\documentclass[12pt,titlepage]{article}

\usepackage{natbib}
\newcommand {\ctn}{\citet} 
\usepackage{graphicx,subfigure,amsmath,latexsym,amssymb}
\usepackage{float,epsfig,multirow,rotating,times}
\usepackage{upgreek,wrapfig}
\usepackage{comment}
\usepackage{multirow}
\usepackage{slashbox}

\newcommand{\blambda}{\boldsymbol{\lambda}}

\newcommand{\bPhi}{\boldsymbol{\Phi}}

\newcommand{\bS}{\boldsymbol{S}}

\newcommand{\bx}{\bm{x}}
\newcommand{\bX}{\boldsymbol{X}}

\newcommand{\bZ}{\boldsymbol{Z}}

\newcommand{\bzero}{\boldsymbol{0}}

\newtheorem{theorem}{Theorem}

\newtheorem{corollary}[theorem]{Corollary}

\newtheorem{lemma}[theorem]{Lemma}

\newtheorem{remark}[theorem]{Remark}

\newenvironment{proof}[1][Proof]{\textbf{#1.} }{\ \rule{0.5em}{0.5em}}

\newcommand{\bm}{\mathbf}

\numberwithin{equation}{section}
\numberwithin{algo}{section}
\numberwithin{table}{section}
\numberwithin{figure}{section}

\usepackage{setspace}
\usepackage[mathscr]{euscript}
\usepackage[margin=1in]{geometry}
\begin{document}


\title{\vspace{-0.8in}
Bayesian Characterizations of Properties of Stochastic Processes with Applications}
\author{Sucharita Roy and Sourabh Bhattacharya\thanks{
Sucharita Roy is an Assistant Professor and HOD of Department of Mathematics, St. Xavier's College, Kolkata, pursuing PhD 
in Interdisciplinary Statistical Research Unit, Indian Statistical Institute, 203, B. T. Road, Kolkata 700108.
Sourabh Bhattacharya is an Associate Professor in Interdisciplinary Statistical Research Unit, Indian Statistical
Institute, 203, B. T. Road, Kolkata 700108.
Corresponding e-mail: sourabh@isical.ac.in.}}
\date{\vspace{-0.5in}}
\maketitle%

\begin{abstract}
In this article, we primarily propose a novel Bayesian characterization of stationary and nonstationary stochastic processes. 
In practice, this theory aims to distinguish between global stationarity and nonstationarity for both parametric and nonparametric stochastic processes. 
Interestingly, our theory builds on our previous work on Bayesian characterization of infinite series, which was applied to verification of the (in)famous Riemann Hypothesis. 
Thus, there seems to be interesting and important connections between pure mathematics and Bayesian statistics, with respect to our proposed ideas. 
We validate our proposed method with simulation and real data experiments associated with different setups. 
In particular, applications of our method include stationarity and nonstationarity determination in various time series models, spatial and spatio-temporal setups, and 
convergence diagnostics of Markov Chain Monte Carlo.
Our results demonstrate very encouraging performance, even in very subtle situations.

Using similar principles, we also provide a novel Bayesian characterization of mutual independence among any number of random variables, using which we characterize
the properties of point processes, including characterizations of Poisson point processes, complete spatial randomness, stationarity and nonstationarity. Applications to 
simulation experiments with ample Poisson and non-Poisson point process models again indicate quite encouraging performance of our proposed ideas.

We further propose a novel recursive Bayesian method for determination of frequencies of oscillatory stochastic processes, based on our general principle.
Simulation studies and real data experiments with varieties of time series models consisting of single and multiple frequencies bring out the worth of our method. 
\\[2mm]
\end{abstract}

\tableofcontents

\pagebreak

\section{Introduction}
\label{sec:intro}

In various areas of statistics dealing with stochastic processes, ascertainment of stationarity or nonstationarity of the process behind the observed data, 
is the primary requirement before postulating a stochastic model. In statistics, empirical plots of the data for visualizing stationarity is quite popular, particularly
in the time series context. However, rigorous ascertainment of stationarity is likely to be carried out via appropriate hypotheses testing procedures. 
In the parametric time series context, stationarity is usually characterized by specific parameters, and by devising suitable testing methods, inference regarding
stationarity can be obtained. Using the result of such a test, appropriate stationarity or nonstationary models can then be built for statistical analysis of the given data. 
Although many tests exist in the time series literature, both parametric and nonparametric, they are meant for specific types of time series. 
Some related works in this regard are \ctn{Dickey79}, \ctn{KPSS92}, \ctn{Philips88}, \ctn{Breitung02}, \ctn{Basu09}, \ctn{Cardinali18}, \ctn{Delft18}. 
In the real data scenario, where the parametric form may itself be called in question, reliability of the tests for stationarity need not be taken for granted.

A very important time series example where studying stationarity property is of utmost importance, is the Markov time series generated by 
Markov Chain Monte Carlo (MCMC) methods, particularly in the Bayesian posterior context. Although in principle there exist many formal theories for addressing
MCMC convergence, they are usually difficult to establish for realistic problems. As a result, plenty of empirical (mostly ad-hoc) methods emerged
for diagnosis of convergence of the MCMC sample to the target posterior distribution, and many such methods are based on visualizing the graphical plots of the
MCMC sample. The available empirical diagnostic tools have the ill reputation of give false impressions about convergence or non-convergence in realistic situations. Moreover,
in reality, the target posteriors can often be multimodal, and in such cases, the performances of such diagnostic tools can be even poorer. For more
about MCMC convergence diagnostics, see, for example, \ctn{Gelman92}, \ctn{Geweke92}, \ctn{Raftery92}, \ctn{Robert95}, \ctn{Gilks96}, \ctn{Cowles96}, 
\ctn{Brooks98a}, \ctn{Brooks98b}, \ctn{Brooks11}, \ctn{Robert04}, \ctn{VRoy19}.

Compared to the time series literature, tests for stationarity in the spatial and spatio-temporal statistics domains are much less developed, and confined 
to checking covariance stationarity only, under assumptions that are often difficult to check in practice. Some relevant works in this regard are
\ctn{Ephraty01}, \ctn{Fuentes02}, \ctn{Guan04}, \ctn{Fuentes05}, \ctn{Li08}, \ctn{Jun12}, \ctn{Soutir17}, \ctn{Soutir17b}. 

In the point process literature, except some simple tests for complete spatial randomness (see, for example, \ctn{OSullivan03},
\ctn{Waller04} and \ctn{Schab05}), there does not seem to exist any formal method to test for Poisson versus
non-Poisson point process, or stationarity versus nonstationarity.

Motivated by the aforementioned problems, we seek a general principle that can attempt to effectively address all such issues. Interestingly, the recursive Bayesian idea
proposed in \ctn{Roy20} to characterize infinite series, turned out to have fruitful extension to our current situations. 
Indeed, the recursive Bayesian concept of \ctn{Roy20} enabled them to study convergence of infinite series whose convergence properties are hitherto unknown. One such
infinite series is also a characterization of the most difficult unsolved problem of mathematics, namely, the Riemann hypothesis. The most surprising result obtained by
\ctn{Roy20} is failure to accept Riemann hypothesis, based on their theory, method and implementation. 
Since the idea of \ctn{Roy20} is primarily about studying deterministic
infinite series, one may be left wondering how this can be useful from the statistical perspective. However, the key concept there is to view the deterministic terms
of the series as realizations from some general stochastic process, then to relate convergence of the series to a quantity that can be interpreted as probability
of convergence of the series under the stochastic process, and finally to build a recursive Bayesian procedure such that the posterior distribution of the probability
of convergence tends to one if and only if the series converges and to zero if and only if it diverges.

From the above summary of the idea of \ctn{Roy20} it can be perceived that the deterministic terms of the infinite series can be easily replaced with random elements
if necessary. For study of stationarity and nonstatonarity, we again relate stationarity to a quantity that admits interpretation as probability that the process is stationary,
and apply the same concept of recursive Bayesian method for characterizations of stationarity and nonstationarity.

In the point process setup, we apply similar principles to characterize complete spatial randomness, using properties of Poisson point process. To characterize
Poisson point process, we first need to characterize mutual independence among a set of random variables. Once we characterize such mutual independence, again using
similar principles and recursive Bayesian concept as before, as we show, characterization of Poisson point process is not difficult to achieve.
For mutual independence we make use of simple break-ups of joint distribution of random variables into products of conditional distributions and Bayesian nonparametrics
based on Dirichlet process (\ctn{Ferguson74}). The latter particularly improves computational efficiency.

Our Bayesian idea can be used in another seemingly unrelated setup, namely, determination of frequencies of oscillations of oscillating stochastic processes.
The basic idea here is to first provide a appropriate bijective transformation to the data such that the transformed process takes values on $0,1]$. The transformed process
can then be raised to some appropriate power such that the oscillations become as explicit as possible. Dividing up the interval $[0,1]$ into appropriate
sub-intervals, we consider the proportions of oscillations contained in the sub-intervals. These can then be related to the frequencies of oscillation of the
underlying stochastic process, and again facilitates characterization with our recursive Bayesian principle. We characterize single and multiple frequencies, as well as 
infinite number of frequencies of oscillation.

The basic aim of this paper is to render our characterization theories amenable to practical applications. To this end, we provide ample illustrations
of our methods and implementations with simulated and real data sets, in each of the aforementioned areas of statistics. Most of our codes are written in C, parallelised
using MPI (Message Passing Interface), and implemented in parallel architectures. Some parallelized R codes are also used in conjunction with our parallel C codes.
Very fast computation is the result of our efforts.

The rest of our article is structured as follows. We begin our treatise in Section \ref{sec:prelude} with some necessary definitions and prove results associated with them.  
With these, we elucidate the key concept behind our proposed ideas in Section \ref{sec:key_concept}, and then introduce our recursive Bayesian procedure for 
studying stationarity in Section \ref{sec:recursive1}. In Section \ref{sec:characterization}, we characterize stationarity and nonstationarity using the recursive
Bayesian procedure. Some relevant computational techniques and their theoretical validation are provided in Section \ref{sec:comp_sup_norm}, and issues related to
discretization associated with our method are discussed in Section \ref{sec:cardinality}. Characterization of second order stationarity, that is stationarity of covariance
structure, is considered in Section \ref{sec:covariance}. Discussion of the role of non-recursive Bayesian procedures for characterizations is provided in Section
\ref{sec:non_recursive}. 

In Section \ref{sec:ar1}, we provide detailed illustration of our theory on characterization of stationarity and nonstationarity with AR(1) models,
along with comparisons with classical tests for stationarity. In Section \ref{sec:other_time_series}, we illustrate our theory and methods on more 
complicated time series models, such as AR(2), ARCH(1) and GARCH(1,1). 
MCMC convergence diagnostics with our Bayesian method is considered in Section \ref{sec:mcmc}. 

In Section \ref{sec:spatial} we illustrate detection of strict and covariance stationarity and nonstationarity in spatial setups, along with comparisons with
existing tests for covariance stationarity. Section \ref{sec:spatio_temporal} is about application of our ideas in spatio-temporal contexts, with comparisons
with existing tests for covariance stationarity. Applications to real spatial and spatio-temporal data sets are considered in Section \ref{sec:realdata_spacetime}.

In Section \ref{sec:point_processes}, using our main principles, we provide Bayesian characterizations of properties of point processes, such as complete
spatial randomness, Poisson point processes along with stationarity and nonstationarity. As a necessary part of such characterizations, we also characterize
mutual independence among a set of random variables, in the same section. In different subsections of the same section we illustrate our theories and methods
with various instances of point processes.

Our Bayesian characterization associated with (multiple) frequency determination of oscillating stochastic processes is detailed in Section \ref{sec:osc_mult},
and illustrated with many examples.

\section{Requisite definitions and associated results -- prelude to the key concept}
\label{sec:prelude}
Consider a stochastic process $\bX=\left\{X_s:s\in\mathcal S\right\}$, where $\mathcal S$ is some arbitrary index set. 
We assume that $\mathcal S=\cup_{i=1}^{\infty}\mathcal M_i$ such that $\mathcal M_i$ are disjoint, and
$\{X_s:s\in\mathcal M_i\}$ is stationary. 
In other words, we assume that $\bX$ is locally stationary. 
We show below that most stochastic processes are approximately locally stationary.
For simplicity of exposition, we consider the case where $s$ is one-dimensional; the higher-dimensional case is a simple generalization.
\begin{theorem}
\label{theorem:local}
For any $(s_1,\ldots,s_m)$, for $m\geq 1$, let $F_{s_1,\ldots,s_m}$ denote the joint distribution function of $\left(X_{s_1},\ldots,X_{s_m}\right)$.	
Assume that for any $(x_1,\ldots,x_m)$, $F_{s_1,\ldots,s_m}$ is differentiable in sufficiently small neighborhoods of $(x_1,\ldots,x_m)$, and that
for $i=1,\ldots,m$, $X_{s_i+h}=X_{s_i}+O_P(h)$, as $h\rightarrow 0$.
Then for any $(x_1,\ldots,x_m)$, $F_{s_1+h,\ldots,s_m+h}\left(x_1,\ldots,x_m\right) = F_{s_1,\ldots,s_m}\left(x_1,\ldots,x_m\right)+O_P(h)$, as $h\rightarrow 0$.
\end{theorem}
\begin{proof}
Let us first assume that $X_s$ are deterministic variables satisfying $X_{s_i+h}=X_{s_i}+O(h)$, as $h\rightarrow 0$, $i=1,\ldots,m$.
Then by Taylor's series expansion up to the first order, using the above condition, reveals that 
$F_{s_1+h,\ldots,s_m+h}\left(x_1,\ldots,x_m\right) = F_{s_1,\ldots,s_m}\left(x_1,\ldots,x_m\right)+O(h)$. Hence, the result follows
by an application of Theorem 7.15 of \ctn{Schervish95}.
\end{proof}
\begin{remark}
\label{remark:local}
The condition $X_{s+h}=X_{s}+O_P(h)$, as $h\rightarrow 0$ is satisfied by stochastic processes $X_s$ with almost surely differentiable paths,
for example, Gaussian processes, with sufficiently smooth covariance structure (see, for example, \ctn{Adler81}, \ctn{Adler07}). Also, non-smooth processes that are 
mean square continuous, in the sense
that $E\left(X_{s+h}-X_s\right)^2\rightarrow 0$, as $h\rightarrow 0$, for any $s$, also satisfy the property. Furthermore, discrete processes such as Poisson processes
satisfy the above property.
Also note that the differentiability condition of $F_{s_1,\ldots,s_m}$ is satisfied by most distribution functions, including the step functions
corresponding to discrete distributions.
\end{remark}

Note that local stationarity does not imply that the entire process is even asymptotically stationary.
However, as we show below, global stationarity is also possible under our setup. Our goal is to distinguish between global (asymptotic) stationarity and nonstationarity.

For all practical purposes, we shall consider realizations of $\bX$ at discrete index points, that is, points on the set 
$\tilde{\mathcal S}=\cup_{i=1}^{\infty}\mathcal N_i$, where
$\mathcal N_i$ is a discretization of $\mathcal M_i$ and
$\{X_s:s\in\mathcal N_i, |\mathcal N_i|=n_i\}$, where $|\mathcal N_i|$ is the cardinality of $\mathcal N_i$, is stationary.
We assume that $|\mathcal N_i|\rightarrow\infty$, for each $i$.
In particular, if $s$ is one-dimensional, then 
$\mathcal N_i=\left\{s_r:\sum_{k=1}^{i-1}n_k\leq r\leq\sum_{k=1}^in_k\right\}$, 
and $|\mathcal N_i|=n_i\rightarrow\infty$ for each $i$; we set $n_0=0$. 

In practice, one can not observe the entire stochastic process $\bX$, even on the discrete set $\tilde{\mathcal S}$. 
Hence, let us assume that only $\bX_K=\left\{X_s:s\in \cup_{i=1}^K\mathcal N_i\right\}$
has been observed, for sufficiently large $K$.

For any Borel set $C$, consider
\begin{equation}
\hat P_i(C)=n_i^{-1}\sum_{s\in\mathcal N_i}I(X_s\in C).
\label{eq:eq1}
\end{equation}

Now let
\begin{align}
\tilde P_K(C)&=\frac{\sum_{s\in\cup_{i=1}^K\mathcal N_i}I(X_s\in C)}{\sum_{i=1}^Kn_i}\notag\\
&=\frac{\sum_{i=1}^Kn_i\hat P_i(C)}{\sum_{i=1}^Kn_i}=\sum_{i=1}^K\hat p_{iK}\hat P_i(C),
\label{eq:eq2}
\end{align}
where $\hat p_{ik}=n_i/\sum_{j=1}^Kn_j$.
By the Glivenko-Cantelli theorem for stationary random variables (see \ctn{Stute80})
\begin{equation}
\underset{C}{\sup}~\left|\hat P_i(C)-P_i(C)\right|\stackrel{a.s.}{\longrightarrow} 0,~\mbox{as}~n_i\rightarrow\infty,
\label{eq:eq3}
\end{equation}
where $P_i(C)$ is the probability that any random variable in $\mathcal N_i$ belongs to $C$.
Note that $P_i(C)$ may itself be a random variable unless $\{X_s:s\in\mathcal N_i, |\mathcal N_i|=n_i\}$ is also ergodic.
Randomness of $P_i(C)$ is not a cause for concern, however, for the methodology that we propose.

Let us now assume that 
\begin{equation}
\hat p_{iK}=\frac{n_i}{\sum_{j=1}^Kn_j}\rightarrow p_{iK}=\frac{p_i}{\sum_{j=1}^Kp_j},
\label{eq:eq4}
\end{equation}
as $n_j\rightarrow\infty$, for $j=1,\ldots,K$. Here $0\leq p_i\leq 1$, such that $\sum_{i=1}^{\infty}p_i=1$.

Let $P_\infty(C)=\sum_{i=1}^\infty p_iP_i(C)$. Then we have the following theorem.
\begin{theorem}
\label{theorem:gc1}
\begin{equation}
	\underset{K\rightarrow\infty}{\lim}\underset{n_i\rightarrow\infty,i=1,\ldots,K}{\lim}~\underset{C}{\sup}~\left|\tilde P_K(C)-P_\infty(C)\right|=0,~\mbox{almost surely}.
\label{eq:gc1}
\end{equation}
\end{theorem}
\begin{proof}
\begin{align}
&\underset{C}{\sup}~\left|\tilde P_K(C)-P_\infty(C)\right|\notag\\
&\quad = \underset{C}{\sup}~\left|\sum_{i=1}^K\hat p_{iK}\hat P_i(C)-\sum_{i=1}^Kp_iP_i(C)-\sum_{i=K+1}^{\infty}p_iP_i(C)\right|\notag\\
&\quad\leq \underset{C}{\sup}~\left|\sum_{i=1}^K\hat p_{iK}\hat P_i(C)-\sum_{i=1}^Kp_iP_i(C)\right|+\underset{C}{\sup}~\left|\sum_{i=K+1}^{\infty}p_iP_i(C)\right|\notag\\
&\quad\leq\sum_{i=1}^Kp_i\left[\underset{C}{\sup}~\left|\hat P_i(C)-P_i(C)\right|\right]+\sum_{i=1}^K\left[\underset{C}{\sup}~\hat P_i(C)\right]\left|\hat p_{iK}-p_i\right|
+\sum_{i=K+1}^{\infty}p_i\left[\underset{C}{\sup}~P_i(C)\right].
\label{eq:eq5}
\end{align}
Now, due to (\ref{eq:eq3}), given $K$, 
\begin{equation*}
\sum_{i=1}^Kp_i\left[\underset{C}{\sup}~\left|\hat P_i(C)-P_i(C)\right|\right]\rightarrow 0,~\mbox{almost surely as}~n_i\rightarrow\infty,~i=1,\ldots,K.
\end{equation*}
Hence, 
\begin{equation}
\underset{K\rightarrow\infty}{\lim}\underset{n_i\rightarrow\infty,i=1,\ldots,K}{\lim}~\sum_{i=1}^Kp_i\left[\underset{C}{\sup}~\left|\hat P_i(C)-P_i(C)\right|\right]= 0,
	~\mbox{almost surely}.
\label{eq:lim1}
\end{equation}

As $n_i\rightarrow\infty$ for $j=1,\ldots,K$ and $K\rightarrow\infty$, the second term of (\ref{eq:eq5}) can be shown to converge to zero in the following way:
\begin{align}
&\underset{n_i\rightarrow\infty,i=1,\ldots,K}{\lim}~\sum_{i=1}^K\left[\underset{C}{\sup}~\hat P_i(C)\right]\left|\hat p_{iK}-p_i\right|\notag\\
&\leq\underset{n_i\rightarrow\infty,i=1,\ldots,K}{\lim}~\sum_{i=1}^K\left|\hat p_{iK}-p_i\right|
=\sum_{i=1}^K\left|p_{iK}-p_i\right|
=\sum_{i=K+1}^{\infty}p_i\rightarrow 0,~\mbox{as}~K\rightarrow\infty.
\label{eq:eq7}
\end{align}
For the third term of (\ref{eq:eq5}), note that
\begin{equation}
\sum_{i=K+1}^{\infty}p_i\left[\underset{C}{\sup}~P_i(C)\right]
\leq \sum_{i=K+1}^{\infty}p_i\rightarrow 0,~\mbox{as}~K\rightarrow\infty.
\label{eq:eq8}
\end{equation}

The result follows by combining (\ref{eq:eq5}), (\ref{eq:lim1}), (\ref{eq:eq7}) and (\ref{eq:eq8}).
\end{proof}

Note that stationarity of the process $\bX$ is characterized by $P_i=P$ for $i=1,2,\ldots$, in which case $P_{\infty}=P$. 
Observe that if $P_i=P_{\infty}$ for $i=1,\ldots,\infty$, it then follows that $P_{\infty}=P$.
Asymptotic stationarity is characterized by $P_i=P$ for $i\geq i_0$, for some $i_0>1$. In this case, if $P_j=P_{i_0,\infty}=
\frac{\sum_{i=i_0+1}^{\infty}p_iP_i}{\sum_{i=i_0+1}^{\infty}p_i}$, for $j>i_0$, then $P_i=P$ for $i>i_0$.
On the other hand, if $\bX$ is nonstationary and not even asymptotically stationary, then $P_i\neq P_j$ for infinitely many $j\neq i$. The latter condition also implies
that there does not exist $i_0>1$ such that $P_j=P_{i_0,\infty}$ for $j>i_0$. Hence, there exists no $i_0>1$ such that $P_i=P$ for $i>i_0$.

\begin{theorem}
\label{theorem:gc2}
$\bX$ is stationary if and only if for $i\geq 1$, $\underset{C}{\sup}~\left|\hat P_i(C)-\tilde P_K(C)\right|\rightarrow 0$ almost surely, as $n_i\rightarrow\infty$
satisfying (\ref{eq:eq4}), $i=1,\ldots,K$, $K\rightarrow\infty$.
\end{theorem}
\begin{proof}
Note that
$\underset{C}{\sup}~\left|\hat P_i(C)-\tilde P_K(C)\right|\leq \underset{C}{\sup}~\left|\hat P_i(C)-P_{\infty}(C)\right|
+\underset{C}{\sup}~\left|\tilde P_K(C)-P_{\infty}(C)\right|$. 
The first part of the right hand side tends to zero almost surely as $n_i\rightarrow\infty$ satisfying (\ref{eq:eq4}), $i=1,\ldots,K$, $K\rightarrow\infty$,
if and only if $\bX$ is stationary, and the second part tends to zero almost surely by Theorem \ref{theorem:gc1}.
\end{proof}

\begin{theorem}
\label{theorem:gc3}
$\bX$ is nonstationary if and only if $\underset{C}{\sup}~\left|\hat P_i(C)-\tilde P_K(C)\right|> 0$ almost surely, as $n_i\rightarrow\infty$
satisfying (\ref{eq:eq4}), $i=1,\ldots,K$, $K\rightarrow\infty$.
\end{theorem}
\begin{proof}
Note that
\begin{equation}
\left|\hat P_i(C)-\tilde P_K(C)\right|\geq \left|\left|\hat P_i(C)-P_{\infty}(C)\right|
-\left|\tilde P_K(C)-P_{\infty}(C)\right|\right|.
\label{eq:gc3_1}
\end{equation}
By Theorem \ref{theorem:gc1}, for any $\epsilon_1>0$, 
\begin{equation}
\left|\tilde P_K(C)-P_{\infty}(C)\right|<\epsilon_1,
\label{eq:gc3_2}
\end{equation}
for all $C$, for sufficiently large $n_i$
satisfying (\ref{eq:eq4}) and sufficiently large $K$. Also, 
\begin{equation}
\left|\hat P_i(C)-P_{\infty}(C)\right|\geq \left|\left|P_i(C)-P_{\infty}(C)\right|-
\left|\hat P_i(C)-P_i(C)\right|\right|. 
\label{eq:gc3_3}
\end{equation}
By (\ref{eq:eq3}), for any $\epsilon_2>0$, $\left|\hat P_i(C)-P_i(C)\right|<\epsilon_2$, for all $C$, as $n_i\rightarrow\infty$.
But $\left|P_i(C)-P_{\infty}(C)\right|>0$, at least for some $C$, since $P_i\neq P_j$ for infinitely many $j\neq i$. Since $\epsilon_2~(>0)$
is arbitrary, it follows from these arguments and (\ref{eq:gc3_3}), that
\begin{equation}
\left|\hat P_i(C)-P_{\infty}(C)\right|>0,~\mbox{for some}~C,~\mbox{for sufficiently large}~n_i.
\label{eq:gc3_4}
\end{equation}
Since $\epsilon_1~(>0)$ in (\ref{eq:gc3_2}) is also arbitrary, combining (\ref{eq:gc3_4}), (\ref{eq:gc3_2}) and (\ref{eq:gc3_1}) it is evident that
the right hand side of (\ref{eq:gc3_1}) is positive for some $C$ for sufficiently large $n_i$ satisfying (\ref{eq:eq4}) and sufficiently large $K$.
Hence, $$\underset{C}{\sup}~\left|\hat P_i(C)-\tilde P_K(C)\right|>0$$
almost surely, as $n_i\rightarrow\infty$ satisfying (\ref{eq:eq4}), $i=1,\ldots,K$, $K\rightarrow\infty$.
\end{proof}

\section{The key concept}
\label{sec:key_concept}

Let $p_{j,n_j}=P\left(\underset{C}{\sup}~\left|\hat P_j(C)-\tilde P_K(C)\right|\leq c_j\right)$. As will be seen later, this can be interpreted as 
the probability that the underlying process is stationary when the observed data is $\mathbb I\left\{\underset{C}{\sup}~\left|\hat P_j(C)-\tilde P_K(C)\right|\leq c_j\right\}$.
Note that, for stationarity, due to Theorem \ref{theorem:gc2}, for $j=1,\ldots,K$, as $n_j\rightarrow\infty$, $K\rightarrow\infty$, 
the latter converges to one almost surely. 
Since $p_{j,n_j}=E\left[\mathbb I\left\{\underset{C}{\sup}~\left|\hat P_j(C)-\tilde P_K(C)\right|\leq c_j\right\}\right]$,
uniform integrability leads one to expect that for $j\geq 1$, 
for any choice of the non-negative monotonically decreasing sequence $\{c_j\}_{j=1}^{\infty}$,
\begin{align}
&\underset{K\rightarrow\infty}{\lim}\underset{n_j\rightarrow\infty,j=1,\ldots,K}{\lim}~p_{j,n_j}\notag\\
&\qquad=\underset{K\rightarrow\infty}{\lim}\underset{n_j\rightarrow\infty,j=1,\ldots,K}{\lim}
~P\left(\underset{C}{\sup}~\left|\hat P_i(C)-\tilde P_K(C)\right|\leq c_j\right)\notag\\
&\qquad=\underset{K\rightarrow\infty}{\lim}\underset{n_j\rightarrow\infty,j=1,\ldots,K}{\lim}~
E\left[\mathbb I\left\{\underset{C}{\sup}~\left|\hat P_i(C)-\tilde P_K(C)\right|\leq c_j\right\}\right]\notag\\
&\qquad=1.\notag
\end{align}
Similarly, for nonstationarity, we expect, using Theorem \ref{theorem:gc3} that for $j\geq j_0\geq 1$,
\begin{equation*}
\underset{K\rightarrow\infty}{\lim}\underset{n_j\rightarrow\infty,j=1,\ldots,K}{\lim}~p_{j,n_j}=0
\end{equation*}
almost surely, for any choice of the non-negative monotonically decreasing sequence $\{c_j\}_{j=1}^{\infty}$.

In reality it is not known if $p_{j,n_j}$ converges to zero or one, since it is not known if $\bX$ is stationary or nonstationary.
Thus, we consider learning about $p_{j,n_j}$ from the data $\bX_K$ and some appropriate prior on $p_{j,n_j}$ in the form of 
the posterior $\pi\left(p_{j,n_j}|\bX_K\right)$. As we will show, 
$$\underset{K\rightarrow\infty}{\lim}\underset{n_j\rightarrow\infty,j=1,\ldots,K}{\lim}~\pi\left(p_{j,n_j}|\bX_K\right)=1,~\mbox{almost surely}$$ for $j\geq 1$ and 
any choice of the non-negative monotonically decreasing sequence $\{c_j\}_{j=1}^{\infty}$, characterizes stationarity of $\bX$ and
$$\underset{K\rightarrow\infty}{\lim}\underset{n_j\rightarrow\infty,j=1,\ldots,K}{\lim}~\pi\left(p_{j,n_j}|\bX_K\right)=0,~\mbox{almost surely}$$ for $j\geq j_0\geq 1$,
for any choice of the non-negative monotonically decreasing sequence $\{c_j\}_{j=1}^{\infty}$, characterizes nonstationarity of $\bX$.


In Section \ref{sec:recursive1} we devise a recursive Bayesian methodology that achieves the goal discussed above.

\section{A recursive Bayesian procedure for studying stationarity}
\label{sec:recursive1}

Since we view $X_i$ as realizations from some random process, we first formalize the notion
in terms of the relevant probability space.
Let $(\Omega,\mathcal A,\mu)$ be a probability space, where $\Omega$ is the sample space,
$\mathcal A$ is the Borel $\sigma$-field on $\Omega$, and $\mu$ is some probability measure.
Let, for $i=1,2,3,\ldots$, $X_i:\Omega\mapsto\mathbb R$ be real valued random variables
measurable with respect to the Borel $\sigma$-field $\mathcal B$ on $\mathbb R$.
As in \ctn{Schervish95}, we can then define a $\sigma$-field of subsets of $\mathbb R^{\infty}$ with
respect to which $X=(X_1,X_2,\ldots)$ is measurable. Indeed, let us define $\mathbb B^{\infty}$ to be
the smallest $\sigma$-field containing sets of the form 
\begin{align}
B&=\left\{X:X_{i_1}\leq r_1,X_{i_2}\leq r_2,\ldots,X_{i_p}\leq r_p,~\mbox{for some}~p\geq 1,\right.\notag\\
&\quad\quad\left.~\mbox{some integers}~
i_1,i_2,\ldots,i_p,~\mbox{and some real numbers}~r_1,r_2,\ldots,r_p\right\}.\notag
\end{align}
Since $B$ is an intersection of finite number of sets  
of the form $\left\{X:X_{i_j}\leq r_j\right\}$; $j=1,\ldots,p$, all of which belong to $\mathcal A$ (since
$X_{i_j}$ are measurable)
it follows that $X^{-1}(B)\in\mathcal A$, so that $X$ is measurable with respect to 
$(\mathbb R^{\infty},\mathbb B^{\infty},P)$, where $P$ is the probability measure induced by $\mu$.

Alternatively, note that it is possible to represent any stochastic process $\{X_i:i\in \mathfrak I\}$, for fixed
$i$ as a random variable $\omega\mapsto X_i(\omega)$, where $\omega\in\mathfrak S$;
$\mathfrak S$ being the set of all functions from $\mathfrak I$ into $\mathbb R$. 
Also, fixing $\omega\in\mathfrak S$, the function $i\mapsto X_i(\omega);~i\in \mathfrak I$,
represents a path of $X_i;~i\in\mathfrak I$. Indeed, we can identify $\omega$ with the function
$i\mapsto X_i(\omega)$ from $\mathfrak I$ to $\mathbb R$; see, for example, \ctn{Oksendal00}, for
a lucid discussion.

This latter identification will be convenient for our purpose, and we adopt this in this article.
Note that the $\sigma$-algebra $\mathcal F$ induced by $X$
is generated by sets of the form
\[
\left\{\omega:\omega(i_1)\in B_1,\omega(i_2)\in B_2,\ldots,\omega(i_k)\in B_k\right\},
\]
where $B_j\subset\mathbb R;~j=1,\ldots,k$, are Borel sets in $\mathbb R$.   

\subsection{Development of the stage-wise likelihoods}
\label{subsec:Bayesian_method}

Let $\{c_j\}_{j=1}^{\infty}$ be a non-negative decreasing sequence and
\begin{equation}
Y_{j,n_j}=\mathbb I{\left\{\underset{C}{\sup}~\left|\hat P_j(C)-\tilde P_K(C)\right|\leq c_j\right\}}.
\label{eq:Y_j_n}
\end{equation}
Let, for $j\geq 1$,
\begin{equation}
P\left(Y_{j,n_j}=1\right)=p_{j,n_j}.
\label{eq:p_j_n}
\end{equation}
Hence, the likelihood of $p_{j,n_j}$, given $y_{j,n_j}$, is given by
\begin{equation}
L\left(p_{j,n_j}\right)=p^{y_{j,n_j}}_{j,n_j}\left(1-p_{j,n_j}\right)^{1-y_{j,n_j}}
\label{eq:likelihood}
\end{equation}
It is important to relate $p_{j,n_j}$ to stationarity of the underlying series.
Note that $p_{j,n_j}$ is the probability that $\underset{C}{\sup}~\left|\hat P_j(C)-\tilde P_K(C)\right|$ falls below $c_j$. Thus, 
$p_{j,n_j}$ can be interpreted as the probability that the
process $\bX$ is stationary when the data observed is $Y_{j,n_j}$. 
If $\bX$ is stationary, then due to Theorem \ref{theorem:gc2} it is to be expected {\it a posteriori}, that for $j\geq 1$, 
for any non-negative decreasing sequence $\{c_j\}_{j=1}^{\infty}$, 
\begin{equation}
p_{j,n_j}\rightarrow 1\quad\mbox{as}~n_j\rightarrow\infty,~\mbox{satisfying}~(\ref{eq:eq4}).
\label{eq:convergent_p}
\end{equation}
Indeed, as we will formally show, condition (\ref{eq:convergent_p}) 
is both necessary and sufficient for stationarity of $\bX$. 

On the other hand, if $\bX$ is nonstationary, then there exists $j_0\geq 1$ 
such that for every $j>j_0$, as $n_j\rightarrow\infty$ satisfying (\ref{eq:eq4}), $\underset{C}{\sup}~\left|\hat P_j(C)-\tilde P_K(C)\right|>c_j$, 
for any non-negative decreasing sequence $\{c_j\}_{j=1}^{\infty}$, due to Theorem \ref{theorem:gc3}. Here we expect, 
{\it a posteriori}, that
\begin{equation}
p_{j,n_j}\rightarrow 0\quad\mbox{as}~n_j\rightarrow\infty,~\mbox{satisfying}~(\ref{eq:eq4}),
\label{eq:divergent_p}
\end{equation}
for $j\geq j_0\geq 1$.
Again, we will prove formally that the above condition is both necessary and sufficient for divergence.

In what follows we shall first construct a recursive Bayesian methodology that formally characterizes
convergence and divergence in terms of formal posterior convergence related to (\ref{eq:convergent_p})
and (\ref{eq:divergent_p}).


\subsection{Development of recursive Bayesian posteriors}
\label{subsec:recursive_posteriors}


We assume that $\left\{y_{j,n_j};j=1,2,\ldots\right\}$ is observed successively at stages indexed by $j$.
That is, we first observe $y_{1,n_1}$, and based on our prior belief regarding the first stage probability, 
$p_{1,n_1}$, compute the posterior distribution of $p_{1,n_1}$ given $y_{1,n_1}$, which we denote by
$\pi(p_{1,n_1}|y_{1,n_1})$.
Based on this posterior we construct a prior for the second stage, and compute the posterior
$\pi(p_{2,n_2}|y_{2,n_2})$. We continue this procedure for as many stages as we desire.
Details follow.

Consider the sequences $\left\{\alpha_j\right\}_{j=1}^{\infty}$ and $\left\{\beta_j\right\}_{j=1}^{\infty}$,
where $\alpha_j=\beta_j=1/j^2$ for $j=1,2,\ldots$.
At the first stage of our recursive Bayesian algorithm, that is, when $j=1$, 
let us assume that the prior is given by
\begin{equation}
\pi(p_{1,n_1})\equiv Beta(\alpha_1,\beta_1),
\label{eq:prior_stage_1}
\end{equation}
where, for $a>0$ and $b>0$, $Beta(a,b)$ denotes the Beta distribution with mean $a/(a+b)$
and variance $(ab)/\left\{(a+b)^2(a+b+1)\right\}$.
Combining this prior with the
likelihood (\ref{eq:likelihood}) (with $j=1$), we obtain the following posterior of $p_{1,n_1}$ given $y_{1,n_1}$:
\begin{equation}
\pi(p_{1,n_1}|y_{1,n_1})\equiv Beta\left(\alpha_1+y_{1,n_1},\beta_1+1-y_{1,n_1}\right).
\label{eq:posterior_stage_1}
\end{equation}
At the second stage (that is, for $j=2$), for the prior of $p_{2,n_2}$ we consider the posterior
of $p_{1,n_1}$ given $y_{1,n_1}$ associated with the $Beta(\alpha_1+\alpha_2,\beta_1+\beta_2)$ prior.
That is, our prior on $p_{2,n_2}$ is given by:
\begin{equation}
\pi(p_{2,n_2})\equiv Beta\left(\alpha_1+\alpha_2+y_{1,n_1},\beta_1+\beta_2+1-y_{1,n_1}\right).
\label{eq:prior_stage_2}
\end{equation}
The reason for such a prior choice is that the uncertainty regarding convergence of the series
is reduced once we obtain the posterior at the first stage, so that at the second stage the uncertainty 
regarding the prior is expected to be lesser compared to the first stage posterior. With our choice, it 
is easy to see that the prior variance at the second stage, given by 
$$\left\{(\alpha_1+\alpha_2+y_{1,n_1})(\beta_1+\beta_2+1-y_{1,n_1})\right\}/\left\{(\alpha_1+\alpha_2+\beta_1+\beta_2+1)^2
(\alpha_1+\alpha_2+\beta_1+\beta_2+2)\right\},$$ 
is smaller than the first stage posterior variance, given by
$$\left\{(\alpha_1+y_{1,n_1})(\beta_1+1-y_{1,n_1})\right\}/\left\{(\alpha_1+\beta_1+1)^2
(\alpha_1+\beta_1+2)\right\}.$$ 

The posterior of $p_{2,n_2}$ given $y_{2,n_2}$ is then obtained by combining the second stage prior
(\ref{eq:prior_stage_2}) with (\ref{eq:likelihood}) (with $j=2$). The form of the posterior
at the second stage is thus given by
\begin{equation}
\pi(p_{2,n_2}|y_{2,n_2})\equiv Beta\left(\alpha_1+\alpha_2+y_{1,n_1}+y_{2,n_2},\beta_1+\beta_2+2-y_{1,n_1}-y_{2,n_2}\right).
\label{eq:posterior_stage_2}
\end{equation}

Continuing this way, at the $k$-th stage, where $k>1$, we obtain the following posterior of $p_{k,n_k}$:
\begin{equation}
\pi(p_{k,n_k}|y_{k,n_k})\equiv Beta\left(\sum_{j=1}^k\alpha_j+\sum_{j=1}^ky_{j,n_j},
k+\sum_{j=1}^k\beta_j-\sum_{j=1}^ky_{j,n_j}\right).
\label{eq:posterior_stage_k}
\end{equation}

It follows from (\ref{eq:posterior_stage_k}) that
\begin{align}
E\left(p_{k,n_k}|y_{k,n_k}\right)&=\frac{\sum_{j=1}^k\alpha_j
+\sum_{j=1}^ky_{j,n_j}}{k+\sum_{j=1}^k\alpha_j+\sum_{j=1}^k\beta_j};
\label{eq:postmean_p_k}\\
Var\left(p_{k,n_k}|y_{k,n_k}\right)&=
\frac{(\sum_{j=1}^k\alpha_j+\sum_{j=1}^ky_{j,n_j})(k+\sum_{j=1}^k\beta_j-\sum_{j=1}^ky_{j,n_j})}
{(k+\sum_{j=1}^k\alpha_j+\sum_{j=1}^k\beta_j)^2(1+k+\sum_{j=1}^k\alpha_j+\sum_{j=1}^k\beta_j)}.
\label{eq:postvar_p_k}
\end{align}
Since $\sum_{j=1}^k\alpha_j=\sum_{j=1}^k\beta_j=\sum_{j=1}^k\frac{1}{j^2}$, (\ref{eq:postmean_p_k})
and (\ref{eq:postvar_p_k}) admit the following simplifications:
\begin{align}
E\left(p_{k,n_k}|y_{k,n_k}\right)&=\frac{\sum_{j=1}^k\frac{1}{j^2}+\sum_{j=1}^ky_{j,n_j}}
{k+2\sum_{j=1}^k\frac{1}{j^2}};
\label{eq:postmean_p_k_2}\\
Var\left(p_{k,n_k}|y_{k,n_k}\right)&=
\frac{(\sum_{j=1}^k\frac{1}{j^2}+\sum_{j=1}^ky_{j,n_j})(k+\sum_{j=1}^k\frac{1}{j^2}-\sum_{j=1}^ky_{j,n_j})}
{(k+2\sum_{j=1}^k\frac{1}{j^2})^2(1+k+2\sum_{j=1}^k\frac{1}{j^2})}.
\label{eq:postvar_p_k_2}
\end{align}

\section{Characterization of stationarity properties of the underlying process}
\label{sec:characterization}
Based on our recursive Bayesian theory we have the following theorem that characterizes 
stationarity of $\bX$ in terms of the limit of the posterior
probability of $p_{k,n_k}$, as $n_k\rightarrow\infty$ satisfying (\ref{eq:eq4}) and $K\rightarrow\infty$.
We also assume, for the sake of generality, that for any $\omega\in\mathfrak S\cap\mathfrak N^c$, where
$\mathfrak N~(\subset\mathfrak S)$ has zero probability measure, the non-negative monotonically 
decreasing sequence $\{c_j\}_{j=1}^{\infty}$
depends upon $\omega$, so that we shall denote the sequence by $\{c_j(\omega)\}_{j=1}^{\infty}$.
In other words, we allow $\left\{c_j(\omega)\right\}_{j=1}^{\infty}$ to depend upon the corresponding data $X(\omega)$. 
Since $\underset{C}{\sup}~\left|\hat P_j(C)-\tilde P_K(C)\right|\leq 1$ and tends to zero in the case of stationarity, 
there exists a 
monotonically decreasing sequence $\left\{c_j(\omega)\right\}_{j=1}^{\infty}$ such that for $n_j;~j=1,\ldots,K$ sufficiently large
satisfying (\ref{eq:eq4}),
\begin{equation}
\underset{C}{\sup}~\left|\hat P_j(C)(\omega)-\tilde P_K(C)(\omega)\right|\leq c_j(\omega),~\mbox{for}~j\geq 1. 
\label{eq:bound_S}
\end{equation}

\begin{theorem}
\label{theorem:convergence}
For all $\omega\in\mathfrak S\cap\mathfrak N^c$, where $\mathfrak N$ is some null set having probability measure zero,
$\bX$ is stationary if and only if 
for any monotonically decreasing sequence 
$\left\{c_j(\omega)\right\}_{j=1}^{\infty}$,
\begin{equation}
\pi\left(\mathcal N_1|y_{k,n_k}(\omega)\right)\rightarrow 1,
\label{eq:consistency_at_1}
\end{equation}
as $k\rightarrow\infty$ and $n_j\rightarrow\infty$ for $j=1,\ldots,K$ satisfying (\ref{eq:eq4}) and $K\rightarrow\infty$, 
where $\mathcal N_1$ is any neighborhood of 1 (one).
\end{theorem}
\begin{proof}
Let, for $\omega\in\mathfrak S\cap\mathfrak N^c$, where
$\mathfrak N$ is some null set having probability measure zero, $\bX$ be stationary. 
Then, by (\ref{eq:bound_S}), $\underset{C}{\sup}~\left|\hat P_j(C)(\omega)-\tilde P_K(C)(\omega)\right|\leq c_j(\omega)$ for $n_j$ sufficiently large
satisfying (\ref{eq:eq4}), given any choice of the monotonically decreasing sequence $\left\{c_j(\omega)\right\}_{j=1}^{\infty}$. Hence, $y_{j,n_j}(\omega)=1$ 
for sufficiently large $n_j$, satisfying (\ref{eq:eq4}), for $j\geq 1$. 
Hence, in this case, $\sum_{j=1}^ky_{j,n_j}(\omega)=k$,
Also, $\sum_{j=1}^k\frac{1}{j^2}\rightarrow\frac{\pi^2}{6}$, as $k\rightarrow\infty$. 
Consequently, it is easy to see that
\begin{align}
\mu_k=E\left(p_{k,n_k}|y_{k,n_k}(\omega)\right)&\sim\frac{\frac{\pi^2}{6}+k}{k+\frac{\pi^2}{3}}
\rightarrow 1,~\mbox{as}~k\rightarrow\infty,~\mbox{and},
\label{eq:postmean_p_k_limit}\\
\sigma^2_k=Var\left(p_{k,n_k}|y_{k,n_k}(\omega)\right)&\sim
\frac{(\frac{\pi^2}{6}+k)(\frac{\pi^2}{6})}{(k+\frac{\pi^2}{3})^2(1+k+\frac{\pi^2}{3})}
\rightarrow 0~\mbox{as}~k\rightarrow\infty.
\label{eq:postvar_p_k_limit}
\end{align}
In the above, for any two sequences $\left\{a_k\right\}_{k=1}^{\infty}$ and $\left\{b_k\right\}_{k=1}^{\infty}$,
$a_k\sim b_k$ indicates $\frac{a_k}{b_k}\rightarrow 1$, as $k\rightarrow\infty$. 
Now let $\mathcal N_1$ denote any neighborhood of 1, and let $\epsilon>0$ be sufficiently small such that
$\mathcal N_1\supseteq\left\{1-p_{k,n_k}<\epsilon\right\}$. Combining (\ref{eq:postmean_p_k_limit})
and (\ref{eq:postvar_p_k_limit}) with Chebychev's inequality ensures
that (\ref{eq:consistency_at_1}) holds. 

Now assume that (\ref{eq:consistency_at_1}) holds. 
Then for any given $\epsilon>0$, 
\begin{equation}
\pi\left(p_{k,n_k}>1-\epsilon|y_{k,n_k}(\omega)\right)\rightarrow 1,~\mbox{as}~k\rightarrow\infty.
\label{eq:post1}
\end{equation}
Hence,
\begin{align}
E\left(p_{k,n_k}|y_{k,n_k}(\omega)\right)&\rightarrow 1;
\label{eq:postmean1}\\
Var\left(p_{k,n_k}|y_{k,n_k}(\omega)\right)&\rightarrow 0,
\label{eq:postvar1}
\end{align}
as $k\rightarrow\infty$.
If $\bX$ is nonstationary, then there exists $j_0(\omega)$ such that for each 
$j\geq j_0(\omega)$, for sufficiently large $n_j$ satisfying 
$\underset{C}{\sup}~\left|\hat P_j(C)(\omega)-\tilde P_K(C)(\omega)\right|>c_j(\omega)$, for $j\geq j_0(\omega)$, for
any choice of non-negative sequence $\{c_j(\omega)\}_{j=1}^{\infty}$ monotonically converging to zero. 
Hence, in this situation, 
$0\leq \sum_{j=1}^ky_{j,n_j}(\omega)\leq j_0(\omega)$.
Substituting this in (\ref{eq:postmean_p_k_2}) and (\ref{eq:postvar_p_k_2}), it is easy to see that,
as $k\rightarrow\infty$,
\begin{align}
E\left(p_{k,n_k}|y_{k,n_k}(\omega)\right)\rightarrow 0;
\label{eq:postmean_div}\\
Var\left(p_{k,n_k}|y_{k,n_k}(\omega)\right)\rightarrow 0,
\label{eq:postvar_div}
\end{align}
so that (\ref{eq:postmean1}) is contradicted.

\end{proof}

We now prove the following theorem that provides necessary and sufficient conditions for
nonstationarity of $\bX$ in terms of the limit of the posterior
probability of $p_{k,n_k(\omega)}$, as $n_k\rightarrow\infty$ satisfying (\ref{eq:eq4}).
\begin{theorem}
\label{theorem:divergence}
$\bX$ is nonstationary if and only if 
for any $\omega\in\mathfrak S\cap\mathfrak N^c$ where $\mathfrak N$ is some null set having probability measure zero, 
for any choice of the non-negative, monotonically decreasing sequence $\{c_j(\omega)\}_{j=1}^{\infty}$,
\begin{equation}
\pi\left(\mathcal N_0|y_{k,n_k(\omega)}(\omega)\right)\rightarrow 1,
\label{eq:consistency_at_0}
\end{equation}
as $k\rightarrow\infty$ and $n_j\rightarrow\infty$, $j=1,\ldots,K$ satisfying (\ref{eq:eq4}), and $K\rightarrow\infty$, 
where $\mathcal N_0$ is any neighborhood of 0 (zero).
\end{theorem}
\begin{proof}
Assume that $\bX$ is nonstationary. Then 
there exists $j_0(\omega)\geq 1$ such that for every $j\geq j_0(\omega)$, 
$\underset{C}{\sup}~\left|\hat P_j(C)(\omega)-\tilde P_K(C)(\omega)\right|>c_j(\omega)$, for sufficiently large $n_j$, 
for any choice of non-negative sequence $\{c_j(\omega)\}_{j=1}^{\infty}$
monotonically converging to zero. 
From the proof of the sufficient condition of Theorem \ref{theorem:convergence} it follows that
(\ref{eq:postmean_div}) and (\ref{eq:postvar_div}) hold.
Let $\epsilon>0$ be small enough so that $\mathcal N_0\supseteq\left\{p_{k,n_k}<\epsilon\right\}$. Then
combining Chebychev's inequality with (\ref{eq:postmean_div}) and (\ref{eq:postvar_div})
it is easy to see that (\ref{eq:consistency_at_0}) holds.

Now assume that (\ref{eq:consistency_at_0}) holds. 
Then for any given $\epsilon>0$, 
\begin{equation}
\pi\left(p_{k,n_k}<\epsilon|y_{k,n_k}(\omega)\right)\rightarrow 1,~\mbox{as}~k\rightarrow\infty.
\label{eq:post2}
\end{equation}
It follows that 
\begin{align}
E\left(p_{k,n_k}|y_{k,n_k}(\omega)\right)&\rightarrow 0; 
\label{eq:postmean2}\\
Var\left(p_{k,n_k}|y_{k,n_k}(\omega)\right)&\rightarrow 0,
\label{eq:postvar2}
\end{align}
as $k\rightarrow\infty$.  

If $\bX$ is stationary, then by Theorem \ref{theorem:convergence}, 
$\pi\left(\mathcal N_1|y_{k,n_k}(\omega)\right)\rightarrow 1$ as $k\rightarrow\infty$, for
all sequences $\{n_j\}_{j=1}^{\infty}$, so that
$E\left(p_{k,n_k}|y_{k,n_k}(\omega)\right)\rightarrow 1$, which is a contradiction to (\ref{eq:postmean2}). 

\end{proof}


\section{Computation of the sup norm between empirical distribution functions associated with $\hat P_j$ and $\tilde P_K$}
\label{sec:comp_sup_norm}
In all practical applications that involves identifying stationarity or nonstationarity by our method, it is needed to compute the sup norms
$\underset{C}{\sup}~|\hat P_j(C)-\tilde P_K(C)|$; $j\geq 1$. For this purpose, it is sufficient to compute 
$\underset{-\infty<x<\infty}{\sup}~|\hat F_j(x)-\tilde F_K(x)|$, where $\hat F_j(x)$ and $\tilde F_K(x)$ stand for
the empirical distribution functions corresponding to $\hat P_j$ and $\tilde P_K$. Lemma \ref{lemma:lemma_tv} provides the formula for
the desired sup norm. 
\begin{lemma}
\label{lemma:lemma_tv}	
Let $\hat F_j(x)$ and $\tilde F_K(x)$ denote the empirical distribution functions corresponding to empirical probability distributions 
$\hat P_j$ and $\tilde P_K$, respectively. Then it holds that	
\begin{equation}
\underset{-\infty<x<\infty}{\sup}~|\hat F_j(x)-\tilde F_K(x)|=1-\tilde F_K(\hat x_j),
\label{eq:tv_norm1}
\end{equation}
where $\hat x_j=\max \mathcal N_j$, provided that $\hat x_j\neq \max \left\{\cup_{k=1}^K\mathcal N_k\right\}$.
\end{lemma}
\begin{proof}
Since both $\hat F_j(x)$ and $\tilde F_K(x)$ are empirical distribution functions, their jumps occur at the order statistics associated with the
sample data. Now, by inspection it can be seen that, if $\hat x_j\neq \max \left\{\cup_{k=1}^K\mathcal N_k\right\}$, then 
\begin{equation}
|\hat F_j(\hat x_j)-\tilde F_K(\hat x_j)=1-\tilde F_K(\hat x_j). 
\label{eq:tv_norm2}
\end{equation}
For the $r$-th order statistic value $x_{(t)}$, $t\geq 1$ such that $x_{(t)}\neq \hat x_j$,
$|\hat F_j(\hat x_j)-\tilde F_K(\hat x_j)$ is of the form $\left|\frac{\ell}{n_j}-\frac{r}{\sum_{k=1}^Kn_k}\right|$, where 
$1<\ell<n_j$, $1<r<\sum_{k=1}^Kn_k$. But, for $1\leq m\leq\sum_{k=1}^Kn_k$, 
\begin{equation}
1-\frac{m}{\sum_{k=1}^Kn_k}\geq\left|\frac{\ell}{n_j}-\frac{r}{\sum_{k=1}^Kn_k}\right|.
\label{eq:tv_norm3}
\end{equation}	
Since $1-\tilde F_K(\hat x_j)$ in (\ref{eq:tv_norm2}) is of the form $1-\frac{m}{\sum_{k=1}^Kn_k}$, it follows from (\ref{eq:tv_norm3})
that (\ref{eq:tv_norm1}) holds. 	
\end{proof}
\begin{remark}
\label{remark:remark_tv}
Lemma \ref{lemma:lemma_tv} gives the formula for the sup norm when $\hat x_j\neq \max \left\{\cup_{k=1}^K\mathcal N_k\right\}$. In fact, (\ref{eq:tv_norm1})
is no longer valid when $\hat x_j=\max \left\{\cup_{k=1}^K\mathcal N_k\right\}$. Note that there exists exactly one $k\geq 1$ such that
$\hat x_{j^*}=\max \left\{\cup_{k=1}^K\mathcal N_k\right\}$. For that $j^*$, there is no direct formula for the sup norm, and it is desirable to compute the sup norm by 
evaluating the differences between the empirical distribution functions at all the sample order statistics. However, just for a single $k$, such elaborate computation 
is not worthwhile. Instead it makes sense to construct $\hat F_{j^*}$ based on all the observations in $\mathcal N_{j^*}$ except $\hat x_{j^*}$.
Hence, if $\tilde x_{j^*}$ is the maximum of $\mathcal N_{j^*}\backslash\left\{\hat x_{j^*}\right\}$, then in that case,
$\underset{-\infty<x<\infty}{\sup}~|\hat F_{j^*}(x)-\tilde F_K(x)|=1-\tilde F_K(\tilde x_{j^*})$, which is what we shall use in our practical applications.
\end{remark}

\section{Choice of the cardinality of $\mathcal N_i$}
\label{sec:cardinality}
An important ingredient of our method, particularly tied to practical implementation, is the choice of the number of random variables in the sets $\mathcal N_i$.
Recall that $\mathcal N_i$ is discretization of an index set ${\mathcal M_i}$, on which $s$ varies
continuously, such that $\left\{X_s:s\in{\mathcal M_i}\right\}$ is stationary. Let the closure of $\mathcal M_i$, denoted by $\overline{\mathcal M_i}$, be compact.

Let the index $s\in\mathbb R^p$, for $p\geq 1$. For $j=1,2,\ldots$, 
consider $p$-dimensional balls $B_p(c_j,r)$ with centers $c_j$ and radius $r>0$ such that for any $s\in\overline{\mathcal M_i}$, there exists $j\geq 1$ such that
$s\in B_p(c_j,\epsilon)$. Then the set $\left\{B_p(c_j,\epsilon):j\geq 1\right\}$ constitutes an open cover for $\overline{\mathcal M_i}$. By compactness, there exists a set
$\left\{B_p(c_{j_k},\epsilon):k=1,\ldots,n_i\right\}$, for finite $n_i\geq 1$ such that $\overline{\mathcal M_i}\subseteq\cup_{k=1}^{n_i}B_p(c_{j_k},\epsilon)$.
It follows that 
\begin{equation}
\mbox{Vol}\left(\overline{\mathcal M_i}\right)\leq\sum_{k=1}^{n_i}\mbox{Vol}\left(B_p(c_{j_k},\epsilon)\right), 
\label{eq:vol1}
\end{equation}
where for any set $S$, $\mbox{Vol}(S)$ denotes the volume of $S$. Since $\mbox{Vol}\left(B_p(c_{j_k},\epsilon)\right)=\mbox{Vol}\left(B_p(\bzero,\epsilon)\right)$, 
the $p$-dimensional
ball with center $\bzero$, and since $\mbox{Vol}\left(B_p(\bzero,\epsilon)\right)=\frac{\pi^{p/2}}{\Gamma(p/2+1)}\epsilon^p$, it follows from (\ref{eq:vol1}) that
\begin{equation}
n_i\geq\left(\frac{\mbox{Vol}\left(\overline{\mathcal M_i}\right)}{\epsilon^p}\right)\left(\frac{\Gamma\left(p/2+1\right)}{\pi^{p/2}}\right).
\label{eq:vol2}
\end{equation}
For example, if $\mathcal M_i$ is a $p$-dimensional hypercube with $c_i~(>0)$ being the length of each edge, then it follows from (\ref{eq:vol2}) that
$n_i\geq\left(\frac{c_i}{\epsilon}\right)^p\left(\frac{\Gamma\left(p/2+1\right)}{\pi^{p/2}}\right)$.
For example, if $p=1$ and $c=3\epsilon$, then $n\geq 1.5$; if $p=2$ and $c=3\epsilon$, then $n\geq 2.865$; $p=3$ and $c=3\epsilon$, implies $n\geq 6.446$, etc.
Similar idea has been considered in Section 1.2.1 of \ctn{Giraud15}, in the context of large $p$.
In our illustrations, the total number of observations are allocated to a substantially large number of cubes of dimensions one, two and three.
Consequently, $c/\epsilon$ is not expected to be significantly larger than one. 
As such, we take care such that the cube containing the minimum number of observations has at least three observations.

\section{Stationarity of covariance structure}
\label{sec:covariance}

Let $Y_{(s_1,s_2)}=X_{s_1}X_{s_2}$, $\mathcal N_{ih}=\left\{(s_1,s_2)\in\mathcal N_i:\|s_1-s_2\|=h\right\}$, and $n_{ih}=\left|\mathcal N_{ih}\right|$.
\begin{equation}
\widehat {Cov}_{ih}=\frac{\sum_{(s_1,s_2)\in\mathcal N_{ih}}Y_{(s_1,s_2)}}{2n_{ih}}-\left(\frac{\sum_{s_1\in\mathcal N_{ih}}X_{s_1}}{n_{ih}}\right)
\left(\frac{\sum_{s_2\in\mathcal N_{ih}}X_{s_2}}{n_{ih}}\right).
\label{eq:cov1}
\end{equation}
Noting that $Y_{(s_1,s_2)}$, where $(s_1,s_2)\in\mathcal N_i$, is stationary, it follows by the ergodic theorem that
\begin{equation}
\widehat {Cov}_{ih}\stackrel{a.s.}{\longrightarrow}{Cov}_{ih}=Cov\left(X_{s_1},X_{s_2}\right)~\mbox{where}~\|s_1-s_2\|=h.
\label{eq:cov2}
\end{equation}
Let
\begin{equation}
\widetilde {Cov}_{Kh}=\sum_{i=1}^K\tilde p_{iKh}\widehat {Cov}_{ih},
\label{eq:cov3}
\end{equation}
where $\tilde p_{iKh}=n_{ih}/\sum_{j=1}^Kn_{jh}$, with $\sum_{i=1}^{\infty}p_{ih}=1$,
and
\begin{equation}
{Cov}_{\infty,h}=\sum_{i=1}^\infty\tilde p_{ih}{Cov}_{ih},
\label{eq:cov4}
\end{equation}

We assume that
\begin{equation}
\tilde p_{iKh}\rightarrow p_{iKh}=\frac{p_{ih}}{\sum_{j=1}^Kp_{jh}},~\mbox{as}~n_{ih}\rightarrow\infty;~i=1,\ldots,K.
\label{eq:cov5}
\end{equation}

\begin{theorem}
\label{theorem:cov1}
Let 
\begin{equation}
\sum_{i=1}^{\infty}p_{ih}\left|Cov_{ih}\right|<\infty.
\label{eq:cov6}
\end{equation}
Then
\begin{equation}
\underset{K\rightarrow\infty}{\lim}\underset{n_{ih}\rightarrow\infty;i=1,\ldots,K}{\lim}~\left|\widetilde {Cov}_{Kh}-{Cov}_{\infty,h}\right|=0.
\label{eq:cov7}
\end{equation}
\end{theorem}
\begin{proof}
\begin{equation}
\left|\widetilde {Cov}_{Kh}-{Cov}_{\infty,h}\right|
\leq\sum_{i=1}^K\left|\widehat{Cov}_{ih}\right|\left|\tilde p_{iKh}-p_{ih}\right|+\sum_{i=1}^Kp_{ih}\left|\widehat{Cov}_{ih}-Cov_{ih}\right|
+\sum_{K+1}^{\infty}p_i\left|Cov_{ih}\right|.
\label{eq:cov8}
\end{equation}
Due to (\ref{eq:cov5}), $\sum_{i=1}^K\left|\widehat{Cov}_{ih}\right|\left|\tilde p_{iKh}-p_{ih}\right|\rightarrow
\sum_{i=1}^K\left|{Cov}_{ih}\right|\left|p_{iKh}-p_{ih}\right|$ as $n_{ih}\rightarrow\infty$; $i=1,\ldots,K$.
Due to (\ref{eq:cov6}), $\left|Cov_{ih}\right|<L$, for some $L>0$, for all $i\geq 1$. Hence, the first term on the right hand side of (\ref{eq:cov8}) is
bounded above by $L\sum_{j=K+1}^{\infty}p_i$, which tends to zero, as $K\rightarrow\infty$, since $\sum_{i=1}^{\infty}p_i=1$.

Using (\ref{eq:cov2}), it is seen that the second term of the right hand side of (\ref{eq:cov8}) also tends to zero as $n_{ih}\rightarrow\infty$; $i=1,\ldots,K$,
satisfying (\ref{eq:cov5}) and as $K\rightarrow\infty$.

The last term on the right hand side of (\ref{eq:cov8}) tends to zero as $K\rightarrow\infty$ due to (\ref{eq:cov6}).
\end{proof}

Note that the covariance structure of $\bX$ is stationary if any only if $Cov_{ih}=Cov_{\infty,h}$ for all $i\geq 1$ and all $h>0$, and is nonstationary if and only if
$Cov_{ih}\neq Cov_{\infty,h}$ for all $i\geq 1$ for some $h>0$.
\begin{theorem}
\label{theorem:cov2}
The covariance structure of $\bX$ is stationary if and only if for $i\geq 1$, for all $h>0$,
\begin{equation*}
\underset{K\rightarrow\infty}{\lim}\underset{n_{jh}\rightarrow\infty;j=1,\ldots,K}{\lim}~\left|\widehat{Cov}_{ih}-\widetilde {Cov}_{Kh}\right|=0.
\end{equation*}
\end{theorem}
\begin{proof}
Using Theorem \ref{theorem:cov1}, the proof follows in the same way as the proof of Theorem \ref{theorem:gc2}, with the probabilities replaced with the respective covariances.


\end{proof}

\begin{theorem}
\label{theorem:cov3}
The covariance structure of $\bX$ is nonstationary if and only if for $i\geq 1$, for some $h>0$,
\begin{equation*}
\underset{K\rightarrow\infty}{\lim}\underset{n_{jh}\rightarrow\infty;j=1,\ldots,K}{\lim}~\left|\widehat{Cov}_{ih}-\widetilde {Cov}_{Kh}\right|>0.
\end{equation*}
\end{theorem}
\begin{proof}
Using Theorem \ref{theorem:cov1}, the proof follows in the same way as the proof of Theorem \ref{theorem:gc3}, with the probabilities replaced with the respective covariances.
\end{proof}

Now define $Y_{j,n_{jh}}=\mathbb I\left\{\left|\widehat{Cov}_{ih}-\widetilde {Cov}_{Kh}\right|<c_{jh}\right\}$.
Then the following characterization theorems hold, the proofs of which are the similar to those of Theorems \ref{theorem:convergence} and \ref{theorem:divergence}.
\begin{theorem}
\label{theorem:convergence2}
For all $\omega\in\mathfrak S\cap\mathfrak N^c$, where $\mathfrak N$ is some null set having probability measure zero,
$\bX$ is stationary if and only if 
for any $h>0$, there exists a monotonically decreasing sequence 
$\left\{c_{jh}(\omega)\right\}_{j=1}^{\infty}$ such that
\begin{equation}
\pi\left(\mathcal N_1|y_{k,n_{kh}}(\omega)\right)\rightarrow 1,
\label{eq:consistency_at_1_h}
\end{equation}
as $k\rightarrow\infty$ and $n_{jh}\rightarrow\infty$ for $j=1,\ldots,K$ satisfying (\ref{eq:eq4}) and $K\rightarrow\infty$, 
where $\mathcal N_1$ is any neighborhood of 1 (one).
\end{theorem}

\begin{theorem}
\label{theorem:divergence2}
$\bX$ is nonstationary if and only if for some $h>0$, and 
for any $\omega\in\mathfrak S\cap\mathfrak N^c$ where $\mathfrak N$ is some null set having probability measure zero, 
for any choice of the non-negative, monotonically decreasing sequence $\left\{c_{jh}(\omega)\right\}_{j=1}^{\infty}$,
\begin{equation}
\pi\left(\mathcal N_0|y_{k,n_{kh}(\omega)}(\omega)\right)\rightarrow 1,
\label{eq:consistency_at_0_h}
\end{equation}
as $k\rightarrow\infty$ and $n_{jh}\rightarrow\infty$, $j=1,\ldots,K$ satisfying (\ref{eq:eq4}), and $K\rightarrow\infty$, 
where $\mathcal N_0$ is any neighborhood of 0 (zero).
\end{theorem}

\section{Characterization of stationarity and nonstationarity using non-recursive Bayesian posteriors}
\label{sec:non_recursive}
Observe that it is not strictly necessary for the prior at any stage to depend upon the previous stage.
Indeed, we may simply assume that $\pi\left(p_{j,n_j}\right)\equiv Beta\left(\alpha_j,\beta_j\right)$,
for $j=1,2,\ldots$. In this case, the posterior of $p_{k,n_k}$ given $y_{k,n_k}$ is simply
$Beta\left(\alpha_k+y_{k,n_k},1+\beta_k-y_{k,n_k}\right)$. The posterior mean and variance are then given by
\begin{align}
E\left(p_{k,n_k}|y_{k,n_k}(\omega)\right)&=\frac{\alpha_k+y_{k,n_k}(\omega)}
{1+\alpha_k+\beta_k};
\label{eq:postmean_p_k_3}\\
Var\left(p_{k,n_k}|y_{k,n_k}(\omega)\right)&=
\frac{(\alpha_k+y_{k,n_k}(\omega))(1+\beta_k-y_{k,n_k}(\omega))}
{(1+\alpha_k+\beta_k)^2(2+\alpha_k+\beta_k)}.
\label{eq:postvar_p_k_3}
\end{align}
Since $y_{k,n_k}(\omega)$ (or $y_{k,n_{kh}}(\omega)$) converges to $1$ or $0$ as $n_k\rightarrow\infty$, accordingly as
$\bX$ is stationary or nonstationary (or the covariance structure of $\bX$ is stationary or nonstationary), it is easily seen, provided that $\alpha_k\rightarrow 0$
and $\beta_k\rightarrow 0$ as $k\rightarrow\infty$, that (\ref{eq:postmean_p_k_3}) converges to $1$ (respectively, $0$) 
if and only if $\bX$ is (covariance) stationary (respectively, (covariance) nonstationary). 
Importantly, if we choose $\alpha_k=\beta_k=0$ for all $k\geq 1$, then $k\rightarrow\infty$ is no longer needed, and the results continue to hold
if $n_k\rightarrow\infty$.

Thus, characterization of stationarity or nonstationarity of $\bX$ is possible even with the non-recursive approach.
Indeed, note that the prior parameters $\alpha_k$ and $\beta_k$ are more flexible compared to those
associated with the recursive approach. This is because, in the non-recursive approach 
we only require $\alpha_k\rightarrow 0$ and $\beta_k\rightarrow 0$ as 
$k\rightarrow\infty$, so that convergence of the series $\sum_{j=1}^{\infty}\alpha_j$ and $\sum_{j=1}^{\infty}\beta_j$
are not necessary, unlike the recursive approach. However, choosing $\alpha_k$ and $\beta_k$  to be of sufficiently
small order ensures much faster convergence of the posterior mean and variance as compared to the recursive approach.

Unfortunately, an important drawback of the non-recursive approach is that it does not admit extension
to the case of general oscillatory stochastic processes. On the other hand, as we show subsequently, the principles of our recursive theory   
can be easily adopted to develop a Bayesian theory for determining (multiple) frequencies of oscillating stochastic processes. 
In other words, the recursive
approach seems to be more powerful from the perspective of development of a general Bayesian principle for learning about the basic characteristics
of the underlying stochastic process. 
Moreover, as our examples demonstrate, the recursive posteriors converge 
sufficiently fast to the correct degenerate distributions, obviating the need to consider the non-recursive approach.
Consequently, we do not further pursue the non-recursive approach in this article but reserve the topic for further
investigation in the future.


\section{First illustration: AR(1) model}
\label{sec:ar1}

Let us consider the following AR(1) model: $X_t=\rho X_{t-1}+\epsilon_t$; $t\geq 1$, where $\epsilon_t\stackrel{iid}{\sim}N(0,1)$, and
$X_0\sim U(-1,1)$, the uniform distribution on $(-1,1)$. It is well-known that $\left\{X_t:t\geq 1\right\}$ is (asymptotically) stationary if and only if
$|\rho|<1$. We illustrate the performance of our methodology after generating the data from the above AR(1) model for
various values of $\rho$, which we pretend to be unknown for illustration. In particular, we consider three different setups in this regard.
In the first setup, we consider samples of sizes $2\times 10^8$ from from the AR(1) model, and assume that the form of the true model is known, and that only
$\rho$ is unknown. In the second setup, we generate samples of sizes $2500$ from from the AR(1) model, and assume as before that only $\rho$ is unknown.
In the last setup, we draw samples of sizes $2500$ from from the AR(1) model, and assume that the entire data-generating model is unknown.

\subsection{Case 1: Large sample size, form of the model known}
\label{subsec:ar1_case1}

\subsubsection{Sample size}
We draw samples of sizes $2\times 10^8$ from the AR(1) model for various values of $\rho$ and evaluate the performance of our Bayesian methodology, setting
$n=10^4$ and $K=2\times 10^4$.

\subsubsection{Construction of bound}
An important ingredient of our proposed method is the construction of the bounds $c_j(\omega)$. In this case, we construct the bounds as follows.
We first draw a sample of size $2\times 10^8$ from the AR(1) model with $\rho=0.99999$. With this sample, for $j=1,\ldots,K$, we form the sup norms
$\tilde c_j=\underset{-\infty<x<\infty}{\sup}~|\hat F_j(x)-\tilde F_K(x)|$ according to Lemma \ref{lemma:lemma_tv} and Remark \ref{remark:remark_tv}.
We then set $c_j$ as
\begin{equation}
c_j=\tilde c_j+10^6\times\left(0.99999-|\hat\rho|\right)/\log(\log(j+1)),
\label{eq:ar1_bound1}
\end{equation}	
where $\hat\rho$ is the maximum likelihood estimator (MLE) of $\rho$ based on the observed sample. If the MLE of $\rho$ does not exist, we set $\hat\rho\equiv 1$.

To explain the strategy behind (\ref{eq:ar1_bound1}), note that for $\rho=0.99999$, the AR(1) process, although stationary, is very close to nonstationarity.
So, for any value of $\rho$ such that $|\rho|<0.99999$, $\tilde c_j$ is expected to be larger than $c_j$. Hence, in such cases, stationarity is to be expected.
On the other hand, if $|\rho|\geq 1$, $\tilde c_j$ is expected to be smaller than $c_j$, so that nonstationarity is implied. For simplicity we assume that
values of $\rho$ such that $0.99999<|\rho|<1$ are not of interest.

To further improve the bound, we add the quantity
$10^6\times\left(0.99999-|\hat\rho|\right)/\log(\log(j+1))$ to $\tilde c_j$. The significance of this addition is as follows. If $|\hat\rho|<0.99999$,
this quantity is positive but tends to zero at a slow rate. This enhances the conclusion of stationarity. Similarly, if $|\hat\rho|>0.99999$, the quantity
is negative and tends to zero slowly, favouring nonstationarity. Multiplication with $10^6$ inflates the quantity for more prominence.

\subsubsection{Implementation}

Note that at each stage $j$, we need to compute the sup norm given by Lemma \ref{lemma:lemma_tv} (also, Remark \ref{remark:remark_tv}). This requires
evaluation of $\tilde F_K$ at $\hat x_j$ (or $\tilde x_{j^*}$). We carry out this evaluations by splitting the summations of the indicator functions
associated with $\tilde F_K$ on $104$ parallel cores on a VMWare, and obtaining the final result on a single node, which also carries out the iterative procedure.
The entire exercise takes about 6 minutes in the case of stationarity and about 3 minutes in the case of nonstationarity.

\subsubsection{Results}
We implement our method when the data is generated from the AR(1) model with $\rho$ randomly selected from $U(-1,1)$, and with $\rho$
taking the values $0.99$, $0.995$, $0.999$, $0.9999$, $1$, $1.00005$, $1.05$ and $2$. Figure \ref{fig:example1} shows that in all the cases, 
our method correctly detects stationarity and nonstationarity.
That even with such subtle differences among the true values of $\rho$ our method performs so well, is quite encouraging.

\begin{figure}
\centering
\subfigure [Stationary: $|\rho|<1$.]{ \label{fig:abs_rho_less_1}
\includegraphics[width=4.5cm,height=4.5cm]{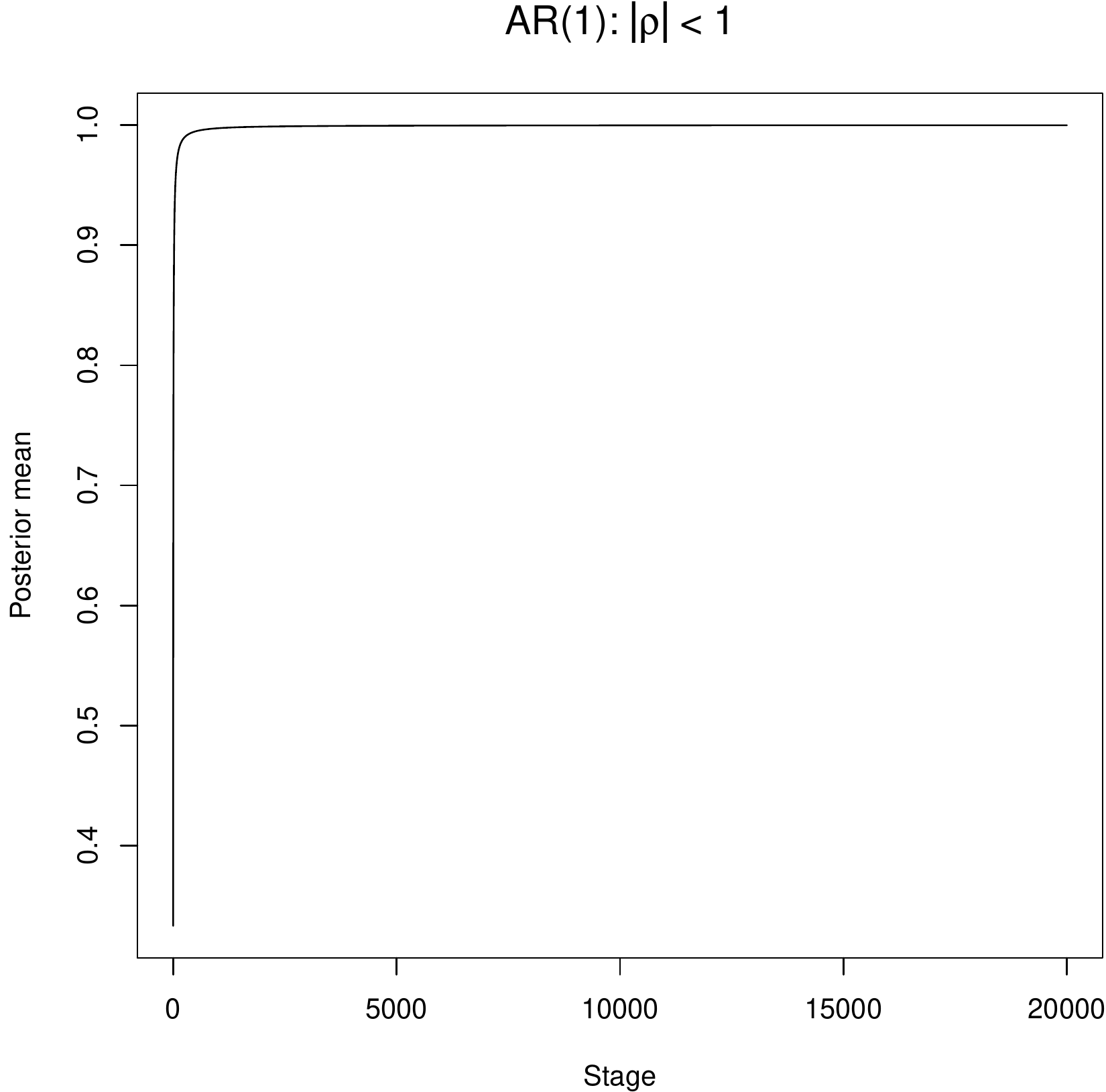}}
\hspace{2mm}
\subfigure [Stationary: $\rho=0.99$.]{ \label{fig:rho_99}
\includegraphics[width=4.5cm,height=4.5cm]{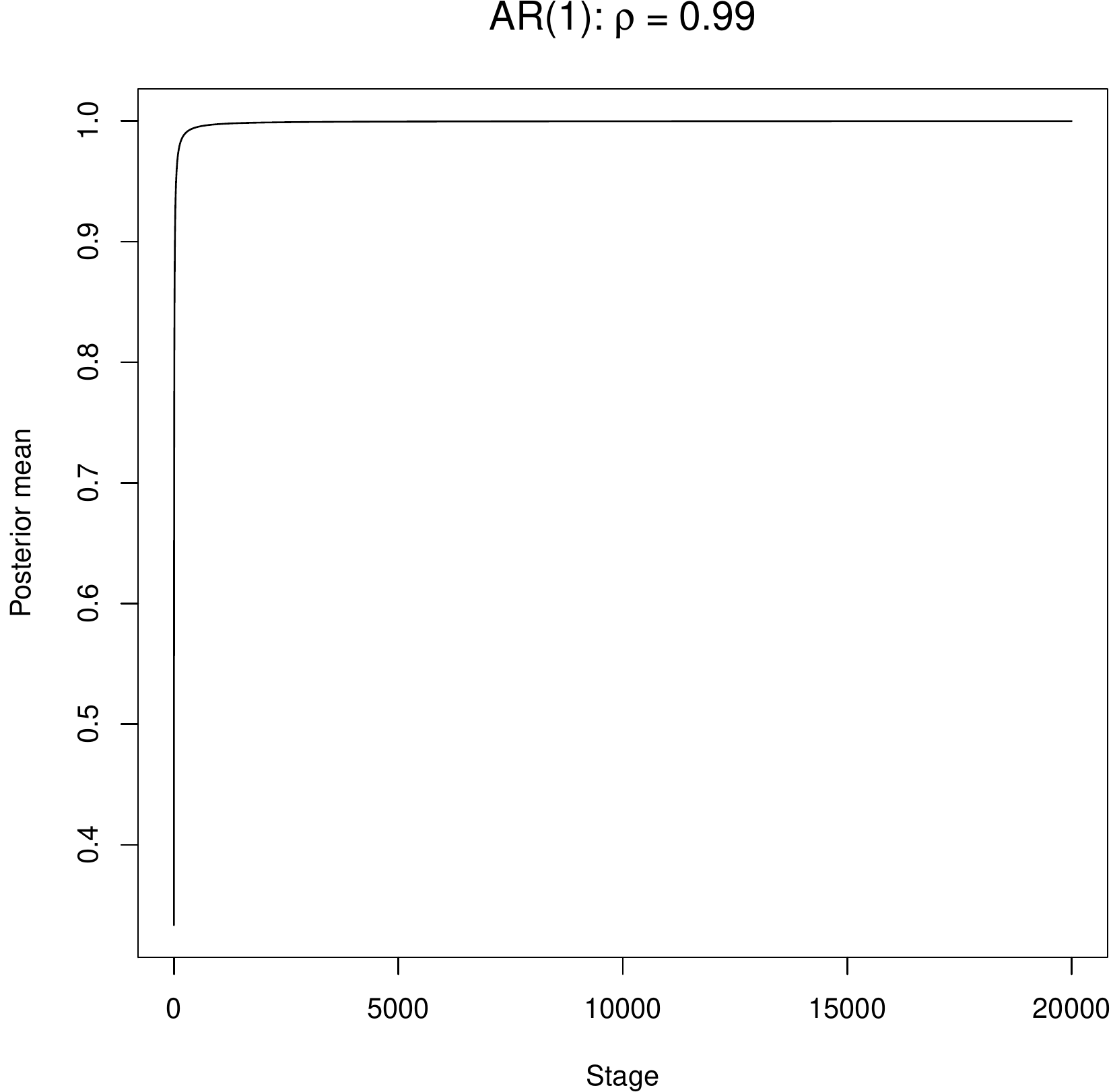}}
\hspace{2mm}
\subfigure [Stationary: $\rho=0.995$.]{ \label{fig:rho_995}
\includegraphics[width=4.5cm,height=4.5cm]{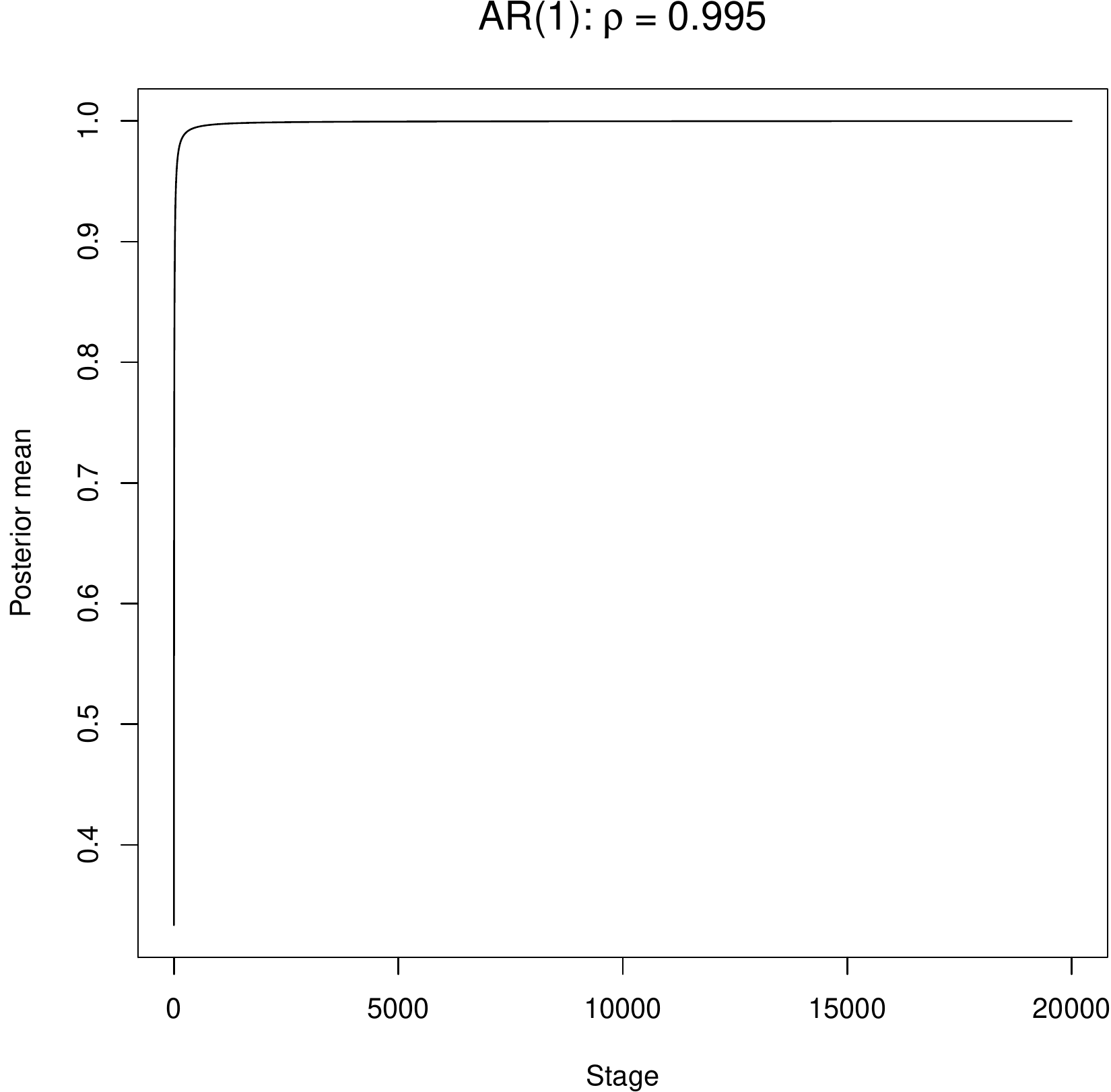}}\\
\vspace{2mm}
\subfigure [Stationary: $\rho=0.999$.]{ \label{fig:rho_999}
\includegraphics[width=4.5cm,height=4.5cm]{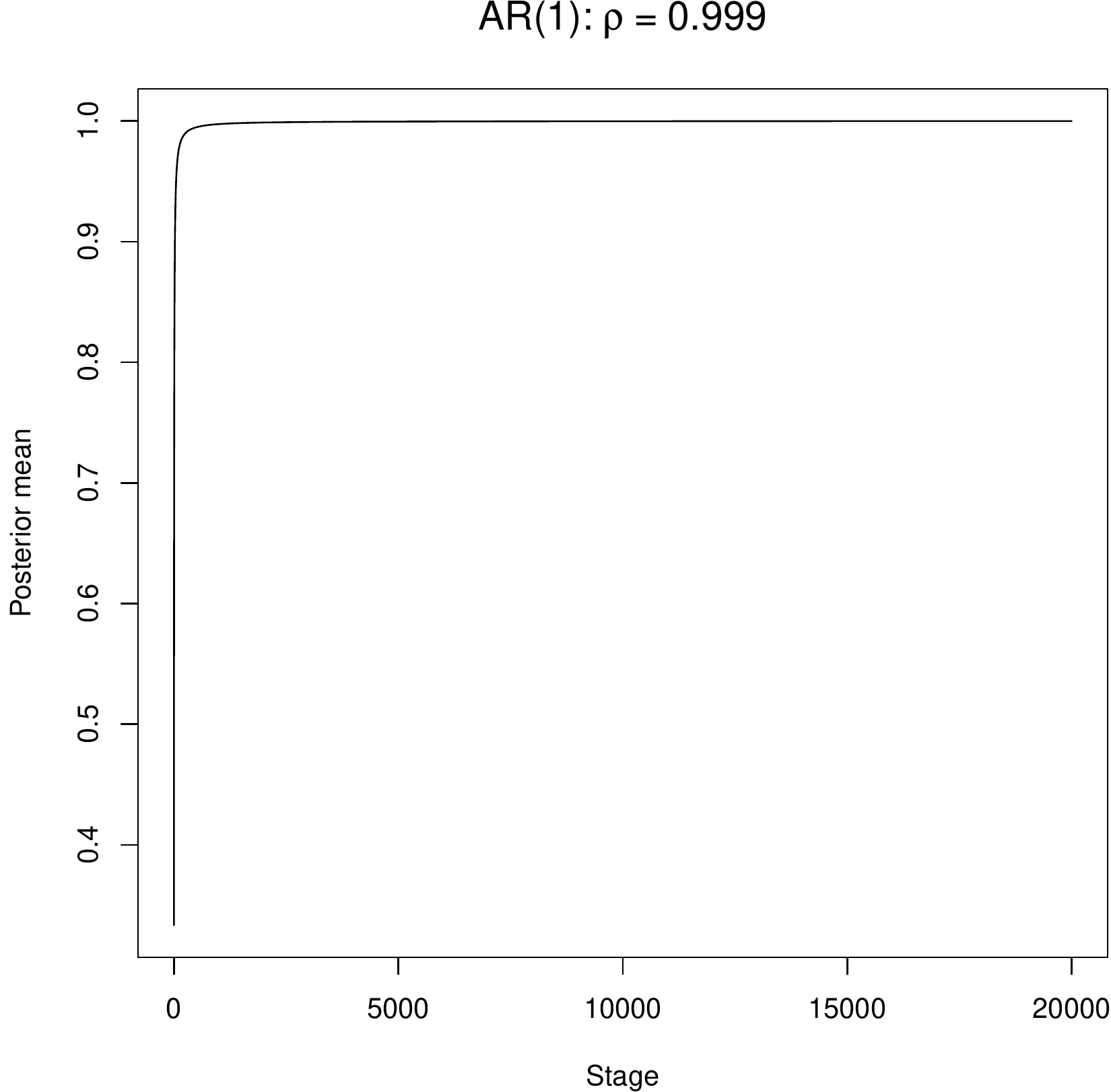}}
\hspace{2mm}
\subfigure [Nonstationary: $\rho=0.9999$.]{ \label{fig:rho_9999}
\includegraphics[width=4.5cm,height=4.5cm]{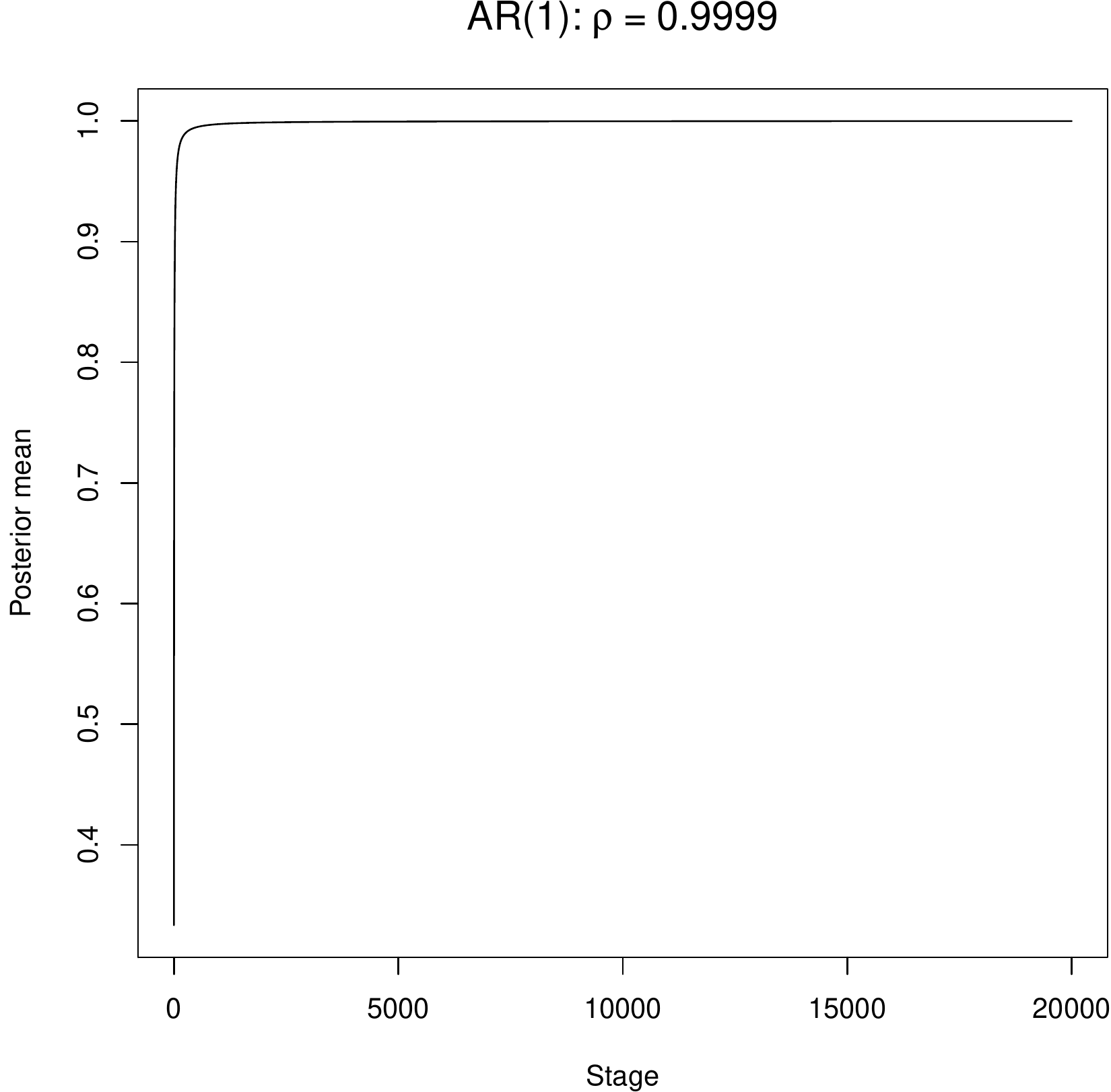}}
\hspace{2mm}
\subfigure [Nonstationary: $\rho=1$.]{ \label{fig:rho_1}
\includegraphics[width=4.5cm,height=4.5cm]{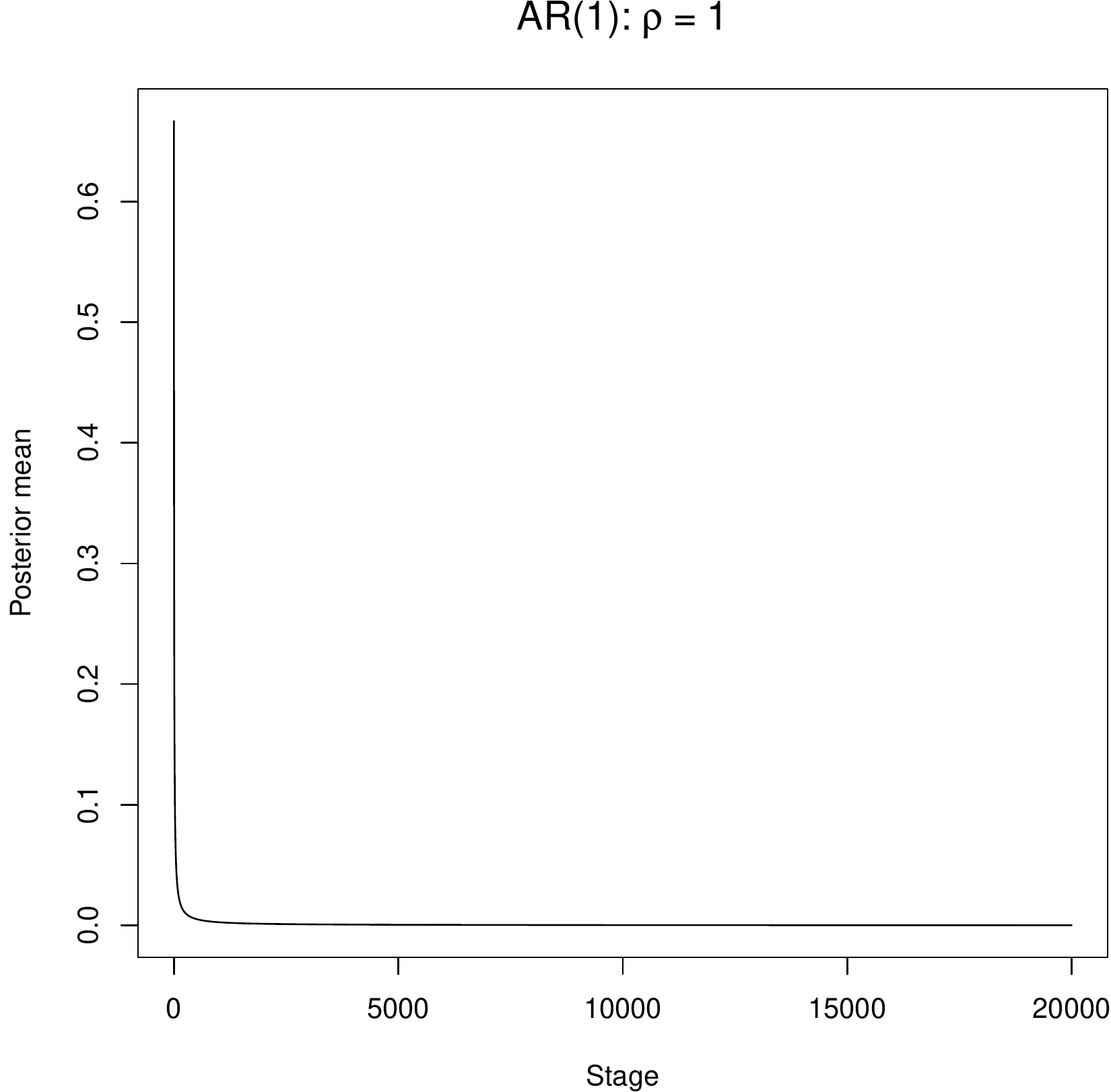}}\\
\vspace{2mm}
\subfigure [Nonstationary: $\rho=1.00005$.]{ \label{fig:rho_100005}
\includegraphics[width=4.5cm,height=4.5cm]{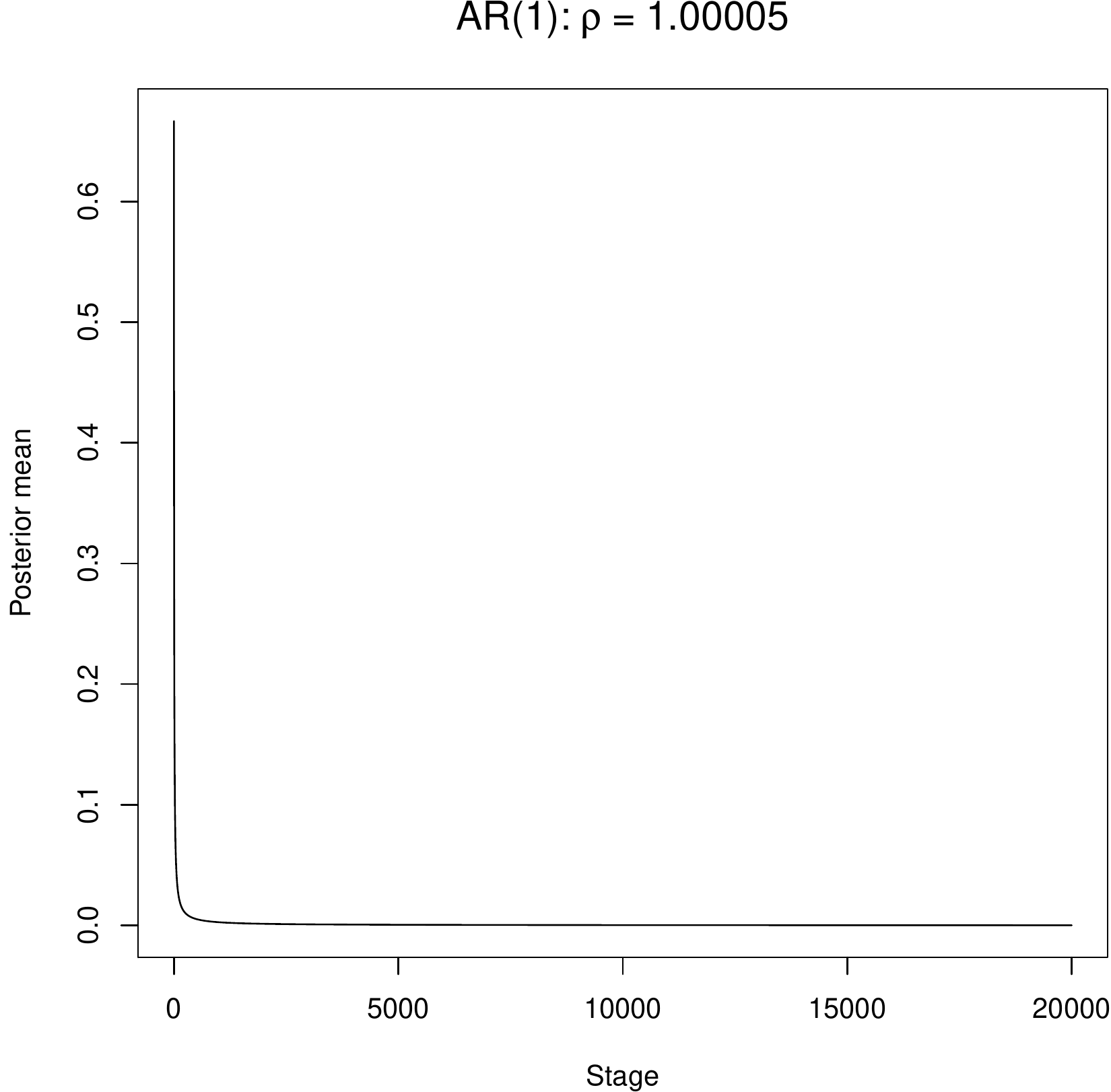}}
\hspace{2mm}
\subfigure [Nonstationary: $\rho=1.05$.]{ \label{fig:rho_105}
\includegraphics[width=4.5cm,height=4.5cm]{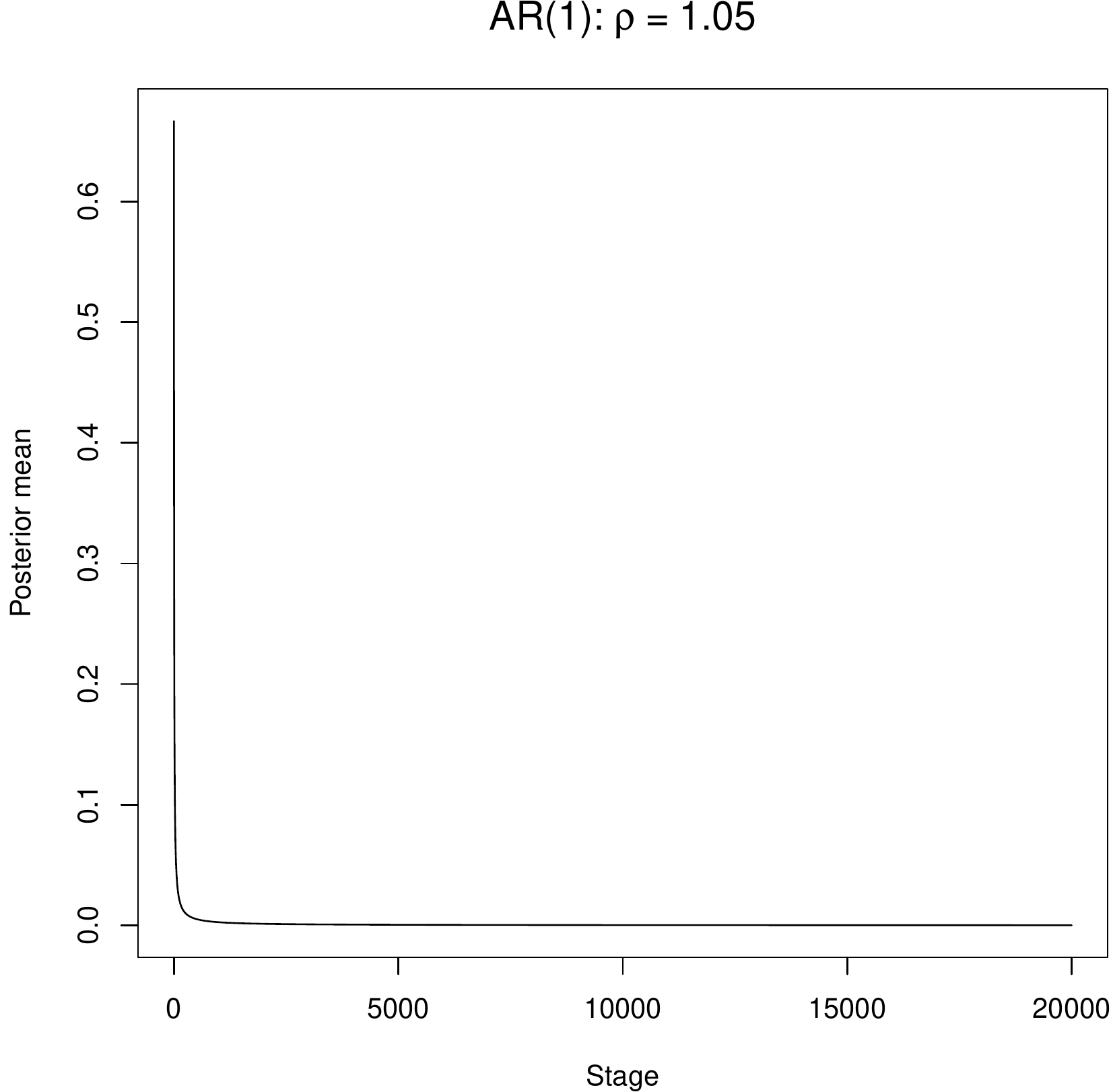}}
\hspace{2mm}
\subfigure [Nonstationary: $\rho=2$.]{ \label{fig:rho_2}
\includegraphics[width=4.5cm,height=4.5cm]{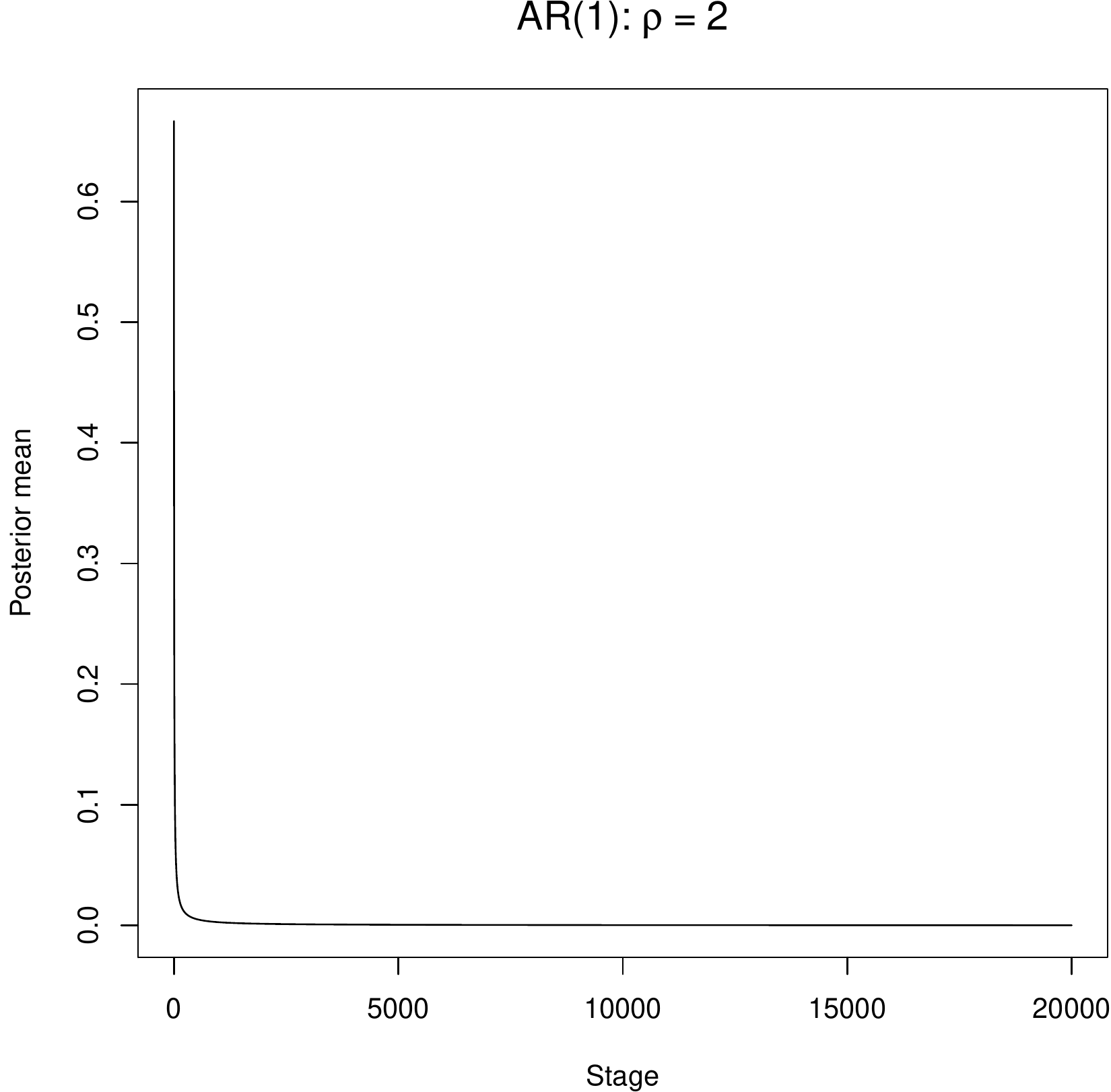}}
\caption{Parametric AR(1) example with $K=20000$ and $n=10000$.}
\label{fig:example1}
\end{figure}

\subsection{Case 2: Relatively small sample size, form of the model known}
\label{subsec:ar1_case2}

\subsubsection{Sample size}
We draw samples of sizes $2500$ from the AR(1) model for those values of $\rho$ as in Section \ref{subsec:ar1_case1} 
and evaluate the performance of our Bayesian methodology, setting $n=50$ and $K=50$.

\subsubsection{Construction of bound}
In this case, we choose the basic form of the bounds in a similar manner as in Section \ref{subsec:ar1_case1}, but make it adaptive with the iterations
to suit the small sample situation.

As before, we first draw a sample of size $2\times 10^8$ from the AR(1) model with $\rho=0.99999$. With this sample, for $j=1,\ldots,K$, we form the sup norms
$\tilde c_j=\underset{-\infty<x<\infty}{\sup}~|\hat F_j(x)-\tilde F_K(x)|$ according to Lemma \ref{lemma:lemma_tv} and Remark \ref{remark:remark_tv}.
We then set $c_j$ as
\begin{equation}
c_j=\tilde c_j+\hat C_j\times\left(0.99999-|\hat\rho|+\hat\epsilon_j\right)/\log(\log(j+1)),
\label{eq:ar1_bound2}
\end{equation}	
where $\hat C_1=1$, $\hat\epsilon_1=0$, and for $j>1$, we adaptively modify these values as follows:
\begin{itemize}
\item If $|\hat\rho|>0.9985$, 
	\begin{enumerate}
	 \item  If $y_j=1$, then $\hat\epsilon_{j+1}=\hat\epsilon_j+0.001$ and $\hat C_{j+1}=\hat C_j+1$.
	 \item  If $y_j=0$, then $\hat\epsilon_{j+1}=\hat\epsilon_j-0.001$ and $\hat C_{j+1}=\hat C_j+1$.
	\end{enumerate}
\item If $0.9955<|\hat\rho|\leq 0.9985$, 
	\begin{enumerate}
	\item $y_j=1$, then $\hat\epsilon_{j+1}=\hat\epsilon_j+0.01$ and $\hat C_{j+1}=\hat C_j+1$.
	\item $y_j=0$, then $\hat\epsilon_{j+1}=\hat\epsilon_j-0.01$ and $\hat C_{j+1}=\hat C_j+1$.
	\end{enumerate}
\item If $0<|\hat\rho|\leq 0.9955$, 
	\begin{enumerate}
         \item If $y_j=1$, then $\hat\epsilon_{j+1}=\hat\epsilon_j+0.05$ and $\hat C_{j+1}=\hat C_j+1$.
         \item If $y_j=0$, then $\hat\epsilon_{j+1}=\hat\epsilon_j-0.05$ and $\hat C_{j+1}=\hat C_j+1$.
	\end{enumerate}
\end{itemize}
To appreciate the above strategy, first note that for small samples, the MLE of $\rho$ need not be adequately close to the true value of $\rho$, and hence
we need to add a quantity $\hat\epsilon_j$ to make up for the inadequacy. We select $\hat\epsilon_j$ adaptively, increasing its value for the next iteration
if $y_j=1$, so that in the next iteration stationarity is preferred, given the current value of $y_j$. If $y_j=0$ in the current iteration, we decrease the current
value of $\hat\epsilon_j$, so that nonstationarity is favoured in the next iteration. We also increase the value of $\hat C_j$ by one, at every iteration, rather
than keeping it constant over the iterations. Thus, the prominence of the quantity $\hat C_j\times\left(0.99999-|\hat\rho|+\hat\epsilon_j\right)/\log(\log(j+1))$
increases with the iterations. 

The increment and decrement of $\hat\epsilon_j$ depends upon the magnitude of $\hat\rho$. If $|\hat\rho|>0.9985$, that is, when the model is close to
nonstationarity, we increase/decrease $\hat\epsilon_j$ by $0.001$ only, since larger quantities, if added, can wrongly indicate stationarity.

When $0.9955<|\hat\rho|\leq 0.9985$, we consider adding/subtracting $0.01$ to $\hat\epsilon_j$; this larger quantity is expected to make up for the uncertainty
associated with stationarity and nonstationarity when $0.9955<|\hat\rho|\leq 0.9985$.

On the other hand, when $0<|\hat\rho|\leq 0.9955$, we add/subtract $0.05$ to $\hat\epsilon_j$, since we expect our algorithm to favour stationarity in this situation.
The choice $0.05$, which is larger than the quantities in the previous cases, is expected to facilitate diagnosis of stationarity.

\subsubsection{Implementation}

The implementation remains the same as before. For this small sample, even with 2 cores, the results are delivered almost instantly.

\subsubsection{Results}
As before, we implement our method when the data is generated from the AR(1) model with $\rho$ randomly selected from $U(-1,1)$, and with $\rho$
taking the values $0.99$, $0.995$, $0.999$, $0.9999$, $1$, $1.00005$, $1.05$ and $2$. Figure \ref{fig:example2} shows that, except
in the case where the true value of $\rho$ is $0.9999$, our method correctly detects stationarity and nonstationarity.
That even with such small sample, and with such subtle differences among the true values of $\rho$, our method performs well, is quite encouraging, despite
its fallibility at $\rho=0.9999$. Indeed, with such small sample, correct detection of stationarity in the case of so subtle difference with nonstationarity is perhaps
not to be expected.
\begin{figure}
\centering
\subfigure [Stationary: $|\rho|<1$.]{ \label{fig:abs_rho_less_1_short}
\includegraphics[width=4.5cm,height=4.5cm]{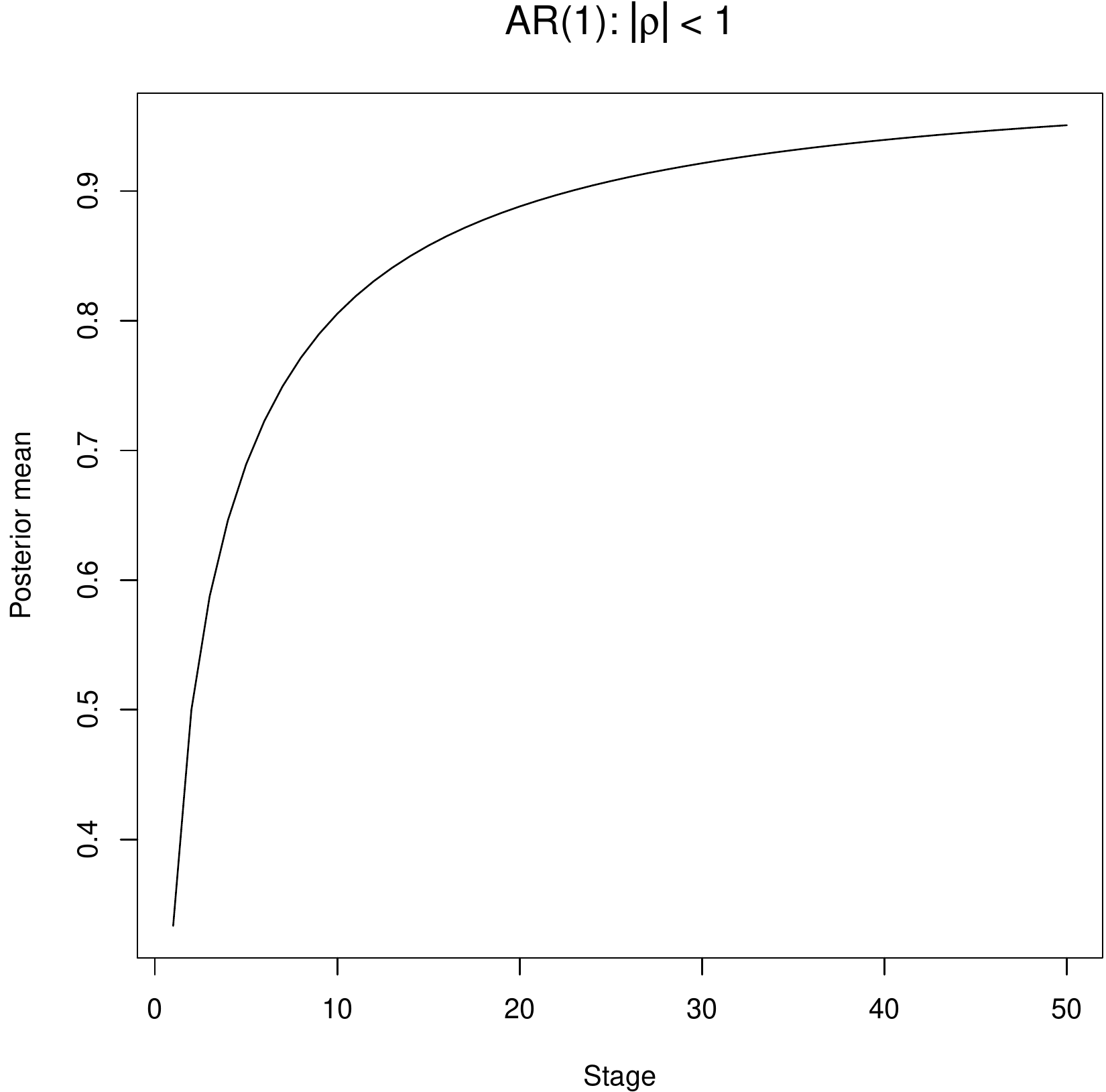}}
\hspace{2mm}
\subfigure [Stationary: $\rho=0.99$.]{ \label{fig:rho_99_short}
\includegraphics[width=4.5cm,height=4.5cm]{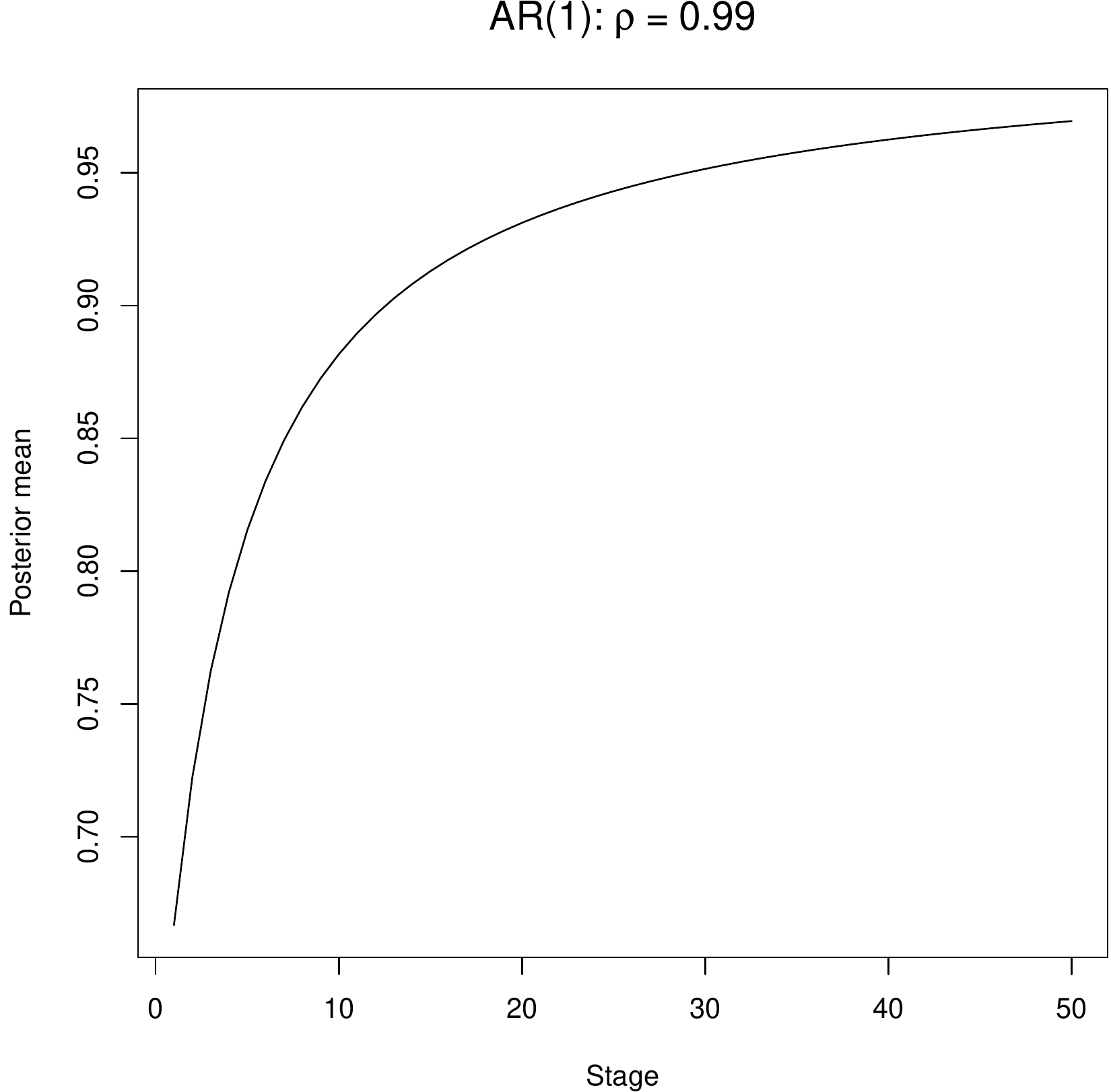}}
\hspace{2mm}
\subfigure [Stationary: $\rho=0.995$.]{ \label{fig:rho_995_short}
\includegraphics[width=4.5cm,height=4.5cm]{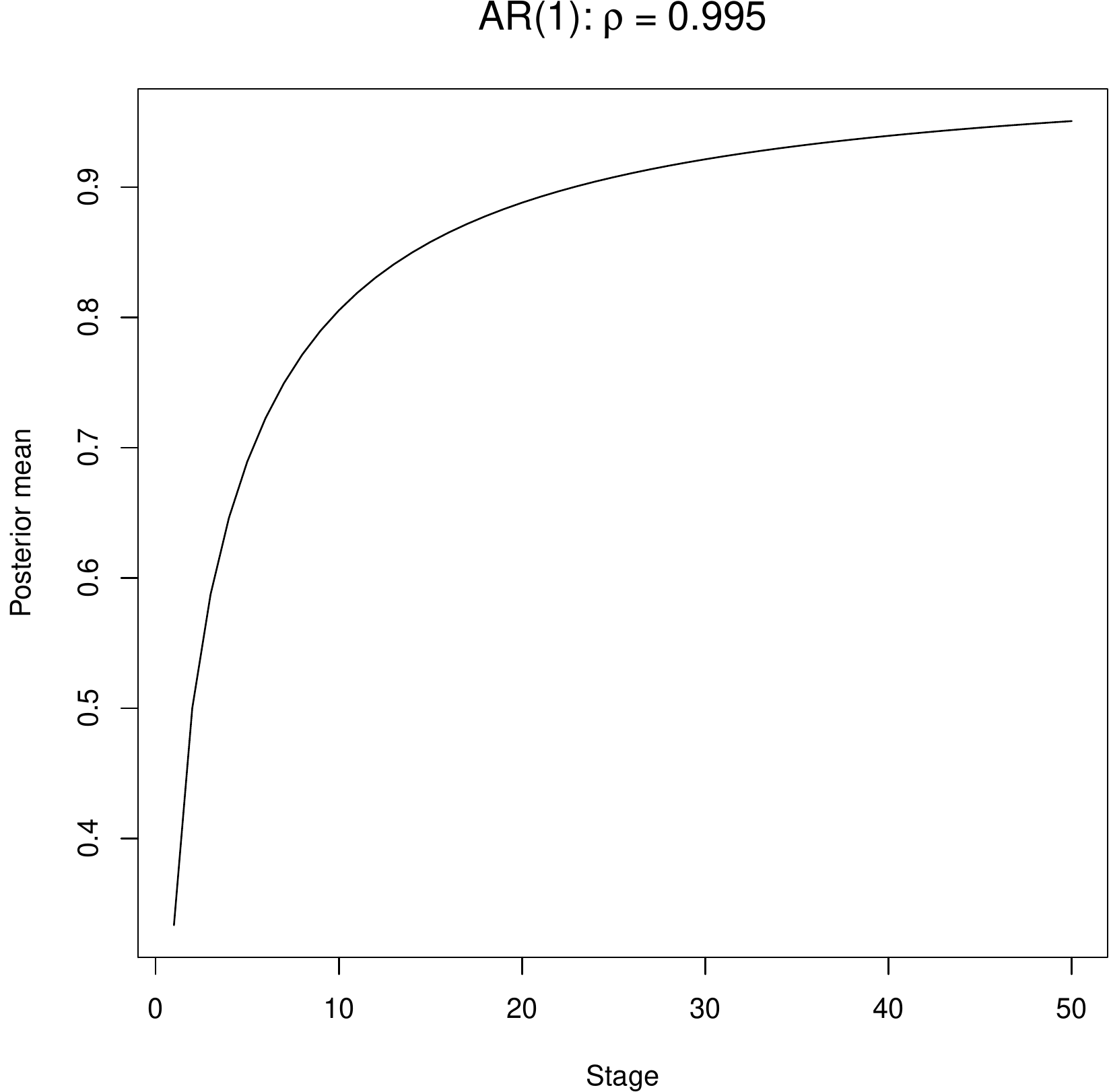}}\\
\vspace{2mm}
\subfigure [Stationary: $\rho=0.999$.]{ \label{fig:rho_999_short}
\includegraphics[width=4.5cm,height=4.5cm]{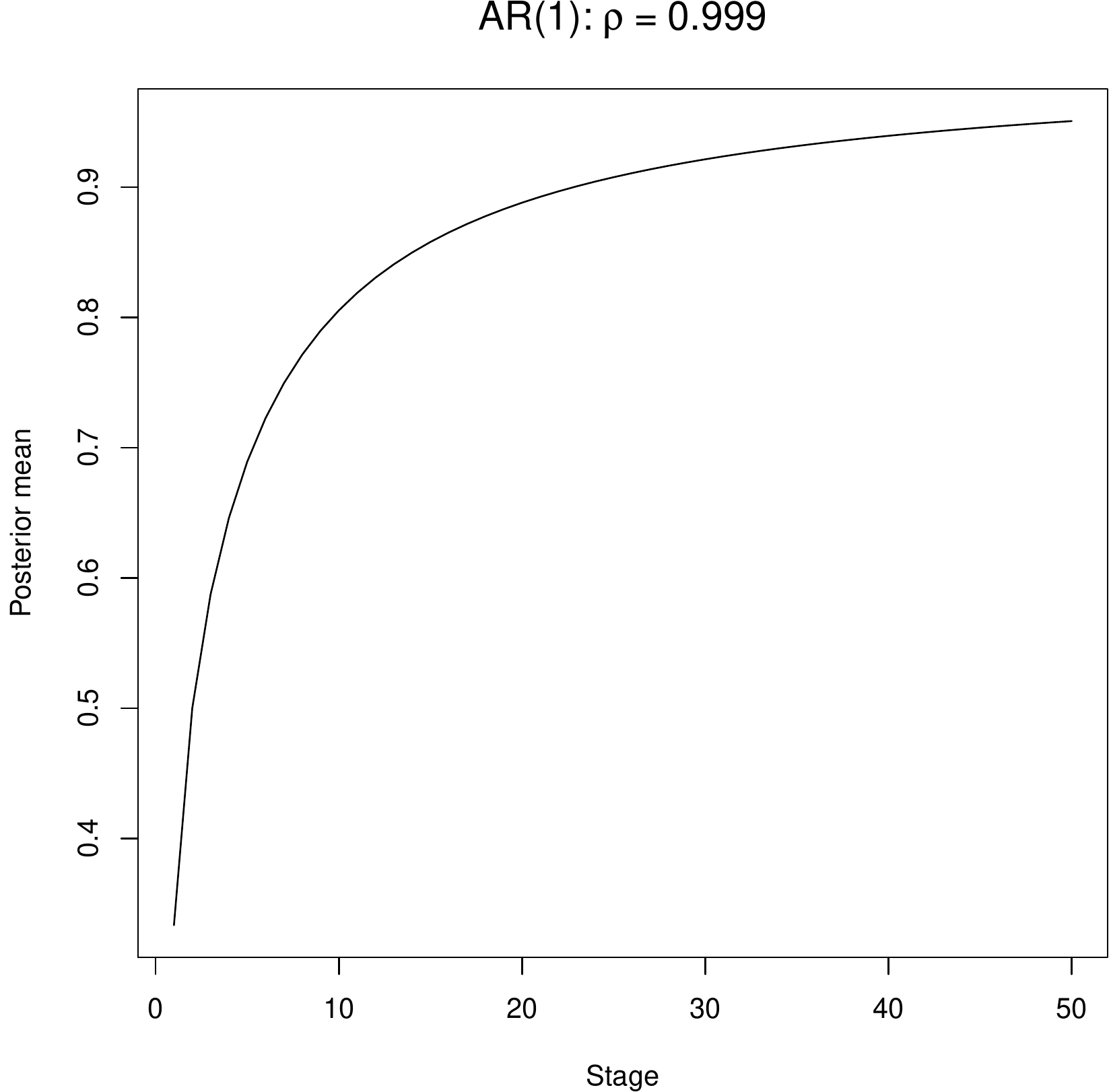}}
\hspace{2mm}
\subfigure [Nonstationary: $\rho=0.9999$.]{ \label{fig:rho_9999_short}
\includegraphics[width=4.5cm,height=4.5cm]{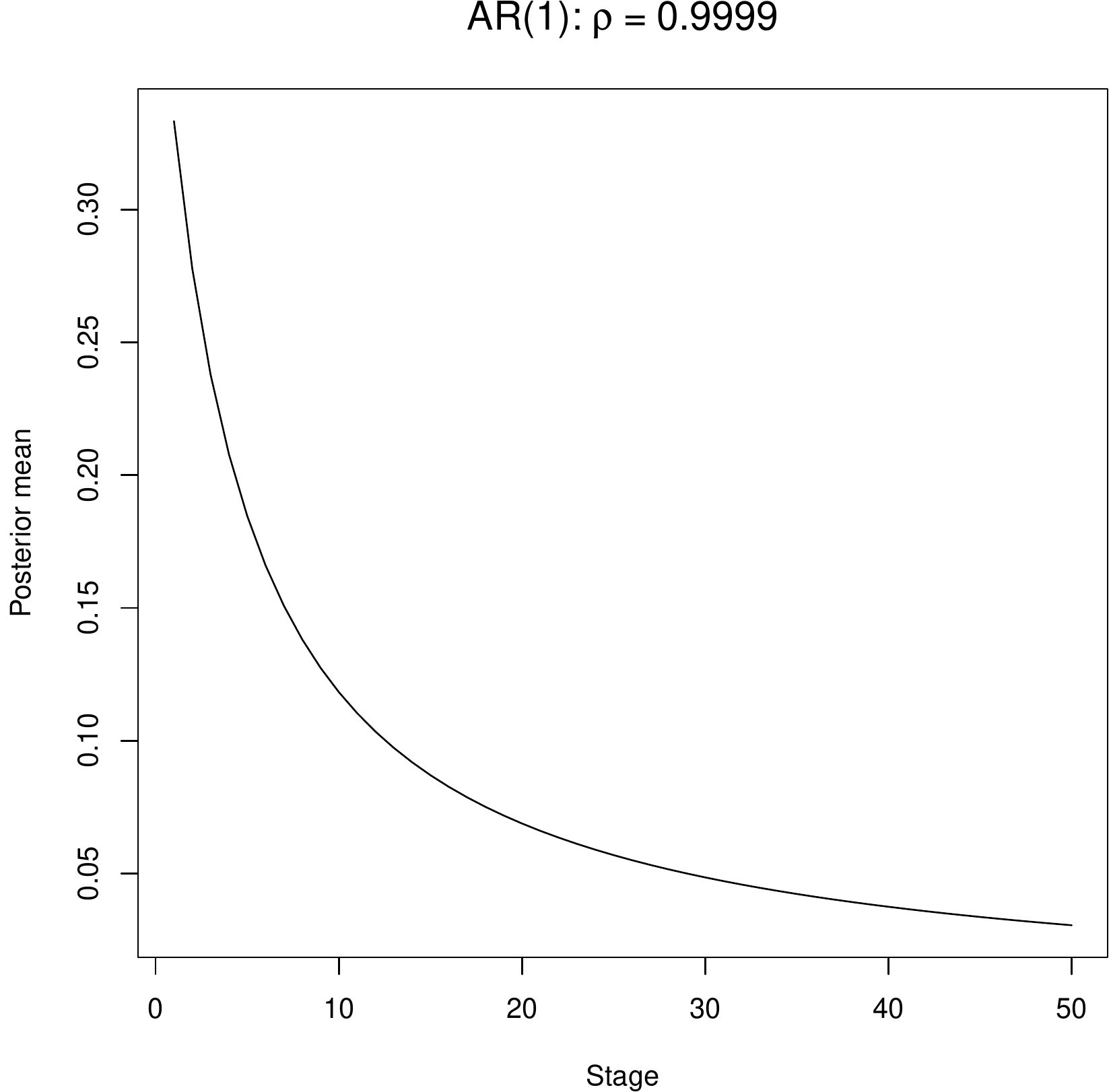}}
\hspace{2mm}
\subfigure [Nonstationary: $\rho=1$.]{ \label{fig:rho_1_short}
\includegraphics[width=4.5cm,height=4.5cm]{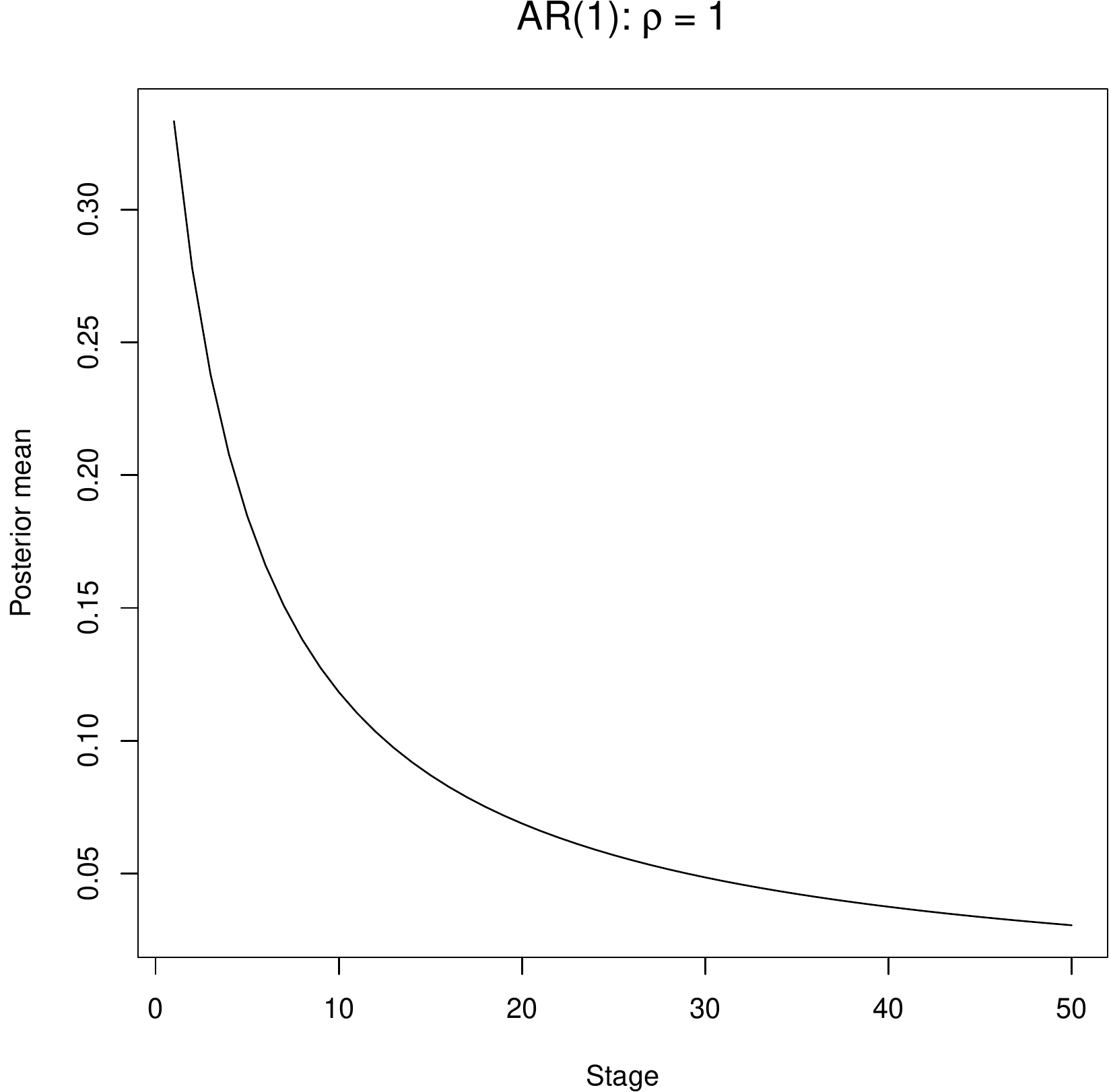}}\\
\vspace{2mm}
\subfigure [Nonstationary: $\rho=1.00005$.]{ \label{fig:rho_100005_short}
\includegraphics[width=4.5cm,height=4.5cm]{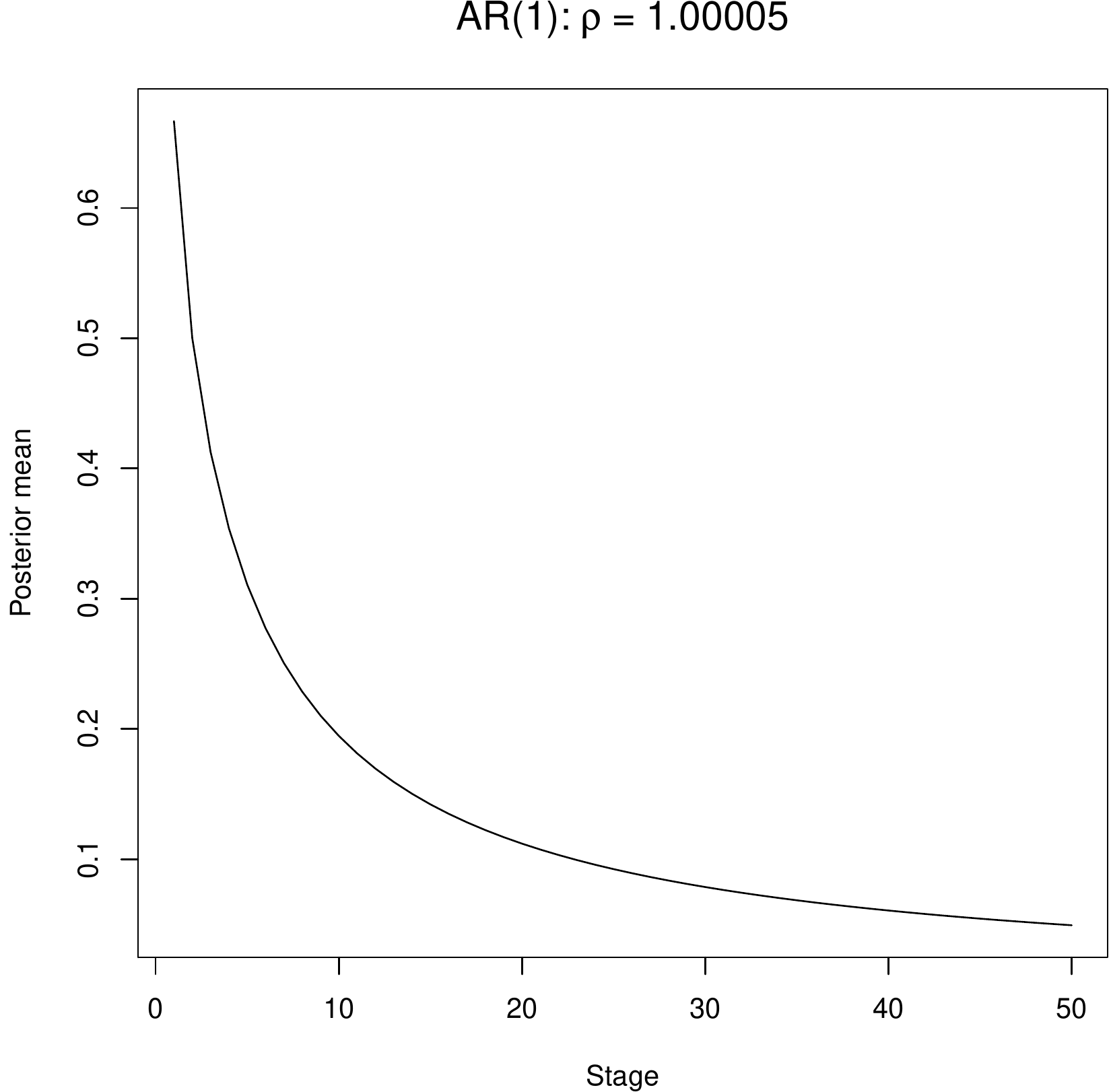}}
\hspace{2mm}
\subfigure [Nonstationary: $\rho=1.05$.]{ \label{fig:rho_105_short}
\includegraphics[width=4.5cm,height=4.5cm]{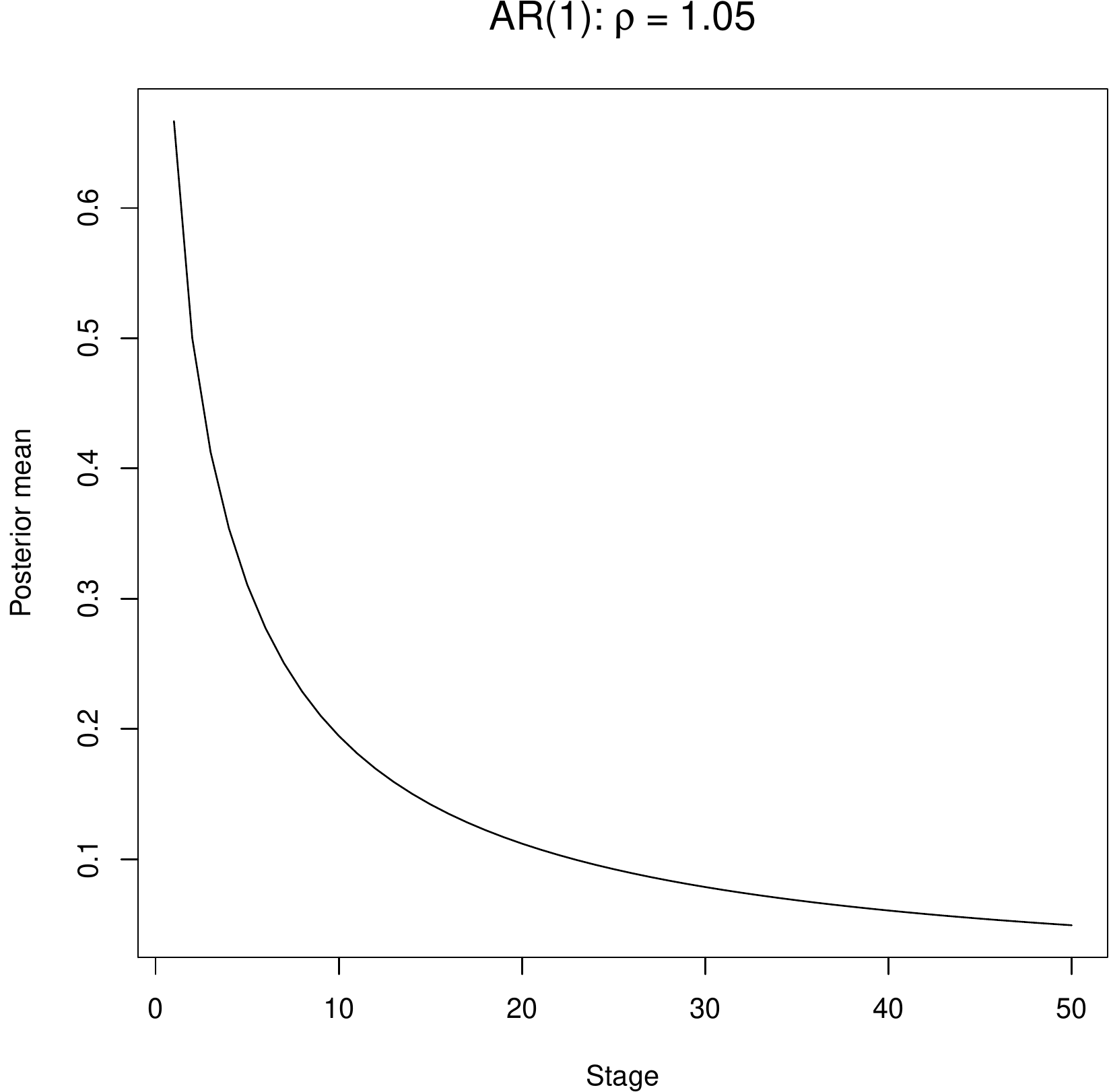}}
\hspace{2mm}
\subfigure [Nonstationary: $\rho=2$.]{ \label{fig:rho_2_short}
\includegraphics[width=4.5cm,height=4.5cm]{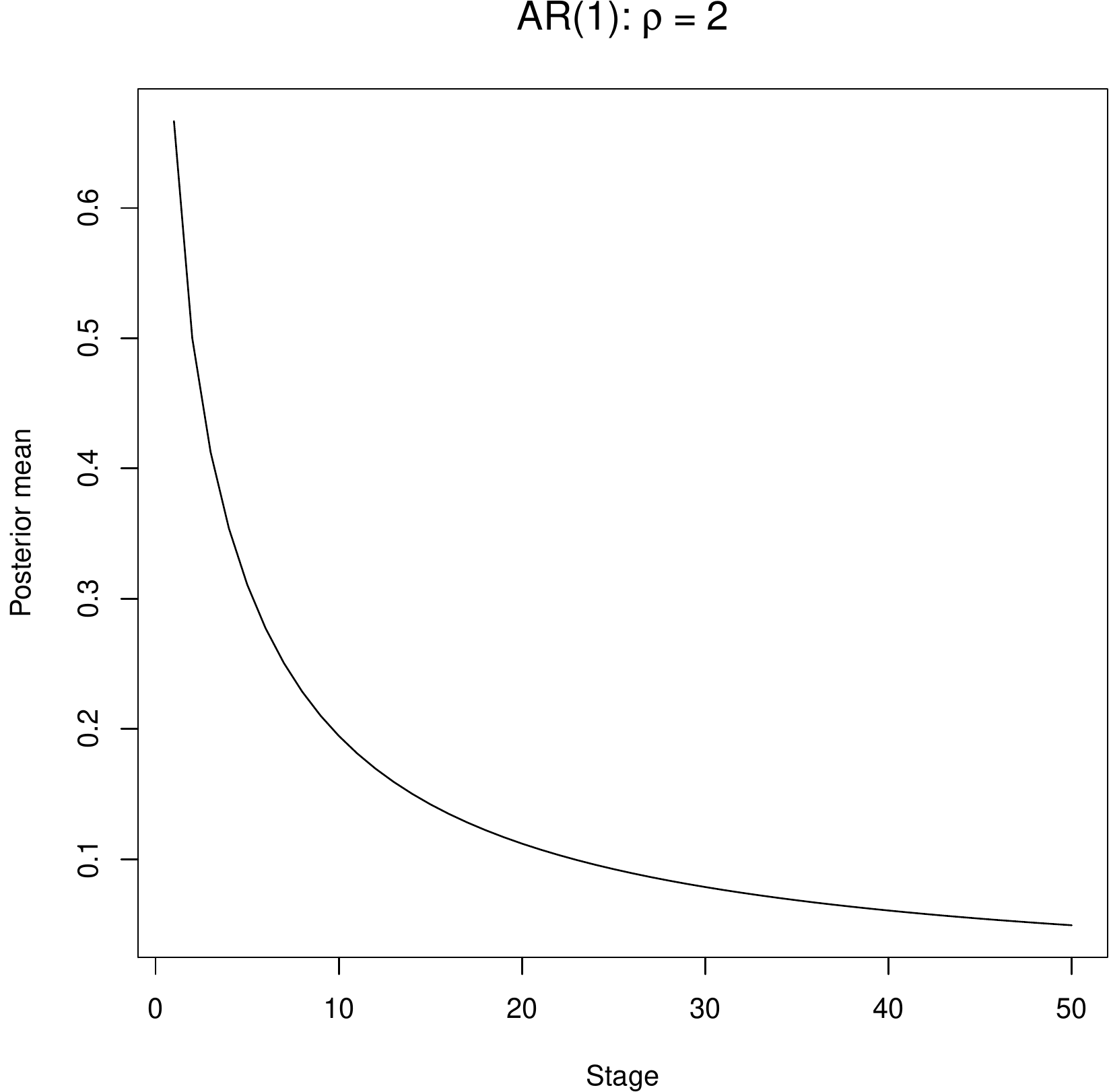}}
\caption{Parametric AR(1) example with $K=50$ and $n=50$.}
\label{fig:example2}
\end{figure}

\subsection{Case 3: Relatively small sample size, form of the model unknown}
\label{subsec:ar1_case3}

\subsubsection{Sample size}
We draw samples of sizes $2500$ from the AR(1) model for those values of $\rho$ as in Sections \ref{subsec:ar1_case1} and \ref{subsec:ar1_case2}
and evaluate the performance of our Bayesian methodology, setting $n=50$ and $K=50$, assuming that the model itself is unknown.

\subsubsection{Construction of bound}

Since we assume now that the model itself is unknown, there is no provision of obtaining the MLE of $\rho$ and constructing bounds on its basis. We also can not
compute $\tilde c_j$, since it requires knowledge of the underlying model. Hence, in the absence of such information, we set
\begin{equation}
c_j=\hat C_j/\log(j+1),
\label{eq:ar1_bound3}
\end{equation}	
where $\hat C_1=1$, and for $j>1$, $\hat C_j=\hat C_{j-1}+0.05$ if $y_{j-1}=1$ and $\hat C_j=\hat C_{j-1}-0.05$ if $y_{j-1}=0$. 

Thus, as before, we favour stationarity at the next stage if at the current stage stationarity is favoured ($y_j=1$) and nonstationarity otherwise. 
Note that unlike the previous cases, we have considered $\log(j+1)$ instead of $\log(\log(j+1))$. This faster rate turned out to be more appropriate
in this situation of very less information about the true model.

\subsubsection{Implementation}

The implementation remains the same as before, only that here it is much simpler because of the simple structure of the bound. 
Again, for this small sample, even with 2 cores, the results are delivered almost instantaneously.

\subsubsection{Results}
As before, we implement our method when the data is generated from the AR(1) model with $\rho$ randomly selected from $U(-1,1)$, and with $\rho$
taking the values $0.99$, $0.995$, $0.999$, $0.9999$, $1$, $1.00005$, $1.05$ and $2$. Figure \ref{fig:example3} shows that, again except
in the case where the true value of $\rho$ is $0.9999$, our method correctly detects stationarity and nonstationarity, albeit in a less precise manner
as in Figure \ref{fig:example2}. That even with such small sample, with no assumption about the true model, 
and with such subtle differences among the true values of $\rho$, our method performs well, is quite encouraging, again, despite
its fallibility at $\rho=0.9999$, which is perhaps not expected to be detected correctly in this situation of so less information.
\begin{figure}
\centering
\subfigure [Stationary: $|\rho|<1$.]{ \label{fig:abs_rho_less_1_short_nonpara}
\includegraphics[width=4.5cm,height=4.5cm]{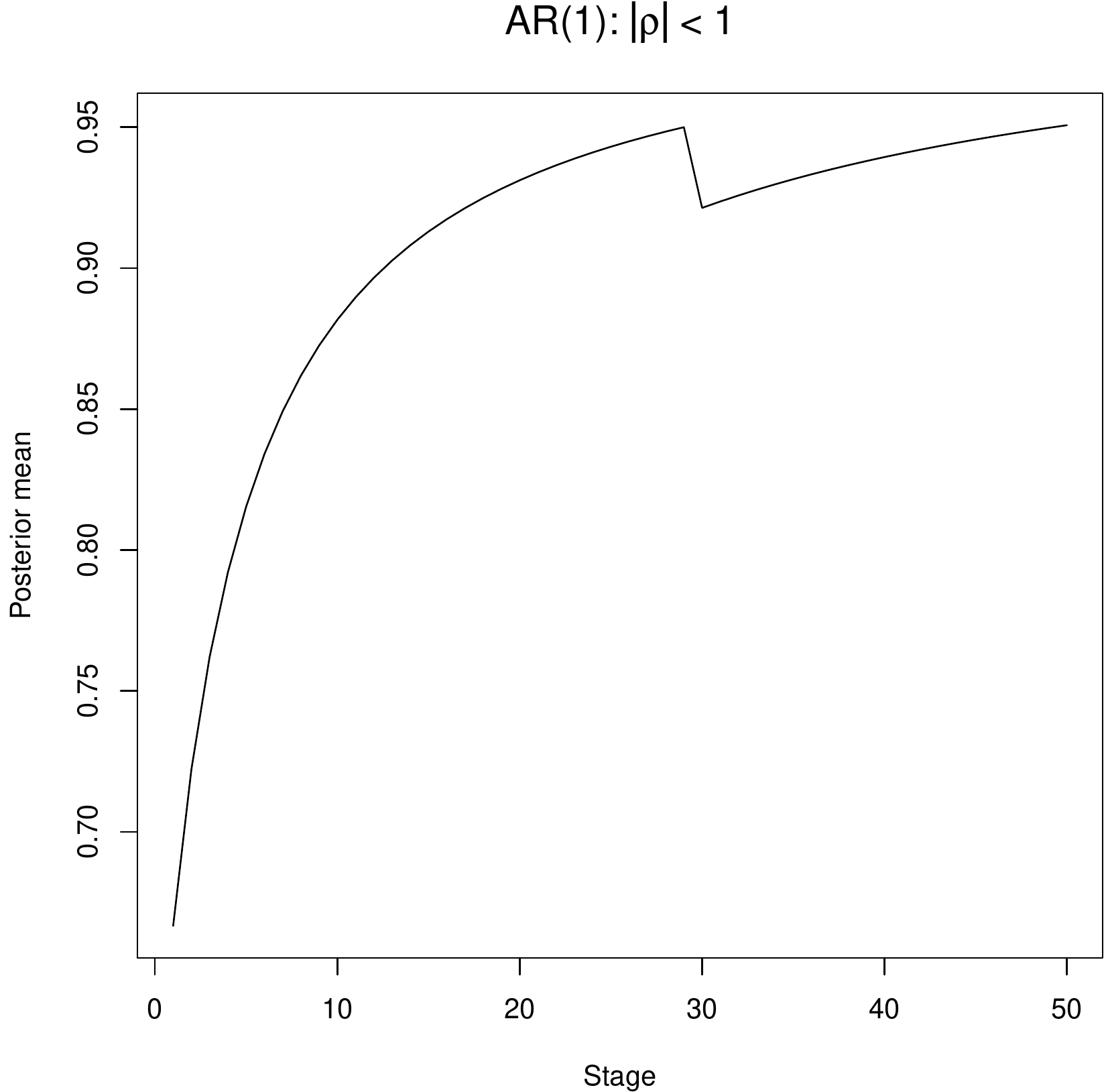}}
\hspace{2mm}
\subfigure [Stationary: $\rho=0.99$.]{ \label{fig:rho_99_short_nonpara}
\includegraphics[width=4.5cm,height=4.5cm]{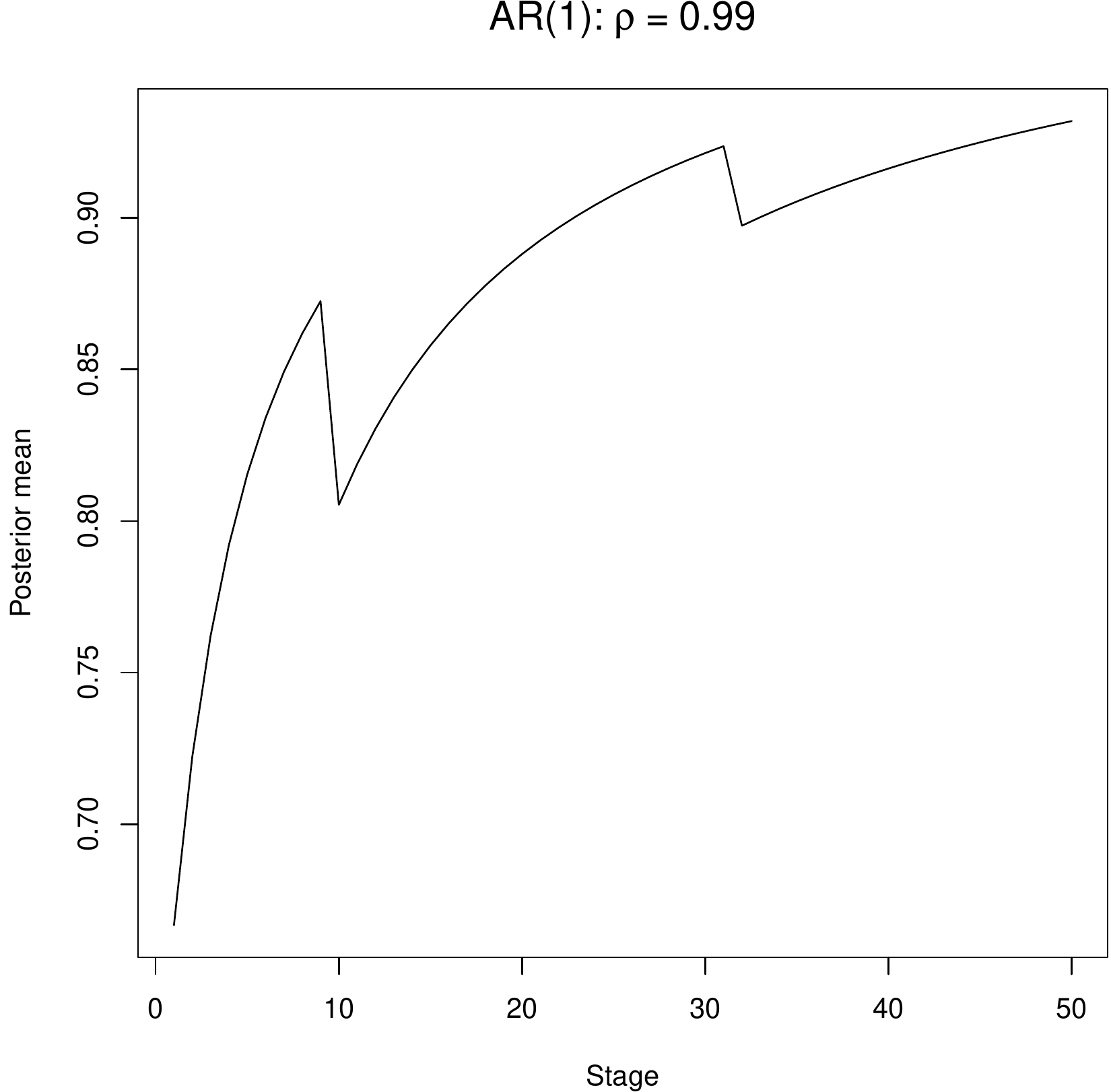}}
\hspace{2mm}
\subfigure [Stationary: $\rho=0.995$.]{ \label{fig:rho_995_short_nonpara}
\includegraphics[width=4.5cm,height=4.5cm]{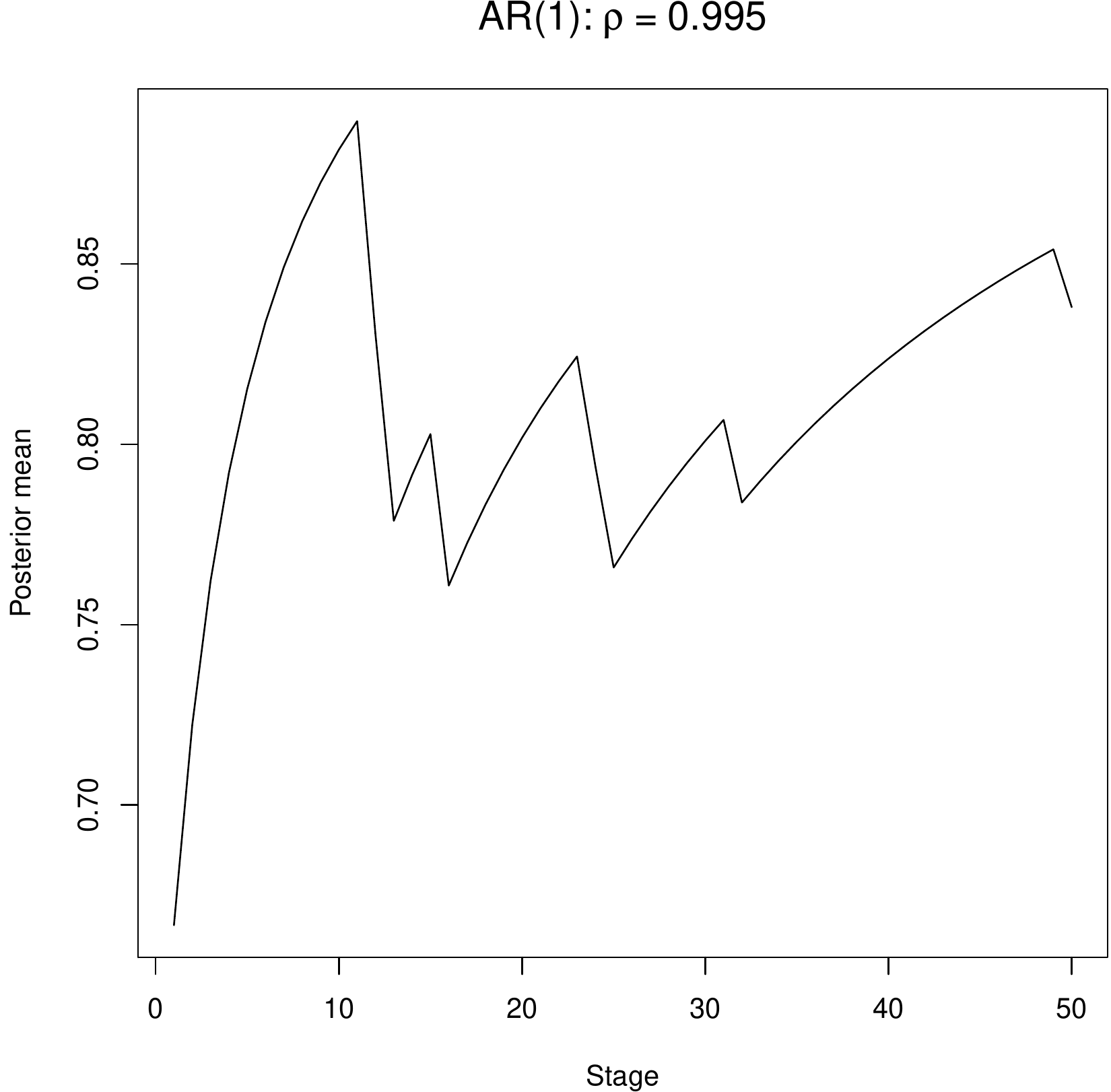}}\\
\vspace{2mm}
\subfigure [Stationary: $\rho=0.999$.]{ \label{fig:rho_999_short_nonpara}
\includegraphics[width=4.5cm,height=4.5cm]{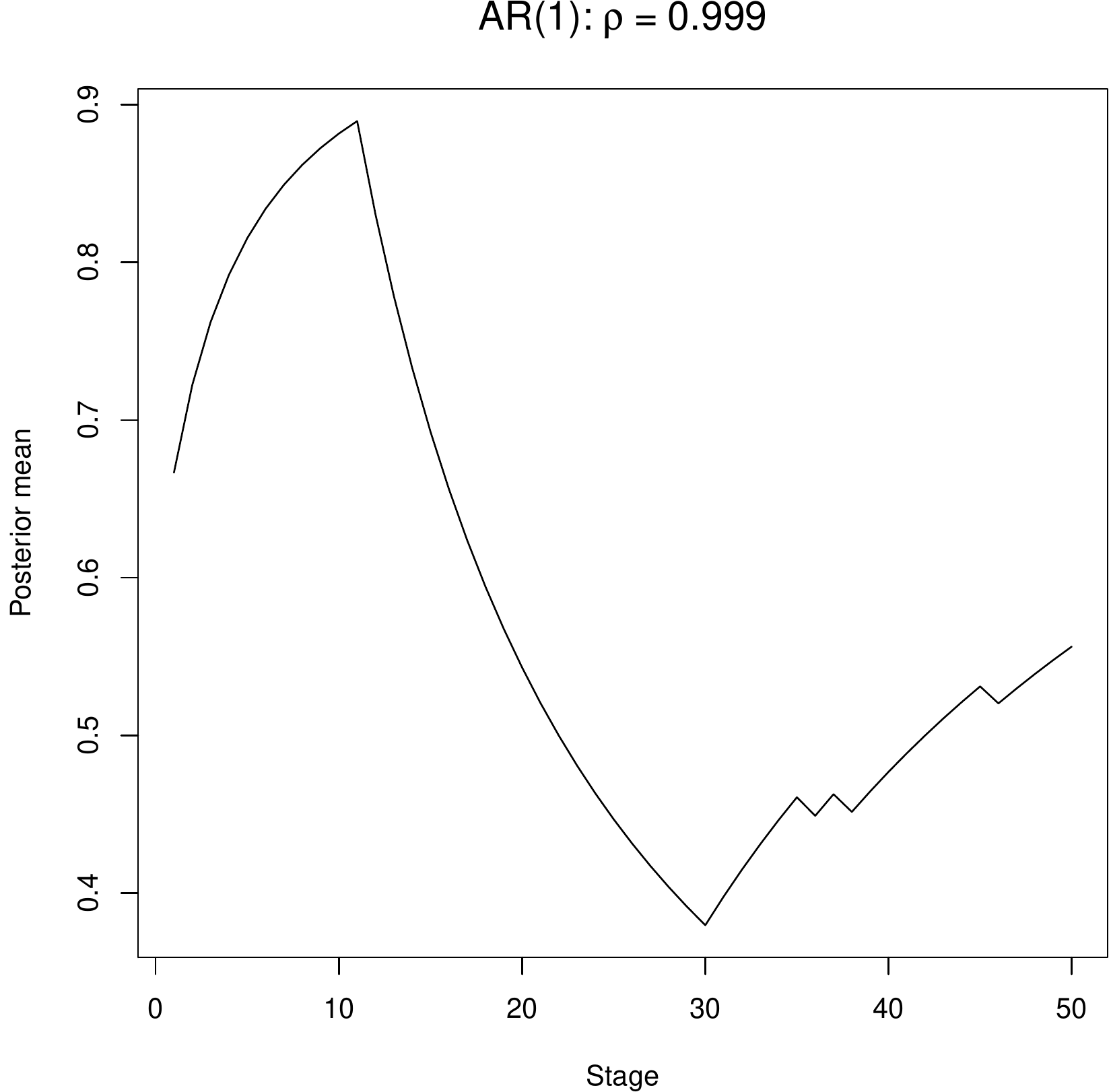}}
\hspace{2mm}
\subfigure [Stationary: $\rho=0.9999$.]{ \label{fig:rho_9999_short_nonpara}
\includegraphics[width=4.5cm,height=4.5cm]{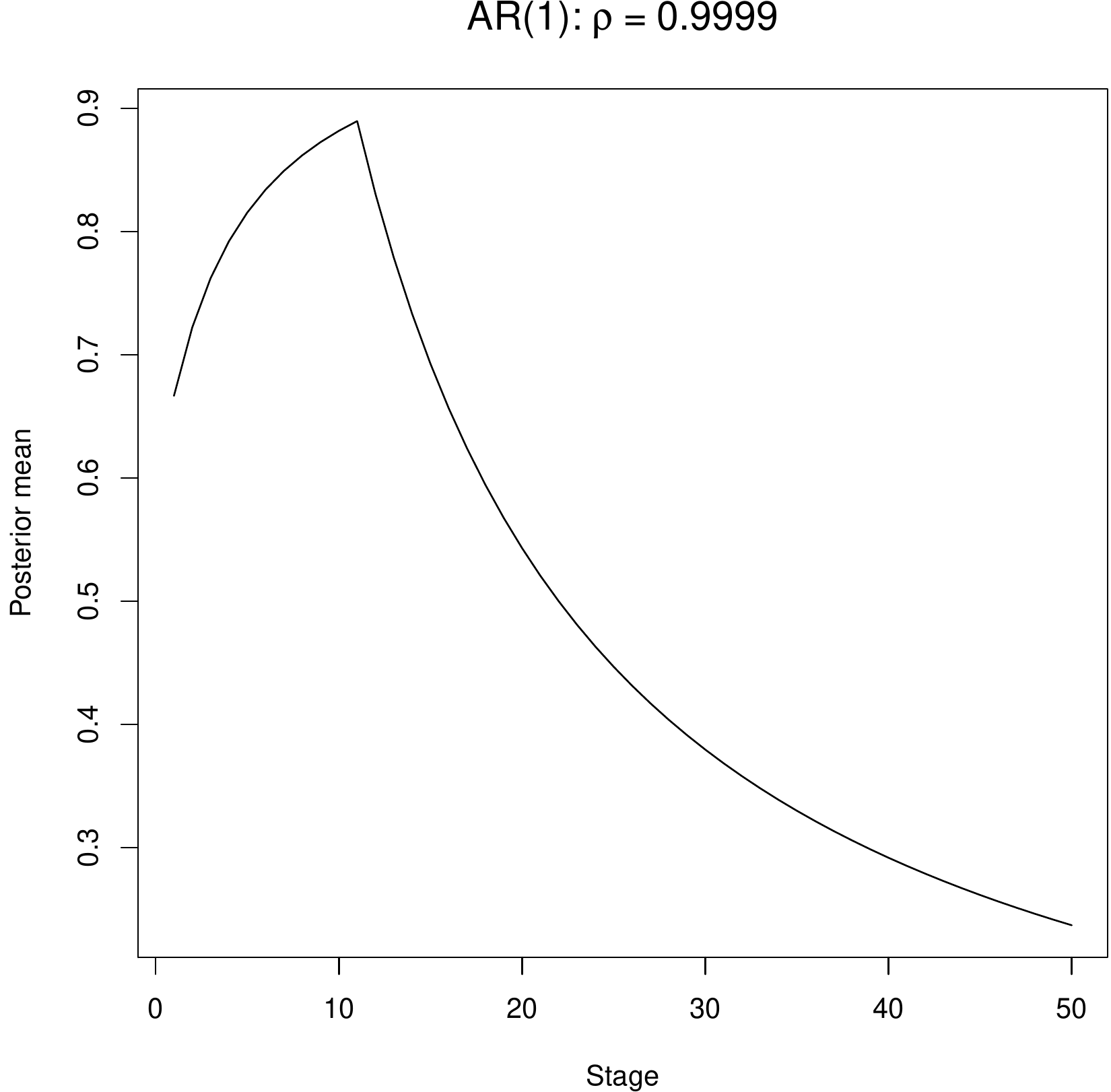}}
\hspace{2mm}
\subfigure [Nonstationary: $\rho=1$.]{ \label{fig:rho_1_short_nonpara}
\includegraphics[width=4.5cm,height=4.5cm]{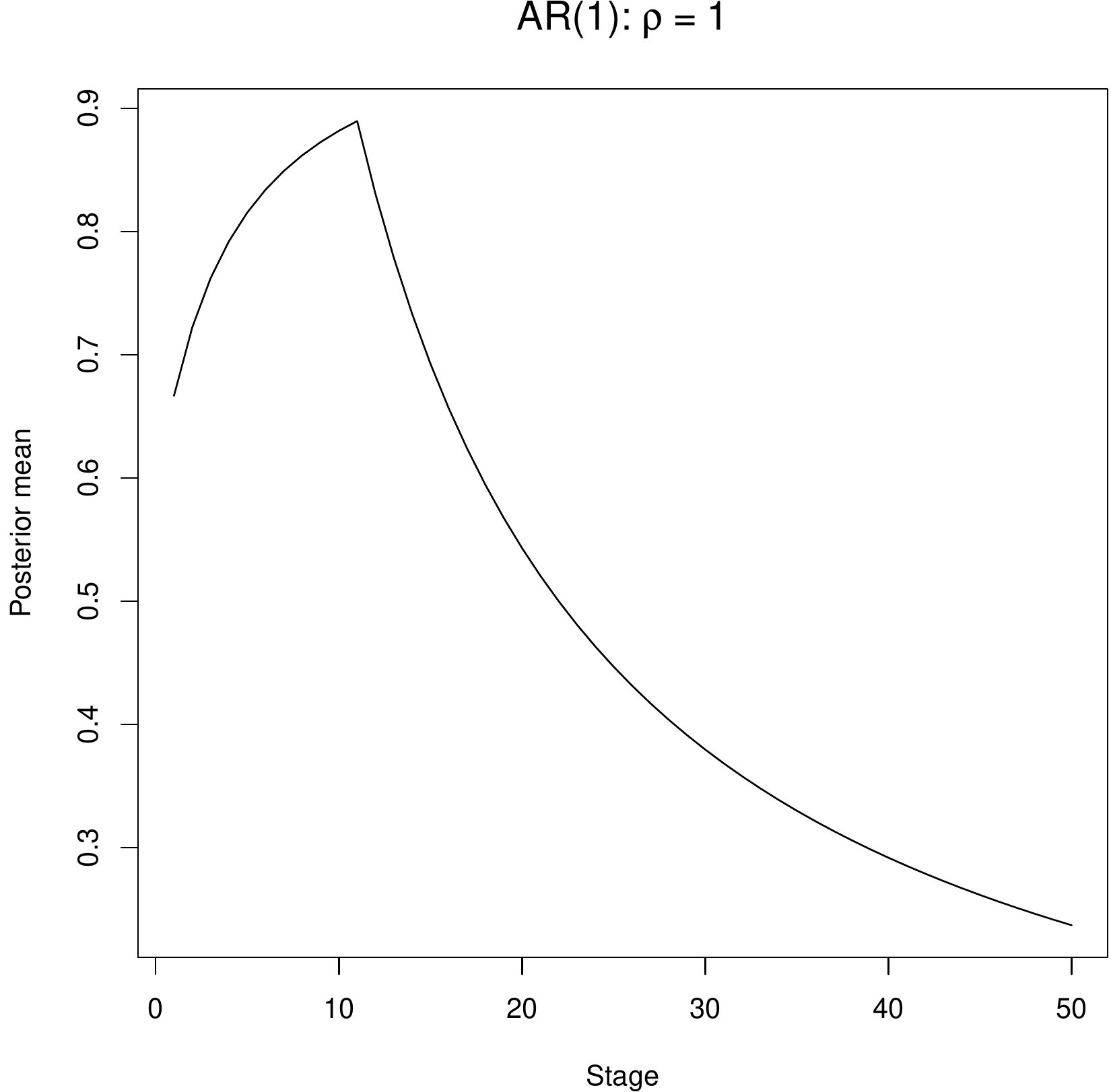}}\\
\vspace{2mm}
\subfigure [Nonstationary: $\rho=1.00005$.]{ \label{fig:rho_100005_short_nonpara}
\includegraphics[width=4.5cm,height=4.5cm]{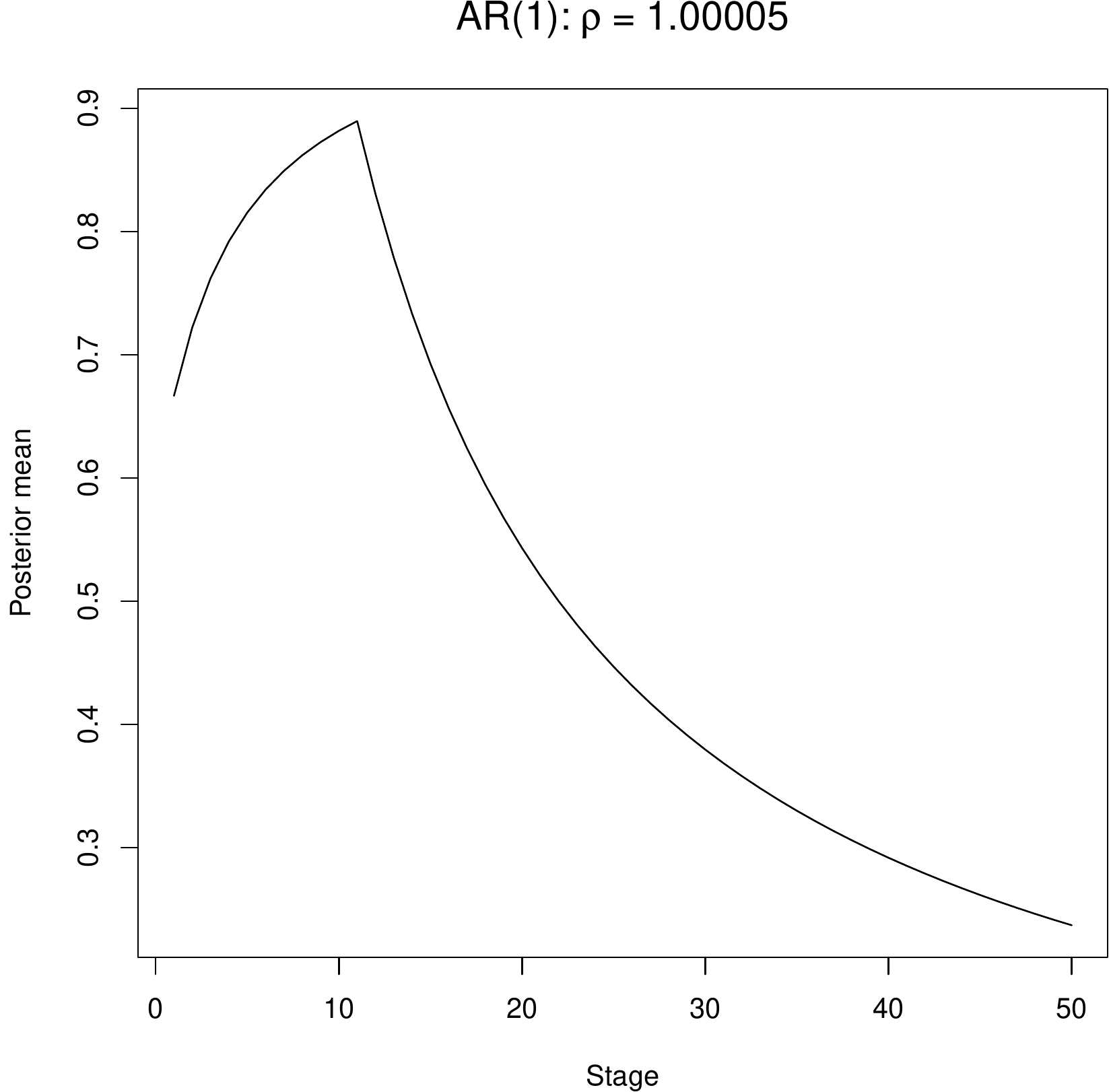}}
\hspace{2mm}
\subfigure [Nonstationary: $\rho=1.05$.]{ \label{fig:rho_105_short_nonpara}
\includegraphics[width=4.5cm,height=4.5cm]{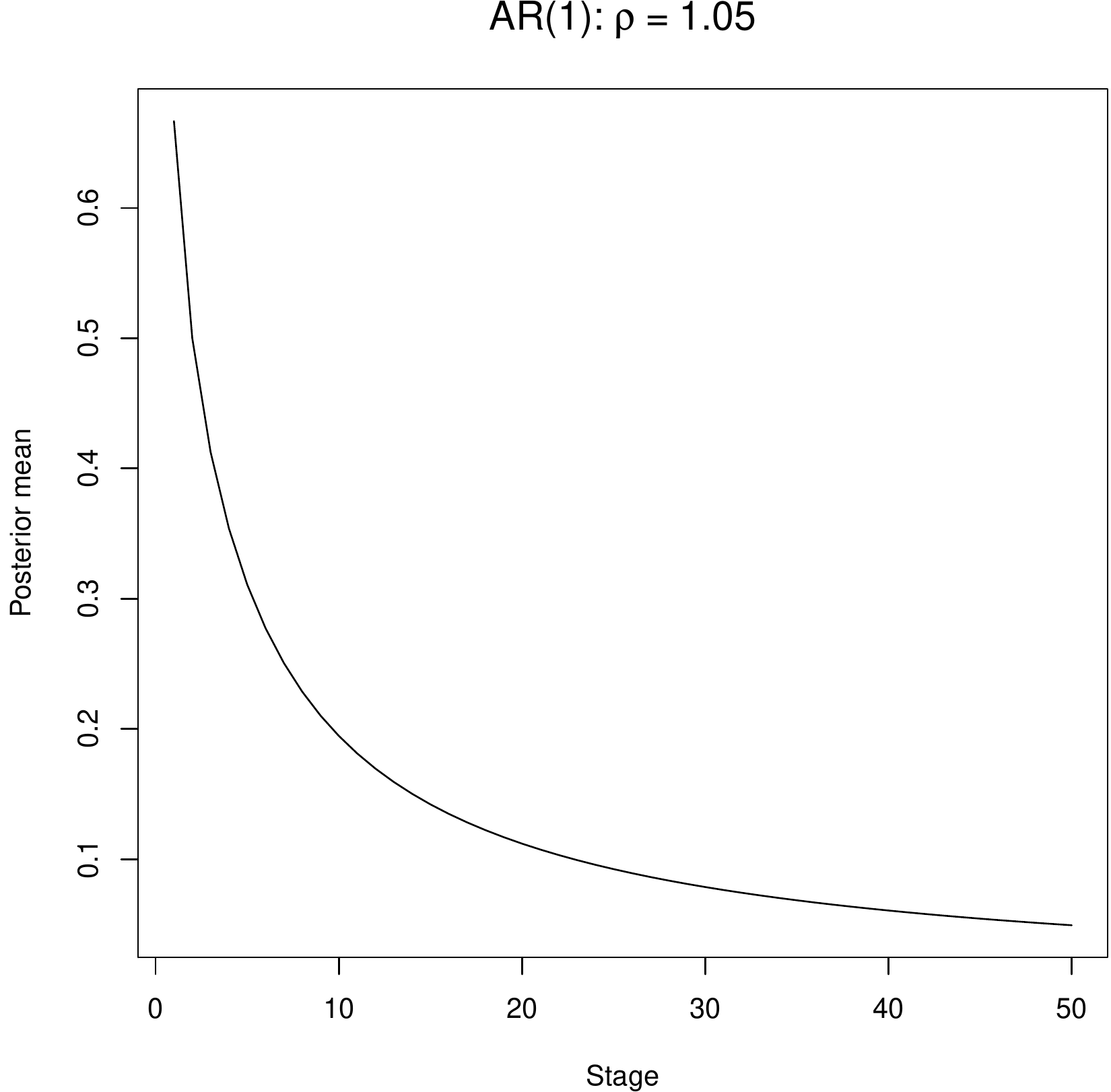}}
\hspace{2mm}
\subfigure [Nonstationary: $\rho=2$.]{ \label{fig:rho_2_short_nonpara}
\includegraphics[width=4.5cm,height=4.5cm]{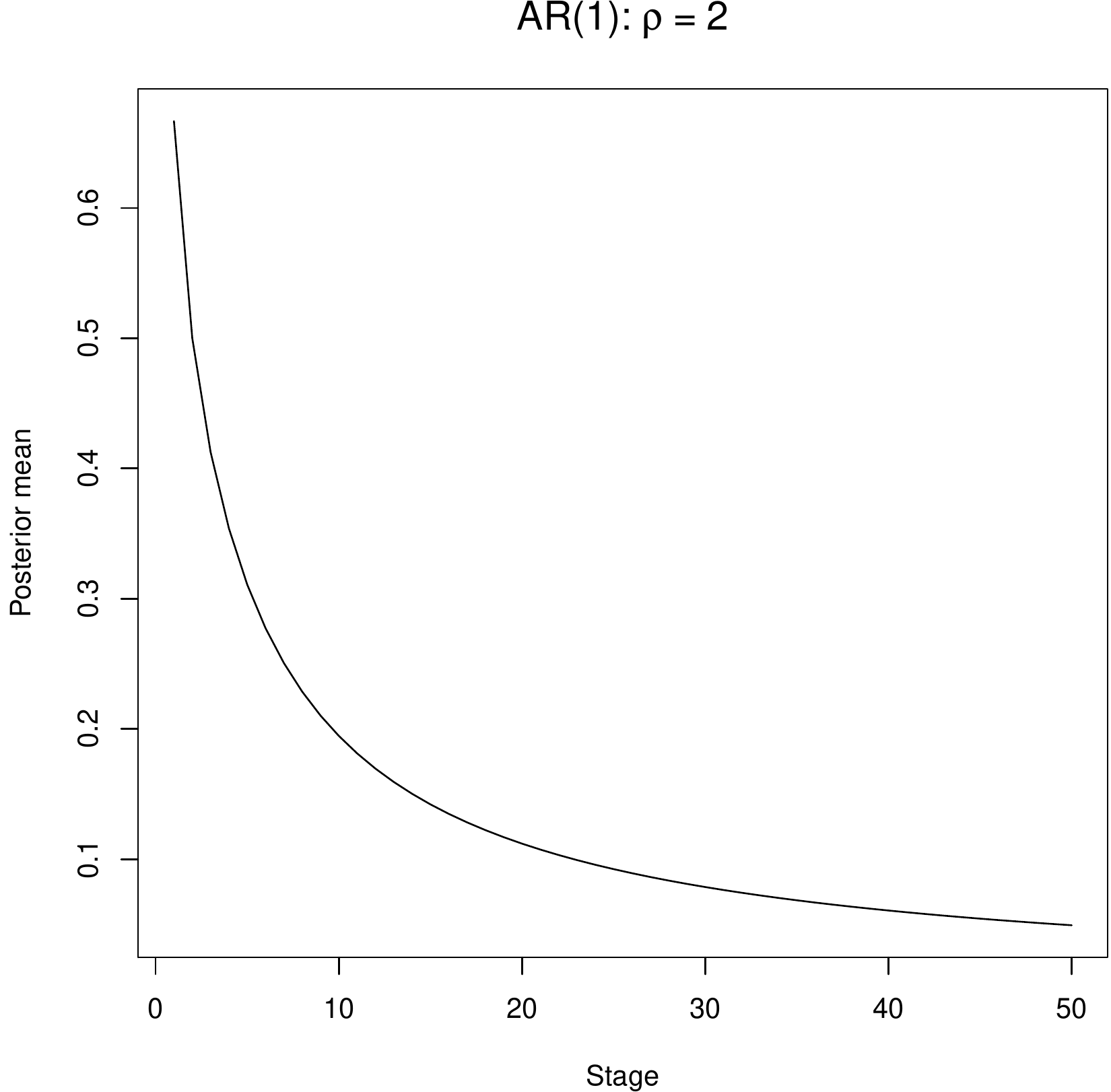}}
\caption{Nonparametric AR(1) example with $K=50$ and $n=50$.}
\label{fig:example3}
\end{figure}

\subsubsection{Comparison with classical tests of nonstationarity}
\label{subsubsec:classical_tests}
To test stationarity of AR(1) model, there are well-known classical hypotheses tests, namely, the augmented Dickey-Fuller (ADF) test (\ctn{Dickey79}), 
the Philips-Perron (PP) test (\ctn{Philips88}), and the Kwiatkowski, Phillips, Schmidt, Shin (KPSS) test (\ctn{KPSS92}).

Researchers have noticed that the first two tests, PP and ADF, are not very efficient in distinguishing between stationarity and nonstationarity when the process
is stationary, but at the verge of stationarity and nonstationarity. Indeed, when we apply these tests on our datasets with sample size 2500,
we find that these two tests correctly determines stationarity/nonstationarity of the process when $\rho$ is randomly chosen between $(-1,1)$, $\rho=0.99$
and $\rho=0.995$, at the $5\%$ level of significance, but fails when $\rho=0.999$, $0.9999$ and $1.05$. However, both these tests correct conclude nonstationarity
when $\rho=1$ and $1.00005$. For $\rho=2$, both the tests turn out to be inapplicable.

On the other hand, at the 5\% level of significance, the KPSS test provides correct answers whenever $|\rho|<1$, but fails when $\rho\geq 1$.

Thus, our proposed method outperforms all the three existing popular methods of testing stationarity in AR(1) models. Here we emphasize that the testing methods
ADF, PP and KPSS are particularly designed to detect stationarity in autoregressive models, while ours is a completely general method. That our method still
managed to outperform the existing specialized testing methods, is very encouraging.

\section{Second illustration: AR(2), ARCH(1) and GARCH(1,1) models}
\label{sec:other_time_series}

We now test our ideas on relatively more complex time series models. In particular, we consider autoregressive models of order $2$ (AR(2)), 
first order autoregressive conditional heteroscedastic model (ARCH(1)) and generalized ARCH of order one (GARCH(1,1)). We consider samples of size $2500$
for our investigation, since the relatively small sample size, as we observed in the context of AR(1), can pose beneficial challenge to our Bayesian method. 

\subsection{Application to AR(2)}
\label{subsec:ar2}
The AR(2) model is given by
\begin{equation}
x_t=\alpha x_{t-1}+\beta x_{t-2}+\epsilon_t;~t=1,2,\ldots,
\label{eq:ar2}
\end{equation}
where we set $x_1=x_2=0$ and $\epsilon_t\stackrel{iid}{\sim}N(0,1)$, for $t=1,2,\ldots$. The necessary and sufficient conditions for stationarity of the AR(2) model 
(\ref{eq:ar2}) are given by (see, for example, \ctn{Shumway06}) 
\begin{align}
	\alpha+\beta<1;\notag\\
	\beta-\alpha<1;\notag\\
	\beta>-1.
\label{eq:ar2_stationarity}
\end{align}
We simulate samples of size $2500$ from (\ref{eq:ar2}) with various fixed values of $\alpha$ and $\beta$ that satisfy and do not satisfy (\ref{eq:ar2_stationarity}),
and apply our Bayesian procedure to ascertain stationarity and nonstationarity, with the bound of the form (\ref{eq:ar1_bound3}), starting
with $\hat C_1=1$. We initially consider $(n=50,K=50)$ but in a few nonstationary cases 
(($\alpha=1,\beta=0$), ($\alpha=0,\beta=1$) and ($\alpha=0.5,\beta=0.5$))
this failed to work satisfactorily, since a relatively large value of $n$ in the context of relatively small sample size 
has the tendency to create overlaps among neighboring regions of local stationarity, in effect, destroying local stationarity which is at the heart of our Bayesian procedure. 
This happens when the underlying time series diverges slowly, as in the aforementioned values of $(\alpha,\beta)$. Figure \ref{fig:slow_divergence}
captures such behaviours of such slowly diverging nonstationary processes in comparison to fast diverging nonstationary processes.
\begin{figure}
\centering
\subfigure [Slow divergence: $\alpha=0.5$, $\beta=0.5$.]{ \label{fig:ar2plot_5_5_short_nonpara2}
\includegraphics[width=4.5cm,height=4.5cm]{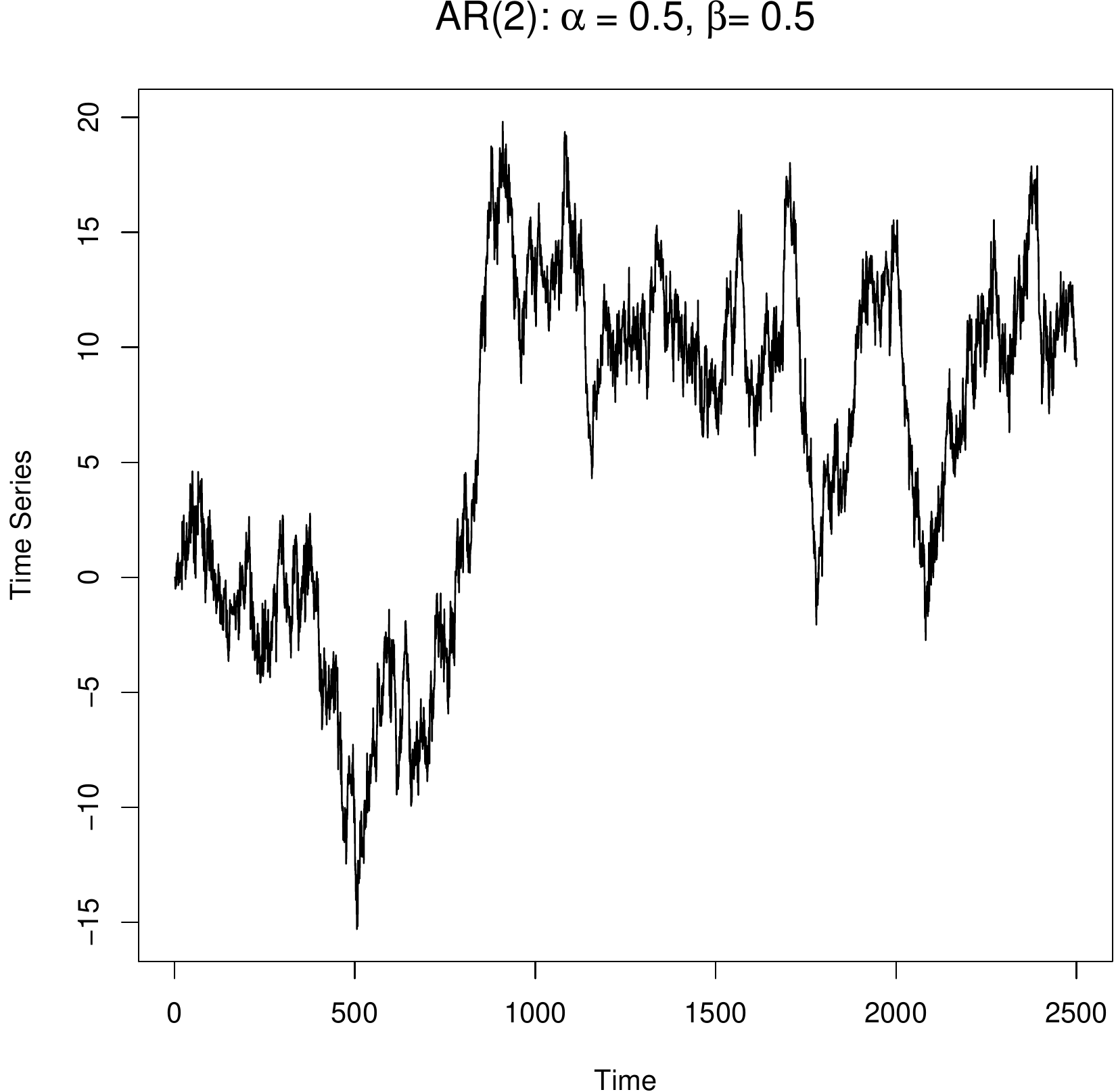}}
\hspace{2mm}
\subfigure [Slow divergence: $\alpha=0$, $\beta=1$.]{ \label{fig:ar2plot_0_1_short_nonpara2}
\includegraphics[width=4.5cm,height=4.5cm]{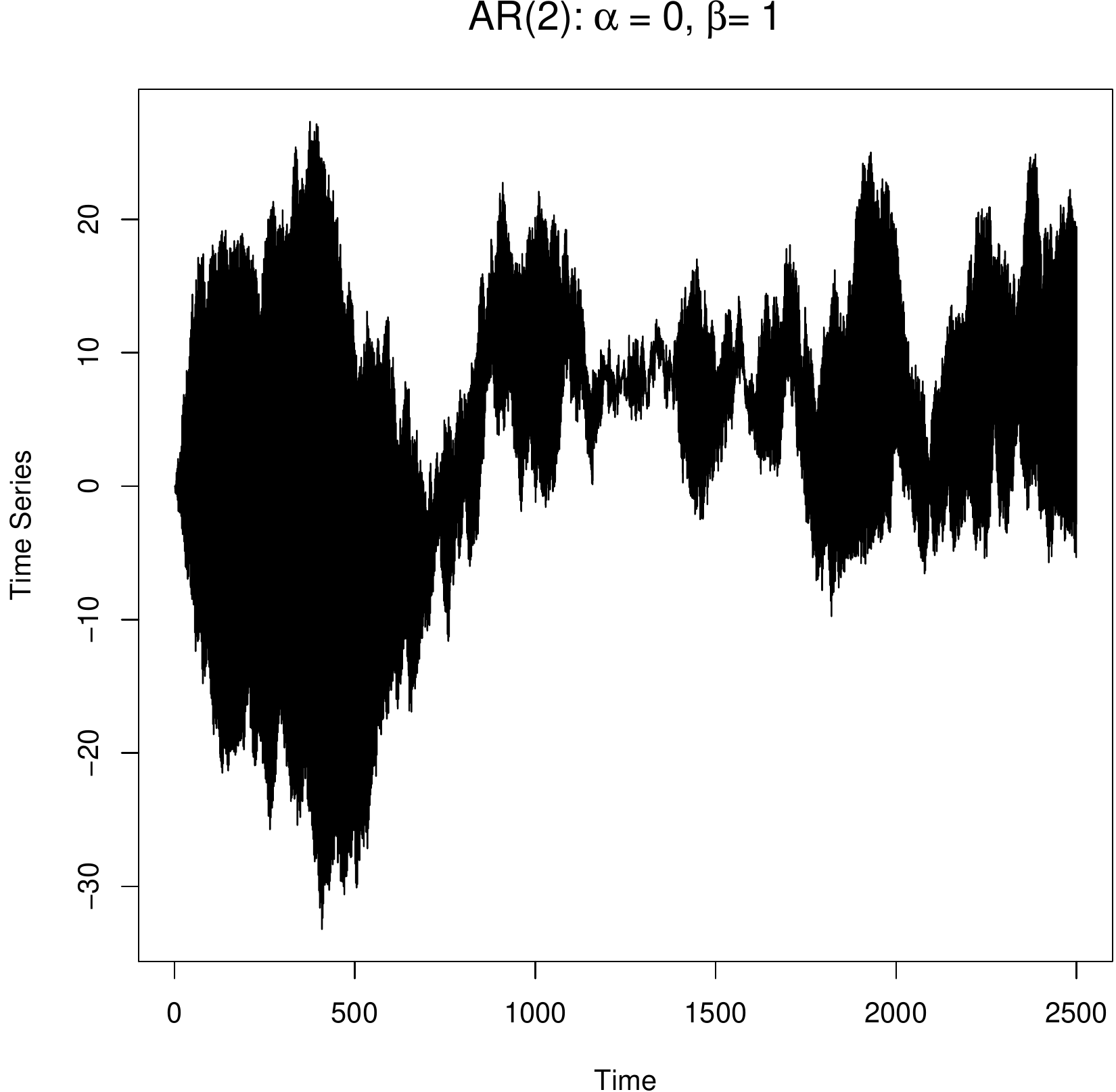}}
\hspace{2mm}
\subfigure [Slow divergence: $\alpha=1$, $\beta=0$.]{ \label{fig:ar2plot_1_0_short_nonpara2}
\includegraphics[width=4.5cm,height=4.5cm]{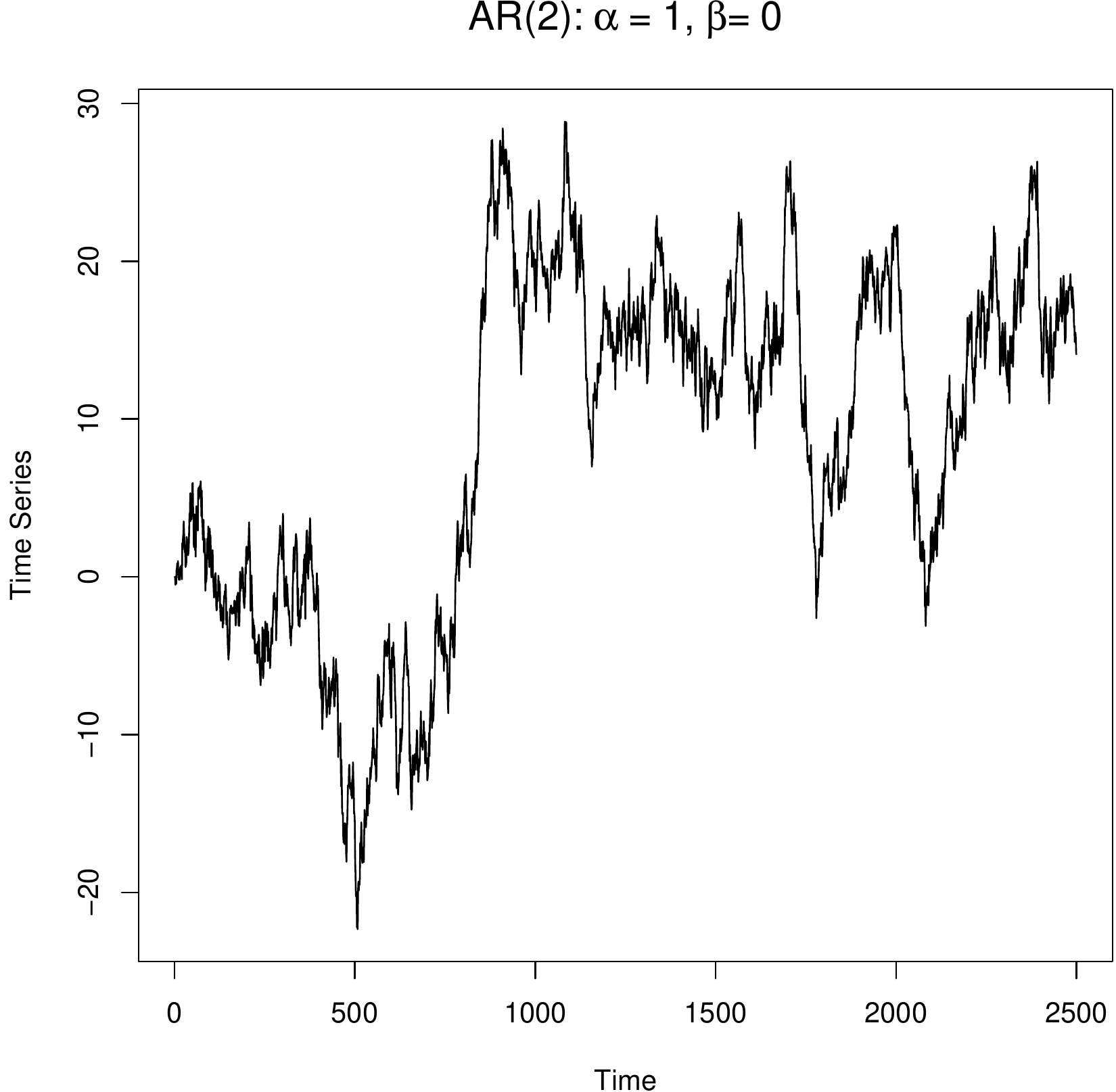}}\\
\vspace{2mm}
\subfigure [Fast divergence: $\alpha=0.5$, $\beta=0.9$.]{ \label{fig:ar2plot_5_9_short_nonpara2}
\includegraphics[width=4.5cm,height=4.5cm]{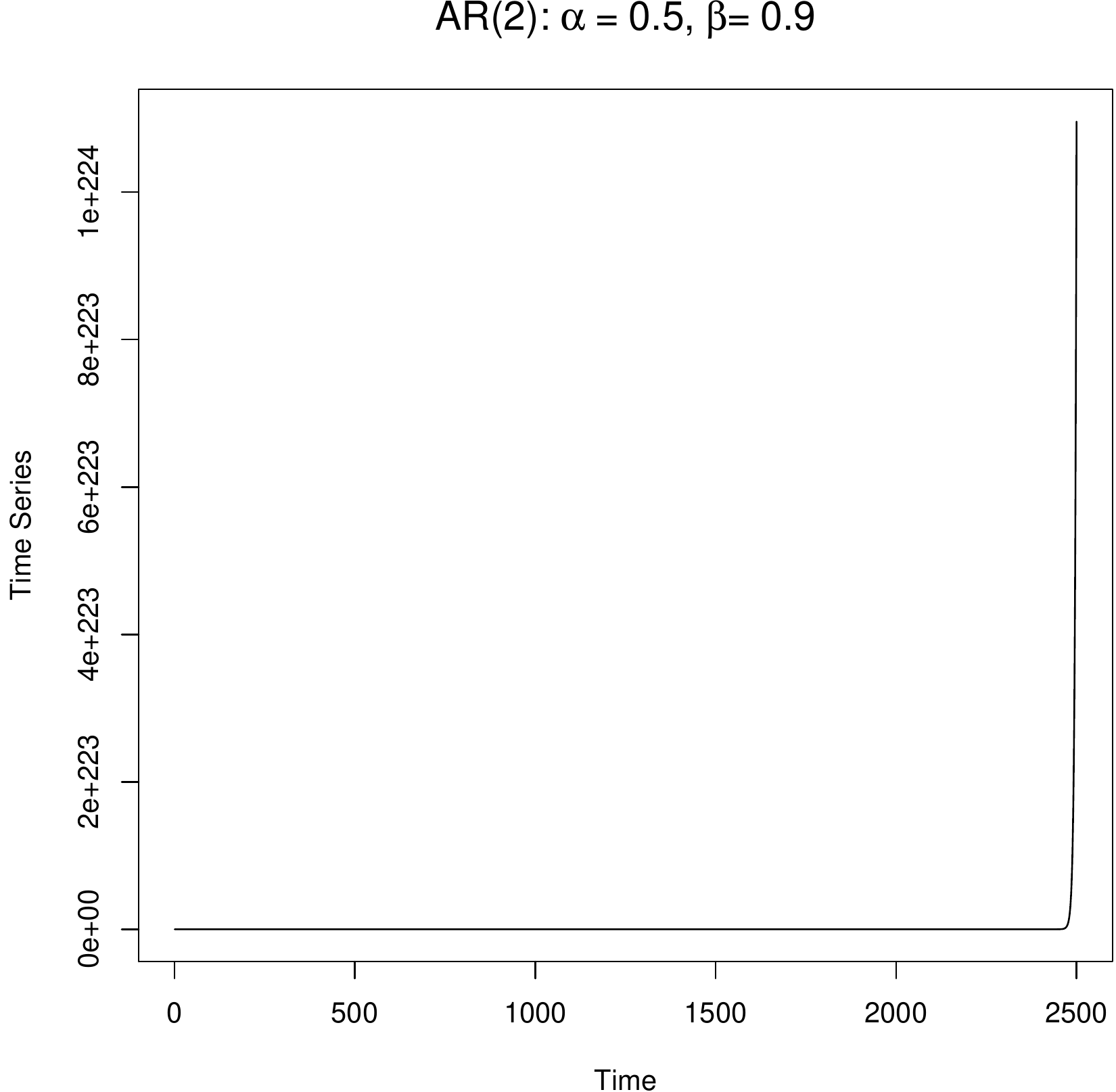}}
\hspace{2mm}
\subfigure [Fast divergence: $\alpha=0.6$, $\beta=0.6$.]{ \label{fig:ar2plot_6_6_short_nonpara2}
\includegraphics[width=4.5cm,height=4.5cm]{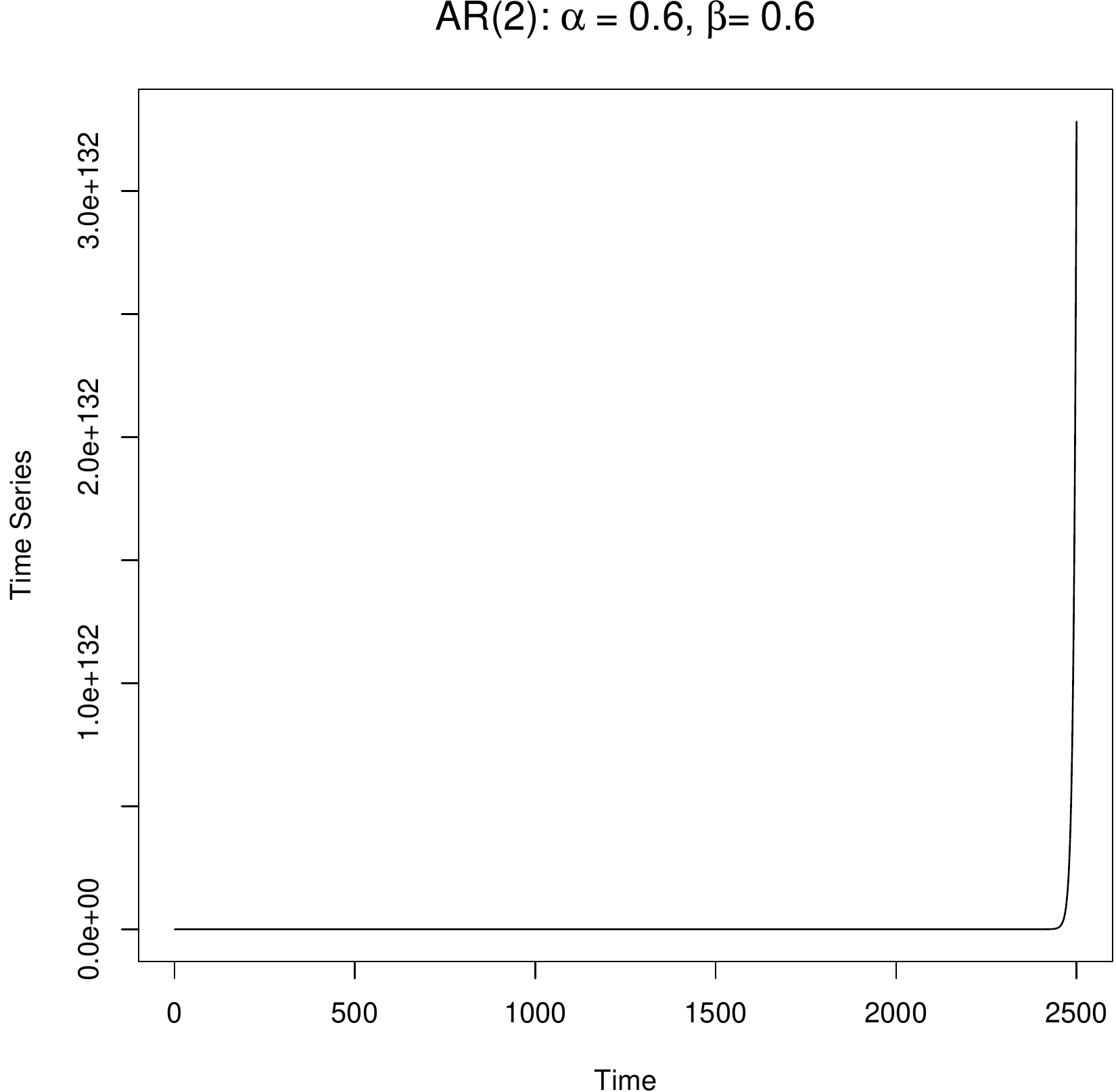}}\\
\caption{Slow and fast divergence tendencies of AR(2) model for several values of $\alpha$ and $\beta$.}
\label{fig:slow_divergence}
\end{figure}

On the other hand, the choice $(n=5,K=500)$ turned out to work very well in all the cases that we considered. Figure \ref{fig:example_ar2}, depicting the results
of our Bayesian method for various values of $\alpha$ and $\beta$ for $(n=5,K=500)$, shows that all the stationarity and nonstationarity situations are correctly
identified.  
\begin{figure}
\centering
\subfigure [Stationary: $\alpha=0.3$, $\beta=0.4$.]{ \label{fig:ar2_3_4_short_nonpara2}
\includegraphics[width=4.5cm,height=4.5cm]{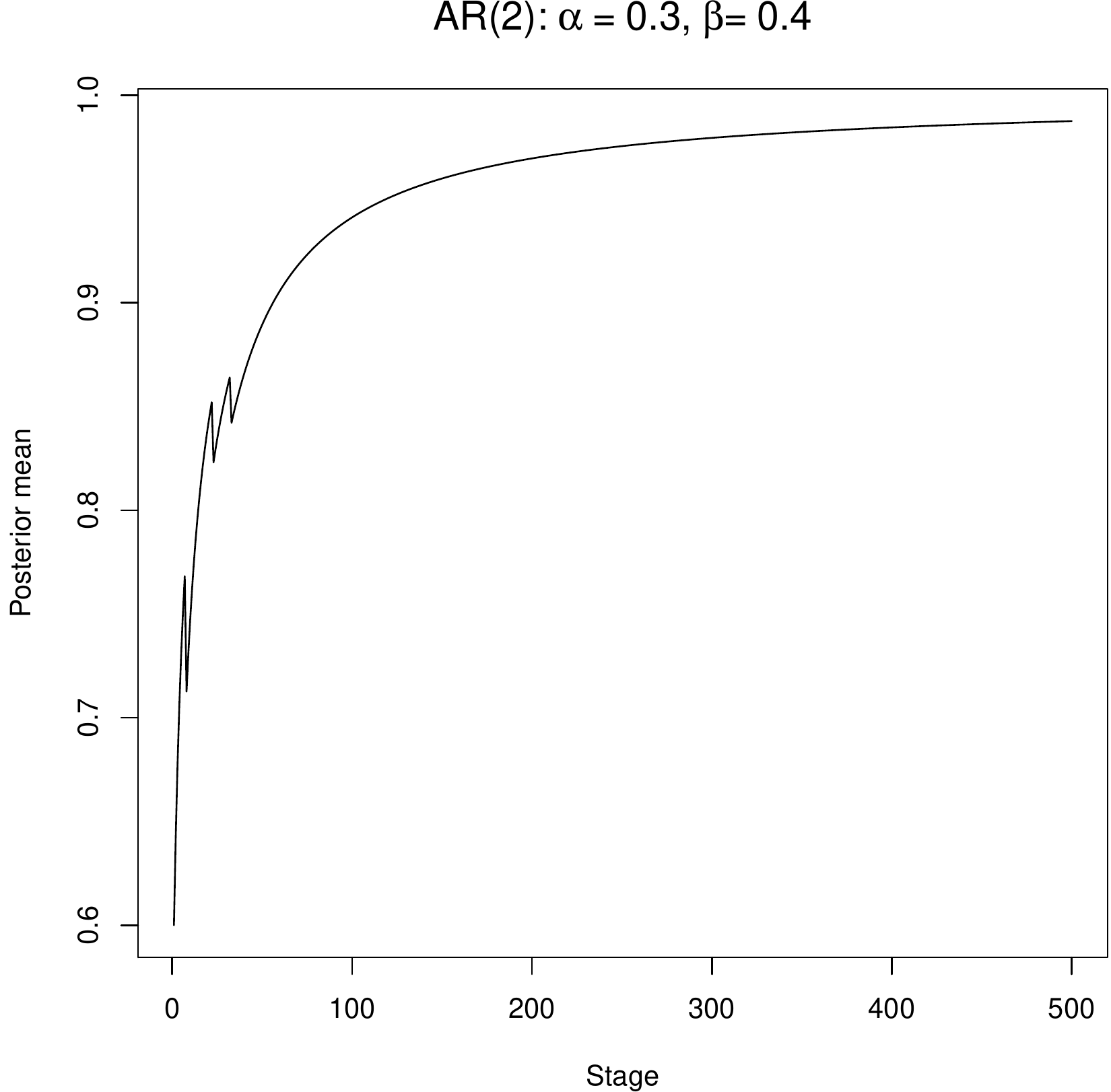}}
\hspace{2mm}
\subfigure [Stationary: $\alpha=0.4$, $\beta=0.3$.]{ \label{fig:ar2_4_3_short_nonpara2}
\includegraphics[width=4.5cm,height=4.5cm]{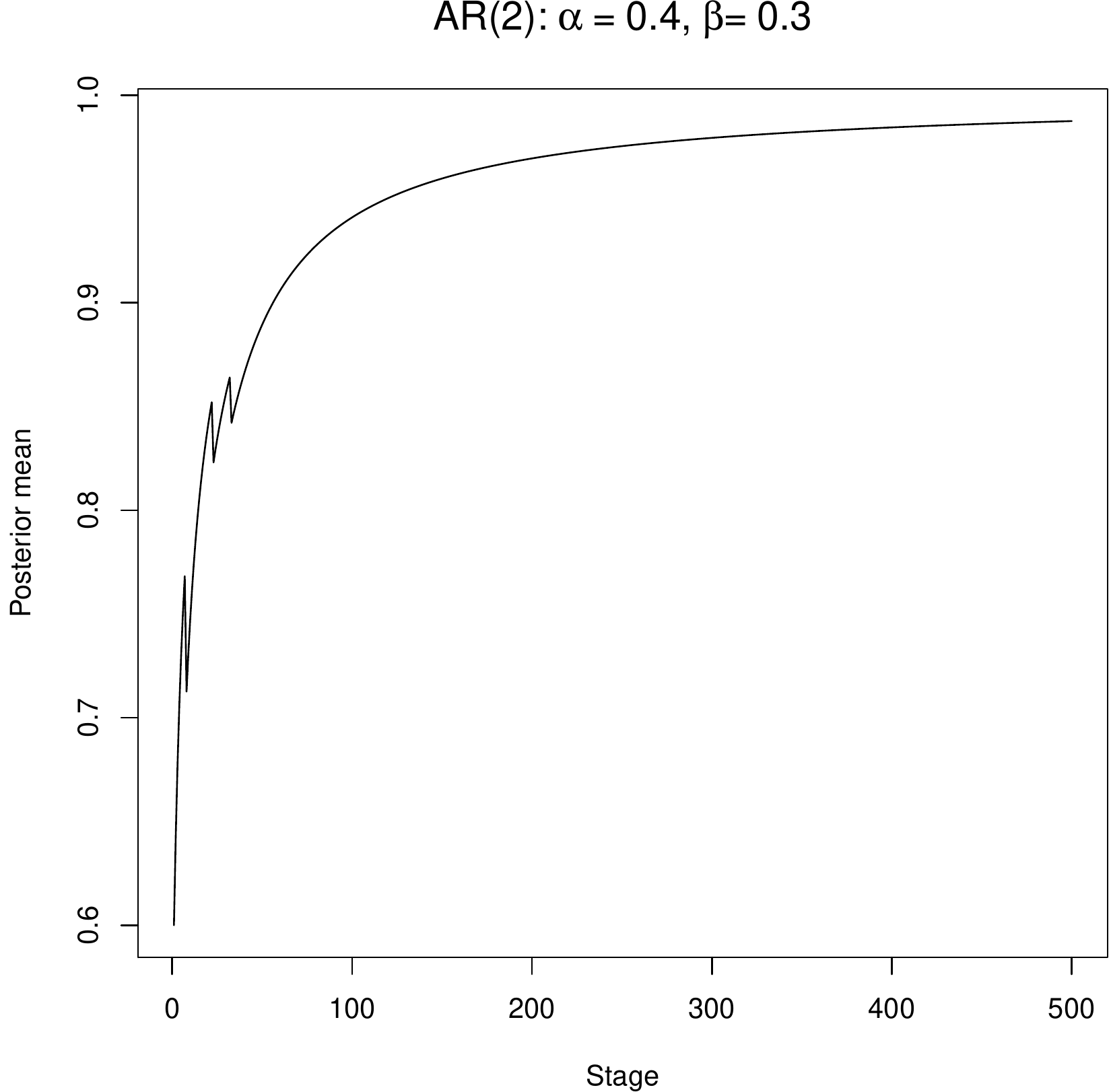}}
\hspace{2mm}
\subfigure [Stationary: $\alpha=0.4$, $\beta=0.5$.]{ \label{fig:ar2_4_5_short_nonpara2}
\includegraphics[width=4.5cm,height=4.5cm]{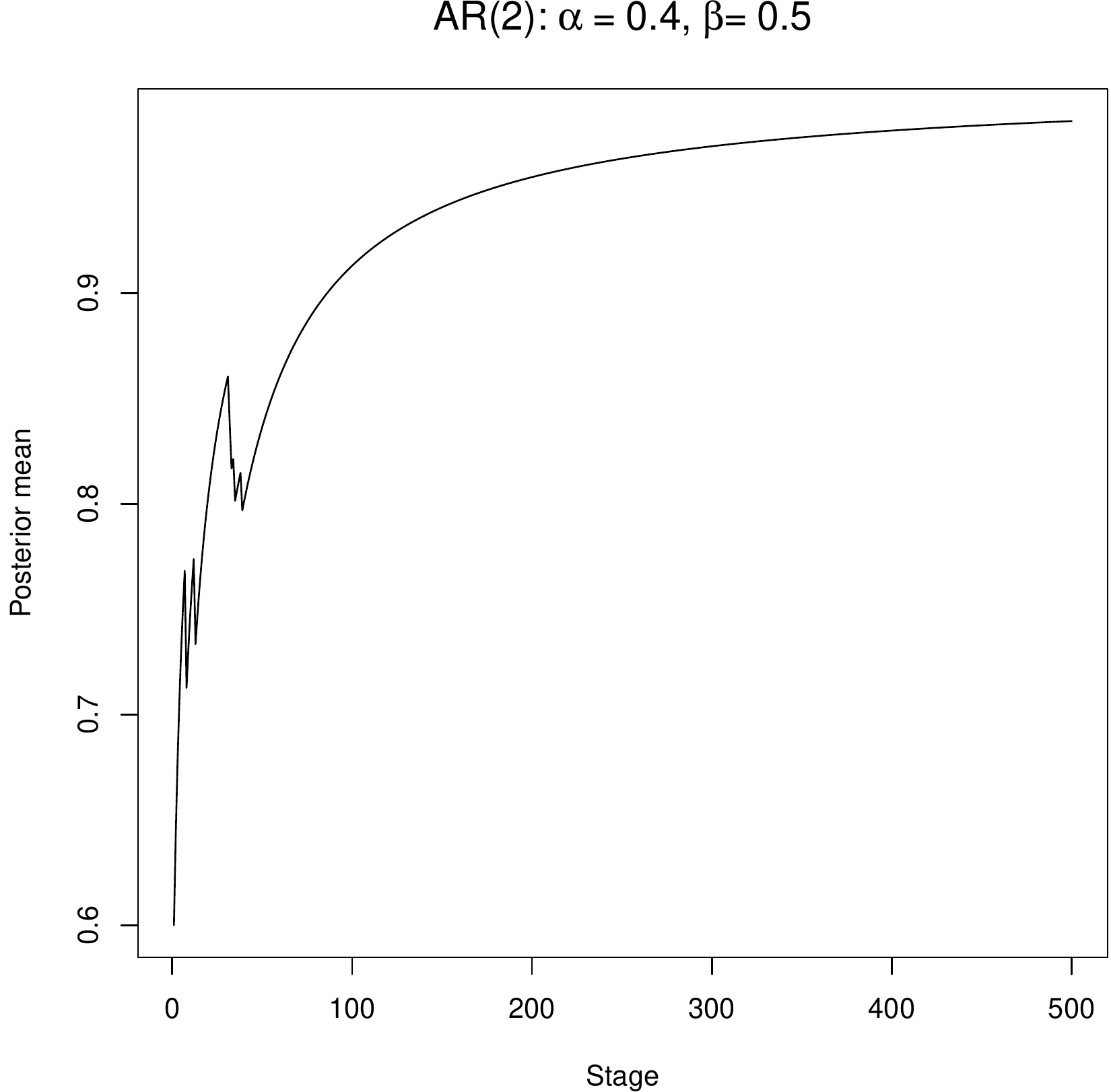}}\\
\vspace{2mm}
\subfigure [Stationary: $\alpha=0.5$, $\beta=0.4$.]{ \label{fig:ar2_5_4_short_nonpara2}
\includegraphics[width=4.5cm,height=4.5cm]{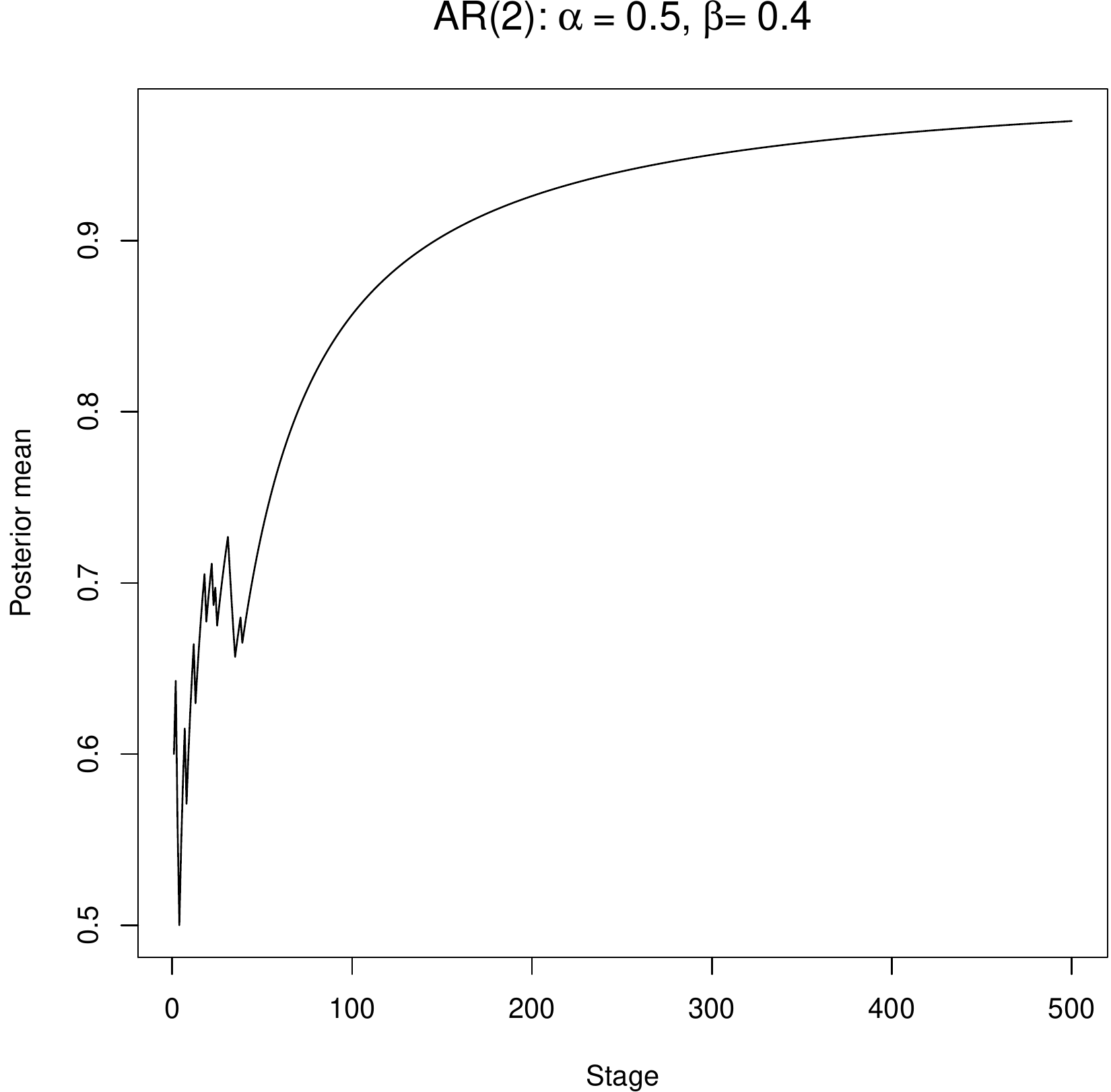}}
\hspace{2mm}
\subfigure [Nonstationary: $\alpha=0.5$, $\beta=0.9$.]{ \label{fig:ar2_5_9_short_nonpara2}
\includegraphics[width=4.5cm,height=4.5cm]{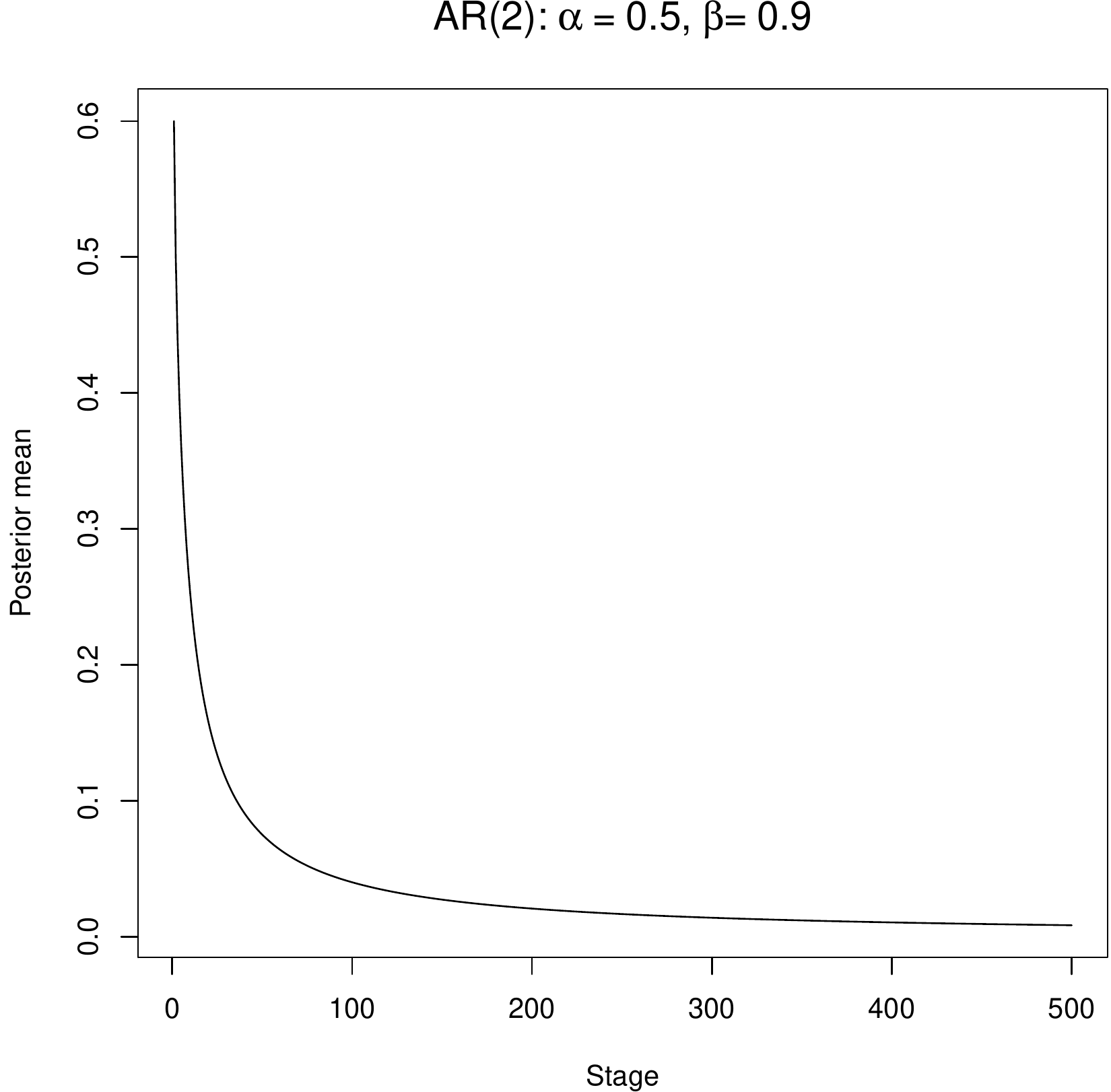}}
\hspace{2mm}
\subfigure [Nonstationary: $\alpha=0.6$, $\beta=0.6$.]{ \label{fig:ar2_6_6_short_nonpara2}
\includegraphics[width=4.5cm,height=4.5cm]{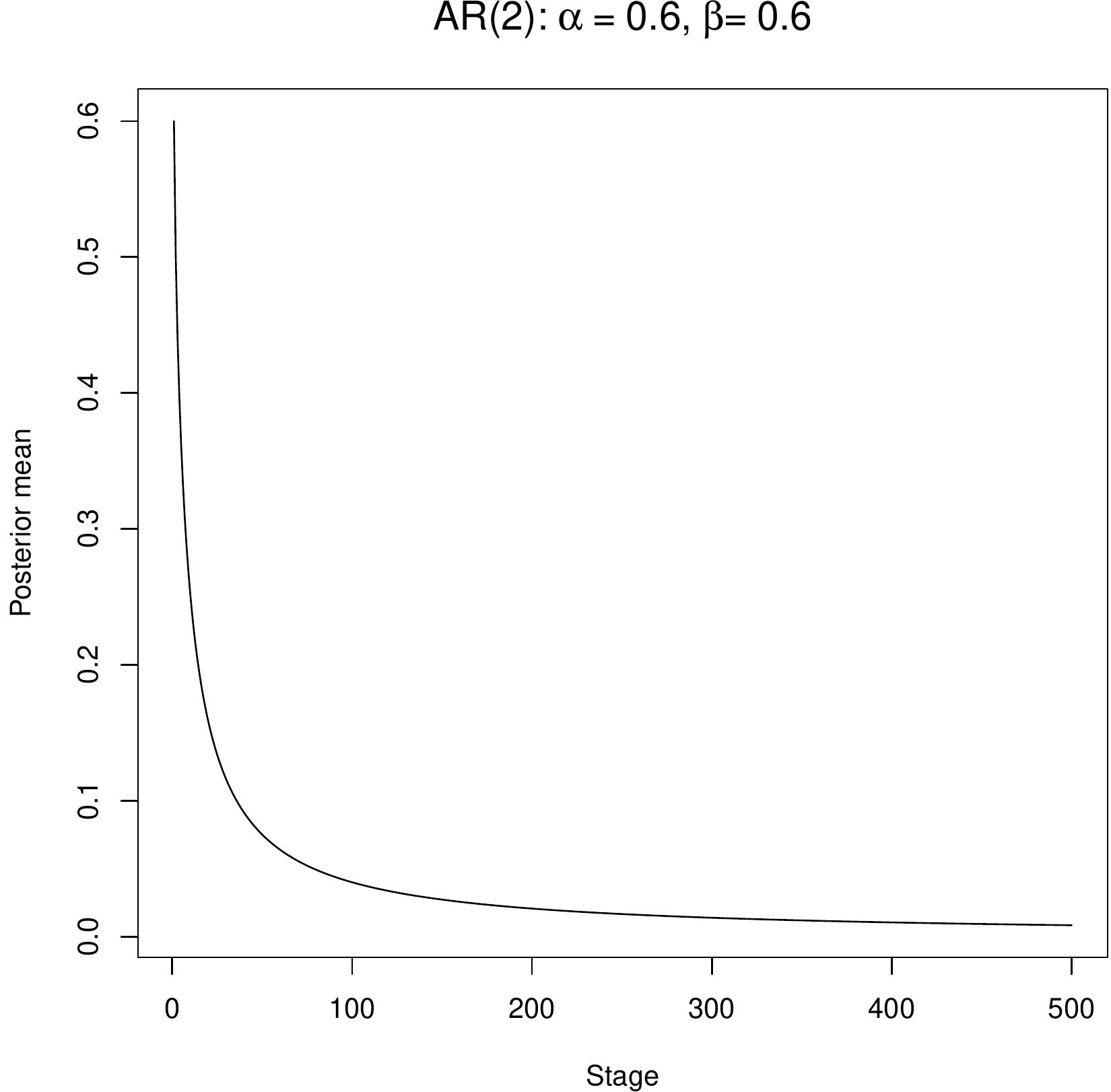}}\\
\vspace{2mm}
\subfigure [Nonstationary: $\alpha=0.5$, $\beta=0.5$.]{ \label{fig:ar2_5_5_short_nonpara2}
\includegraphics[width=4.5cm,height=4.5cm]{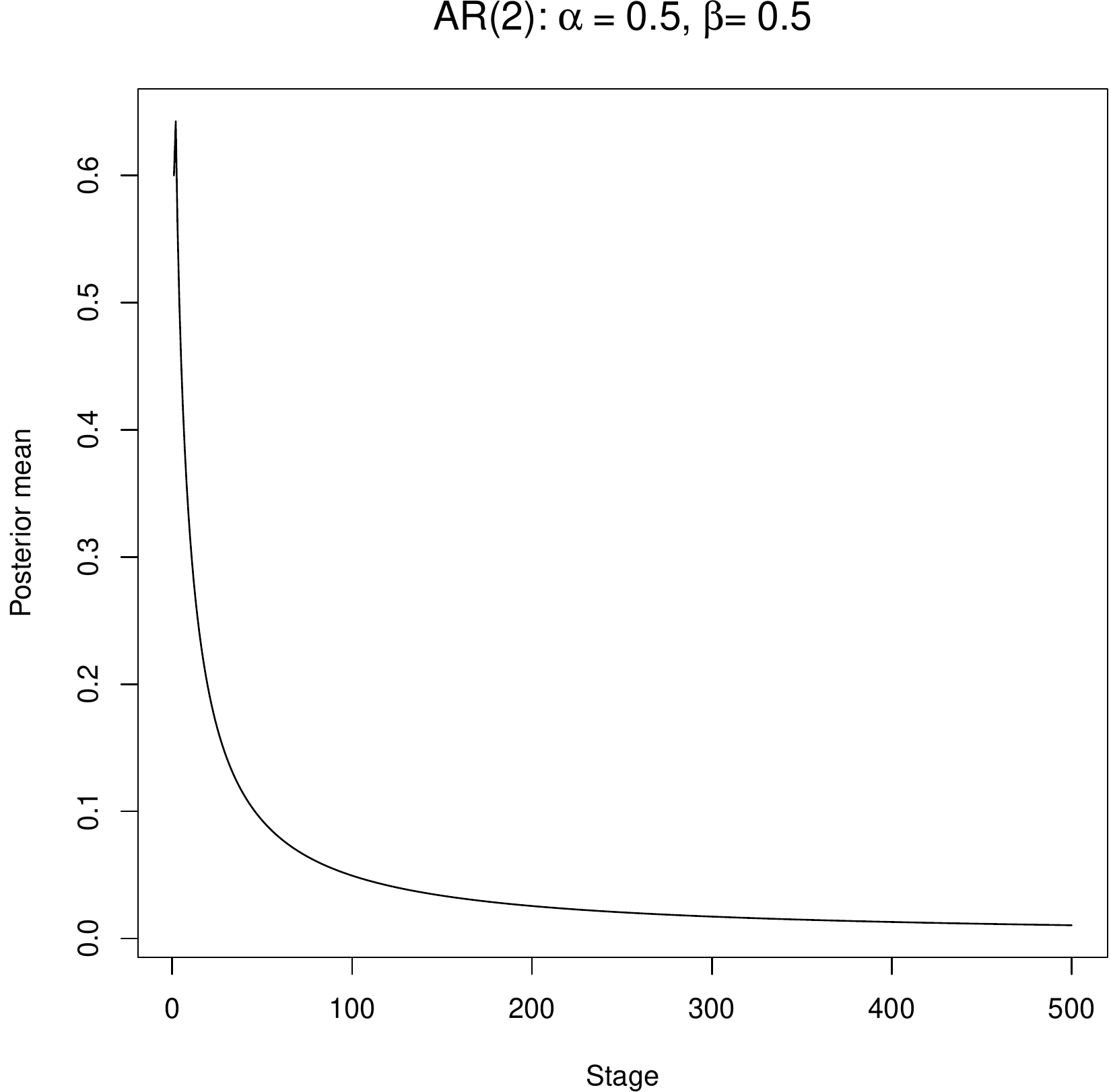}}
\hspace{2mm}
\subfigure [Nonstationary: $\alpha=0$, $\beta=1$.]{ \label{fig:ar2_0_1_short_nonpara2}
\includegraphics[width=4.5cm,height=4.5cm]{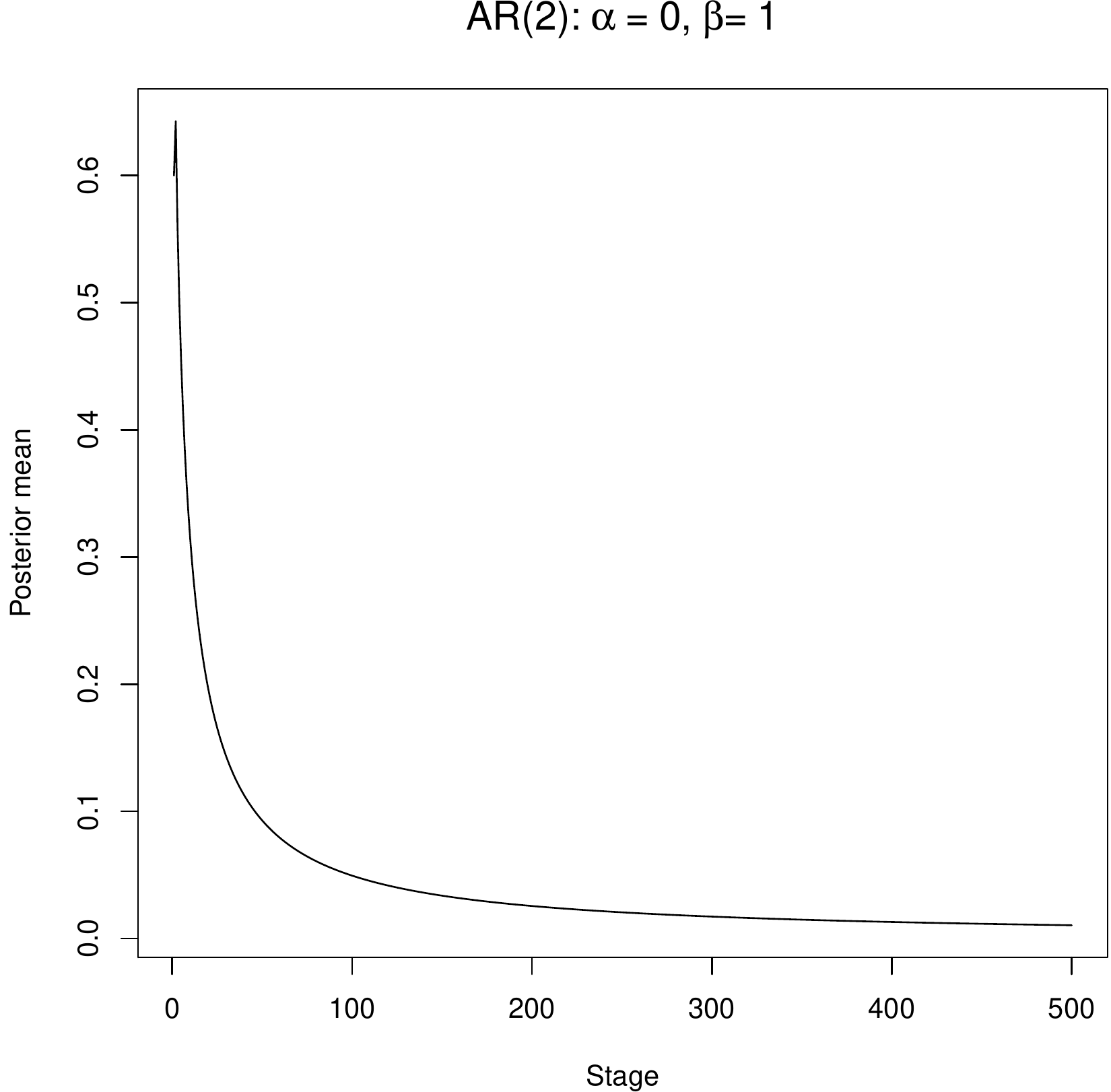}}
\hspace{2mm}
\subfigure [Nonstationary: $\alpha=1$, $\beta=0$.]{ \label{fig:ar2_1_0_short_nonpara2}
\includegraphics[width=4.5cm,height=4.5cm]{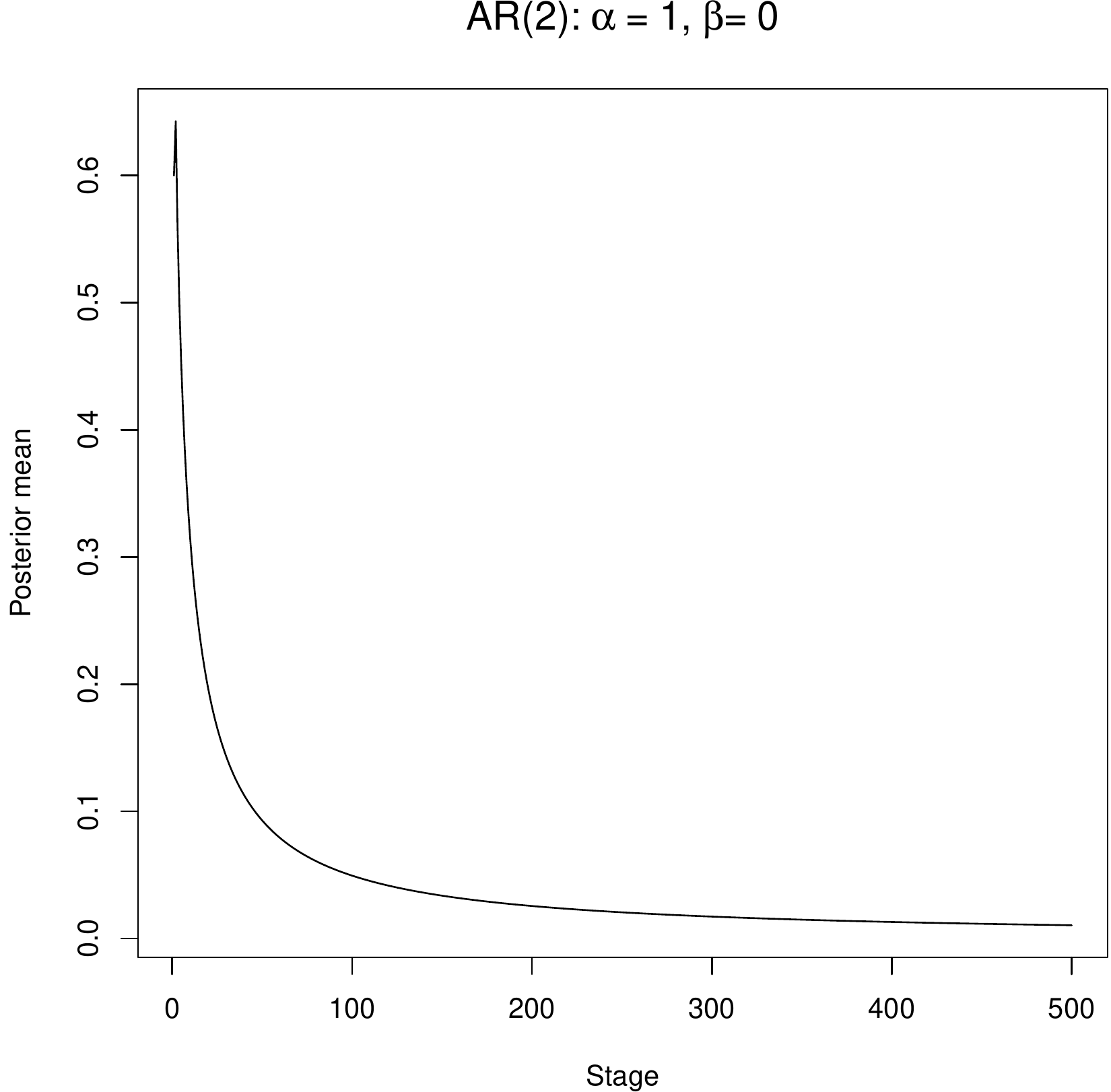}}
\caption{Nonparametric AR(2) example with $K=500$ and $n=5$.}
\label{fig:example_ar2}
\end{figure}

\subsection{Application to ARCH(1)}
\label{subsec:arch}
The ARCH models introduced by \ctn{Engle82} attempts to take into account the heteroscedasticity of financial time series, which is often
ignored by other popular financial models such as Black-Scholes (\ctn{Black73}) and the Ornstein-Uhlenbeck processi (\ctn{OU30}). In the ARCH(p) model,
the conditional variance is modeled as an autoregressive process of order $p$. For details on ARCH models, see \ctn{Bera93}, \ctn{Giraitis05}, \ctn{Straumann05}. 

The ARCH(1) model is of the following form: for $t=1,2,\ldots$,
\begin{align}
	x_t&=\epsilon_t\sigma_t\notag\\
	\sigma^2_t&=\omega+\alpha x^2_{t-1},
	\label{eq:arch}
\end{align}
where $\omega>0$, $\alpha\geq 0$ and $\epsilon_t\stackrel{iid}{\sim}N(0,1)$, for $t=1,2,\ldots$. The necessary and sufficient condition for stationarity of (\ref{eq:arch})
is $0<\alpha<1$. We set $\omega=1$ and $x_1=0$ for our purpose.

As in the AR(2) situations, here we considered $n=5$, $K=500$, and the bound (\ref{eq:ar1_bound3}) with $\hat C_1=1$. With these, Figure \ref{fig:example_arch}
provides the results of our Bayesian analyses of the realizations of (\ref{eq:arch}) for $\omega=1$ and various values of $\alpha$.
Although for $0<\alpha<1$, our method correctly identifies stationarity in all the cases, for $\alpha=1, 1.5, 2$, our procedure falsely declares nonstationarity as
stationarity. 

\begin{figure}
\centering
\subfigure [Stationary: $\alpha=0.5$.]{ \label{fig:arch_5_short_nonpara2}
\includegraphics[width=4.5cm,height=4.5cm]{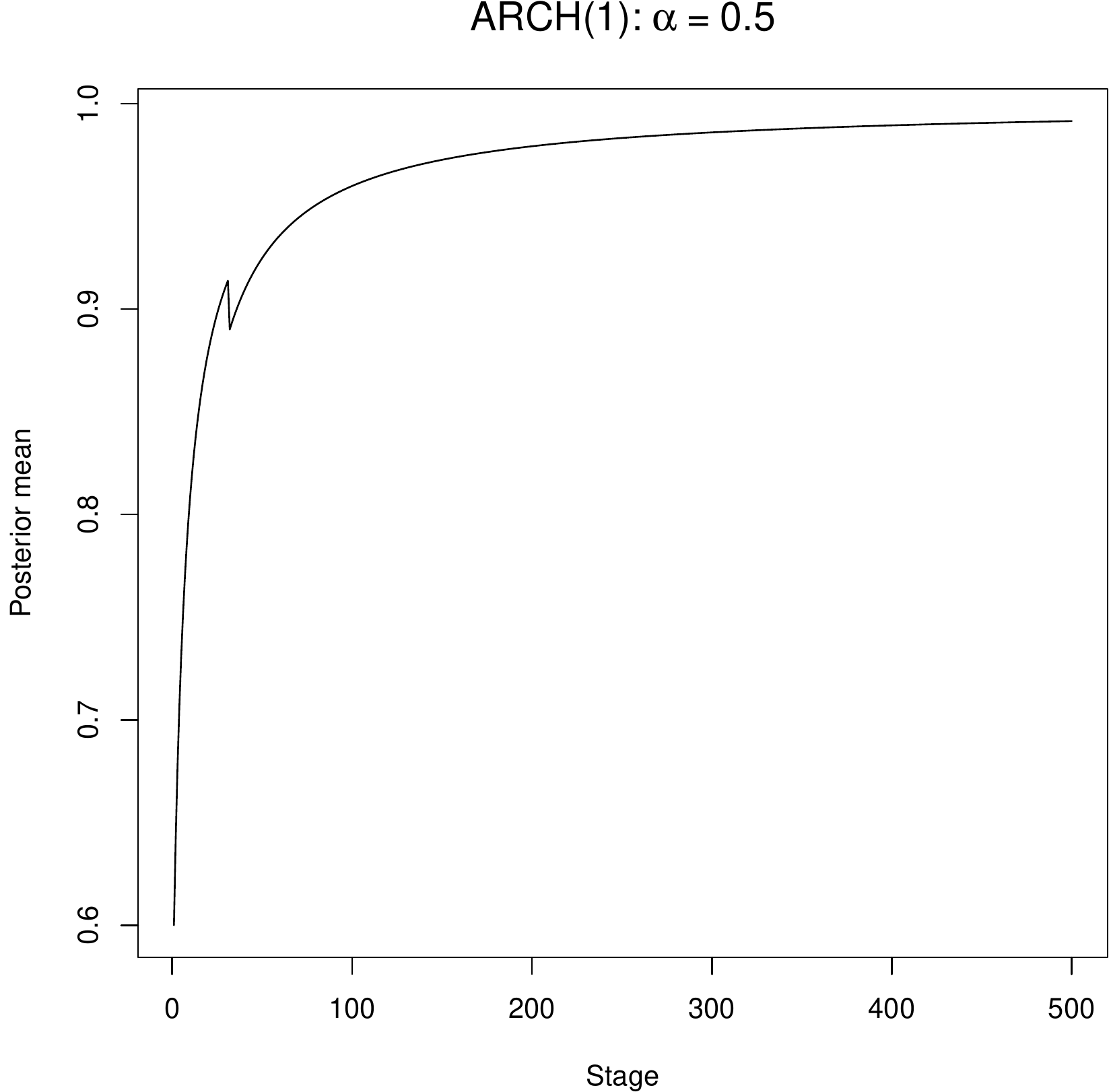}}
\hspace{2mm}
\subfigure [Stationary: $\alpha=0.9$.]{ \label{fig:arch_9_short_nonpara2}
\includegraphics[width=4.5cm,height=4.5cm]{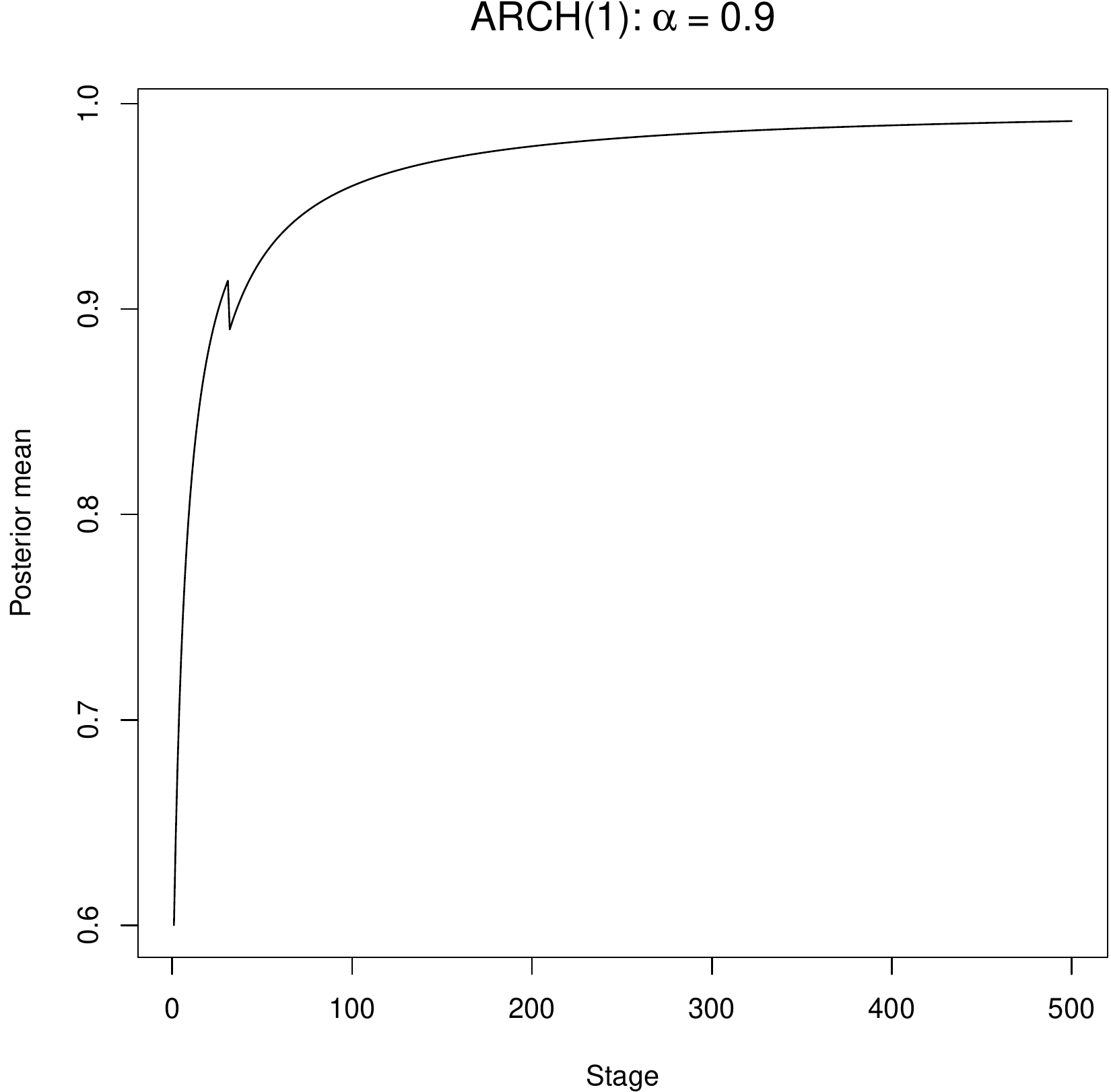}}
\hspace{2mm}
\subfigure [Stationary: $\alpha=99$.]{ \label{fig:arch_99_short_nonpara2}
\includegraphics[width=4.5cm,height=4.5cm]{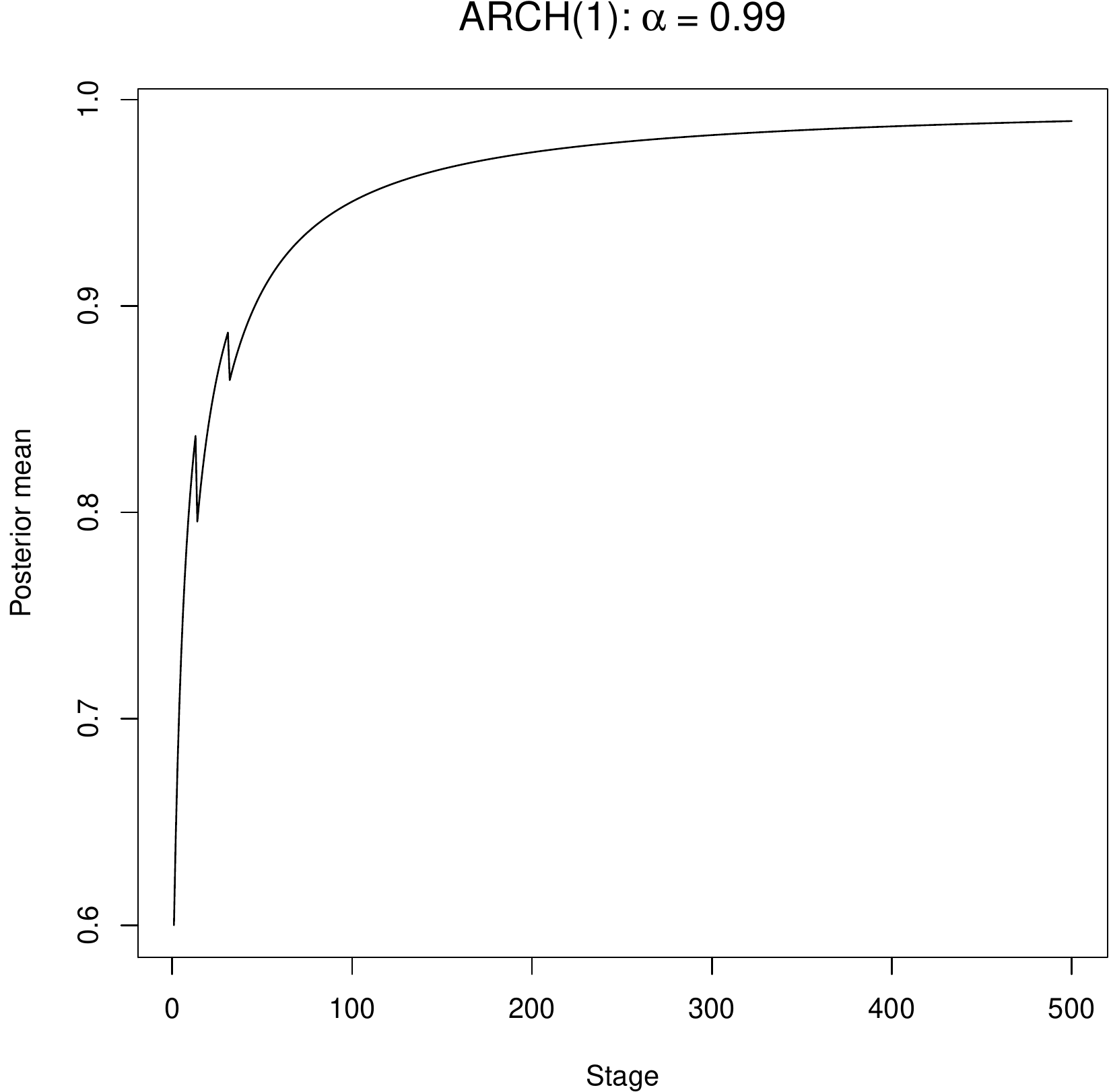}}\\
\vspace{2mm}
\subfigure [Stationary: $\alpha=0.999$.]{ \label{fig:arch_999_short_nonpara2}
\includegraphics[width=4.5cm,height=4.5cm]{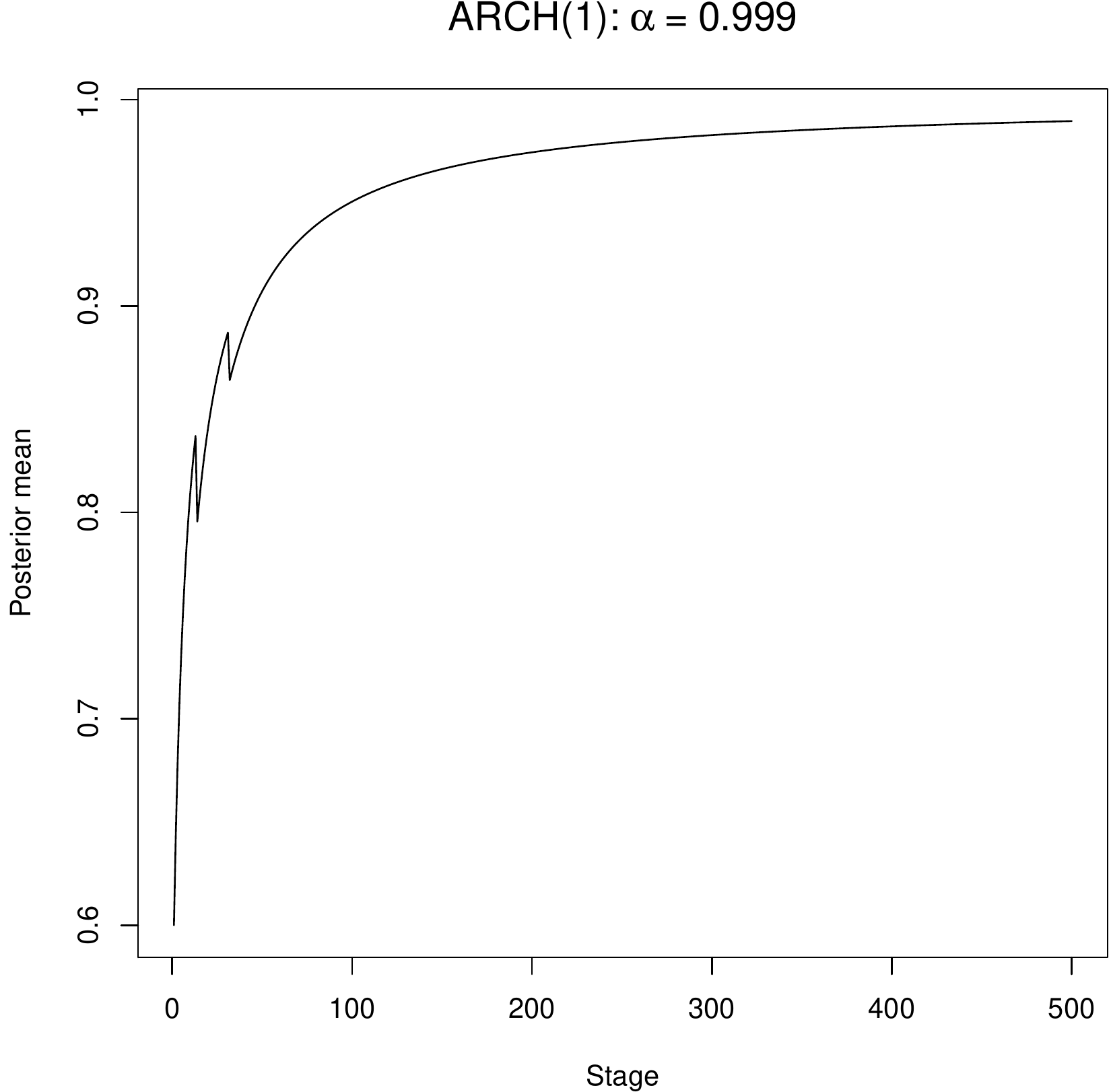}}
\hspace{2mm}
\subfigure [Stationary: $\alpha=0.9999$.]{ \label{fig:arch_9999_short_nonpara2}
\includegraphics[width=4.5cm,height=4.5cm]{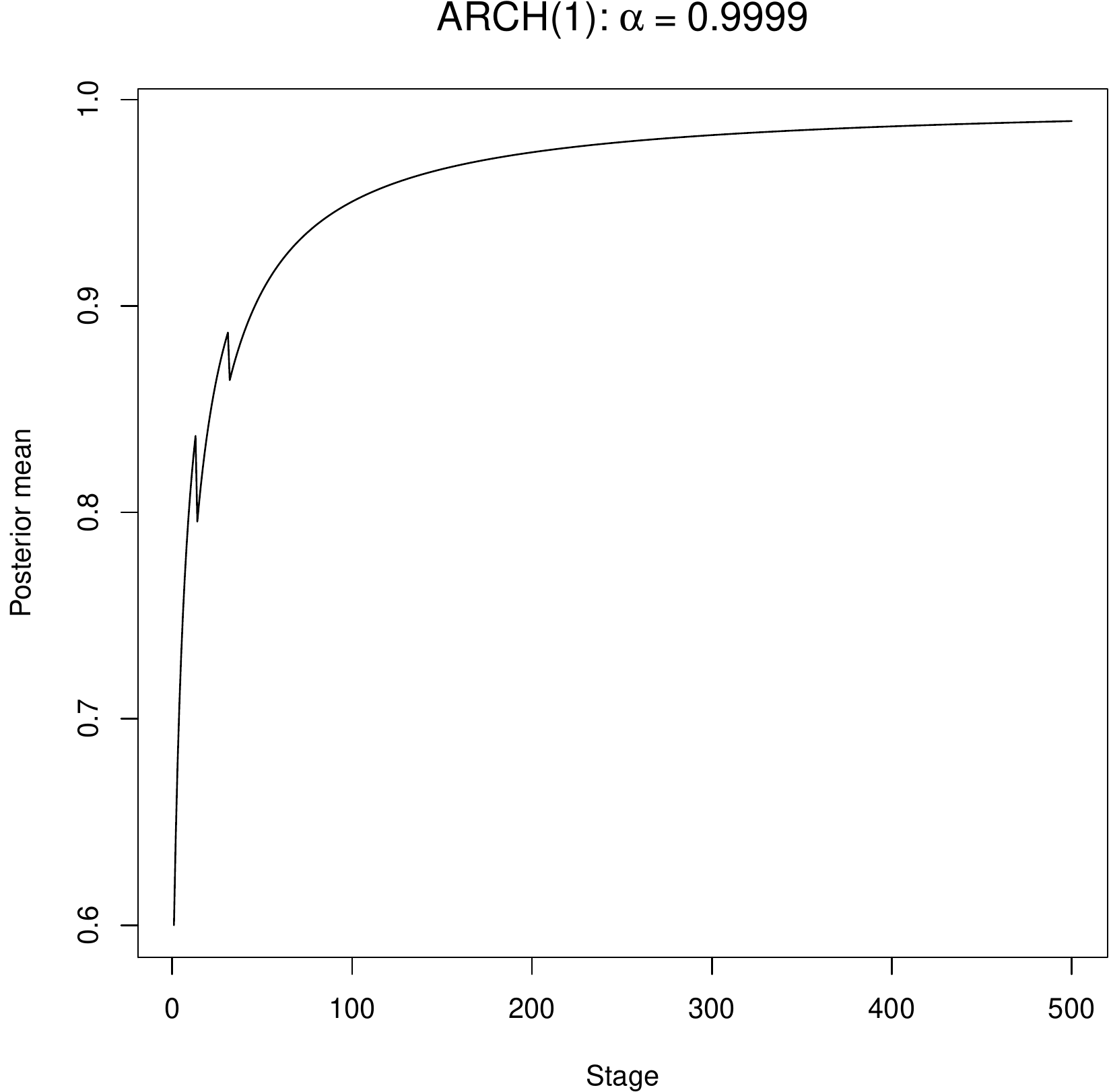}}
\hspace{2mm}
\subfigure [Stationary: $\alpha=0.99999$.]{ \label{fig:arch_99999_short_nonpara2}
\includegraphics[width=4.5cm,height=4.5cm]{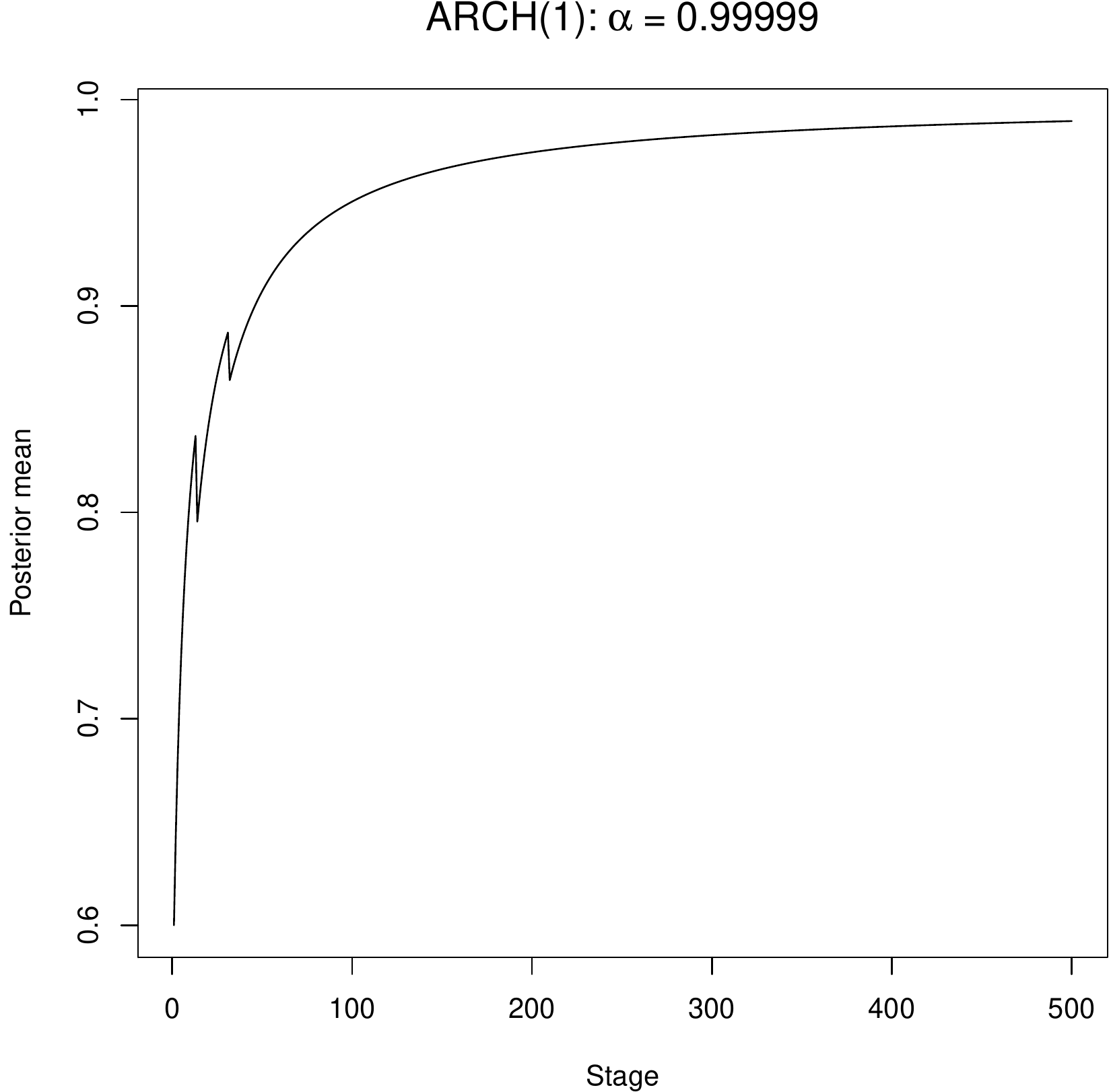}}\\
\vspace{2mm}
\subfigure [Nonstationary: $\alpha=1$.]{ \label{fig:arch_1_short_nonpara2}
\includegraphics[width=4.5cm,height=4.5cm]{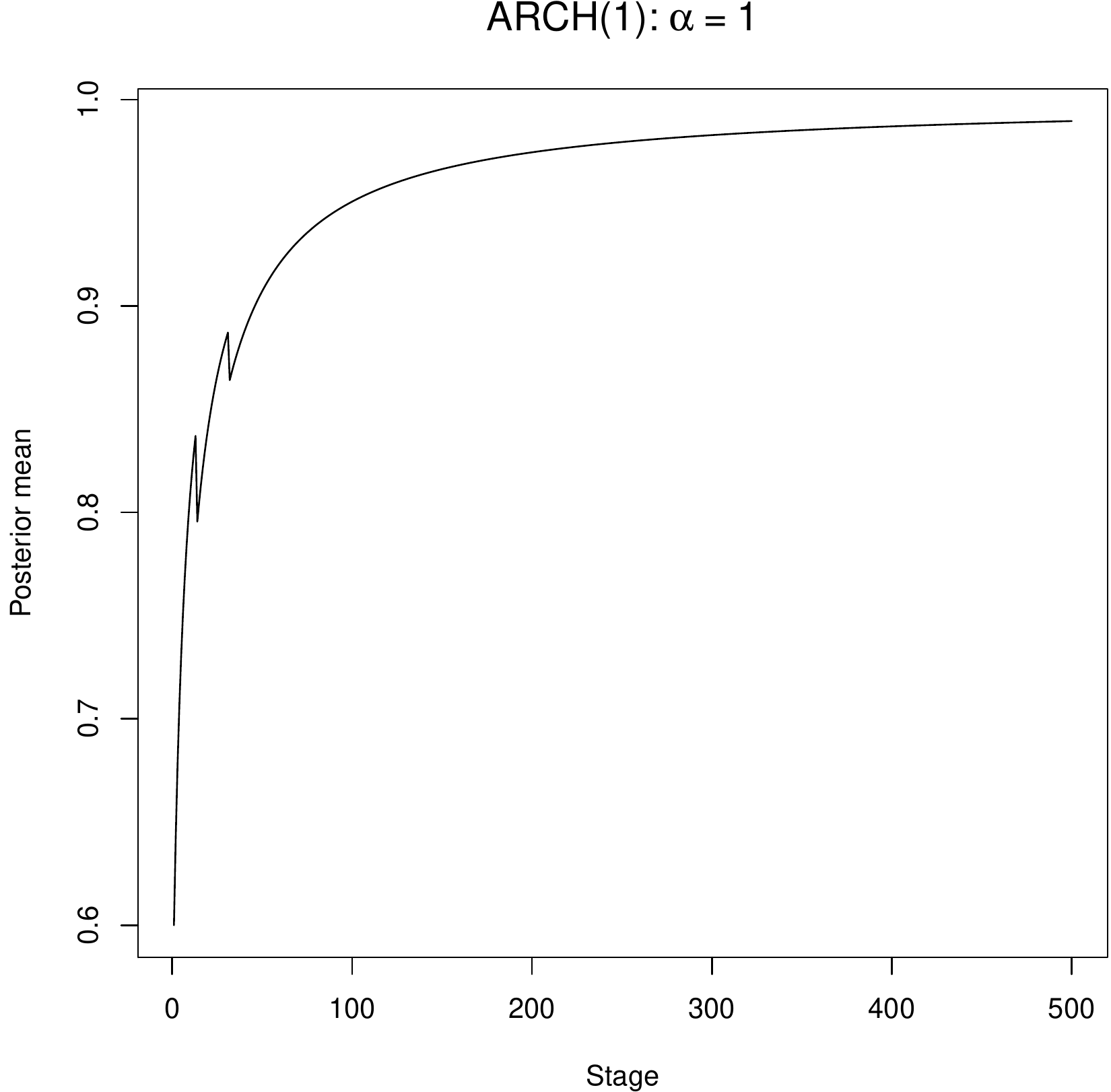}}
\hspace{2mm}
\subfigure [Nonstationary: $\alpha=1.5$.]{ \label{fig:arch_15_short_nonpara2}
\includegraphics[width=4.5cm,height=4.5cm]{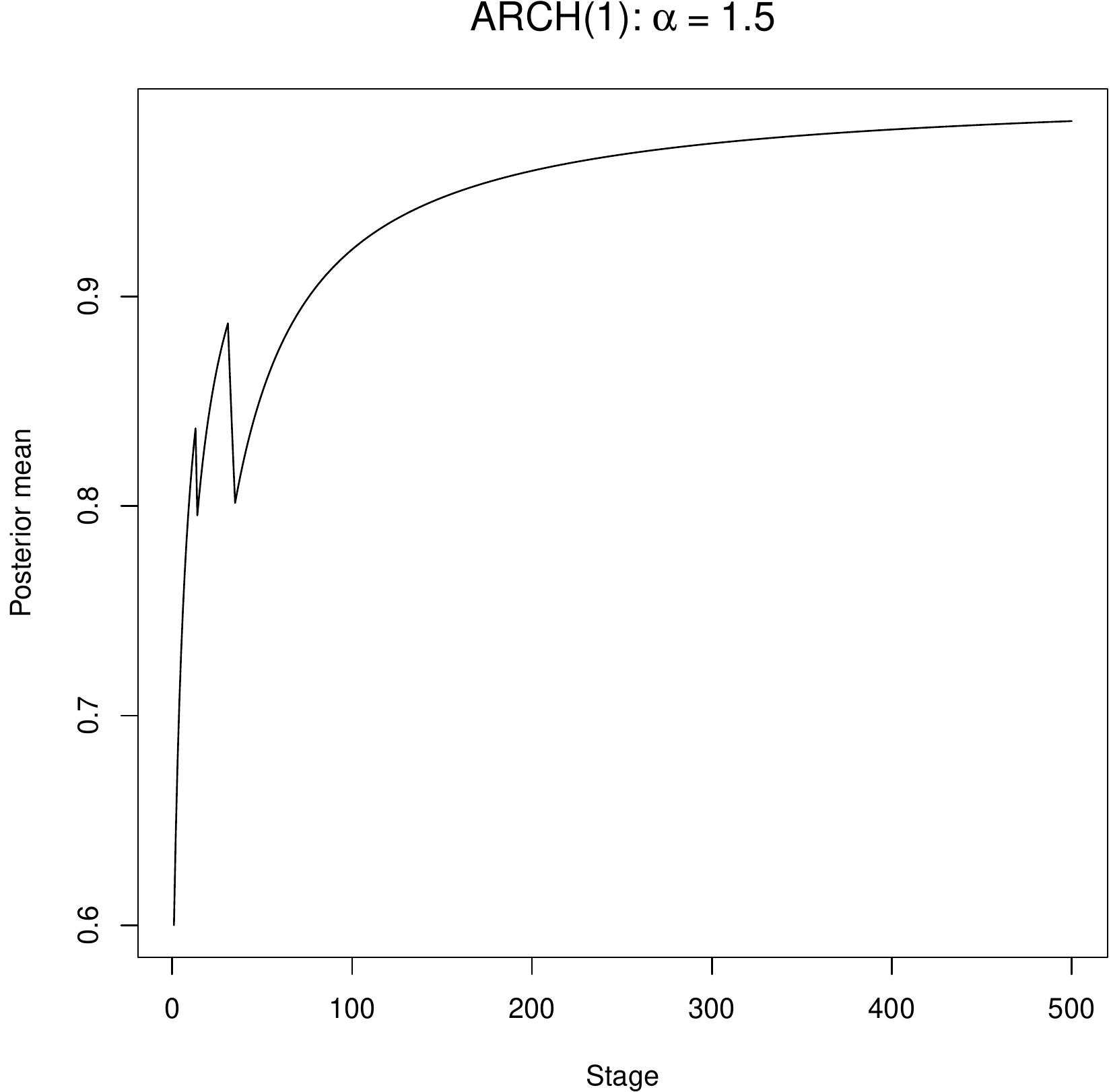}}
\hspace{2mm}
\subfigure [Nonstationary: $\alpha=2$.]{ \label{fig:arch_2_short_nonpara2}
\includegraphics[width=4.5cm,height=4.5cm]{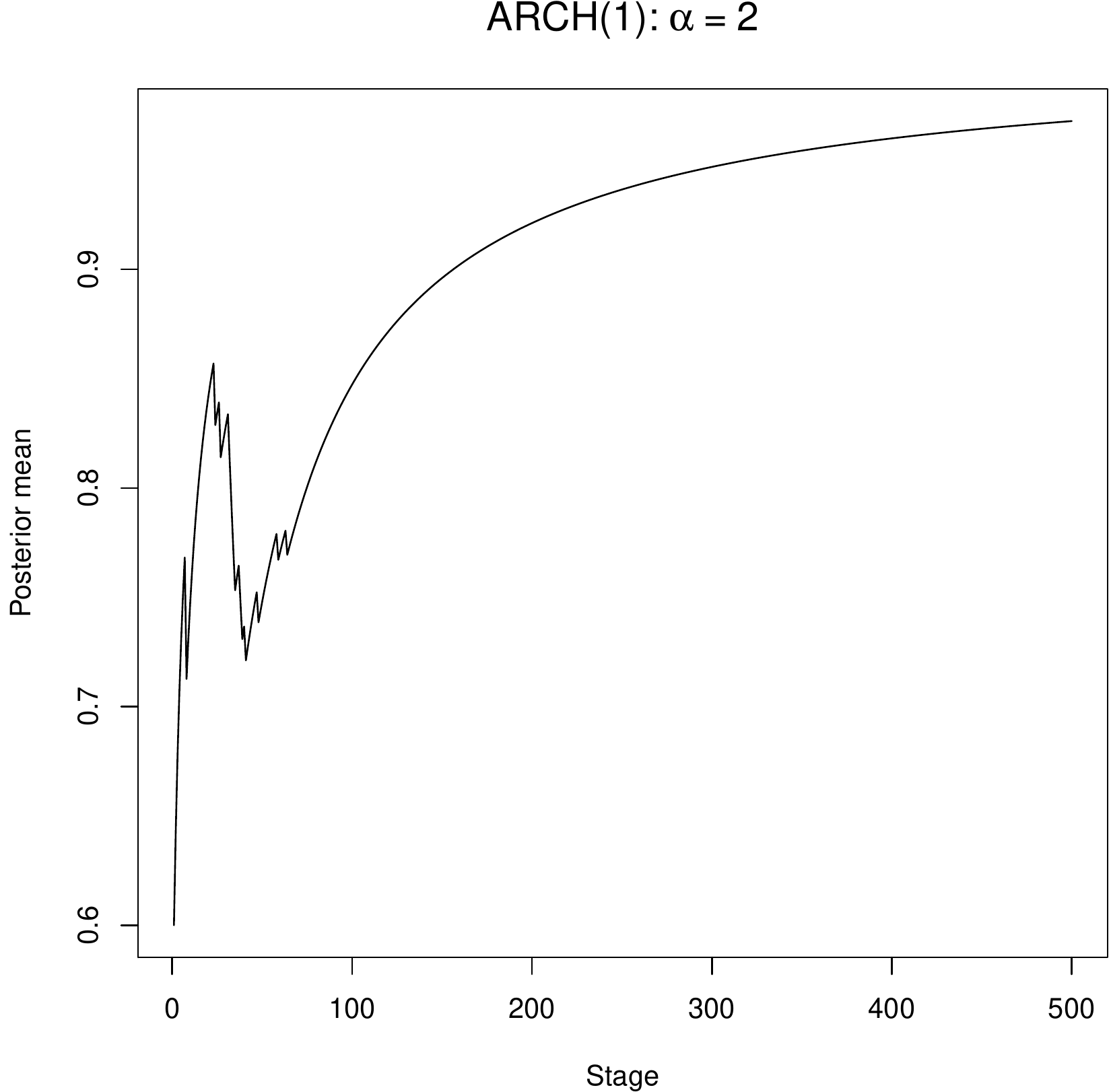}}
\caption{Nonparametric ARCH(1) example with $K=500$ and $n=5$.}
\label{fig:example_arch}
\end{figure}

To understand the reason for this, it is necessary to recall some of the properties of the ARCH(1) model. Note that $E(x_t)=0$ for $t\geq 1$
and for any $t\geq 1$, $Var(x_t)=\frac{\omega}{1-\alpha}$, provided $0<\alpha<1$. For $\alpha\geq 1$, $Var(x_t)$ increases with $t$. Moreover,
$Cov(x_t,x_{t+j})=0$ for $j\geq 1$. The last fact shows that the ARCH(1) model is serially uncorrelated. Thus, even though for $\alpha\geq 1$, 
$Var(x_t)$ increases with $t$, the realizations will be centered around zero and will be serially uncorrelated, and these are instrumental in rendering
the pattern of the realizations seem like stationary time series. Although the variances are increasing in such cases, the realizations need not have 
an increasing range pattern due to absence of serial correlation. Figure \ref{fig:failed_Bayesian_arch} shows ARCH(1) realizations for $\alpha=0.9$, $1$, $1.5$ and $2$. 
Note that none of the realizations exhibit any trend of increasing range, even though only $\alpha=0.9$ corresponds to
stationarity. Moreover, the pattern of the nonstationary realization for $\alpha=1$ is quite similar to that of the stationary realization $\alpha=0.9$.
Indeed, all the four realizations shown in Figure \ref{fig:failed_Bayesian_arch} have similar patterns; they essentially differ only at a few time points,
where the realizations have different ranges. 

In other words, the realizations for $\alpha=1$, $1.5$ and $2$ shown in Figure \ref{fig:failed_Bayesian_arch} do not seem to have enough information 
to distinguish them from stationarity. Hence, it is not surprising that our Bayesian method declared these realizations as stationary.
\begin{figure}
\centering
\subfigure [Stationary: $\alpha=0.9$.]{ \label{fig:archplot_9_short_nonpara2}
\includegraphics[width=4.5cm,height=4.5cm]{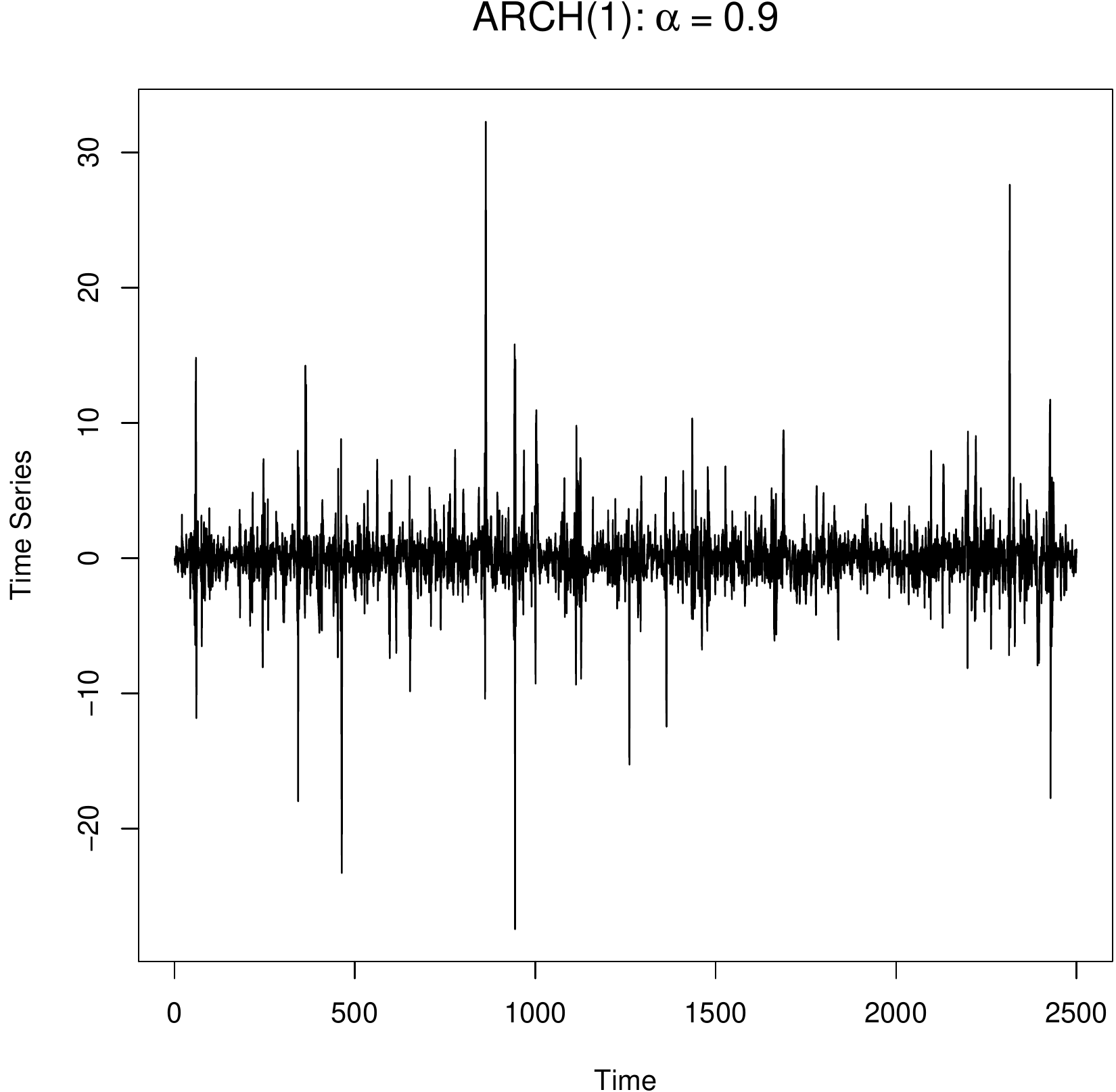}}
\hspace{2mm}
\subfigure [Nonstationary: $\alpha=1$.]{ \label{fig:archplot_1_short_nonpara2}
\includegraphics[width=4.5cm,height=4.5cm]{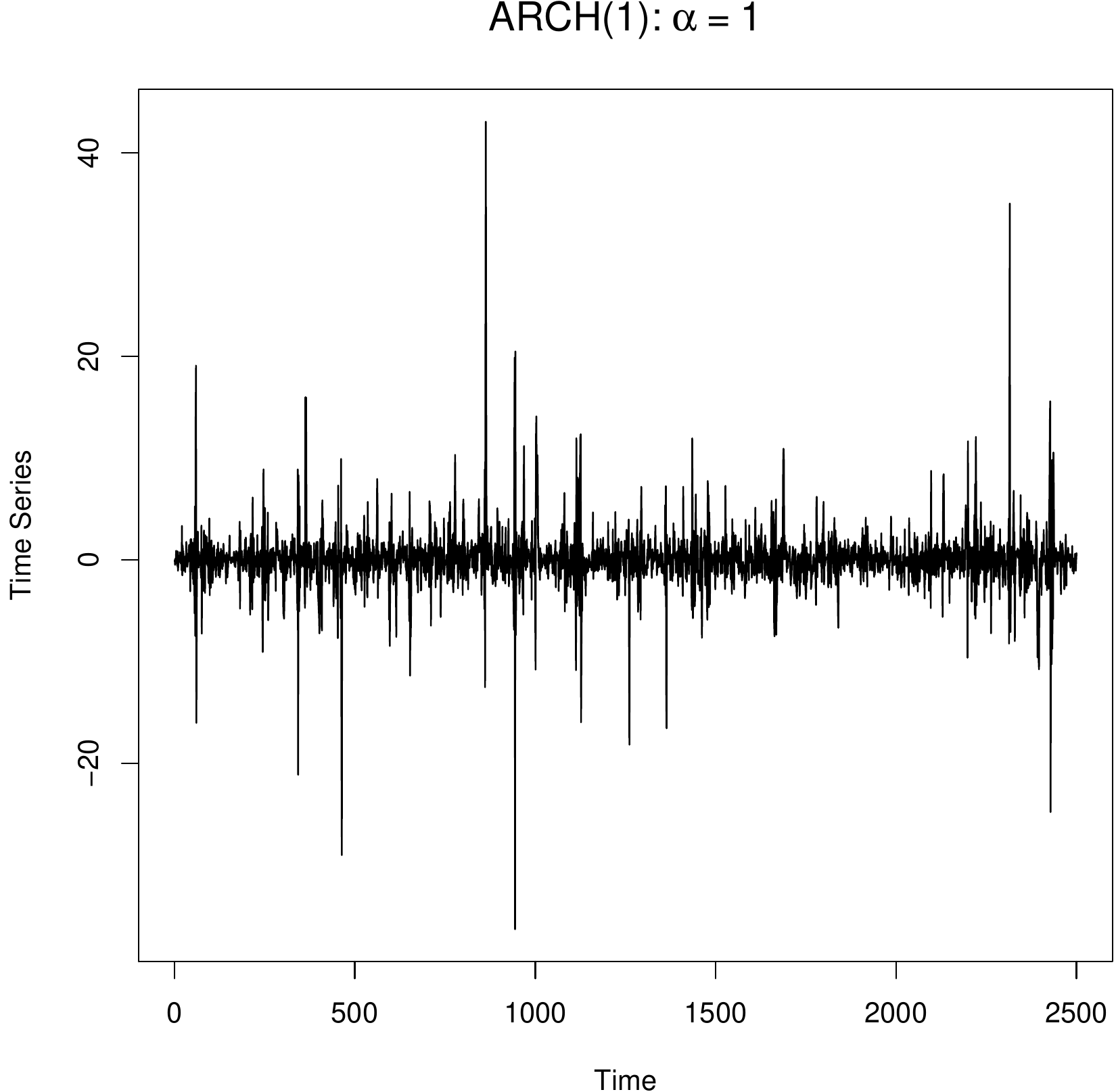}}\\
\vspace{2mm}
\subfigure [Nonstationary: $\alpha=1.5$.]{ \label{fig:archplot_15_short_nonpara2}
\includegraphics[width=4.5cm,height=4.5cm]{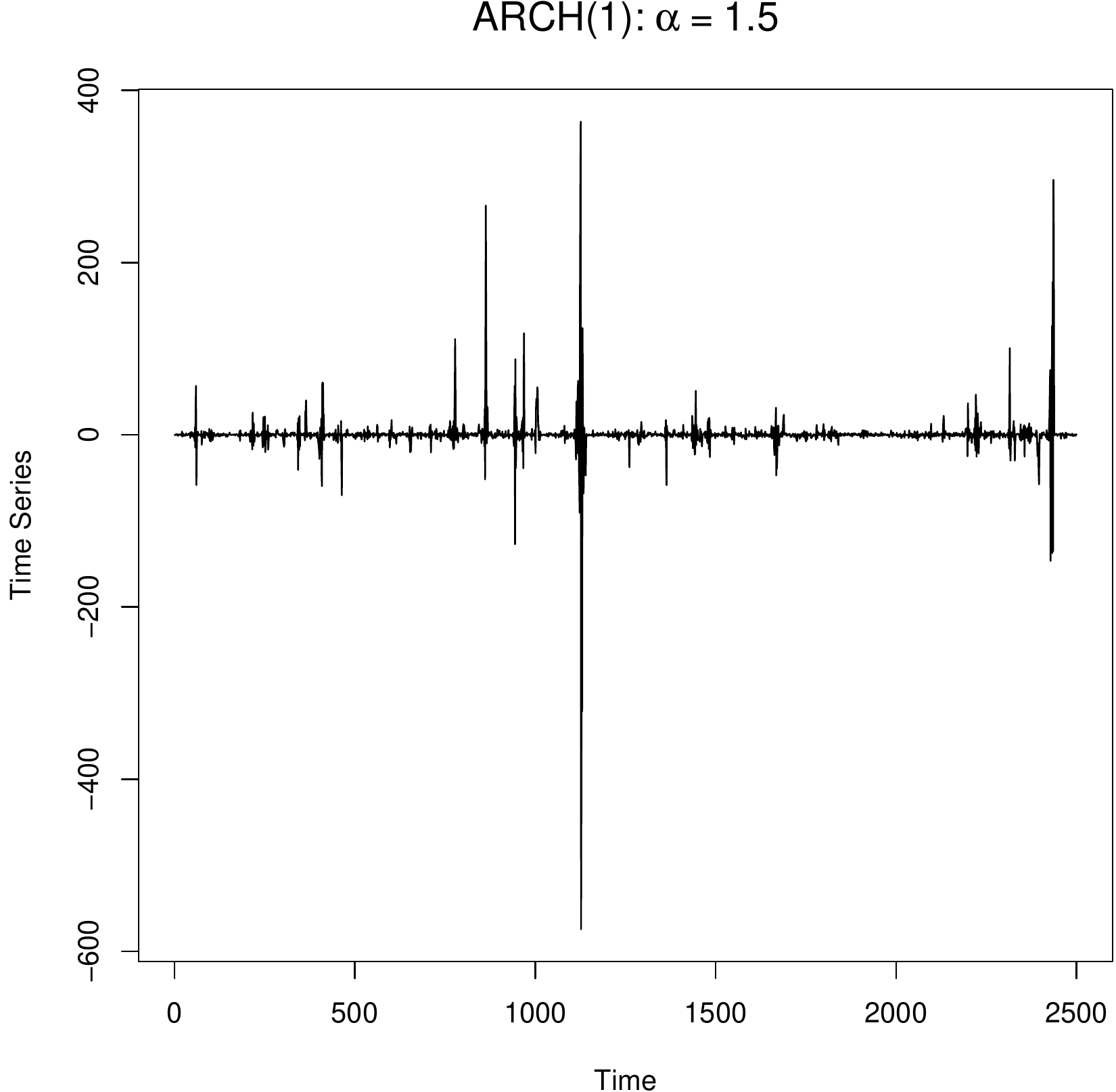}}
\hspace{2mm}
\subfigure [Nonstationary: $\alpha=2$.]{ \label{fig:archplot_2_short_nonpara2}
\includegraphics[width=4.5cm,height=4.5cm]{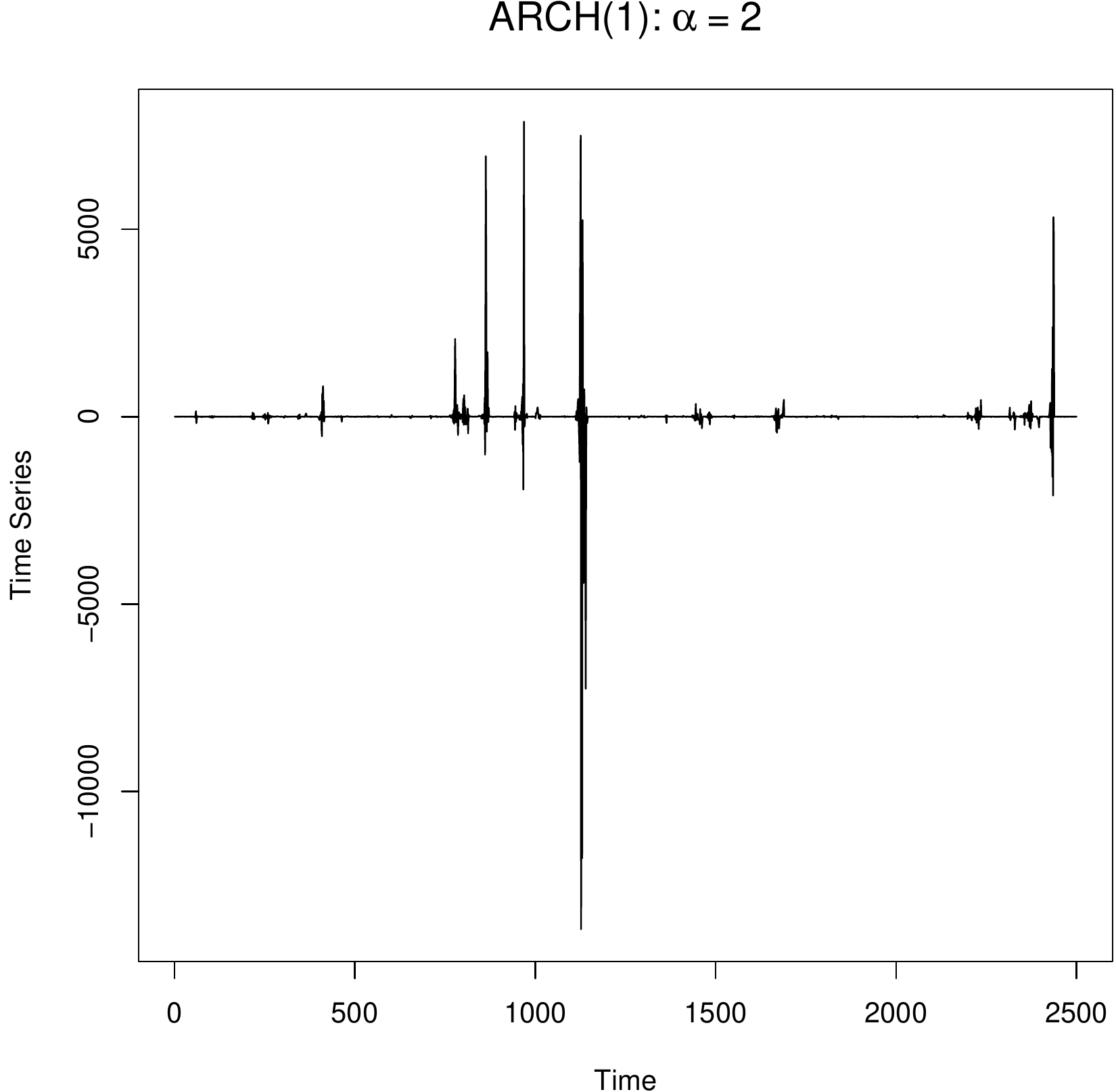}}\\
\caption{Comparison of ARCH(1) samples for several values of $\alpha$ where our Bayesian method failed.}
\label{fig:failed_Bayesian_arch}
\end{figure}

\subsection{Application to GARCH(1,1)}
\label{subsec:garch}

The ARCH model has been generalized by \ctn{Bollerslev86} and \ctn{Taylor86} independently 
to let $\sigma^2_t$ to have an autoregressive structure as well. This generalized ARCH, or GARCH model, 
is arguably the most widely used model in financial time series, particularly, for modeling stochastic volatility.
For details on GARCH, see \ctn{Bougerol92}, \ctn{Giraitis05}, \ctn{Berkes03} and \ctn{Straumann05}.

The GARCH(1,1) model, which generalizes ARCH(1), is of the following form: for $t=1,2,\ldots$,
\begin{align}
	x_t&=\epsilon_t\sigma_t\notag\\
	\sigma^2_t&=\omega+\alpha x^2_{t-1}+\beta\sigma^2_{t-1},
	\label{eq:garch}
\end{align}
where $\omega>0$, $\alpha\geq 0$, $\beta\geq 0$ and $\epsilon_t\stackrel{iid}{\sim}N(0,1)$, for $t=1,2,\ldots$. The necessary and sufficient condition 
for stationarity of (\ref{eq:garch})
is $0<\alpha+\beta<1$. We set $\omega=1$ and $x_1=0$ and $\sigma_1=0$ for our purpose.

Again we set $n=5$ and $K=500$ and consider the nonparametric bound (\ref{eq:ar1_bound3}) for applying our Bayesian idea to model (\ref{eq:garch}) for 
different values of $\alpha$ and $\beta$ leading to stationarity and nonstationarity. Figure \ref{fig:example_garch}, summarizing the results of our
Bayesian experiments, show that all the cases have been correctly identified, except the cases of $(\alpha=1,\beta=0)$ and $(\alpha=0.5,\beta=0.5)$.
Note that the first case is the same as ARCH(1) with $\alpha=1$, and the reason for failure of our Bayesian method for this case has already been explained
in Section \ref{subsec:arch}.

\begin{figure}
\centering
\subfigure [Stationary: $\alpha=0.3$, $\beta=0.4$.]{ \label{fig:garch_3_4_short_nonpara2}
\includegraphics[width=4.5cm,height=4.5cm]{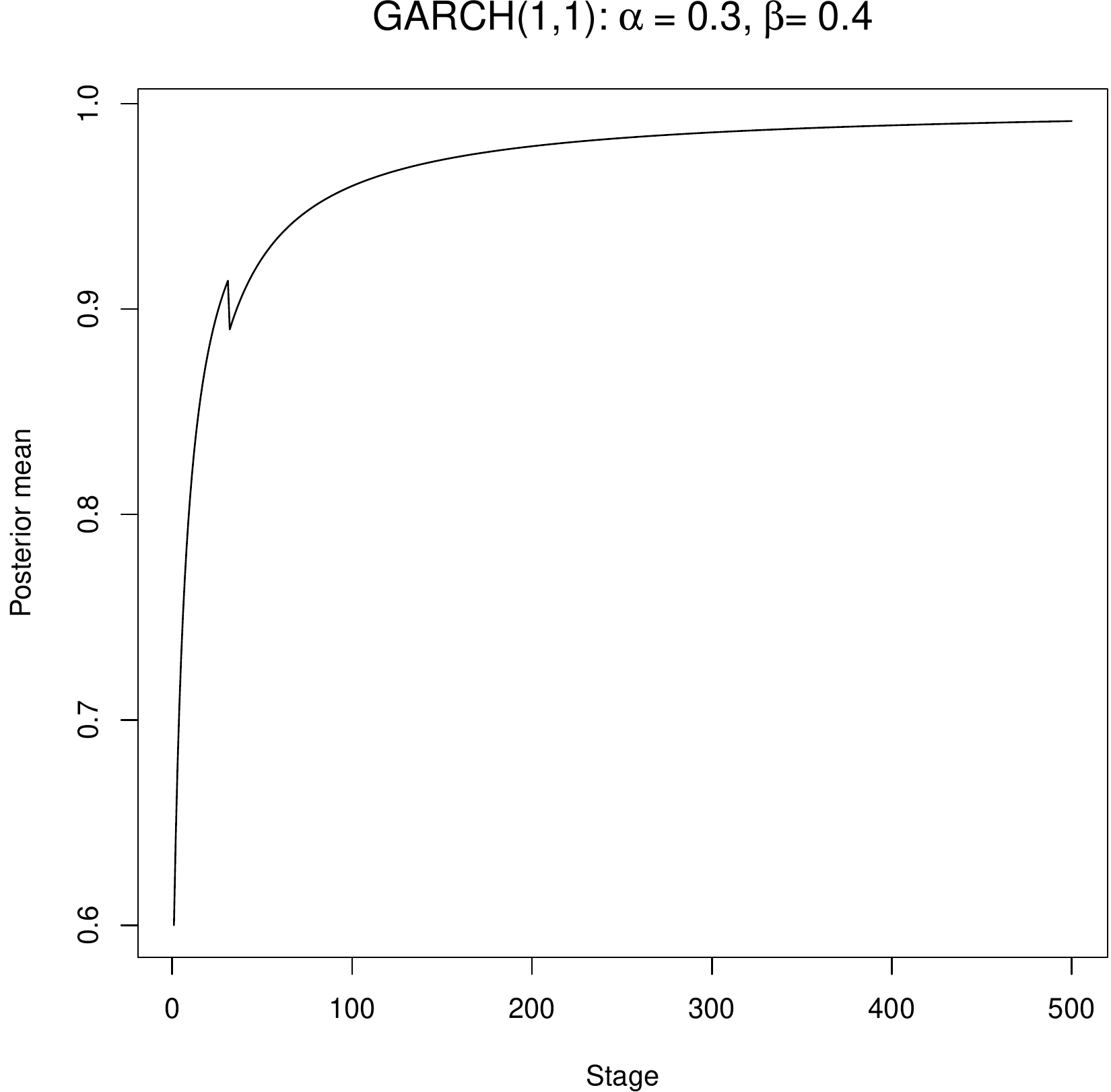}}
\hspace{2mm}
\subfigure [Stationary: $\alpha=0.4$, $\beta=0.3$.]{ \label{fig:garch_4_3_short_nonpara2}
\includegraphics[width=4.5cm,height=4.5cm]{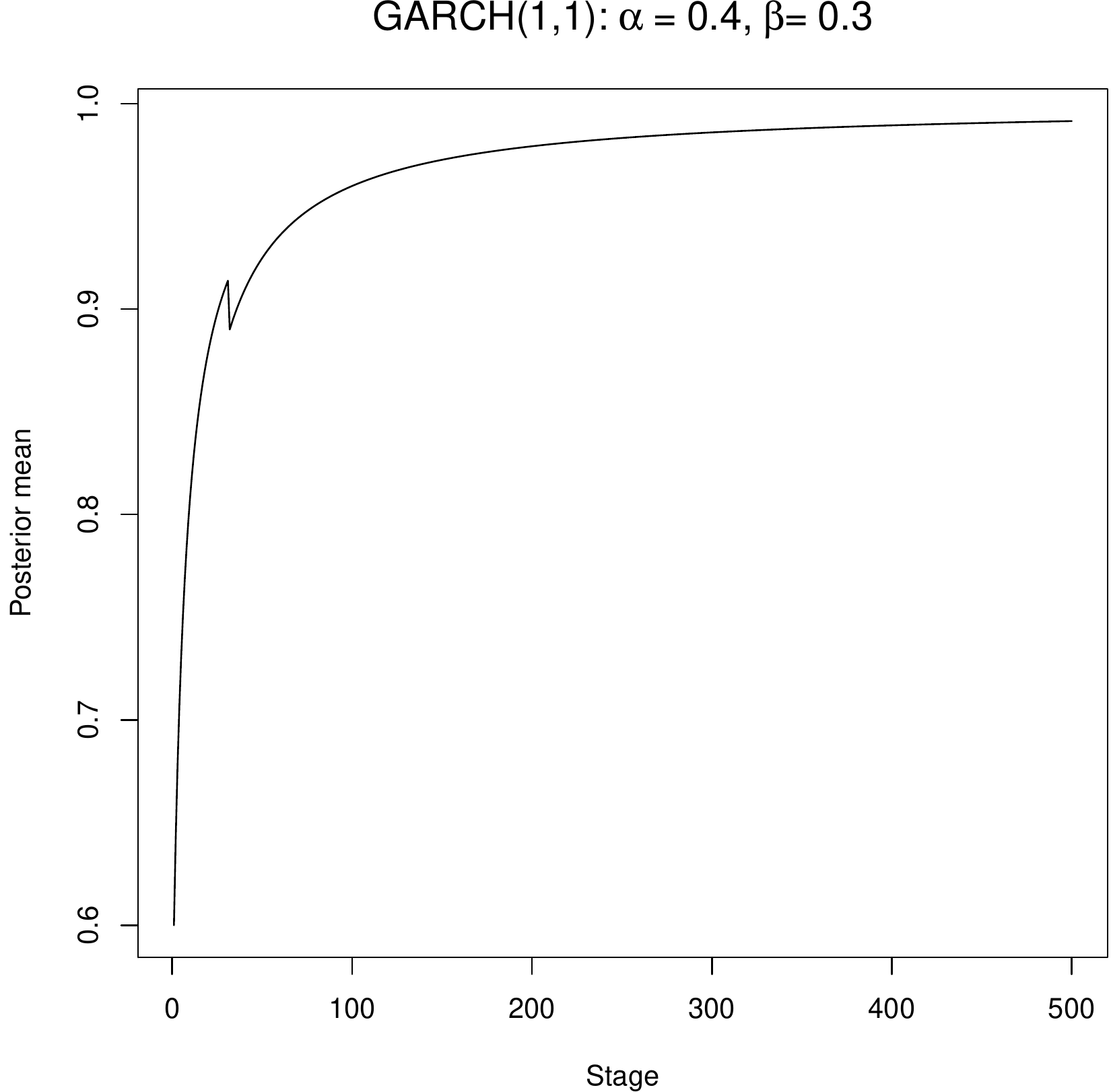}}
\hspace{2mm}
\subfigure [Stationary: $\alpha=0.4$, $\beta=0.5$.]{ \label{fig:garch_4_5_short_nonpara2}
\includegraphics[width=4.5cm,height=4.5cm]{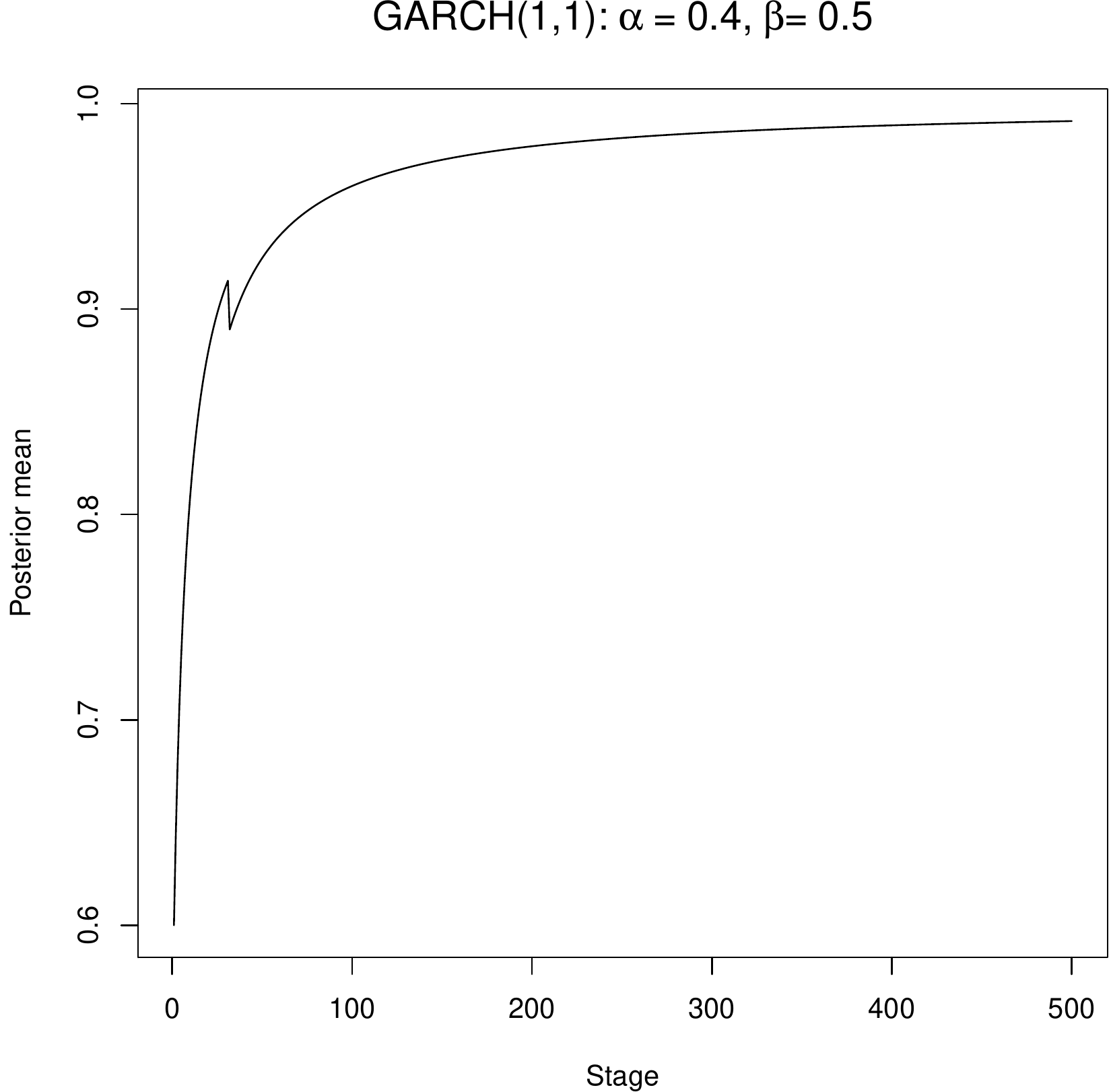}}\\
\vspace{2mm}
\subfigure [Stationary: $\alpha=0.5$, $\beta=0.4$.]{ \label{fig:garch_5_4_short_nonpara2}
\includegraphics[width=4.5cm,height=4.5cm]{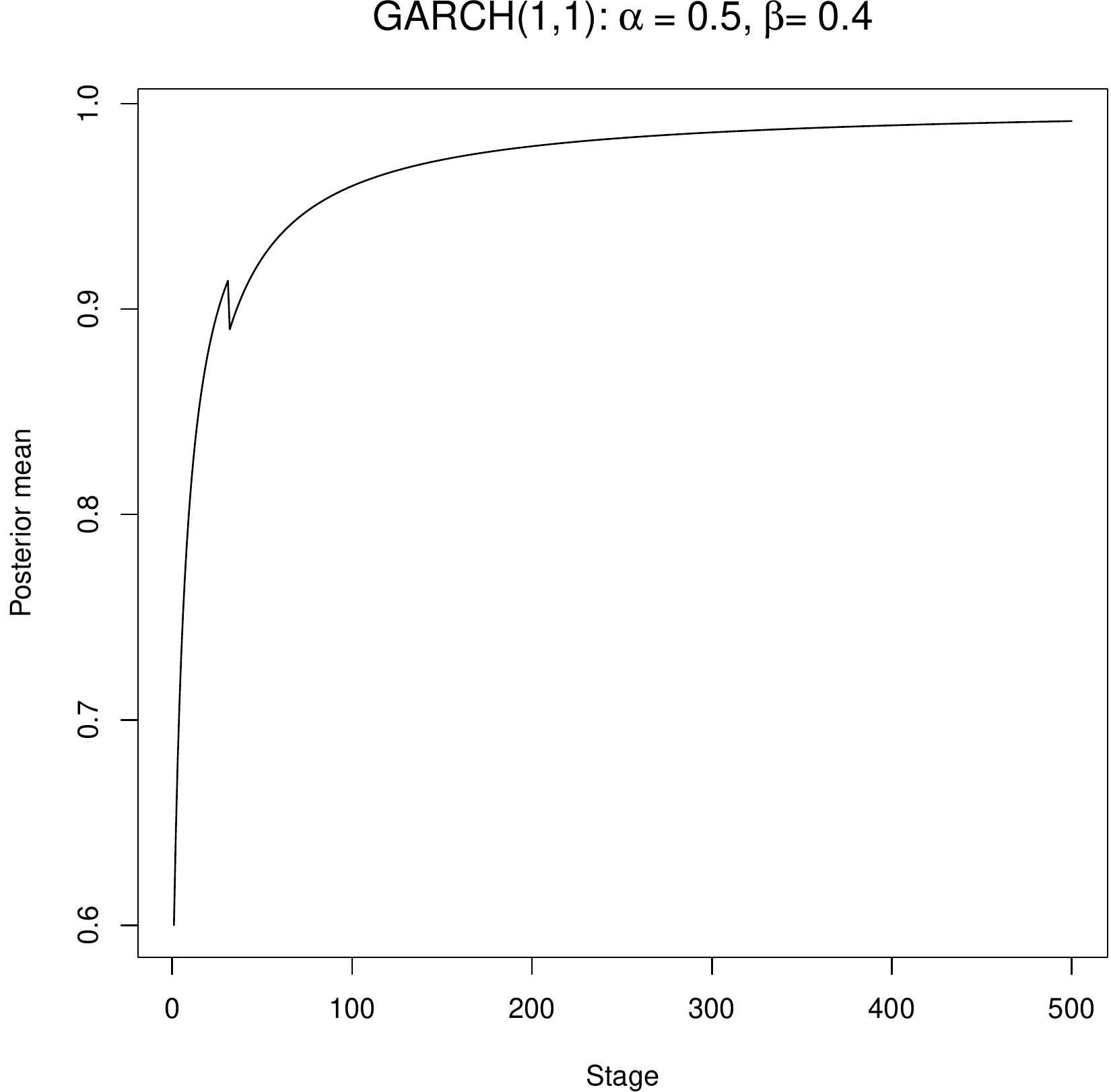}}
\hspace{2mm}
\subfigure [Nonstationary: $\alpha=0.5$, $\beta=0.6$.]{ \label{fig:garch_5_6_short_nonpara2}
\includegraphics[width=4.5cm,height=4.5cm]{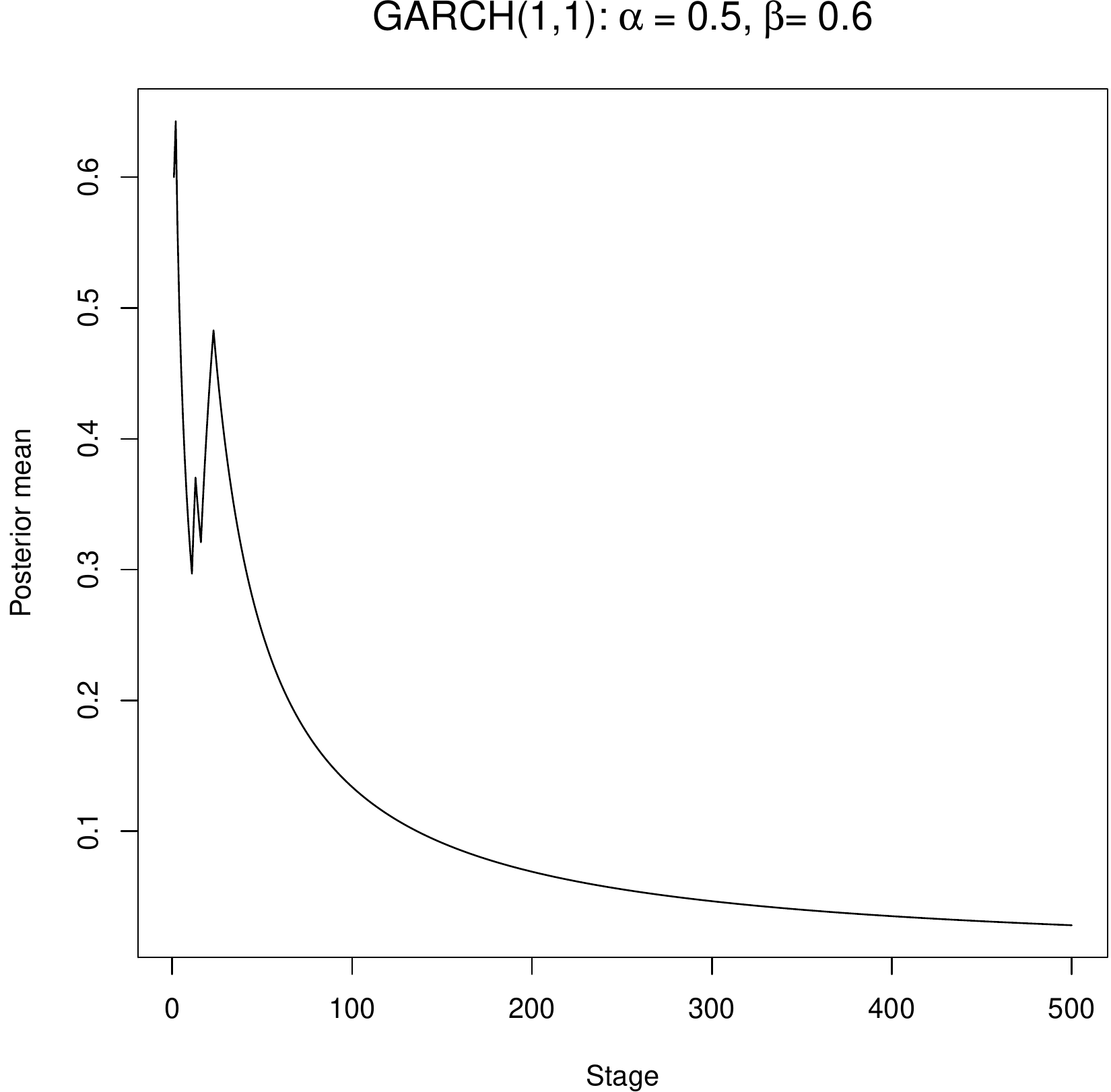}}
\hspace{2mm}
\subfigure [Nonstationary: $\alpha=0.6$, $\beta=0.6$.]{ \label{fig:garch_6_6_short_nonpara2}
\includegraphics[width=4.5cm,height=4.5cm]{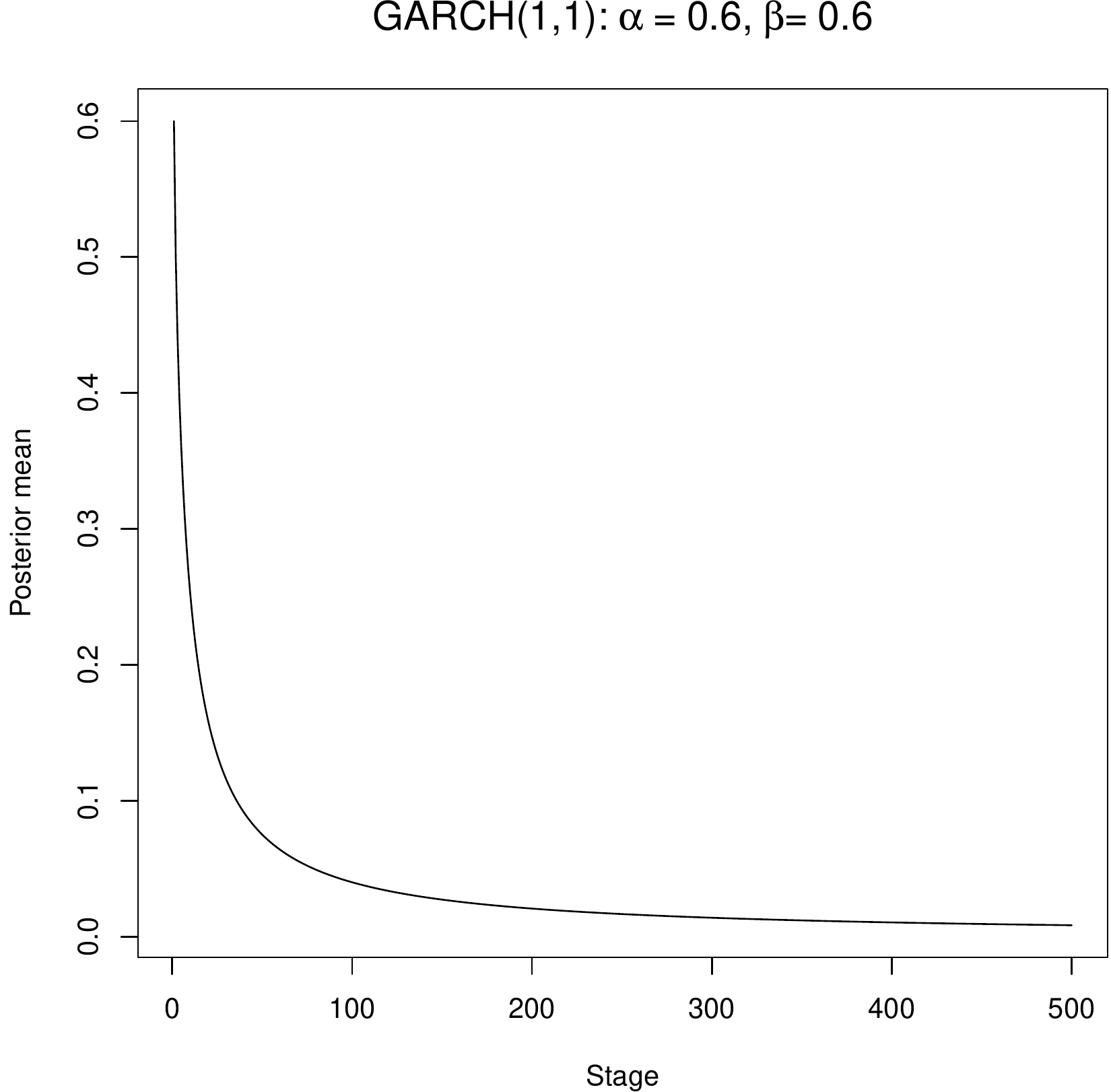}}\\
\vspace{2mm}
\subfigure [Nonstationary: $\alpha=0.5$, $\beta=0.5$.]{ \label{fig:garch_5_5_short_nonpara2}
\includegraphics[width=4.5cm,height=4.5cm]{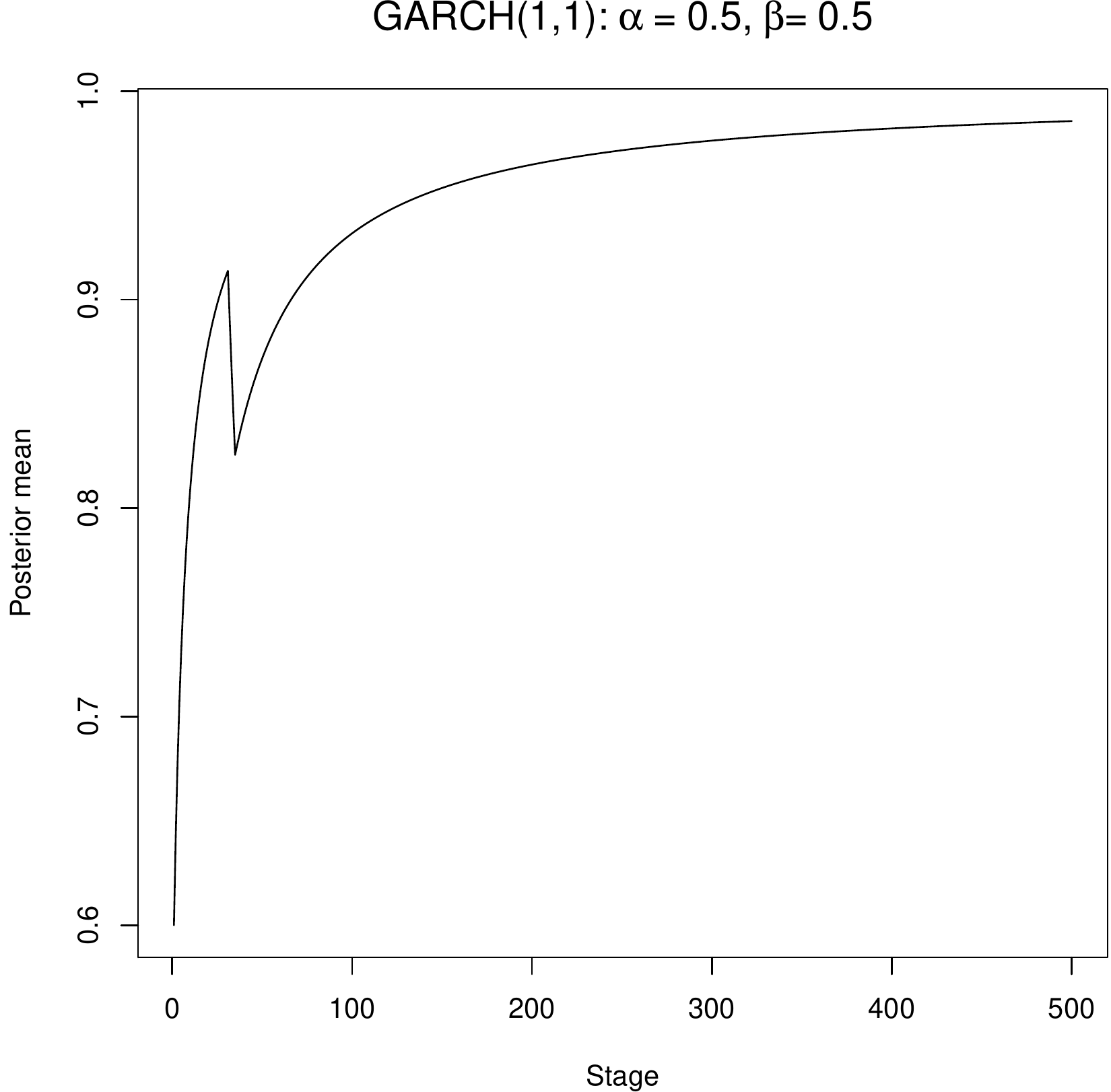}}
\hspace{2mm}
\subfigure [Nonstationary: $\alpha=0$, $\beta=1$.]{ \label{fig:garch_0_1_short_nonpara2}
\includegraphics[width=4.5cm,height=4.5cm]{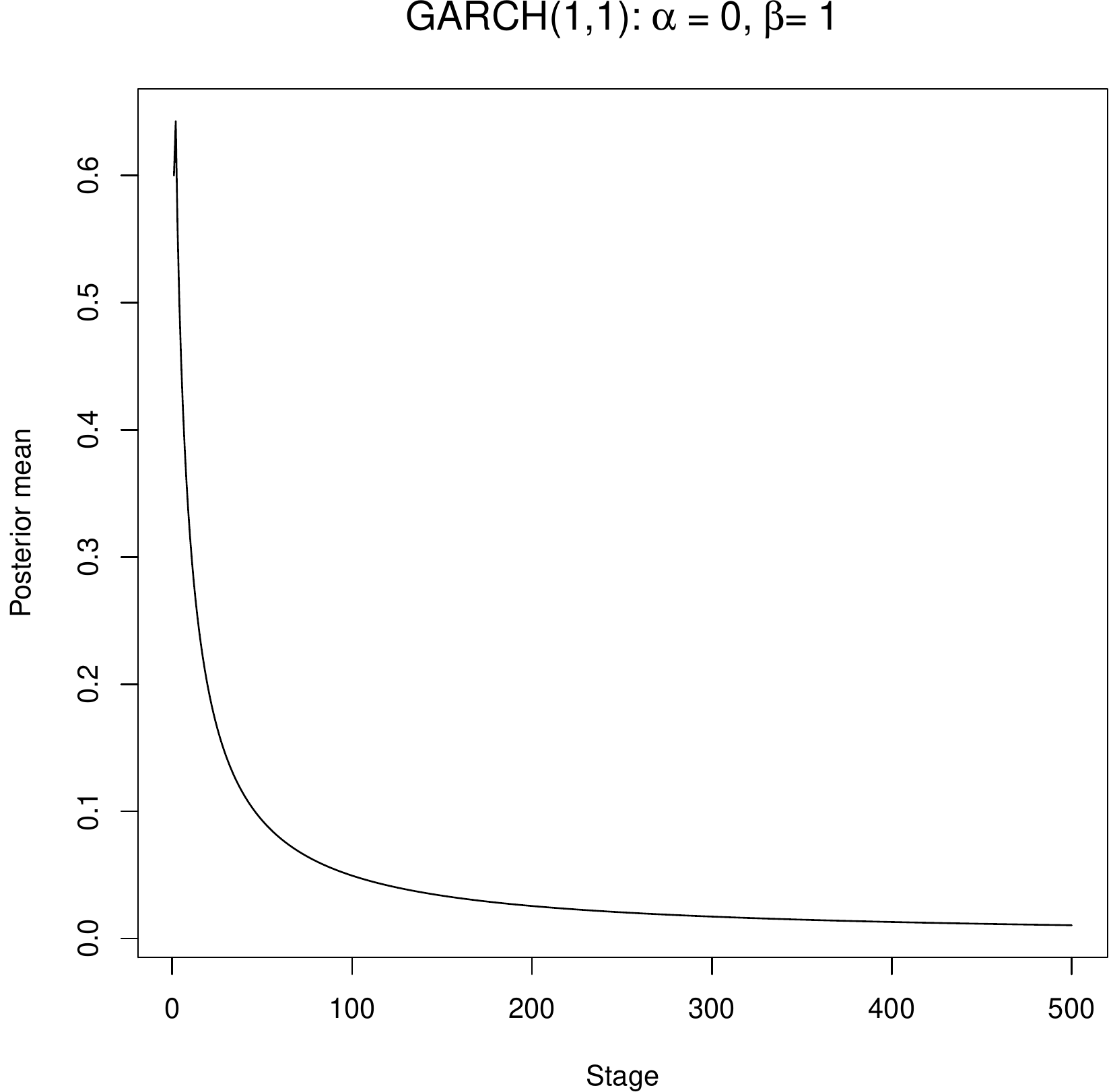}}
\hspace{2mm}
\subfigure [Nonstationary: $\alpha=1$, $\beta=0$.]{ \label{fig:garch_1_0_short_nonpara2}
\includegraphics[width=4.5cm,height=4.5cm]{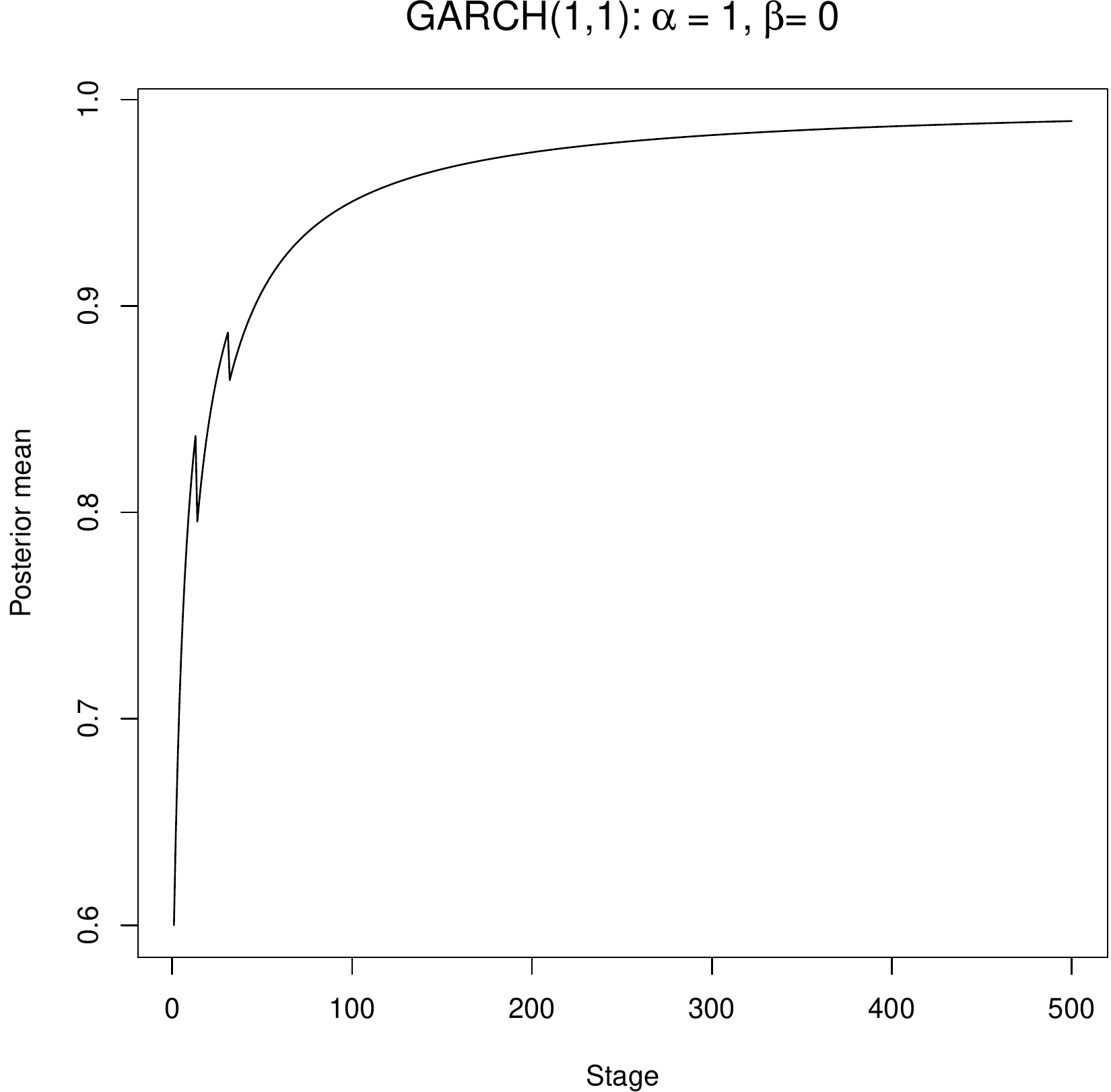}}
\caption{Nonparametric GARCH(1,1) example with $K=500$ and $n=5$.}
\label{fig:example_garch}
\end{figure}

The diagram for the case of $(\alpha=0.5,\beta=0.5)$ is provided in Figure \ref{fig:failed_Bayesian_garch}. Note that
this realization is essentially of the same pattern as panels (a) and (b) of Figure \ref{fig:failed_Bayesian_garch} associated with ARCH(1) models
with $\alpha=0.9$ and $1$, respectively, which do not seem to show any evidence of nonstationarity. Hence, again, quite unsurprisingly, our Bayesian
method declared this case as stationary.
\begin{figure}
\centering
\subfigure [Nonstationary: $\alpha=0.5$, $\beta=0.5$.]{ \label{fig:garchplot_5_5_short_nonpara2}
\includegraphics[width=6.5cm,height=6.5cm]{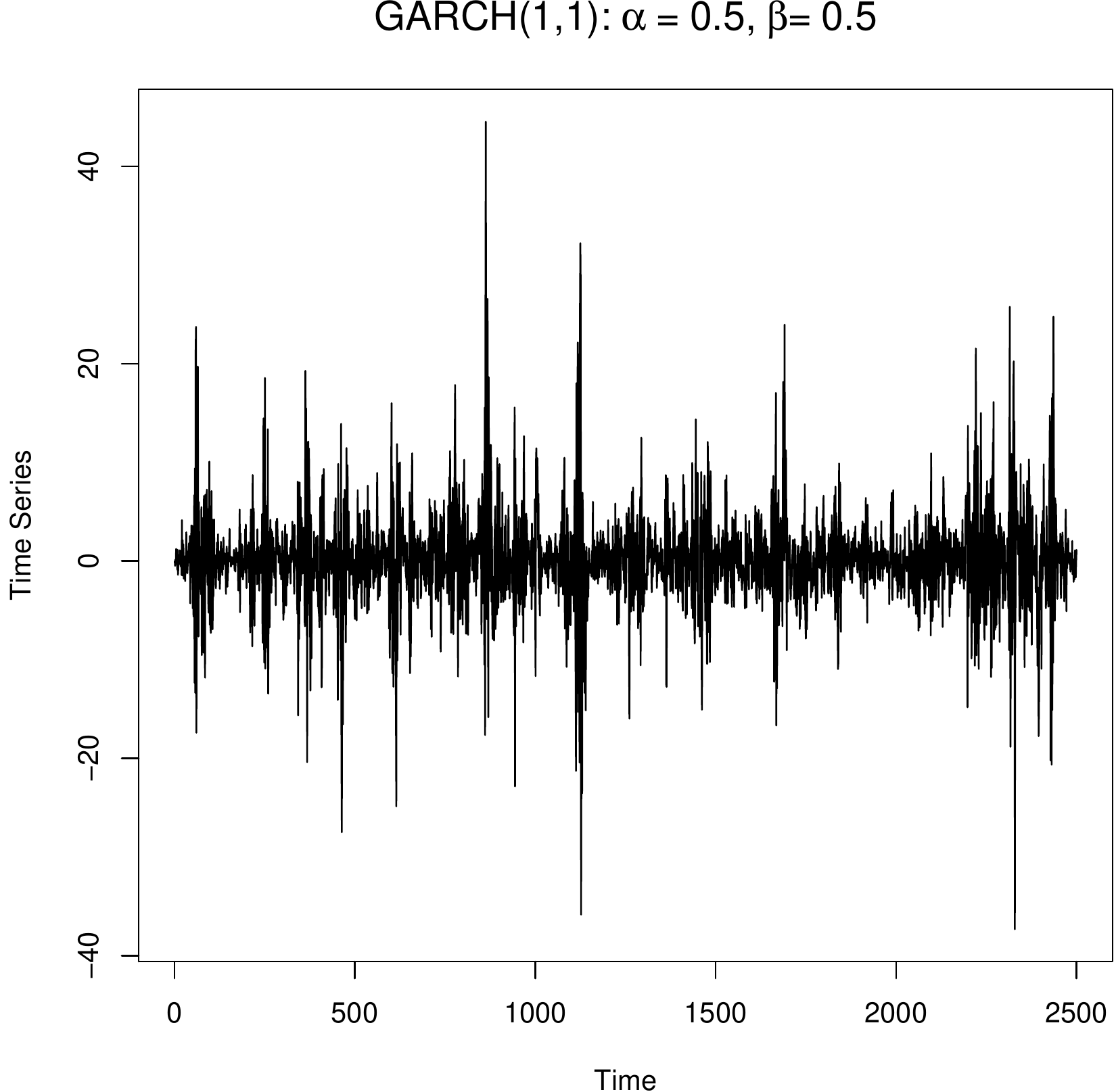}}
\caption{GARCH(1,1) sample for $\alpha=0.5$ and $\beta=0.5$ where our Bayesian method failed.}
\label{fig:failed_Bayesian_garch}
\end{figure}

\section{Third illustration: MCMC convergence diagnostics}
\label{sec:mcmc}

We now test our Bayesian method on the very relevant problem of MCMC convergence diagnosis.
For our purpose, we focus attention on transformation based Markov chain Monte Carlo (TMCMC) introduced by \ctn{Dutta14}. 
We consider three examples: in the first example, we assume that the target distribution is a product 
of $100$ standard normal densities, and consider seven instances of additive TMCMC. Here we make use of the optimal scaling theory for additive TMCMC.
In the next two examples, we consider mixtures of two normal densities. In all the cases, we evaluate convergence of TMCMC using our proposed Bayesian method. 

\subsection{A brief overview of TMCMC}
\label{subsec:tmcmc_overview}
TMCMC enables updating an entire block of parameters using deterministic bijective transformations of some arbitrary low-dimensional random variable. 
Thus very high-dimensional parameter
spaces can be explored using simple transformations of very low-dimensional random variables.
In fact, transformations of some one-dimensional random variable always suffices, which we shall
adopt in our examples. The underlying idea also greatly improves computational
speed and acceptance rate compared to block Metropolis-Hastings methods. Interestingly, the
TMCMC acceptance ratio is independent of the proposal distribution chosen for the arbitrary low-
dimensional random variable. For implementation in our cases, we shall consider the additive
transformation, since it is shown in \ctn{Dutta14} that many fewer number of
``move types" are required by this transformation compared to non-additive transformations.
To elaborate the additive TMCMC mechanism, assume that a block of parameters $\bx = (x_1,\ldots,x_r)$
is to be updated simultaneously using additive TMCMC, where $r~(\geq 2)$ is some positive integer.
At the $t$-th iteration ($t\geq 1$) we shall then simulate $\theta\sim g(x)I_{\left\{x>0\right\}}$ , where $g(\cdot)$ is some arbitrary distribution and 
$I_{\left\{x>0\right\}}$ is the indicator function of the set $\left\{x > 0\right\}$. 

We then propose, for $j = 1,\ldots,r$, $x^{(t)}_j = x^{(t-1)}_j\pm a_j\eta$, with equal probability (although equal probability is
a convenience, not a necessity), where $(a_1,\ldots,a_r)$ are appropriate scaling constants. Thus, using
additive transformations of a single, one-dimensional $x$, we update the entire block $\bx$ at once. 

\subsection{Optimal scaling of TMCMC}
\label{subsec:tmcmc_optimal_scaling}
In our examples we shall choose $r=d$, where $d$ is the 
total number of parameters to be updated. In other words, we shall update all the parameters simultaneously, in a single block.
We shall consider $a_i=1$, for $i=1,\ldots,d$ and $g(\cdot)$ to be the $N(0,\frac{\ell^2}{d})$ density, so that $\eta$ is simulated from a truncated normal distribution, 
with mean zero and variance $\ell^2/d$. The optimum choice of $\ell$ is directly related to the optimal scaling problem (see \ctn{Dey17} and \ctn{Dey19}).
Under appropriate regularity conditions it turns out that the optimal value of $\ell$ corresponds to the optimal additive TMCMC acceptance rate $0.439$.
When the target distribution $\pi(x_1,\ldots,x_d)$ is a product of $d$ $iid$ standard normal densities, as we consider, 
then it turns out that the optimum choice of $\ell$ is $2.4$.

\subsection{TMCMC example 1: product of $100$ standard normal densities}
\label{subsec:product_normal}

We apply additive TMCMC to generate $10^6$ realizations from $\pi(x_1,\ldots,x_d)$ being a product of $d$ standard normal densities with $d=100$.
We consider seven values of $\ell$, and hence seven different TMCMC chains, each corresponding to a value of $\ell$. In particular, we set $\ell=0.001$, $0.01$,
$0.1$, $2.4$, $10$, $100$ and $1000$. Of these, $\ell=2.4$ is the optimum value that maximizes the ``diffusion speed" associated with the TMCMC chain. The values
relatively closer to $2.4$, although not optimal, can still generate TMCMC chains with reasonable convergence properties. Significantly small values of $\ell$ generates TMCMC
chains with very high acceptance rates but with very slow convergence rates, as at each iteration, the chain is allowed to take only small steps for movement. 
On the other hand, for significantly large values of $\ell$, large steps are generally proposed, which are often rejected. Thus, the chain again has slow convergence,
with poor acceptance rate. 

It transpires from the above discussion that for values of $\ell$ equal to, or relatively close to $2.4$, good convergence properties of the TMCMC chains
can be expected, and it is desirable that our Bayesian method indicates convergence to stationarity for such cases. For other values of $\ell$, since
the convergence properties of the chains are expected to be poor, our Bayesian method must reflect so.

Generation of $10^6$ TMCMC realizations from $\pi(x_1,\ldots,x_d)$ with $d=100$ takes less than $0.05$ seconds on an ordinary 64 bit laptop.
For implementation of our Bayesian idea, we need the bounds $c_j$. The general-purpose nonparametric bound (\ref{eq:ar1_bound3}) turned out to be
quite appropriate in all the TMCMC examples that we consider. Indeed, in general there is no provision for parametric bounds in MCMC situations,
as such bounds would require direct generation from $\pi$ or some distribution close to $\pi$, but if such direct generation were at all possible, MCMC would not
be needed in the first place.

For $K=1000$ and $n=1000$, Figures \ref{fig:example4} and \ref{fig:example5} display the trace plots (presented
after thinning the original chain of length $10^6$ by $100$, to reduce the file sizes) and the corresponding Bayesian posterior means
associated with our Bayesian stationarity detection idea, for different values of $\ell$, for the first co-ordinate $x_1$ of $(x_1,\ldots,x_{100})$. 
It takes a few seconds even on a 64-bit dual core laptop for parallel implementation of our Bayesian idea in these cases.

The results are very much in keeping with our prior expectation
that for significantly small and large values of $\ell$ convergence to stationarity for the given sample size is not expected, while for $\ell=2.4$ and
values relatively close to $2.4$, stationarity is expected. Specifically, the figures for Bayesian stationarity detection strongly indicate convergence
for $\ell=0.1$, $2.4$ and $10$, but strongly indicate that the chains corresponding to $\ell=0.001$, $0.01$, $100$ and $1000$, are yet to achieve stationarity.
Note that these results are also in accordance with the visual information obtained from the corresponding trace plots.
\begin{figure}
\centering
\subfigure [TMCMC trace plot ($\ell=0.001$).]{ \label{fig:trace_plot1}
\includegraphics[width=5.5cm,height=5.5cm]{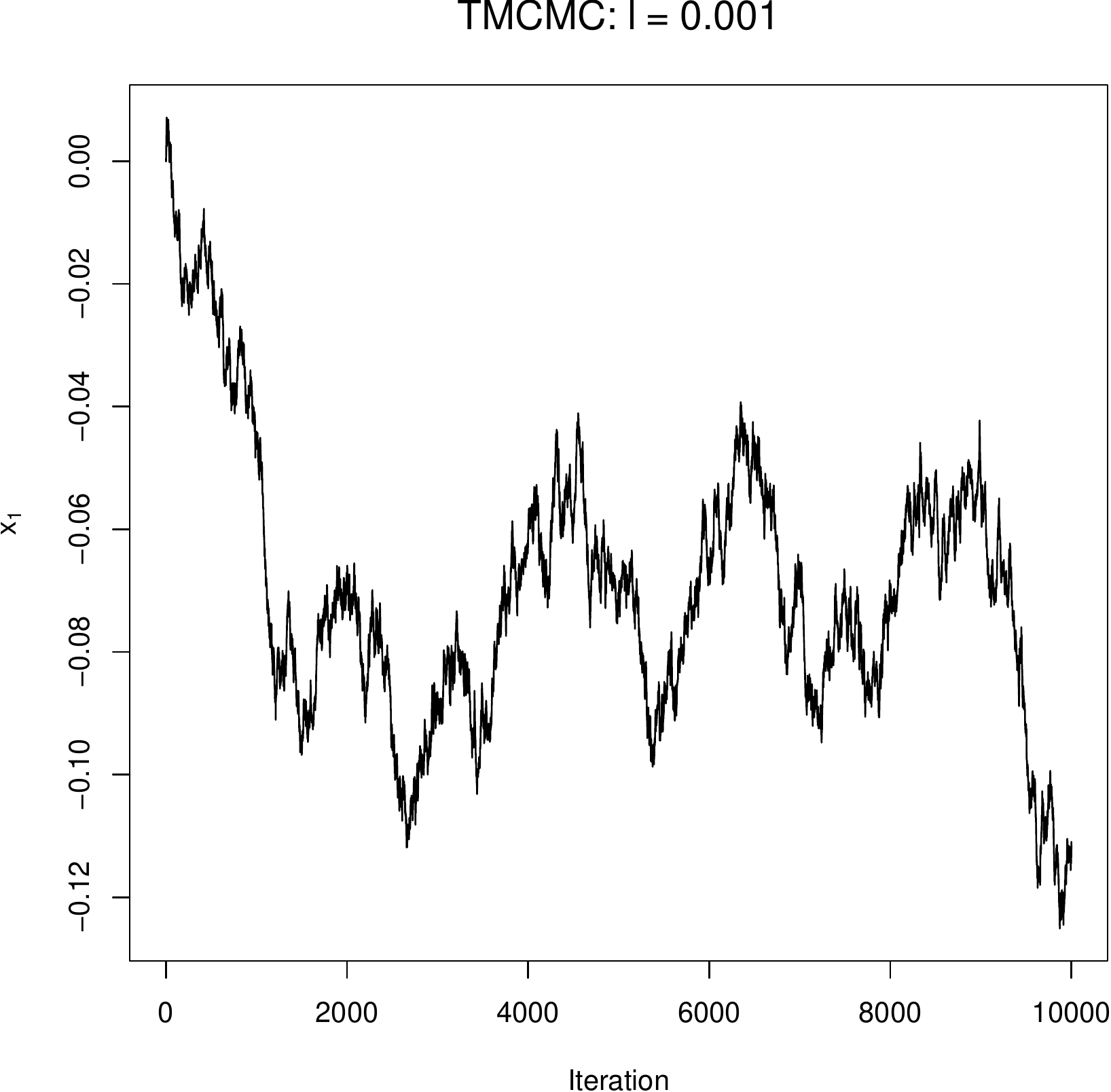}}
\hspace{2mm}
\subfigure [Convergence ($\ell=0.001$): Nonstationary.]{ \label{fig:diag1}
\includegraphics[width=5.5cm,height=5.5cm]{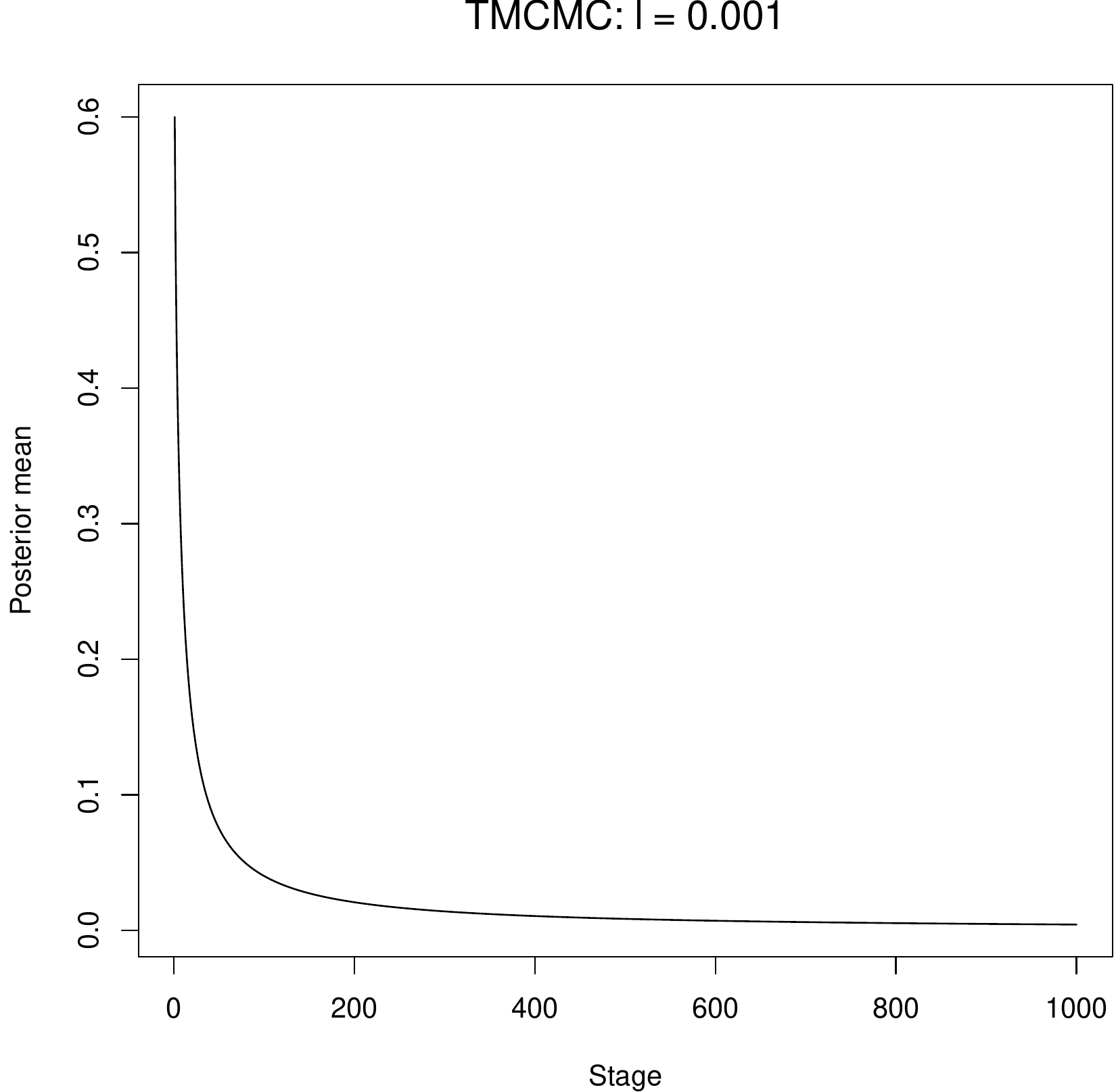}}\\
\vspace{2mm}
\subfigure [TMCMC trace plot ($\ell=0.01$).]{ \label{fig:trace_plot2}
\includegraphics[width=5.5cm,height=5.5cm]{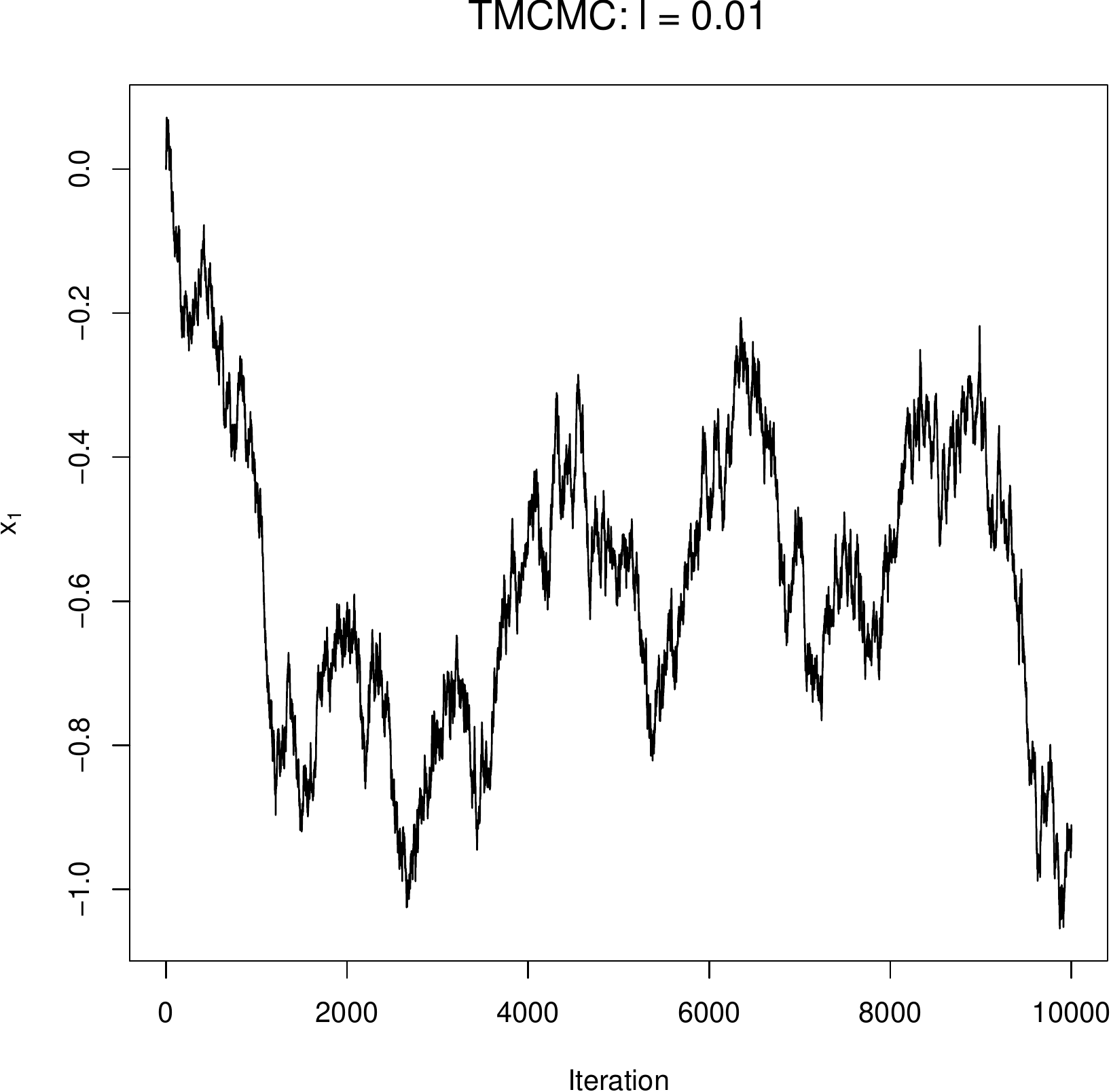}}
\hspace{2mm}
\subfigure [Convergence ($\ell=0.01$): Nonstationary.]{ \label{fig:diag2}
\includegraphics[width=5.5cm,height=5.5cm]{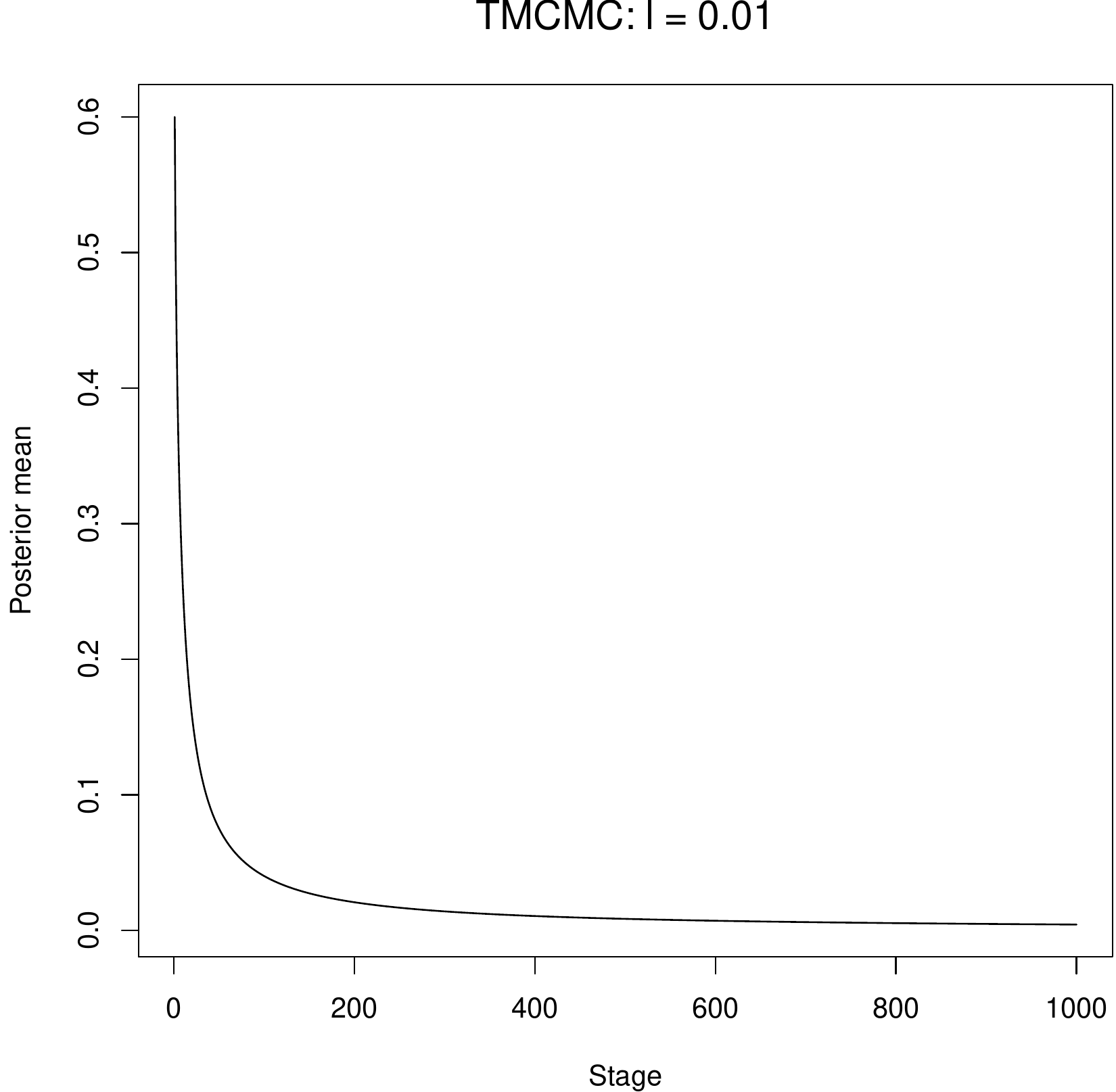}}\\
\vspace{2mm}
\subfigure [TMCMC trace plot ($\ell=0.1$).]{ \label{fig:trace_plot3}
\includegraphics[width=5.5cm,height=5.5cm]{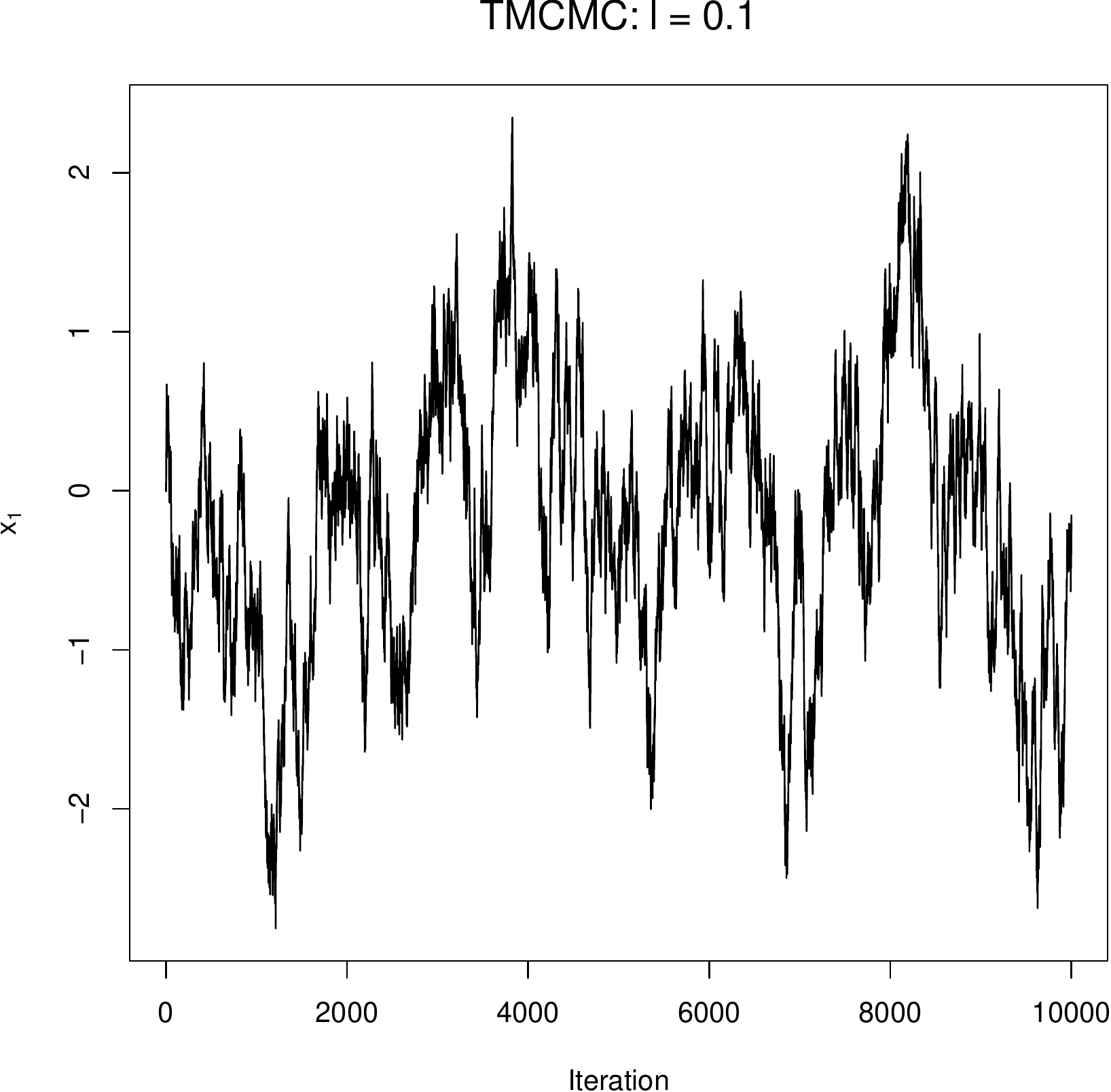}}
\hspace{2mm}
\subfigure [Convergence ($\ell=0.1$): Stationary.]{ \label{fig:diag3}
\includegraphics[width=5.5cm,height=5.5cm]{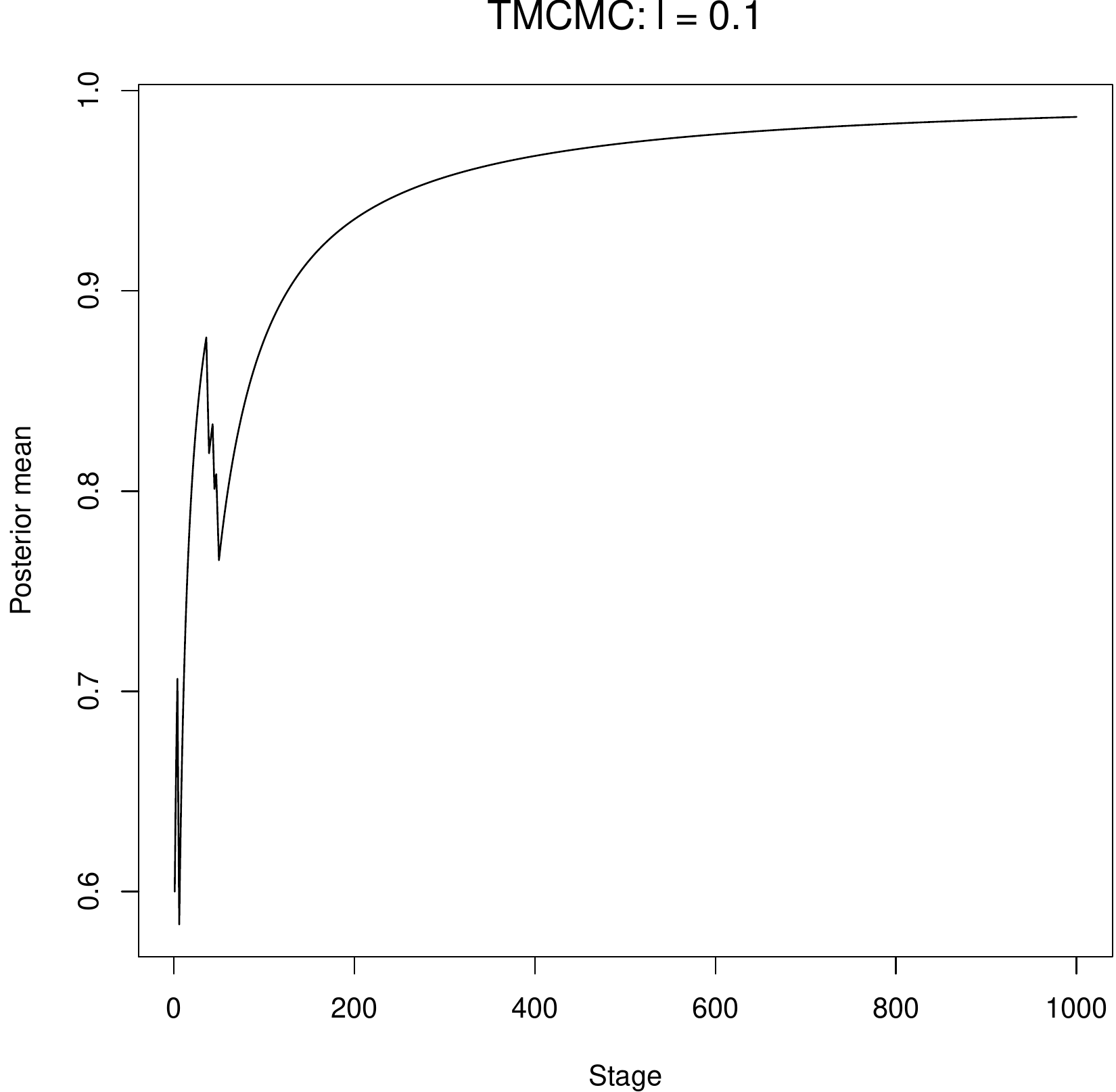}}\\
\vspace{2mm}
\subfigure [TMCMC trace plot ($\ell=2.4$).]{ \label{fig:trace_plot4}
\includegraphics[width=5.5cm,height=5.5cm]{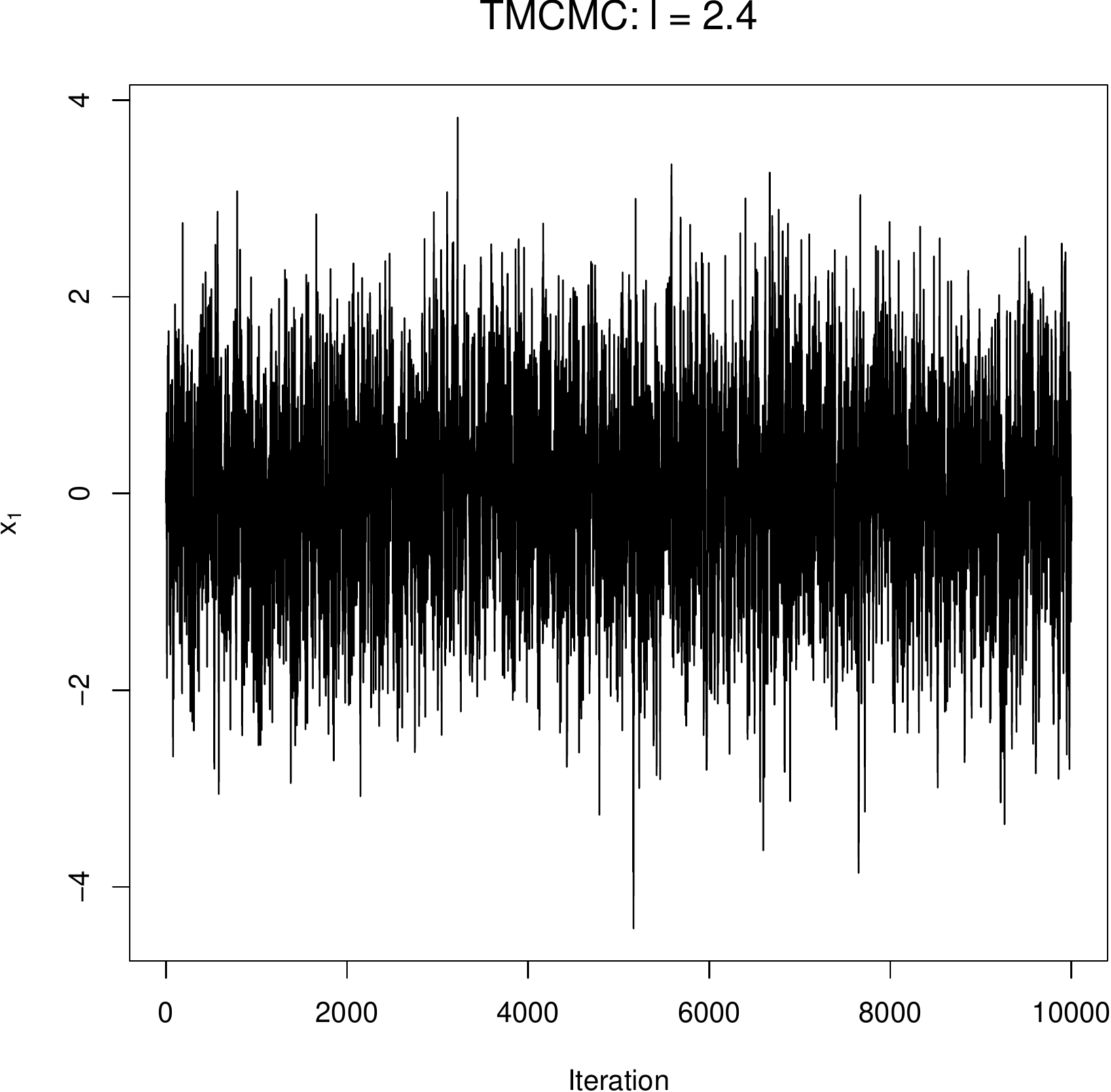}}
\hspace{2mm}
\subfigure [Convergence ($\ell=2.4$): Stationary.]{ \label{fig:diag4}
\includegraphics[width=5.5cm,height=5.5cm]{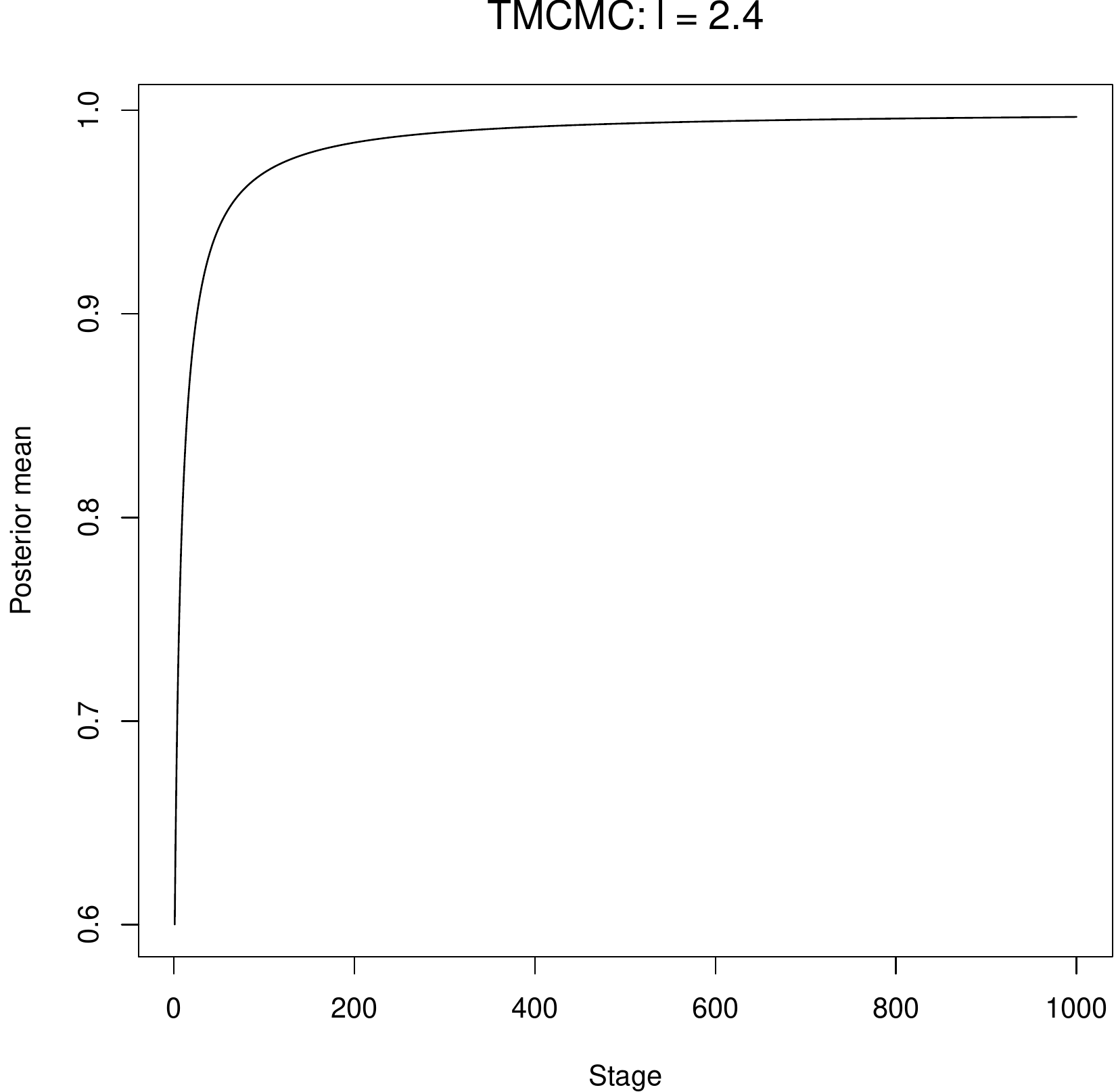}}\\
\caption{Additive TMCMC convergence example, with $K=1000$ and $n=1000$.}
\label{fig:example4}
\end{figure}

\begin{figure}
\centering
\subfigure [TMCMC trace plot ($\ell=10$).]{ \label{fig:trace_plot5}
\includegraphics[width=5.5cm,height=5.5cm]{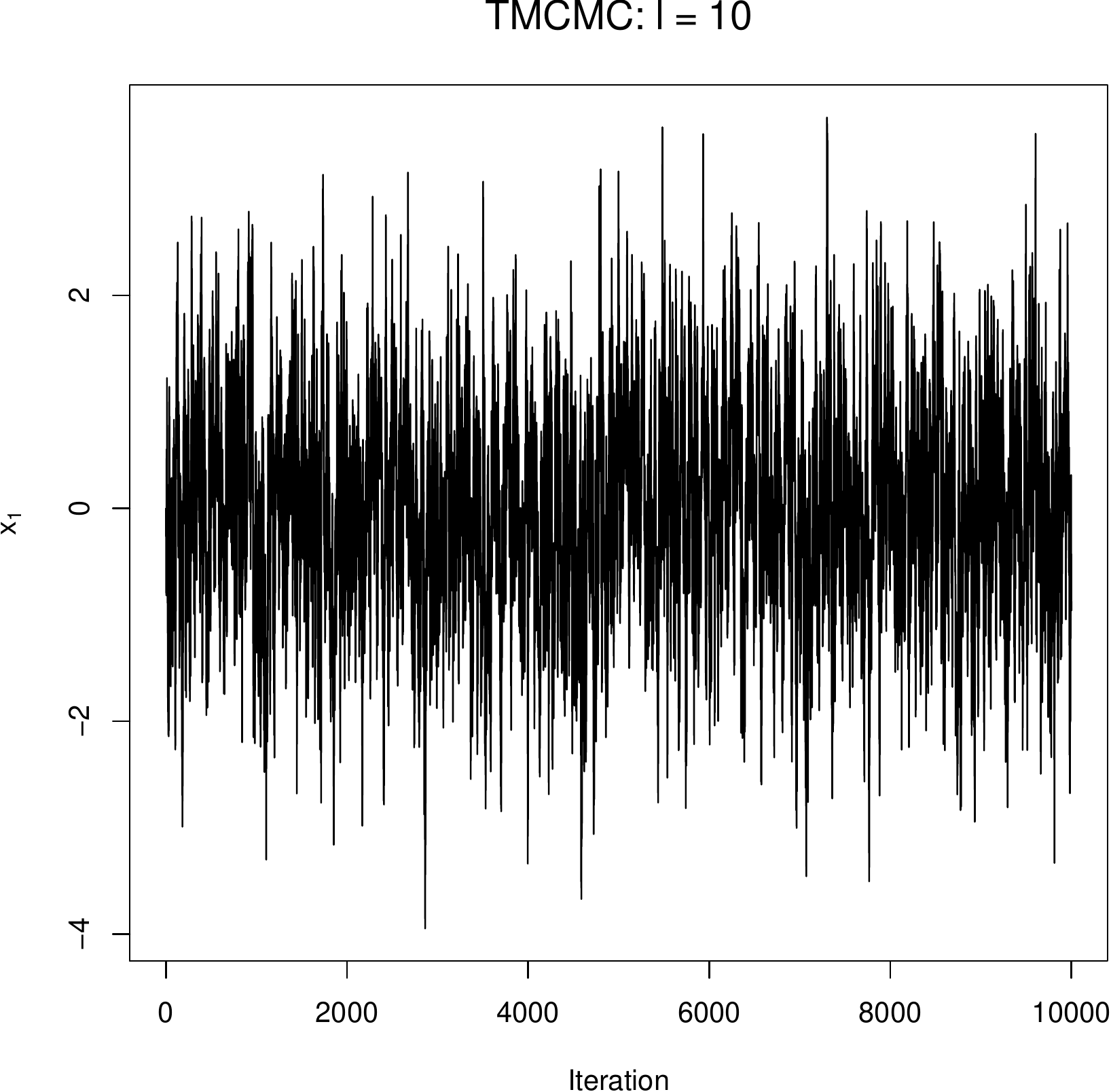}}
\hspace{2mm}
\subfigure [Convergence ($\ell=10$): Stationary.]{ \label{fig:diag5}
\includegraphics[width=5.5cm,height=5.5cm]{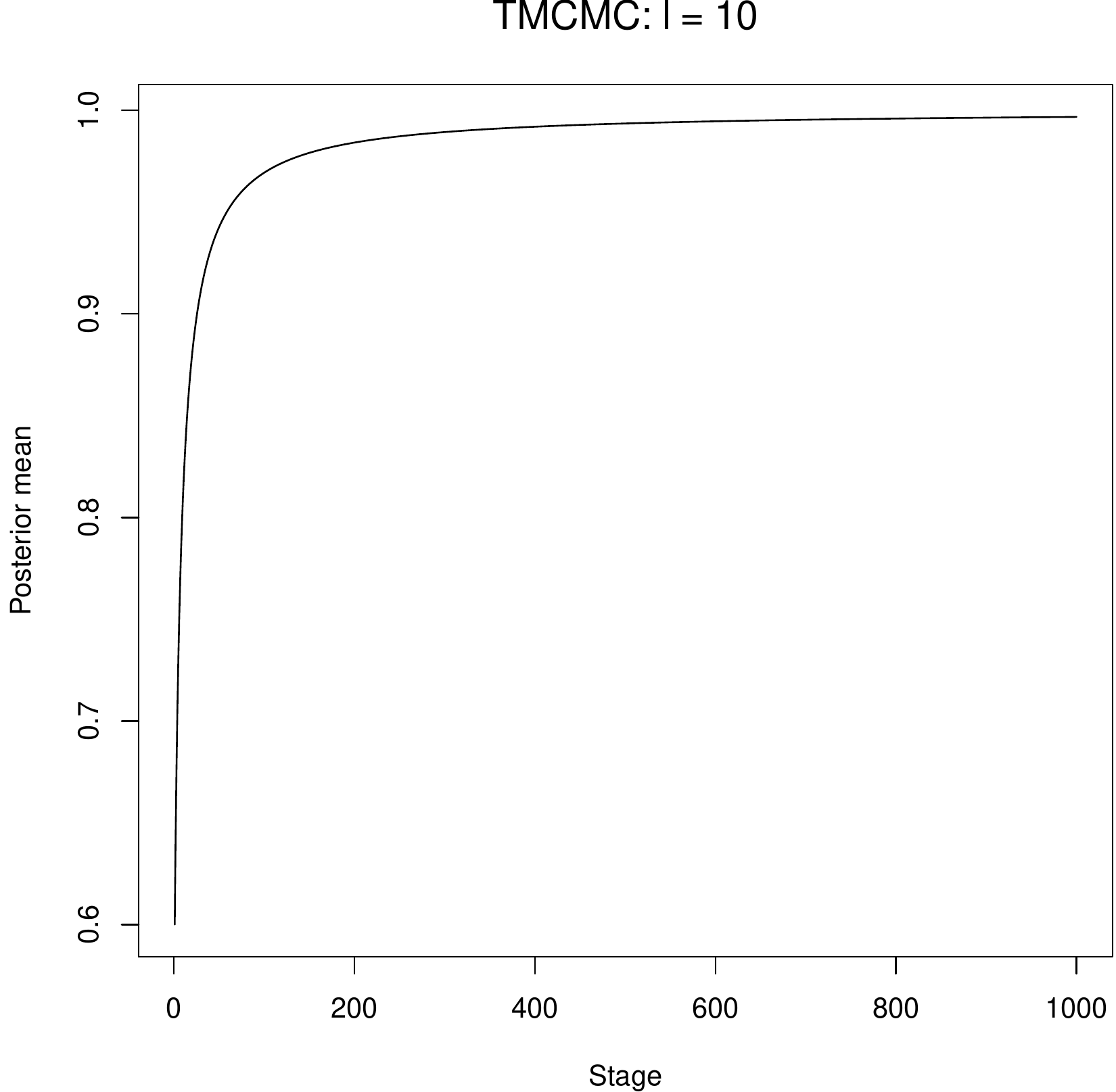}}\\
\vspace{2mm}
\subfigure [TMCMC trace plot ($\ell=100$).]{ \label{fig:trace_plot6}
\includegraphics[width=5.5cm,height=5.5cm]{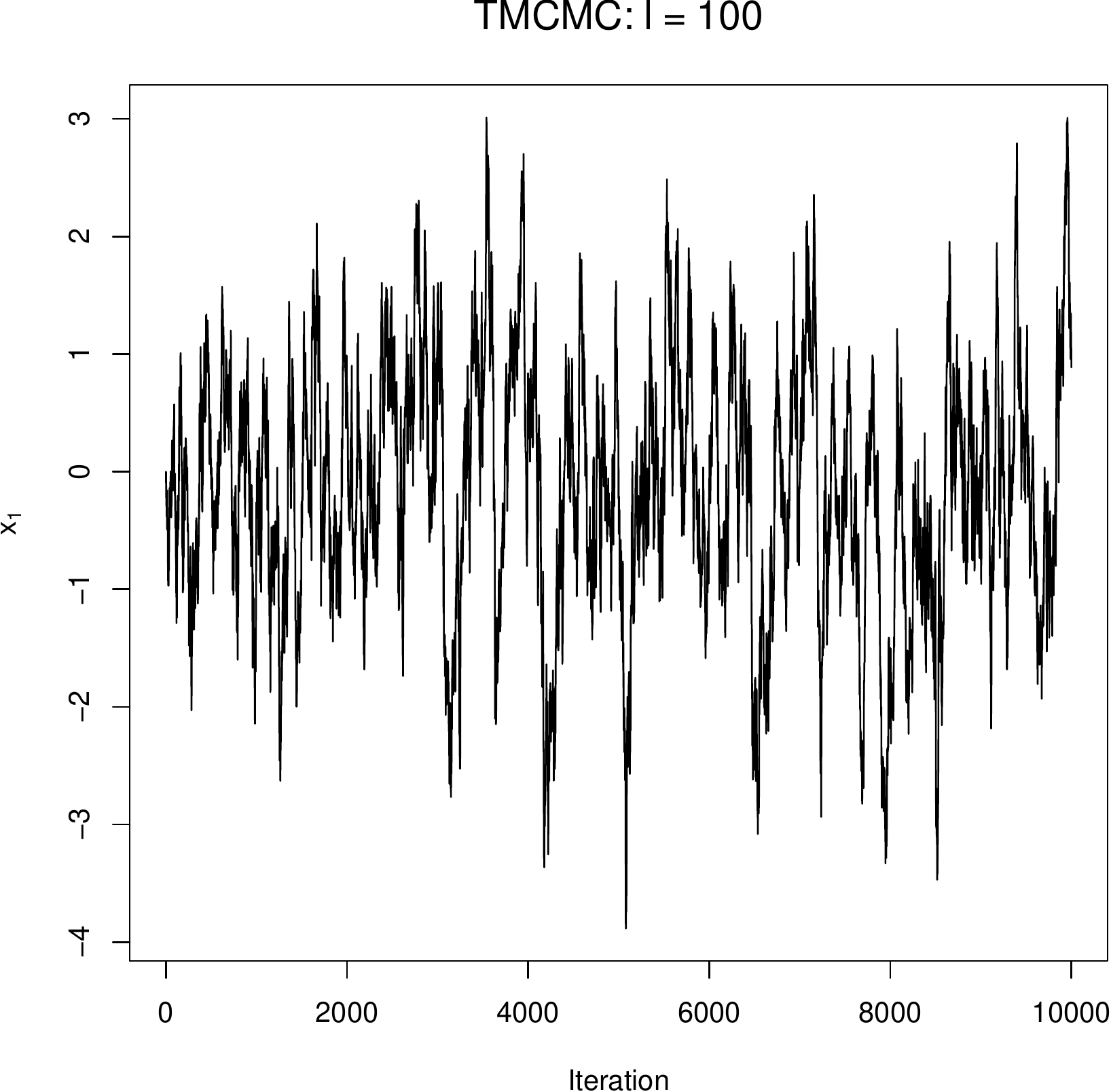}}
\hspace{2mm}
\subfigure [Convergence ($\ell=100$): Nonstationary.]{ \label{fig:diag6}
\includegraphics[width=5.5cm,height=5.5cm]{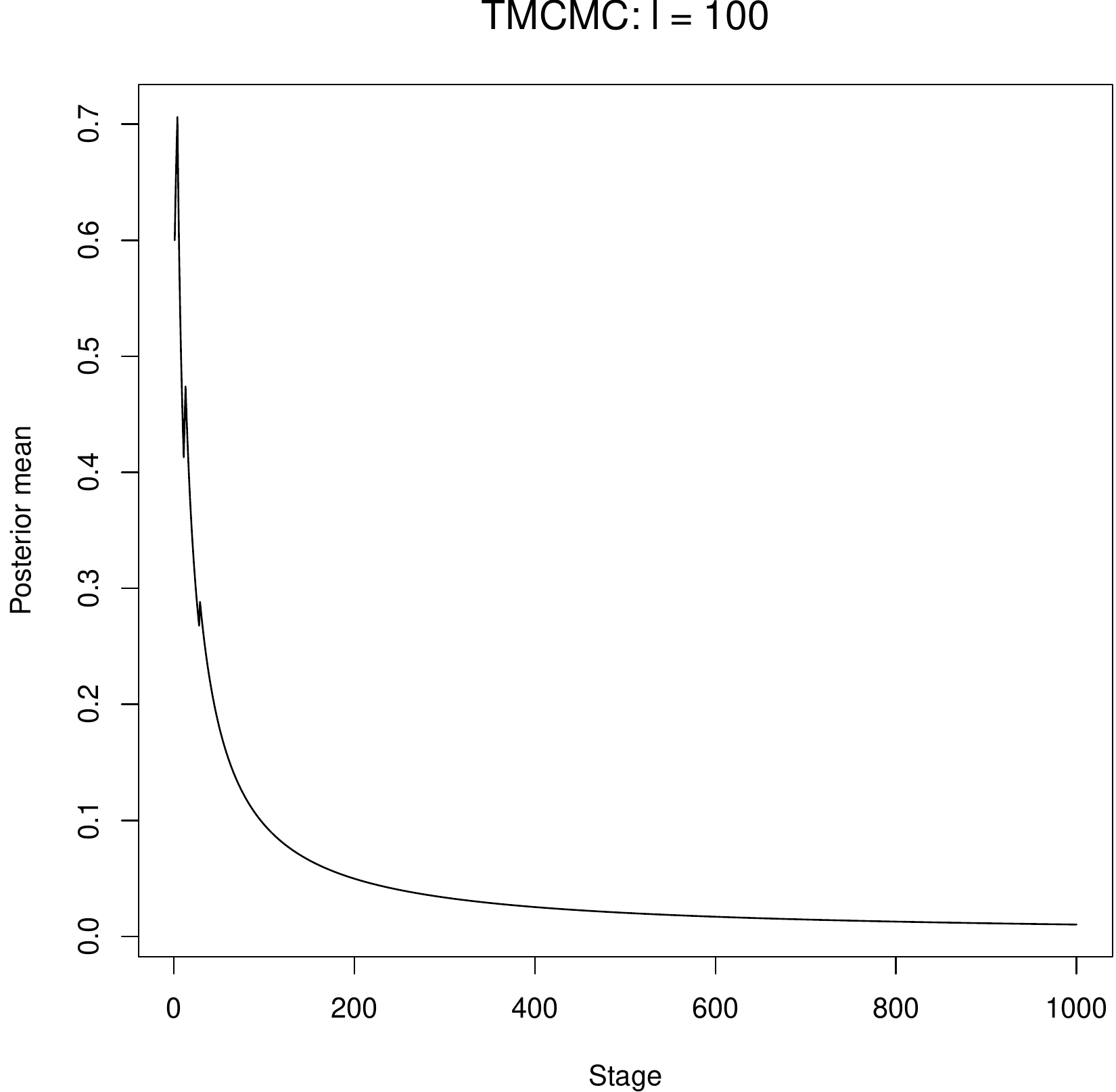}}\\
\vspace{2mm}
\subfigure [TMCMC trace plot ($\ell=1000$).]{ \label{fig:trace_plot7}
\includegraphics[width=5.5cm,height=5.5cm]{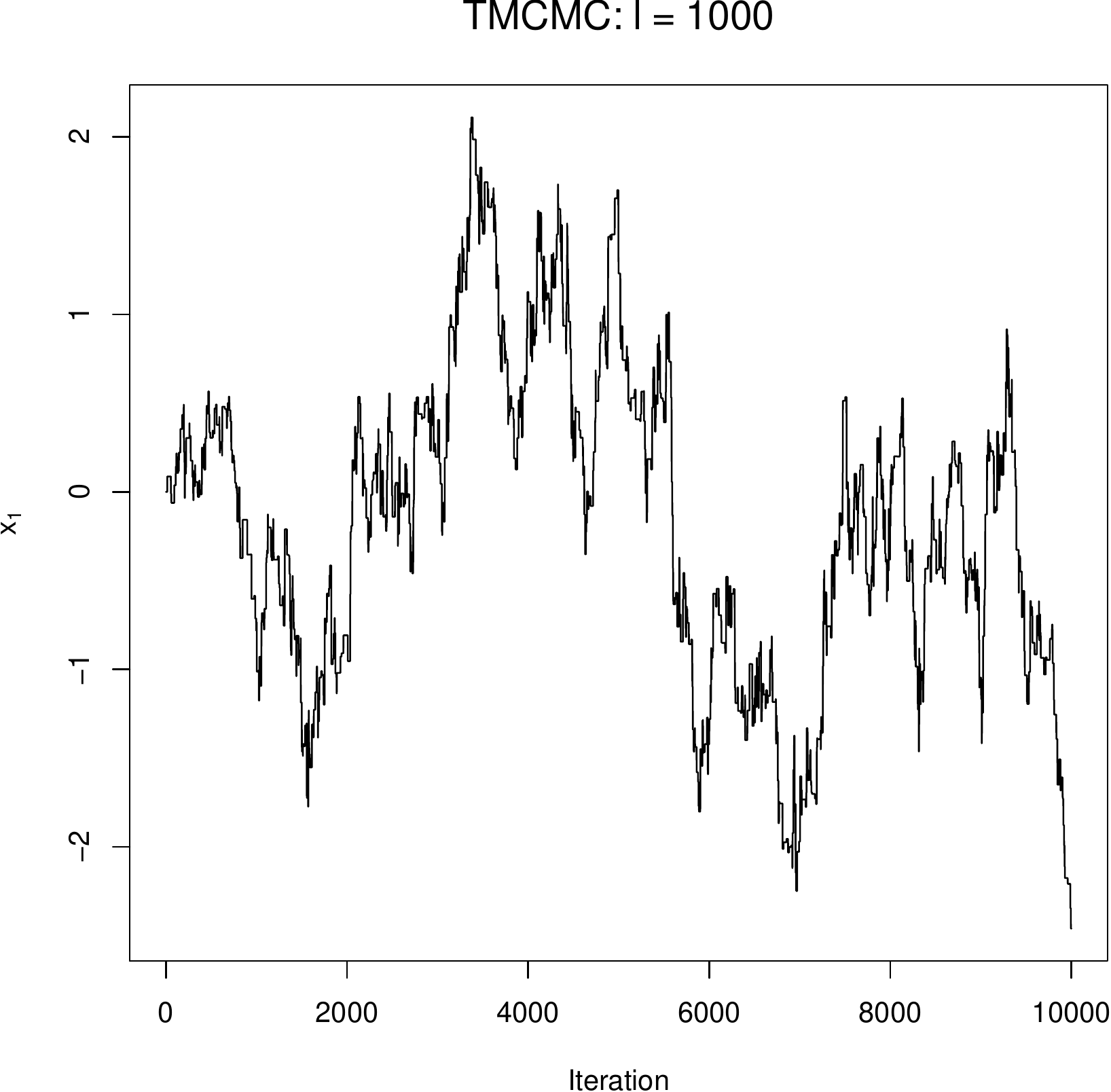}}
\hspace{2mm}
\subfigure [Convergence ($\ell=1000$): Nonstationary.]{ \label{fig:diag7}
\includegraphics[width=5.5cm,height=5.5cm]{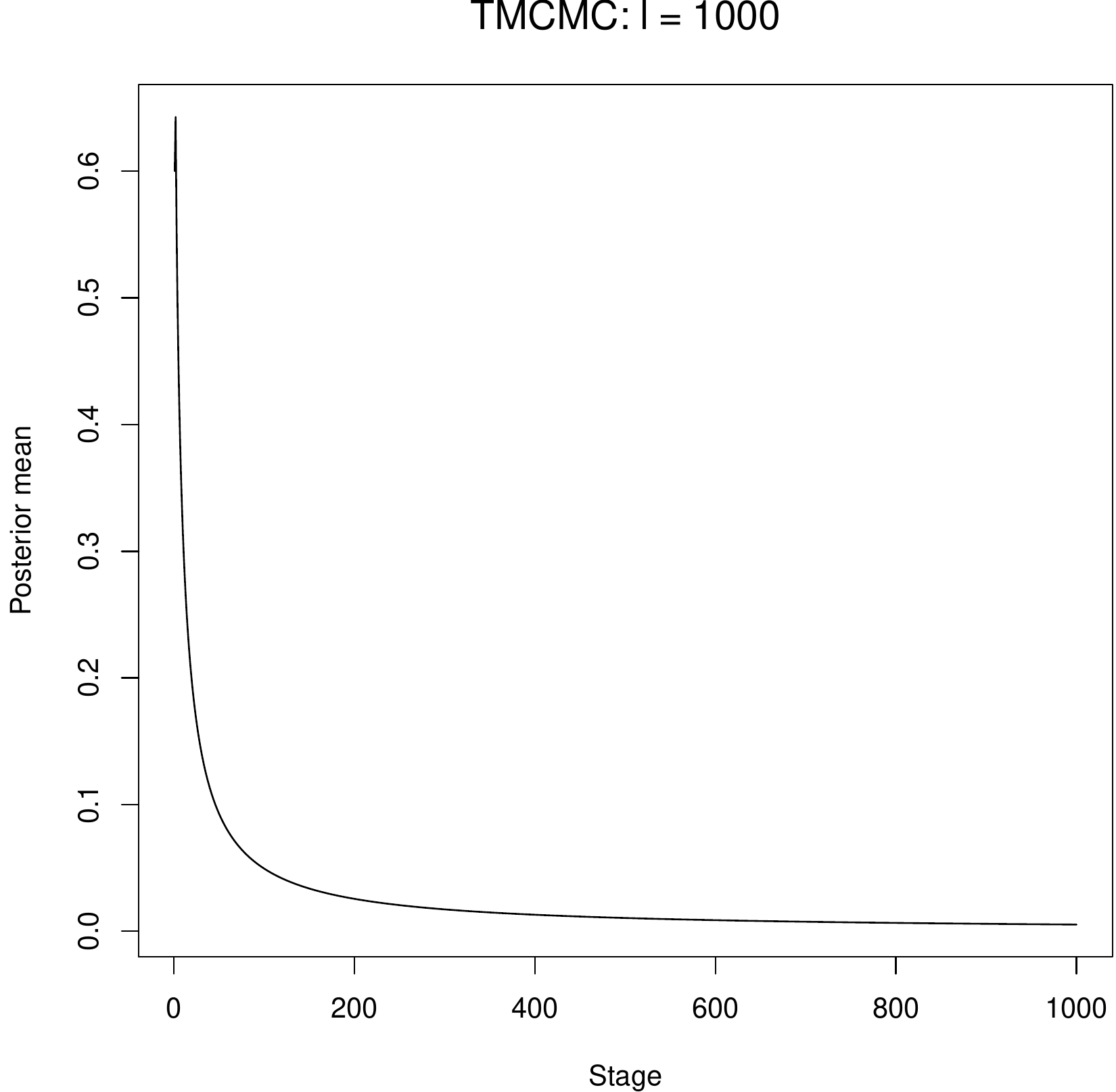}}\\
\caption{Additive TMCMC convergence example, with $K=1000$ and $n=1000$.}
\label{fig:example5}
\end{figure}

\subsection{TMCMC example 2: mixture normal densities}
\label{subsec:mixtures}
We now consider two mixtures of normal densities. The first mixture is of the form
\begin{equation}
\pi(x)=\frac{1}{2}N(x:0,1)+\frac{1}{2}N(x,10,1),
\label{eq:mixture1}
\end{equation}
where $N(x,\mu,\sigma^2)$ denotes the normal density with mean $\mu$ and variance $\sigma^2$, evaluated at $x$.
The second mixture is of the form
\begin{equation}
\pi(x)=\frac{1}{2}N(x:0,1)+\frac{1}{2}N(x,15,1),
\label{eq:mixture2}
\end{equation}
The mixtures differ slightly only in the means of the second mixture, but with TMCMC implementation, they reveal significant
difference.

With the same implementation as before, with $\ell=2.4$, and with the same bound $c_j$, we obtain Figure \ref{fig:example6}.
The TMCMC trace plot and the Bayesian idea of stationarity detection reveals that for (\ref{eq:mixture1}) stationarity is clearly reached.
That this is achieved even though the chain concentrates around two values $0$ and $10$, is quite encouraging.

The trace plot for (\ref{eq:mixture2}), with the same implementation as before displays two instances of very distinct and significant local stationarity. 
Consequently, for stationarity detection for this case, $K=1000$ and $n=1000$ is no 
longer appropriate. Rather, $K=2$ and $n=500000$, seems to be natural and appropriate. With this we obtain the posterior means for the two iterations
(corresponding to $K=2$) to be $0.6$ and $0.5$, respectively, with the associated posterior variances $0.04$ and $0.03125$. This is an indication that the chain
did not yet reach stationarity, which is also evident from the trace plot. Indeed, for just two instances of significant local stationarities, global stationarity
can not be ensured.
\begin{figure}
\centering
\subfigure [TMCMC trace plot for first mixture.]{ \label{fig:trace_plot8}
\includegraphics[width=5.5cm,height=5.5cm]{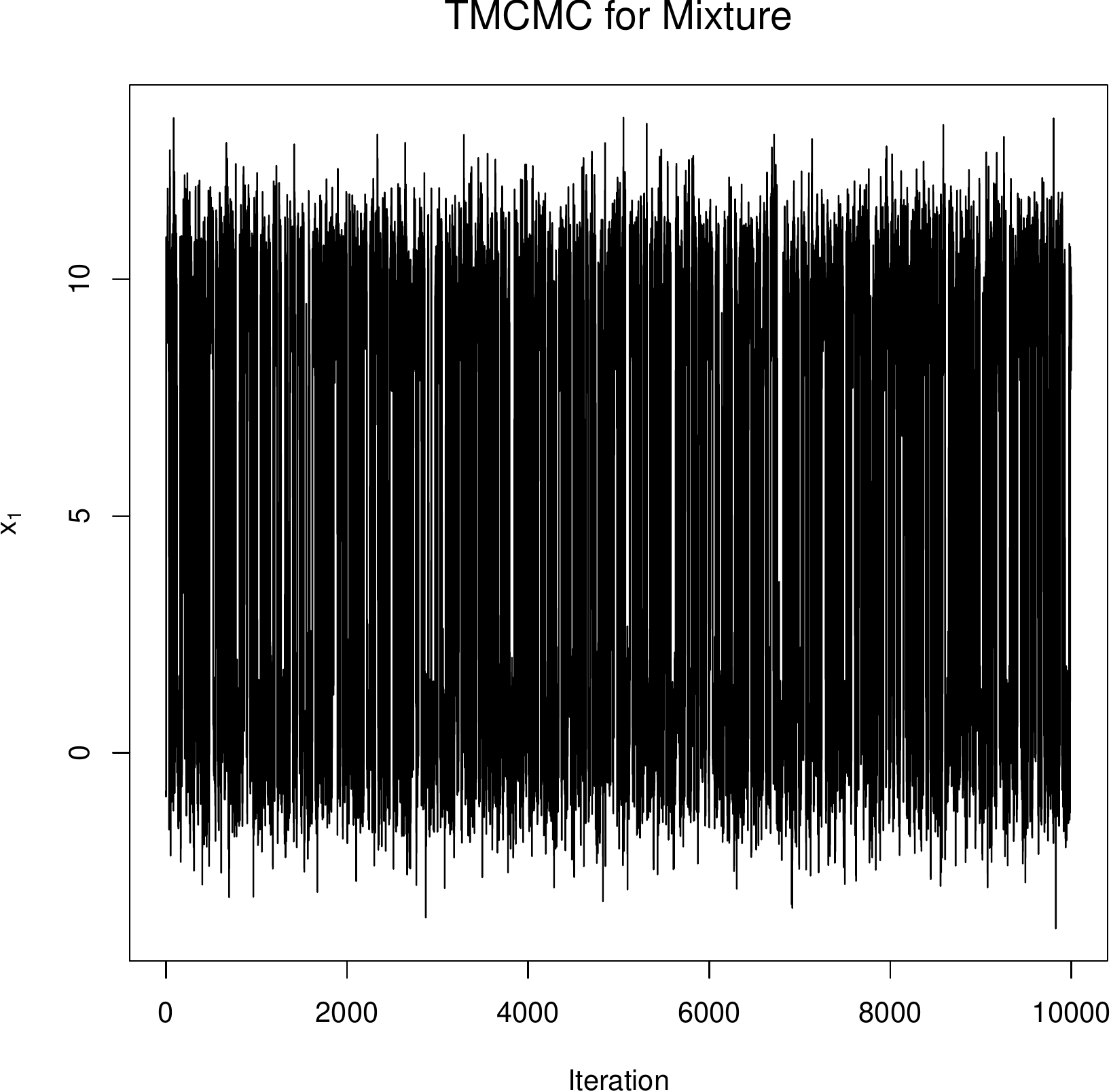}}
\hspace{2mm}
\subfigure [Convergence: Stationary ($K=1000$, $n=1000$).]{ \label{fig:diag8}
\includegraphics[width=5.5cm,height=5.5cm]{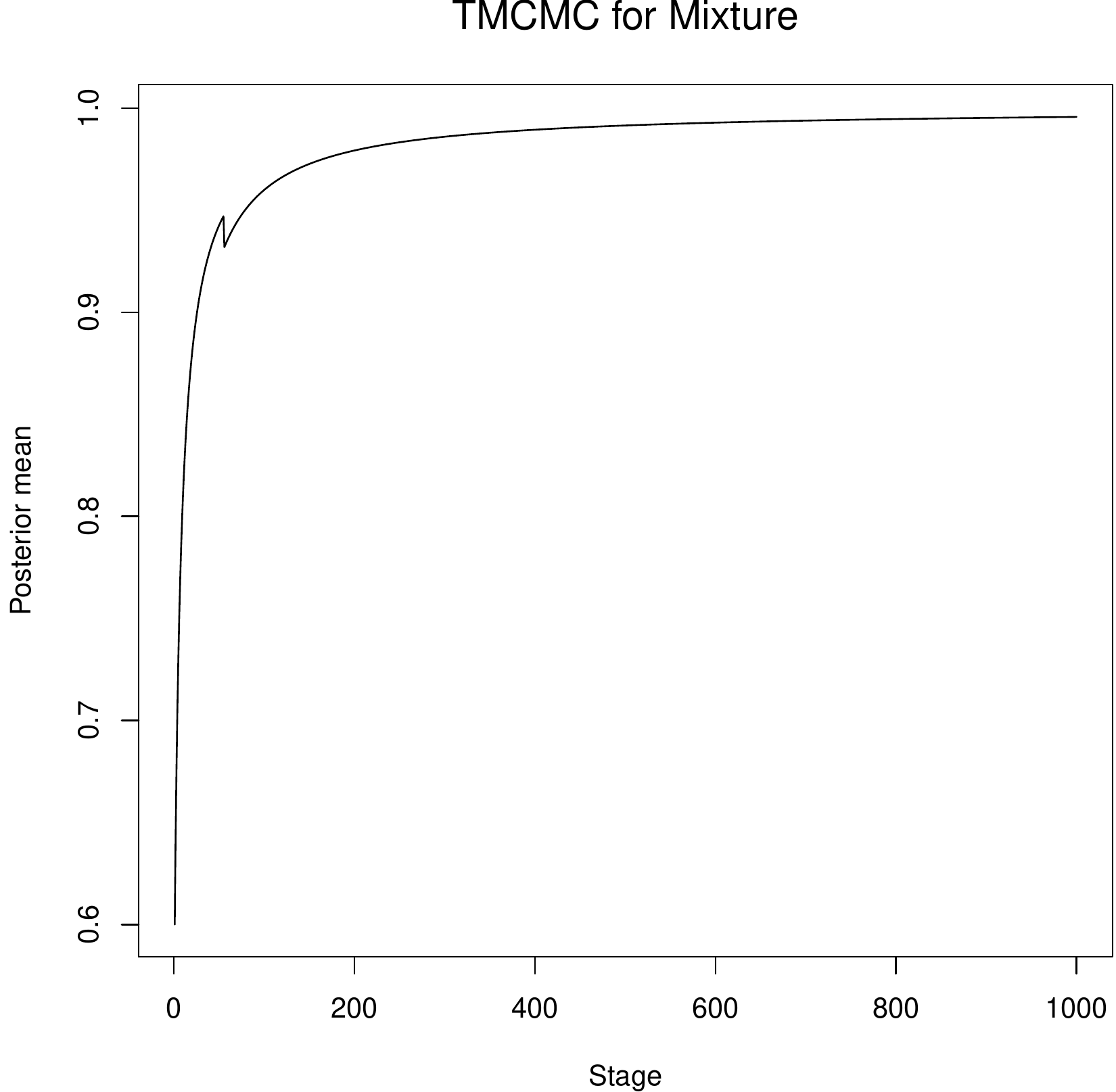}}\\
\vspace{2mm}
\subfigure [TMCMC trace plot for second mixture. Convergence: Nonstationary ($K=2$, $n=500000$)]{ \label{fig:trace_plot9}
\includegraphics[width=5.5cm,height=5.5cm]{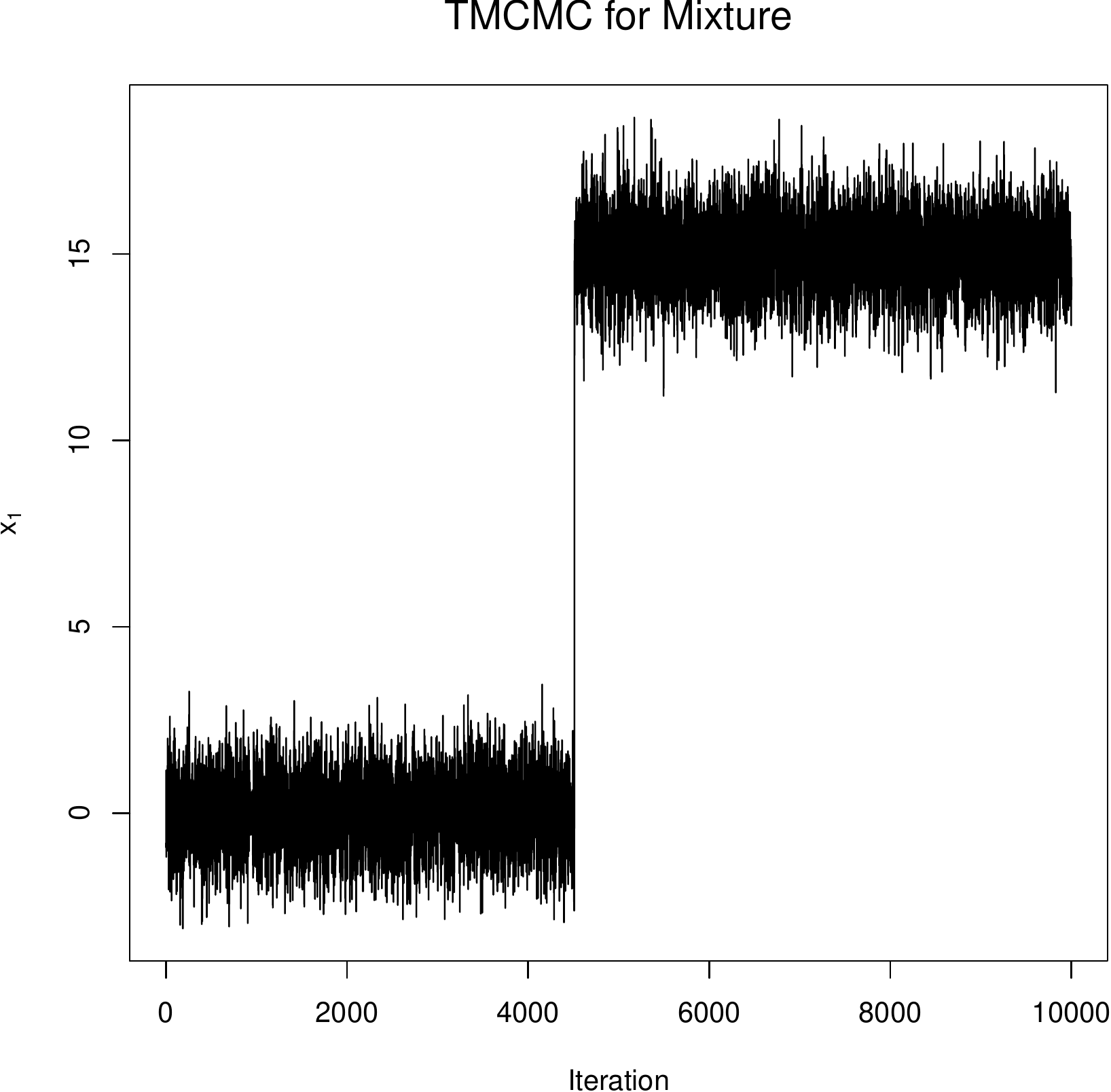}}
\caption{Additive TMCMC convergence example for mixture densities.}
\label{fig:example6}
\end{figure}

\section{Fourth illustration: detection of stationarity and nonstationarity in spatial data}
\label{sec:spatial}
In this illustration, we shall consider detecting both strict and weak stationarity of the spatial processes that gave rise to the observed data.

\subsection{Data generation}
\label{subsec:spatial_data_generation}
We now conduct simulation experiments with our theory for detecting stationarity and nonstationarity in spatial data. 
To conduct the experiment, we simulate two datasets from stationary and nonstationary zero-mean Gaussian processes (GPs)
with covariance functions 
\begin{equation}
Cov(X_{s_1},X_{s_2})=\exp(-5\|s_1-s_2\|^2)
\label{eq:stationary1}
\end{equation}
and 
\begin{equation}
Cov(X_{s_1},X_{s_2})=C_1(\|s_1-s_2\|)=\exp(-5\|\sqrt{s_1}-\sqrt{s_2}\|^2),
\label{eq:nonstationary1}
\end{equation}
for all spatial locations $s_1,s_2\in \mathbb R^2$. For our simulation studies, we restrict the spatial locations to $[0,1]^2$. 
We simulate partial realizations of length 10000
from the two GPs. We begin by simulating first, for $i=1,\ldots,10000$, $\tilde s_i\sim U\left([0,1]^2\right)$, and then setting $s_i=\sqrt{\tilde s_i}$.
Here for any $s=(u,v)^T\in[0,1]^2$, $\sqrt{s}=(\sqrt{u},\sqrt{v})^T$.
The strategy of taking square roots of the components of $\tilde s_i$ ensured numerical stability of the corresponding covariance matrices. 
We then simulate from 10000 zero-mean multivariate normals
with covariance matrices defined by the above stationary and nonstationary covariance functions. Generating from the multivariate normal distributions
by parallelising the required Cholesky decomposition of the covariance matrix and subsequent multiplication of the Cholesky factor with the vector of standard normal 
random variables using ScaLAPACK (Scalable Linear Algebra Package) takes less than $40$ seconds in our C code implementation on our 64 bit laptop
(8 GB RAM and 2.3 GHz CPU speed), with just 4 cores.

\subsection{Implementation of our method to detect strict stationarity}
\label{subsec:spatial_implementation_strict}
For our purpose, we first need to form $\mathcal N_i$; $i=1,\ldots,K$. In the spatial setting, the $K$-means clustering of the locations $s_i$; $i=1,\ldots,10000$,
seems to be very appropriate. The nearby locations based on the distances from the centroid, will be classified within the same cluster, which is desirable
from the spatial perspective. Thus, once we select $K$, the $K$-means clustering yields the $K$ clusters, which are $\mathcal N_i$; $i=1,\ldots,K$ in our notation.
In our example, we select $K=250$, so that there are about $40$ observations per cluster on the average.
We choose the clusterings such that there are at least $15$ observations per cluster.
As before, we consider the general purpose nonparametric bound $c_j$ given by (\ref{eq:ar1_bound3}) for implementation of our method. 

\subsubsection{Choice of $\hat C_1$}
\label{subsubsec:spatial_C1_hat}
For the choice of $\hat C_1$, we first generate a sample of size $10000$ from a zero mean GP with the Whittle covariance function of the form 
\begin{equation}
	Cov(X_{s_1},X_{s_2})=(\|s_1-s_2\|/\psi)\mathcal K_1(\|s_1-s_2\|/\psi), 
\label{eq:spatial_bound}
\end{equation}
where $\mathcal K_1$ is the second kind modified Bessel function of order $1$. For the same value of $\psi$, this covariance function has thicker tails
than exponential correlation functions of the forms $\exp(-\|s_1-s_2\|^2/\psi)$ and $\exp(-\|s_1-s_2\|/\psi)$. We set $\psi=0.8$ to achieve reasonable thickness
of the tail of (\ref{eq:spatial_bound}). With this covariance function, we then use the bound (\ref{eq:ar1_bound3}) and set $\hat C_1$ to be the minimum
positive value such that convergence to $1$ is achieved. This $\hat C_1$ can be interpreted as providing a reasonable bound for spatial processes with
covariance functions with reasonably thick tails, but thinner than that of (\ref{eq:spatial_bound}) with $\psi=0.8$. With this method, we obtain
$\hat C_1=0.89$. This value, being close to $1$, suggests that the default choice $\hat C_1=1$ still makes sense. Indeed, both the choices yielded the same results
regarding the decision on stationarity or nonstationarity of the underlying process.

\subsection{Results}
\label{subsec:spatial_results}
Figure \ref{fig:spatial1} shows the results of implementation of our theory to detect strong stationarity and nonstationarity of the data obtained from
the two GPs. The bounds (\ref{eq:ar1_bound3}) correspond to $\hat C_1=0.89$ obtained using the strategic procedure using (\ref{eq:spatial_bound}). 
Panel (a) correctly asserts strict stationarity when the covariance is of the form (\ref{eq:stationary1}) and correctly detects
strict nonstationarity when the covariance is of the form (\ref{eq:nonstationary1}). 
The entire methodology takes less than a second for parallel implementation on our 64 bit laptop using 4 cores.
\begin{figure}
\centering
\subfigure [Correct detection of stationarity.]{ \label{fig:stationary}
\includegraphics[width=5.5cm,height=5.5cm]{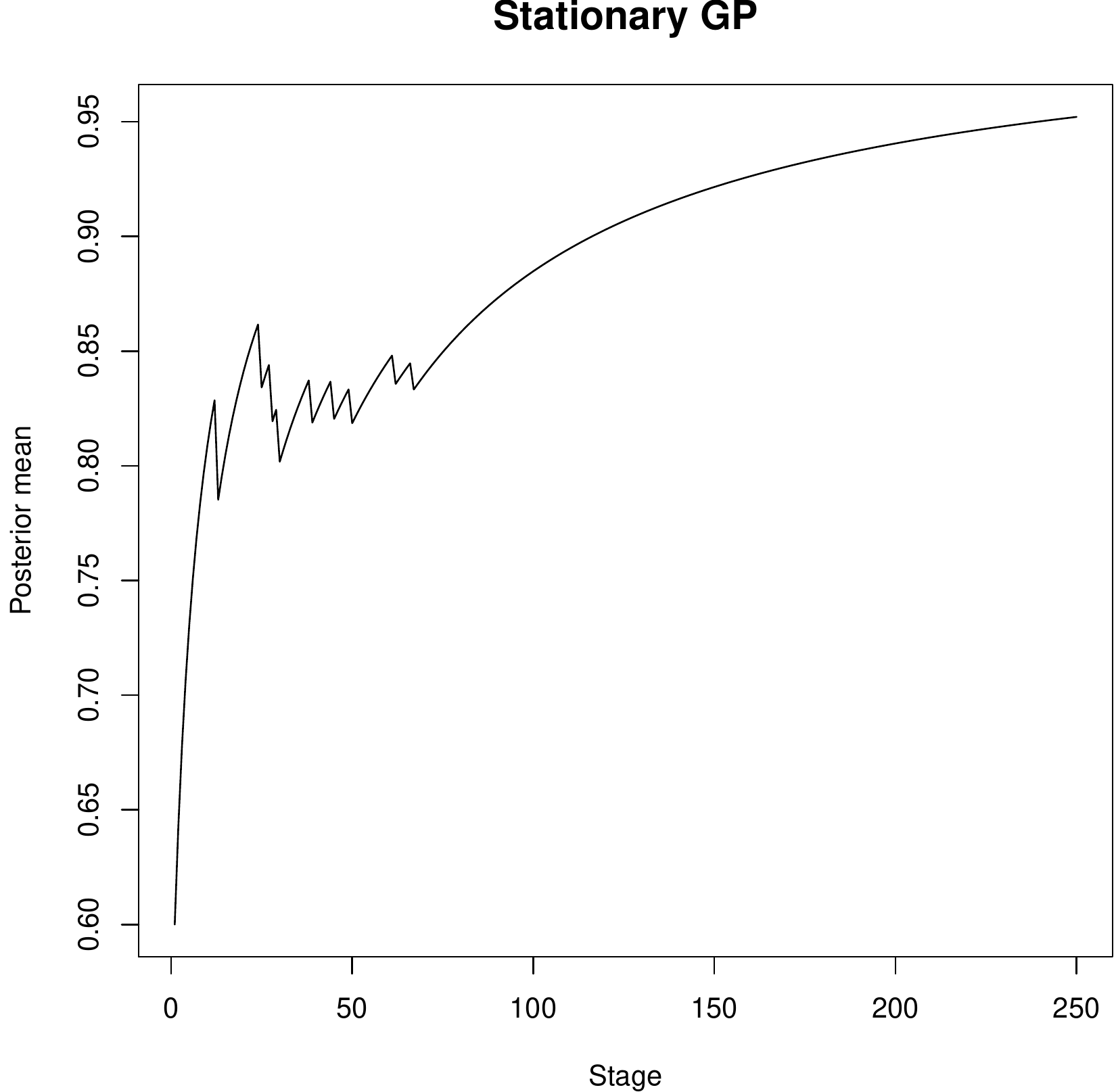}}
\hspace{2mm}
\subfigure [Correct detection of nonstationarity.]{ \label{fig:nonstationary}
\includegraphics[width=5.5cm,height=5.5cm]{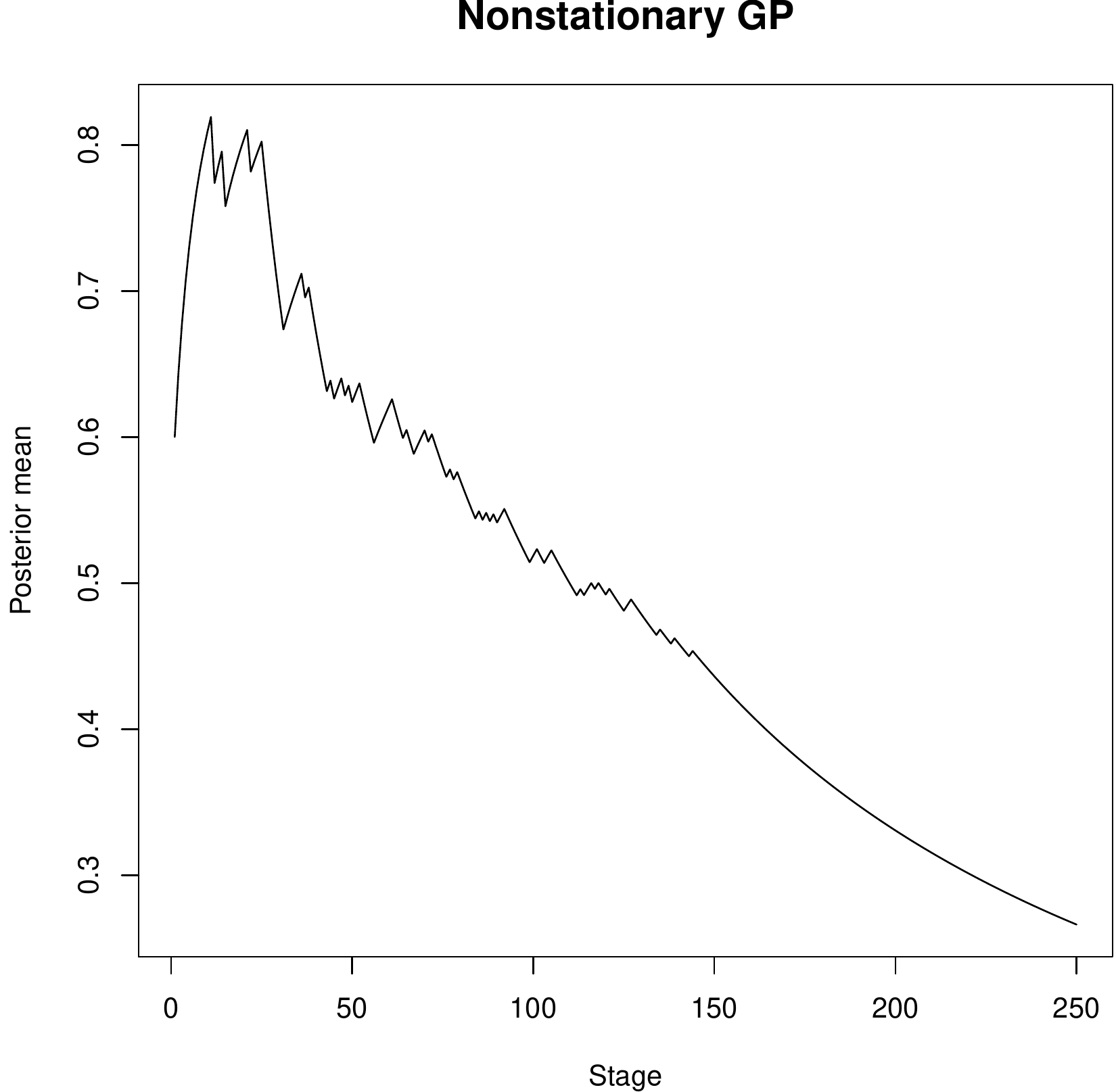}}
\caption{Detection of strong stationarity and nonstationarity in spatial data drawn from GPs.}
\label{fig:spatial1}
\end{figure}

\subsection{Implementation of our method to detect covariance stationarity}
\label{subsec:spatial_implementation_weak}
As we demonstrated, our proposed method does an excellent job in capturing strict stationarity and nonstationarity of the underlying spatial stochastic process. 
In routine spatial modeling, however, strict stationarity and nonstationarity plays little role compared to covariance stationarity and covariance nonstationarity.
Thus, it is more important to detect if the covariance in question is stationary or not. Although in our example it directly follows from our tests of strict stationarity
that the covariances for the two GPs must be stationary and nonstationary, we directly check covariance stationarity using our Bayesian method
formalized in Theorems \ref{theorem:convergence2} and \ref{theorem:divergence2}.

For practical implementation, we convert the covariances $\widehat {Cov}_{ih}$ given by (\ref{eq:cov1}) into correlations by dividing them by the relevant standard errors 
and initially set $\mathcal N_{i,h_j,h_{j+1}}=\left\{(s_1,s_2)\in\mathcal N_i:h_j\leq\|s_1-s_2\|<h_{j+1}\right\}$; $j=1,\ldots,10$, where
$h_1=0$ and $h_j=h_{j-1}+0.1$, for $j=2,\ldots,10$. We consider the nonparametric bound $c_j$ given by (\ref{eq:ar1_bound3}) for all $j=1,\ldots,10$, for both the GPs.
But we found that these $\mathcal N_{i,h_j,h_{j+1}}$ are too large to be useful, as $0<\|s_1-s_2\|<0.04$, for all $(s_1,s_2)$ in most of the $K$-means clusters that we obtained.
Indeed, only three neighborhoods defined by $h_1=0$, $h_2=0.02$, $h_3=0.03$ and $h_4=0.04$, turned out to be appropriate.

We again fix $K=250$ clusters such that each cluster contains at least $15$ observations.

\subsubsection{Choice of $\hat C_1$}
\label{subsubsec:spatial_cov_C1_hat}

To obtain appropriate choice of $\hat C_1$ for detecting covariance stationarity, we consider three strategies. 
Our first method in this regard corresponds to using $\hat C_1$ for strict stationarity. Thus, the first startegy yields $\hat C_1=0.89$. 

For the second strategy, we utilize the GP realization with covariance function (\ref{eq:spatial_bound}). Here we choose the minimum value of $\hat C_1$
such that (\ref{eq:ar1_bound3}) yielded convergence to $1$ for all $\mathcal N_{i,h_j,h_{j+1}}$; $j=1,2,3$. This gave $\hat C_1=0.412$. 

In the third strategy, we chose the minimum value of $\hat C_1$ that yielded convergence to $1$ for all $\mathcal N_{i,h_j,h_{j+1}}$; $j=1,2,3$ for one dataset and 
convergence to $0$ for the other dataset.
In our case, this strategy again gave $\hat C_1=0.412$. 

The strategic choice $\hat C_1=0.412$ 
successfully detected covariance stationarity and nonstationarity. However, the choice $\hat C_1=0.89$ turned out
to be too large to detect covariance nonstationarity. This is in keeping with the issue that detection of strict stationarity requires a bound that
must also ensure covariance stationarity, and hence such a bound must be larger than that for covariance stationarity.

Again, our parallel implementation takes less than a second on our laptop, for each $\mathcal N_{i,h_j,h_{j+1}}$. This quick computation ensures that choice of $\hat C_1$
is not a computationally demanding exercise.

Figure \ref{fig:spatial2} shows the results associated with $\mathcal N_{i,h_1,h_2}$, $\mathcal N_{i,h_2,h_3}$ and $\mathcal N_{i,h_3,h_4}$,
for $i=1,\ldots,K$, where $K=250$ as before, and $\hat C_1=0.89$. The figure shows that whenever the data arises from the GP with covariance of the form (\ref{eq:stationary1}), 
our Bayesian method correctly identifies covariance stationarity for every $j$. Indeed, for all $j=1,2,3$, covariance stationarity is clearly indicated. 
On the other hand, when the data arises from the GP with the nonstationary covariance (\ref{eq:nonstationary1}), 
convergence to $0$ is indicated with $\mathcal N_{i,h_3,h_4}$. As per Theorem \ref{theorem:divergence2},
this shows nonstationarity of the covariance structure. 

\begin{figure}
\centering
\subfigure [$0\leq\|h\|<0.02$]{ \label{fig:covs1}
\includegraphics[width=5.5cm,height=5.5cm]{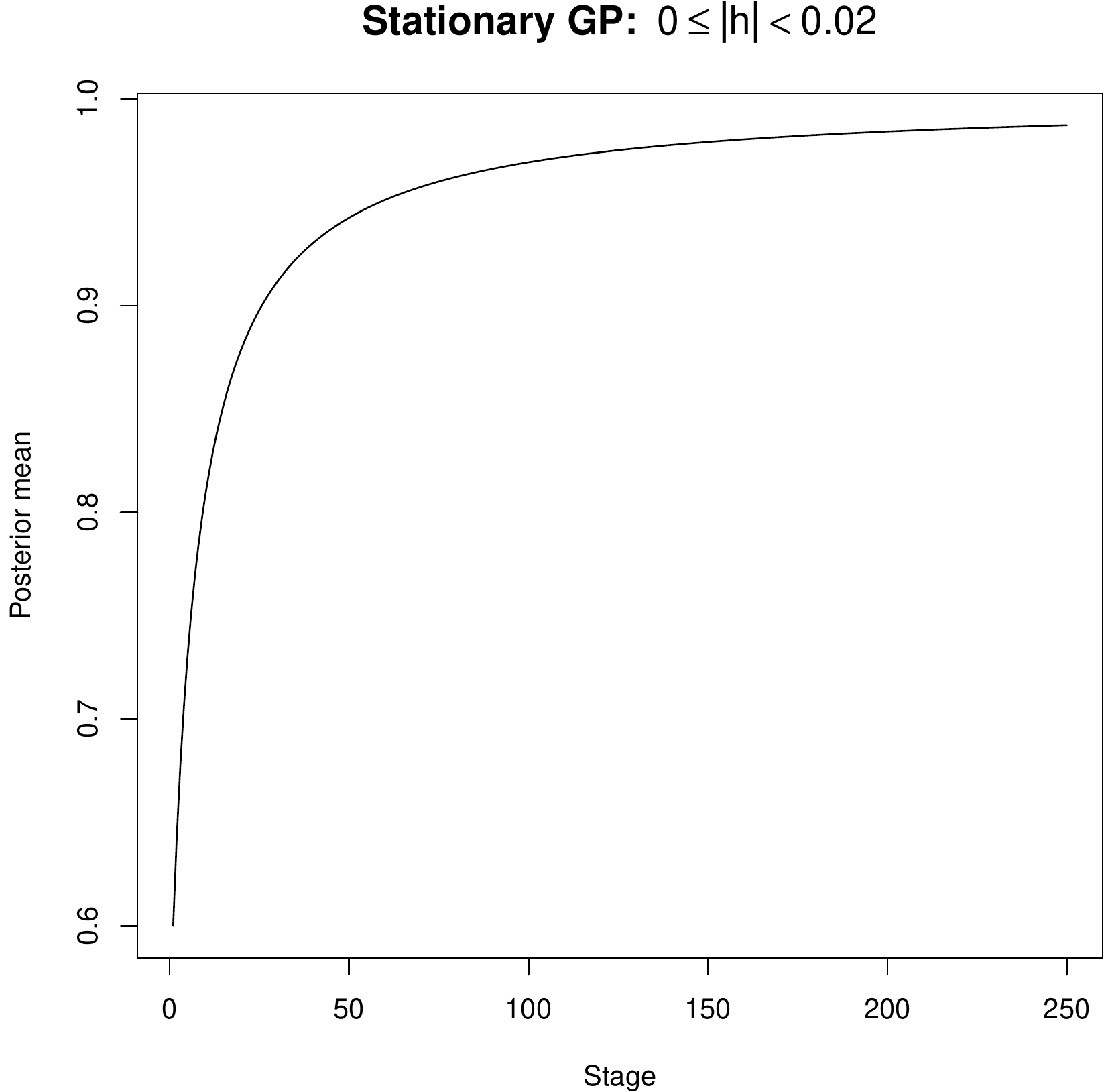}}
\hspace{2mm}
\subfigure [$0\leq\|h\|<0.02$.]{ \label{fig:covns1}
\includegraphics[width=5.5cm,height=5.5cm]{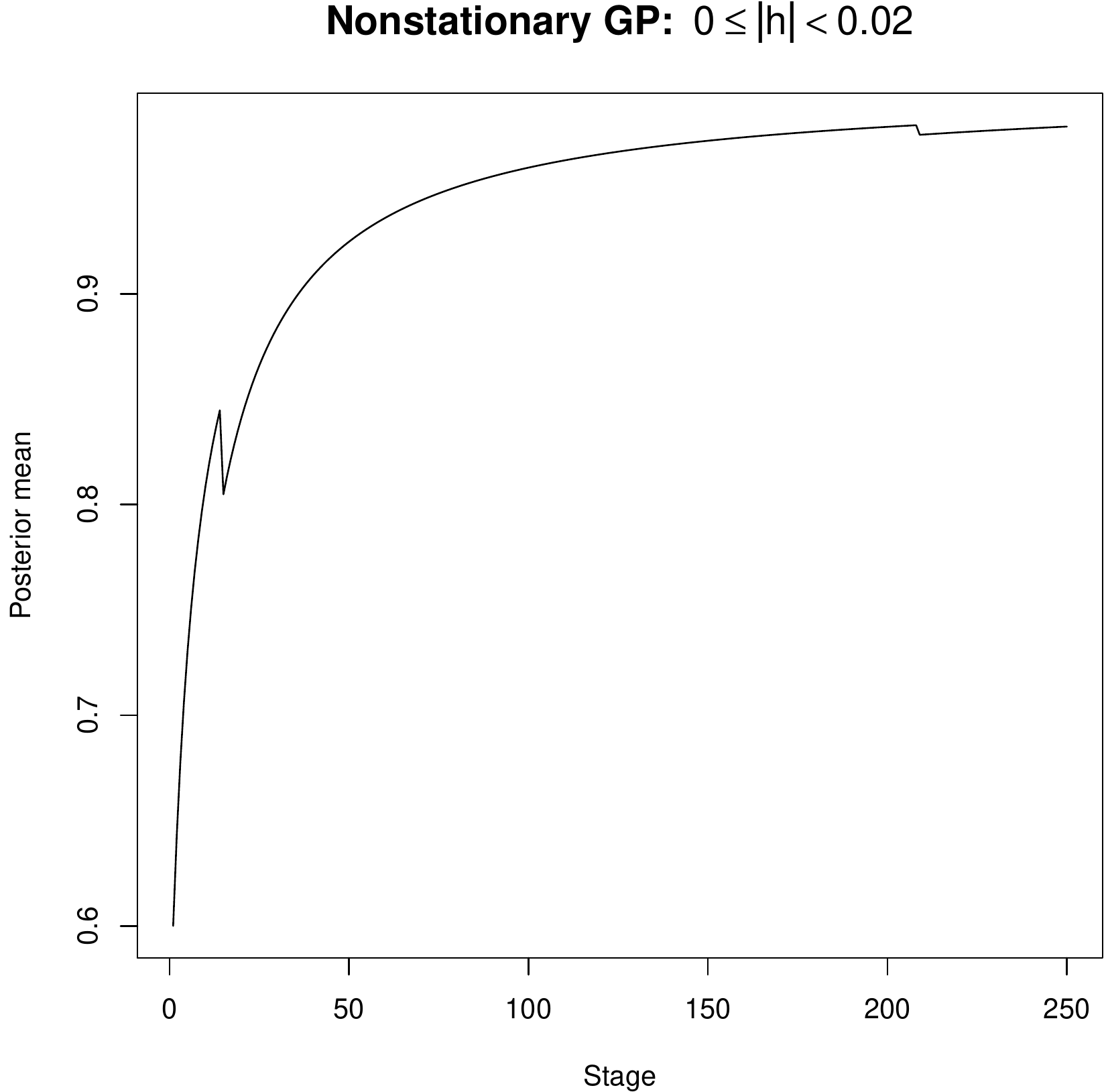}}\\
\vspace{2mm}
\subfigure [$0.02\leq\|h\|<0.03$]{ \label{fig:covs2}
\includegraphics[width=5.5cm,height=5.5cm]{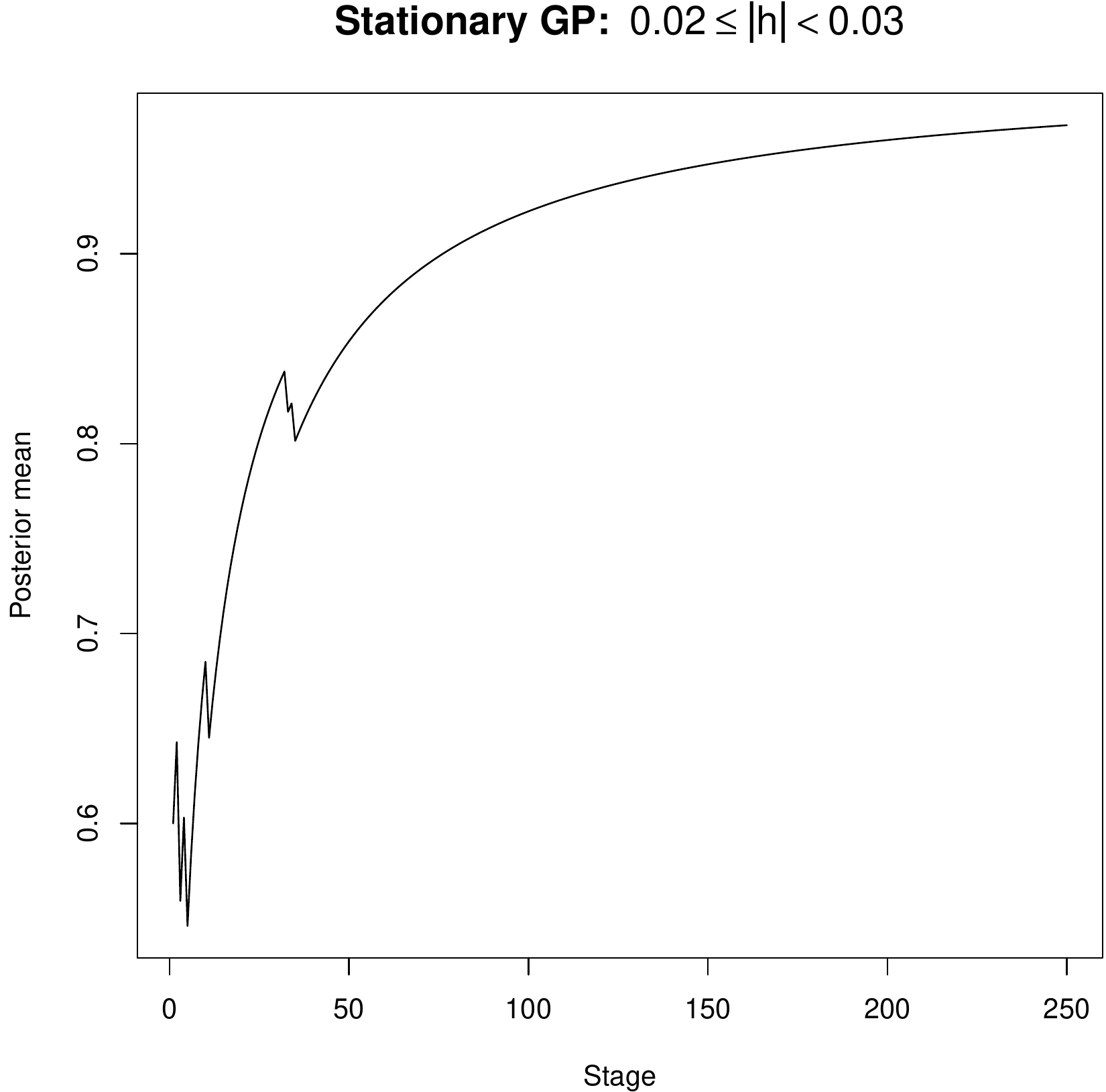}}
\hspace{2mm}
\subfigure [$0.02\leq\|h\|<0.03$.]{ \label{fig:covns2}
\includegraphics[width=5.5cm,height=5.5cm]{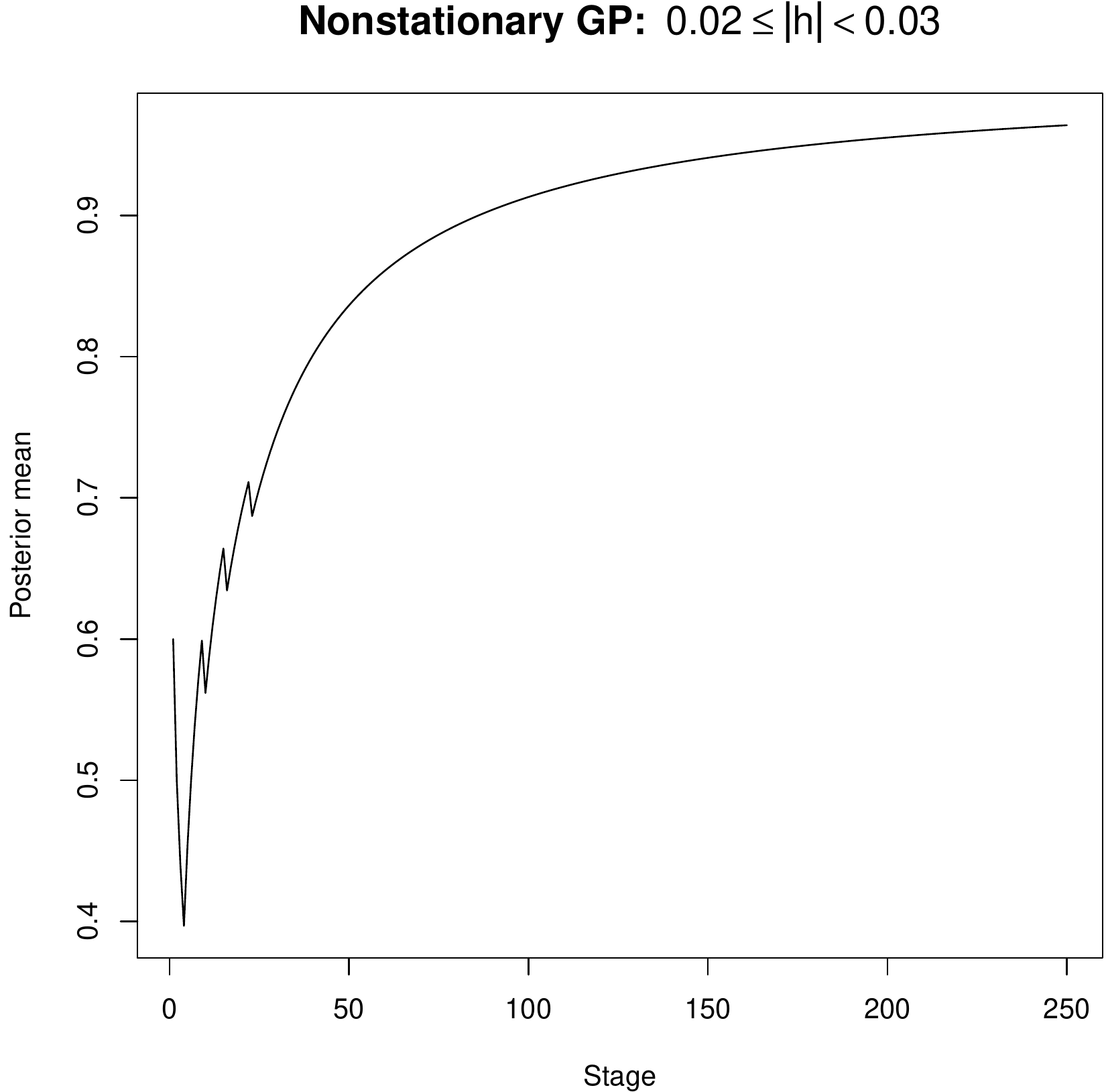}}\\
\vspace{2mm}
\subfigure [$0.03\leq\|h\|<0.04$]{ \label{fig:covs3}
\includegraphics[width=5.5cm,height=5.5cm]{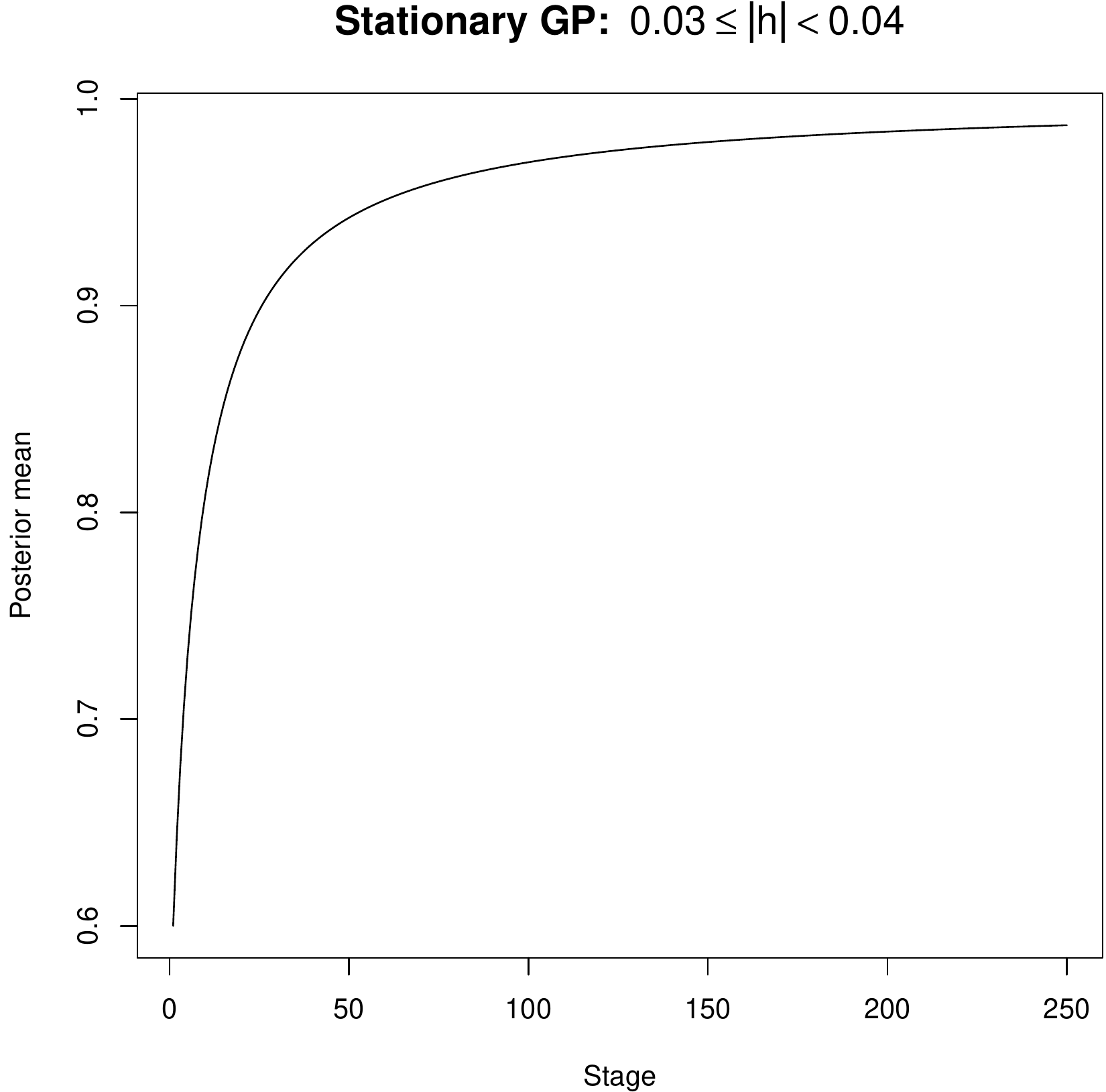}}
\hspace{2mm}
\subfigure [$0.03\leq\|h\|<0.04$.]{ \label{fig:covns3}
\includegraphics[width=5.5cm,height=5.5cm]{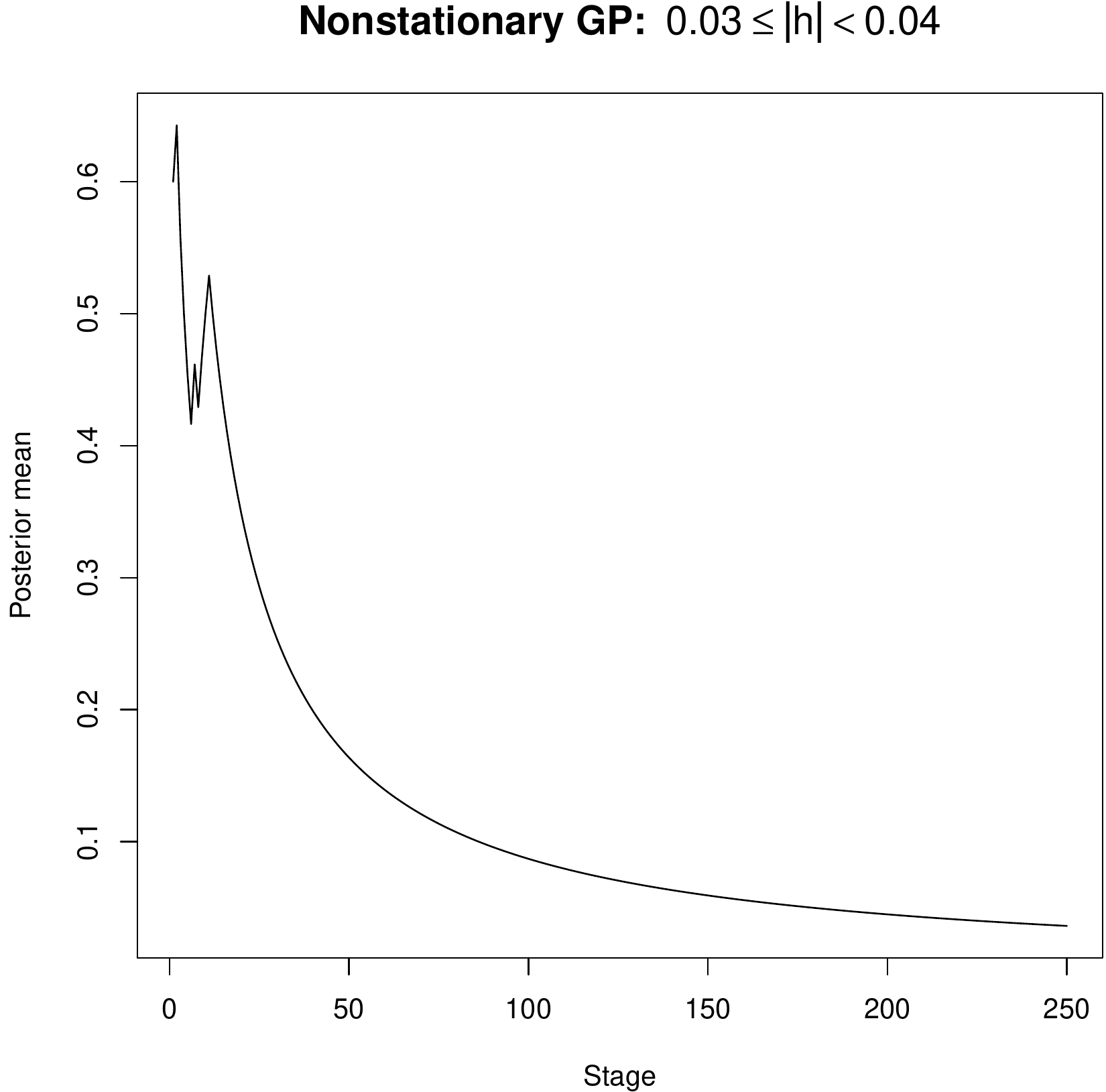}}\\
\caption{Detection of covariance stationarity and nonstationarity in spatial data drawn from GPs.}
\label{fig:spatial2}
\end{figure}

\subsection{Detection of strict nonstationarity in mixtures of stationary and nonstationary covariances}
\label{subsec:mixture_stationarity}
We now consider realizations from zero-mean GPs with covariances of the form
\begin{equation}
Cov(X_{s_1},X_{s_2})=p\exp(-5\|s_1-s_2\|^2)+(1-p)\exp(-5\|\sqrt{s_1}-\sqrt{s_2}\|^2),
\label{eq:nonstationary2}
\end{equation}
where $0<p<1$. In particular, using our Bayesian theory, we attempt to detect strict and weak nonstationarity of the process when $p=0.9,0.99,0.999,0.9999,0.99999$.  
Note that in theses cases, although most of the weight concentrates on the stationary part of (\ref{eq:nonstationary2}), the little mass on the nonstationary part
makes the covariance nonstationary, and it is important to detect such subtle difference between stationarity and nonstationarity. As before, we set $K=250$
clusters with each cluster containing at least $15$ observations.

We consider the same way of data generation from GP as before, and the same way of implementation. We again use the same form of the bound $c_j$ as (\ref{eq:ar1_bound3}),
with $\hat C_1=0.89$ and $\hat C_1=1$ for detection of strict nonstationarity, as before. These choices put up excellent performances and are in agreement
with each other, in spite of the subtlety involved in this exercise.
Figure \ref{fig:spatial3}, corresponding to $\hat C_1=0.89$, shows that our Bayesian method correctly identifies nonstationarity in all the cases.

\begin{figure}
\centering
\subfigure [$p = 0.9$.]{ \label{fig:mixed_9}
\includegraphics[width=5.5cm,height=5.5cm]{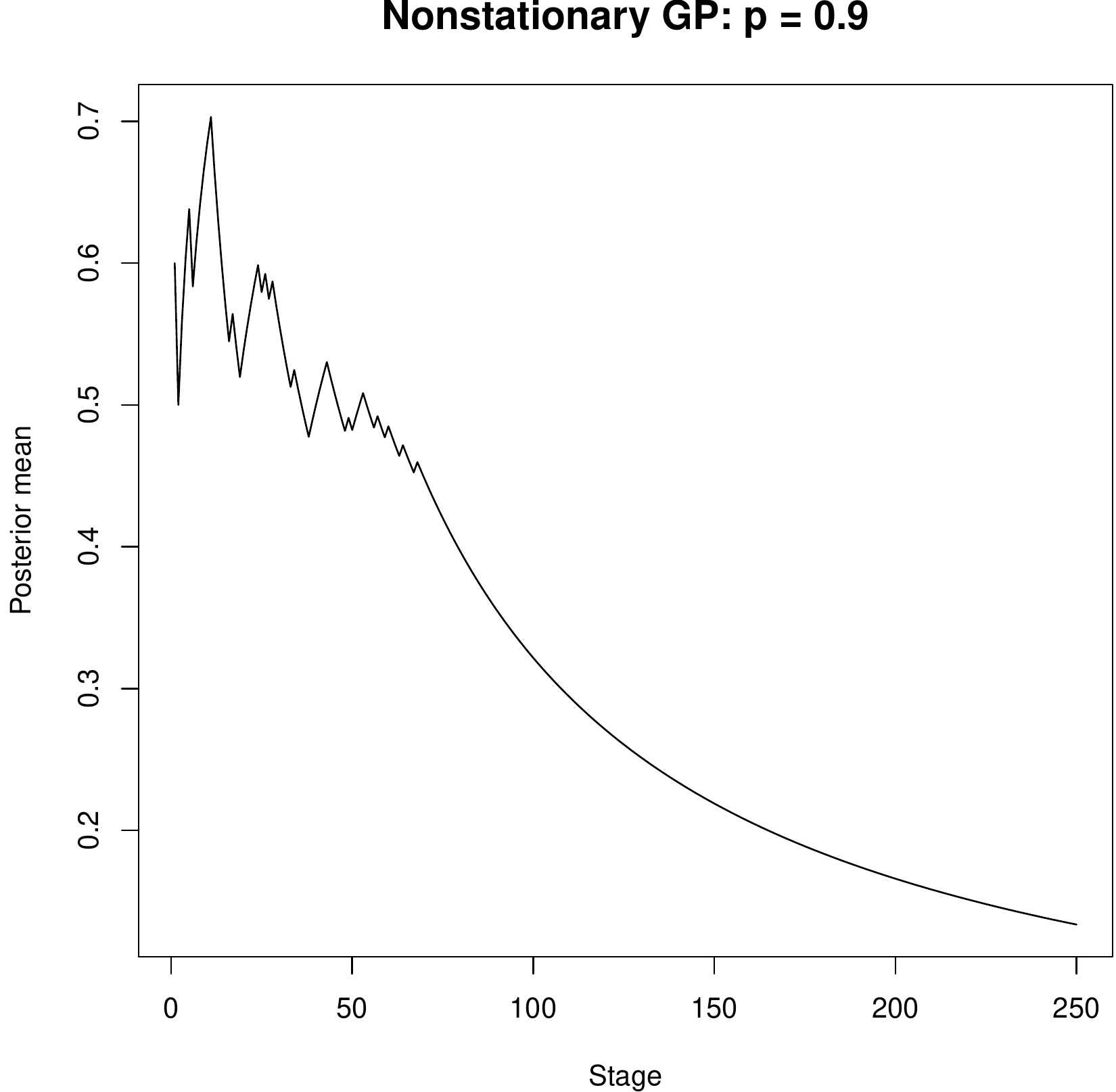}}
\hspace{2mm}
\subfigure [$p = 0.99$.]{ \label{fig:mixed_99}
\includegraphics[width=5.5cm,height=5.5cm]{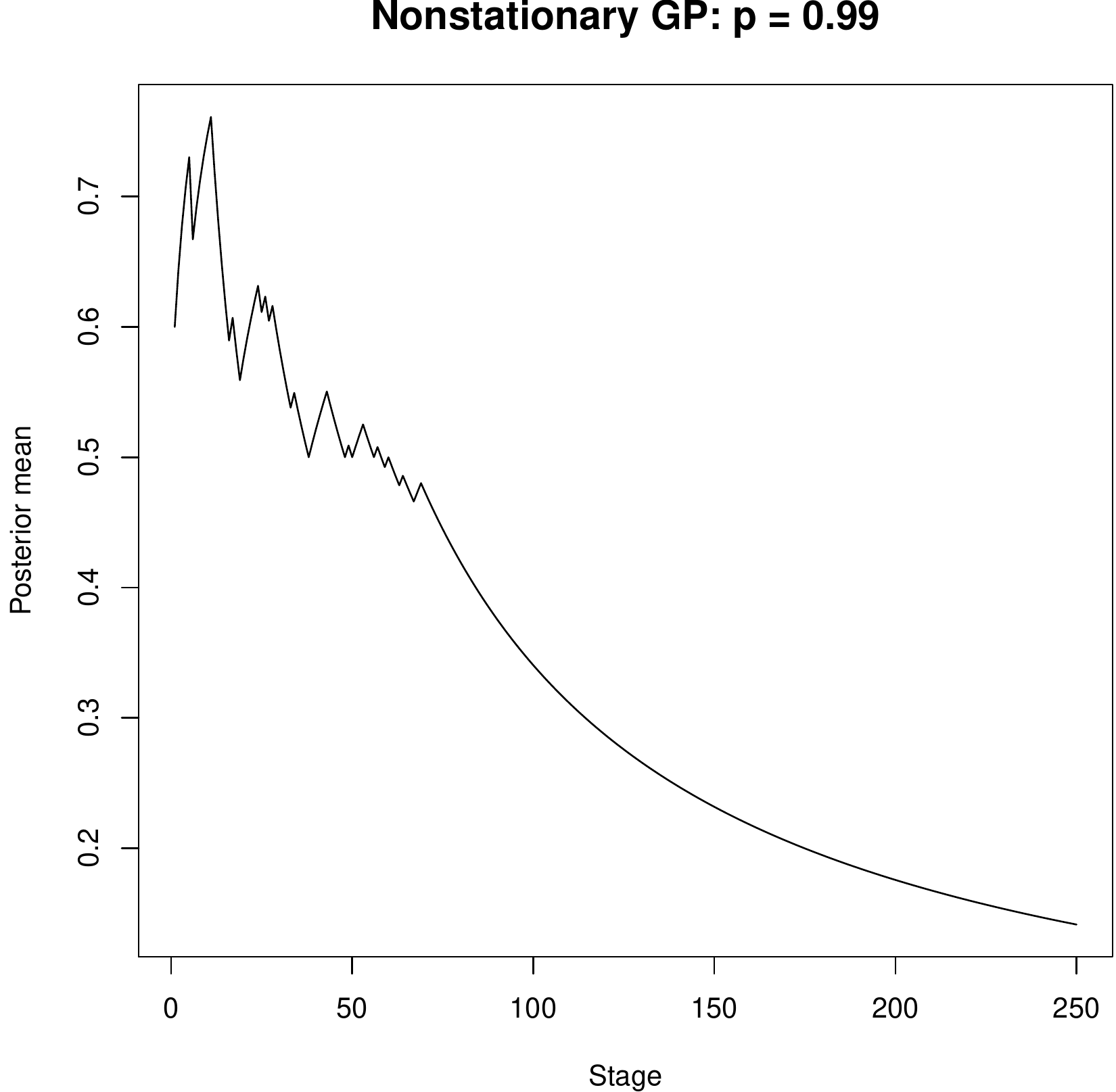}}\\
\vspace{2mm}
\subfigure [$p = 0.999$.]{ \label{fig:mixed_999}
\includegraphics[width=5.5cm,height=5.5cm]{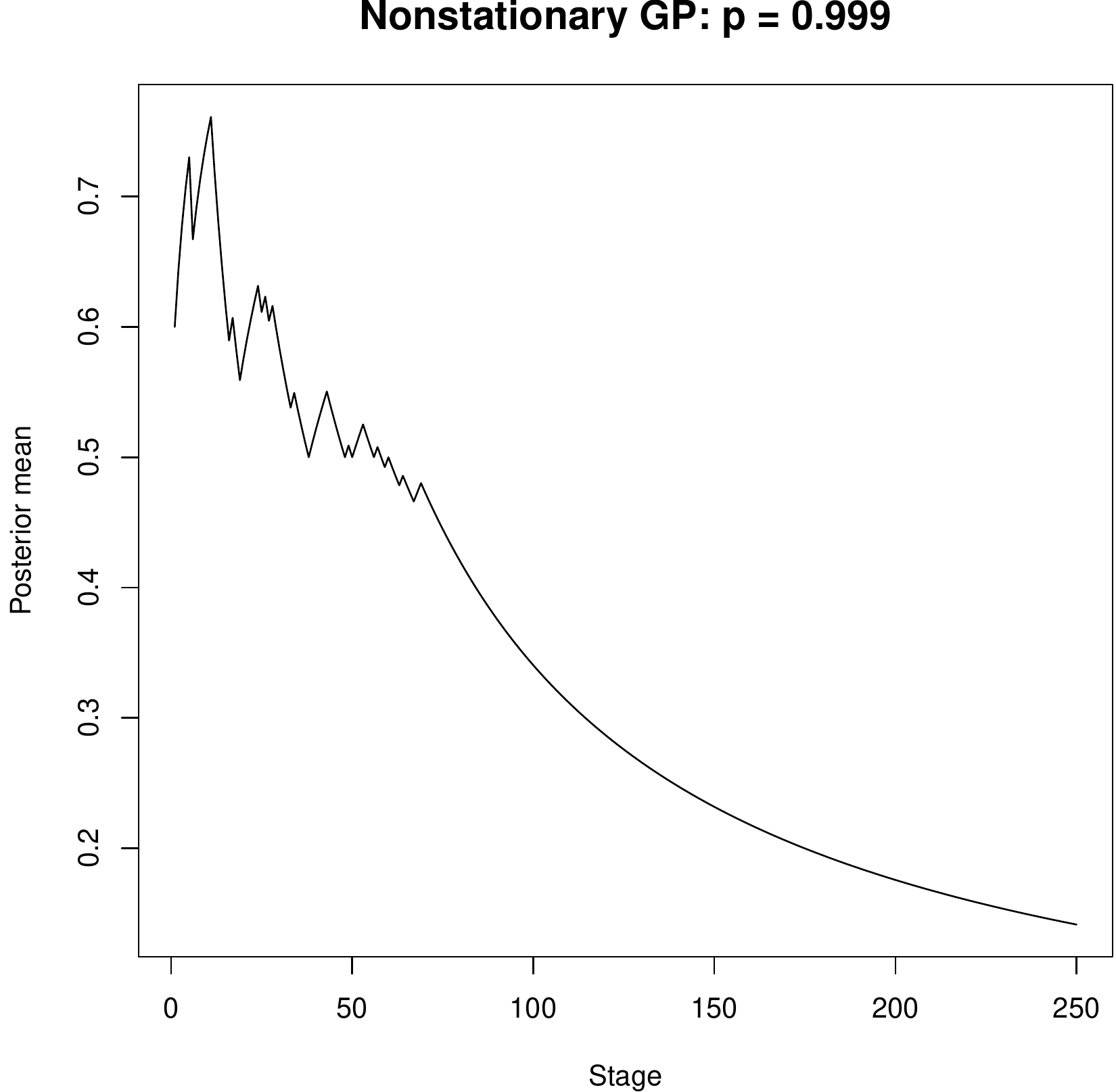}}
\hspace{2mm}
\subfigure [$p = 0.9999$.]{ \label{fig:mixed_9999}
\includegraphics[width=5.5cm,height=5.5cm]{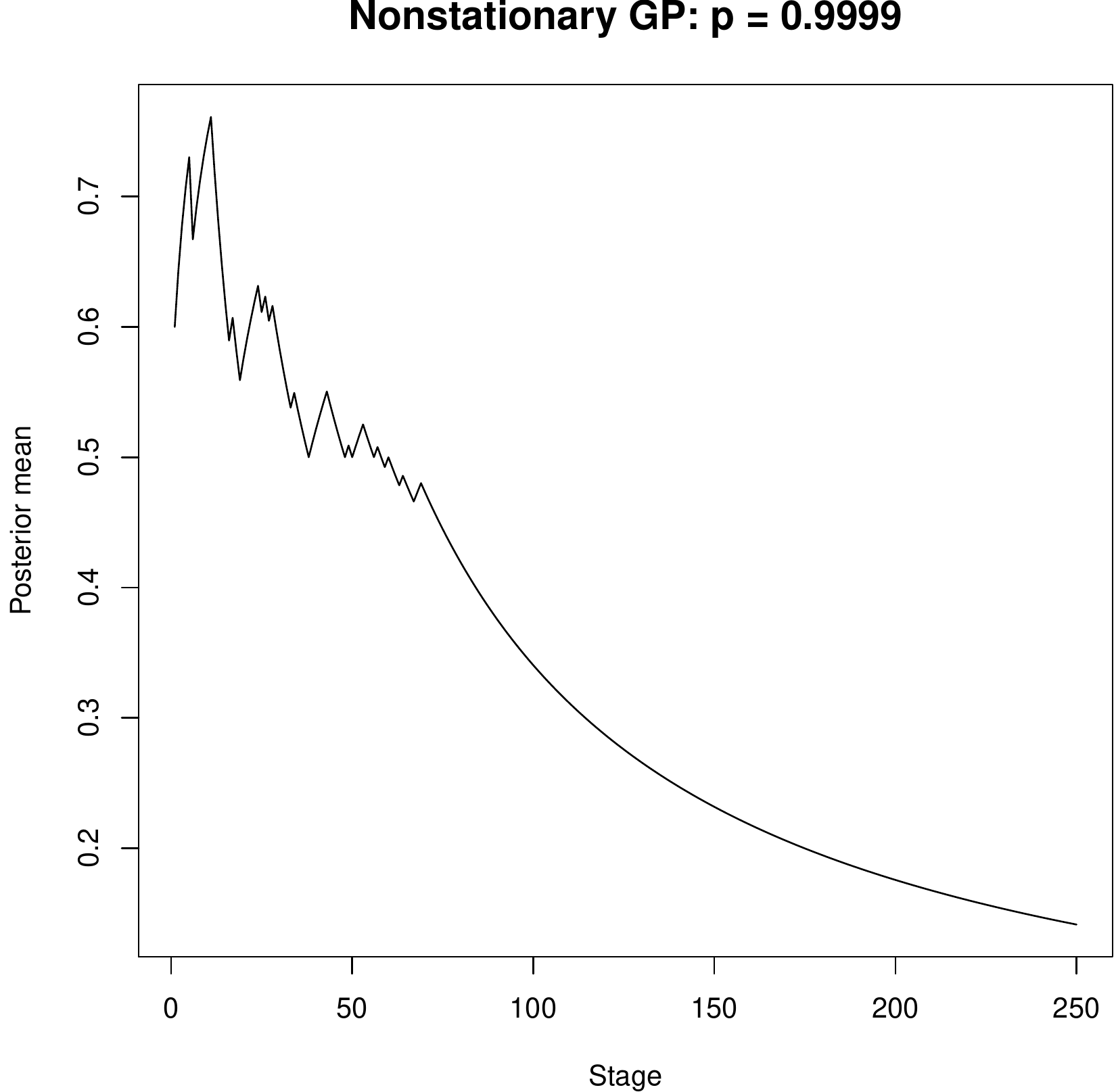}}\\
\vspace{2mm}
\subfigure [$p = 0.99999$.]{ \label{fig:mixed_99999}
\includegraphics[width=5.5cm,height=5.5cm]{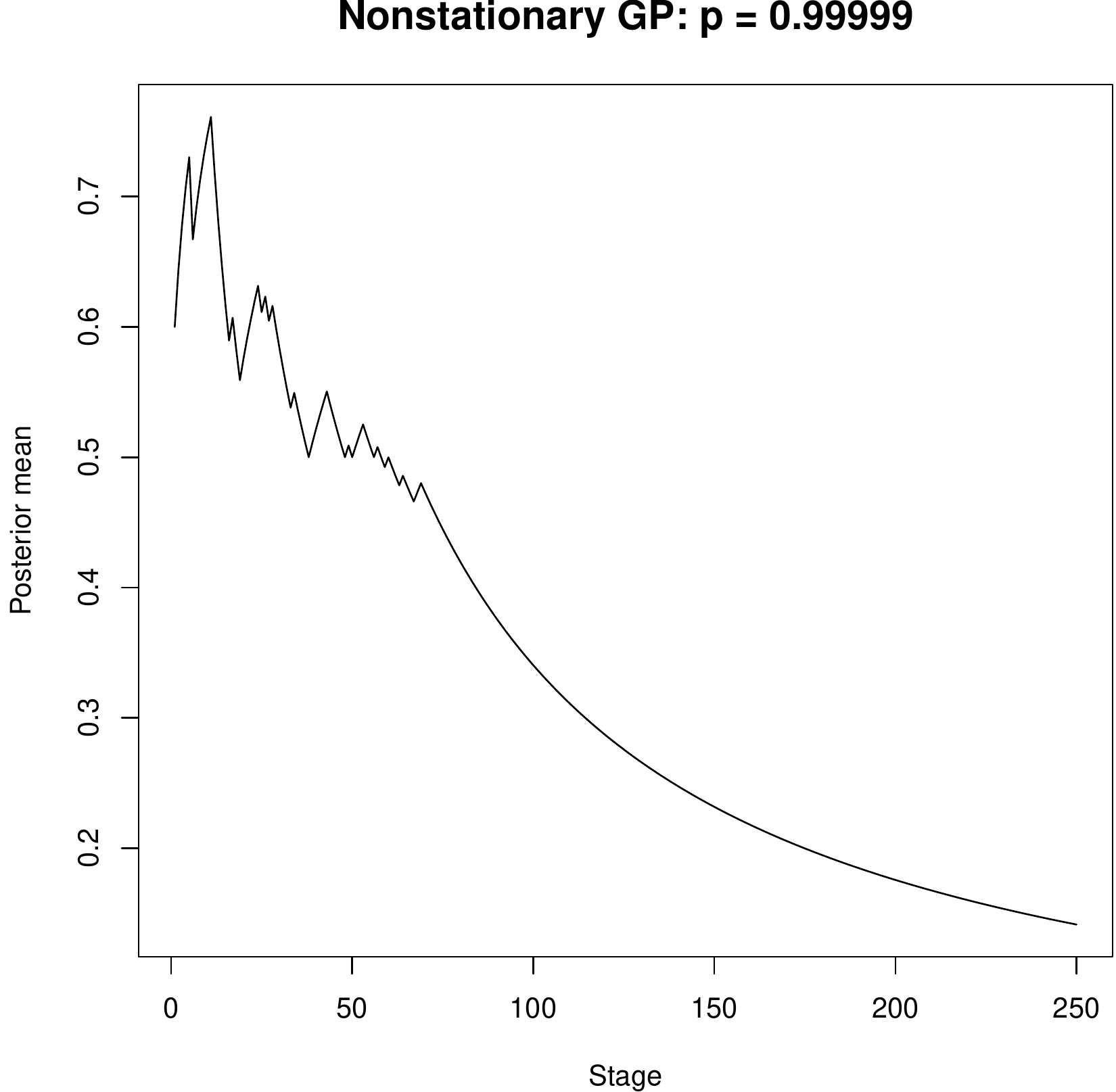}}\\
\caption{Detection of strong nonstationarity in spatial data drawn from GP with covariance structure (\ref{eq:nonstationary2}) with $p=0.99999$.}
\label{fig:spatial3}
\end{figure}

\subsection{Detection of covariance nonstationarity in mixtures of stationary and nonstationary covariances}
\label{subsec:mixture_cov}

The same strategies
discussed in Section \ref{subsec:spatial_implementation_weak}, adapted in this situation, yielded effective bounds of the form (\ref{eq:ar1_bound3}) with 
$\hat C_1=0.412$, 
as before. 
We briefly discuss the second procedure of adapting the strategy to the current scenario. Note that the first procedure does not need any change at all. 

To implement our second strategy in this case, we need a benchmark dataset for which covariance stationarity has been established. We thus consider the GP data with 
covariance of the form (\ref{eq:stationary1}), whose covariance stationarity is established. For any new dataset for which
covariance stationarity needs to be checked, in this case, any dataset
with covariance structure of the form (\ref{eq:nonstationary2}), we consider the same bound starting with $\hat C_1=0.89$. We then gradually decrease $\hat C_1$
for both the datasets until we arrive at a point that discriminates covariance stationarity and nonstationarity, in the same way as discussed in 
Section \ref{subsec:spatial_implementation_weak}. With this method, we obtain $\hat C_1=0.412$, 
which shows covariance stationarity for (\ref{eq:stationary1}) but covariance nonstationarity for (\ref{eq:nonstationary2}). 
Recall that $\hat C_1=0.412$ also resulted with respect to the GP realization for the Whittle covariance function (\ref{eq:spatial_bound}).

Again we set $K=250$, with each cluster consisting of a minimum of $15$ observations.
%
Figure \ref{fig:spatial4}, corresponding to $\hat C_1=0.412$ 
and $p=0.99999$ in the covariance structure (\ref{eq:nonstationary2}), shows that this procedure
does an excellent job in detecting covariance nonstationarity even in such a subtle situation. Indeed, the same $\hat C_1=0.412$ 
very successfully captured covariance nonstationarity for all other values of $p$, namely, $p=0.9,0.99,0.999,0.9999$ (figures omitted for brevity).

\begin{figure}
\centering
\subfigure [$0\leq\|h\|<0.02$]{ \label{fig:mixed_covns1}
\includegraphics[width=5.5cm,height=5.5cm]{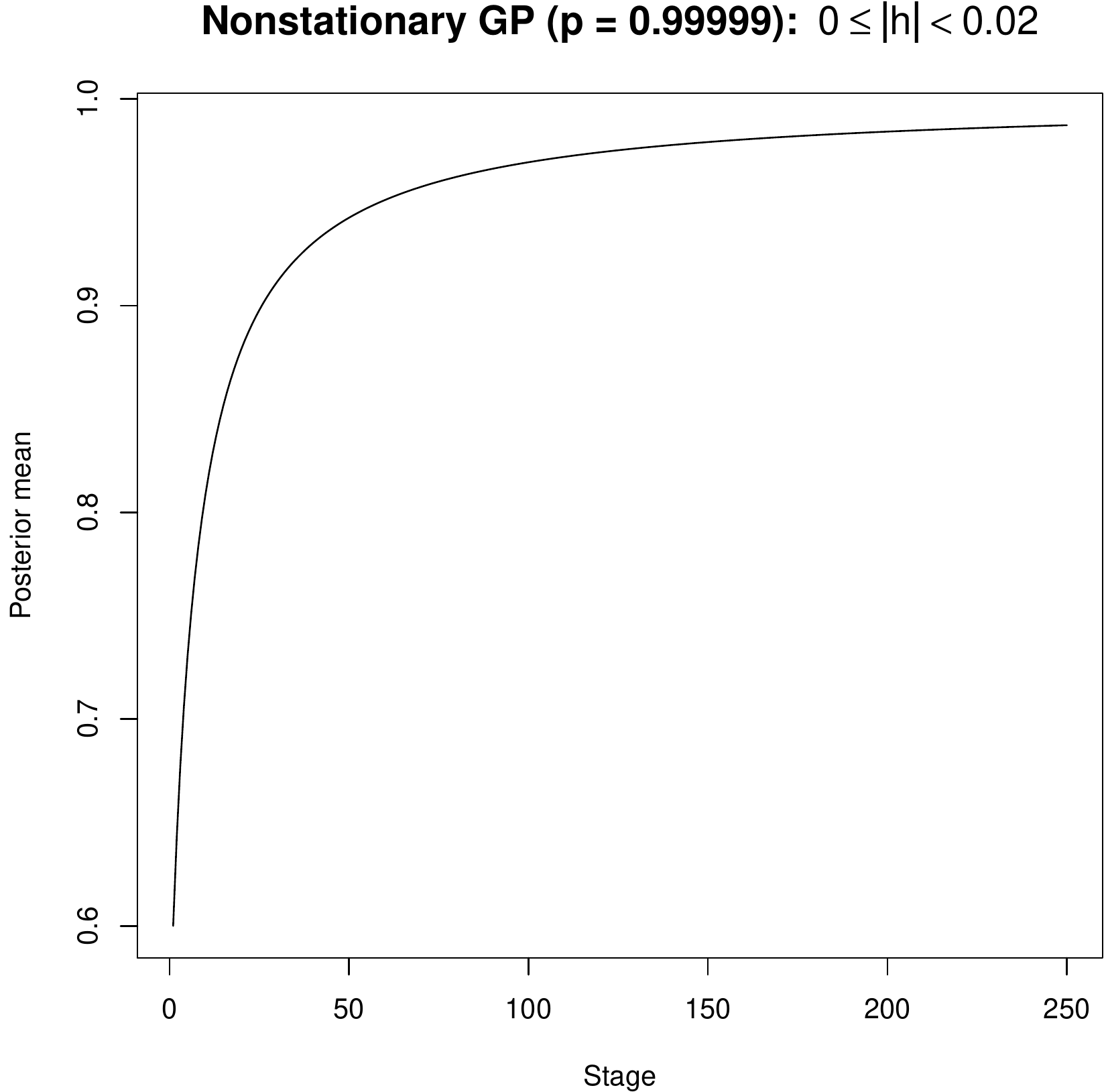}}
\hspace{2mm}
\subfigure [$0.02\leq\|h\|<0.03$.]{ \label{fig:mixed_covns2}
\includegraphics[width=5.5cm,height=5.5cm]{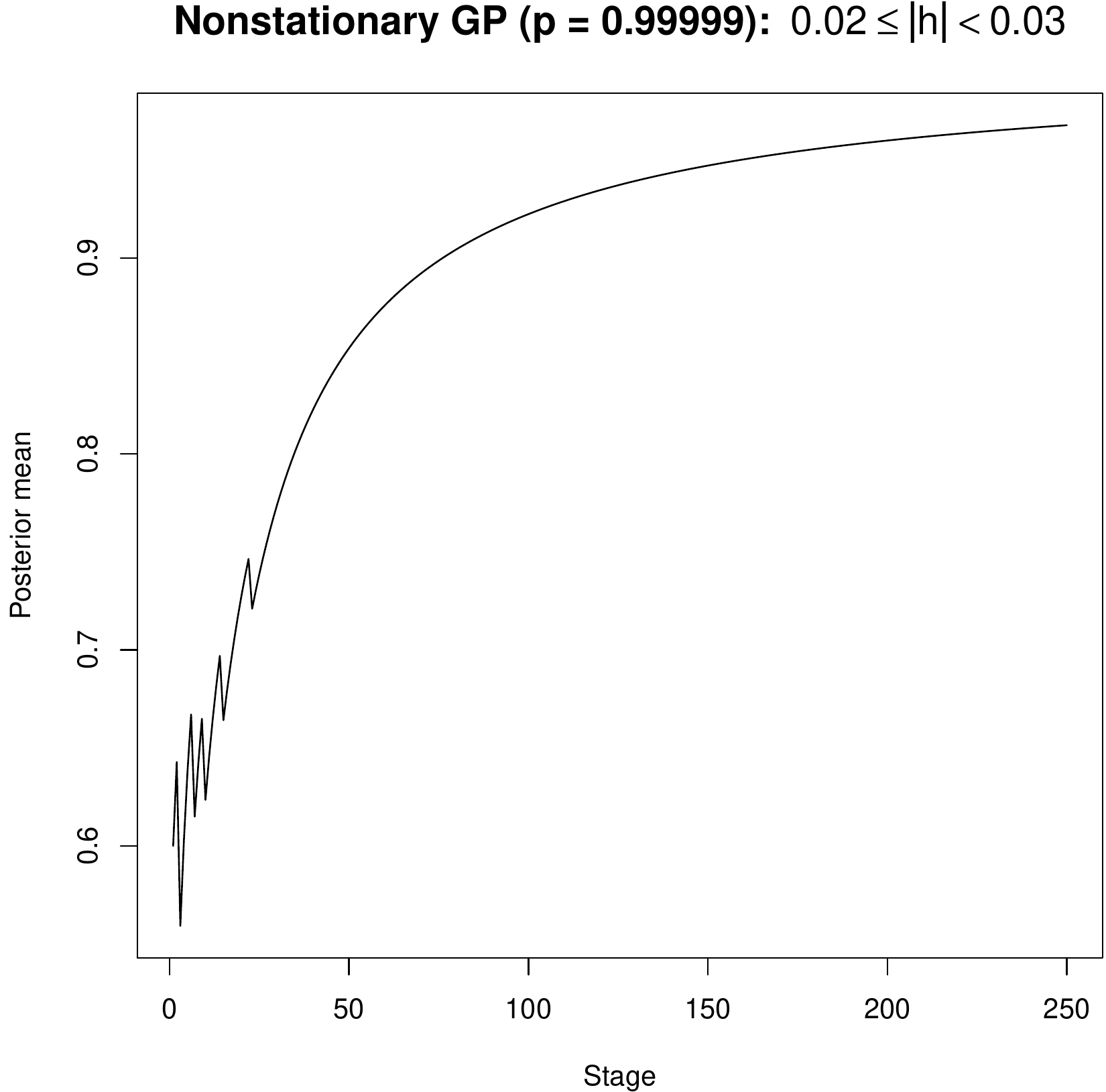}}\\
\vspace{2mm}
\subfigure [$0.03\leq\|h\|<0.04$]{ \label{fig:mixed_covns3}
\includegraphics[width=5.5cm,height=5.5cm]{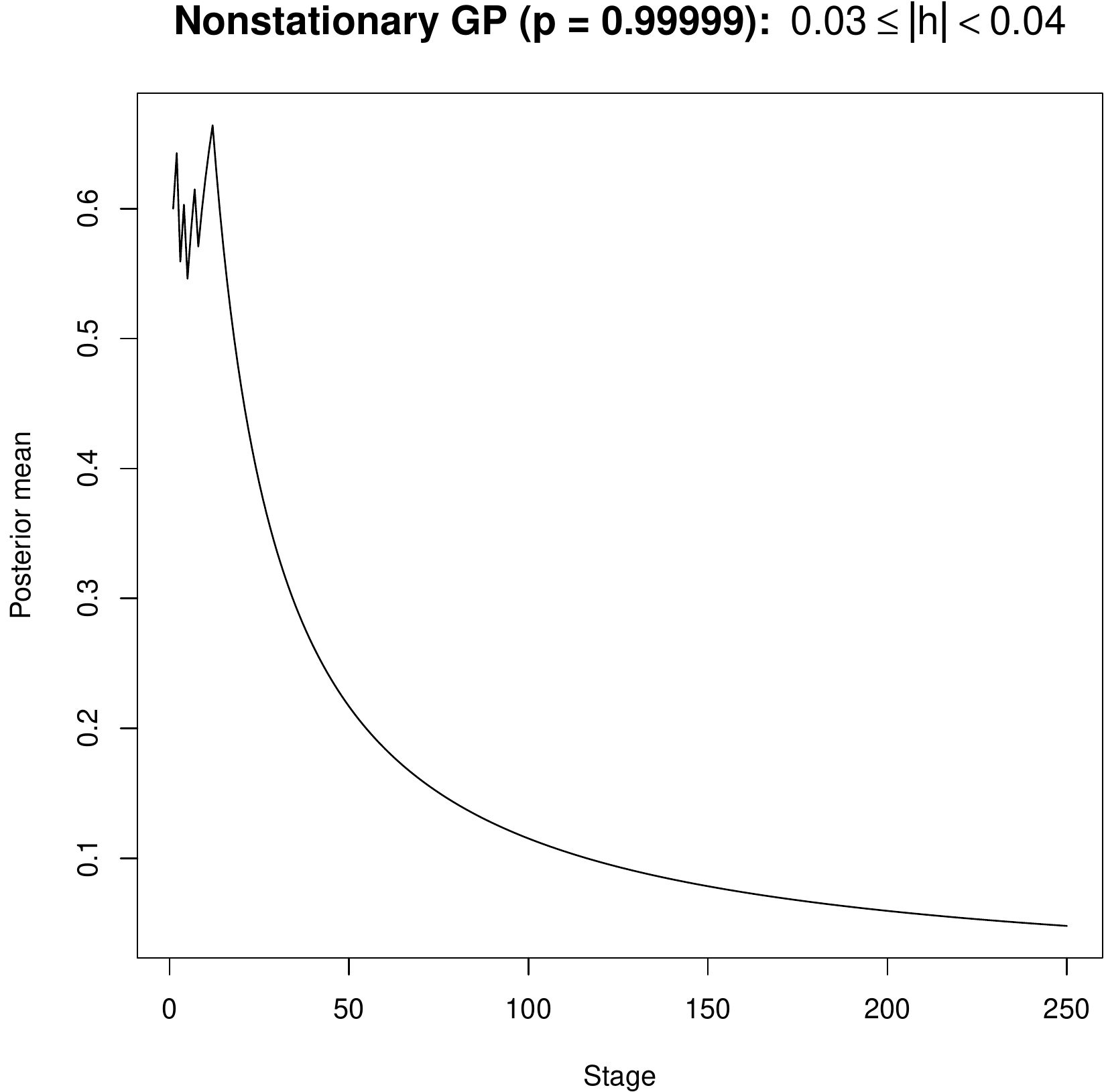}}\\
\caption{Detection of covariance nonstationarity in spatial data drawn from GP with covariance structure (\ref{eq:nonstationary2}) with $p=0.99999$.}
\label{fig:spatial4}
\end{figure}

\subsection{Spatial experiments with smaller data sets}
\label{subsec:spatial_smalldata}
We now repeat all the above experiments with datasets of sizes $1000$. We consider $K=100$ clusters with average cluster size $10$.
For checking strict stationarity, our first strategy of fixing $\hat C_1$, using the Whittle covariance function (\ref{eq:spatial_bound}) yielded $\hat C_1=0.02$, which produced 
too small bounds to be useful. On the other hand, the second procedure gave $\hat C_1=1.24$, which yielded reliable results, even for these small data sets.
Figures \ref{fig:spatial_short} and \ref{fig:spatial_mixture_short} depict the results for $\hat C_1=1.24$. For covariance stationarity, these small data sets 
were able to produce a single valid region $\mathcal N_{i,h_1,h_2}$, defined by $h_1=0$ and $h_2=0.1$, and hence, with only this region, verification of covariance
stationarity or nonstationarity is not possible. But since the underlying model is GP, covariance stationarity is equivalent to strict stationarity, and even for
non-Gaussian processes, strict stationarity would imply covariance stationarity (although strict nonstationarity need not imply covariance nonstationarity).

\begin{figure}
\centering
\subfigure [Correct detection of stationarity.]{ \label{fig:stationary_short}
\includegraphics[width=5.5cm,height=5.5cm]{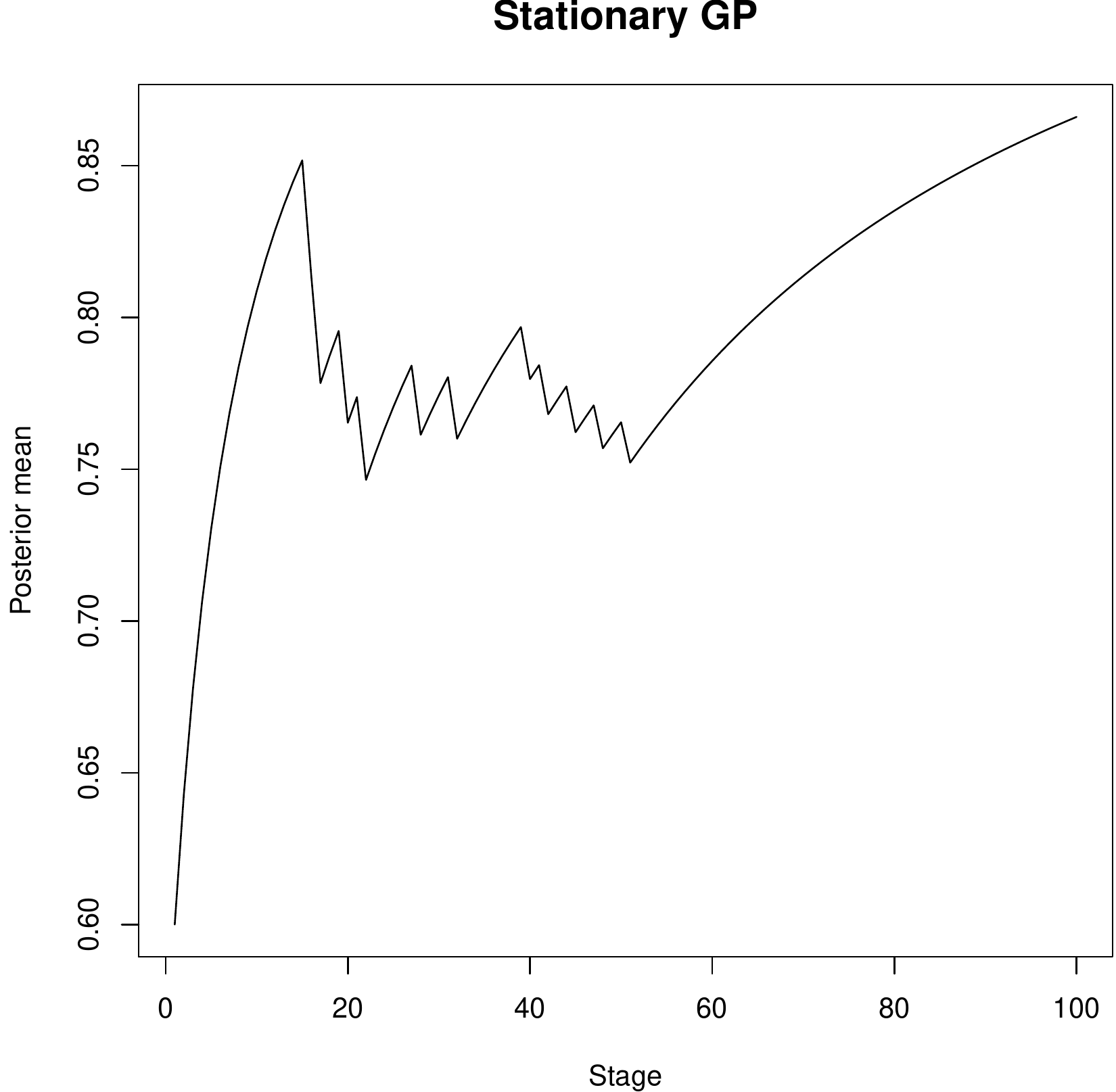}}
\hspace{2mm}
\subfigure [Correct detection of nonstationarity.]{ \label{fig:nonstationary_short}
\includegraphics[width=5.5cm,height=5.5cm]{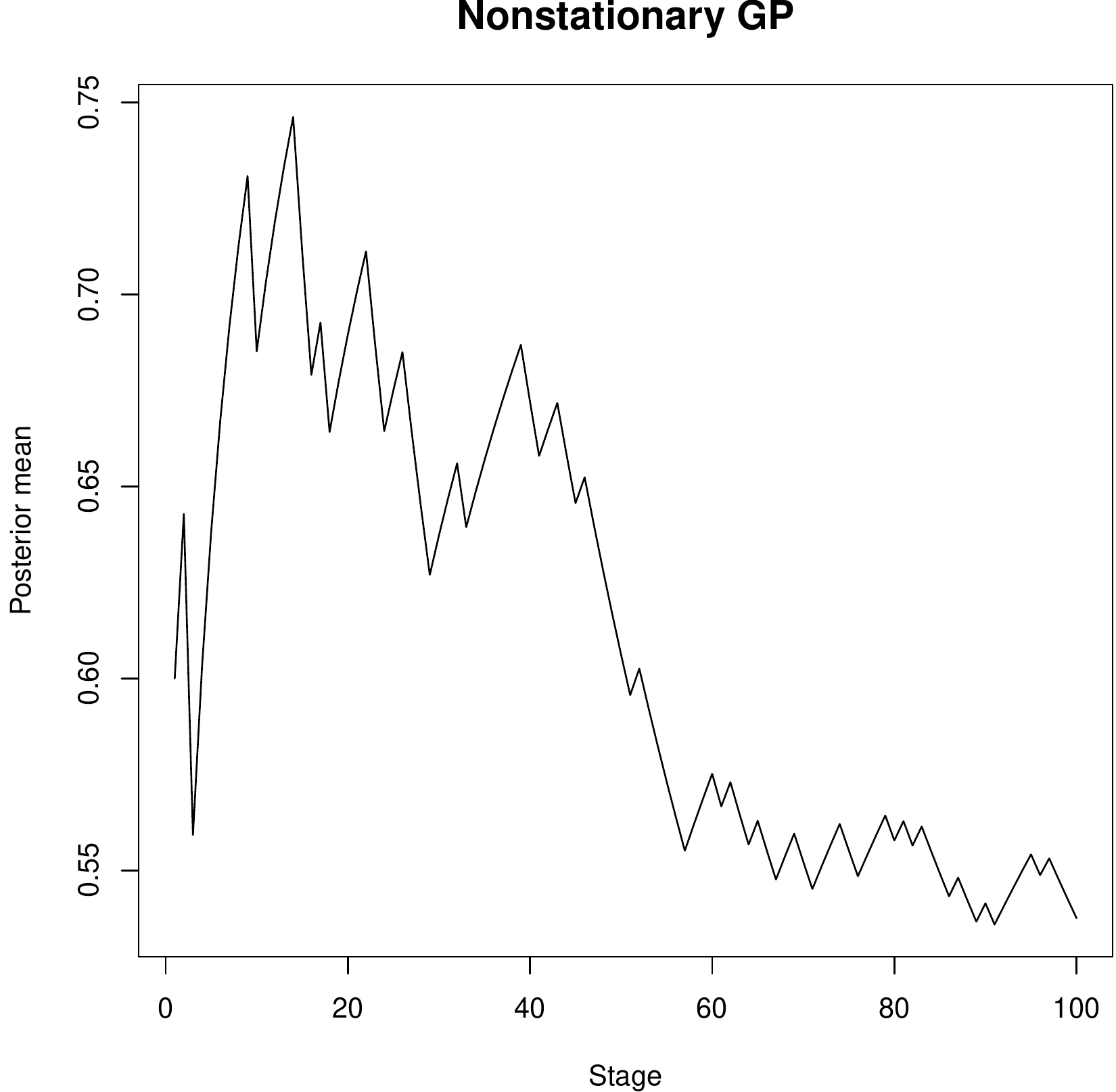}}
\caption{Detection of strong stationarity and nonstationarity in spatial data of size $1000$ drawn from GPs.}
\label{fig:spatial_short}
\end{figure}

\begin{figure}
\centering
\subfigure [$p = 0.9$.]{ \label{fig:mixed_9_short}
\includegraphics[width=5.5cm,height=5.5cm]{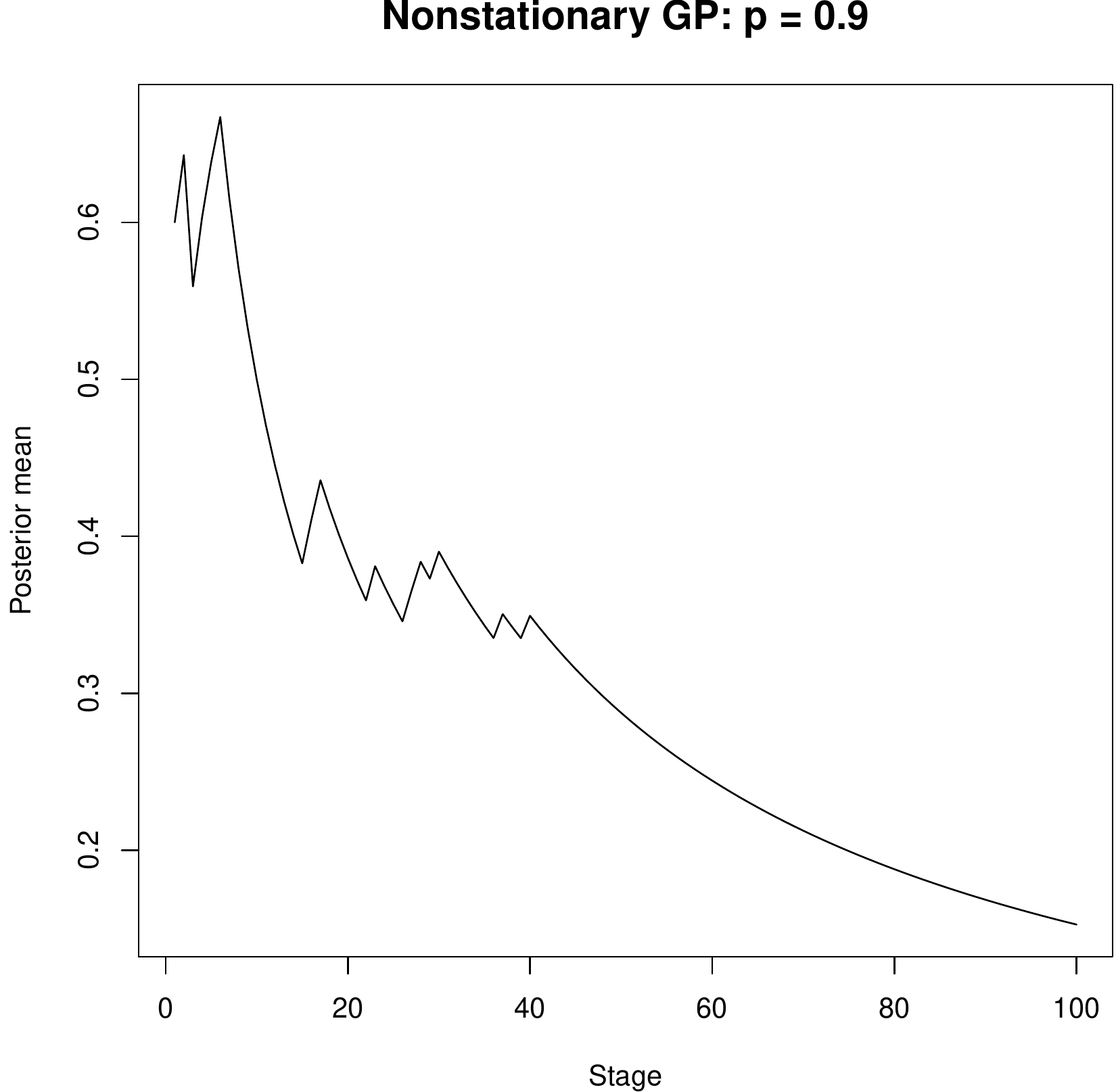}}
\hspace{2mm}
\subfigure [$p = 0.99$.]{ \label{fig:mixed_99_short}
\includegraphics[width=5.5cm,height=5.5cm]{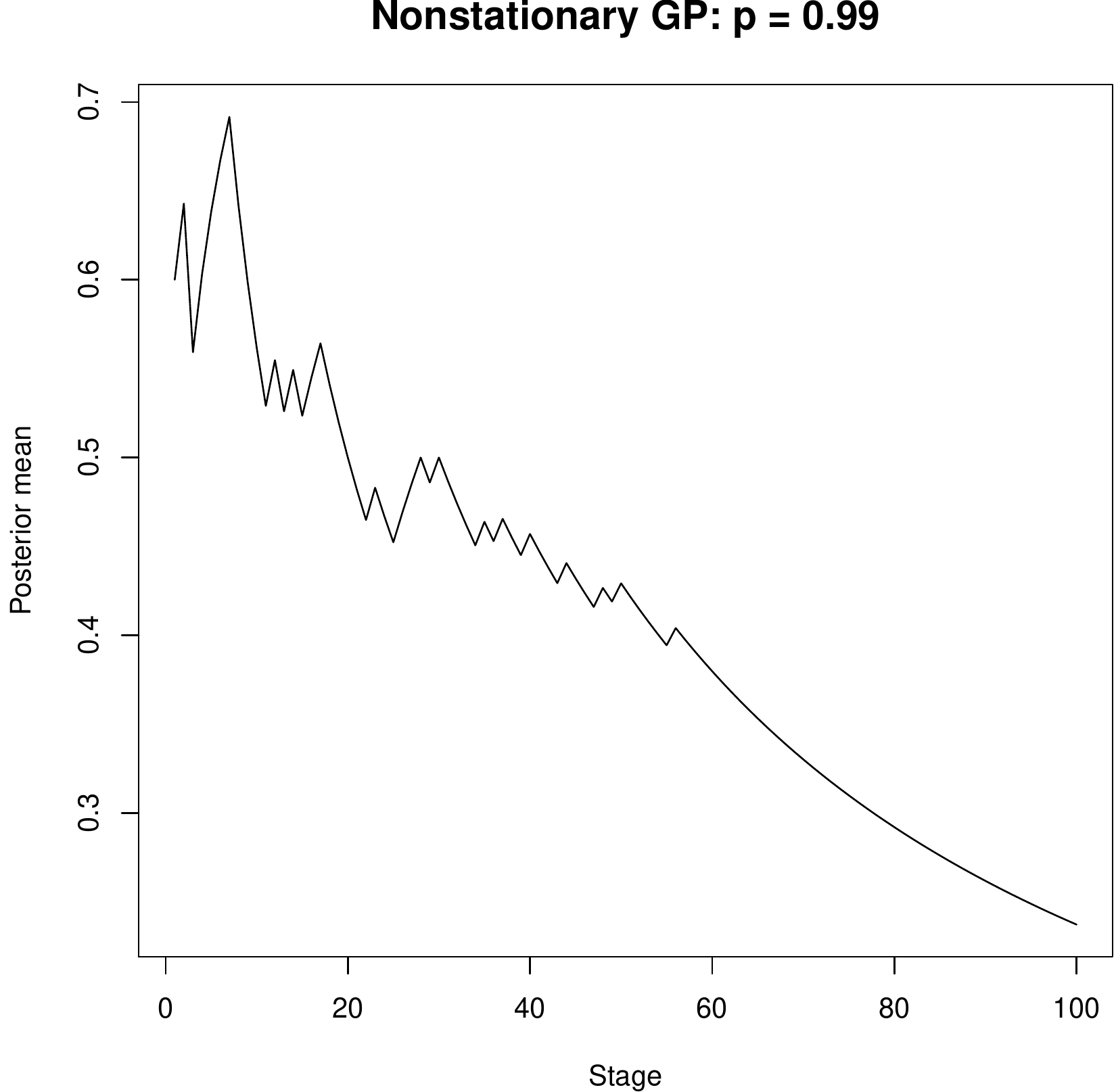}}\\
\vspace{2mm}
\subfigure [$p = 0.999$.]{ \label{fig:mixed_999_short}
\includegraphics[width=5.5cm,height=5.5cm]{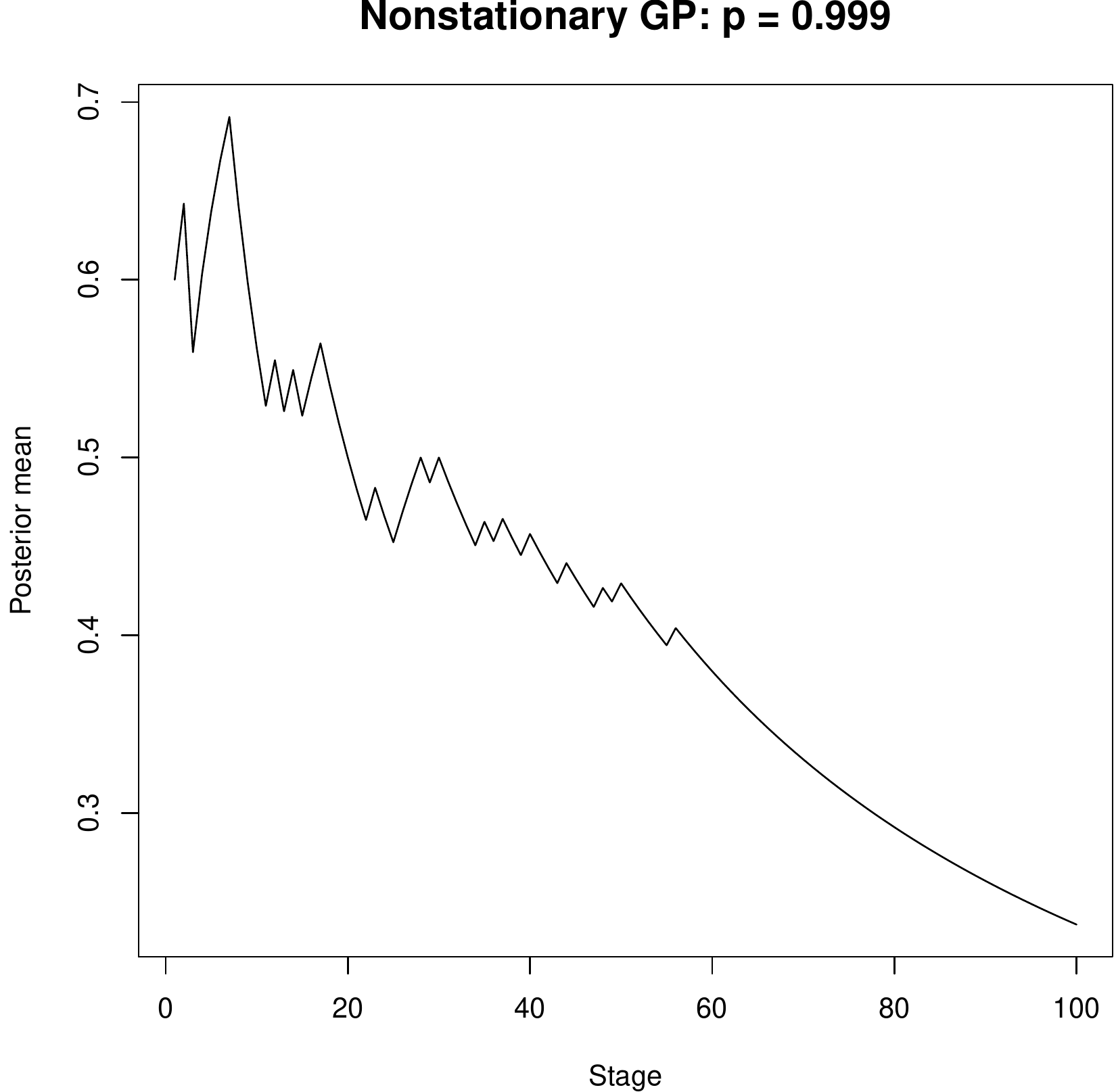}}
\hspace{2mm}
\subfigure [$p = 0.9999$.]{ \label{fig:mixed_9999_short}
\includegraphics[width=5.5cm,height=5.5cm]{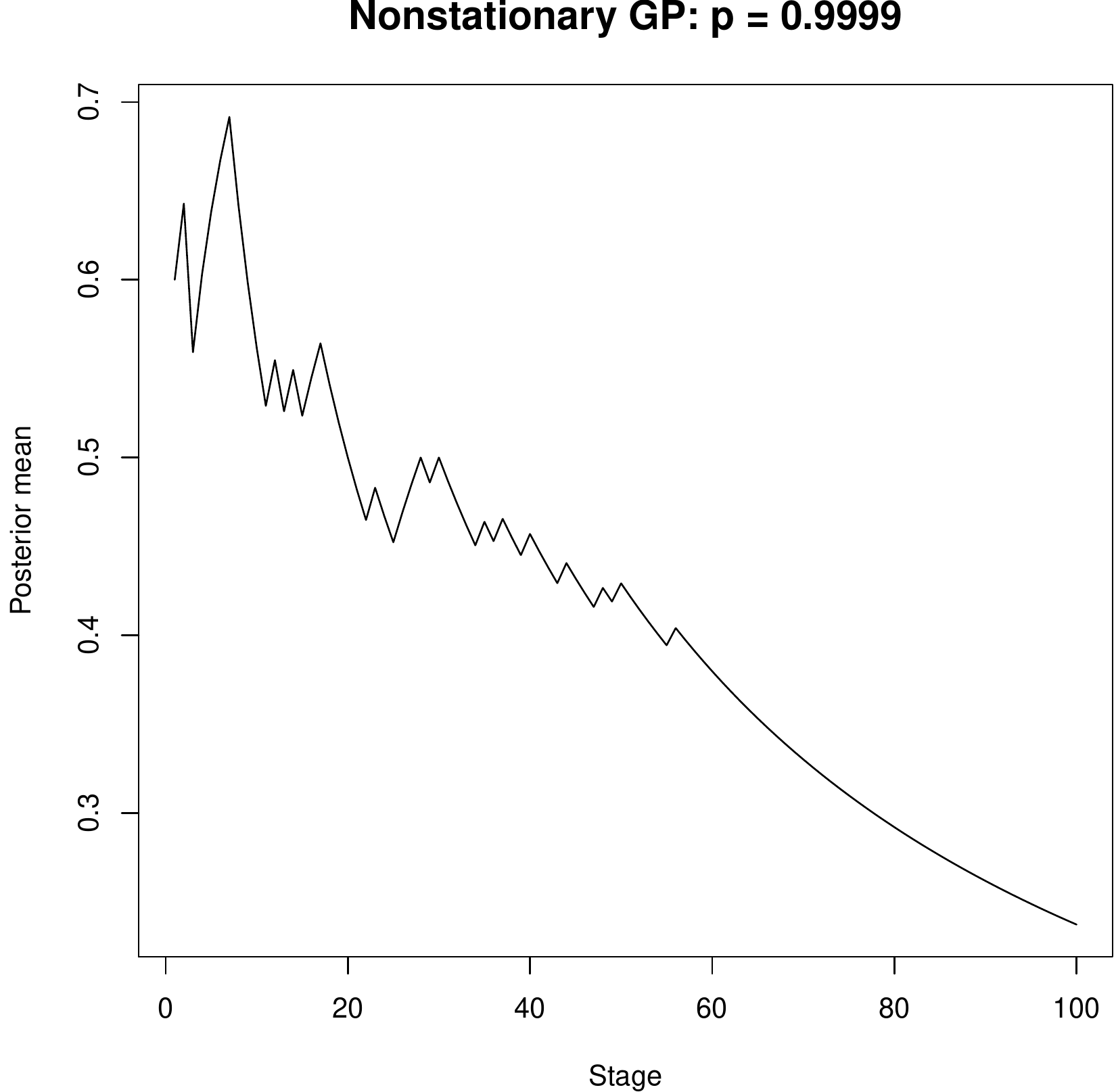}}\\
\vspace{2mm}
\subfigure [$p = 0.99999$.]{ \label{fig:mixed_99999_short}
\includegraphics[width=5.5cm,height=5.5cm]{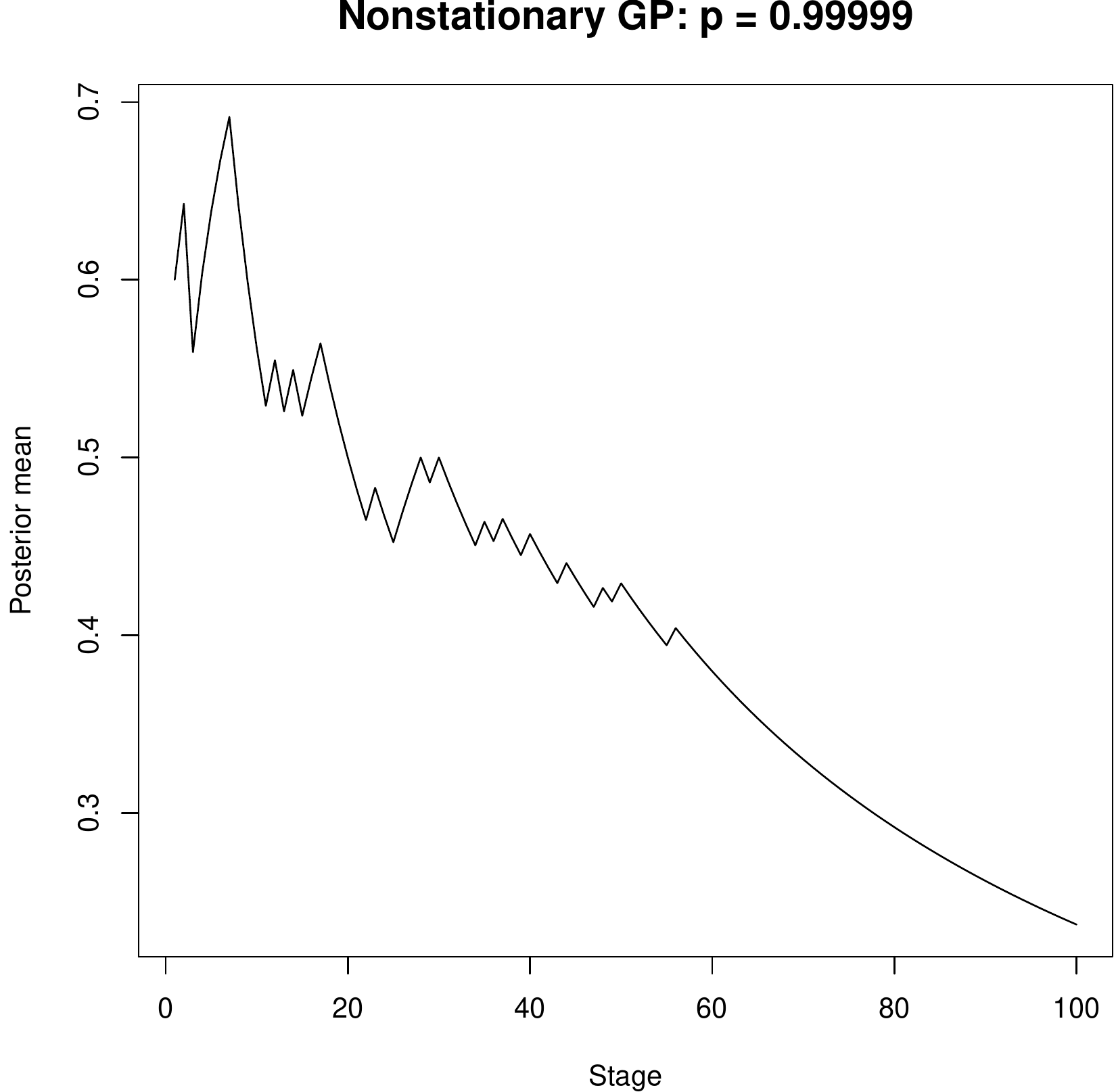}}\\
\caption{Detection of strong nonstationarity in spatial data of size $1000$ drawn from GP with covariance structure (\ref{eq:nonstationary2}) with $p=0.99999$.}
\label{fig:spatial_mixture_short}
\end{figure}

\subsection{Comparison with existing methods}
\label{subsec:spatial_comparison}
In spatial statistics, formal methods of testing stationarity or nonstationarity are rare, and mostly exploratory data analysis is used to informally
check stationarity. However, \ctn{Soutir17} have introduced some tests for checking covariance stationarity, under a variety of assumptions. These methods
seem to be more general compared to the existing ones.
An $R$-code for implementing their method is available at the webpage of the first author. Given a dataset, the code calculates two test statistics,
denoted by $T$ and $V$, along with the corresponding $P$-values under the null hypothesis of stationarity.
The statistic $V$ has been proposed in \ctn{Soutir17b}.

We apply their methods to our simulated spatial datasets in order to compare with our results. However, with data size $10000$, it turned out that
obtaining a result within reasonable time limits with the aforementioned $R$ code is almost infeasible. Instead, we applied their methods to data sets
of sizes $1000$, $3000$ and $5000$. The run times for the $R$ code for these data sizes are about $28$ seconds, $5$ minutes and $12$ minutes, respectively.

Table \ref{tab:table1} presents the results of the tests applied to our simulated datasets. In all the cases, the $T$ statistic failed to reject
the null hypothesis of stationarity, even though there is only one case of true null stationarity. On the other hand, the $V$-statistic performs much better,
with its performance consistently improving with increasing sample size, as vindicated by the corresponding $P$-values. But observe that for sample size
$1000$, even the $V$-statistic fails to reject the null hypothesis of stationarity at the 5\% level for most cases where the actual model is nonstationary.
Moreover, at the 5\% level, this statistic rejects the true null stationary model for sample sizes $3000$ and $5000$.

Thus, compared to our Bayesian idea, the overall performance of both the statistics $T$ and $V$ does not seem to be satisfactory for the models that we considered.

Moreover, from the methodological perspective, the tests of \ctn{Soutir17} check covariance stationarity only, not strict stationarity. Various assumptions,
which may be difficult to verify in practice, are also required. In contrast, our Bayesian method requires the only assumption of local stationarity that is expected
to hold in practice, and allows for identification of both weak and strict stationarity.

\begin{sidewaystable}[h]
\centering
\caption{The performance evaluation of the test statistics $T$ and $V$ of \ctn{Soutir17} and \ctn{Soutir17b} applied to our simulated spatial datasets.}
\vspace{2cm}

\begin{tabular}{|c|c|c|c|c|c|c|c|c|c|c|}
\hline
	& \multicolumn{1}{|c|}{\rule[-5mm]{0mm}{10mm}\backslashbox{Test}{Model}}
	& \multicolumn{1}{|c|}{\rule[-5mm]{0mm}{10mm}\mbox{Stationary}} & \multicolumn{1}{|c|}{\rule[-5mm]{0mm}{10mm}\mbox{Nonstationary}}
	& \multicolumn{1}{|c|}{\rule[-5mm]{0mm}{10mm}$p=0.9$} & \multicolumn{1}{|c|}{\rule[-5mm]{0mm}{10mm}$p=0.99$}
	& \multicolumn{1}{|c|}{\rule[-5mm]{0mm}{10mm}$p=0.999$} & \multicolumn{1}{|c|}{\rule[-5mm]{0mm}{10mm}$p=0.9999$}
	& \multicolumn{1}{|c|}{\rule[-5mm]{0mm}{10mm}$p=0.99999$}\\
	\hline

\multirow{4}{*}{1000} 
& $T$             & 7.692 & 2.444 & 2.290 & 4.717 & 9.151 & 8.105 & 9.254 \\ 
& $P$-value ($T$) & 0.158 & 0.751 & 0.776 & 0.387 & 0.103 & 0.141 & 0.099\\ 
& $V$             & 11.405 & 4.643 & 4.429 & 9.999 & 12.411 & 11.766 & 12.603\\ 
& $P$-value ($V$) & 0.056 & 0.398 & 0.424 & 0.080 & 0.043 & 0.051 & 0.041\\ 
\hline
\multirow{4}{*}{3000} 
& $T$             & 4.921 & 11.466 & 6.743 & 5.286 & 5.162 & 4.964 & 5.106 \\ 
& $P$-value ($T$) & 0.361 & 0.055 & 0.206 & 0.322 & 0.335 & 0.356 & 0.341\\ 
& $V$             & 13.483 & 16.527 & 16.631 & 14.073 & 13.432 & 13.508 & 13.432\\ 
& $P$-value ($V$) & 0.031 & 0.014 & 0.014 & 0.023 & 0.031 & 0.031 & 0.031\\ 
\hline
\multirow{4}{*}{5000} 
& $T$             & 3.307 & 4.234 & 3.196 & 3.313 & 3.385 & 3.342 & 3.385 \\ 
& $P$-value ($T$) & 0.595 & 0.451 & 0.615 & 0.595 & 0.583 & 0.589 & 0.583\\ 
& $V$             & 18.160 & 20.233 & 16.787 & 17.843 & 18.160 & 18.160 & 18.238\\ 
& $P$-value ($V$) & 0.010 & 0.006 & 0.014 & 0.011 & 0.010 & 0.010 & 0.010\\ 
\hline
\end{tabular}
\label{tab:table1}
\end{sidewaystable}

\section{Fifth illustration: detection of stationarity and nonstationarity in spatio-temporal data}
\label{sec:spatio_temporal}

We now apply our techniques in ascertaining stationarity and nonstationarity in spatio-temporal data, where both spatial and temporal components play
important roles. For our simulation studies, we consider covariance functions of the following forms:
\begin{equation}
	Cov(X_{(s_1,t_1)},X_{(s_2,t_2)})=\exp(-5\|s_1-s_2\|^2)\times\frac{\rho^{|t_1-t_2|}}{1-\rho^2},
\label{eq:spacetime1}
\end{equation}
\begin{equation}
	Cov(X_{(s_1,t_1)},X_{(s_2,t_2)})=\exp(-5\|\sqrt{s_1}-\sqrt{s_2}\|^2)\times\frac{\rho^{|t_1-t_2|}}{1-\rho^2},
\label{eq:spacetime2}
\end{equation}
and
\begin{equation}
	Cov(X_{(s_1,t_1)},X_{(s_2,t_2)})=\left(p\exp(-5\|s_1-s_2\|^2)+(1-p)\exp(-5\|\sqrt{s_1}-\sqrt{s_2}\|^2)\right)\times\frac{\rho^{|t_1-t_2|}}{1-\rho^2},
\label{eq:spacetime3}
\end{equation}
for all $s_1,s_2\in \mathbb R^2$, $t_1,t_2\in\mathbb R^+$ and $\rho\in\mathbb R$. Note that $\frac{\rho^{|t_1-t_2|}}{1-\rho^2}$ is the covariance function
associated with an $AR(1)$ model with parameter $\rho$. The forms of the covariance functions (\ref{eq:spacetime1}), (\ref{eq:spacetime2}) and (\ref{eq:spacetime3})
show that the covariance parts associated with spatial and temporal components are separated from each other, thanks to the product forms. Covariance functions 
with such a property are known as separable covariance functions. In (\ref{eq:spacetime3}), $p\in [0,1]$, as before. If $p=0$, then (\ref{eq:spacetime3})
reduces to (\ref{eq:spacetime3}) and to (\ref{eq:spacetime1}) if $p=1$.

Note that if $|\rho|<1$, then (\ref{eq:spacetime1}) is a stationary covariance function, and nonstationary otherwise. On the other hand, (\ref{eq:spacetime2})
and (\ref{eq:spacetime3}) are both nonstationary covariance functions, irrespective of the value of $\rho$. 

For our simulation experiments, we consider zero-mean GPs $X_{(s,t)}$ with the above covariance functions, restricting the spatial locations on $[0,1]^2$
and setting the time points $t_i=i$, for $i\geq 1$. We simulate, for $i=1,\ldots,100$, $\tilde s_i\sim U\left([0,1]^2\right)$ and set $s_i=\sqrt{s_i}$.
We set $t_i=i$, for $i=1,\ldots,100$. This defines covariance matrices for $10000$-dimensional multivariate normal associated with the underlying GPs.
Note that such covariance matrices are Kronecker products of the spatial and temporal covariance matrices, thanks to separability.

Observe that the above separable covariance matrices correspond to separable spatio-temporal processes of the form
\begin{equation}
	X_{(s,t)}=X_{(s,t-1)}+\epsilon_{(s,t)}, 
\label{eq:spatial_ar}
\end{equation}
for $t=1,2,\ldots$, where $X_{(s,0)}=\bzero$ (null vector), and $\epsilon_{(s,t)}$ are zero-mean GPs independent in time, but with spatial covariance
with forms same as the spatial parts in (\ref{eq:spacetime1}), (\ref{eq:spacetime2}) and (\ref{eq:spacetime3}).
With the above representation, generation of $10000$ realization takes about a second, even in $R$.


To construct $\mathcal N_i$, $i=1,\ldots,K$, we consider $K$-means clustering of the points $$\left\{(s_i,t_j); i=1,\ldots,100;j=1,\ldots,100\right\},$$ 
into $K=250$ clusters.

\subsection{Choice of the bound $c_j$ in the spatio-temporal case}
\label{subsec:C1_hat_spacetime}
We consider the bound of the form (\ref{eq:ar1_bound3}) as before. As regards, $\hat C_1$, we found that $\hat C_1=0.5$ performed adequately for the entire
suite of our simulation experiments in the spatio-temporal scenario. However, we also consider a strategy for obtaining $\hat C_1$ using ideas similar to the
spatial setup, detailed below. 

We first generate a sample of size $10000$ from a zero mean GP with the covariance function of the following form: 
\begin{equation}
	Cov(X_{(s_1,t_1)},X_{(s_2,t_2)})=(\|s_1-s_2\|/\psi)\mathcal K_1(\|s_1-s_2\|/\psi)\times\frac{\xi^{|t_1-t_2|}}{1-\xi^2}, 
\label{eq:spacetime_bound}
\end{equation}
with $\psi=0.8$ and $\xi=0.999999$. Note that this covariance function corresponds to a model of the form (\ref{eq:spatial_ar}) with $X_{(s,0)}=\bzero$
and zero-mean GPs $\epsilon_{(s,t)}$ independent in time, with spatial covariance given by the spatial form in (\ref{eq:spacetime_bound}).
The parameter values $\psi=0.8$ and $\xi=0.999999$ are chosen to make the underlying spatio-temporal process reasonably close to nonstationarity
with respect to space and time.

We then choose that minimum value of $\hat C_1$ such that the spatio-temporal process remains stationary. This minimum value, for checking
strict stationarity, is given by $\hat C_1=0.37$, which is reasonably close to $\hat C_1=0.5$ that worked well for our experiments.
Again, we obtained same results for both the values of $\hat C_1$, and we report results for $\hat C_1=0.37$. 

However, for weak stationarity, we again failed to obtain multiple valid intervals 
for realizations of size $10000$ from the zero-mean GP with covariance (\ref{eq:spacetime_bound}). Indeed, we could obtain only a single interval $[0,0.15]$. 
Hence, in that case we consider $\hat C_1=0.5$.

Below we discuss the experimental designs for our various simulation experiments.

\subsection{Spatial and temporal stationarity}
\label{subsec:spacetime_stationary}
We generate partial realizations of length $10000$ from the zero mean GP with covariance function (\ref{eq:spacetime1}) using the formulation (\ref{eq:spatial_ar}), 
with $\rho=0.8$
and also with $\rho=0.99999$. Thus,
the spatio-temporal GPs are strictly stationary, and our Bayesian method is expected to reflect this. 
The latter situation is quite subtle, as the difference with temporal nonstationarity is negligible.

Apart from strict stationarity, we also investigate weak stationarity, focussing on the subtle situation where $\rho=0.99999$.

\subsection{Spatio-temporal nonstationarity}
\label{subsec:spacetime_nonstationary}
Recall that spatio-temporal nonstationarity occurs in our cases when $|\rho|\geq 1$ in (\ref{eq:spacetime1}) and when covariances (\ref{eq:spacetime2})
or (\ref{eq:spacetime3}) are chosen. We experiment with (\ref{eq:spacetime1}) with $\rho=1$, (\ref{eq:spacetime2}) with $\rho=0.8$ and $\rho=1$, 
(\ref{eq:spacetime3}) with $p=0.99999$ and $\rho=0.8$. The latter is a subtle situation where nonstationarity is quite difficult to ascertain.
Note that if nonstationarity can be captured by our Bayesian method in this situation, then so is possible for larger values of $\rho$ taking the temporal
part closer to nonstationarity.
With the last, subtle situation, we also investigate covariance nonstationarity.

\subsection{Results}
\label{subsec:results_spacetime}
Figure \ref{fig:spacetime1}, diagrammatically representing our Bayesian procedure, vindicates that the stochastic processes associated with 
covariance function (\ref{eq:spacetime1}) with $\rho=0.8$ and $\rho=0.99999$, are indeed
strictly stationary. On the other hand, the processes corresponding to 
(\ref{eq:spacetime1}) with $\rho=1$, (\ref{eq:spacetime2}) with $\rho=0.8$ and $\rho=1$, (\ref{eq:spacetime3}) with $p=0.99999$ and $\rho=0.99999$, are all 
correctly detected by our Bayesian method as strictly nonstationary.

Figure \ref{fig:spacetime2} depicts the results of investigation of weak stationarity for the covariance (\ref{eq:spacetime1}) with $\rho=0.99999$.
For the covariance (\ref{eq:spacetime3}) with $p=0.99999$ and $\rho=0.8$, Figure \ref{fig:spacetime3} presents the results of our Bayesian technique.
In both the cases, success of our Bayesian proposal is clearly borne out.

\begin{figure}
\centering
\subfigure [Correct detection of stationarity.]{ \label{fig:stationary1}
\includegraphics[width=5.5cm,height=5.5cm]{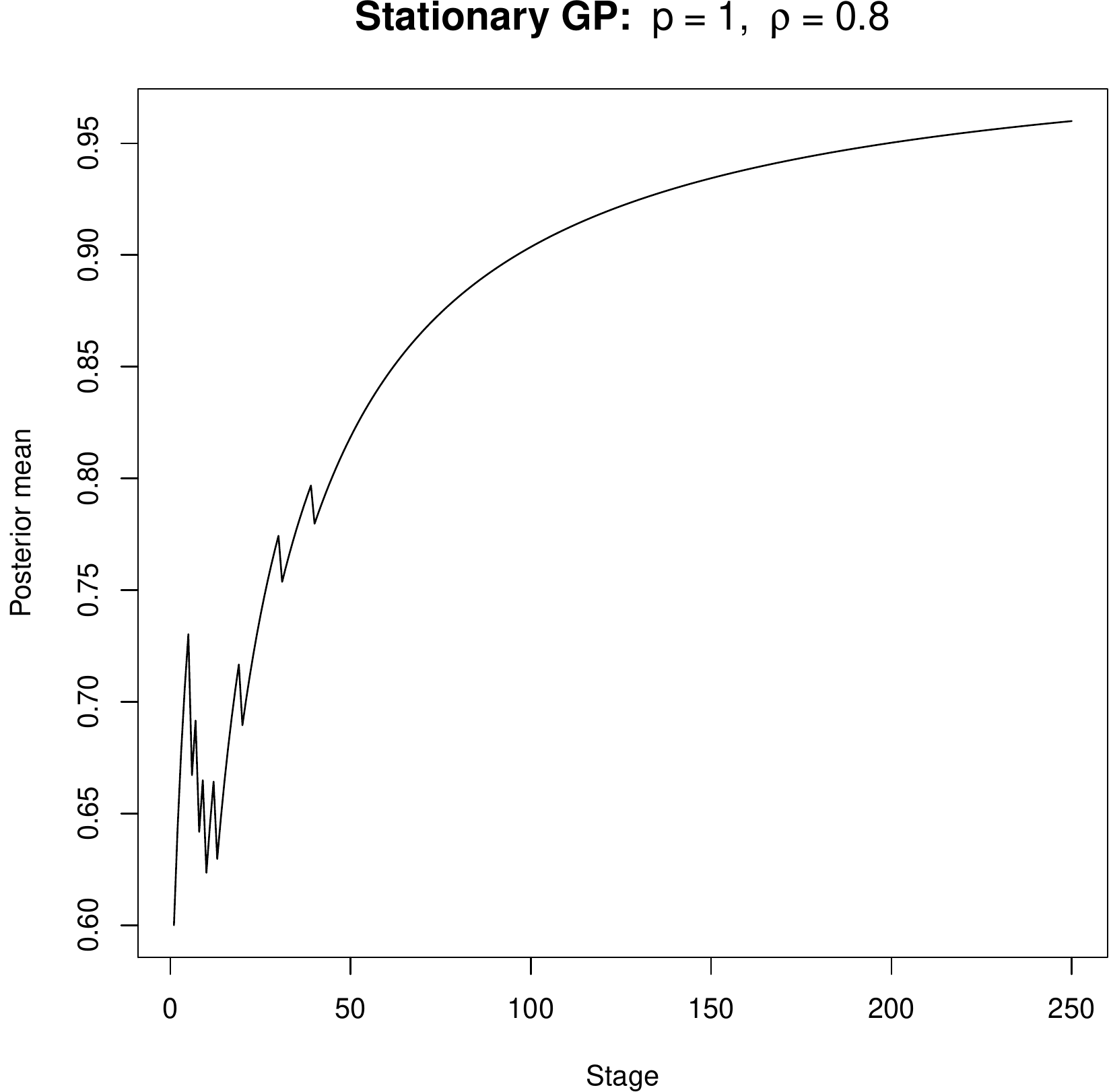}}
\hspace{2mm}
\subfigure [Correct detection of stationarity.]{ \label{fig:stationary2}
\includegraphics[width=5.5cm,height=5.5cm]{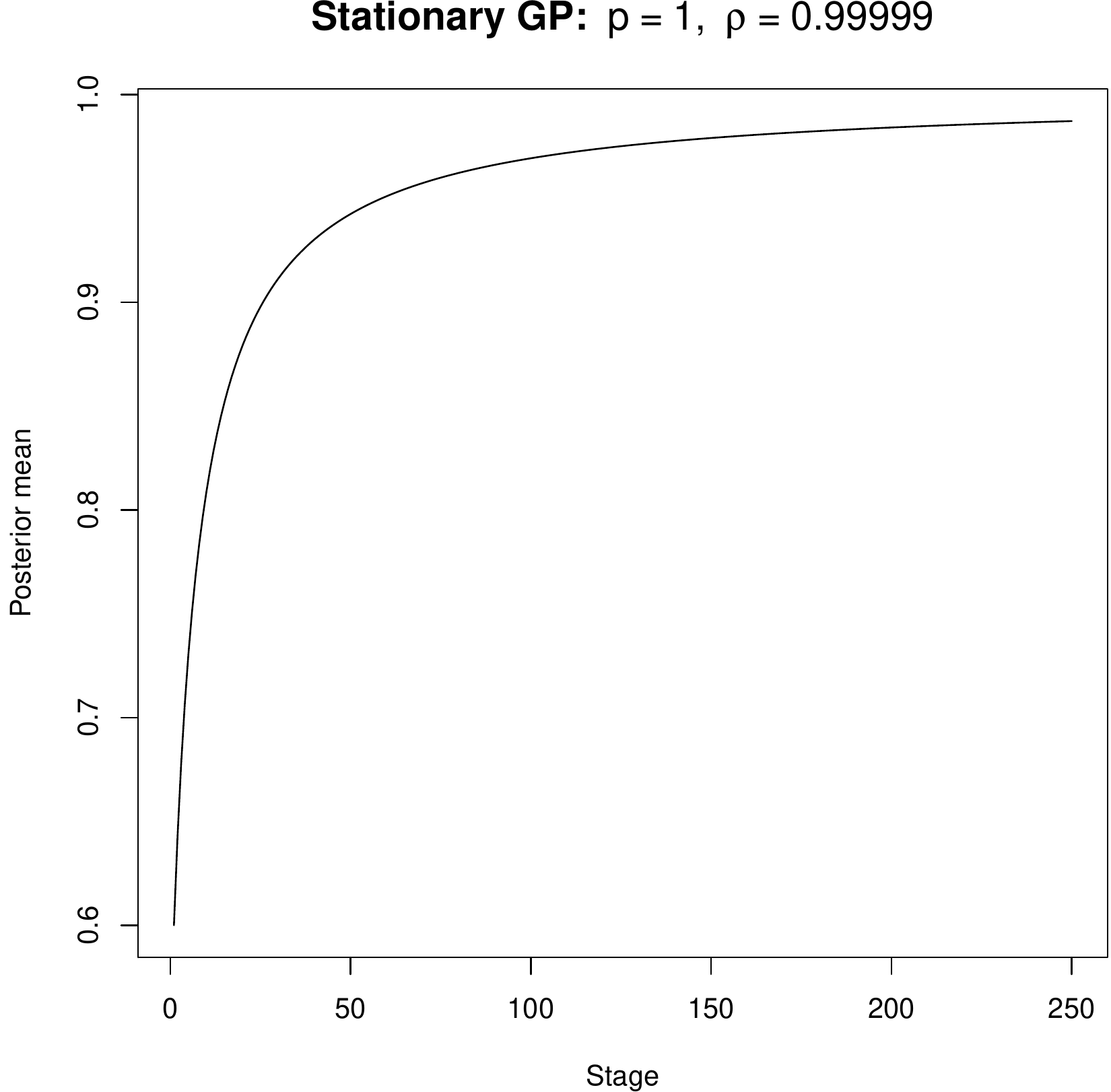}}\\
\vspace{2mm}
\subfigure [Correct detection of nonstationarity.]{ \label{fig:nonstationary1}
\includegraphics[width=5.5cm,height=5.5cm]{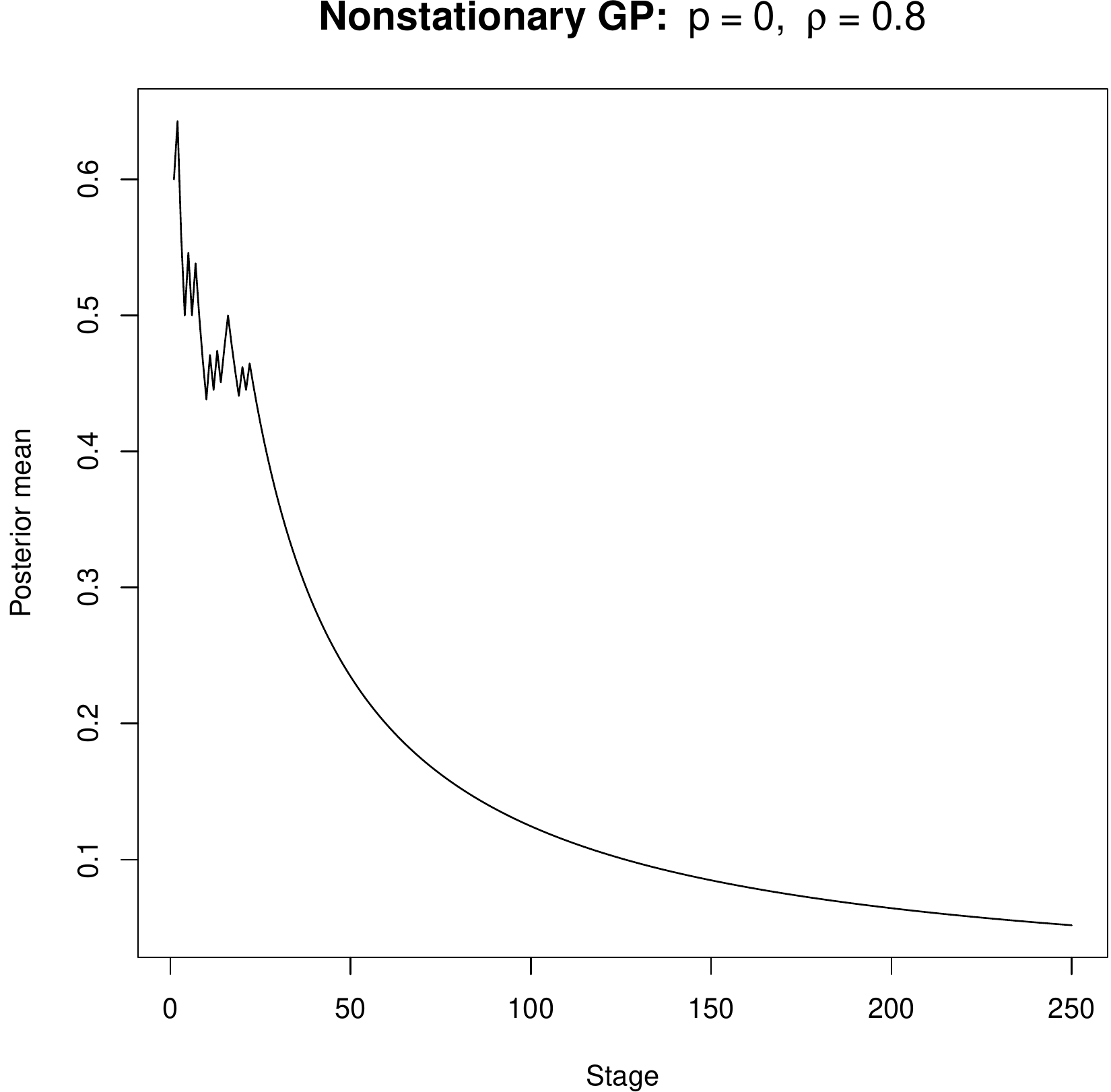}}
\hspace{2mm}
\subfigure [Correct detection of nonstationarity.]{ \label{fig:nonstationary2}
\includegraphics[width=5.5cm,height=5.5cm]{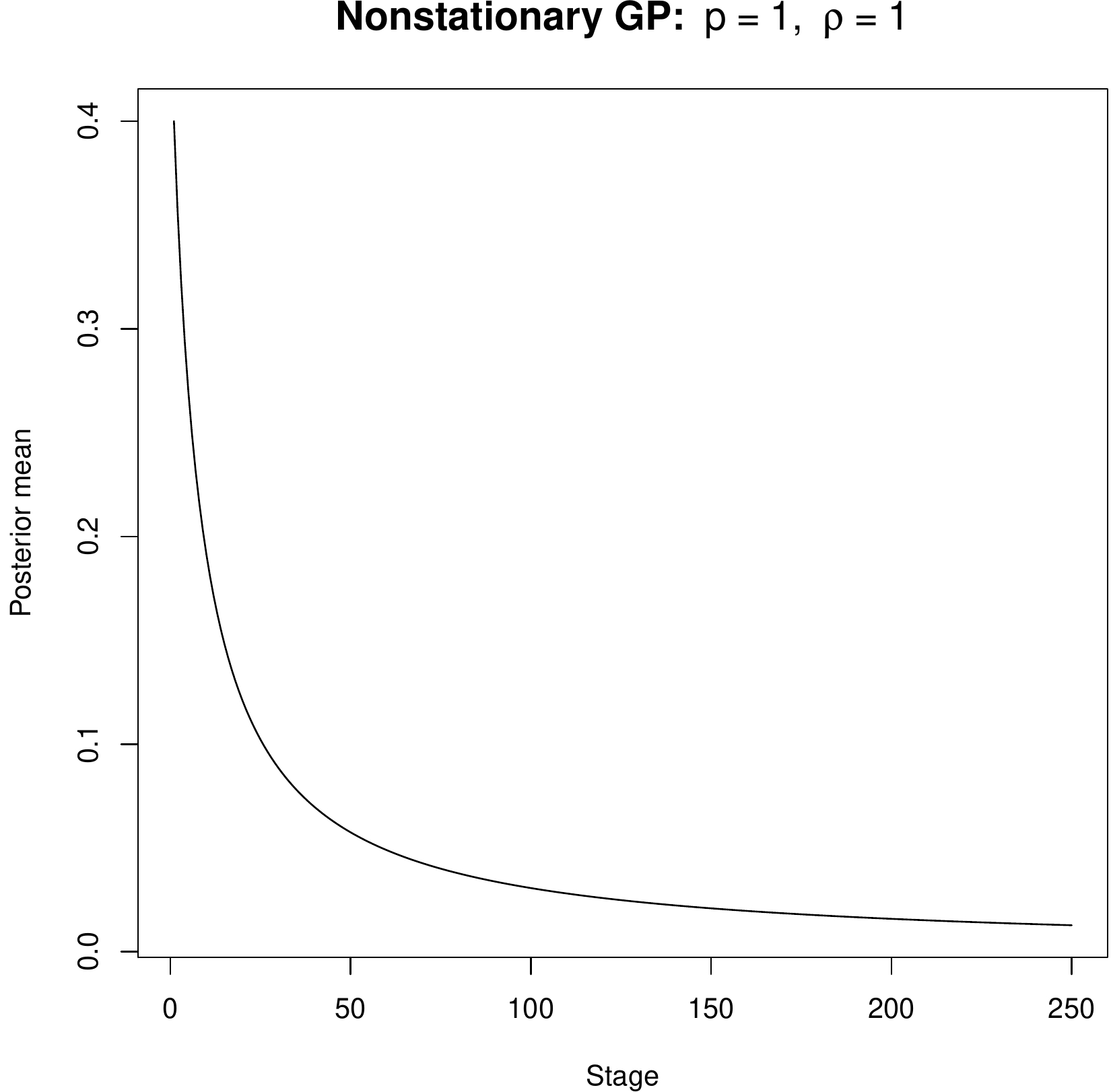}}\\
\vspace{2mm}
\subfigure [Correct detection of nonstationarity.]{ \label{fig:nonstationary3}
\includegraphics[width=5.5cm,height=5.5cm]{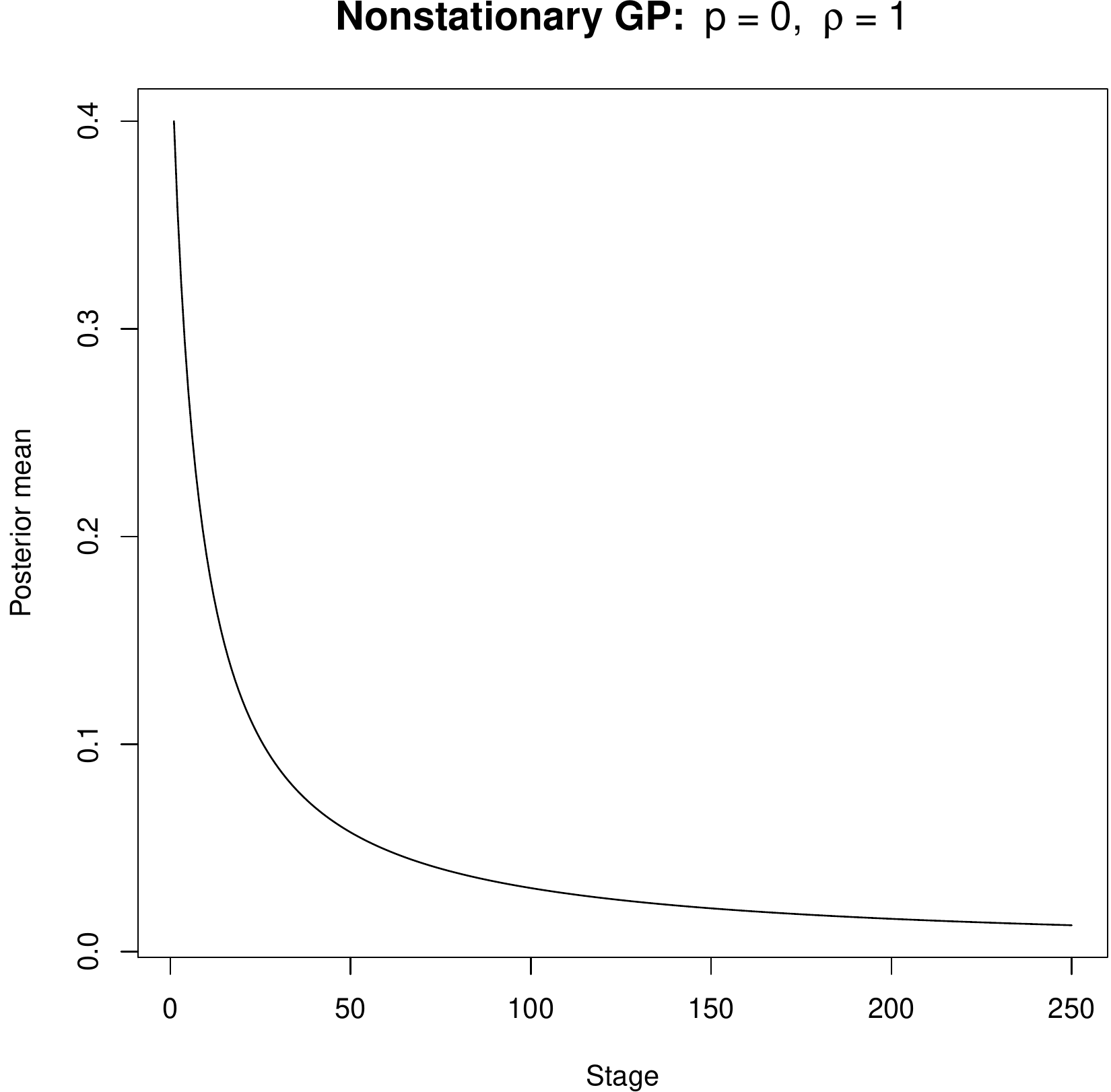}}\\
\caption{Detection of strong stationarity and nonstationarity in spatio-temporal data drawn from GPs.}
\label{fig:spacetime1}
\end{figure}

\begin{figure}
\centering
\subfigure [$0\leq\|h\|<0.15$]{ \label{fig:spacetime_covns1}
\includegraphics[width=5.5cm,height=5.5cm]{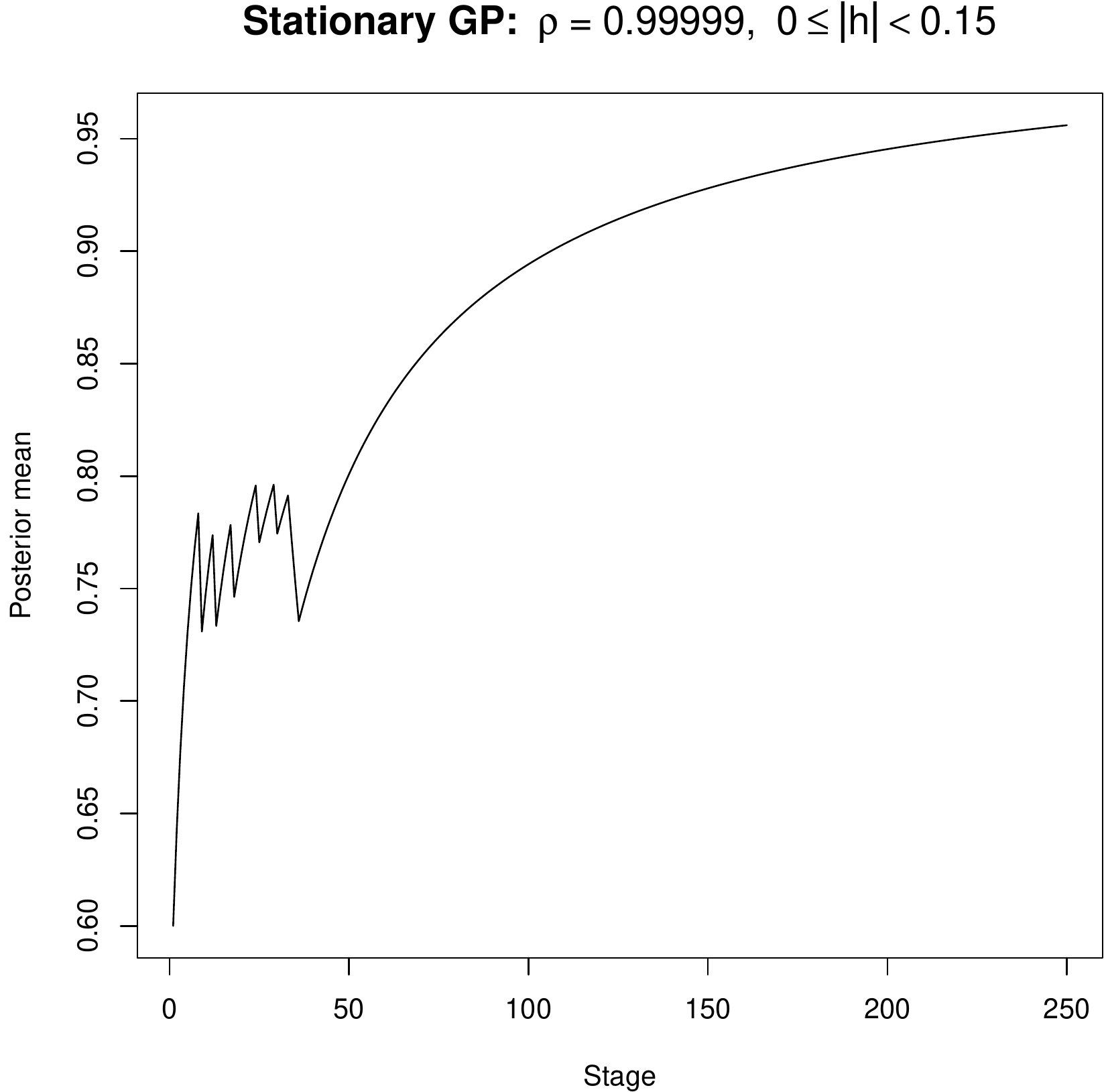}}
\hspace{2mm}
\subfigure [$0.15\leq\|h\|<0.2$.]{ \label{fig:spacetime_covns2}
\includegraphics[width=5.5cm,height=5.5cm]{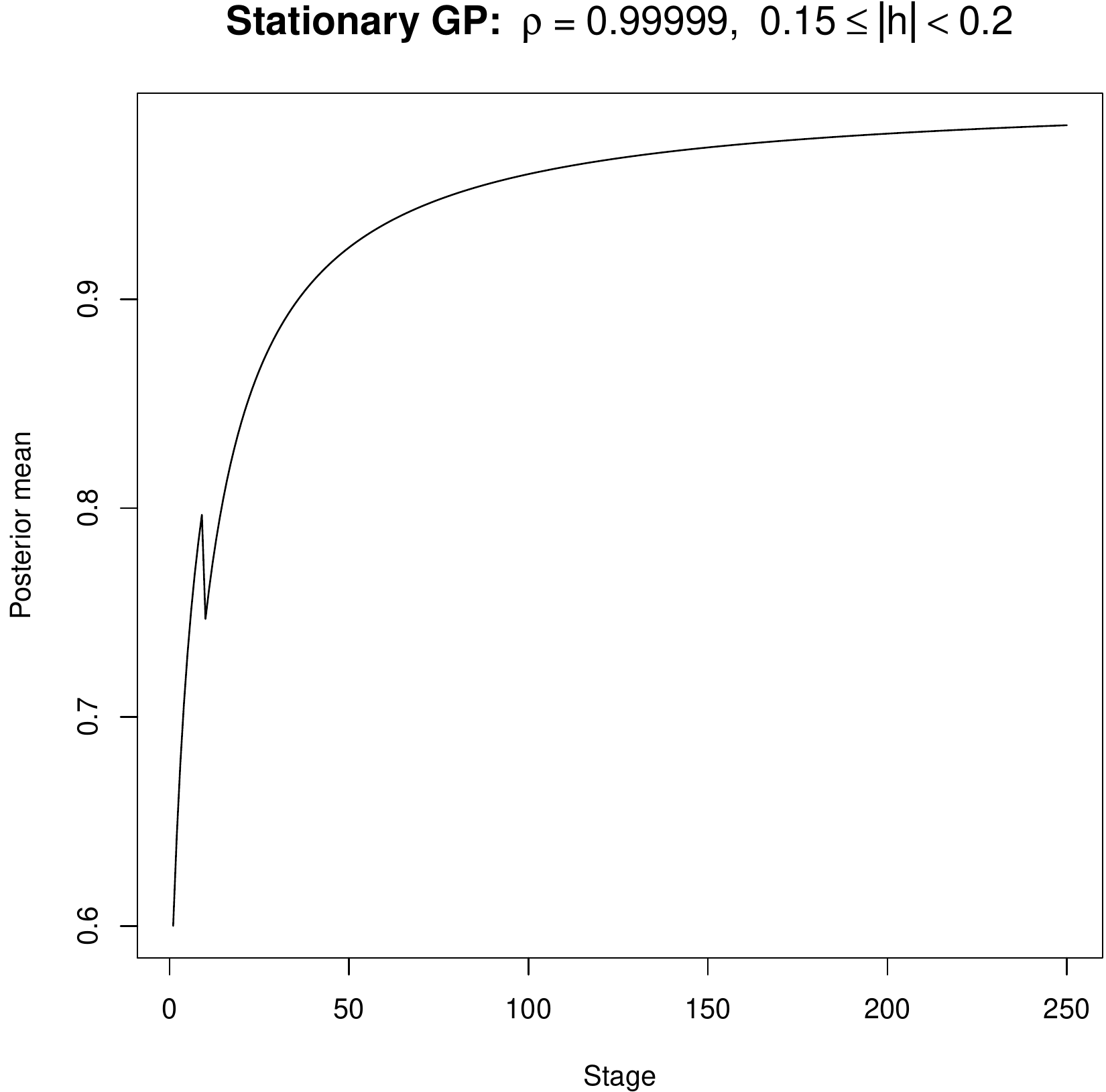}}\\
\vspace{2mm}
\subfigure [$0.2\leq\|h\|<0.35$]{ \label{fig:spacetime_covns3}
\includegraphics[width=5.5cm,height=5.5cm]{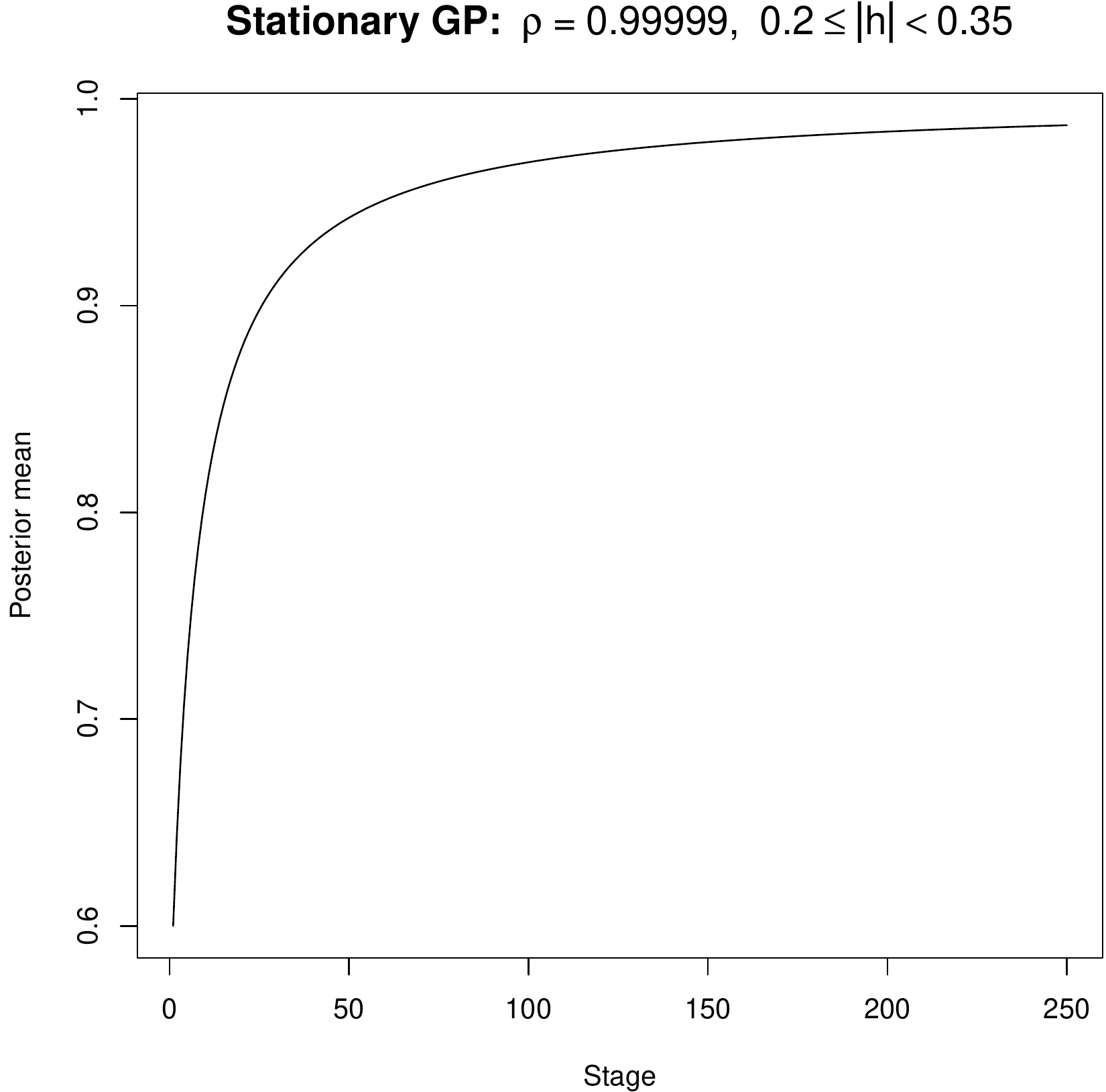}}\\
\caption{Detection of covariance stationarity in spatio-temporal data drawn from GP with covariance structure (\ref{eq:spacetime1}) with $\rho=0.99999$.}
\label{fig:spacetime2}
\end{figure}

\begin{figure}
\centering
\subfigure [$0\leq\|h\|<0.15$]{ \label{fig:spacetime_mixed_covns1}
\includegraphics[width=5.5cm,height=5.5cm]{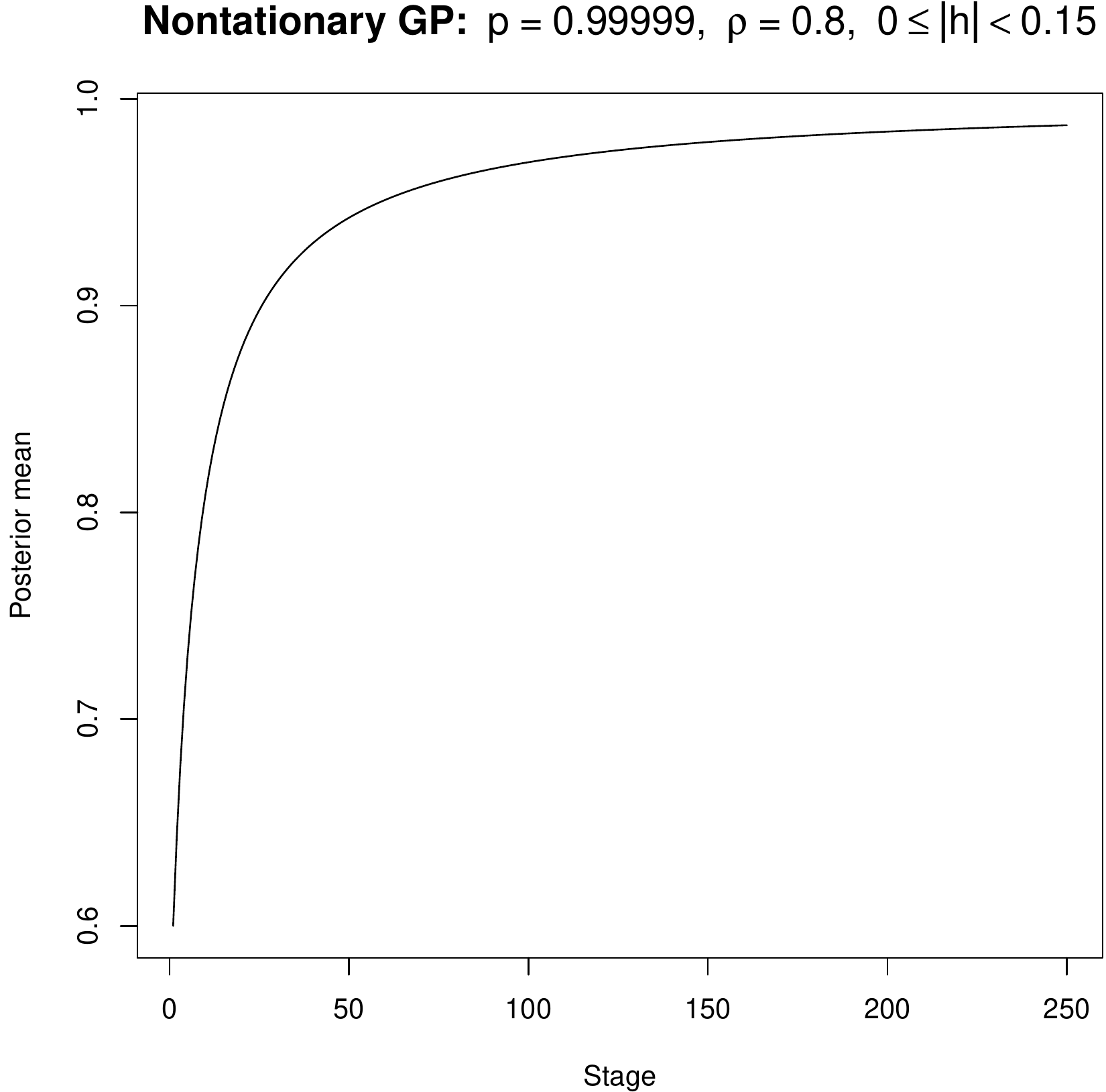}}
\hspace{2mm}
\subfigure [$0.15\leq\|h\|<0.2$.]{ \label{fig:spacetime_mixed_covns2}
\includegraphics[width=5.5cm,height=5.5cm]{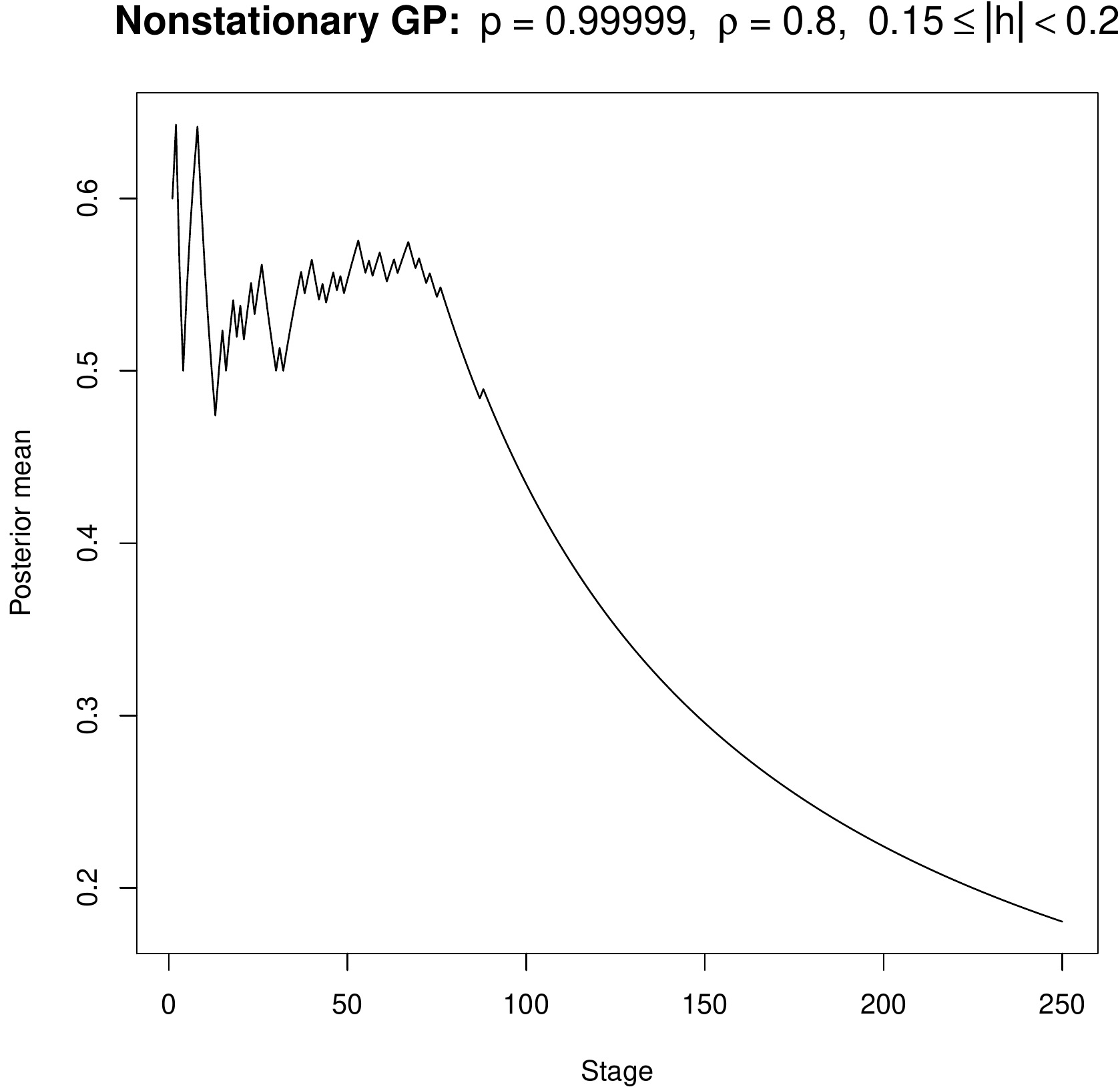}}\\
\vspace{2mm}
\subfigure [$0.2\leq\|h\|<0.35$]{ \label{fig:spacetime_mixed_covns3}
\includegraphics[width=5.5cm,height=5.5cm]{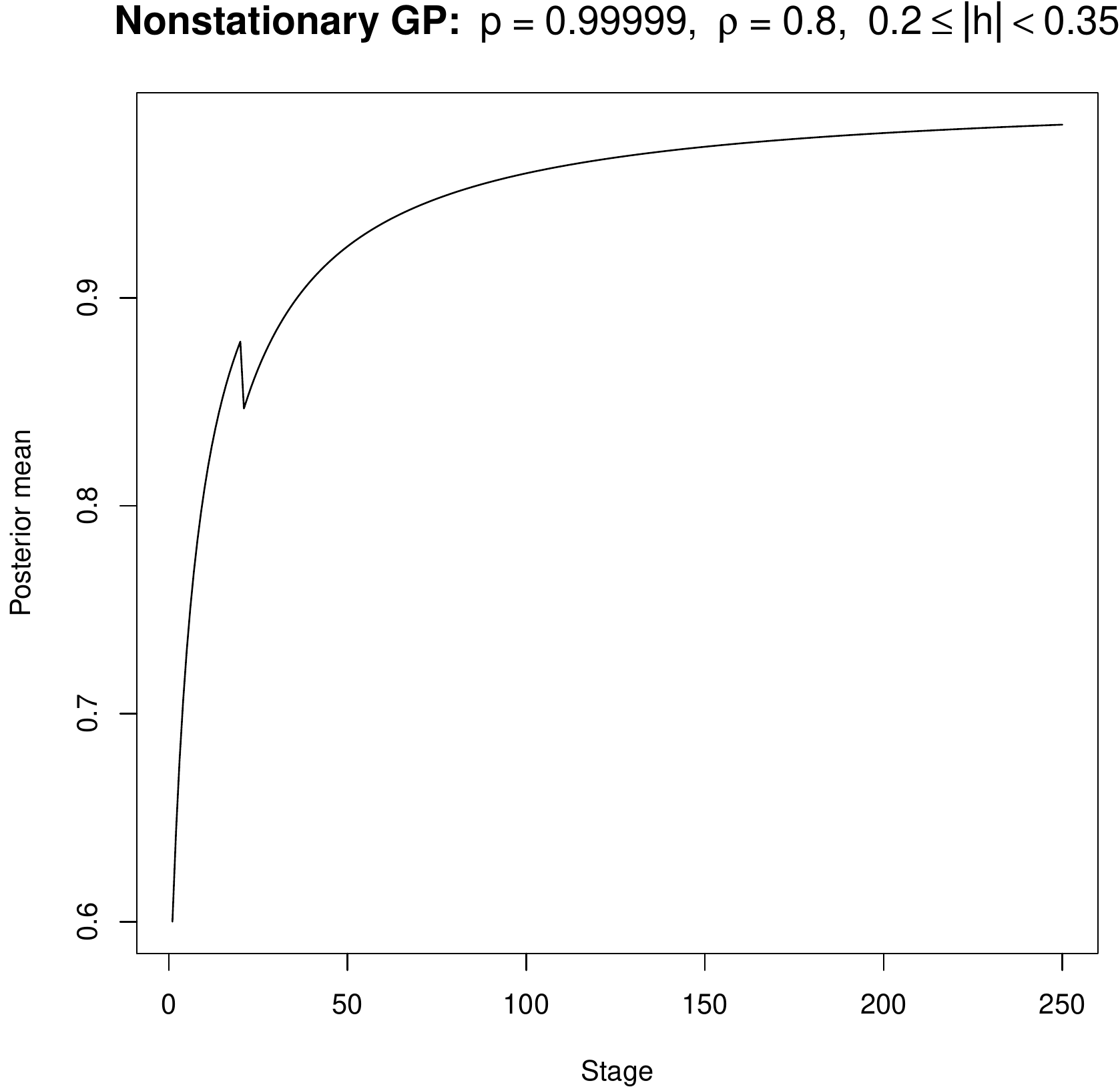}}\\
\caption{Detection of covariance nonstationarity in spatio-temporal data drawn from GP with covariance structure (\ref{eq:spacetime3}) with $p=0.99999$ and $\rho=0.8$.}
\label{fig:spacetime3}
\end{figure}

\subsection{Investigation of spatio-temporal stationarity with smaller sample size}
\label{subsec:spacetime_small_sample}
We now investigate stationarity of the above spatio-temporal models using much smaller sample sizes. In particular, we consider $50$ locations
and $20$ time points only, and $K=100$ clusters. We ensured at least $3$ data points in each cluster. 
Our strategy for choosing $\hat C_1$, detailed in Section \ref{subsec:C1_hat_spacetime}, gave
$\hat C_1=0.87$ for investigating strict stationarity. Again, $\hat C_1=0.5$ yielded the same conclusions.
Figure \ref{fig:spacetime1_short}, depicting the results of our analysis for $\hat C_1=0.87$, indicates correct decisions on strict stationarity and
nonstationarity in all the cases, even for such small data size.

However, validating covariance stationarity could not be achieved for such small samples, as we again ended up with the single interval $\mathcal N_{i,h_1,h_2}$
with $h_1=0$ and $h_2=0.2$.
\begin{figure}
\centering
\subfigure [Correct detection of stationarity.]{ \label{fig:stationary1_short}
\includegraphics[width=5.5cm,height=5.5cm]{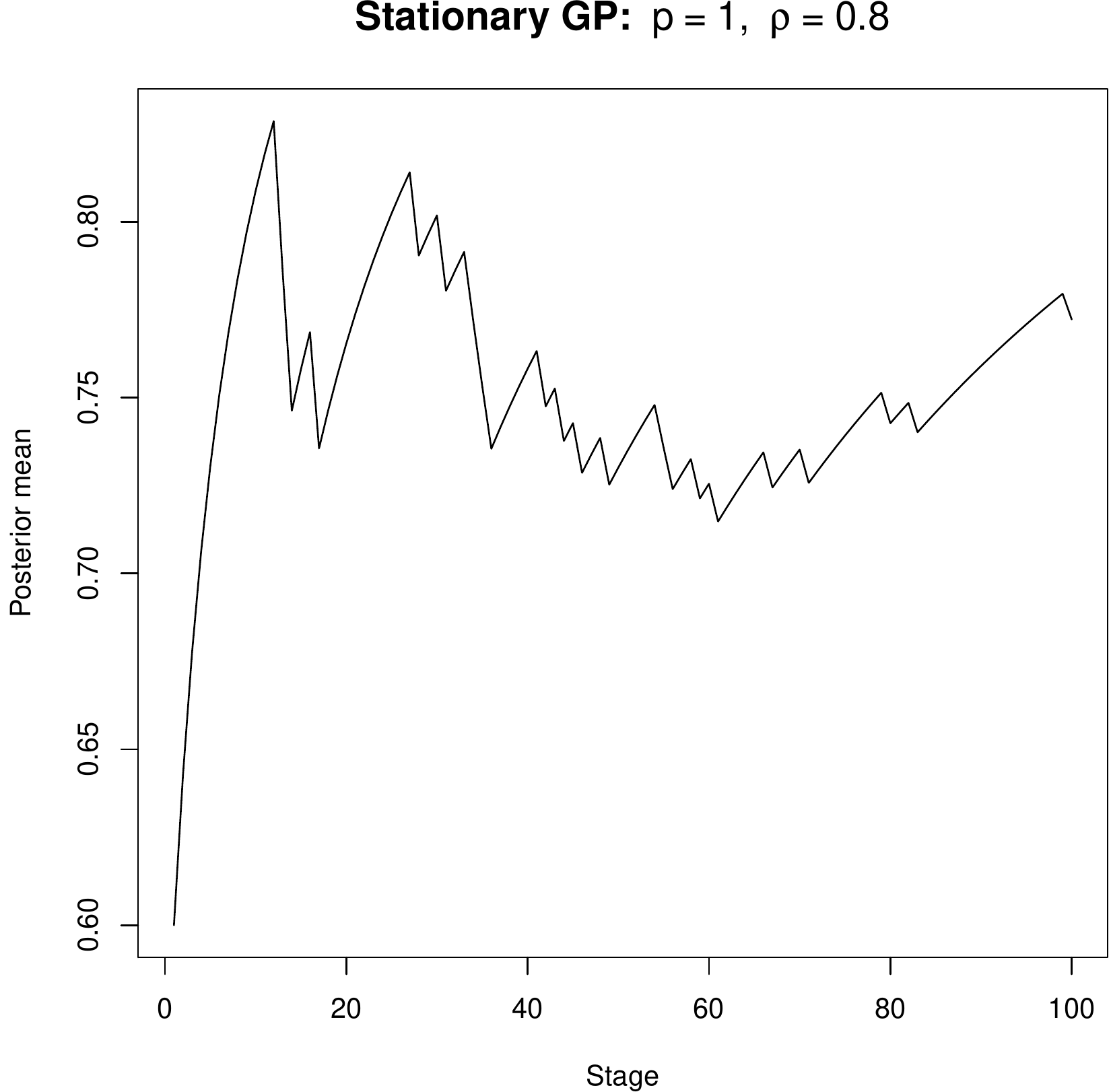}}
\hspace{2mm}
\subfigure [Correct detection of stationarity.]{ \label{fig:stationary2_short}
\includegraphics[width=5.5cm,height=5.5cm]{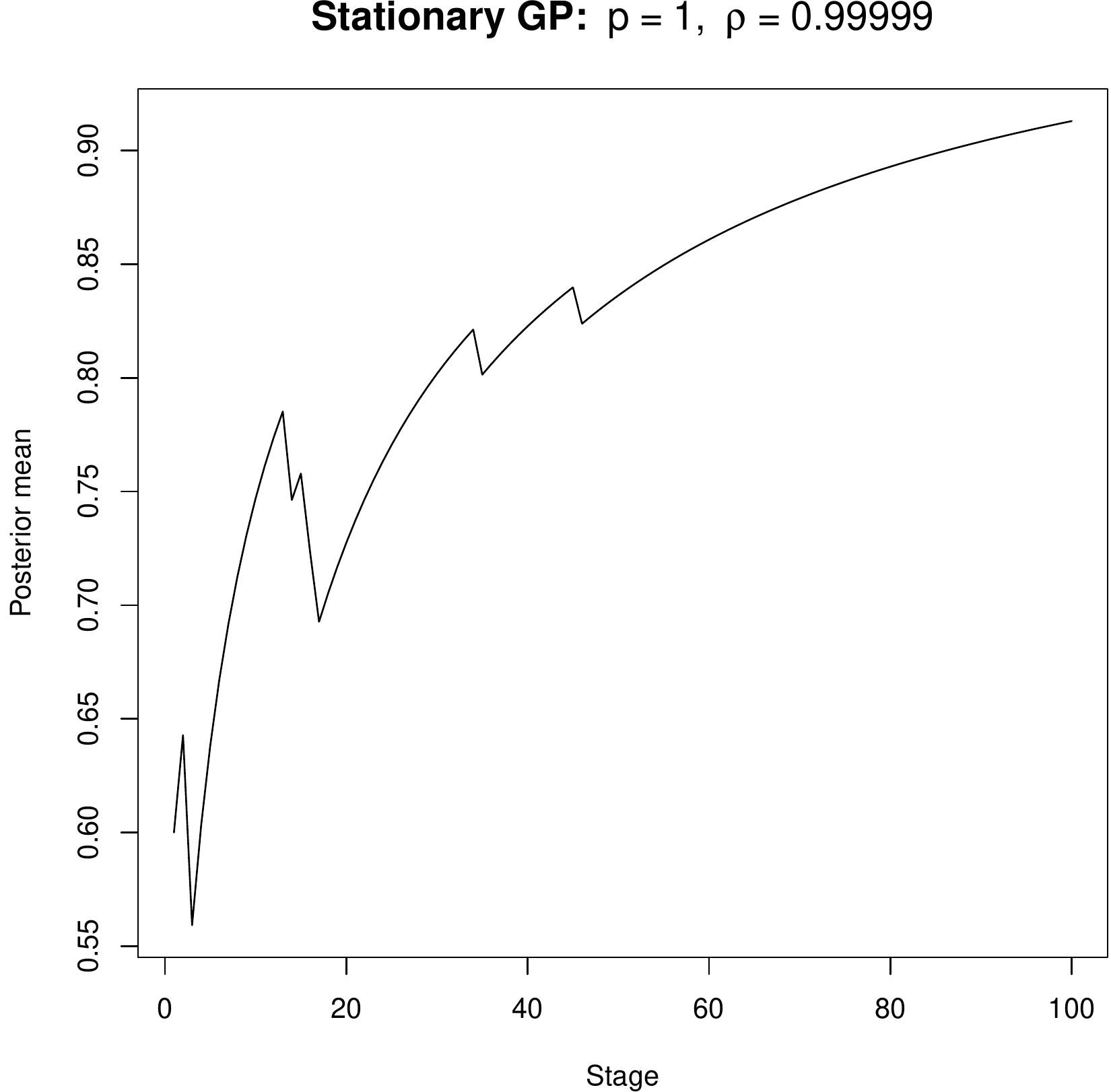}}\\
\vspace{2mm}
\subfigure [Correct detection of nonstationarity.]{ \label{fig:nonstationary1_short}
\includegraphics[width=5.5cm,height=5.5cm]{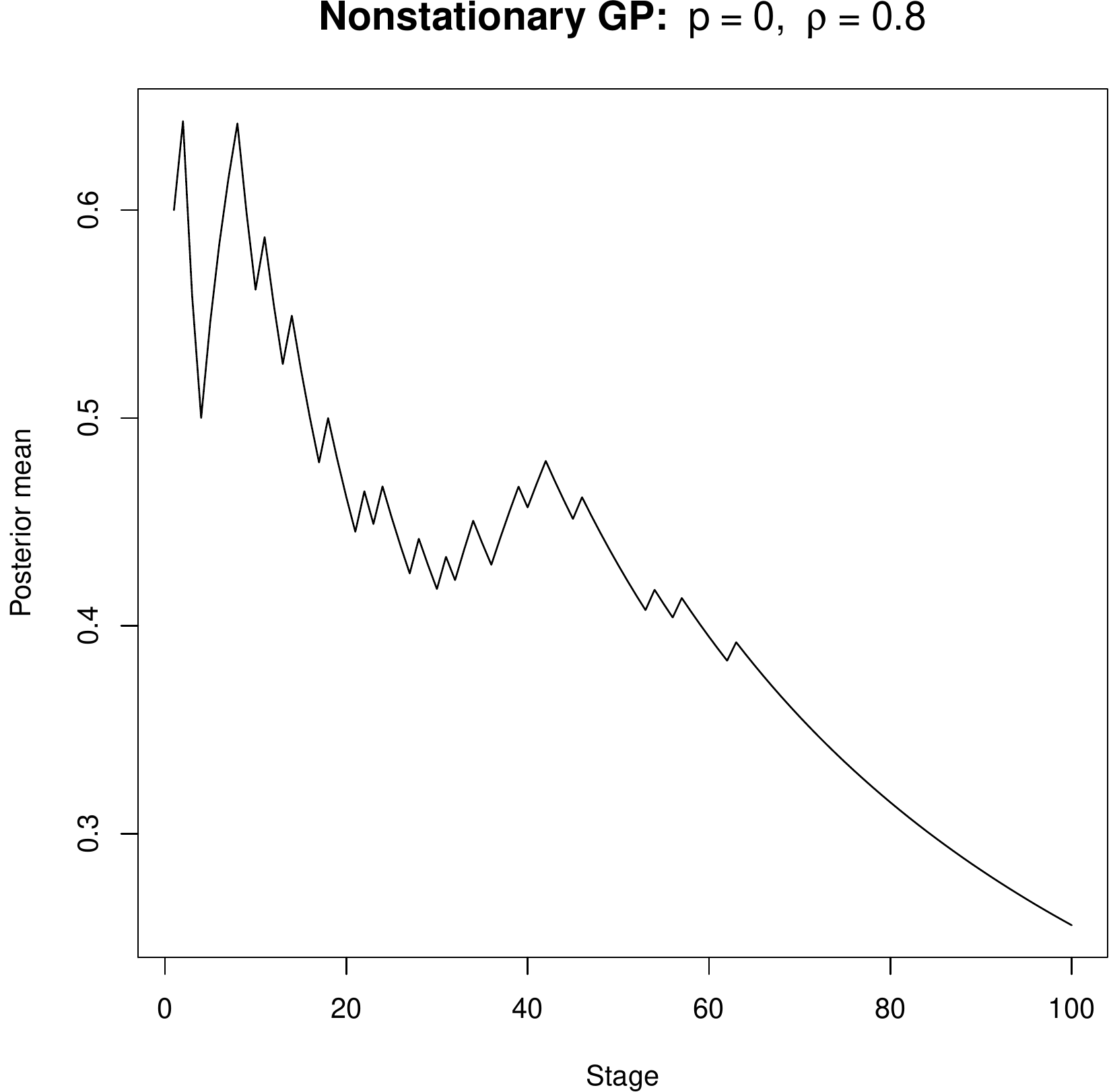}}
\hspace{2mm}
\subfigure [Correct detection of nonstationarity.]{ \label{fig:nonstationary2_short}
\includegraphics[width=5.5cm,height=5.5cm]{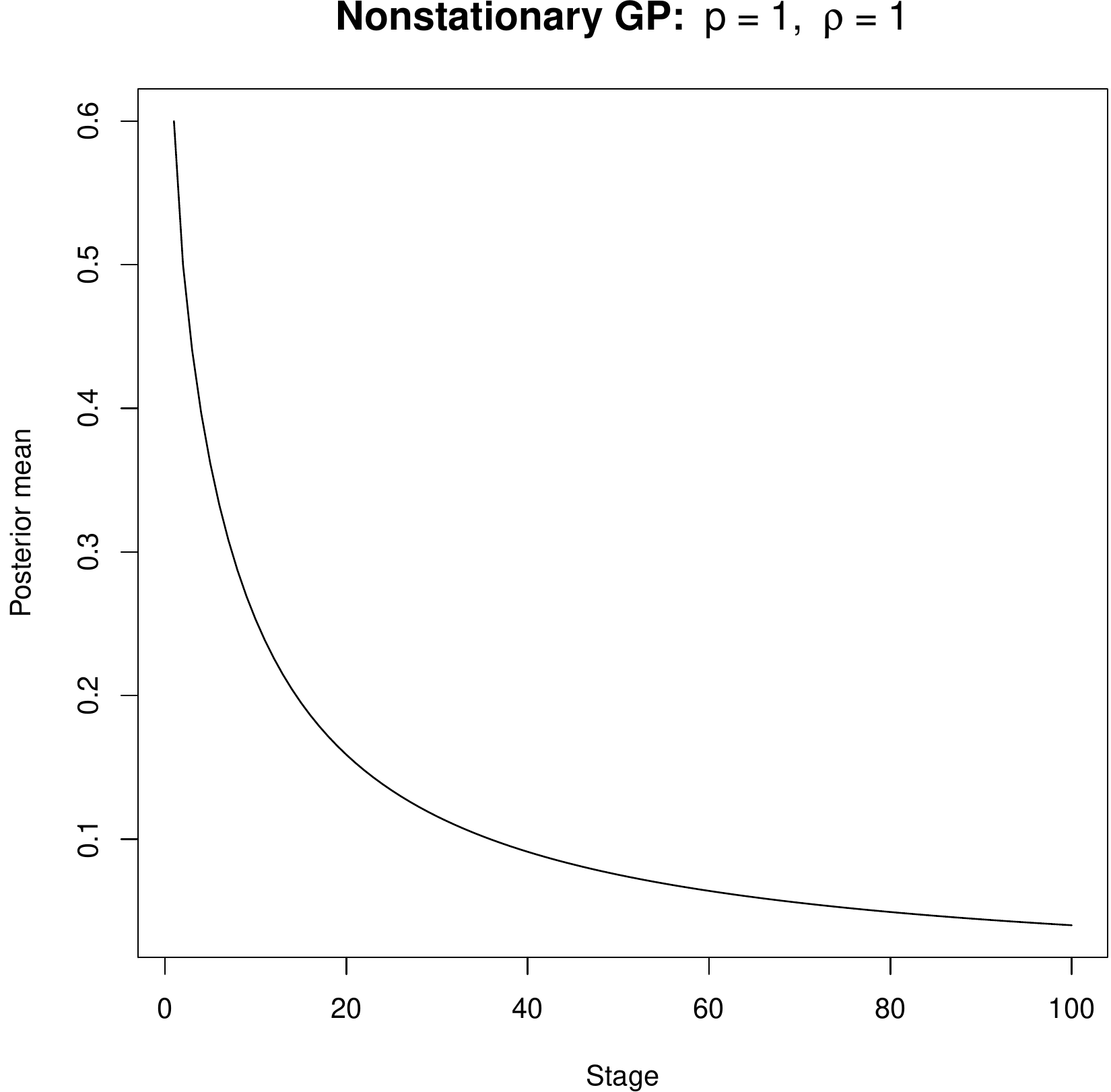}}\\
\vspace{2mm}
\subfigure [Correct detection of nonstationarity.]{ \label{fig:nonstationary3_short}
\includegraphics[width=5.5cm,height=5.5cm]{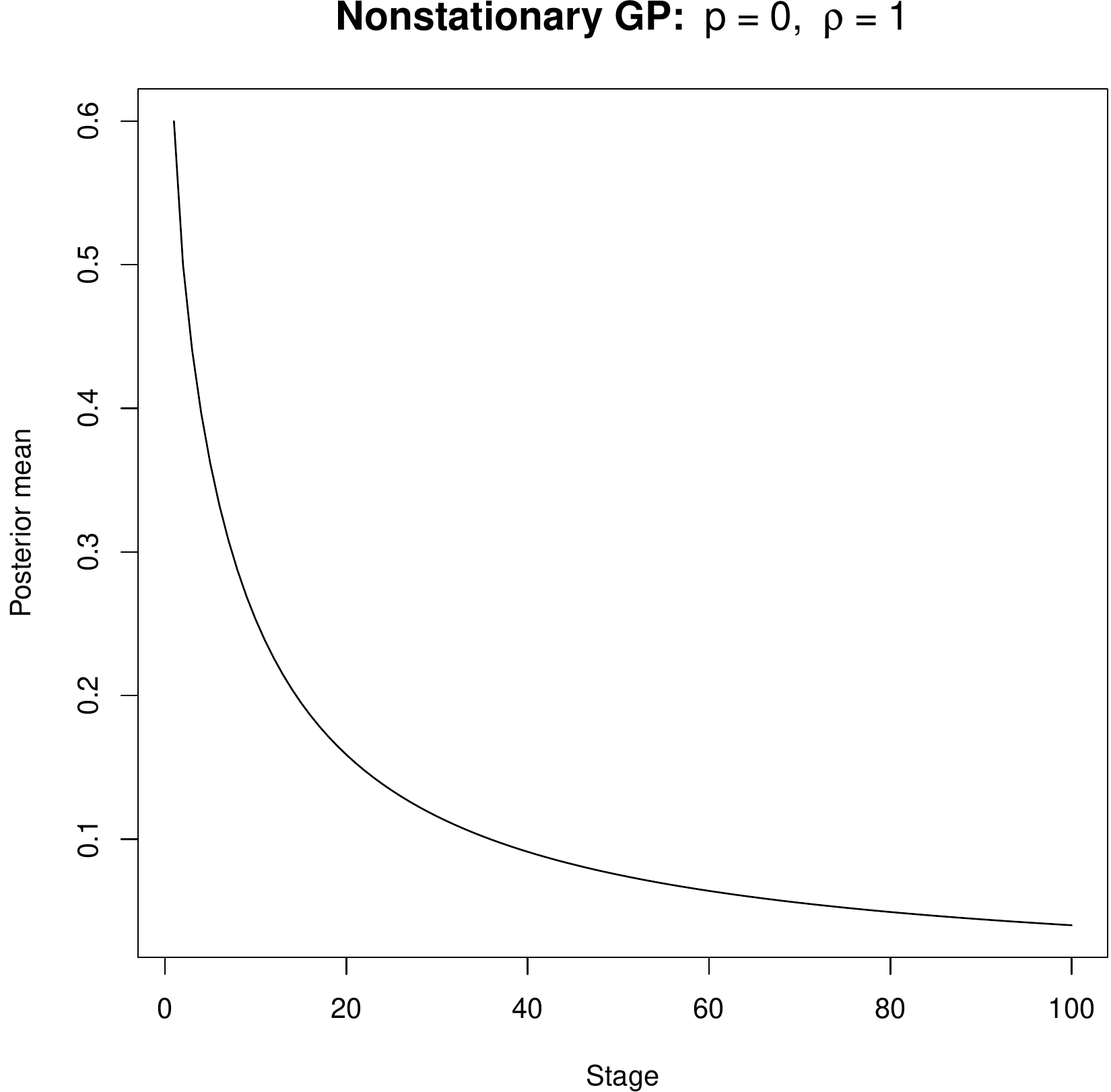}}\\
\caption{Detection of strong stationarity and nonstationarity in spatio-temporal data drawn from GPs with $50$ locations and $20$ time points.}
\label{fig:spacetime1_short}
\end{figure}

\subsection{Comparison with existing methods}
\label{subsec:comp_spacetime}
As in the spatial case, for the spatio-temporal setup, formal methods of testing stationarity are very rare in the literature.
Recently, some methods in this direction are proposed in \ctn{Soutir17b}. Indeed, the authors propose as many as $10$ test statistics to detect
covariance stationarity, under a variety of assumptions. The main ideas are similar to the testing ideas in the spatial setup proposed in \ctn{Soutir17}. 
A relevant $R$ code is provided in the webpage of the first author, but it failed to work for our simulated spatio-temporal datasets, possibly because the methods are
heavily dependent on choices of the underlying parameters involved in their methods. Instead, we apply our Bayesian methodology on the spatio-temporal
models and simulation designs to which \ctn{Soutir17b} applied their testing methods.

Following \ctn{Soutir17b}, we consider zero mean spatio-temporal processes, with $T=200$ time points and $m=100$ or $500$ locations drawn uniformly from
$\left[-\frac{\lambda}{2},\frac{\lambda}{2}\right]$. We then apply our Bayesian procedure to the $5$ spatio-temporal models considered by \ctn{Soutir17b},
under the same setups, described below.

\subsubsection{Simulations under stationarity with exponential spatial covariance function}
\label{subsubsec:spacetime_stationarity}
We generate data from the following stationary models:
\begin{itemize}
\item[(S1)] $X_{(s,t)}=0.5 X_{(s,t-1)}+\epsilon_{(s,t)}$, where $X_{s,0}=\bzero$ and $\epsilon_{(s,t)}$ are zero mean GPs independent over time
with spatial covariance structure 
\begin{equation}
Cov\left(\epsilon_{(s_1,t)},\epsilon_{(s_2,t)}\right)=\exp\left(-\|s_1-s_2\|/\psi\right).
\label{eq:expcov_spacetime}
\end{equation}
The above model defines a spatially and temporally stationary Gaussian random field.

\item[(S2)] $X_{(s,t)}=0.5 X_{(s,t-1)}+0.4 X_{(s,t-1)}\epsilon_{(s,t-1)}+\epsilon_{(s,t)}$, where $X_{s,0}=\bzero$ and $\epsilon_{(s,t)}$ are zero 
mean GPs independent over time with spatial covariance (\ref{eq:expcov_spacetime}). This model is a spatially and temporally non-Gaussian random field.
\end{itemize}
For both the above models, we set $\lambda=5$ for simulating the locations, and fix $\psi=0.5$ and $1$ for two sets of data simulations for each of $(m=100,T=200)$
and ($m=500,T=200$) sample sizes.

For checking strict stationarity, for sample size $(m=100,T=200)$, our strategy for choosing $\hat C_1$, detailed in Section \ref{subsec:C1_hat_spacetime}, 
gave $\hat C_1=0.042$, and for ($m=500,T=200$), we obtained $\hat C_1=0.045$.
As before, we consider $K=250$ clusters in both the cases. 

For covariance stationarity, we obtained $\hat C_1=0.4$ for both $(m=100,T=200)$ and 
$(m=500,T=200)$.
For the first sample size, we obtained $\mathcal N_{i,h_j,h_{j+1}}$ defined by $h_1 = 0$, $h_2 = 0.4$, $h_3 = 0.7$, $h_4=0.9$, $h_5=2$, $h_6=3$.
For the second sample size, we also obtained $h_7=4$ for model $S1$ when $\psi=5$ and for model $S2$ when $\psi=1$ and $\psi=5$.

For brevity we show the strict and weak stationarity convergence results only for $(m=100,T=200)$, with $\psi=1$, depicted as Figures \ref{fig:spacetime_soutir},
\ref{fig:spacetime_cov_soutir1} and \ref{fig:spacetime_cov_soutir2}.

\begin{figure}
\centering
\subfigure [Correct detection of stationarity.]{ \label{fig:stationary1_soutir}
\includegraphics[width=5.5cm,height=5.5cm]{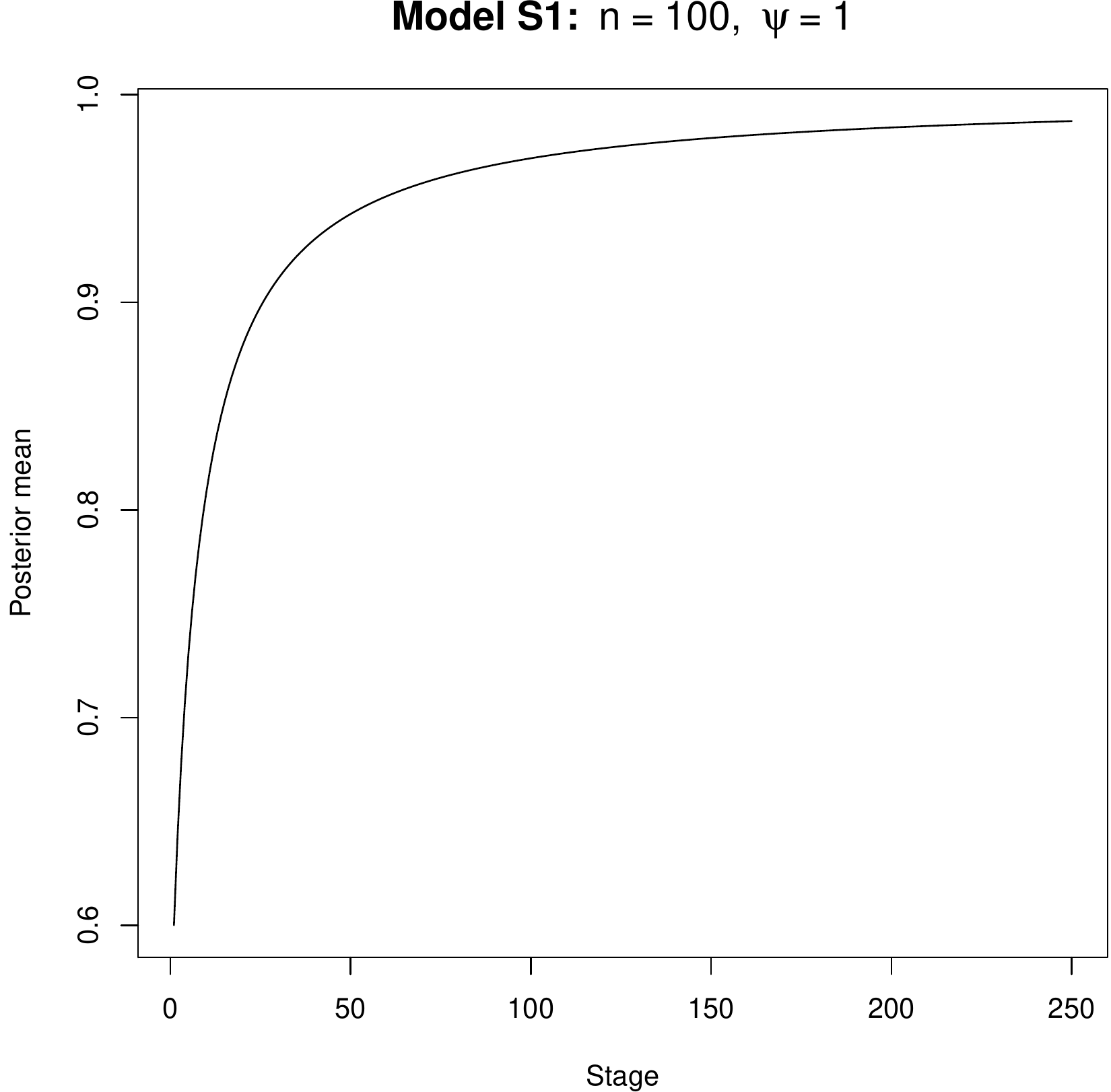}}
\hspace{2mm}
\subfigure [Correct detection of stationarity.]{ \label{fig:stationary2_soutir}
\includegraphics[width=5.5cm,height=5.5cm]{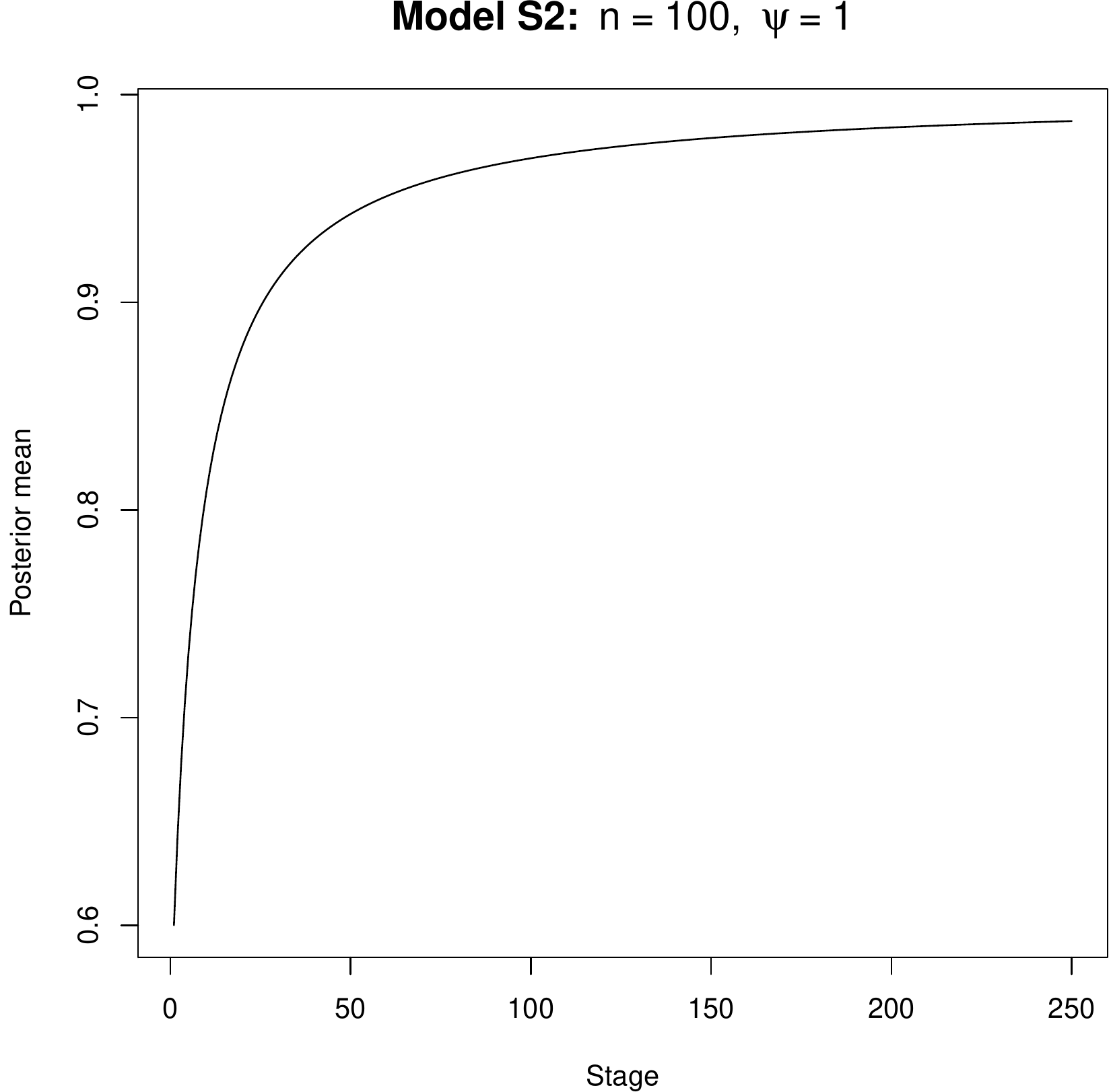}}\\
\caption{Detection of strong stationarity in spatio-temporal data drawn from models $S1$ and $S2$ 
with sample size $100$ locations and $200$ time points, with $\psi=1$ and $\lambda=5$.}
\label{fig:spacetime_soutir}
\end{figure}

\begin{figure}
\centering
\subfigure [$0\leq\|h\|<0.4$]{ \label{fig:spacetime_covns1_soutir1}
\includegraphics[width=5.5cm,height=5.5cm]{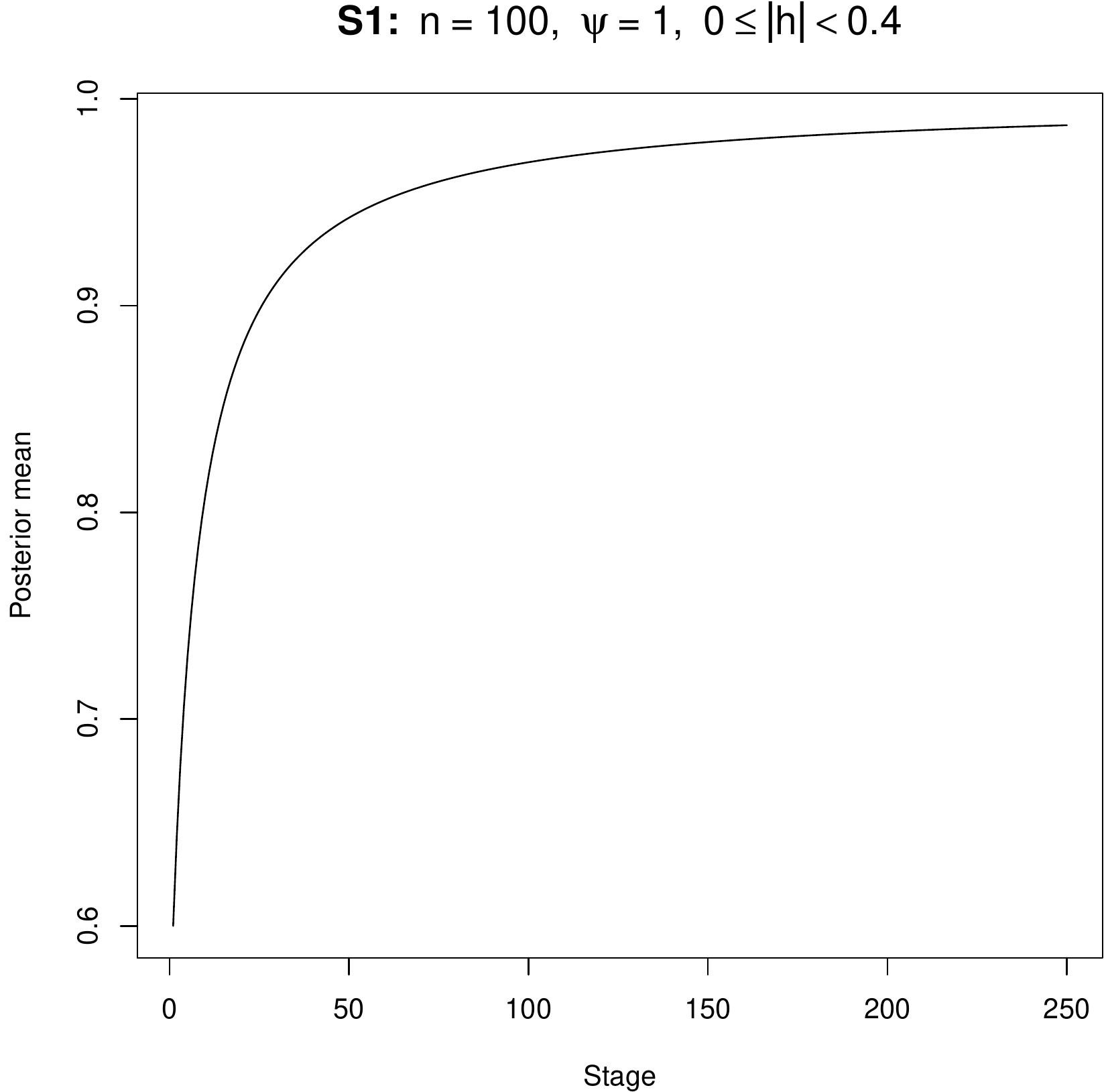}}
\hspace{2mm}
\subfigure [$0.4\leq\|h\|<0.7$.]{ \label{fig:spacetime_covns2_soutir1}
\includegraphics[width=5.5cm,height=5.5cm]{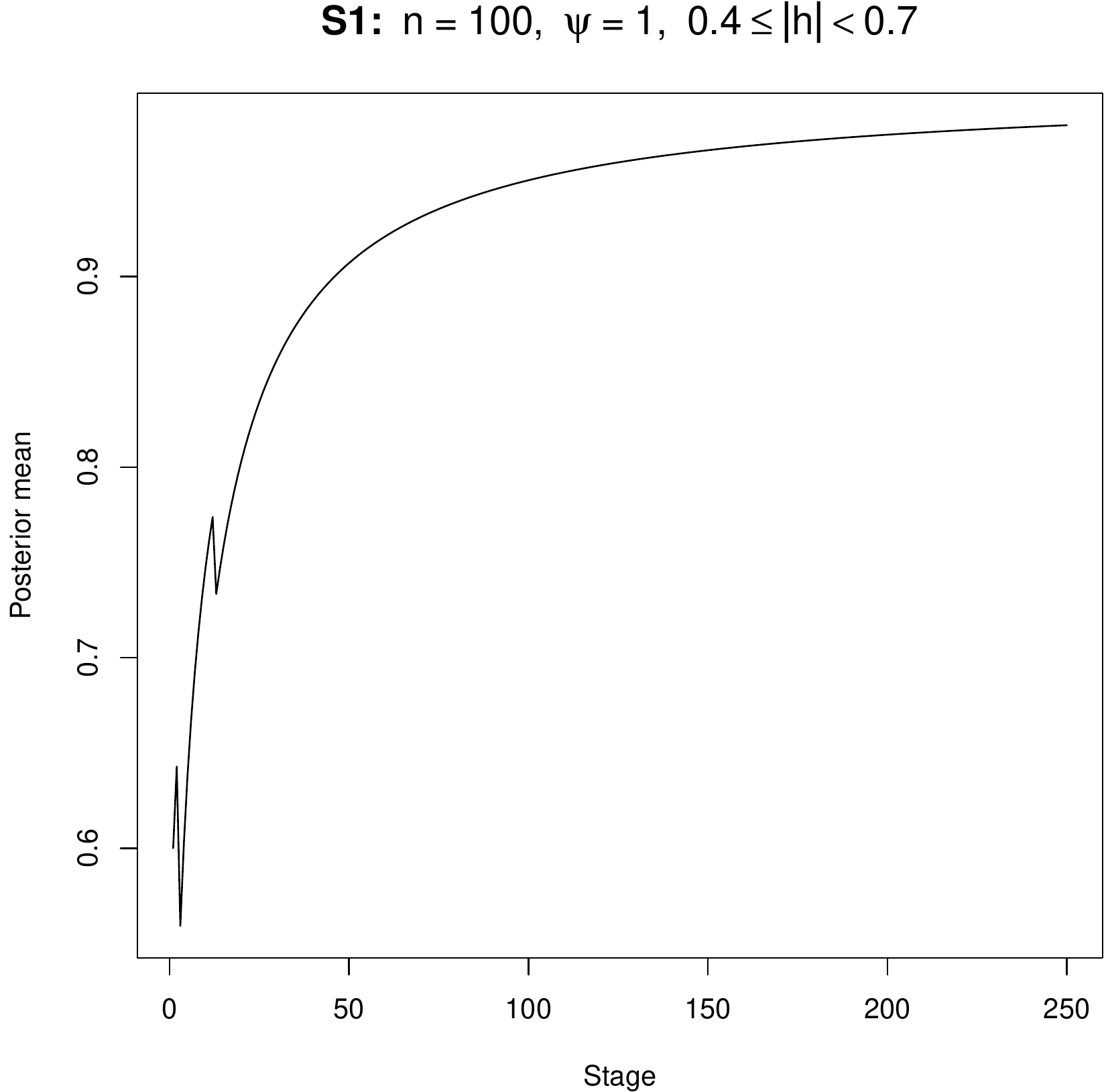}}\\
\vspace{2mm}
\subfigure [$0.7\leq\|h\|<0.9$]{ \label{fig:spacetime_covns3_soutir1}
\includegraphics[width=5.5cm,height=5.5cm]{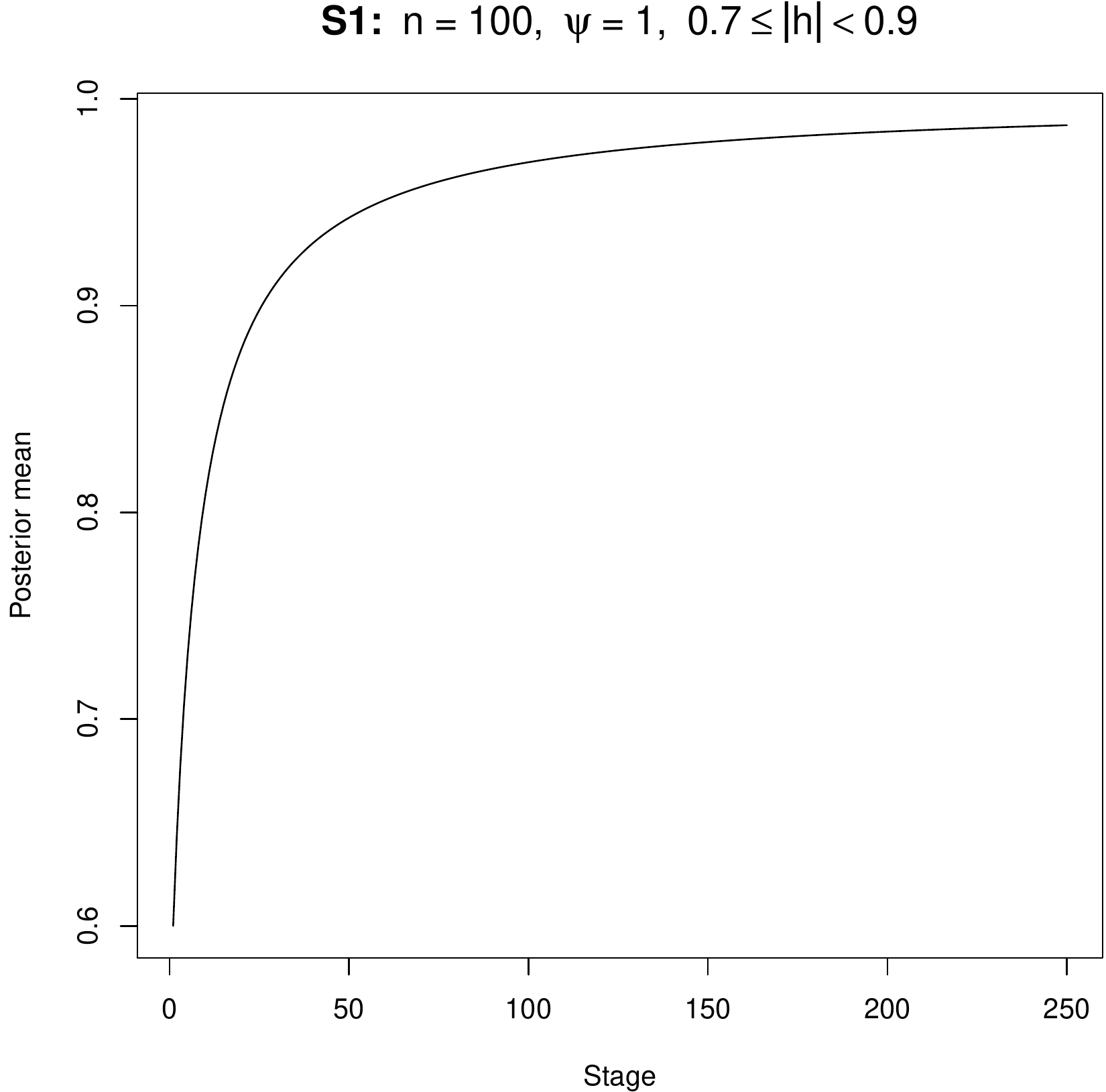}}
\hspace{2mm}
\subfigure [$0.9\leq\|h\|<2$.]{ \label{fig:spacetime_covns4_soutir1}
\includegraphics[width=5.5cm,height=5.5cm]{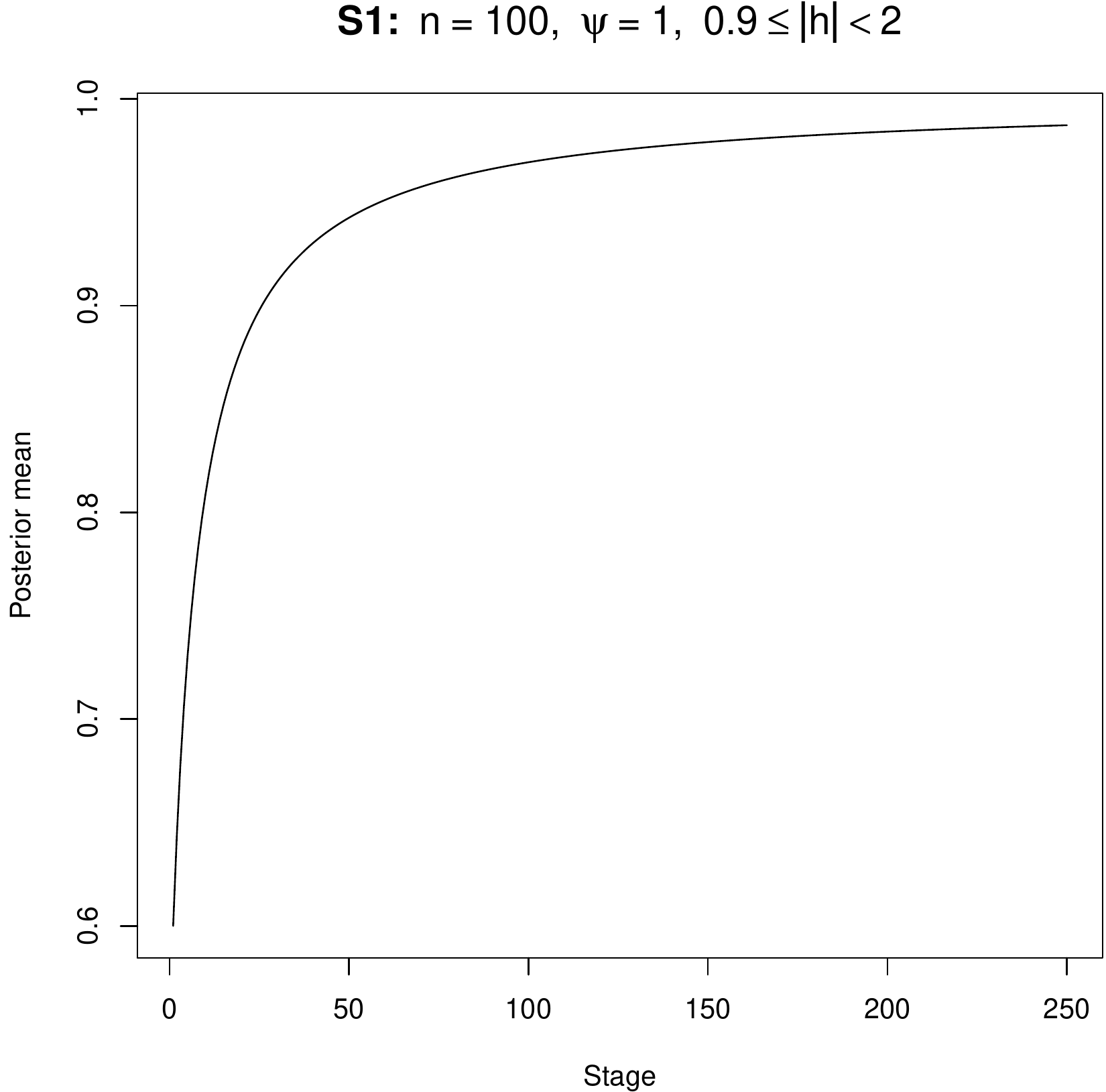}}\\
\vspace{2mm}
\subfigure [$2\leq\|h\|<3$.]{ \label{fig:spacetime_covns5_soutir1}
\includegraphics[width=5.5cm,height=5.5cm]{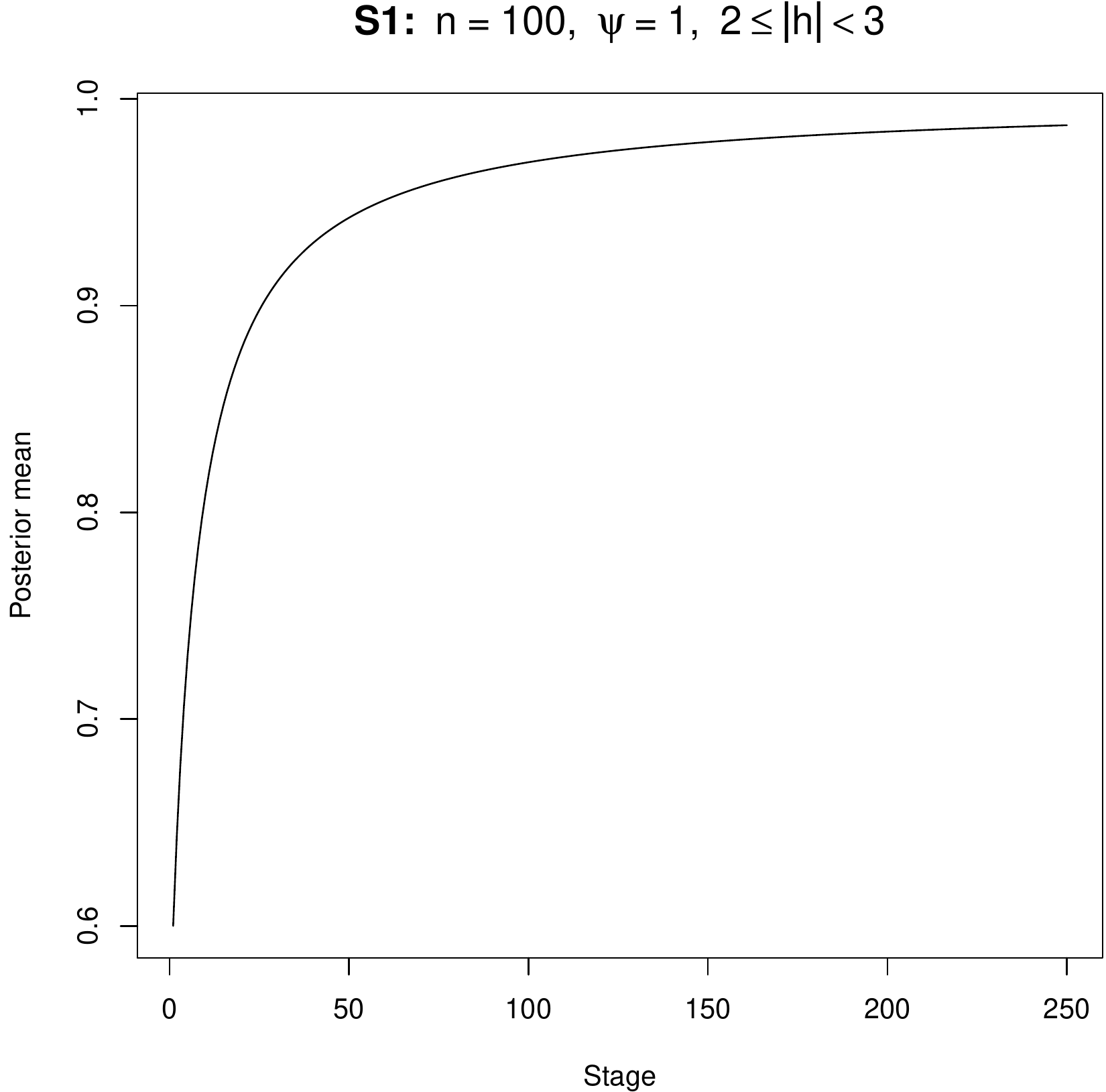}}\\
\caption{Detection of covariance stationarity in spatio-temporal data drawn from model $S1$ with sample size $100$ locations and $200$ time points, 
with $\psi=1$ and $\lambda=5$.}
\label{fig:spacetime_cov_soutir1}
\end{figure}

\begin{figure}
\centering
\subfigure [$0\leq\|h\|<0.4$]{ \label{fig:spacetime_covns1_soutir2}
\includegraphics[width=5.5cm,height=5.5cm]{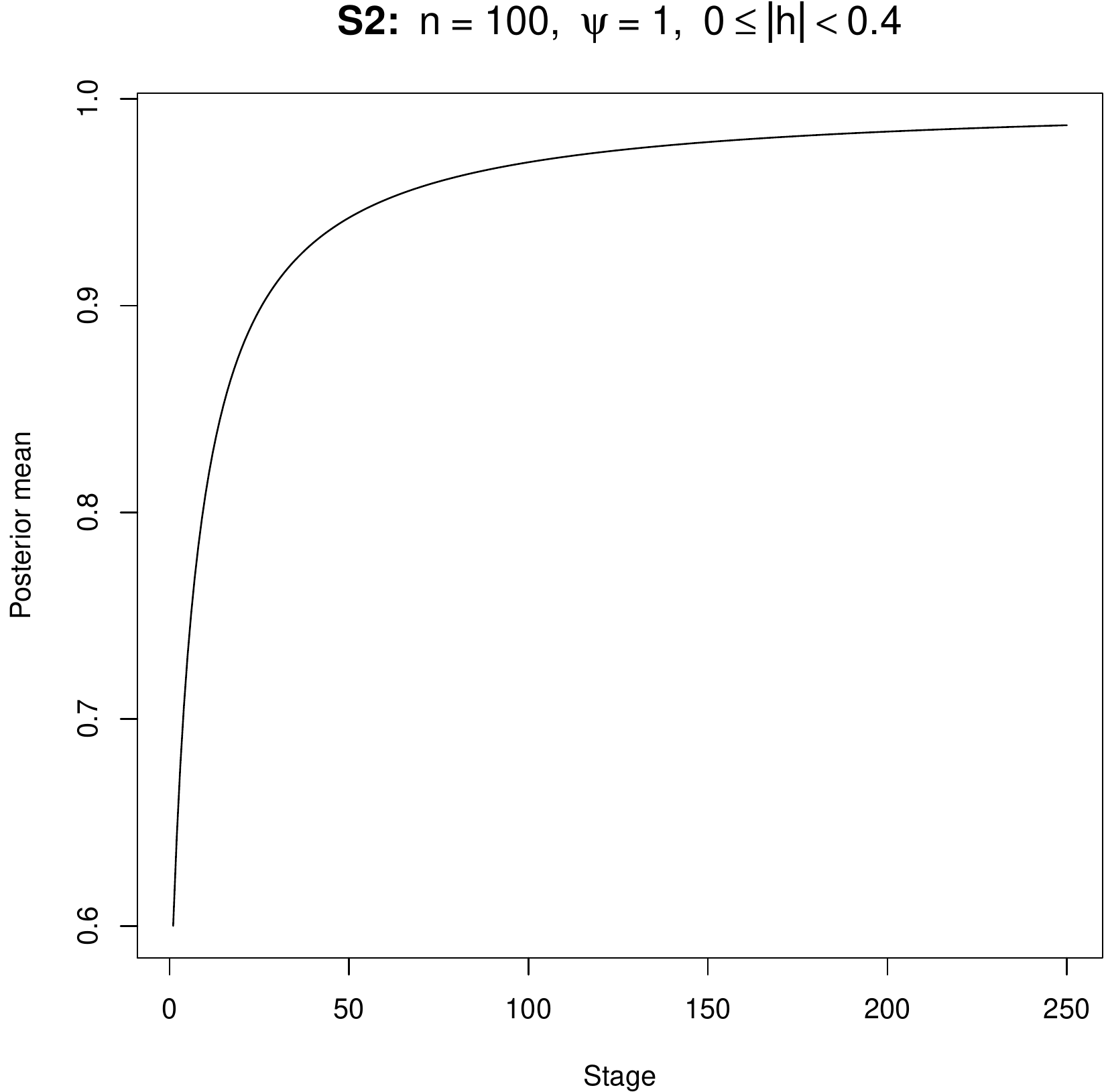}}
\hspace{2mm}
\subfigure [$0.4\leq\|h\|<0.7$.]{ \label{fig:spacetime_covns2_soutir2}
\includegraphics[width=5.5cm,height=5.5cm]{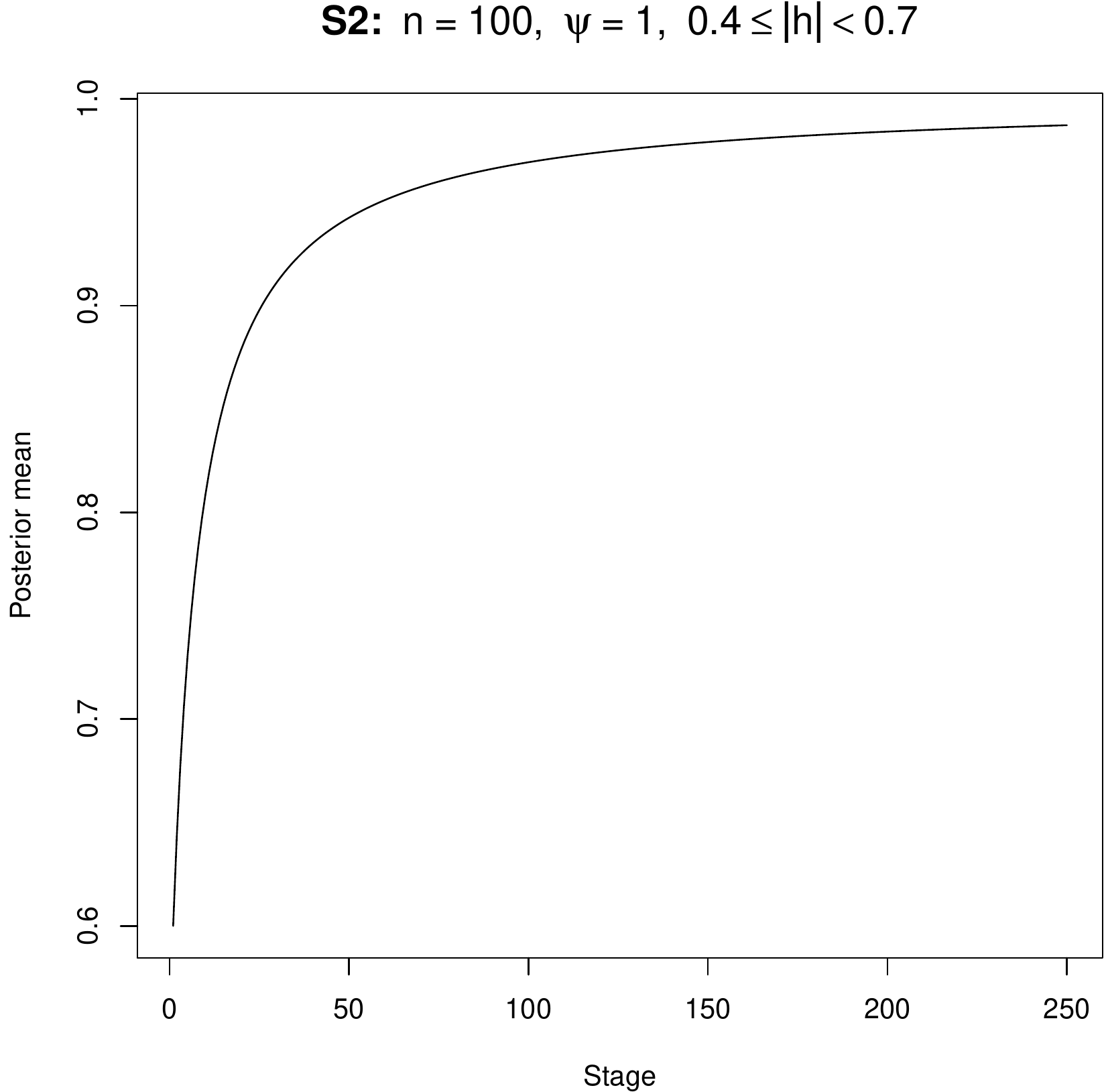}}\\
\vspace{2mm}
\subfigure [$0.7\leq\|h\|<0.9$]{ \label{fig:spacetime_covns3_soutir2}
\includegraphics[width=5.5cm,height=5.5cm]{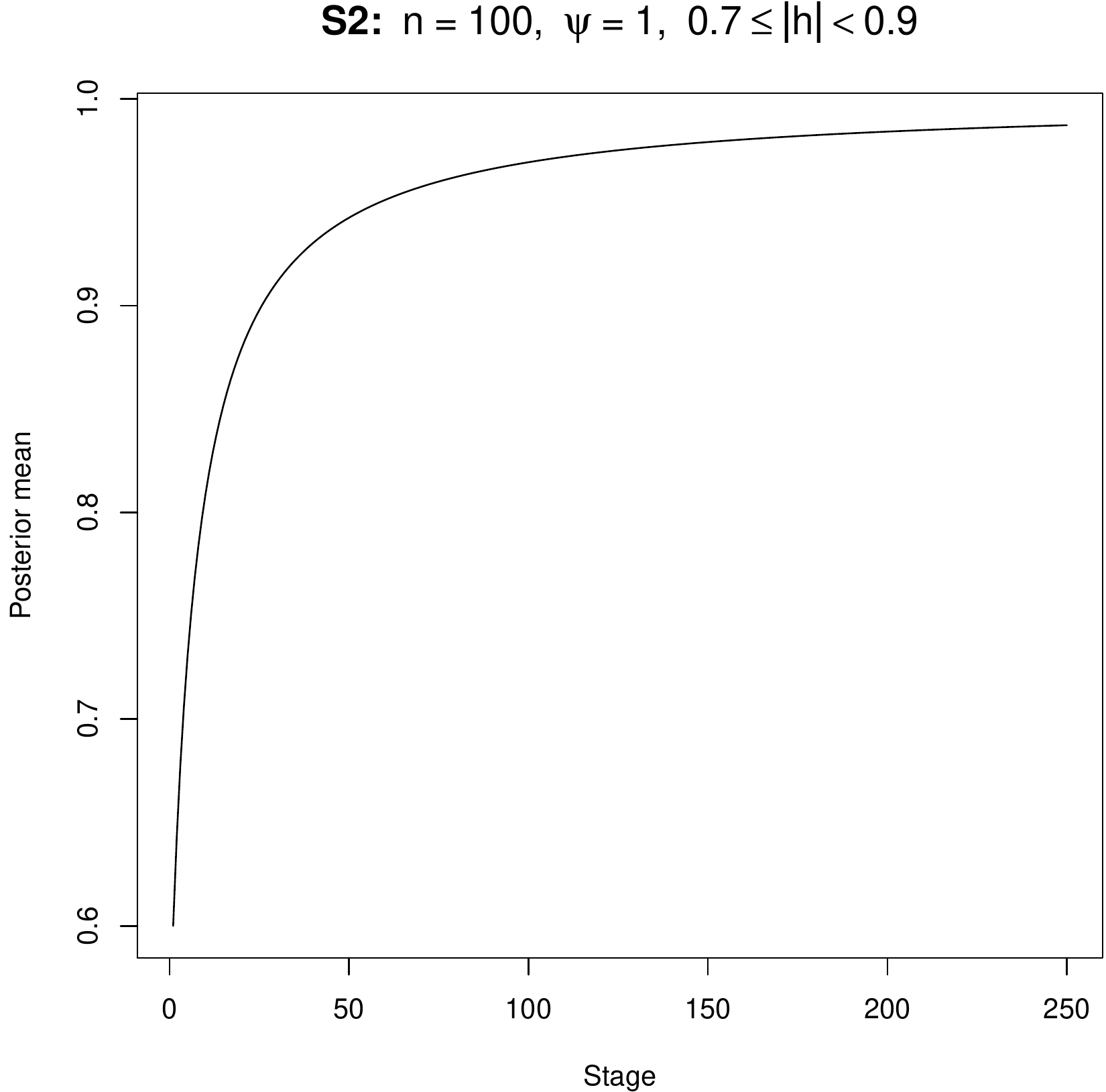}}
\hspace{2mm}
\subfigure [$0.9\leq\|h\|<2$.]{ \label{fig:spacetime_covns4_soutir2}
\includegraphics[width=5.5cm,height=5.5cm]{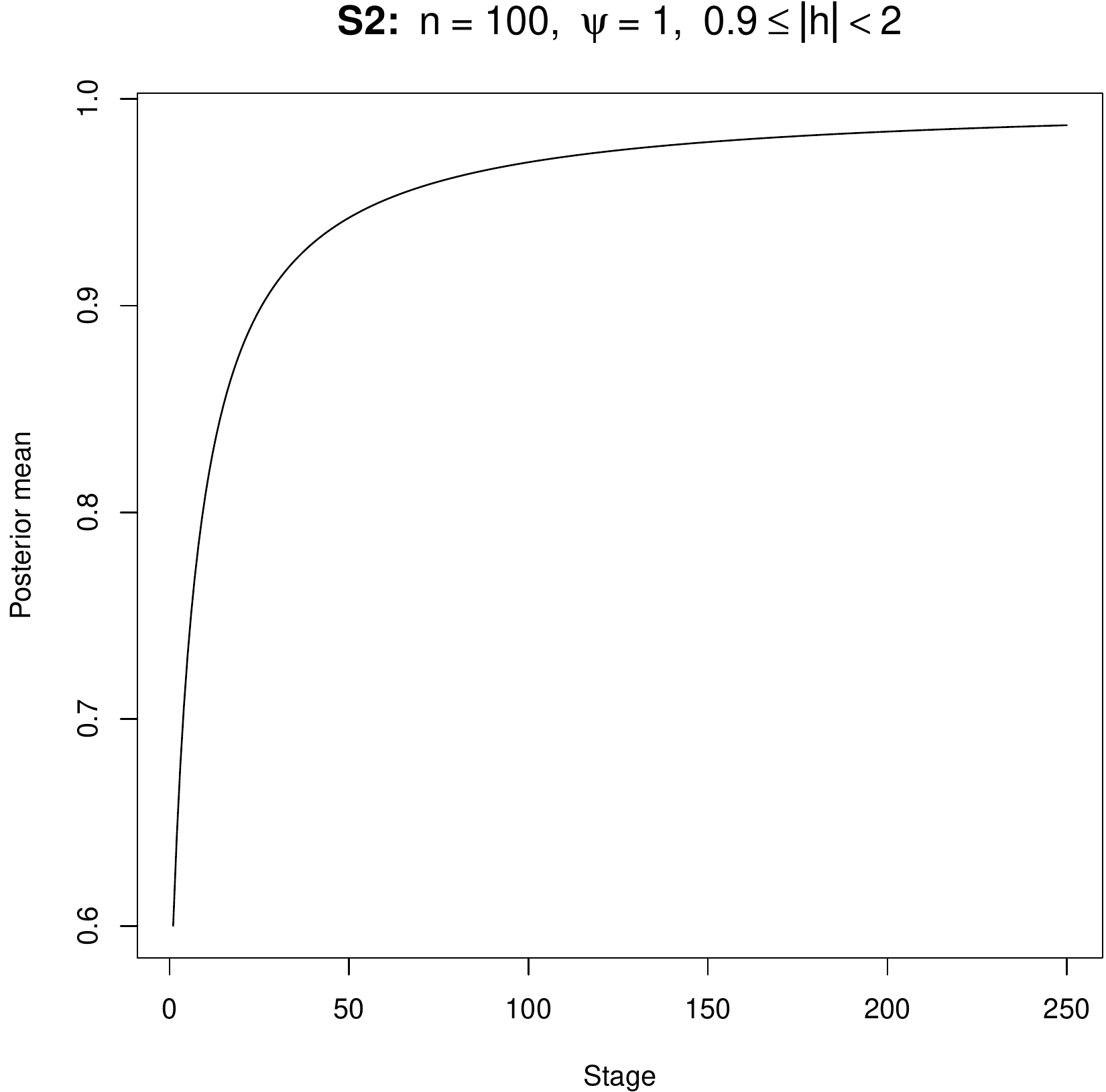}}\\
\vspace{2mm}
\subfigure [$2\leq\|h\|<3$.]{ \label{fig:spacetime_covns5_soutir2}
\includegraphics[width=5.5cm,height=5.5cm]{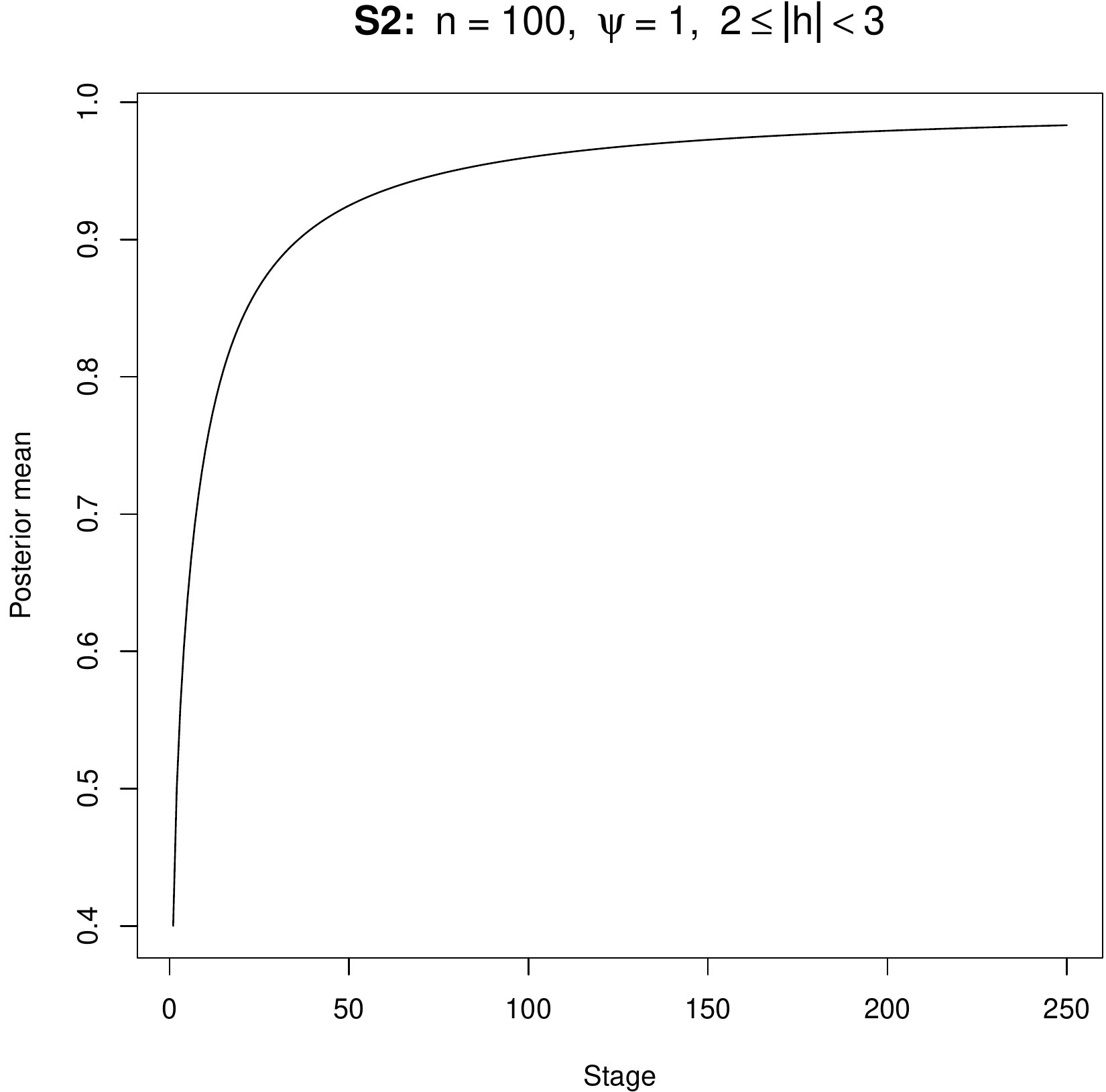}}\\
\caption{Detection of covariance stationarity in spatio-temporal data drawn from model $S2$ with sample size $100$ locations and $200$ time points, 
with $\psi=1$ and $\lambda=5$.}
\label{fig:spacetime_cov_soutir2}
\end{figure}

\subsubsection{Simulations under stationarity with Whittle spatial covariance function}
\label{subsubsec:spacetime_stationarity_whittle}
Following \ctn{Soutir17b} we now repeat the above experiments with the same models $S1$ and $S2$ but with the exponential covariance functions 
replaced with the Whittle covariance function (\ref{eq:spatial_bound}), with $\psi=0.37$ and $0.72$.  
Note that the values of $\hat C_1$ remain the same as before; however, the minimum values of $\hat C_1$ for which covariance stationarities were achieved, varied between
$0.15$, $0.2$ and $0.3$.

As expected, we obtained excellent results in all the cases, but present the results corresponding to $(m=100,T=200)$ and $\psi=0.72$ for brevity.
Figures \ref{fig:spacetime_soutir_w}, \ref{fig:spacetime_cov_soutir1_w} and \ref{fig:spacetime_cov_soutir2_w} depict our Bayesian results regarding
strict and weak stationarities of the models $S1$ and $S2$.

\begin{figure}
\centering
\subfigure [Correct detection of stationarity.]{ \label{fig:stationary1_soutir_w}
\includegraphics[width=5.5cm,height=5.5cm]{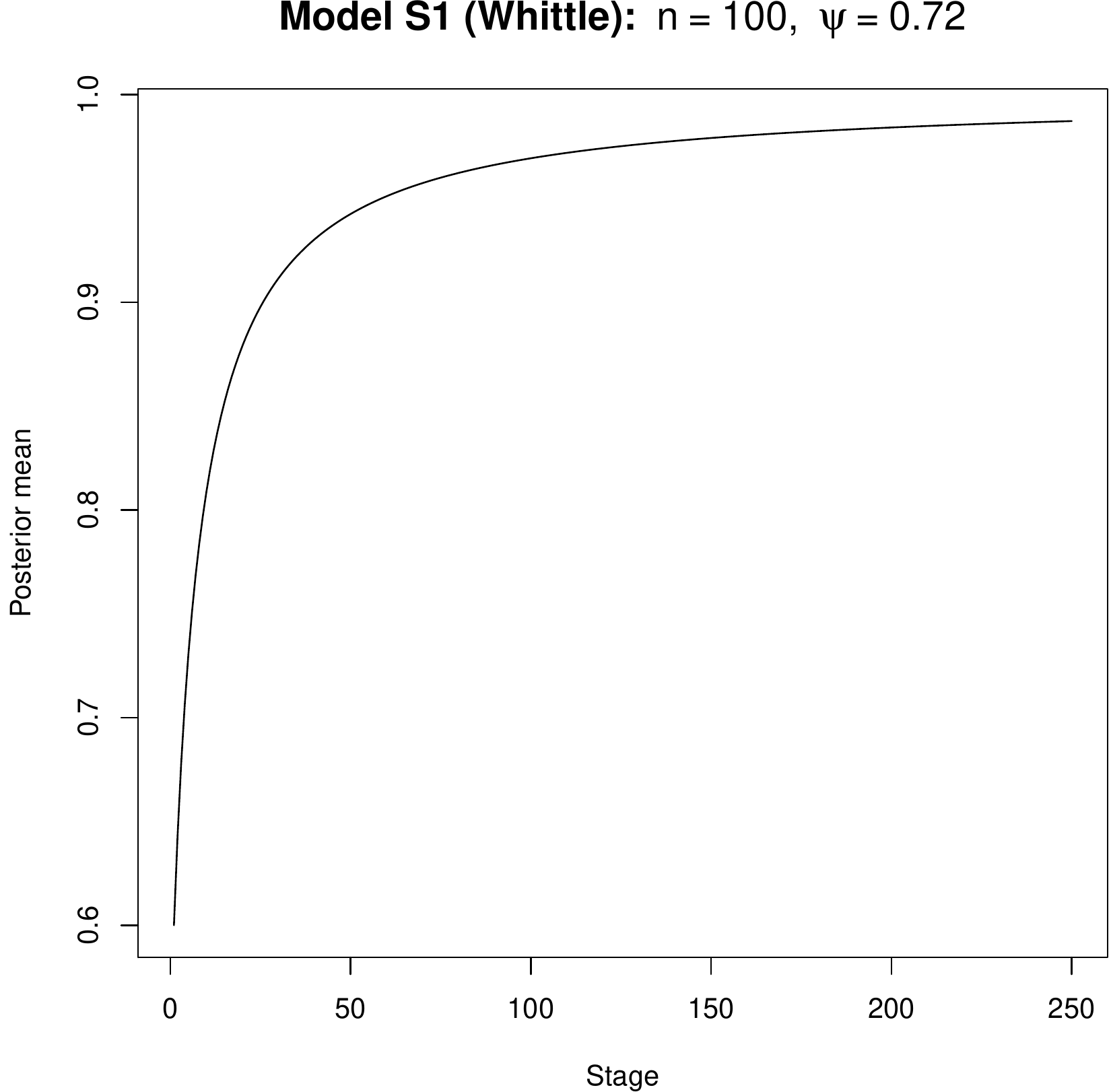}}
\hspace{2mm}
\subfigure [Correct detection of stationarity.]{ \label{fig:stationary2_soutir_w}
\includegraphics[width=5.5cm,height=5.5cm]{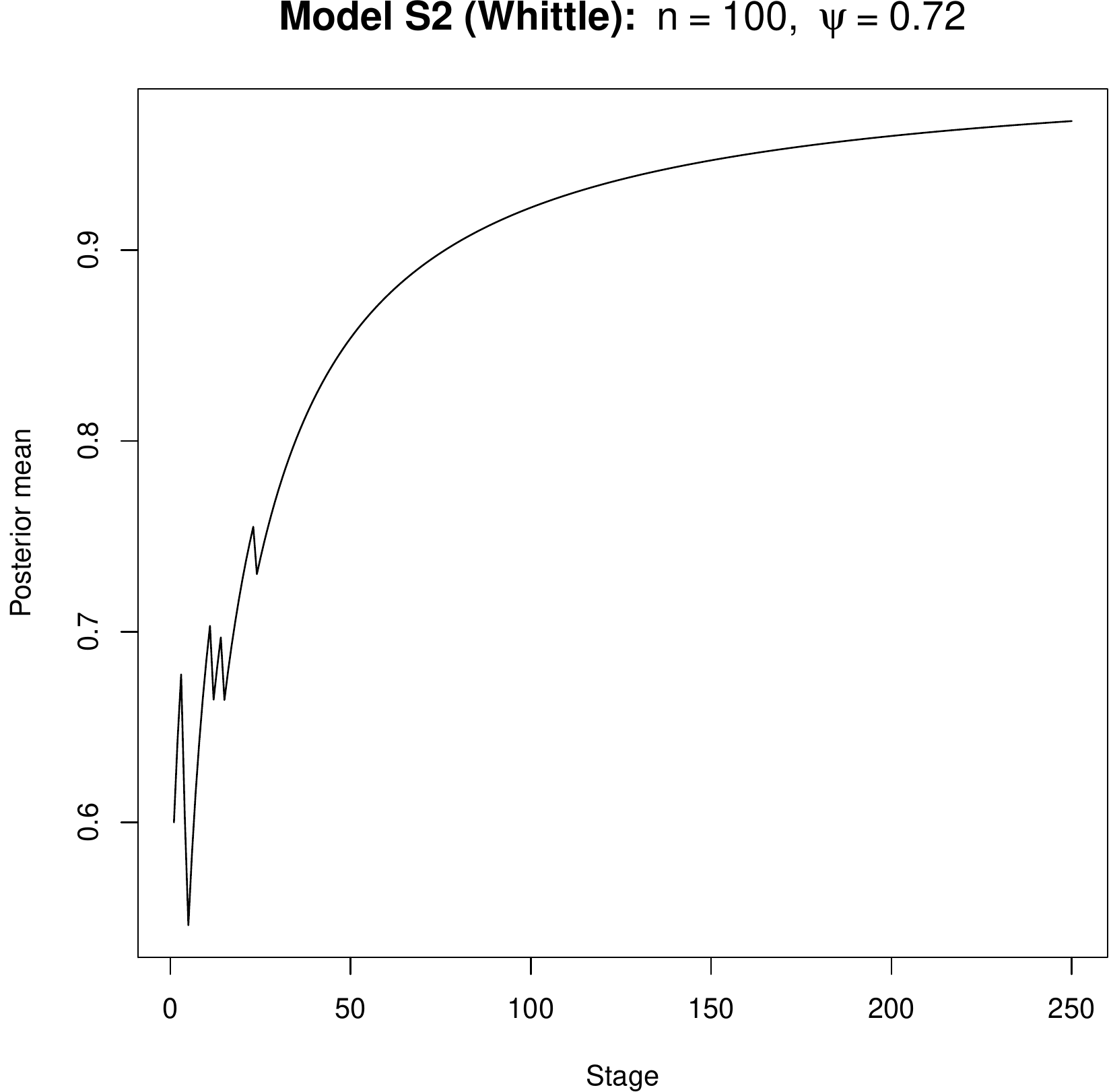}}\\
\caption{Detection of strong stationarity in spatio-temporal data drawn from models $S1$ and $S2$ 
with sample size $100$ locations and $200$ time points, corresponding to Whittle spatial covariance with $\psi=0.72$ and $\lambda=5$.}
\label{fig:spacetime_soutir_w}
\end{figure}

\begin{figure}
\centering
\subfigure [$0\leq\|h\|<0.4$]{ \label{fig:spacetime_covns1_soutir1_w}
\includegraphics[width=5.5cm,height=5.5cm]{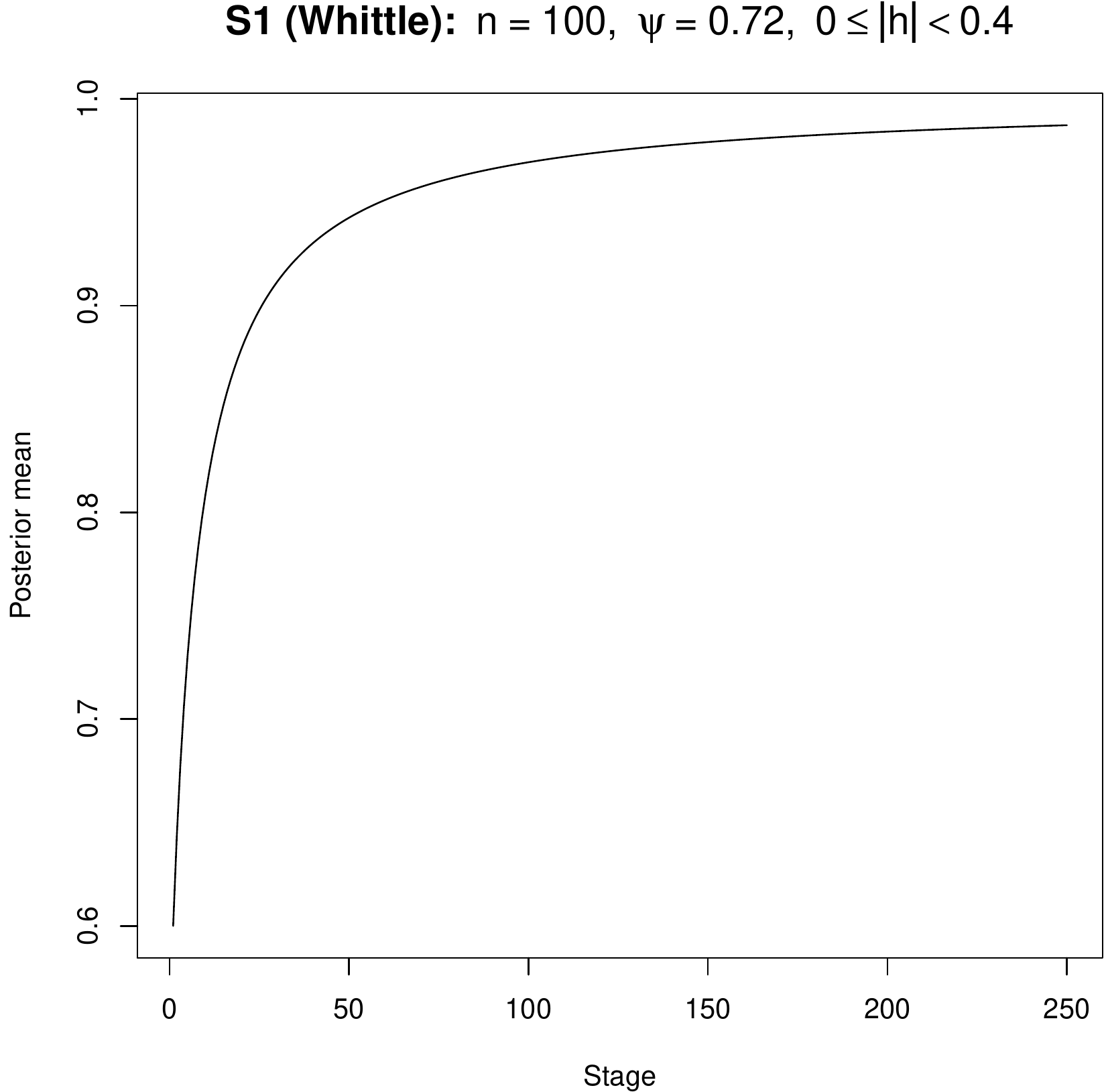}}
\hspace{2mm}
\subfigure [$0.4\leq\|h\|<0.7$.]{ \label{fig:spacetime_covns2_soutir1_w}
\includegraphics[width=5.5cm,height=5.5cm]{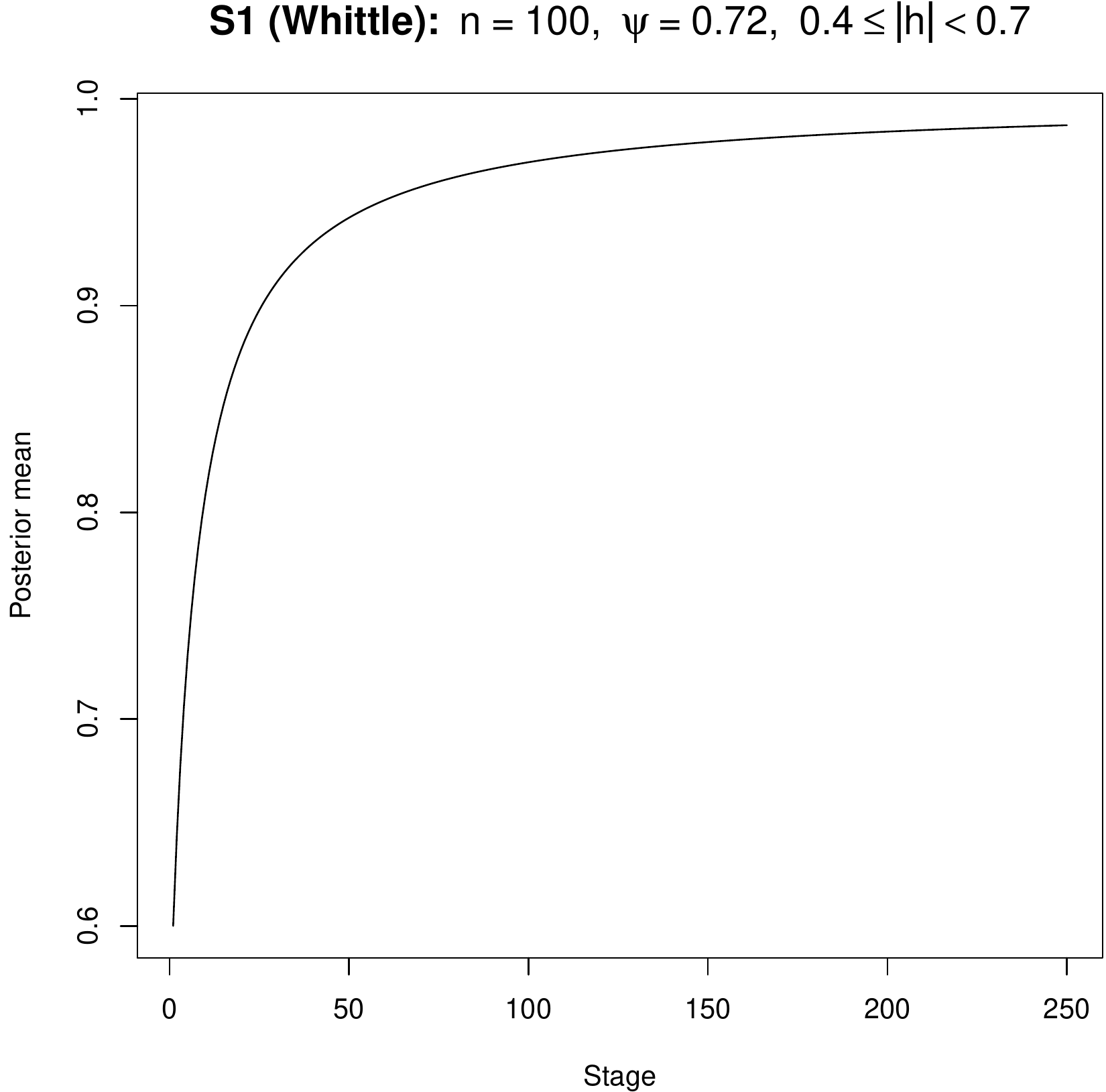}}\\
\vspace{2mm}
\subfigure [$0.7\leq\|h\|<0.9$]{ \label{fig:spacetime_covns3_soutir1_w}
\includegraphics[width=5.5cm,height=5.5cm]{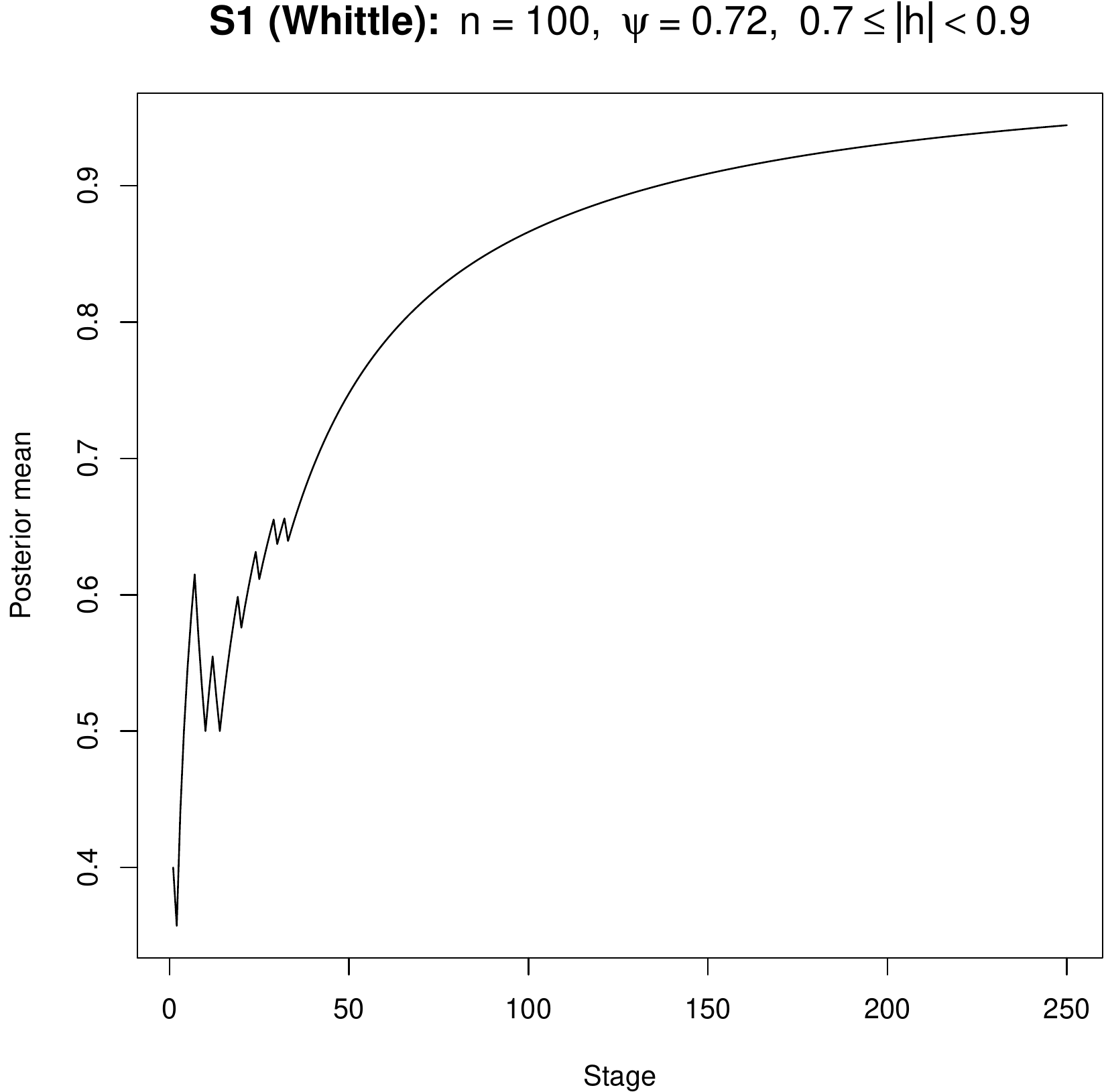}}
\hspace{2mm}
\subfigure [$0.9\leq\|h\|<2$.]{ \label{fig:spacetime_covns4_soutir1_w}
\includegraphics[width=5.5cm,height=5.5cm]{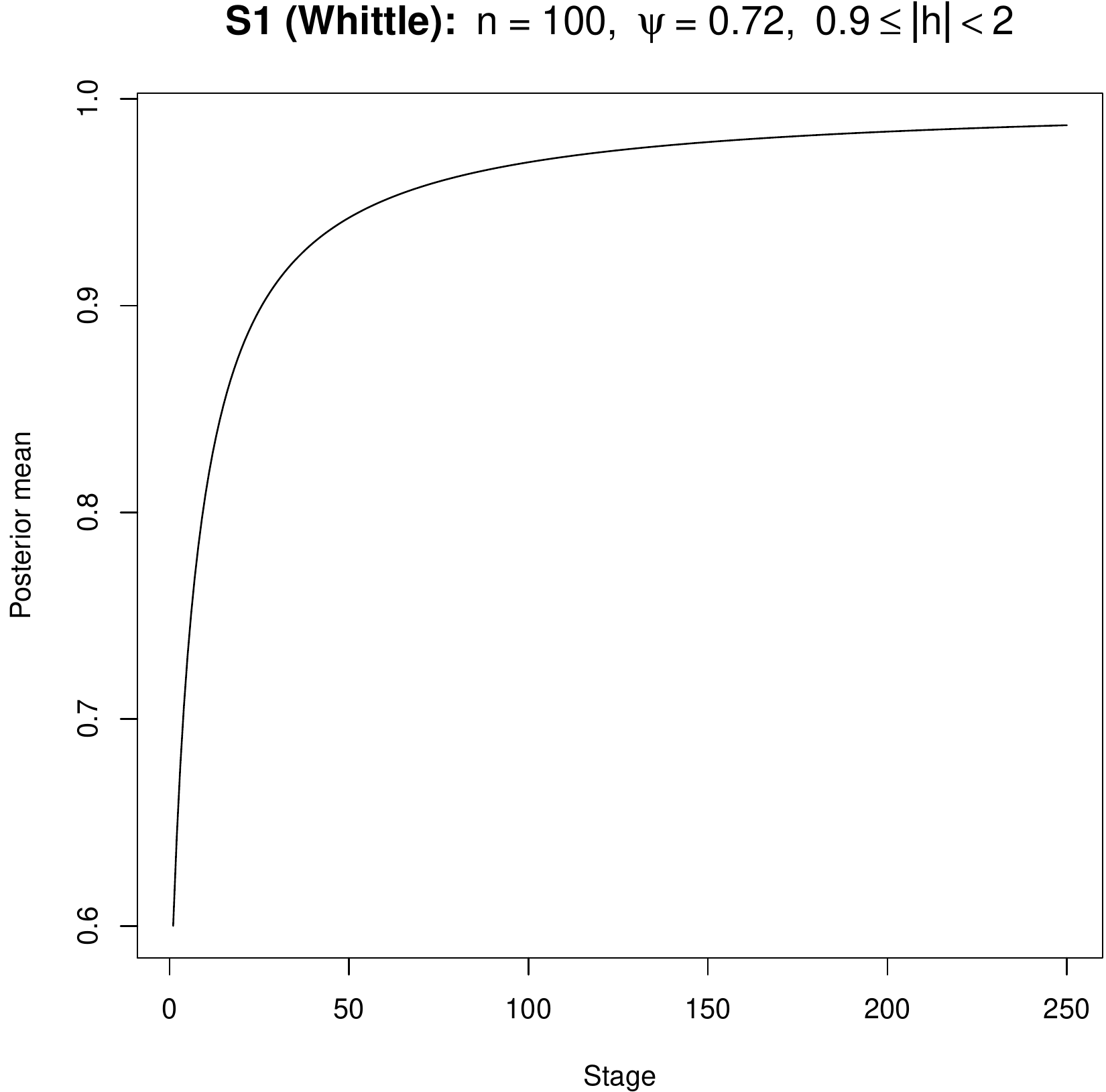}}\\
\vspace{2mm}
\subfigure [$2\leq\|h\|<3$.]{ \label{fig:spacetime_covns5_soutir1_w}
\includegraphics[width=5.5cm,height=5.5cm]{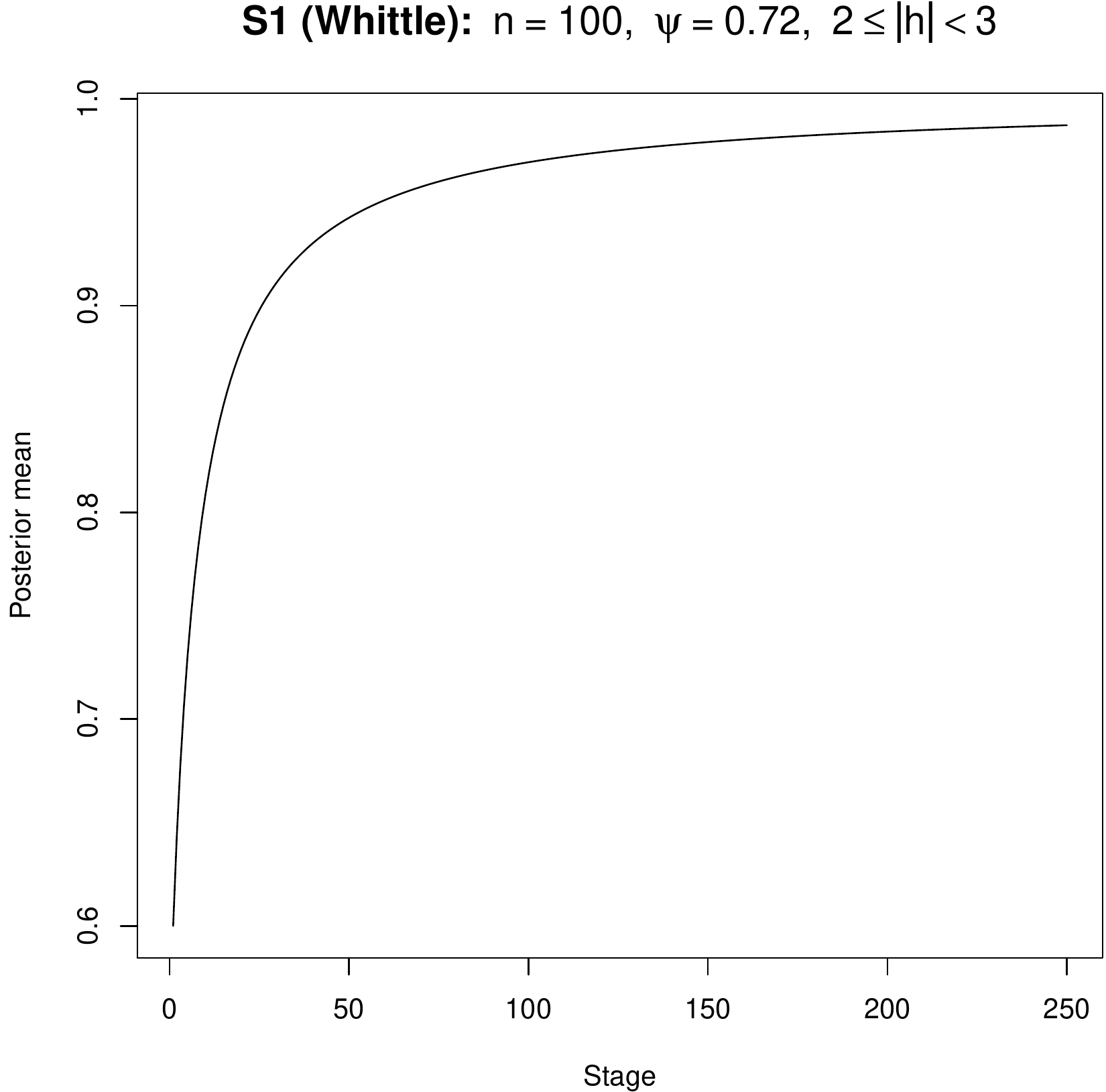}}\\
\caption{Detection of covariance stationarity in spatio-temporal data drawn from model $S1$ with sample size $100$ locations and $200$ time points, 
corresponding to Whittle spatial covariance with $\psi=0.72$ and $\lambda=5$.}
\label{fig:spacetime_cov_soutir1_w}
\end{figure}

\begin{figure}
\centering
\subfigure [$0\leq\|h\|<0.4$]{ \label{fig:spacetime_covns1_soutir2_w}
\includegraphics[width=5.5cm,height=5.5cm]{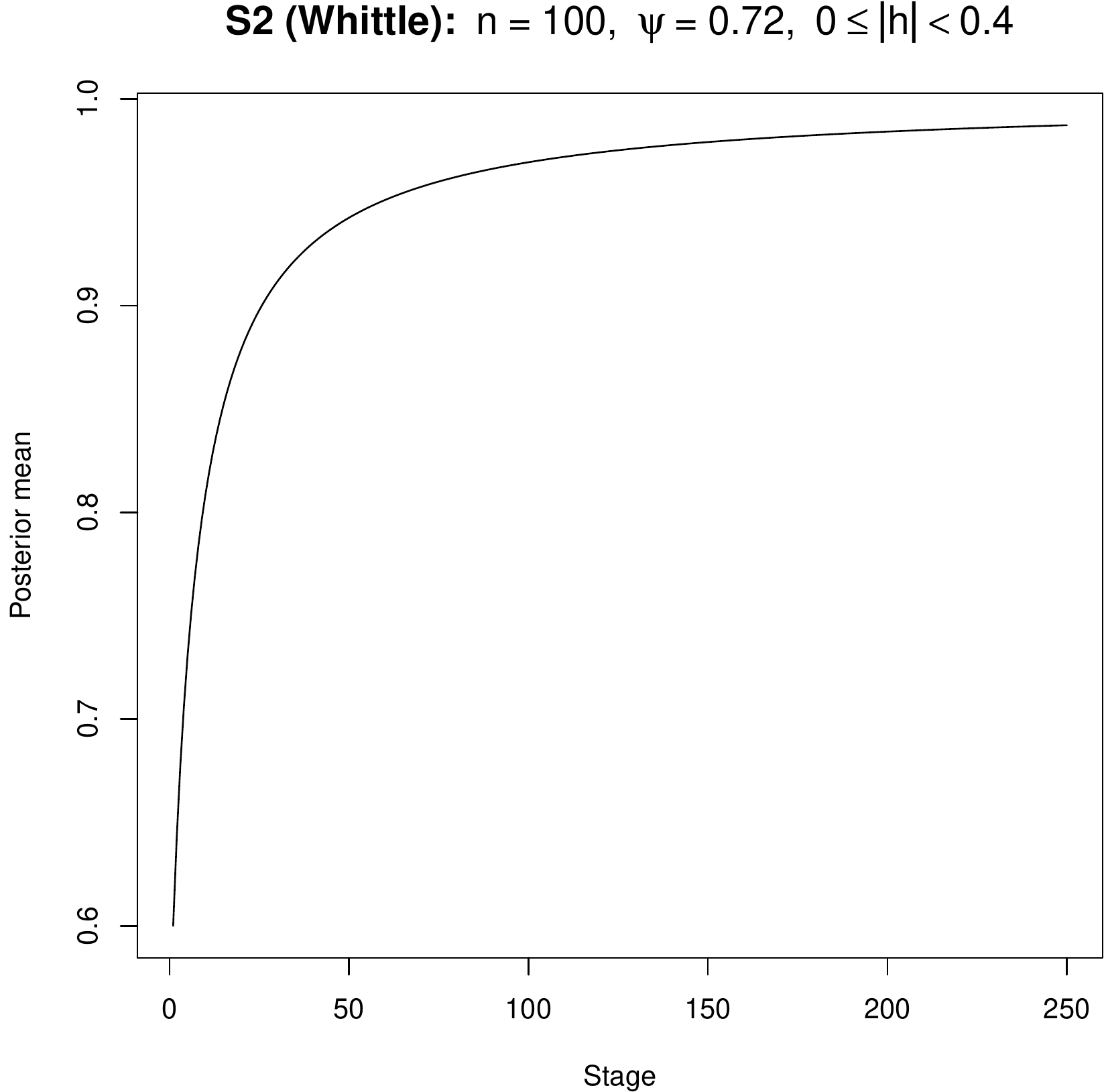}}
\hspace{2mm}
\subfigure [$0.4\leq\|h\|<0.7$.]{ \label{fig:spacetime_covns2_soutir2_w}
\includegraphics[width=5.5cm,height=5.5cm]{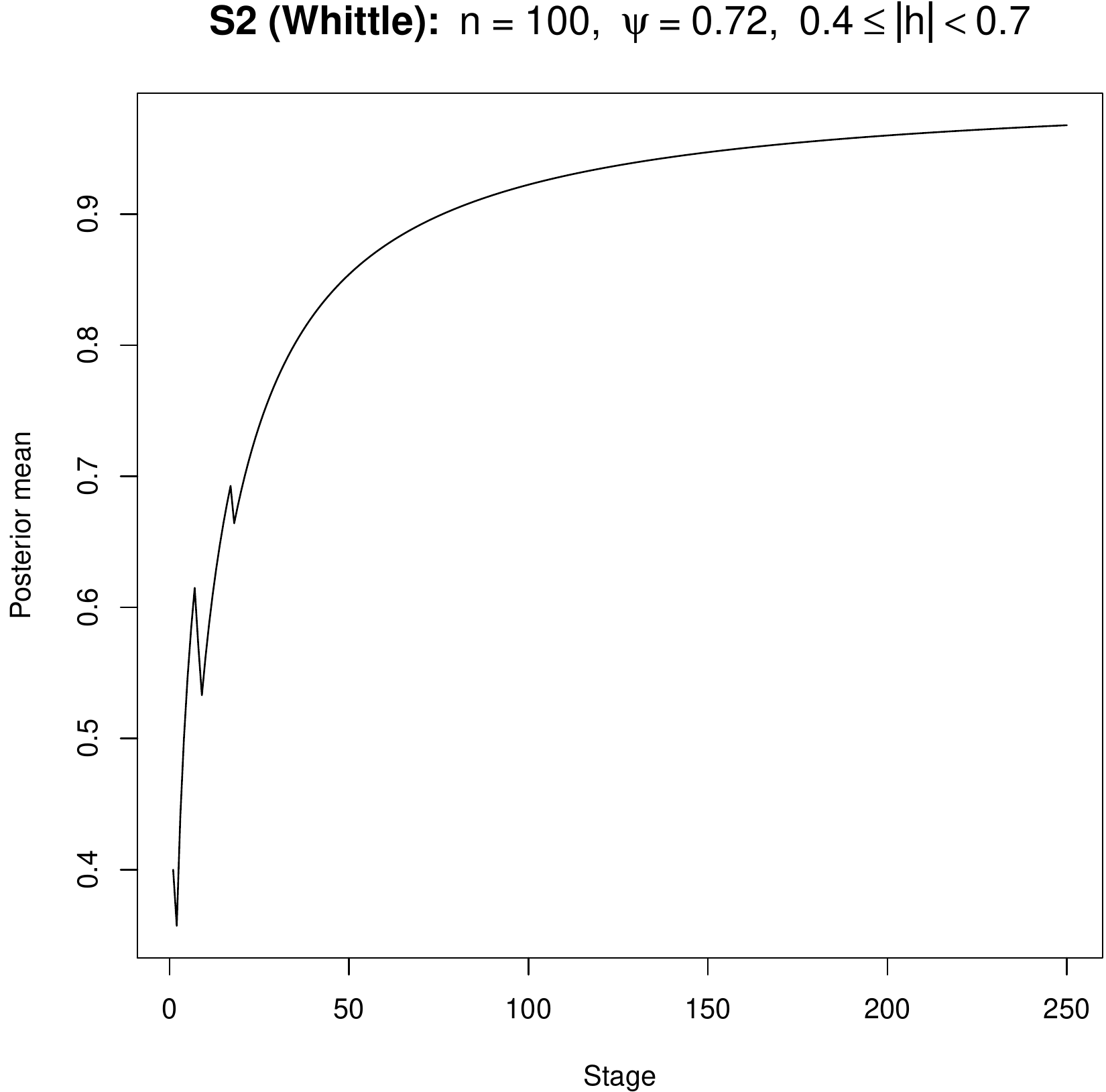}}\\
\vspace{2mm}
\subfigure [$0.7\leq\|h\|<0.9$]{ \label{fig:spacetime_covns3_soutir2_w}
\includegraphics[width=5.5cm,height=5.5cm]{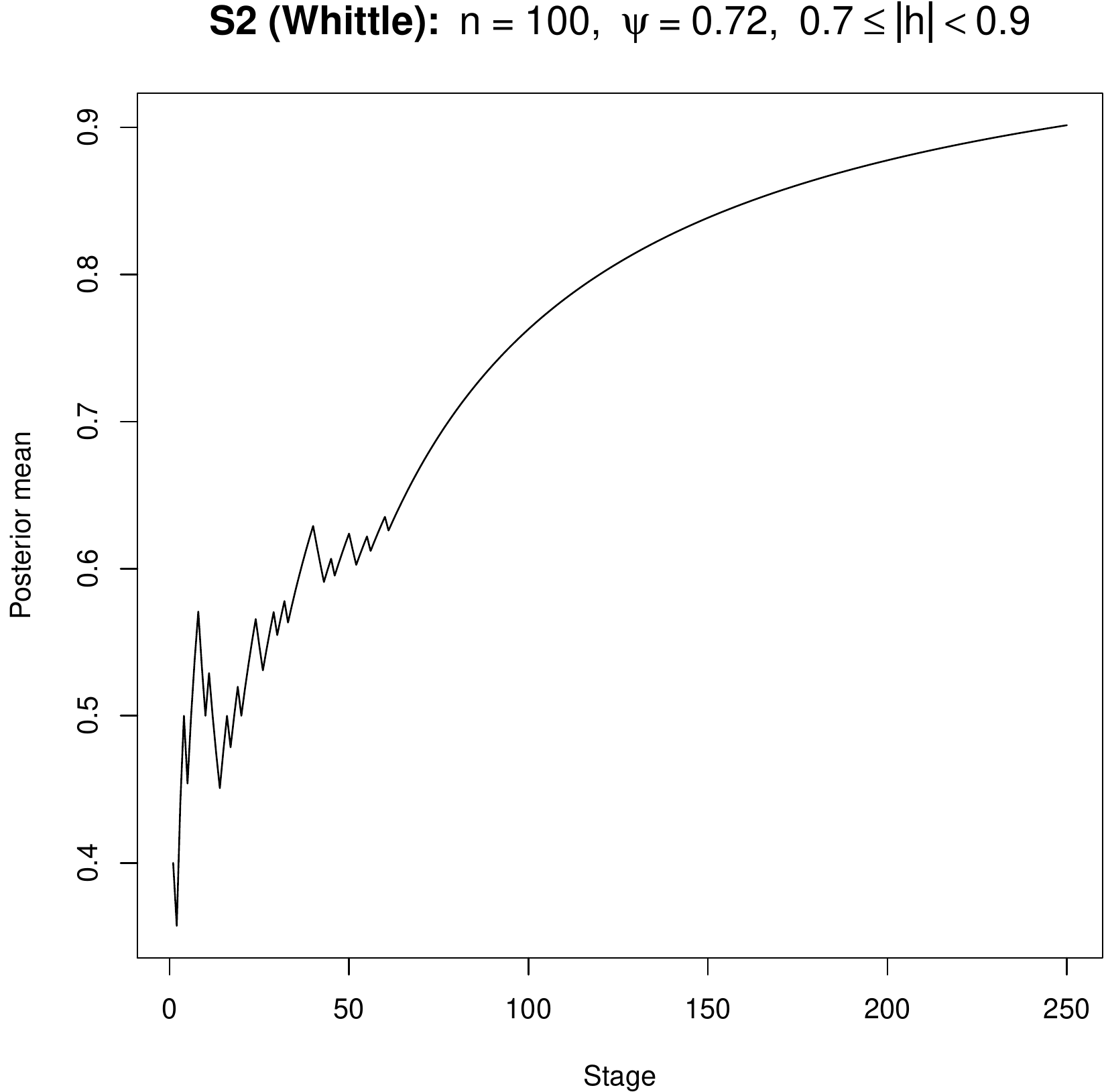}}
\hspace{2mm}
\subfigure [$0.9\leq\|h\|<2$.]{ \label{fig:spacetime_covns4_soutir2_w}
\includegraphics[width=5.5cm,height=5.5cm]{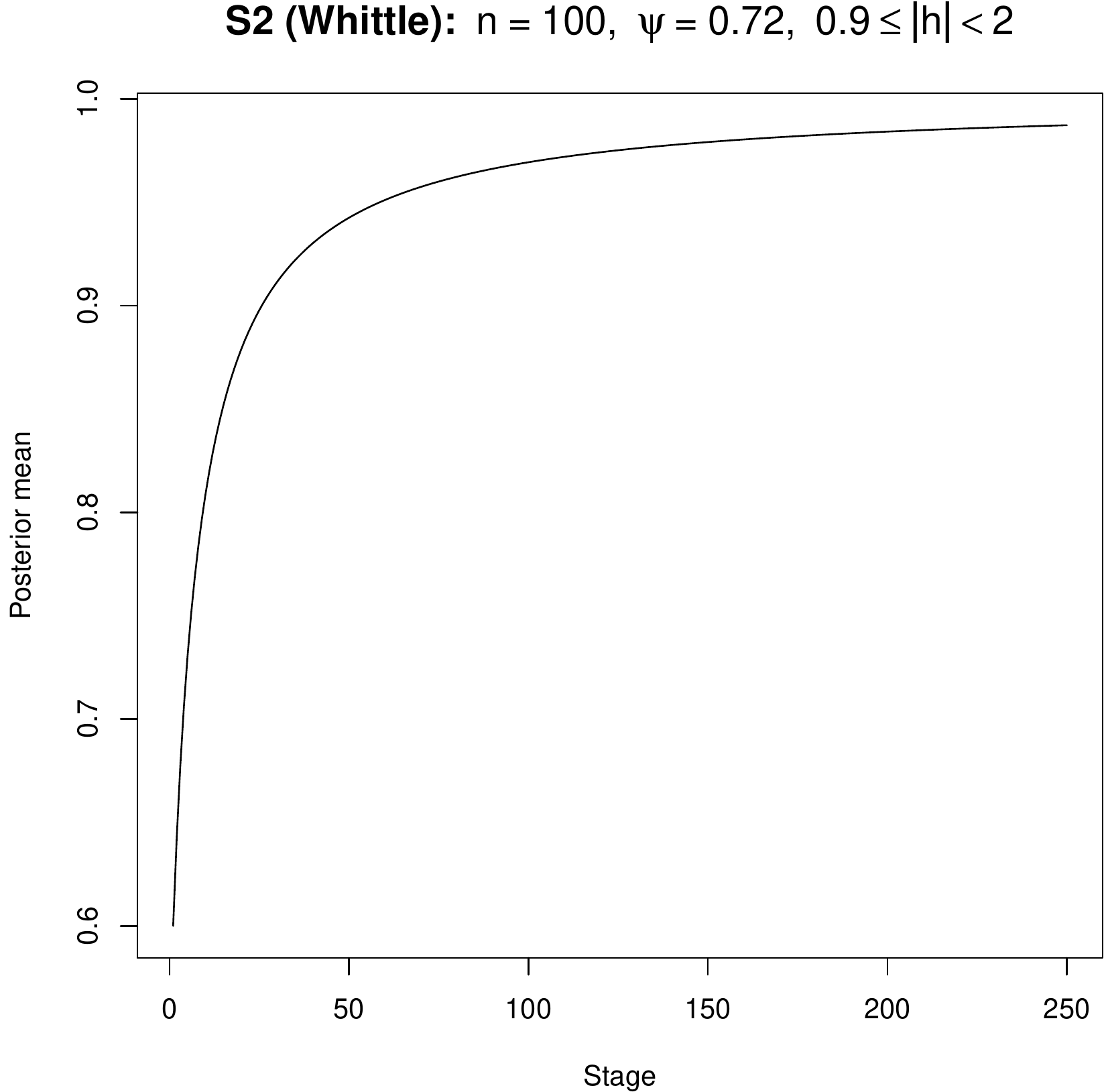}}\\
\vspace{2mm}
\subfigure [$2\leq\|h\|<3$.]{ \label{fig:spacetime_covns5_soutir2_w}
\includegraphics[width=5.5cm,height=5.5cm]{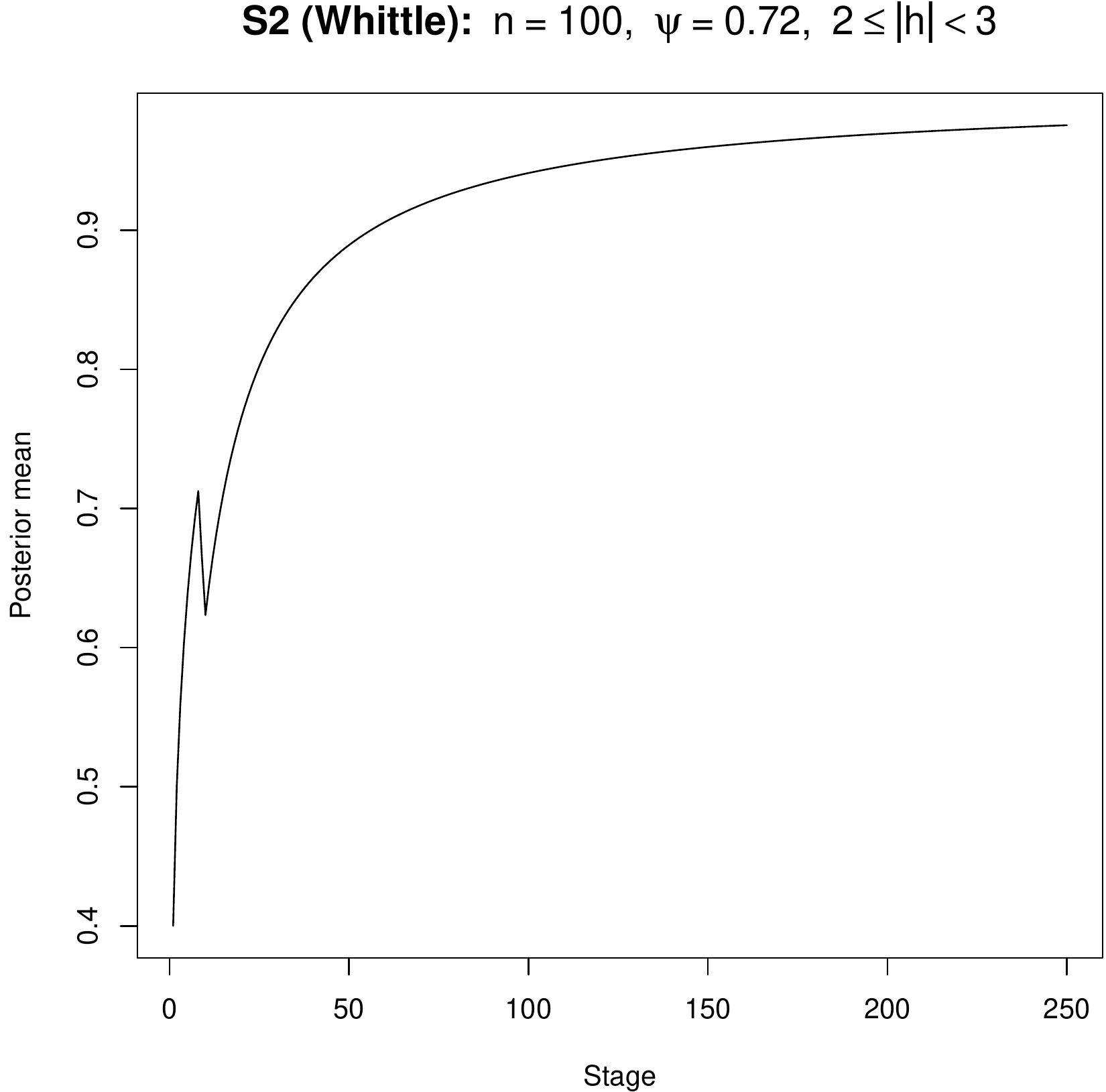}}\\
\caption{Detection of covariance stationarity in spatio-temporal data drawn from model $S2$ with sample size $100$ locations and $200$ time points, 
corresponding to Whittle spatial covariance with $\psi=0.72$ and $\lambda=5$.}
\label{fig:spacetime_cov_soutir2_w}
\end{figure}

\subsubsection{Simulations under nonstationarity}
\label{subsubsec:spacetime_nonstationarity}
We now apply our Bayesian methodology to the three nonstationary models and setups considered by \ctn{Soutir17b}. 
\begin{itemize}
\item[(NS1)] $X_{(s,t)}=0.5 X_{(s,t-1)}+\left(1.3+\sin\left(\frac{2\pi t}{400}\right)\right)\epsilon_{(s,t)}$, 
where $X_{s,0}=\bzero$ and $\epsilon_{(s,t)}$ are zero mean GPs independent over time
with spatial covariance structure (\ref{eq:expcov_spacetime}).
Note that this is a temporally nonstationary but spatially stationary Gaussian random field.
We consider $\psi=0.5$ and $1$, and $\lambda=5$ for the simulations.

\item[(NS2)] $X_{(s,t)}=0.5 X_{(s,t-1)}+0.4 X_{(s,t-1)}\epsilon_{(s,t-1)}+\eta_{(s,t)}$, where $X_{s,0}=\bzero$ and $\eta_{(s,t)}$ are zero 
mean GPs independent over time with nonstationary spatial covariance given as follows. 
\begin{equation}
Cov\left(\eta_{(s_1,t)},\eta_{(s_2,t)}\right)=\left|\Sigma\left(\frac{s_1}{\lambda}\right)\right|^{\frac{1}{4}}
\left|\Sigma\left(\frac{s_2}{\lambda}\right)\right|^{\frac{1}{4}}
\left|\frac{\Sigma\left(\frac{s_1}{\lambda}\right)+\Sigma\left(\frac{s_2}{\lambda}\right)}{2}\right|^{-\frac{1}{2}}\exp\left[-\sqrt{Q_{\lambda}(s_1,s_2)}\right],
\label{eq:spacetime_nonstationary}
\end{equation}
where $Q_{\lambda}(s_1,s_2)=2(s_1-s_2)^T\left[\Sigma\left(\frac{s_1}{\lambda}\right)+\Sigma\left(\frac{s_2}{\lambda}\right)\right]^{-1}(s_1-s_2)$ and
$\Sigma\left(\frac{s}{\lambda}\right)=\Gamma\left(\frac{s}{\lambda}\right)\Lambda\Gamma\left(\frac{s}{\lambda}\right)^T$.
In the above,
\[\Gamma\left(\frac{s}{\lambda}\right)=\left(\begin{array}{cc}\gamma_1\left(\frac{s}{\lambda}\right) & -\gamma_2\left(\frac{s}{\lambda}\right)\\
\gamma_2\left(\frac{s}{\lambda}\right) & \gamma_1\left(\frac{s}{\lambda}\right)\end{array}\right);~~
\Lambda=\left(\begin{array}{cc}1 & 0\\ 0 & \frac{1}{2}\end{array}\right), 
\]
where $\gamma_1\left(\frac{s}{\lambda}\right)=\log\left(u/\lambda+0.75\right)$, $\gamma_2\left(\frac{s}{\lambda}\right)=(u/\lambda)^2+(v/\lambda)^2$,
and $s=(u,v)^T$.

With this, the model is a temporally stationary and spatially nonstationary Gaussian random field. For simulations, we consider $\lambda=20$,
following \ctn{Soutir17b}.

\item[(NS3)] $X_{(s,t)}=0.5 X_{(s,t-1)}+\left(1.3+\sin\left(\frac{2\pi t}{400}\right)\right)\eta_{(s,t)}$, where $X_{s,0}=\bzero$ and $\eta_{(s,t)}$ are zero
mean GPs independent over time with nonstationary spatial covariance given by (\ref{eq:spacetime_nonstationary}).
This defines a temporally and spatially nonstationary Gaussian random field. Again, we set $\lambda=20$ for simulations, following \ctn{Soutir17b}.
\end{itemize}

We obtained the right results in all the cases of nonstationarity, but present the results corresponding to $(m=100,T=200)$ and $\psi=1$ for brevity. 
Figure \ref{fig:spacetime_soutir_ns} provides the results on strong stationarity and the result on covariance stationarity of $NS1$
is depicted in Figure \ref{fig:spacetime_cov_soutir_ns1}.
For detection of strict nonstationarity, $\hat C_1$ varied between $0.04$ and $0.05$. The same values also yielded respective covariance nonstationarities
in these exampples. However, the maximum values of $\hat C_1$ for detecting covariance nonstationarities varied between $0.05$, $0.15$, $0.2$ and $0.3$.

\begin{figure}
\centering
\subfigure [Correct detection of nonstationarity.]{ \label{fig:nonstationary1_soutir_ns}
\includegraphics[width=5.5cm,height=5.5cm]{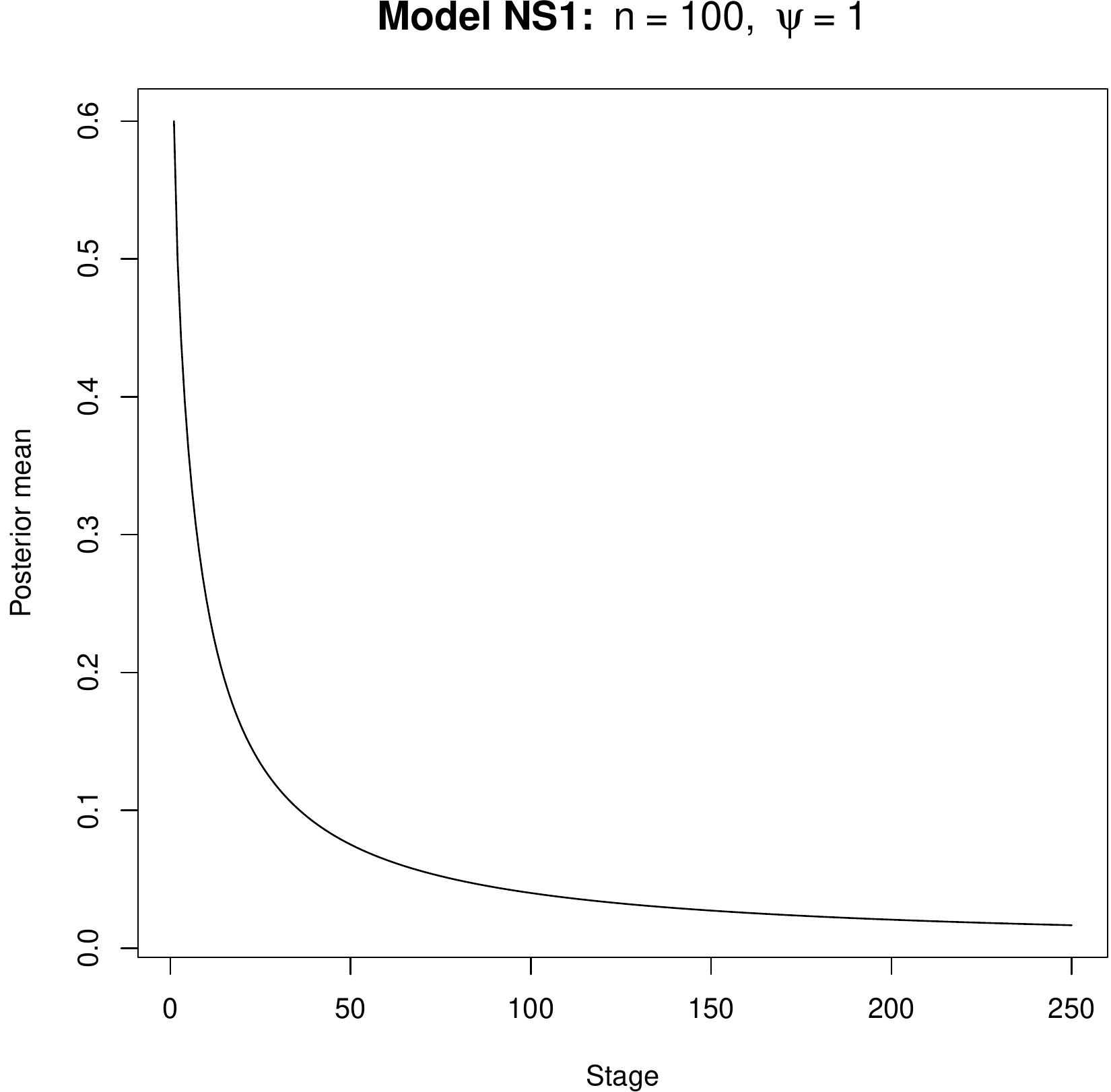}}
\hspace{2mm}
\subfigure [Correct detection of nonstationarity.]{ \label{fig:nonstationary2_soutir_ns}
\includegraphics[width=5.5cm,height=5.5cm]{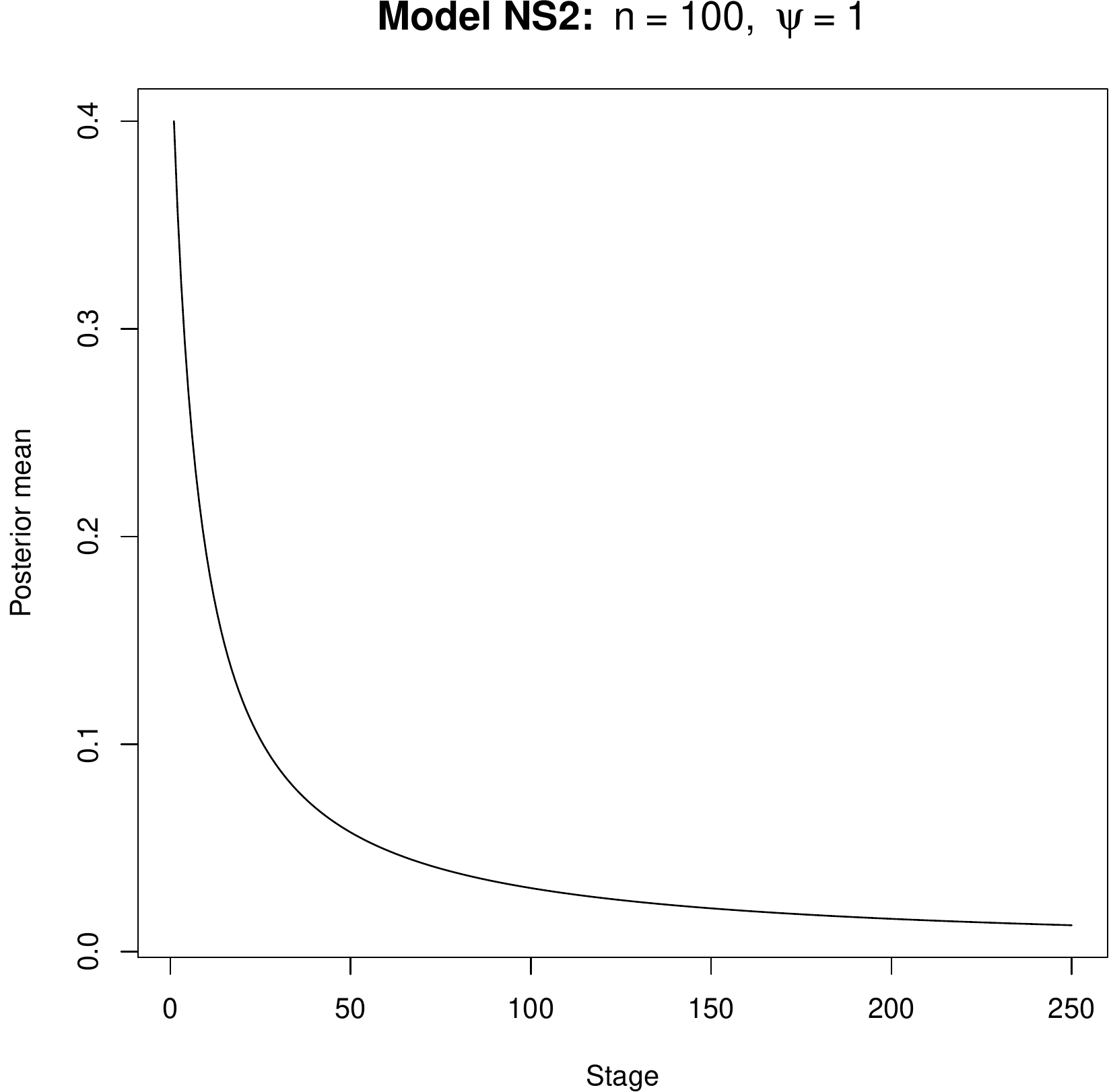}}\\
\vspace{2mm}
\subfigure [Correct detection of nonstationarity.]{ \label{fig:nonstationary3_soutir_ns}
\includegraphics[width=5.5cm,height=5.5cm]{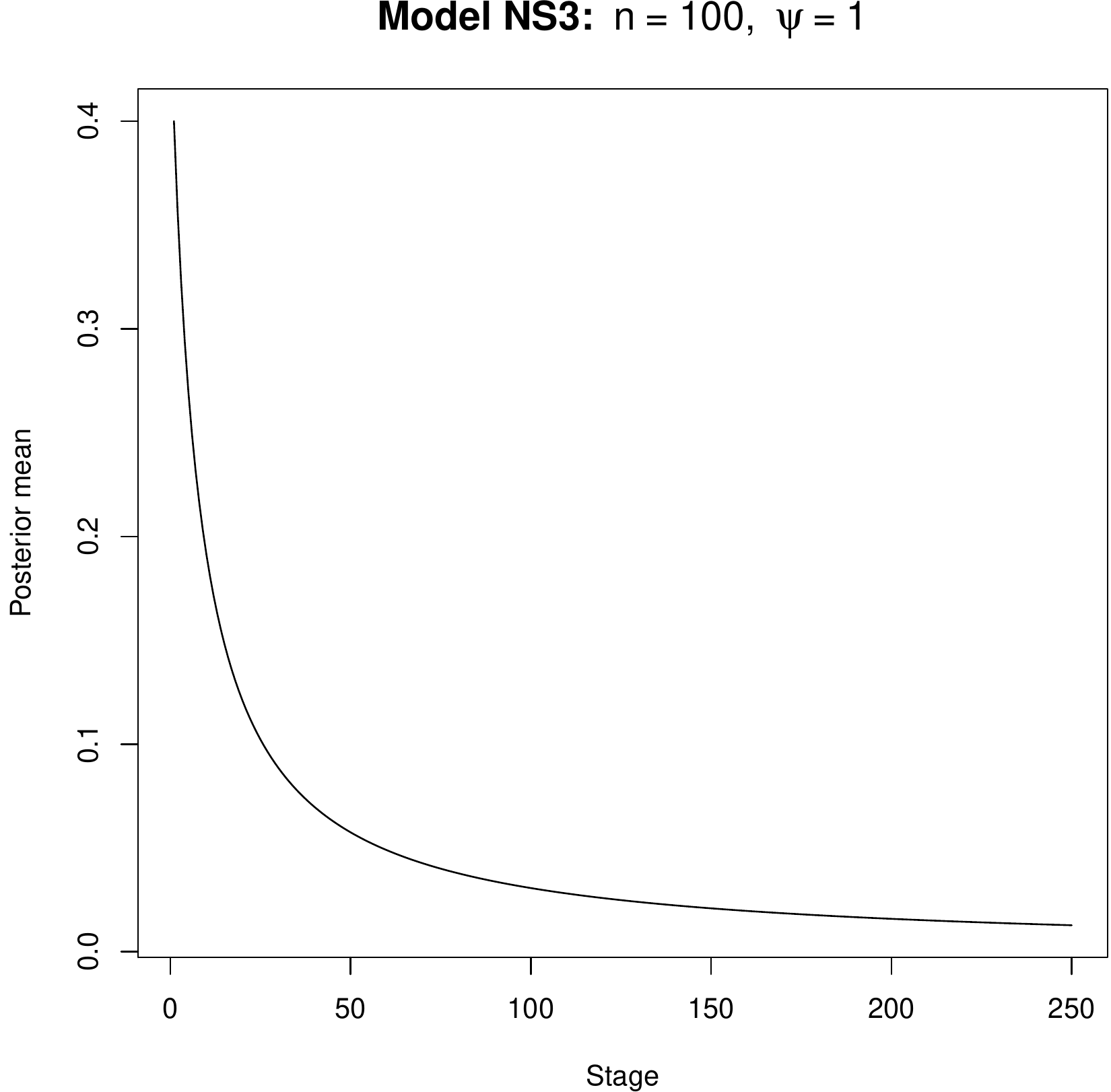}}\\
\caption{Detection of strong nonstationarity in spatio-temporal data drawn from models $NS1$, $NS2$ and $NS3$ 
with sample size $100$ locations and $200$ time points.}
\label{fig:spacetime_soutir_ns}
\end{figure}

\begin{figure}
\centering
\subfigure [$0\leq\|h\|<0.4$]{ \label{fig:spacetime_covns1_soutir_ns1}
\includegraphics[width=5.5cm,height=5.5cm]{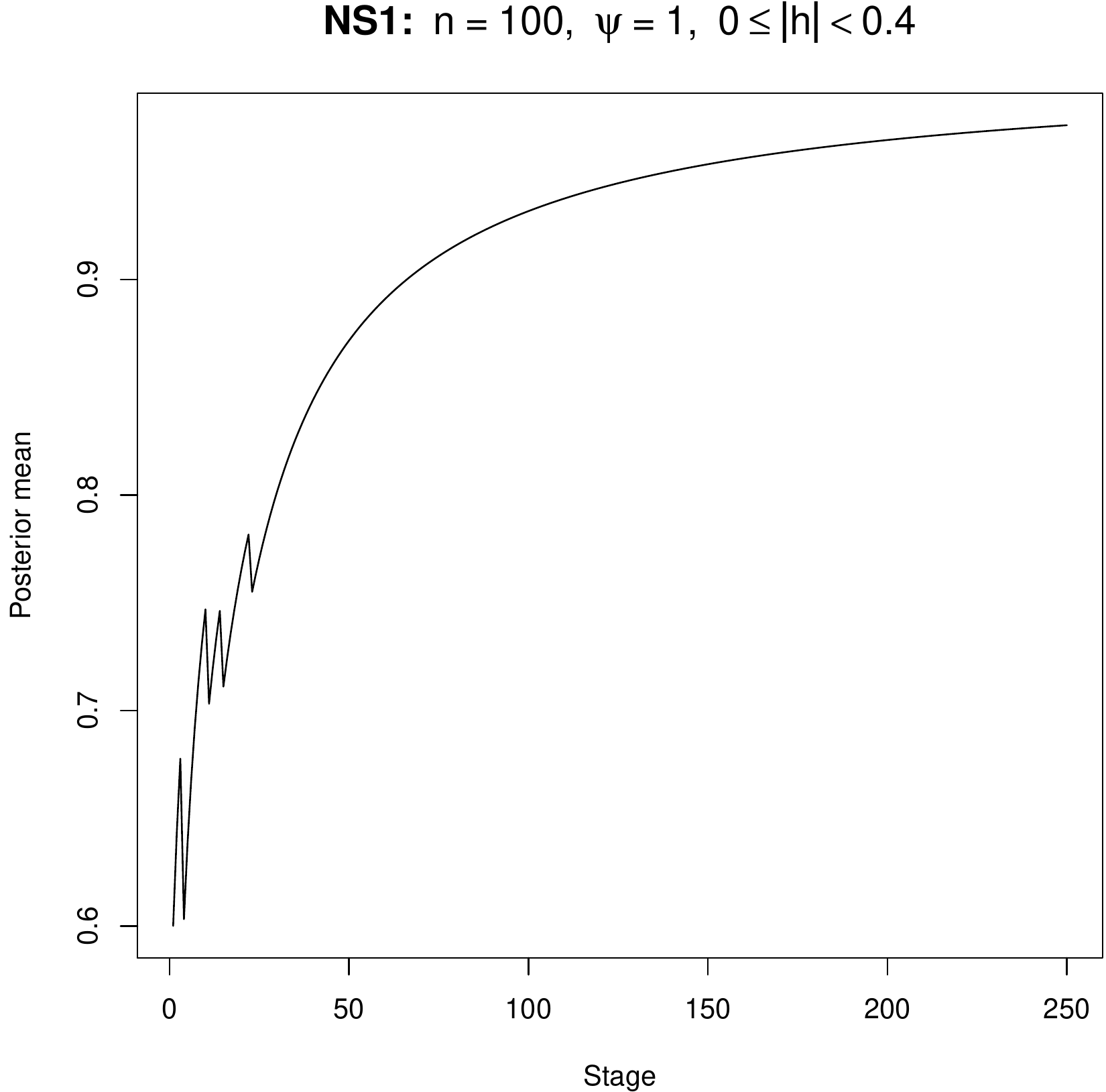}}
\hspace{2mm}
\subfigure [$0.4\leq\|h\|<0.7$.]{ \label{fig:spacetime_covns2_soutir_ns1}
\includegraphics[width=5.5cm,height=5.5cm]{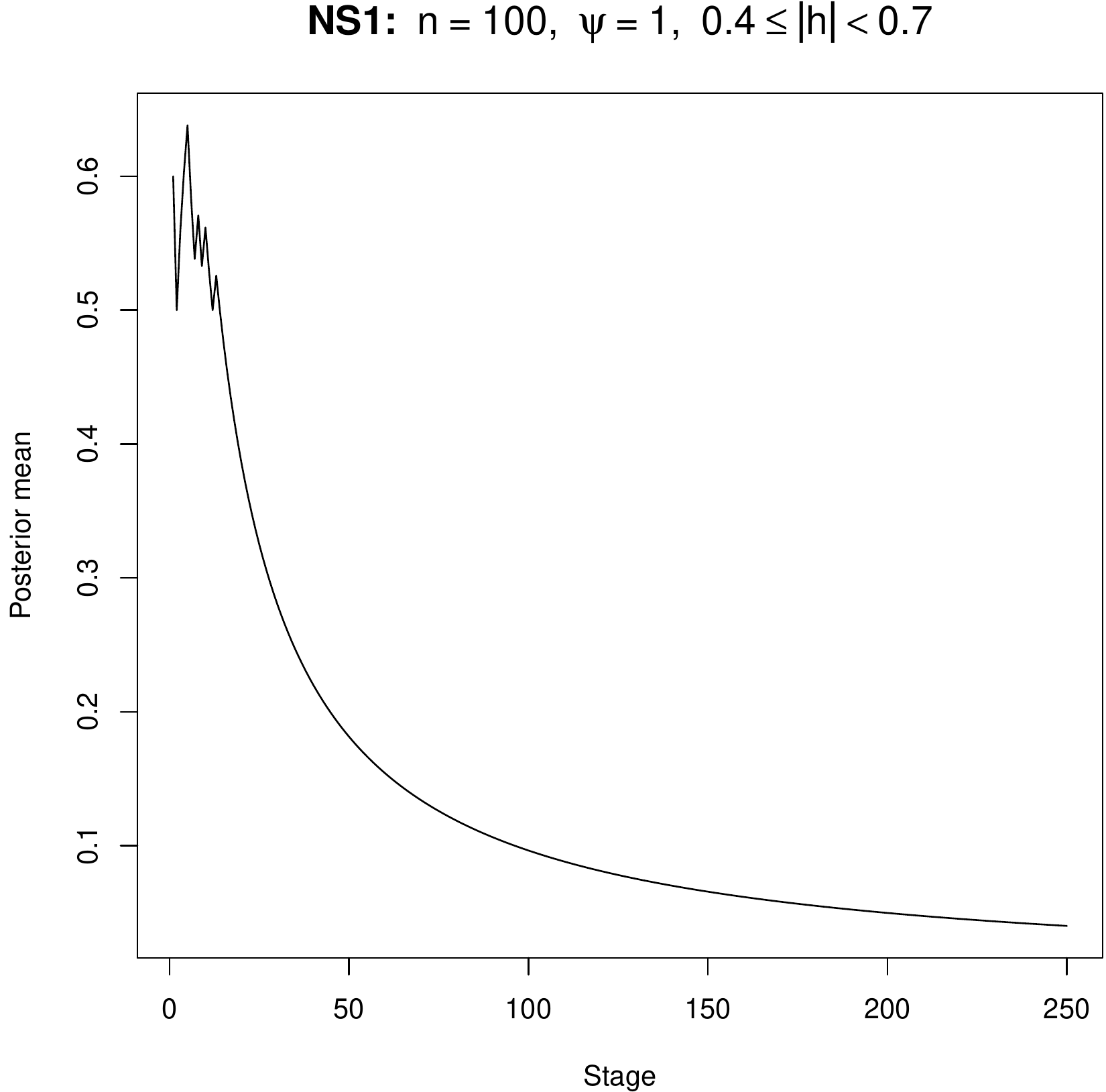}}\\
\vspace{2mm}
\subfigure [$0.7\leq\|h\|<0.9$]{ \label{fig:spacetime_covns3_soutir_ns1}
\includegraphics[width=5.5cm,height=5.5cm]{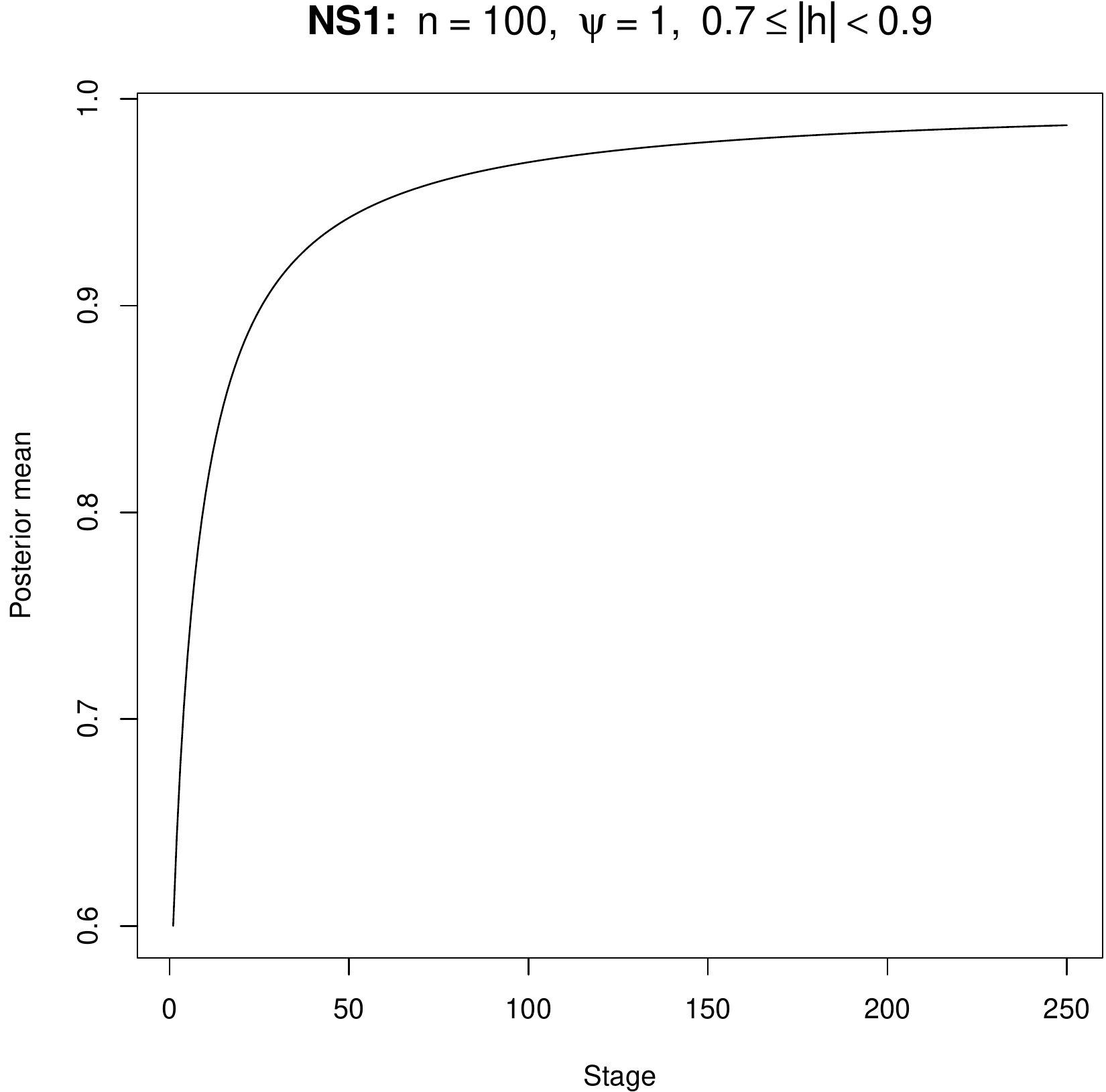}}
\hspace{2mm}
\subfigure [$0.9\leq\|h\|<2$.]{ \label{fig:spacetime_covns4_soutir_ns1}
\includegraphics[width=5.5cm,height=5.5cm]{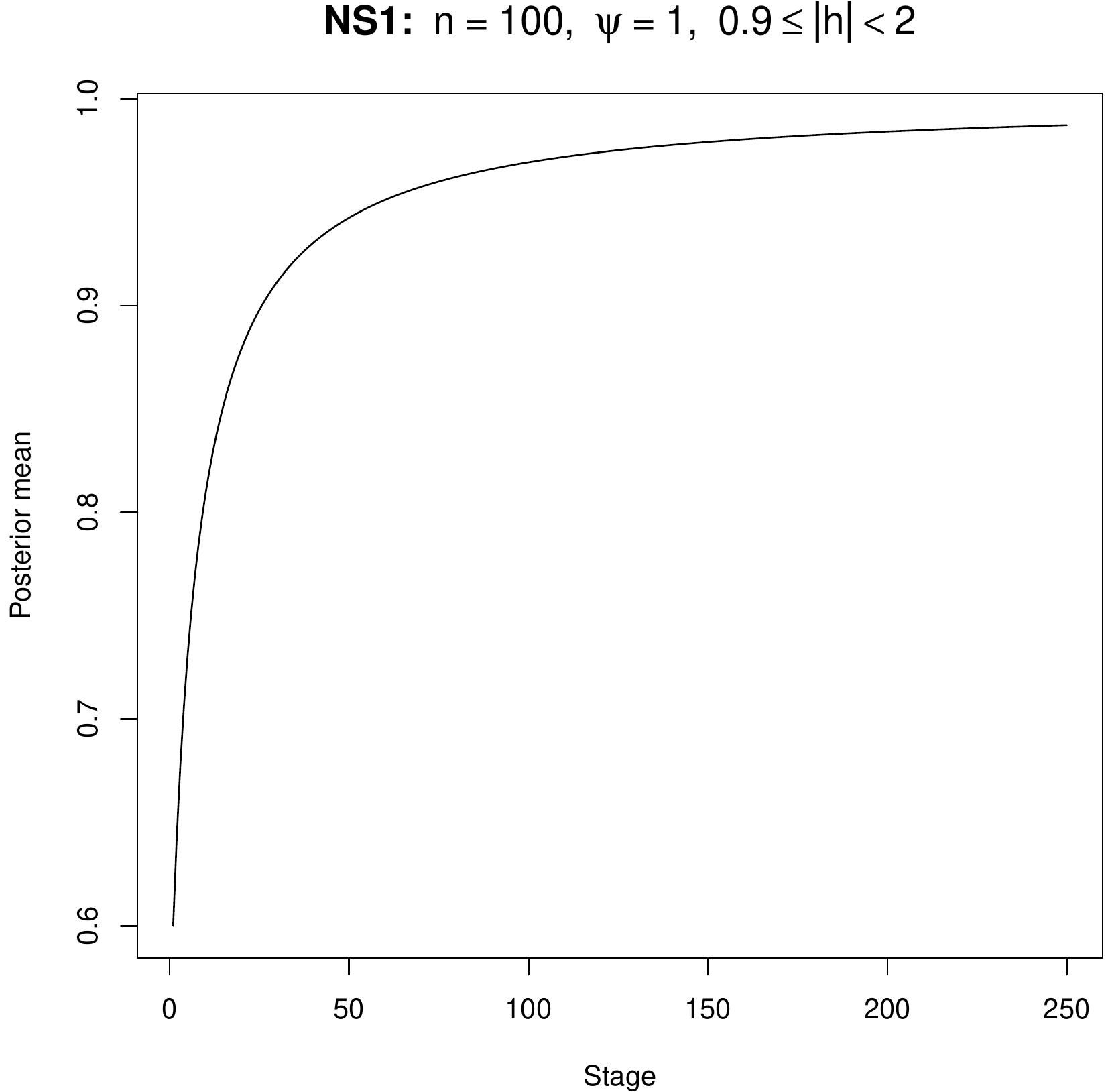}}\\
\caption{Detection of covariance nonstationarity in spatio-temporal data drawn from model $NS1$ with sample size $100$ locations and $200$ time points.} 
\label{fig:spacetime_cov_soutir_ns1}
\end{figure}

\subsubsection{Overall comparison of our results with those of \ctn{Soutir17b}}
\label{subsubsec:soutir_spacetime_comparison}

First, our Bayesian procedure is designed to identify both weak and strict stationarity of the underlying spatio-temporal process, while
the methods of \ctn{Soutir17b} are meant for detection of weak stationarity only, and not for strict stationarity. 

Second, our method requires the only assumption of local stationarity, which is expected to hold in general. In contrast, the methods of \ctn{Soutir17b}
require a variety of assumptions, which may be difficult to verify in practice.

Overall, our Bayesian procedure worked adequately for all the strict stationarity and nonstationarity cases that we considered. The method also performed 
satisfactorily whenever there existed well-defined regions $\mathcal N_{i,h_j,h_{j+1}}$ in the data set. On the other hand, the methods of \ctn{Soutir17b}
did not yield satisfactory results particularly when the underlying process is non-Gaussian.

\section{Real data analyses for spatial and spatio-temporal data}
\label{sec:realdata_spacetime}
\ctn{Das20} considered three real spatial and spatio-temporal data sets on pollutants for illustration of their new general nonparametric 
spatial and spatio-temporal model and methods.
One is an ozone data set, which is a spatial data. Initially, \ctn{Das20} fitted a stationary model, a special case of their general model, to the ozone data, but
obtained unsatisfactory fit. This prompted them to fit the nonstationary instance of their model, which yielded adequate results. 
Thus, nonstationarity of the ozone data seems to be more plausible than stationarity. Here we establish with our Bayesian method that this is indeed the case.

The other two data sets are spatio-temporal data sets on particulate matters (PM), which are mixtures of solid particles and 
liquid droplets found in the air. The data sets correspond to measurements of air concentrations of two different size ranges -- PM 10 and PM 2.5. 
The first one, PM 10, is suspected to be nonstationary, while PM 2.5 is suspected to be stationary in the literature (see, for example, \ctn{Paciorek09}). 
With our Bayesian method for characterizing stationarity and nonstationarity, we establish that such intuitions are correct.

For details regarding the three data sets, see \ctn{Das20}. There are also covariaites associated with the three data sets, which have been utilized
by \ctn{Das20} for their modeling purpose. However, for checking stationarity and nonstationarity, only the responses are necessary. Hence, for our
current purpose, the covariates are unnecessary. We evaluate all the final responses in their log scales.

\subsection{Spatial ozone data}
\label{subsec:ozone}

After appropriate data transformations (see \ctn{Das20}), we obtain $76$ observations, evaluated in the log scale. 
To obtain $\hat C_1$, we first generate $76$ observations from a GP with the Whittle covariance function given by
(\ref{eq:spatial_bound}), with $\psi=0.8$, and with the same set of locations as the ozone data. We set $K=20$ for this small data set, 
and obtain the minimum value of $\hat C_1$ that ensured
stationarity for this GP data with our Bayesian method, to be $0.38$. With this value of $\hat C_1$ and larger (even with $\hat C_1=0.43$), 
we obtained clear evidence of nonstationarity for the ozone data, as depicted
in Figure \ref{fig:ozone_nonstationary}.
\begin{figure}
\centering
\subfigure [Nonstationarity (ozone data).]{ \label{fig:ozone_nonstat}
\includegraphics[width=6.5cm,height=6.5cm]{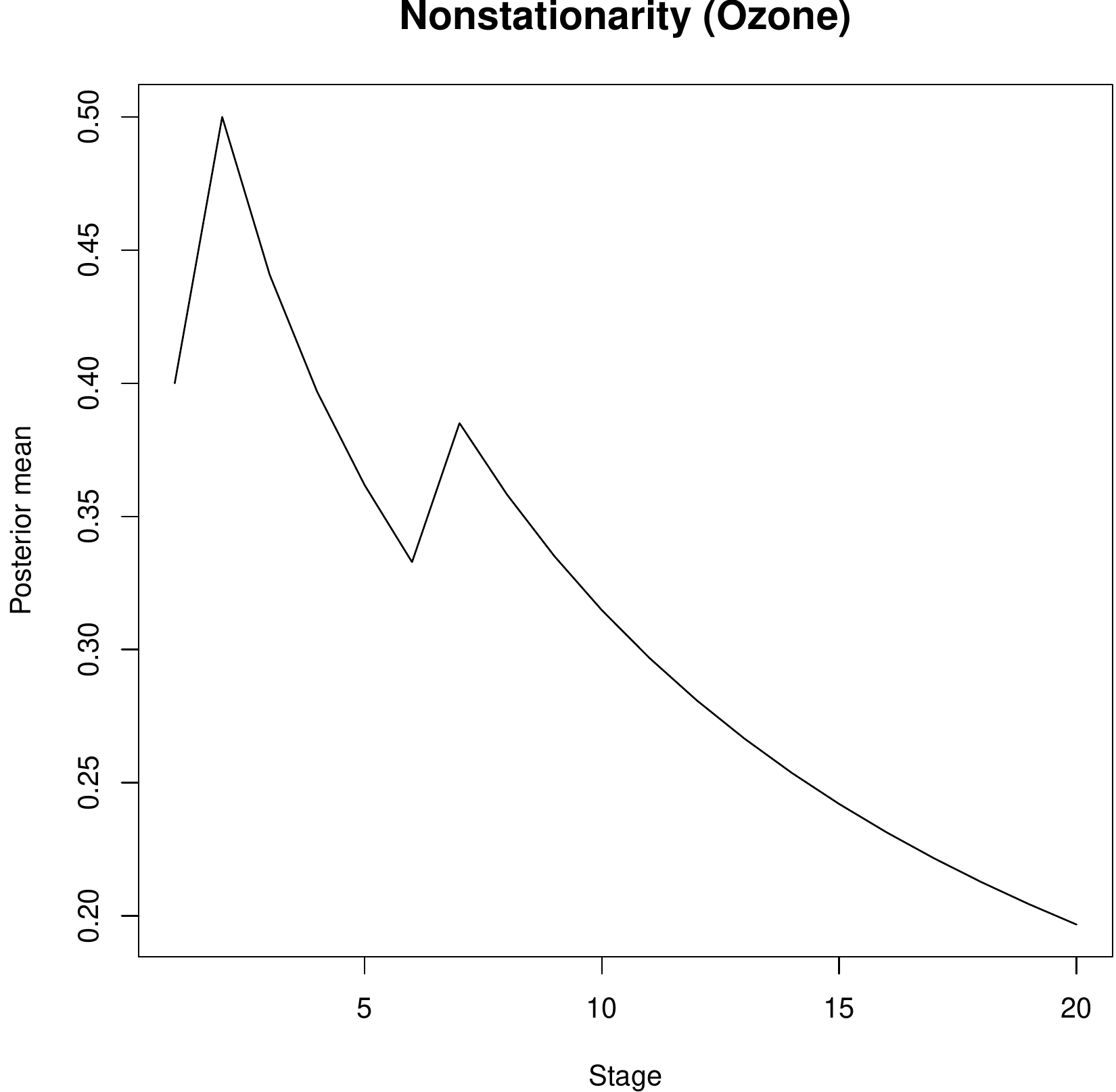}}
\caption{Detection of nonstationarity of the ozone data with our Bayesian method.} 
\label{fig:ozone_nonstationary}
\end{figure}

To check covariance stationarity, we obtain four neighborhoods $\mathcal N_{i,h_j,h_{j+1}}$, for $j=1,2,3,4$, where $h_1=0.0$, $h_2=0.2$, $h_3=0.4$, $h_4=0.6$
and $h_5=0.8$. With $K=20$ and the same Whittle covariance based GP data for strict stationarity, the same value $\hat C_1=0.38$ turned out to be the
minimum value ensuring covariance statonarity for the GP data. Figure \ref{fig:ozone_cov_nonstationary} shows covariance nonstationarity for the ozone data
with $\hat C_1=0.38$. Indeed, convergence to zero is indicated with $\mathcal N_{i,h_2,h_3}$.
\begin{figure}
\centering
\subfigure [$0\leq\|h\|<0.2$.]{ \label{fig:ozone_covns1}
\includegraphics[width=5.5cm,height=5.5cm]{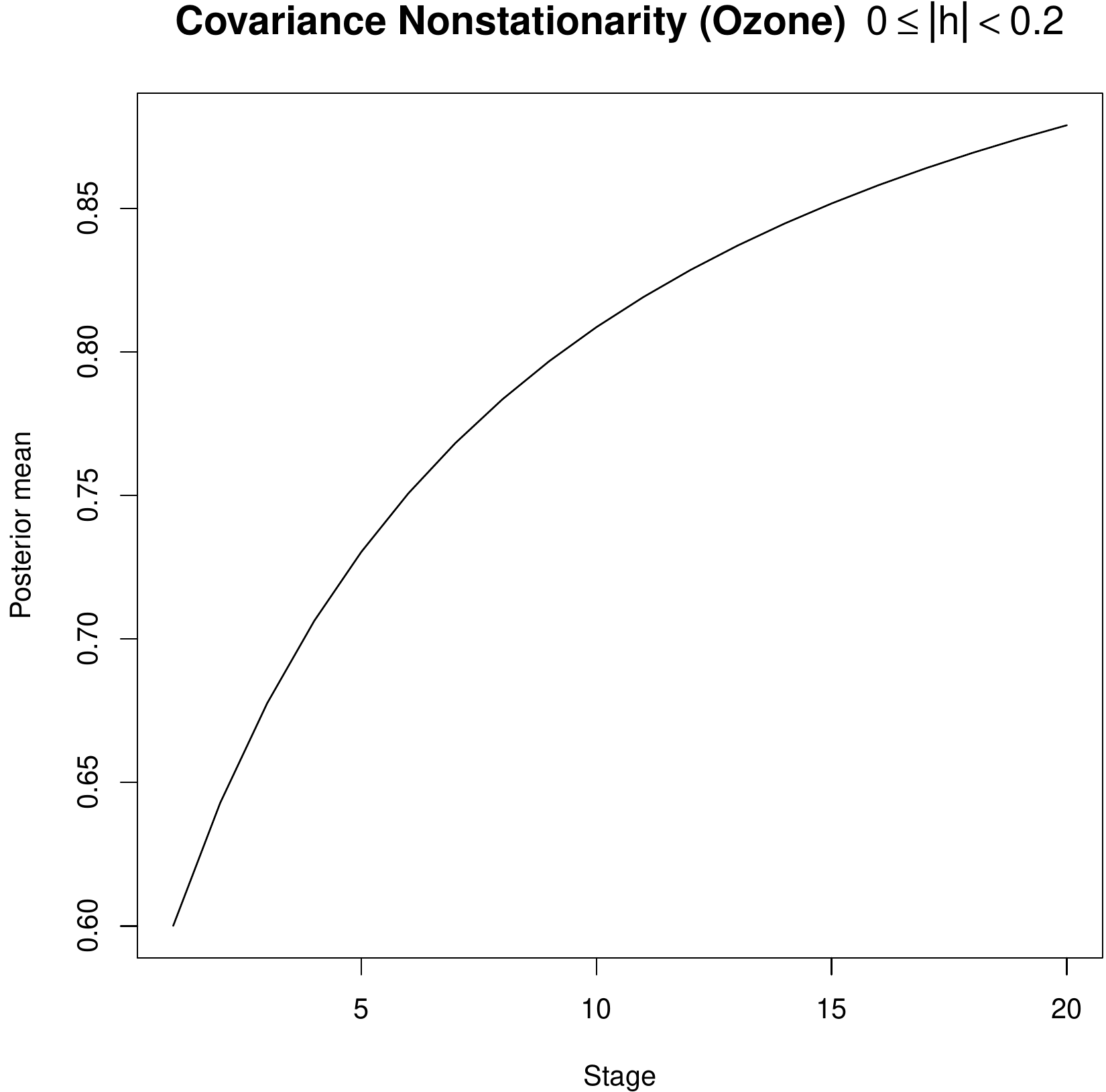}}
\hspace{2mm}
\subfigure [$0.2\leq\|h\|<0.4$.]{ \label{fig:ozone_covns2}
\includegraphics[width=5.5cm,height=5.5cm]{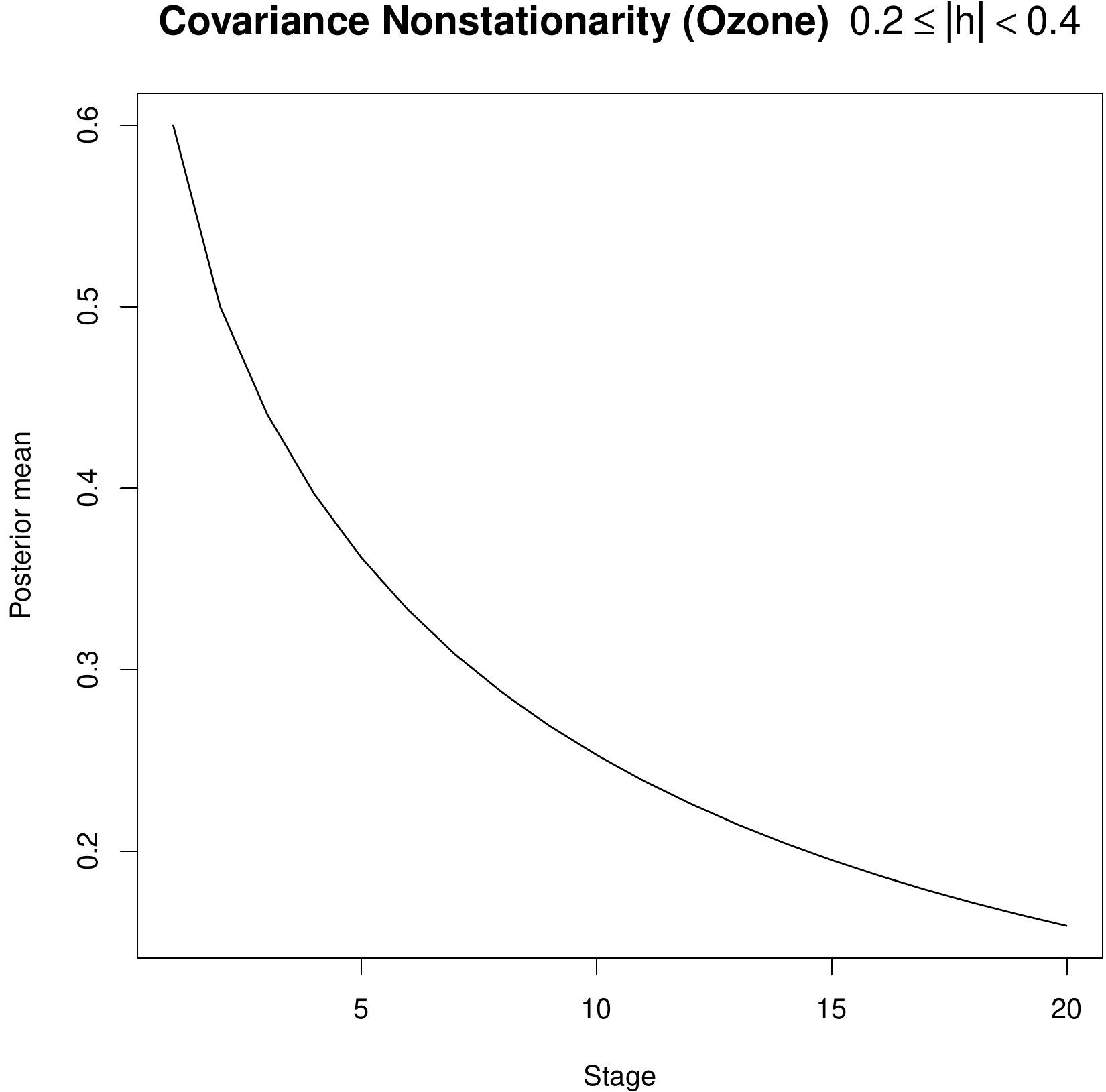}}\\
\vspace{2mm}
\subfigure [$0.4\leq\|h\|<0.6$.]{ \label{fig:ozone_covns3}
\includegraphics[width=5.5cm,height=5.5cm]{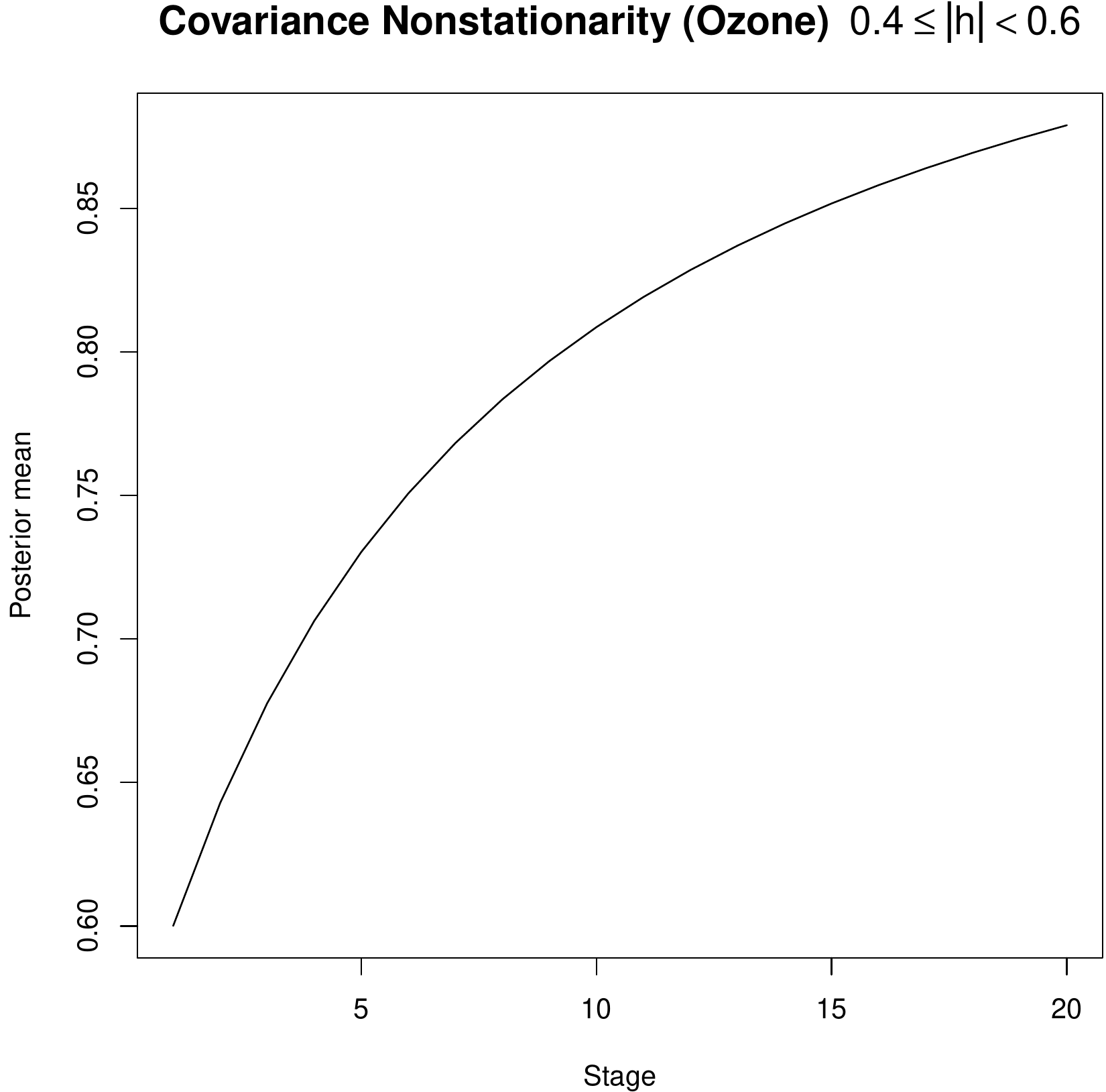}}
\hspace{2mm}
\subfigure [$0.6\leq\|h\|<0.8$.]{ \label{fig:ozone_covns4}
\includegraphics[width=5.5cm,height=5.5cm]{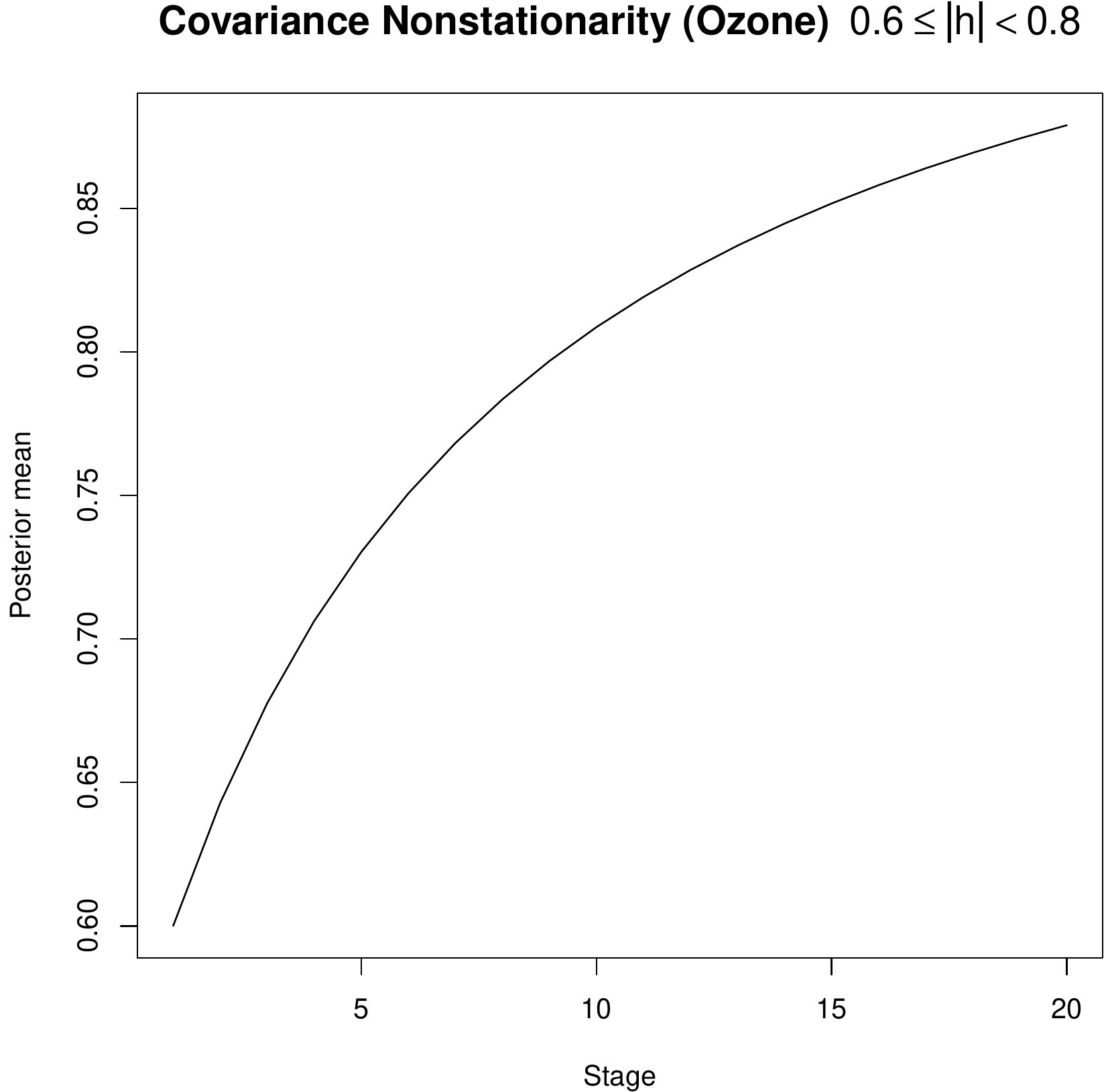}}
\caption{Detection of covariance nonstationarity of the ozone data.}
\label{fig:ozone_cov_nonstationary}
\end{figure}

\subsection{Spatio-temporal PM 10 data}
\label{subsec:PM10}

This data set consists of $70572$ observations, a part of which has been used by \ctn{Das20} for model fitting. However, here we use all $70572$ log-response values
to check strict and covariance stationarity. To obtain $\hat C_1$, we need to generate GP samples of size $70572$ with the Whittle covariance function
and the locations, time points corresponding to the real PM 10 data set. However,
generation of such a large GP sample turned out to be prohibitive with our current infrastructure. But more of concern is the issue that the 
stability of the covariance matrix turned out to steadily deteriorate for dimensions larger than $100000$. Figure \ref{fig:PM10_gp_samples} shows two GP samples
of sizes $10000$ and $20000$ generated using the $R$-package ``mvnfast", using $80$ parallel cores. Although the sample of size $10000$ is stable, the other
shows increasing variability from index $10000$ onwards.
\begin{figure}
\centering
\subfigure [GP sample size $10000$.]{ \label{fig:gpsamp1}
\includegraphics[width=7.5cm,height=5.5cm]{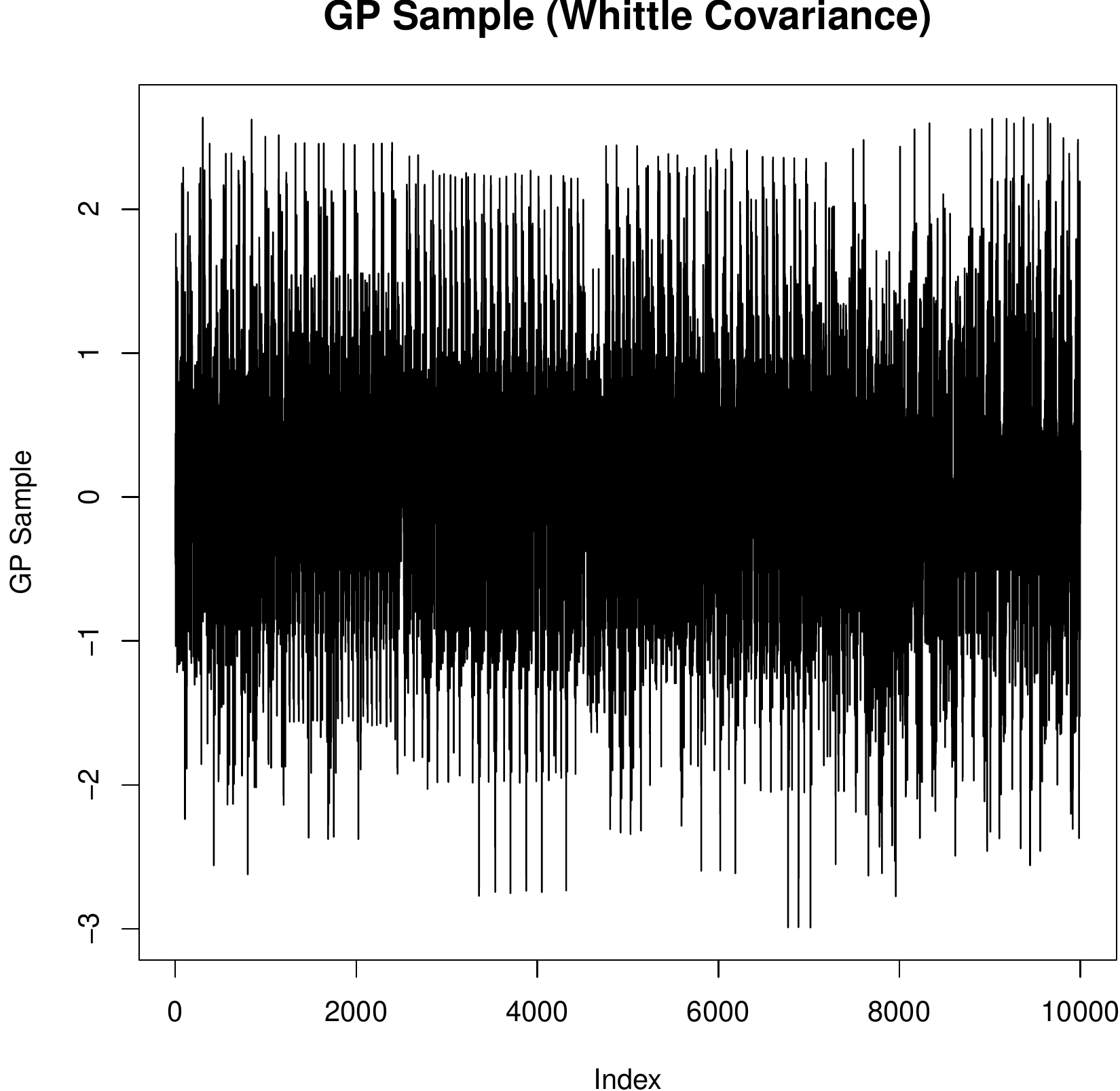}}\\
\vspace{2mm}
\subfigure [GP sample size $20000$.]{ \label{fig:gpsamp2}
\includegraphics[width=7.5cm,height=5.5cm]{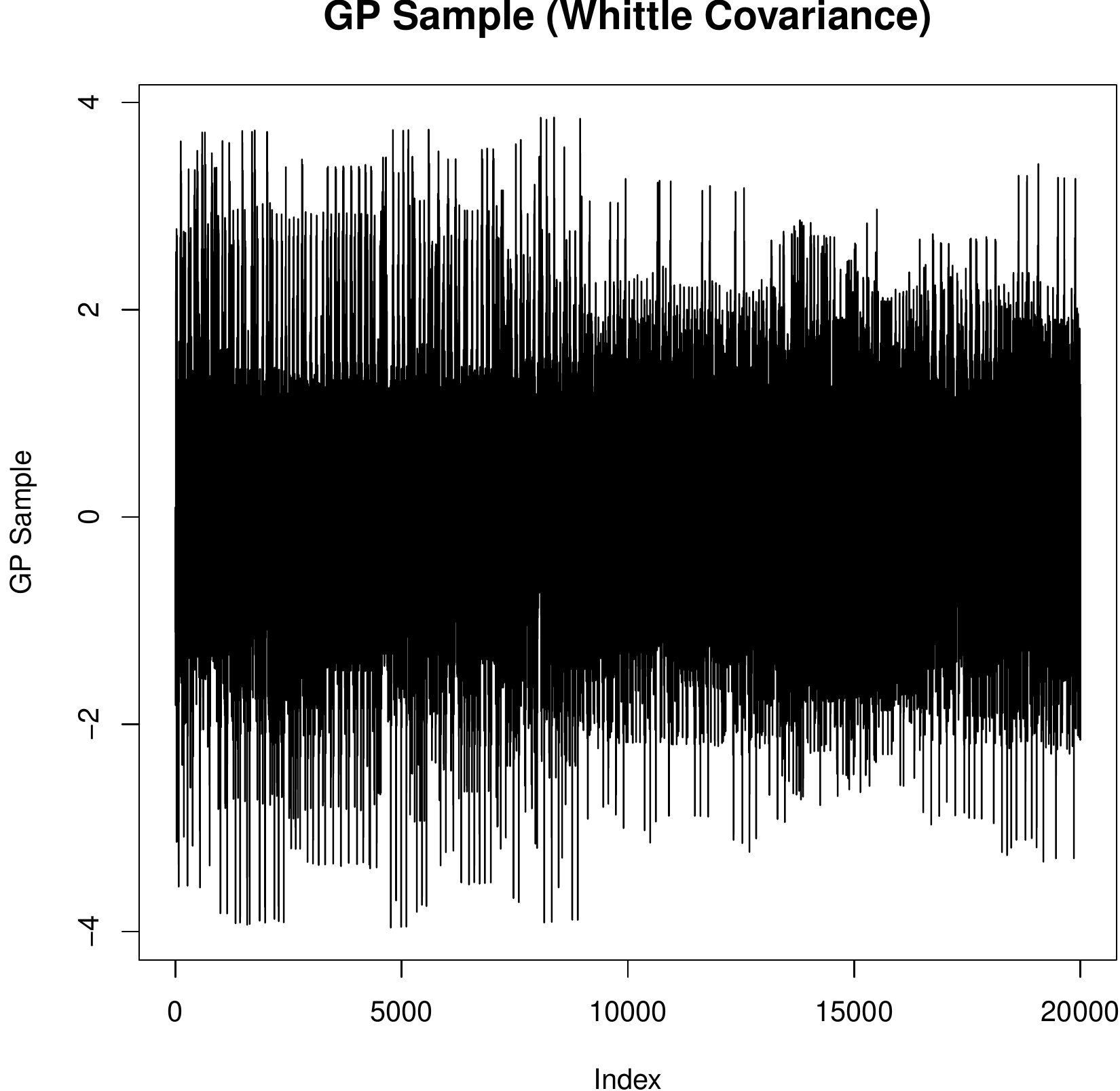}}
\caption{GP samples of sizes $10000$ and $20000$ for Whittle covariance with $\psi=0.8$ for PM 10 data.}
\label{fig:PM10_gp_samples}
\end{figure}
Hence, to obtain $\hat C_1$ we consider the GP sample of size $10000$. Setting $K=250$ as in the simulation studies, we obtain $\hat C_1=0.16$ for checking
strict stationarity. For the real PM 10 data of size $70572$, we then set $\hat C_1=0.16$ and $K=1764$. The latter is chosen such that the number of observations
per cluster is on the average $40$, to match the average number of observations per cluster in the simulated GP data. Figure \ref{fig:PM10_nonstationary}
clearly indicates strict nonstationarity of the PM 10 data.
\begin{figure}
\centering
\subfigure [Nonstationarity (PM 10 data).]{ \label{fig:PM10_nonstat}
\includegraphics[width=6.5cm,height=6.5cm]{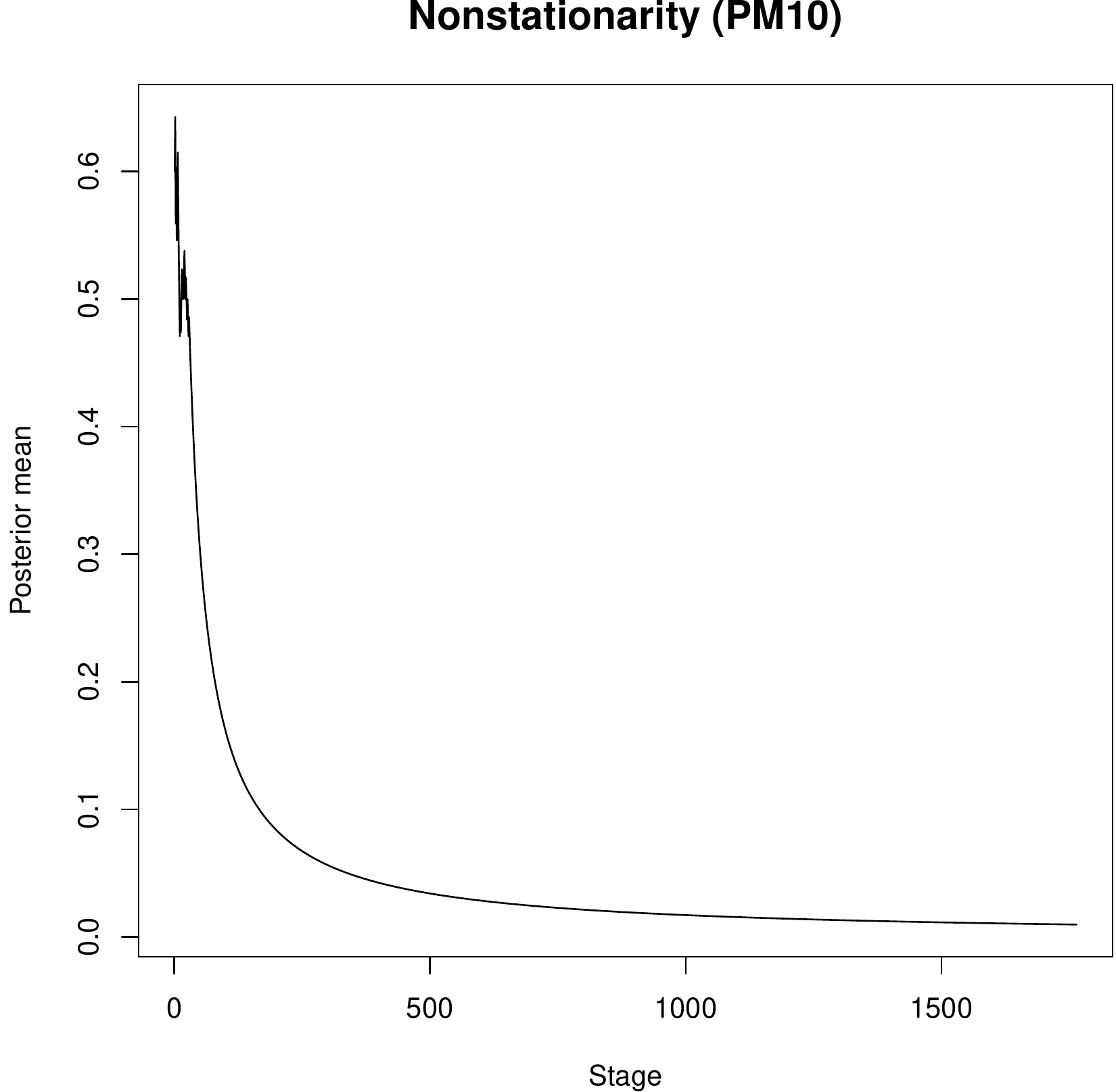}}
\caption{Detection of nonstationarity of the PM 10 data with our Bayesian method.} 
\label{fig:PM10_nonstationary}
\end{figure}

For checking covariance stationarity, our method with Whittle covariance failed to yield a valid $\hat C_1$ since we could obtain only a single neighborhood
$\mathcal N_{i,h_1,h_2}$, with $h_1=0.0$ and $h_2=0.15$. Hence, we set $\hat C_1=0.16$, the same value obtained for checking strict stationarity. Again,
for obtaining valid intervals, we needed to decrease the number of clusters and increase the number of observations per cluster. In this regard, setting
$K=500$ let us obtain four valid neighborhoods $\mathcal N_{i,h_j,h_{j+1}}$; $j=1,2,3,4$, with $h_1=0.0$, $h_2=0.1$, $h_3=0.2$, $h_4=0.3$ and $h_5=0.4$.
Figure \ref{fig:PM10_cov_nonstationary} shows covariance nonstationarity for the PM 10 data, as convergence to zero is 
indicated with $\mathcal N_{i,h_1,h_2}$ and $\mathcal N_{i,h_2,h_3}$.
\begin{figure}
\centering
\subfigure [$0\leq\|h\|<0.1$.]{ \label{fig:PM10_covns1}
\includegraphics[width=5.5cm,height=5.5cm]{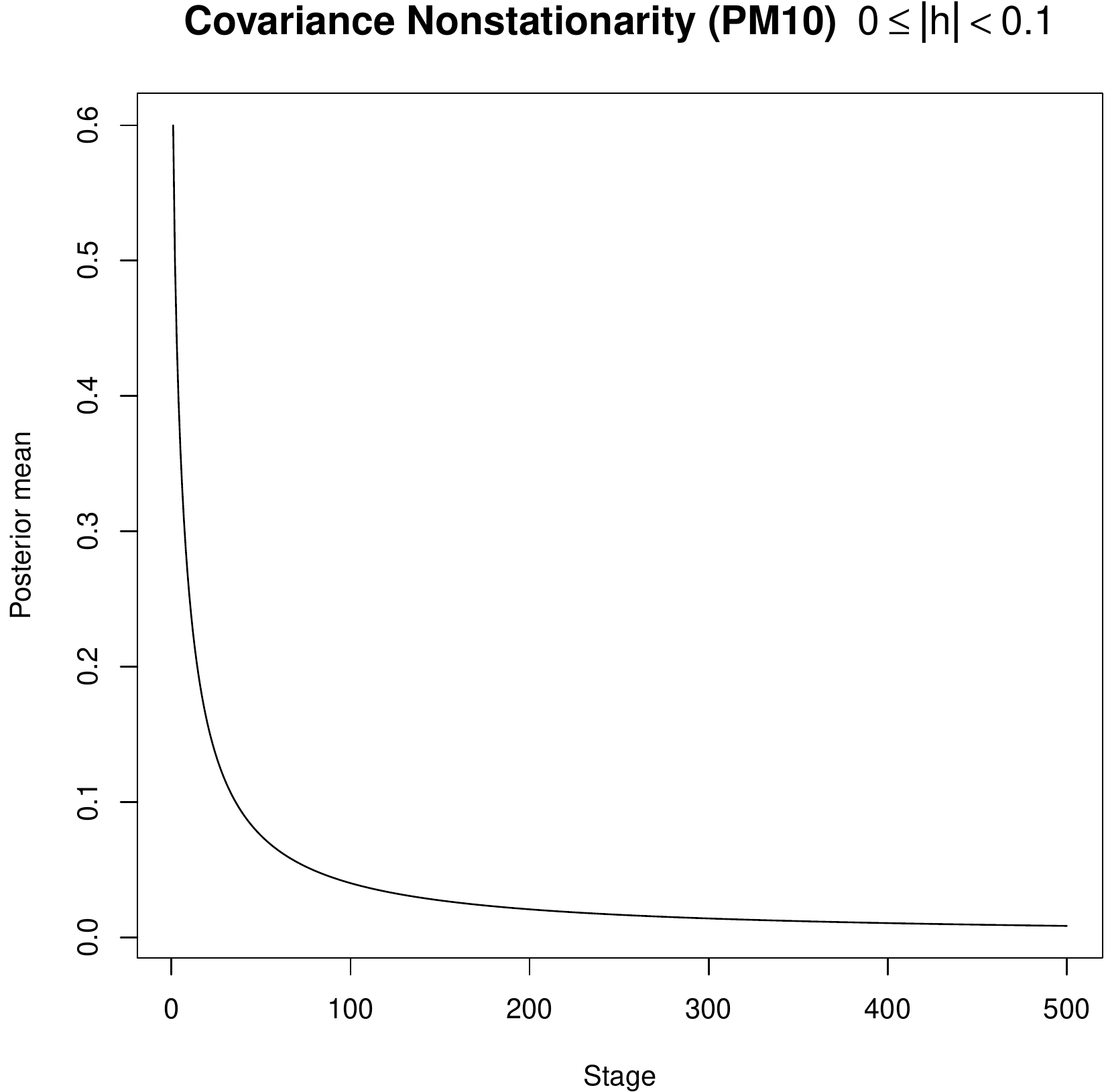}}
\hspace{2mm}
\subfigure [$0.1\leq\|h\|<0.2$.]{ \label{fig:PM10_covns2}
\includegraphics[width=5.5cm,height=5.5cm]{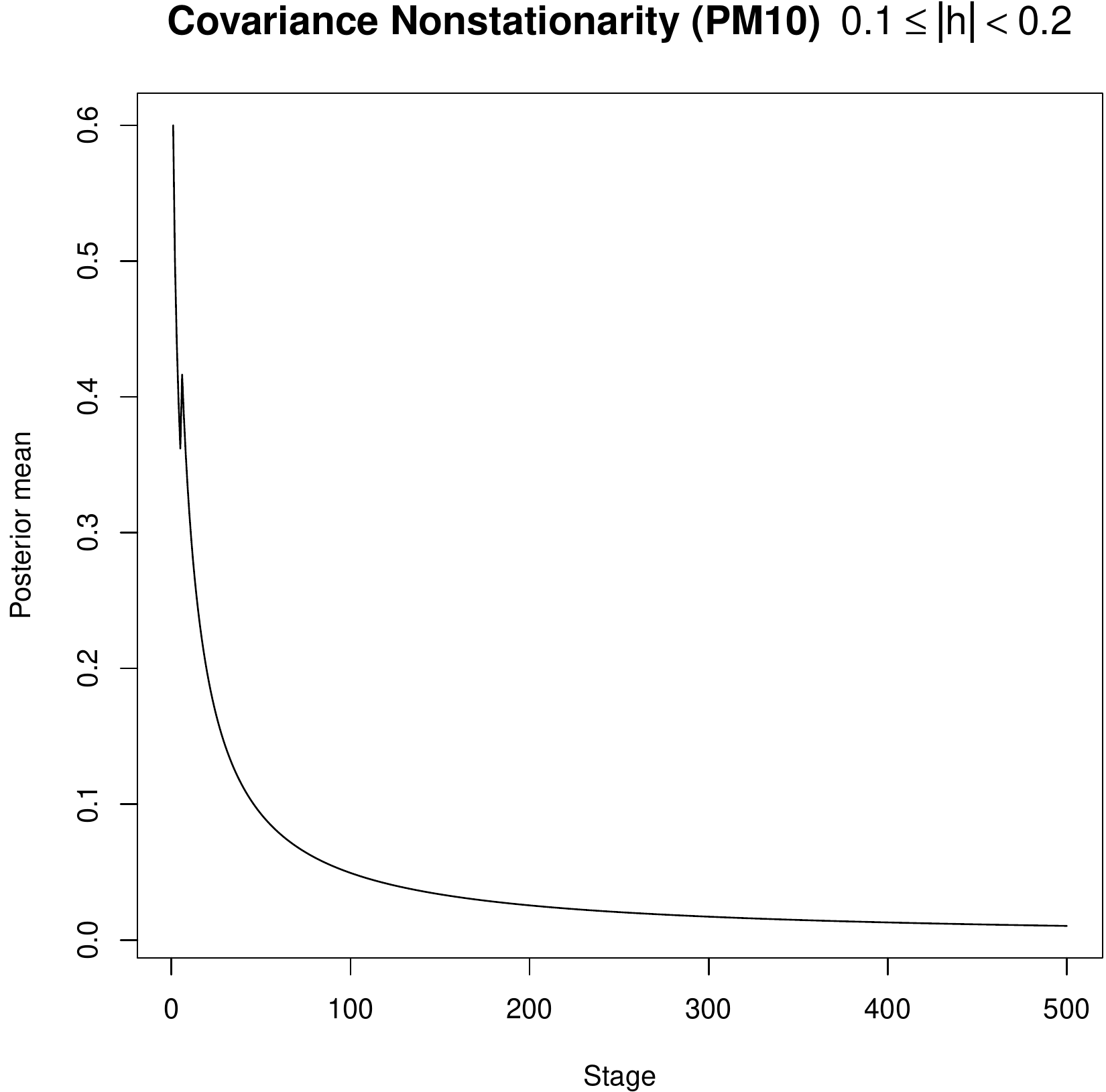}}\\
\vspace{2mm}
\subfigure [$0.2\leq\|h\|<0.3$.]{ \label{fig:PM10_covns3}
\includegraphics[width=5.5cm,height=5.5cm]{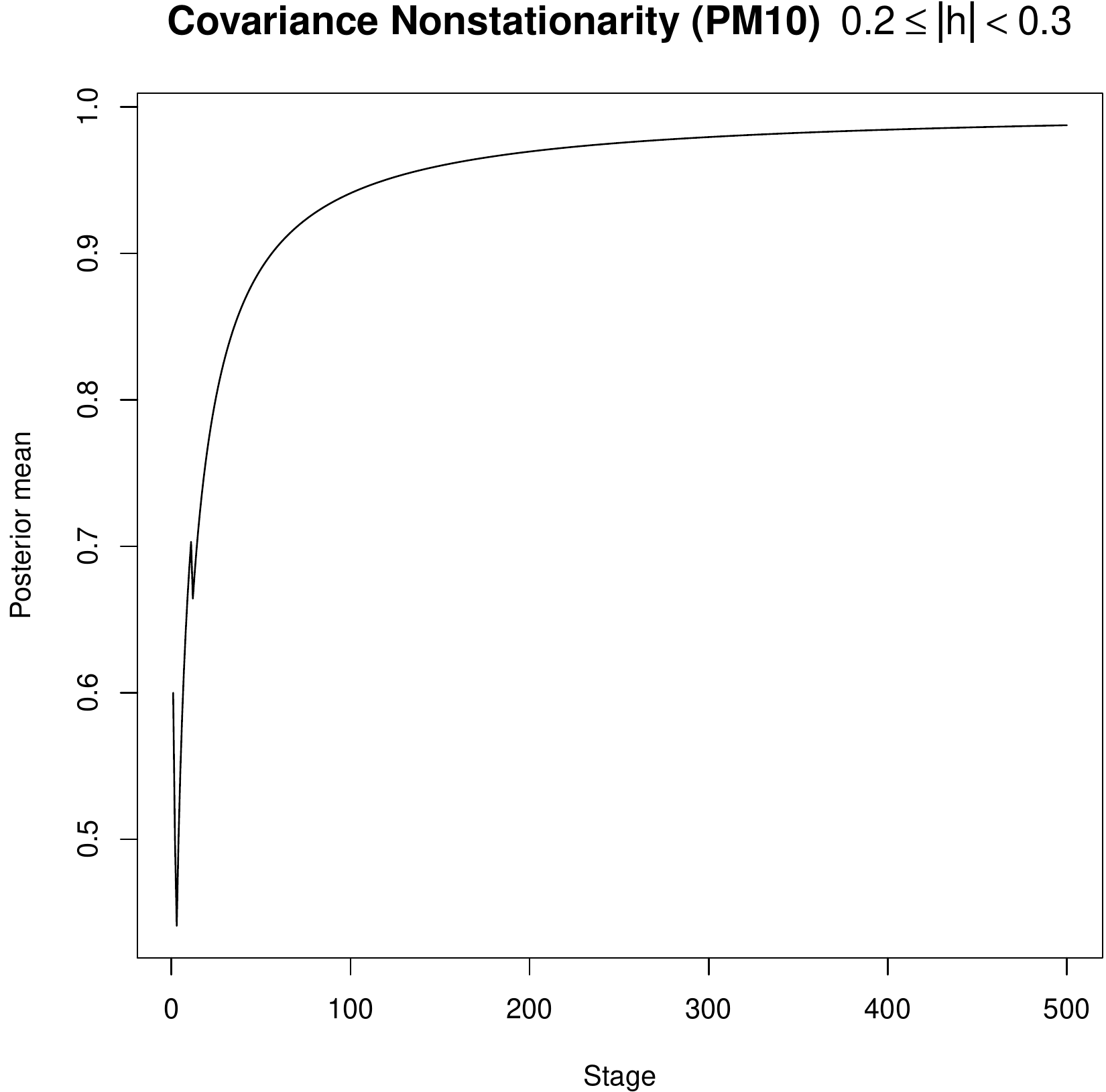}}
\hspace{2mm}
\subfigure [$0.3\leq\|h\|<0.4$.]{ \label{fig:PM10_covns4}
\includegraphics[width=5.5cm,height=5.5cm]{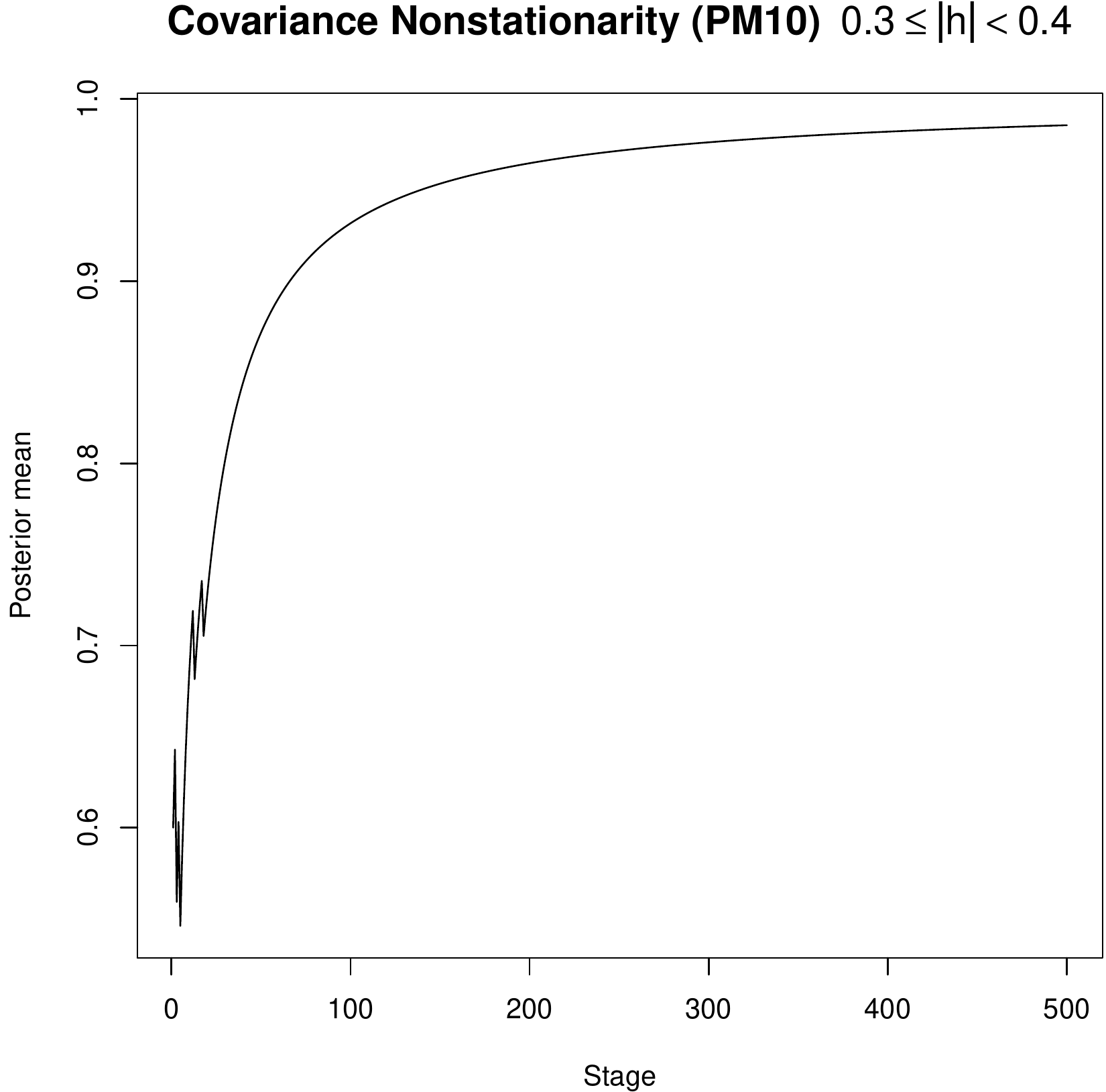}}
\caption{Detection of covariance nonstationarity of the PM 10 data.}
\label{fig:PM10_cov_nonstationary}
\end{figure}

\subsection{Spatio-temporal PM 2.5 data}
\label{subsec:PM25}

The PM 2.5 data set consists of $17496$ observations. For checking strict stationarity, we generated a GP sample 
of size $17496$ with the Whittle covariance function with $\psi=0.8$, with the same locations and time points as the real PM 2.5 data. Unlike the PM 10 case,
here the GP sample turned out to be stable, as shown in Figure \ref{fig:PM25_gp_sample}.
\begin{figure}
\centering
\subfigure [GP sample size $17496$.]{ \label{fig:gpsamp3}
\includegraphics[width=7.5cm,height=5.5cm]{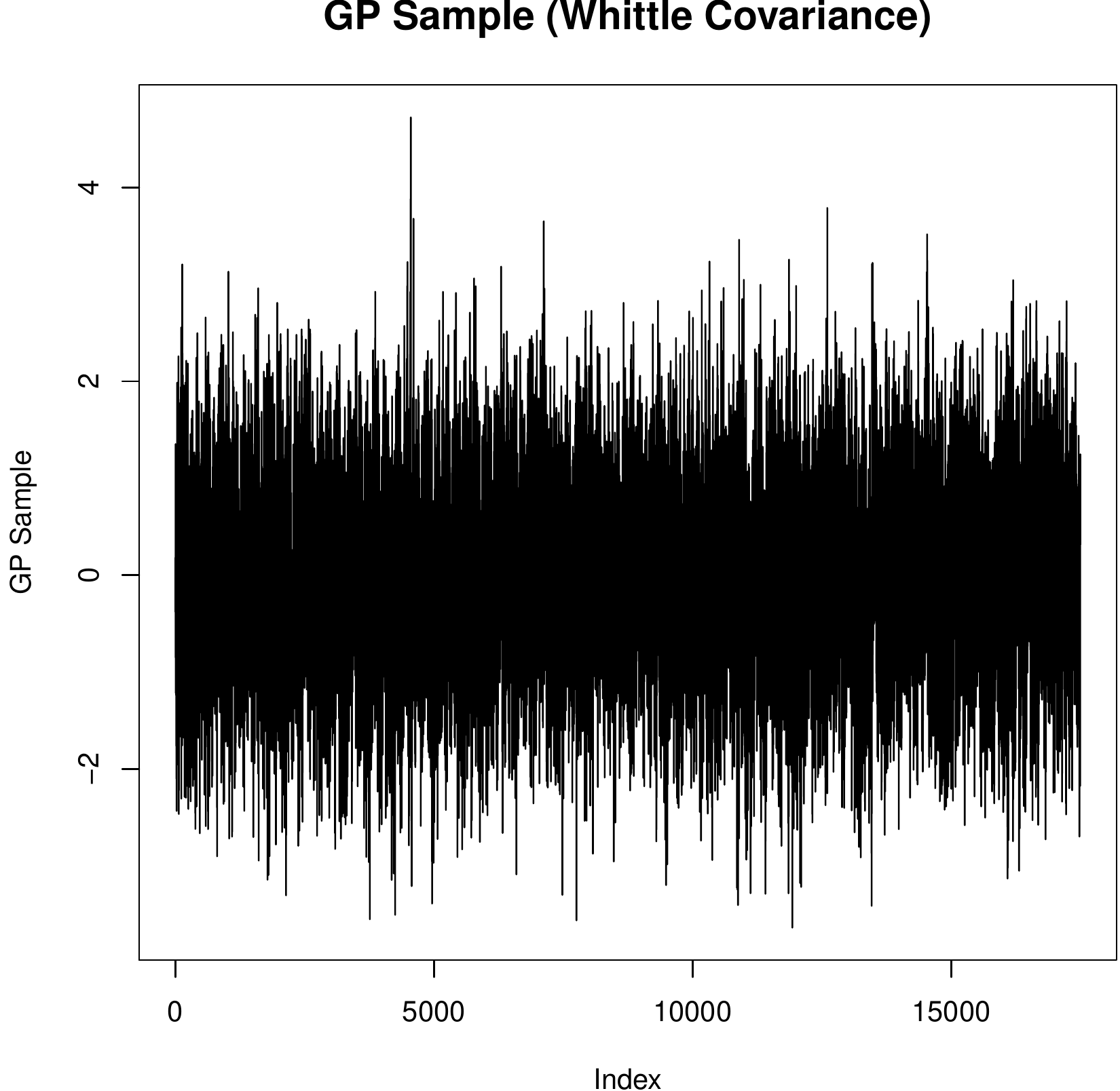}}
\caption{GP sample of size $17496$ for Whittle covariance with $\psi=0.8$ for PM 2.5 data.}
\label{fig:PM25_gp_sample}
\end{figure}
Setting $K=437$, so that there are $40$ observations on the average in each cluster, we obtained $\hat C_1=0.02$ with the Whittle based GP sample.
Figure \ref{fig:PM25_stationary} shows that the PM 2.5 data is strongly stationary. Hence, it is not necessary to check covariance statioanrity of this data.
\begin{figure}
\centering
\subfigure [Stationarity (PM 2.5 data).]{ \label{fig:PM25_stat}
\includegraphics[width=6.5cm,height=6.5cm]{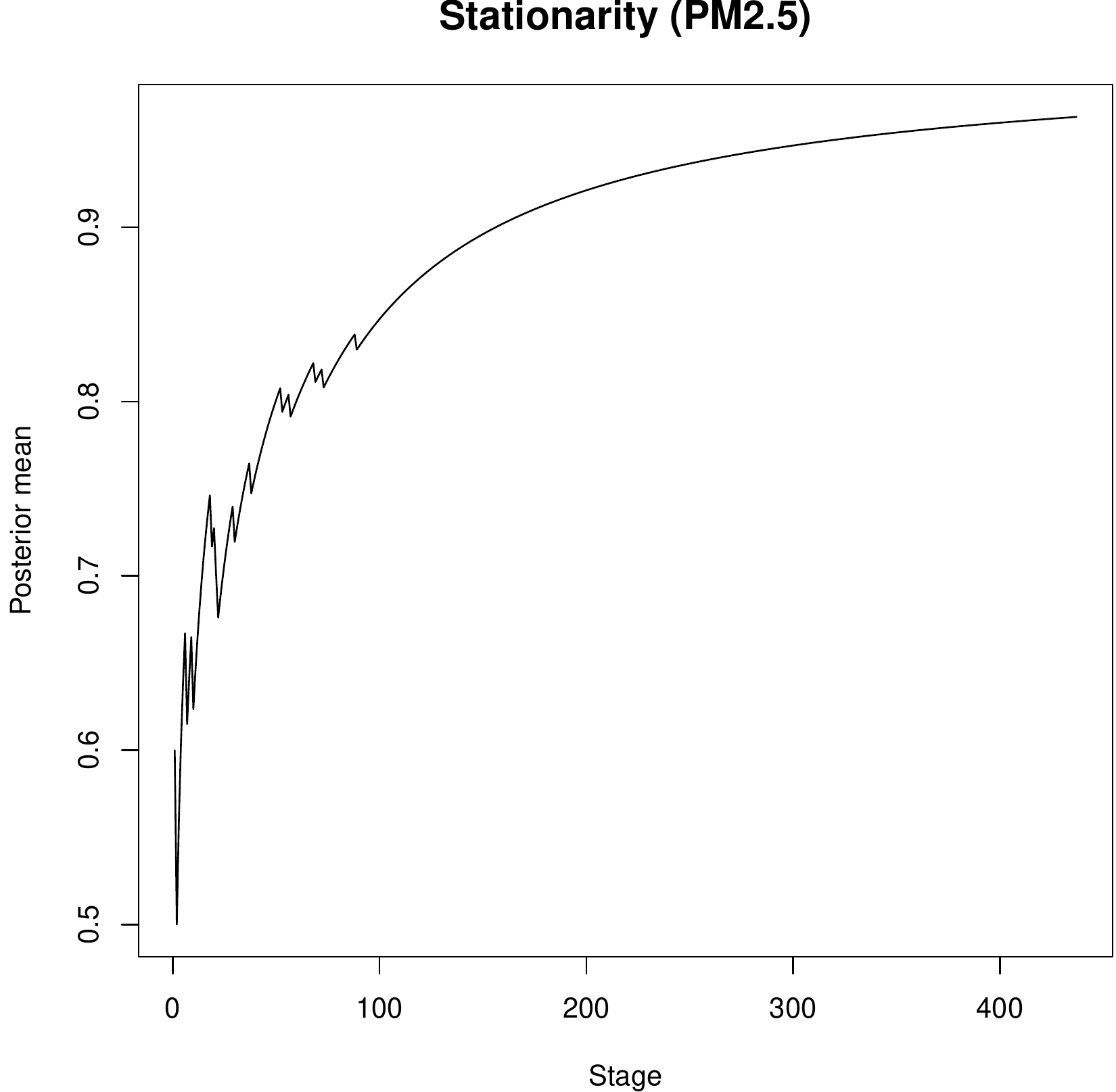}}
\caption{Detection of stationarity of the PM 2.5 data with our Bayesian method.} 
\label{fig:PM25_stationary}
\end{figure}

\section{Bayesian characterization of point processes}
\label{sec:point_processes}

Point pattern analysis is the study involving analysis of the spatial distribution of the observed events and to infer about the underlying data-generating process.
In this regard, an important question to ask is whether or not interactions exist between the events. Hence, a pertinent test that is often used in point pattern
analysis is the test of complete spatial randomness (CSR), that is, if the points are independently and uniformly distributed over the study area.
Theoretically, homogeneous Poisson point process (HPP) corresponds to CSR, and thus tests for CSR can be devised on such basis, assuming the Poisson process framework for
independent disjoint sets of events. However, rejecting CSR only rejects the HPP assumption and does not facilitate conclusion of stationarity
or nonstationarity, Poisson or non-Poisson process. Bayesian characterization of stationarity and nonstationarity can be achieved as before, while 
Bayesian characterization of CSR and Poisson assumption require further work. To characterize the Poisson assumption we exploit mutual independence of
disjoint sets of events, under the assumption of orderliness and almost sure boundedly finite property of the process without fixed atoms.

Testing for CSR can be found in \ctn{OSullivan03}, \ctn{Waller04} and \ctn{Schab05}. The key ingredient in such tests is the so-called $G$ function that
provides the distribution of the distance from any arbitrary event to its nearest event. Specifically, let $d_{ij}$ denote the distance between the
$i$-th and $j$-th events in a set of $n$ events, and for $s=1,\ldots,n$, let $d_s=\min\left\{d_{st}:t\neq s\right\}$. Consider the empirical distribution function
\begin{equation}
\hat G(x)=\frac{\sum_{s=1}^nI(d_s\leq x)}{n}.
\label{eq:hat_G}
\end{equation}
Under CSR, that is, under the assumption of homogeneous Poisson point process, $\hat G(x)$ has expectation
\begin{equation}
G(x)=1-\exp\left(-\lambda\pi x^2\right),
\label{eq:G_pp}
\end{equation}
the $G$-function. Here $\lambda$ is the intensity, or the number of events per unit area, the maximum likelihood estimator of which is given by
$\tilde\lambda=n/|W|$, where $W$ is the bounded region where the points are observed, and $|W|$ denotes the volume of $W$. 
Indeed, the entire point process $\bX$ defined on some region $\bS\subset\mathbb R^d$, for some $d\geq 1$ can not be observed, 
and hence a bounded region $W\subset \bS$ is considered where points are observed.
Let 
\begin{equation}
\tilde G(x)=1-\exp\left(-\tilde\lambda\pi x^2\right),
\label{eq:tilde_G}
\end{equation}

Let us assume that $\bX_K=\left\{X_s:s\in \cup_{i=1}^K\mathcal N_i\right\}$
has been observed, for $K>1$. Here $\cup_{i=1}^K\mathcal N_i$ corresponds to the observation window $W$.
For the purpose of asymptotics, we assume that $|\bS|$, the volume of $\bS$ tends to infinity, so that even though $|W|$ remains finite, $n$, the number
of points in $W$ tends to infinity, almost surely.

For any $x>0$, consider
\begin{equation}
\hat G_i(x)=n_i^{-1}\sum_{s\in\mathcal N_i}I(d_s\leq x),
\label{eq:eqG1}
\end{equation}
where $n_i=|\mathcal N_i|$, as before. Note that $n=\sum_{i=1}^kn_i$.

Now let
\begin{align}
\hat G_K(x)&=\frac{\sum_{s\in\cup_{i=1}^K\mathcal N_i}I(d_s\leq x)}{\sum_{i=1}^Kn_i}\notag\\
&=\frac{\sum_{i=1}^Kn_i\hat G_i(x)}{\sum_{i=1}^Kn_i}=\sum_{i=1}^K\hat p_{iK}\hat G_i(x),
\label{eq:eqG2}
\end{align}
where $\hat p_{ik}=n_i/\sum_{j=1}^Kn_j$, as before.
Let us now assume (\ref{eq:eq4}), which we recall as 
\begin{equation*}
\hat p_{iK}=\frac{n_i}{\sum_{j=1}^Kn_j}\rightarrow p_{iK}=\frac{p_i}{\sum_{j=1}^Kp_j},
\end{equation*}
as $n_j\rightarrow\infty$, for $j=1,\ldots,K$. Here $0\leq p_i\leq 1$, such that $\sum_{i=1}^{\infty}p_i=1$.

Let $W_d$ denote the space where the distances $d_i$, $i=1,\ldots,n$, associated with the observation window $W$, lie upon. 
%
However, for the asymptotic theory, we must let the window $W$ and corresponding $W_d$ to grow, otherwise the number of points $n$ can not tend to infinity. 
Indeed, for fixed $W$, even the MLE $\tilde\lambda=n/|W|$ is not a consistent estimator for $\lambda$ in the HPP case. Thus, in this regard, we consider
the sequences $W_r$, $W_{dr}$, $K=K_r$, $n_{ir}$, $n_r$, $K_r$, $\hat p_{iK_r}$ and $\tilde\lambda_r$, 
for $r=1,2,\ldots$, where the suffix $r$ is incorporated to our previous notation to signify sequences. Let $|W_r|\rightarrow\infty$ as $r\rightarrow\infty$.
Note that $K_r$ may remain finite even as $r\rightarrow\infty$.
Let us also denote by $G_{true}$ the true point process generating the data. Note that for HPP, $G_{true}=G$. In reality, the true point process, and hence $G_{true}$,
is unknown.

A problem associated with HPP is that it is hard to establish $\underset{x\in W_{dr}}{\sup}~\left|\hat G_K-\tilde G(x)\right|\rightarrow 0$, in either weak or strong sense. 
To see this, note that
\begin{align}
	\underset{x\in W_{dr}}{\sup}~\left|\tilde G(x)- G(x)\right|
	&=\underset{x\in W_{dr}}{\sup}~\left|\exp\left(-\lambda\pi x^2\right)\left(1-\exp\left(-\pi x^2\left(\tilde\lambda_r-\lambda\right)\right)\right)\right|\notag\\
	&\leq 1-\underset{x\in W_{dr}}{\inf}~\exp\left(-\pi x^2\left|\tilde\lambda_r-\lambda\right|\right).\notag
	\label{eq:nonconv1}
\end{align}
Since $\exp\left(-\pi x^2\left|\tilde\lambda_r-\lambda\right|\right)$ is decreasing in $x^2$ and $W_{dr}$ is bounded, the infimum over $W_{dr}$ is given by
$\exp\left(-\pi \xi^2_r\left|\tilde\lambda_r-\lambda\right|\right)$, where $\xi_r$ is the maximum interpoint distance in $W_{dr}$. 
In other words,
\begin{equation}
\underset{x\in W_d}{\sup}~\left|\tilde G(x)- G(x)\right|
\leq 1-\exp\left(-\pi \xi^2_r\left|\tilde\lambda_r-\lambda\right|\right).
\label{eq:nonconv2}
\end{equation}
By Markov's inequality, for any $\epsilon>0$, 
$$P\left(\xi^2_r\left|\tilde\lambda_r-\lambda\right|>\epsilon\right)<\epsilon^{-2}\xi^4_rE\left(\frac{n_r}{|W_r|}-\lambda\right)^2
=\epsilon^{-2}\lambda\frac{\xi^4_r}{|W_r|},$$
which tends to zero if $\frac{\xi^4_r}{|W_r|}\rightarrow 0$ as $r\rightarrow\infty$. But as can be easily verified, this does not hold for regular window shapes
such as squares, rectangles, circles, triangles, etc. Indeed, for these shapes, $\frac{\xi^4_r}{|W_r|}\rightarrow \infty$ as $r\rightarrow\infty$.

Instead of $\underset{x\in W_{dr}}{\sup}~\left|\hat G_K(x)- \tilde G(x)\right|$ we shall thus deal with 
$\int_{W_{dr}}\left|\hat G_K(x)- \tilde G(x)\right|dG_{true}(x)$ in the following theorem.


\begin{theorem}
\label{theorem:gc1_pp}
Assume that $\bX$ follows homogeneous Poisson point process, and that the points are observed in the window $W_r$, where $|W_r|\rightarrow\infty$
as $r\rightarrow\infty$. Let $W_{dr}$ denote the space of the distances associated with $W_r$. 
Then, for all values of $K_{\infty}=\underset{r\rightarrow\infty}{\lim}~K_r$,
\begin{equation}
	\underset{r\rightarrow\infty,n_{ir}\rightarrow\infty,i=1,\ldots,K_r}{\lim}~
	~\int_{W_{dr}}\left|\hat G_{K_r}(x)-\tilde G(x)\right|dG_{true}(x)=0,
\label{eq:gc1_pp}
\end{equation}
	almost surely if $\sum_{r=1}^{\infty}|W_r|^{-1}<\infty$.
\end{theorem}
\begin{proof}
%
Observe that
\begin{align}
	&\int_{W_{dr}}\left|\hat G_{K_r}(x)-\tilde G(x)\right|dG_{true}(x)\notag\\
	&\qquad\leq \underset{x\in W_{dr}}{\sup}~\left|\hat G_{K_r}(x)-G(x)\right|G_{true}(W_{dr}) +\int_{W_{dr}}\left|\tilde G(x)- G(x)\right|dG_{true}(x)\notag\\
	&\qquad\leq \underset{x\in W_{dr}}{\sup}~\left|\hat G_{K_r}(x)-G(x)\right| +\int_{W_{dr}}\left|\tilde G(x)- G(x)\right|dG_{true}(x).
\label{eq:pp1}
\end{align}

Since 
\begin{equation}
	\underset{x\in W_{dr}}{\sup}~\left|\hat G_{K_r}(x)-G(x)\right|
	=\underset{x\in W_{dr}}{\sup}~\left|\sum_{i=1}^{K_r}\left(\hat G_i(x)-G(x)\right)\right|
	\leq \sum_{i=1}^{K_r}\hat p_{iK_r}\underset{x\in W_{dr}}{\sup}~\left|\hat G_i(x)-G(x)\right|.
	\label{eq:gv_pp2}	
\end{equation}
Now, as $r\rightarrow\infty$, the right hand side of (\ref{eq:gv_pp2}) converges almost surely to 
	\begin{equation}
		\sum_{i=1}^{K_{\infty}} p_{iK_{\infty}}\underset{r\rightarrow\infty}{\lim}~\underset{x\in W_{dr}}{\sup}~\left|\hat G_i(x)-G(x)\right|,
		\label{eq:conv_pp3}
	\end{equation}
since $\hat p_{iK_r}\rightarrow p_{iK_{\infty}}$ in the same way as (\ref{eq:eq4}). 
	Also, $\underset{x\in W_{dr}}{\sup}~\left|\hat G_i(x)-G(x)\right|\stackrel{a.s.}{\longrightarrow}0$, 
	as $r\rightarrow\infty$ and $n_{ir}\rightarrow\infty$ by
Glivenko-Cantelli theorem for stationary random variables (\ctn{Stute80}). 
That is, given any $K_{\infty}$, (\ref{eq:conv_pp3}) converges to zero almost surely. Thus, (\ref{eq:conv_pp3}) converges to zero almost surely, even as
$K_{\infty}\rightarrow\infty$. 
Hence, it follows from these arguments and (\ref{eq:gv_pp2}) that for all values of $K_{\infty}$,
\begin{equation*}
	\underset{x\in W_{dr}}{\sup}~\left|\hat G_K(x)-G(x)\right|\stackrel{a.s.}{\longrightarrow}0,~\mbox{as}~n_{ir}\rightarrow\infty,~i=1,\ldots,K_r,~r\rightarrow\infty,
\end{equation*}
and hence
\begin{equation}
	\int_{W_{dr}}\left|\hat G_K(x)-G(x)\right|dG_{true}(x)\stackrel{a.s.}{\longrightarrow}0,~\mbox{as}~n_{ir}\rightarrow\infty,~i=1,\ldots,K_r,~r\rightarrow\infty.
\label{eq:gv_pp3}
\end{equation}
Now note that, for $\tilde\lambda_r=n_r/|W_r|$,
\begin{align}
	\int_{W_{dr}}\left|\tilde G(x)- G(x)\right|dG_{true}(x)
	&=\int_{W_{dr}}\left|\exp\left(-\lambda\pi x^2\right)\left(1-\exp\left(-\pi x^2\left(\tilde\lambda-\lambda\right)\right)\right)\right|dG_{true}(x)\notag\\
	&\leq G_{true}\left(W_{dr}\right)-\int_{W_{dr}}\exp\left(-\pi x^2\left|\tilde\lambda-\lambda\right|\right)dG_{true}(x)
	\label{eq:l2_conv1}.
\end{align}
In (\ref{eq:l2_conv1}), 
	\begin{equation}	
		G_{true}\left(W_{dr}\right)\rightarrow 1,~\mbox{as}~r\rightarrow\infty. 
		\label{eq:l2_conv2}
	\end{equation}
Now, by Markov's inequality, for any $\epsilon>0$,
\begin{align}
	&\sum_{r=1}^{\infty}P\left(\left|\tilde\lambda_r-\lambda\right|>\epsilon\right)
	=\sum_{r=1}^{\infty}P\left(\left|\frac{n_r}{|W_r|}-\lambda\right|>\epsilon\right)\notag\\
	&\qquad<\epsilon^{-2}\sum_{r=1}^{\infty}E\left(\frac{n_r}{|W_r|}-\lambda\right)^2=\epsilon^{-2}\lambda\sum_{r=1}^{\infty}\frac{1}{|W_r|}
	<\infty,\notag
\end{align}
where the last step is due to our assumption.
Hence, by Borel-Cantelli lemma, $\left|\tilde\lambda_r-\lambda\right|\stackrel{a.s.}{\longrightarrow}0,~\mbox{as}~r\rightarrow\infty$.
By dominated convergence theorem, it follows that 
\begin{equation}
	\int_{W_{dr}}\exp\left(-\pi x^2\left|\tilde\lambda-\lambda\right|\right)dG_{true}(x)\stackrel{a.s.}{\longrightarrow}1,~\mbox{as}~r\rightarrow\infty.
\label{eq:l2_conv3}
\end{equation}	

It follows from (\ref{eq:l2_conv1}), (\ref{eq:l2_conv2}) and (\ref{eq:l2_conv3}) that
	\begin{equation}
		\int_{W_{dr}}\left|\tilde G(x)- G(x)\right|dG_{true}(x)\stackrel{a.s.}{\longrightarrow}0,~\mbox{as}~r\rightarrow\infty.
		\label{eq:gv_pp5}
	\end{equation}
The result follows by combining (\ref{eq:pp1}), (\ref{eq:gv_pp3}) and (\ref{eq:gv_pp5}).	
\end{proof}

\begin{remark}
\label{remark:on_K}
	Note that unlike in the previous cases where we required $K\rightarrow\infty$, here we did not require the assumption $K_{\infty}\rightarrow\infty$.
	Theorem \ref{theorem:gc1_pp} explicitly mentions that the result holds for all values of $K_{\infty}$. This difference is due to the fact that
	in the asymptotics of point process we assumed that the observation window $W_r$ is growing with $r$, and with such growing observation window,
	the entire point process can be ultimately captured. Hence increasing the number of clusters is not required. 
	From a more mathematical perspective, note that $\hat G_K$ uses all the observations in the observation window, and so the value of $K$ is irrelevant
	mathematically.
\end{remark}

\begin{remark}
\label{remark:gv_pp}
Note that by direct application of Glivenko-Cantelli theorem for stationary random variables we can obtain, 
\begin{equation}
	\underset{x\in W_{dr}}{\sup}~\left|\hat G_{K_r}(x)-G(x)\right|	\stackrel{a.s.}{\longrightarrow}0,~\mbox{as}~r\rightarrow\infty.
\label{eq:gv_pp1}
\end{equation}
This does not require breaking up the observation window $W_r$ into sub-regions $\mathcal N_1,\ldots,\mathcal N_{K_r}$, and the
assumption $n_{ir}\rightarrow\infty$ for $i=1,\ldots,K_r$.
However, it is important to detect which sub-regions of $W_r$ are not representatives of CSR. From this perspective, it is important to consider 
the sub-regions $\mathcal N_1,\ldots,\mathcal N_{K_r}$, and consideration of the form (\ref{eq:eqG2}), which we formalize in our Bayesian characterization.
\end{remark}

Let $\{c_j\}_{j=1}^{\infty}$ be a non-negative decreasing sequence and
\begin{equation}
	Y_{j,n_{jr}}=\mathbb I{\left\{\int_{W_{dr}}\left|\hat G_j(x)-\tilde G(x)\right|dG(x)\leq c_j\right\}}.
\label{eq:Y_j_n_pp}
\end{equation}
In practice, we shall approximate $\int_{W_{dr}}\left|\hat G_j(x)-\tilde G(x)\right|dG(x)$ by $\frac{1}{n_{jr}}\sum_{i=1}^{n_{jr}}\left|\hat G_j(d_i)-\tilde G(d_i)\right|$,
where the distances $d_i$ are assumed to correspond to the true data-generating point process $G_{true}$.

As before, let, for $j\geq 1$,
\begin{equation}
	P\left(Y_{j,n_{jr}}=1\right)=p_{j,n_{jr}}.
\label{eq:p_j_n_pp}
\end{equation}
Hence, the likelihood of $p_{j,n_{jr}}$, given $y_{j,n_{jr}}$, is given by the form (\ref{eq:likelihood}).

As before, we construct a recursive Bayesian methodology that formally characterizes
homogeneous Poisson process and otherwise in terms of formal posterior convergence. 
The relevant theorems in this regard, the proofs of which are similar to stationarity and nonstationarity characterizations, are presented
below as Theorems \ref{theorem:convergence_pp} and \ref{theorem:divergence_pp}.


\begin{theorem}
\label{theorem:convergence_pp}
For all $\omega\in\mathfrak S\cap\mathfrak N^c$, where $\mathfrak N$ is some null set having probability measure zero,
$\bX\cap W$ follows homogeneous Poisson process if and only if 
for any monotonically decreasing sequence 
$\left\{c_j(\omega)\right\}_{j=1}^{\infty}$,
\begin{equation}
	\pi\left(\mathcal N_1|y_{k,n_{kr}}(\omega)\right)\rightarrow 1,
\label{eq:consistency_at_1_pp}
\end{equation}
as $k\rightarrow\infty$ and $n_{jr}\rightarrow\infty$ for $j=1,\ldots,K_r$ satisfying (\ref{eq:eq4}) and $K_r\rightarrow\infty$ as $r\rightarrow\infty$, 
where $\mathcal N_1$ is any neighborhood of 1 (one).
\end{theorem}

\begin{theorem}
\label{theorem:divergence_pp}
$\bX\cap W$ does not follow homogeneous Poisson process if and only if 
for any $\omega\in\mathfrak S\cap\mathfrak N^c$ where $\mathfrak N$ is some null set having probability measure zero, 
for any choice of the non-negative, monotonically decreasing sequence $\{c_j(\omega)\}_{j=1}^{\infty}$,
\begin{equation}
	\pi\left(\mathcal N_0|y_{k,n_{kr}(\omega)}(\omega)\right)\rightarrow 1,
\label{eq:consistency_at_0_pp}
\end{equation}
as $k\rightarrow\infty$ and $n_{jr}\rightarrow\infty$, $j=1,\ldots,K_r$ satisfying (\ref{eq:eq4}), and $K_r\rightarrow\infty$ as $r\rightarrow\infty$, 
where $\mathcal N_0$ is any neighborhood of 0 (zero).
\end{theorem}

\begin{remark}
\label{remark:remark2}
Note that Theorems \ref{theorem:convergence_pp} and \ref{theorem:divergence_pp} require $K_r\rightarrow\infty$ as $r\rightarrow\infty$, even though	
Theorem \ref{theorem:gc1_pp} does not have this requirement. But this arises entirely for convergence of the recursive Bayesian algorithm
as the stage number $k\rightarrow\infty$.
\end{remark}	

\subsection{Discussion on edge correction}
\label{subsec:computation_edge}


Since the data are observed in the bounded window $W$, the minimum distance $d_i$ in the window may be larger than the true minimum distance had the complete
point process $\bX$ been observed. In classical point process analysis, this may induce a bias in estimating the true distribution function, which is known as edge effect. 
Needless to mention, various corrections for such edge effect is available in the literature. 

However, in the way we proceed with our Bayesian method, the edge effects do not influence our final results. The reason for this is the following.
We partition the point pattern in the observation window $W$ into $K$ clusters using the K-means clustering algorithm. Thus, within each cluster
in the interior of $W$, the edge effect is minimized. This is because the K-means clustering algorithm guarantees that within cluster variation
is minimized and the between cluster variation is maximized, which entails that the minimum distance $d_i$ of any point $i$ within each cluster is
often indeed the minimum when all the points are considered. Note that this is actually the case for `empty distances', if the distances are measured from
the centroid of each cluster. 
Our experiments demonstrate the validity of our aforementioned arguments in this regard.


\subsection{Characterization of stationarity and nonstationarity of point processes}
\label{subsec:pp_stationarity}
The characterization of stationarity and nonstationarity in the point process setup remains essentially the same as in the general situation, with the
conceptual difference being consideration of $W_r$ in the point process setup, with $|W_r|\rightarrow\infty$. We present the main results regarding
stationarity and nonstationarity in the point process setup, which are slight modifications of Theorems \ref{theorem:gc1}, \ref{theorem:gc2}, \ref{theorem:gc3},
\ref{theorem:convergence} and \ref{theorem:divergence}.

\begin{theorem}
\label{theorem:gc1_spp}
Let $K_r\rightarrow\infty$ as $r\rightarrow\infty$. Then
\begin{equation*}
%
\underset{r\rightarrow\infty}{\lim}{\lim}\underset{n_{ir}\rightarrow\infty,i=1,\ldots,K_r}{\lim}
	~\underset{C}{\sup}~\left|\tilde P_{K_r}(C)-P_\infty(C)\right|=0,~\mbox{almost surely}.
\end{equation*}
\end{theorem}

\begin{theorem}
\label{theorem:gc2_spp}
The point process $\bX$ is stationary if and only if  
$\underset{C}{\sup}~\left|\hat P_j(C)-\tilde P_{K_r}(C)\right|\rightarrow 0$ almost surely, as $n_{jr}\rightarrow\infty$
satisfying (\ref{eq:eq4}), $j=1,\ldots,K_r$, $K_r\rightarrow\infty$, as $r\rightarrow\infty$.
\end{theorem}

\begin{theorem}
\label{theorem:gc3_spp}
	$\bX$ is nonstationary if and only if $\underset{C}{\sup}~\left|\hat P_j(C)-\tilde P_{K_r}(C)\right|> 0$ almost surely, 
	as $n_{jr}\rightarrow\infty$
satisfying (\ref{eq:eq4}), $j=1,\ldots,K_r$, $K_r\rightarrow\infty$, as $r\rightarrow\infty$.
\end{theorem}

Let $\{c_j\}_{j=1}^{\infty}$ be a non-negative decreasing sequence and
\begin{equation*}
	Y_{j,n_{jr}}=\mathbb I{\left\{\underset{C}{\sup}~\left|\hat P_j(C)-\tilde P_{K_r}(C)\right|\leq c_j\right\}}.
\end{equation*}
Let, for $j\geq 1$,
\begin{equation*}
	P\left(Y_{j,n_{jr}}=1\right)=p_{j,n_{jr}}.
\end{equation*}

\begin{theorem}
\label{theorem:convergence_spp}
For all $\omega\in\mathfrak S\cap\mathfrak N^c$, where $\mathfrak N$ is some null set having probability measure zero,
$\bX$ is stationary if and only if 
for any monotonically decreasing sequence 
$\left\{c_j(\omega)\right\}_{j=1}^{\infty}$,
\begin{equation*}
	\pi\left(\mathcal N_1|y_{k,n_{kr}}(\omega)\right)\rightarrow 1,
\end{equation*}
	as $k\rightarrow\infty$ and $n_{jr}\rightarrow\infty$ for $j=1,\ldots,K_r$ satisfying (\ref{eq:eq4}) and $K_r\rightarrow\infty$ as $r\rightarrow\infty$, 
where $\mathcal N_1$ is any neighborhood of 1 (one).
\end{theorem}

\begin{theorem}
\label{theorem:divergence_spp}
$\bX$ is nonstationary if and only if 
for any $\omega\in\mathfrak S\cap\mathfrak N^c$ where $\mathfrak N$ is some null set having probability measure zero, 
for any choice of the non-negative, monotonically decreasing sequence $\{c_j(\omega)\}_{j=1}^{\infty}$,
\begin{equation*}
	\pi\left(\mathcal N_0|y_{k,n_{kr}(\omega)}(\omega)\right)\rightarrow 1,
\end{equation*}
	as $k\rightarrow\infty$ and $n_{jr}\rightarrow\infty$, $j=1,\ldots,K_r$ satisfying (\ref{eq:eq4}), and $K_r\rightarrow\infty$ as $r\rightarrow\infty$, 
where $\mathcal N_0$ is any neighborhood of 0 (zero).
\end{theorem}

\subsection{Characterization of mutual independence among random variables}
\label{subsec:mutual_independence}
In this section we first characterize mutual independence among a general set of random variables $\bX_K=\left(X_1,\ldots,X_K\right)$, as $K\rightarrow\infty$,
and then specialize the characterization in the point process setup.
Indeed, although characterizations and tests for mutual independence among a set of random variables is available in the literature
(see, for example, \ctn{Puri71}, \ctn{Gieser97}, \ctn{Um01}, \ctn{Cleroux95}, \ctn{Bilodeau05}, \ctn{Hoeffding48}, \ctn{Blum61}, \ctn{Ghoudi01}, \ctn{Beran07},
\ctn{Bilodeau17}), they are meant for a finite set of random variables.
Moreover, such characterizations are often not computationally manageable. Here we attempt to provide a characterization for number of random variables tending to
infinity, with manageable computation. Also, unlike the previous approaches, we need only asymptotic stationarity of the realizations of the random variables, not
even independence.

The key idea is to consider the differences 
\begin{equation}
\zeta_i=\underset{t_1,\ldots,t_i\in\mathbb R}{\sup}~\left|P\left(X_i\leq t_i|X_1\leq t_1,\ldots,X_{i-1}\leq t_{t-1}\right)-P\left(X_i\leq t_i\right)\right|,
	\label{eq:diff1}
\end{equation}
for $i=2,\ldots,K$, with $\zeta_1=0$. If all $\zeta_i$; $i=2,\ldots,K$, are sufficiently small, 
then the random variables $\left(X_1,\ldots,X_K\right)$ are mutually independent.
For practical purposes, we must replace $$P\left(X_i\leq t_i|X_1\leq t_1,\ldots,X_{i-1}\leq t_{t-1}\right)$$ and $P\left(X_i\leq t_i\right)$
with their corresponding empirical probabilities. In other words, we write 
\begin{equation}
P\left(X_i\leq t_i|X_1\leq t_1,\ldots,X_{i-1}\leq t_{t-1}\right)
	=\frac{P\left(X_1\leq t_1,\ldots,X_{i-1}\leq t_{t-1},X_i\leq t_i\right)}{P\left(X_1\leq t_1,\ldots,X_{i-1}\leq t_{t-1}\right)},
\label{eq:diff2}
\end{equation}
and replace $P\left(X_1\leq t_1,\ldots,X_{i-1}\leq t_{t-1},X_i\leq t_i\right)$ and $P\left(X_1\leq t_1,\ldots,X_{i-1}\leq t_{t-1}\right)$
with their corresponding empirical distribution functions $$F_{n,1:i}\left(X_1\leq t_1,\ldots,X_{i-1}\leq t_{t-1},X_i\leq t_i\right)$$
and $$F_{n,1:(i-1)}\left(X_1\leq t_1,\ldots,X_{i-1}\leq t_{t-1}\right),$$ respectively.
We also replace $P\left(X_i\leq t_i\right)$ with its empirical distribution function $F_n,i\left(X_i\leq t_i\right)$.
We denote the differences of the empirical distribution functions corresponding to (\ref{eq:diff1}) by $\hat \zeta_i$; $i=2,\ldots,k$, with $\hat\zeta_1=0$. 

However, computation of the joint empirical distribution functions $F_{n,1:i}$ often turn out to be zero numerically, even if $i$ is not too large. To address
this, we resort to Bayesian nonparametrics, with Dirichlet process prior for the joint distribution of $\bX_K$. In fact, more generally, we consider a stochastic process
prior for the sequence of random variables $\bX=\left(X_1,X_2,X_3,\ldots\right)$. Let $G_0$ denote the expected parametric stochastic process for $\bX$.
Specifically, we assume that $\bX\sim G$ and $G\sim DP\left(\alpha G_0\right)$, where $DP\left(\alpha G_0\right)$ stands for Dirichlet process with base measure $G_0$
and strength parameter $\alpha>0$. More transparently, let $\bX_{i_1,i_2,\ldots,i_K}=\left(X_{i_1},X_{i_2},\ldots,X_{i_K}\right)$, for any set of indices
$i_1,\ldots,i_K$. Then $\bX_{i_1,i_2,\ldots,i_K}\sim G_{i_1,i_2,\ldots,i_K}$ and $G_{i_1,i_2,\ldots,i_K}\sim DP\left(\alpha G_{0,i_1,i_2,\ldots,i_K}\right)$,
where $G_{i_1,i_2,\ldots,i_K}$ and $G_{0,i_1,i_2,\ldots,i_K}$ are $k$-dimensional distributions associated with $\bX_{i_1,i_2,\ldots,i_K}$.

Now, if data $\bX^j_{i_1,i_2,\ldots,i_K}$; $j=1,2,\ldots$, are available which are not necessarily $iid$ or not even independent, we consider the following recursive
strategy for sequentially updating the posterior distribution of the Dirichlet process. We assume that
\begin{equation}
	\bX^1_{i_1,i_2,\ldots,i_K}\sim G_1;~G_1\sim DP\left(\alpha G_{0,i_1,i_2,\ldots,i_K}\right).
	\label{eq:dp_indep1}
\end{equation}
so that the posterior distribution of the random distribution given $\bX^1_{i_1,i_2,\ldots,i_K}$ is given by
\begin{equation}
	[G_1|\bX^1_{i_1,i_2,\ldots,i_K}]\sim DP\left(\alpha G_{0,i_1,i_2,\ldots,i_K}+\delta_{\bX^1_{i_1,i_2,\ldots,i_K}}\right).
	\label{eq:dp_indep2}
\end{equation}
Now, assuming $[G_1|\bX^1_{i_1,i_2,\ldots,i_K}]$ to be the prior for the distribution of $\bX^2_{i_1,i_2,\ldots,i_K}$, we have 
\begin{equation}
	[G_2|\bX^2_{i_1,i_2,\ldots,i_K}]\sim DP\left(\alpha G_{0,i_1,i_2,\ldots,i_K}
	+\delta_{\bX^1_{i_1,i_2,\ldots,i_K}}+\delta_{\bX^2_{i_1,i_2,\ldots,i_K}}\right).
	\label{eq:dp_indep3}
\end{equation}
Continuing as (\ref{eq:dp_indep1}), (\ref{eq:dp_indep2}) and (\ref{eq:dp_indep3}), we obtain in general, for $j\geq 1$,
\begin{equation}
[G_j|\bX^j_{i_1,i_2,\ldots,i_K}]\sim DP\left(\alpha G_{0,i_1,i_2,\ldots,i_K}
	+\sum_{r=1}^j\delta_{\bX^r_{i_1,i_2,\ldots,i_K}}\right).
	\label{eq:dp_indep4}
\end{equation}
Note that the posterior in this case is of the same form as that of
$[G_j|\bX^r_{i_1,i_2,\ldots,i_K};r=1,\ldots,j]$, had $\bX^r_{i_1,i_2,\ldots,i_K};r=1,\ldots,j$ been $iid$ with distribution 
$G_{i_1,i_2,\ldots,i_K}$ and $G_{i_1,i_2,\ldots,i_K}\sim DP\left(\alpha G_{0,i_1,i_2,\ldots,i_K}\right)$.

In particular, for $n$ data points $\left\{\bX^j_K; j=1,2,\ldots,n\right\}$, following (\ref{eq:dp_indep4}) we obtain the posterior mean as
\begin{equation}
	E[G_n|\bX^n_K]=\frac{\alpha G_{0,1:K}+\sum_{r=1}^n\delta_{\bX^r_K}}{\alpha+n},
	\label{eq:dp_mean1}
\end{equation}
which involves all the available data points $\left\{\bX^j_K; j=1,2,\ldots,n\right\}$. 
With (\ref{eq:dp_mean1}), we deal with the following form of the conditional distribution function of $[X_j|X_1,\ldots,X_{j-1}]$ for $j\geq 1$:
\begin{equation}
	\tilde \zeta_{jn}(t_1,\ldots,t_j)=\frac{E[G_n\left(X_1\leq t_1,\ldots,X_j\leq t_j\right)|\bX^n_j]}{E[G_n\left(X_1\leq t_1,\ldots,X_{j-1}\leq t_{j-1}\right)|\bX^n_{j-1}]}.
	\label{eq:dp_cond1}
\end{equation}
The marginal distribution of $X_j$ in this case that we shall consider is 
\begin{equation}
	\tilde \zeta_{jn}(t_j)=\frac{\alpha G_{0,j}(X_j\leq t_j)+\sum_{r=1}^n\delta_{X^r_j}(X^r_j\leq t_j)}{\alpha+n}
	\label{eq:marginal1}
\end{equation}

With these, we have the following result.
\begin{theorem}
\label{theorem:dp_cdf1}
For any $K\geq 2$, let $\bX^j_K$; $j\geq 1$, be stationary. Then $(X_1,\ldots,X_K)$ are mutually independent if and only if, for $j=1,\ldots,K$,
\begin{equation}
	\underset{t_1,\ldots,t_j\in\mathbb R}{\sup}~\left|\tilde \zeta_{jn}(t_1,\ldots,t_j)-\tilde \zeta_{jn}(t_j)\right|\stackrel{a.s.}{\longrightarrow}0,~n\rightarrow\infty.
	\label{eq:dp_diff_conv1}
\end{equation}
\end{theorem}
\begin{proof}
Let $(X_1,\ldots,X_K)$ be mutually independent. Then $[X_j|X_1,\ldots,X_{j-1}]=[X_j]$, for $j\geq 2$. In other words,
it holds that $P\left(X_j\leq t_j|X_1\leq t_1,\ldots,X_{j-1}\leq t_{j-1}\right)=P\left(X_j\leq t_j\right)$, for all $t_1,\ldots,t_j\in\mathbb R$, and $j\geq 2$.
Now, 
\begin{equation}
	\underset{t_j\in\mathbb R}{\sup}~\left|\tilde \zeta_{jn}(t_1,\ldots,t_j)-\tilde \zeta_{jn}(t_j)\right|
	\leq \underset{t_j\in\mathbb R}{\sup}~\left|\tilde \zeta_{jn}(t_1,\ldots,t_j)-P(X_j\leq t_j)\right|
	+\underset{t_j\in\mathbb R}{\sup}~\left|P(X_j\leq t_j)-\tilde \zeta_{jn}(t_j)\right|.
	\label{eq:dpdiff1}
\end{equation}
Let us first focus on the first term of (\ref{eq:dpdiff1}). For fixed $\alpha$, as $n\rightarrow\infty$, due to Glivenko-Cantelli theorem for stationarity,
it is easily seen that
\begin{equation}
	E[G_n\left(X_1\leq t_1,\ldots,X_{j-1}\leq t_{j-1}\right)|\bX^n_{j-1}]\stackrel{a.s.}{\longrightarrow}P\left(X_1\leq t_1,\ldots,X_{j-1}\leq t_{j-1}\right),
	\label{eq:dpdiff2}
\end{equation}
for any $t_1,\ldots,t_{j-1}\in\mathbb R$. Also, for any $t_1,\ldots,t_{j-1}\in\mathbb R$, again due to Glivenko-Cantelli theorem for stationarity,
\begin{equation}
	\underset{t_j\in\mathbb R}{\sup}~\left|E[G_n\left(X_1\leq t_1,\ldots,X_j\leq t_j\right)|\bX^n_j]-
	P\left(X_1\leq t_1,\ldots,X_j\leq t_j\right)\right|\stackrel{a.s.}{\longrightarrow}0,~\mbox{as}~n\rightarrow\infty.
	\label{eq:dpdiff3}
\end{equation}
Combining (\ref{eq:dpdiff2}) and (\ref{eq:dpdiff3}) yields
\begin{equation*}
\underset{t_j\in\mathbb R}{\sup}~\left|\frac{E[G_n\left(X_1\leq t_1,\ldots,X_j\leq t_j\right)|\bX^n_j]}{E[G_n\left(X_1\leq t_1,\ldots,X_{j-1}\leq t_{j-1}\right)|\bX^n_{j-1}]}-
	\frac{P\left(X_1\leq t_1,\ldots,X_j\leq t_j\right)}{P\left(X_1\leq t_1,\ldots,X_{j-1}\leq t_{j-1}\right)}\right|\stackrel{a.s.}{\longrightarrow}0,
	~\mbox{as}~n\rightarrow\infty,
\end{equation*}
for all $t_1,\ldots,t_{j-1}\in\mathbb R$.
That is, for all $t_1,\ldots,t_{j-1}\in\mathbb R$,
\begin{equation*}
\underset{t_j\in\mathbb R}{\sup}~\left|\tilde \zeta_{jn}(t_1,\ldots,t_j)-P\left(X_j\leq t_j|X_1\leq t_1,\ldots,X_{j-1}\leq t_{j-1}\right)\right|
\stackrel{a.s.}{\longrightarrow}0,
\end{equation*}
and since under mutual independence, $P\left(X_j\leq t_j|X_1\leq t_1,\ldots,X_{j-1}\leq t_{j-1}\right)=P\left(X_j\leq t_j\right)$,
\begin{equation*}
\underset{t_j\in\mathbb R}{\sup}~\left|\tilde \zeta_{jn}(t_1,\ldots,t_j)-P\left(X_j\leq t_j\right)\right|
	\stackrel{a.s.}{\longrightarrow}0,~\mbox{as}~n\rightarrow\infty,
\end{equation*}
for all $t_1,\ldots,t_{j-1}\in\mathbb R$, under mutual independence. More transparently, since $\tilde \zeta_{jn}(t_1,\ldots,t_j)$ is asymptotically independent of 
$t_1,\ldots,t_{j-1}$, for any $\epsilon>0$ under mutual independence, there exists $n_0(\epsilon)\geq 1$
such that for $n>n_0(\epsilon)$,
\begin{equation*}
\underset{t_j\in\mathbb R}{\sup}~\left|\tilde \zeta_{jn}(t_1,\ldots,t_j)-P\left(X_j\leq t_j\right)\right|
<\epsilon,
\end{equation*}
for all $t_1,\ldots,t_{j-1}\in\mathbb R$. That is,
	(\ref{eq:dp_diff_conv1})
\begin{equation}
\underset{t_1,\ldots,t_j\in\mathbb R}{\sup}~\left|\tilde \zeta_{jn}(t_1,\ldots,t_j)-P\left(X_j\leq t_j\right)\right|
	\stackrel{a.s.}{\longrightarrow}0,~\mbox{as}~n\rightarrow\infty.
\label{eq:dpdiff4}
\end{equation}

For the second term of (\ref{eq:dpdiff1}), note that
\begin{equation}
\underset{t_j\in\mathbb R}{\sup}~\left|P(X_j\leq t_j)-\tilde \zeta_{jn}(t_j)\right|\stackrel{a.s.}{\longrightarrow}0,
\label{eq:dpdiff5}
\end{equation}
as $n\rightarrow\infty$, due to Glivenko-Cantelli theorem for stationarity,

Combining (\ref{eq:dpdiff1}), (\ref{eq:dpdiff4}) and (\ref{eq:dpdiff5}) yields (\ref{eq:dp_diff_conv1}) under mutual independence.

Now if (\ref{eq:dp_diff_conv1}) holds for $j\geq 2$, then this clearly implies mutual independence of the random variables.
\end{proof}
\begin{remark}
\label{remark:remark_indep1}
Apart from being much more stable numerically compared to the approach of comparison between classical empirical conditional and marginal distributions, our DP-based approach 	
also allows incorporation of the dependence structure, if any, through the base measure $G_0$. This can be achieved by empirically estimating the dependence structure
from the data, and incorporating it in $G_0$. For example, if $G_0$ corresponds to Gaussian process, then its mean and the covariance structure can be estimated from
the data. This is expected to improve efficiency of inference regarding mutual independence. Note that such dependence structure can not be exploited in the
approach of comparison between classical empirical conditional and marginal distributions.
\end{remark}

For our Bayesian characterization of mutual independence, let $n_j$ denote the minimum number of observations associated with $(X_1,\ldots,X_j)$, for $j\geq 2$.
Now let $\{c_j\}_{j=1}^{\infty}$ be a non-negative decreasing sequence and
\begin{equation*}
	Y_{j,n_{j}}=\mathbb I{\left\{\underset{t_1,\ldots,t_j\in\mathbb R}{\sup}~\left|\tilde \zeta_{jn_j}(t_1,\ldots,t_j)-\tilde \zeta_{jn_j}(t_j)\right|\leq c_j\right\}}.
\end{equation*}
Let, for $j\geq 1$,
\begin{equation*}
	P\left(Y_{j,n_{j}}=1\right)=p_{j,n_{j}}.
\end{equation*}
Let the rest of the recursive Bayesian procedure be the same as in Section \ref{sec:recursive1}.
Then, using Theorem \ref{theorem:dp_cdf1}, the following theorem can be proved in almost the same way as Theorem \ref{theorem:convergence}.

\begin{theorem}
\label{theorem:convergence_indep}
Let $\bX^i;i=1,2,\ldots$, be stationary. Then 
$(X_1,X_2,\ldots)$ are mutually independent if and only if 
for all $\omega\in\mathfrak S\cap\mathfrak N^c$, where $\mathfrak N$ is some null set having probability measure zero,
for any monotonically decreasing sequence 
$\left\{c_j(\omega)\right\}_{j=1}^{\infty}$,
\begin{equation*}
	\pi\left(\mathcal N_1|y_{k,n_{k}}(\omega)\right)\rightarrow 1,
\end{equation*}
	as $k\rightarrow\infty$ and $n_{j}\rightarrow\infty$ for $k=2,3,\ldots,K$ and $K\rightarrow\infty$, 
where $\mathcal N_1$ is any neighborhood of 1 (one).
\end{theorem}

\subsection{Mutual independence in the point process setup}
\label{subsec:pp_indep}

Recall that for a Poisson point process, if for any set of disjoint regions $C_i$; $i=1,\ldots,K$, where $C_i\subset\bS$, $\bX_{C_i}$, denoting
the set of points in $C_i$, are independent, for any $K>1$. This is referred to as the complete independence property in \ctn{Daley03}. 
However, complete independence alone is not sufficient to characterize Poisson point process.
In this regard, let us consider the following assumptions.

\begin{itemize}
\item[(A1)]
Let $N(A)$, the number of points in the set $A$, be defined and finite for every bounded set $A$ in the Borel sigma-field generated by the open spheres
of $\bS$. This can be simply expressed by saying that the trajectories of $N(\cdot)$ are almost surely boundedly finite (\ctn{Daley03}). 

\item[(A2)] $P_r\left\{N\left(S_{\epsilon}(x)\right)>1\right\}=o\left(P_r\left\{N\left(S_{\epsilon}(x)\right)>1\right\}\right)$, as $\epsilon\rightarrow 0$.
Here $S_{\epsilon}(x)$ denotes the open sphere with radius $\epsilon$ and center $x$.
This property is called orderliness.	
\end{itemize}

With these, the Poisson process can be characterized as follows.
\begin{theorem}[\ctn{Daley03}]
\label{theorem:poisson1}
Let $N(\cdot)$ be almost surely boundedly finite and without fixed atoms. Then $N(\cdot)$ is a Poisson process if and only if it is orderly
and has the complete independence property.
\end{theorem}

We also note the following lemma.
\begin{lemma}[\ctn{Daley03}]
\label{lemma:poisson1}
A point $x_0$ is an atom of the parameter measure $\Lambda$ if and only if it is a fixed atom of the process.
\end{lemma}

\begin{corollary}
\label{cor:poisson1}
Theorem \ref{theorem:poisson1} and Lemma \ref{lemma:poisson1} together imply that if $\Lambda$ corresponds to a continuous distribution, then
(A1)-(A2) along with complete independence characterize Poisson process. 
\end{corollary}
We now characterize Poisson process in a recursive Bayesian framework using our Bayesian characterization of mutual independence
assuming (A1)--(A2) and non-atomicity of the process. 
In all our examples, we consider $\Lambda$ to be associated with continuous distributions, hence non-atomic; (A1)--(A2) also hold in all our  
simulation studies.

Assume that $\bX_{C_i}$ are locally stationary and let $D_{C_i}$ denote the set of minimum inter-point distances associated with $\bX_{C_i}$.
As before, for $r=1,2,\ldots$, let $W_r$ and $W_{dr}$ be the observation window and the space of inter-point distances corresponding to $W_r$ at the $r$-th stage,
where $|W_r|\rightarrow\infty$ as $r\rightarrow\infty$.
Let us also replace $n_j$ and $K$ with $n_{jr}$ and $K_r$, respectively, as before. 

Now let $\{c_j\}_{j=1}^{\infty}$ be a non-negative decreasing sequence and
\begin{equation*}
Y_{j,n_{jr}}=\mathbb I{\left\{\underset{t_1,\ldots,t_j\in\mathbb R}{\sup}~\left|\tilde \zeta_{jn_{jr}}(t_1,\ldots,t_j)-\tilde \zeta_{jn_{jr}}(t_j)\right|\leq c_j\right\}},
\end{equation*}
and, for $j\geq 1$,
\begin{equation*}
	P\left(Y_{j,n_{jr}}=1\right)=p_{j,n_{jr}}.
\end{equation*}
Then we have the following result for point processes corresponding to Theorem \ref{theorem:convergence_indep}.

\begin{theorem}
\label{theorem:convergence_indep_pp}
Let $\bX$ be a point process in $\bS$. Assume that for the disjoint regions $C_i\subset \bS$; $i=1,\ldots,K_r$, $\bX_{C_i}$ are locally stationary.  
Then $\left(D_{C_1},\ldots,D_{C_{K_r}}\right)$ are mutually independent if and only if 
for all $\omega\in\mathfrak S\cap\mathfrak N^c$, where $\mathfrak N$ is some null set having probability measure zero,
for any monotonically decreasing sequence 
$\left\{c_j(\omega)\right\}_{j=1}^{\infty}$, and any set of disjoint regions $C_i$; $i=1,\ldots,K_r$, where $C_i\subset\bS$,
\begin{equation}
	\pi\left(\mathcal N_1|y_{k,n_{kr}}(\omega)\right)\rightarrow 1,
	\label{eq:convergence_indep_pp}
\end{equation}
	as $k\rightarrow\infty$ and $n_{kr}\rightarrow\infty$ for $k=2,3,\ldots,K_r$ and $K_r\rightarrow\infty$ as $r\rightarrow\infty$, 
where $\mathcal N_1$ is any neighborhood of 1 (one).
\end{theorem}
\begin{proof}
Using Theorem \ref{theorem:dp_cdf1}, the proof follows in almost the same way as that of Theorem \ref{theorem:convergence}.
\end{proof}

\begin{theorem}
\label{theorem:convergence_indep_pp2}
Consider any point process $\bX\in\bS$.
Assume that the $\sigma$-algebra for $\bS$ is separable and generated by the mutually disjoint sets $\left\{C_i;i\geq 1\right\}$,
	and that $\bX_{C_i}$ are locally stationary. Then, provided that (A1)--(A2) hold and the process is non-atomic, 
$\bX$ is a Poisson point process if and only if (\ref{eq:convergence_indep_pp}) holds. 
\end{theorem}
\begin{proof}
By Theorem \ref{theorem:convergence_indep_pp}, $\left(D_{C_1},\ldots,D_{C_{K_r}}\right)$ are mutually independent if and only if (\ref{eq:convergence_indep_pp}) holds.
Since the mutually disjoint sets $\left\{C_i;i\geq 1\right\}$ generates the $\sigma$-field for $\bS$, it follows that any set of mutually disjoint
sets $\left\{B_1,\ldots,B_{\ell}\right\}$ in the $\sigma$-field for $\bS$, for any $\ell>1$, $\left(D_{B_1},\ldots,D_{B_{\ell}}\right)$, are mutually independent. 

Also, it is easy to see that $\left(D_{B_1},\ldots,D_{B_{\ell}}\right)$ are mutually independent if and only if
$\left(\bX_{B_1},\ldots,\bX_{B_{\ell}}\right)$ are mutually independent.

Hence, by the hypothesis of the theorem it follows that $\bX$ is a Poisson point process if and only if (\ref{eq:convergence_indep_pp}) holds. 
\end{proof}

\subsection{Computational strategy for mutual independence assessment} 
\label{subsec:compute_diff}

Note that for relatively large $j$, it may not be feasible to directly compute 
$\underset{t_1,\ldots,t_j\in\mathbb R}{\sup}~\left|\tilde \zeta_{jn_{j}}(t_1,\ldots,t_j)-\tilde \zeta_{jn_{j}}(t_j)\right|$.
Hence we consider the following strategy.
For $j=2$, let $\tilde t_1,\tilde t_2$ be the maximizers of $\left|\tilde \zeta_{jn_{j}}(t_1,\ldots,t_j)-\tilde \zeta_{jn_{j}}(t_j)\right|$, and for $j\geq 3$,
let
\begin{align}
	&\underset{t_1,\ldots,t_j\in\mathbb R}{\sup}~\left|\tilde \zeta_{jn_{j}}(t_1,\ldots,t_j)-\tilde \zeta_{jn_{j}}(t_j)\right|\notag\\
	&= \underset{t_j\in\mathbb R}{\sup}~\left|\frac{E[G_n\left(X_1\leq \tilde t_1,\ldots,X_{j-1}\leq\tilde t_{j-1},X_j\leq t_j\right)|\bX^n_j]}
	{E[G_n\left(X_1\leq \tilde t_1,\ldots,X_{j-1}\leq \tilde t_{j-1}\right)|\bX^n_{j-1}]}-
	\frac{\alpha G_{0,j}(X_j\leq t_j)+\sum_{r=1}^n\delta_{X^r_j}(X^r_j\leq t_j)}{\alpha+n}\right|,
\end{align}
where $\tilde t_3,\ldots,\tilde t_{j-1}$ are the maximizers of 
$\left|\tilde \zeta_{{j-1}n_{j-1}}(t_1,\ldots,t_{j-1})-\tilde \zeta_{{j-1}n_{j-1}}(t_{j-1})\right|$, for $j\geq 3$.

\subsection{Example 1: Detection of HPP and IHPP and their properties}
\label{subsec:example_pp1}

We generate a HPP with intensity $\lambda=1$ on a window of the form $[0,100]\times[0,100]$, using the R package ``spatstat" (\ctn{Baddeley05}),
and obtain $9949$ points in this exercise. 
We also simulate an IHPP using the spatstat package with $\lambda(x,y)=100(x+y)$ on $[0,5]\times[0,5]$, generating $12447$ observations. 
The plots of the point patterns are provided in Figure \ref{fig:pp1_plots}.
Observe that while the HPP pattern in panel (a) is reasonably uniform on the observed window, the IHPP pattern in panel (b) shows sparsity in the bottom
left corner and density in the top right corner of the observation window. 
\begin{figure}
\centering
\subfigure [Homogeneous Poisson point pattern.]{ \label{fig:hpp}
\includegraphics[width=5.5cm,height=5.5cm]{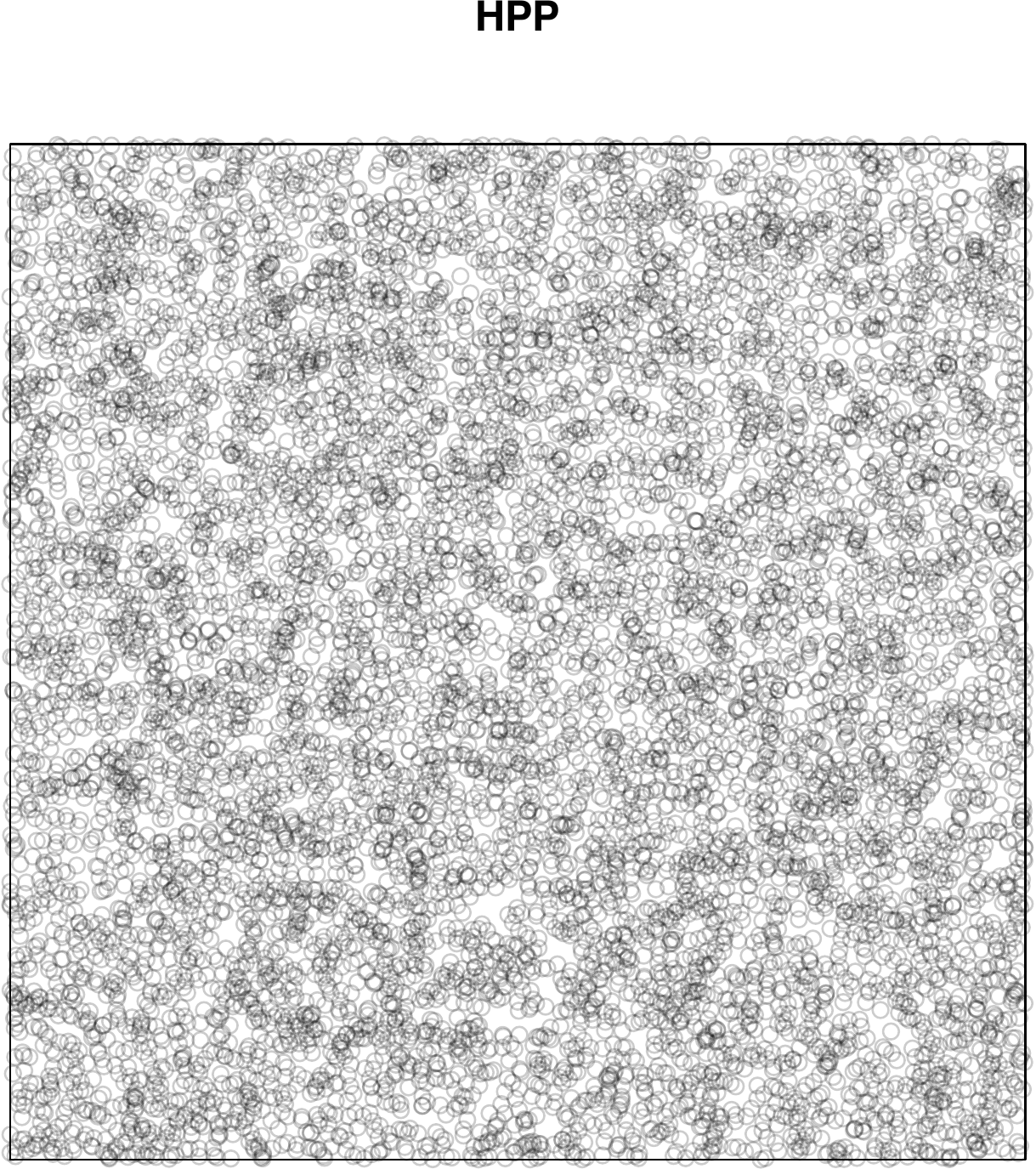}}
\hspace{2mm}
\subfigure [Inhomogeneous Poisson point pattern.]{ \label{fig:ihpp}
\includegraphics[width=5.5cm,height=5.5cm]{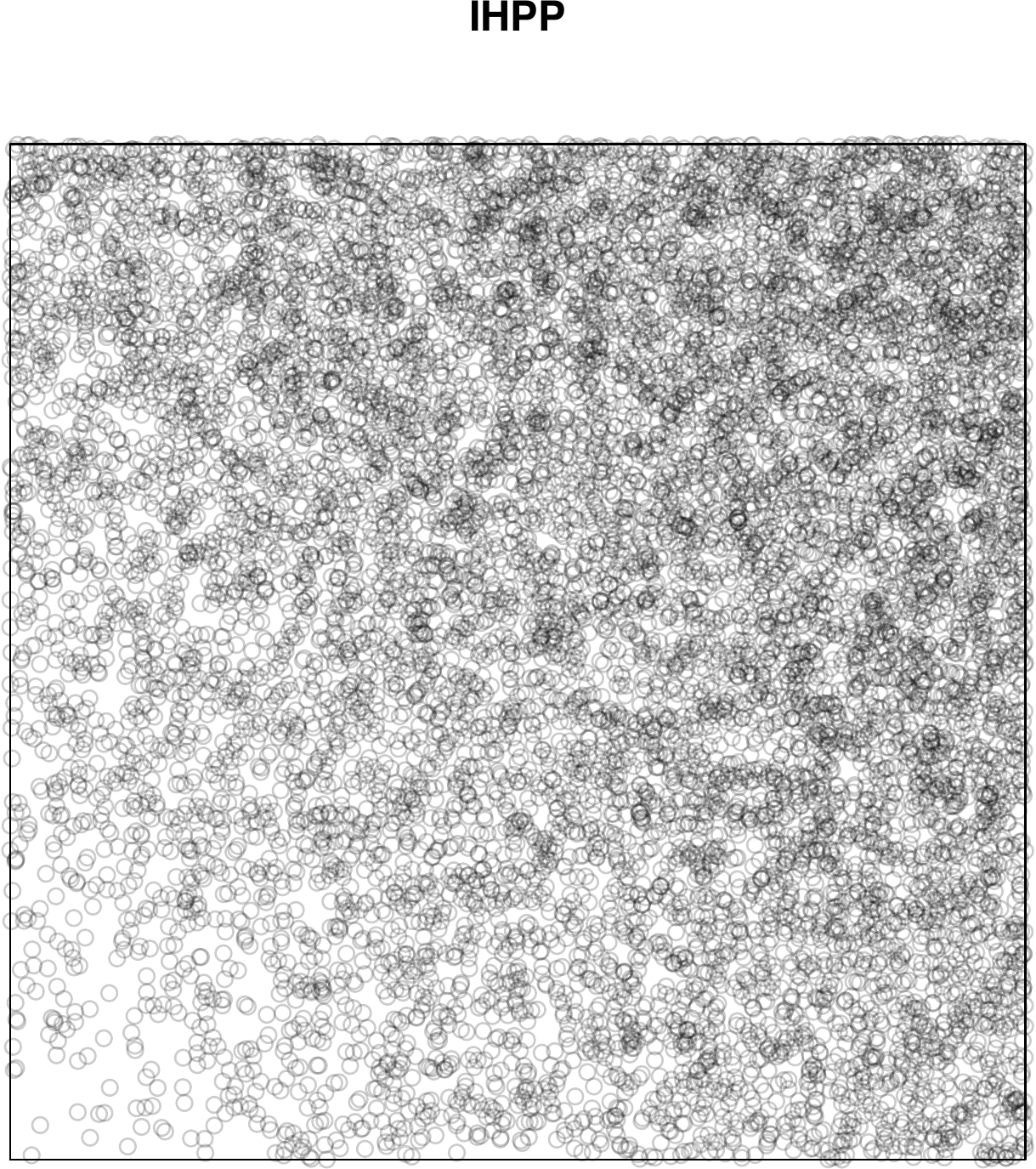}}
\caption{Homogeneous and inhomogeneous Poisson point processes.} 
\label{fig:pp1_plots}
\end{figure}
Our goal is to identify the true point processes that generated the data, pretending that they are unknown and that only the data are observed.

\subsubsection{Homogeneity detection}
\label{subsubsec:homogeneity}
Let us first concentrate on the HPP data.
With $K=1000$ clusters, we use bound (\ref{eq:ar1_bound3}) 
and obtain $\hat C_1=0.25$ as the minimum value of $\hat C_1$ that led to convergence of our recursive Bayesian algorithm to $1$.
The result is depicted in panel (a) of Figure \ref{fig:pp1}. Panel (b) of Figure \ref{fig:pp1} is the simultaneous critical envelope 
associated with classical test of HPP, prepared using spatstat with $1000$ simulations of CSR. Here $r$ stands for the distance argument, 
and $\hat G_{obs}(r)$, $\hat G_{theo}(r)$, $\hat G_{lo}(r)$ and $\hat G_{hi}(r)$ stand for the observed empirical distribution function for the distances
with Kaplan-Meier edge correction, the theoretical distribution function under CSR, the lower critical boundary and the upper critical boundary for the distribution
functions under CSR, respectively. Here the significance level of simultaneous Monte Carlo test is given by $0.000999$. 
Since the observed distribution function fall well within the lower and upper critical boundaries, the result is in agreement with our Bayesian result and indeed, the truth.

We now analyse the point pattern obtained from the IHPP.
Panel (c) of Figure \ref{fig:pp1}
shows the result of our Bayesian analysis with $K=1000$ clusters and $\hat C_1=0.25$. Divergence to zero, that is, inhomogeneity is clearly indicated. However,
this does not validate or invalidate Poisson process. To validate Poisson process, we need to create a characterization of mutual independence between the points
contained in the $K$ clusters. Panel (d) of Figure \ref{fig:pp1} is similar to panel (b) except that the observed distribution function in this case now corresponds
to IHPP. Note that the observed distribution function $\hat G_{obs}(r)$ falls almost entirely within the limits $\hat G_{lo}(r)$ and $\hat G_{hi}(r)$, which makes
it considerably difficult to distinguish this IHPP from HPP. The advantage of our Bayesian method depicted in panel (c) is clearly pronounced over this classical method
in this regard.

\begin{figure}
\centering
\subfigure [HPP detection with Bayesian method.]{ \label{fig:hpp_bayesian}
\includegraphics[width=5.5cm,height=5.5cm]{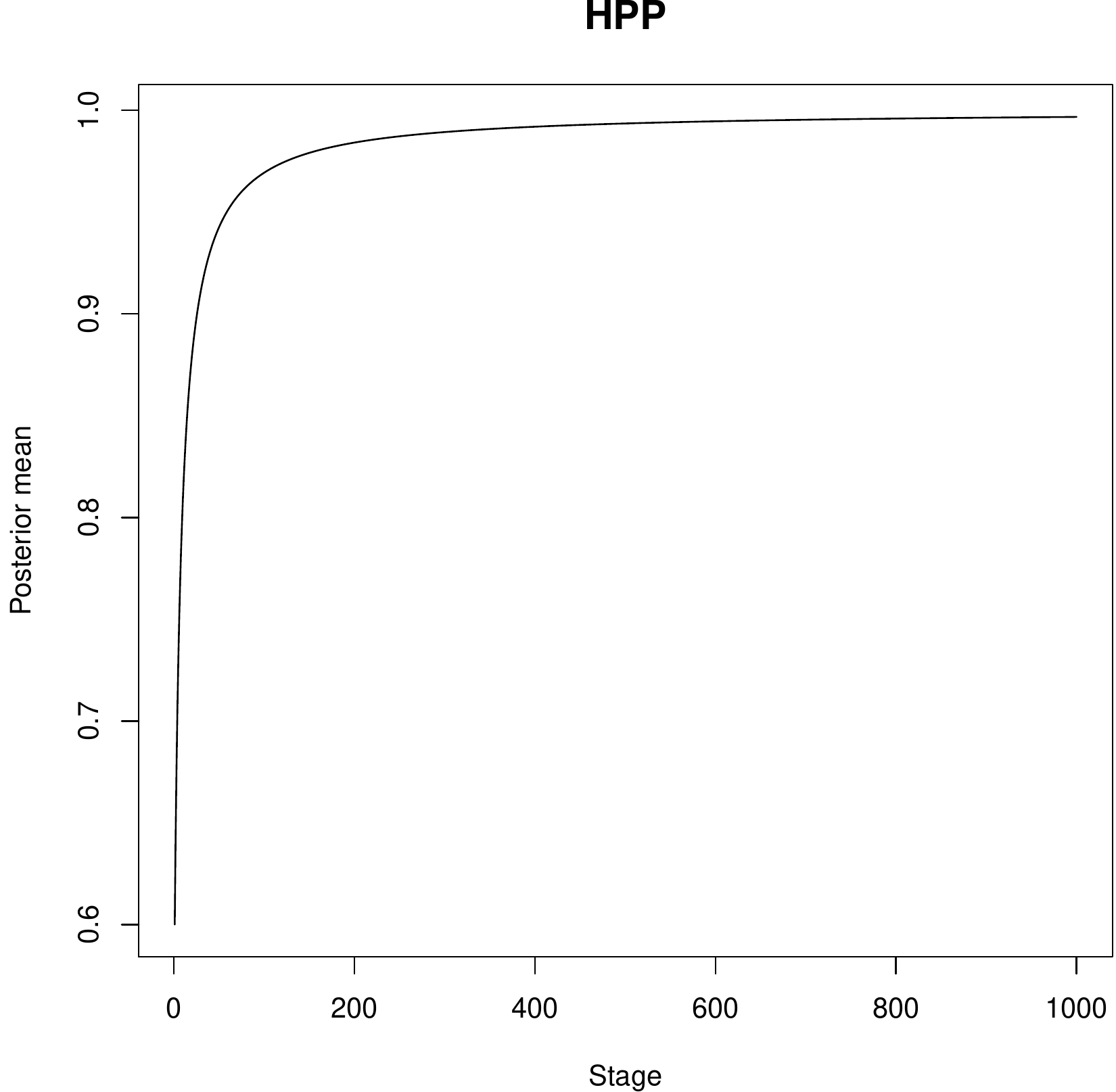}}
\hspace{2mm}
\subfigure [HPP detection with classical method.]{ \label{fig:hpp_classical}
\includegraphics[width=5.5cm,height=5.5cm]{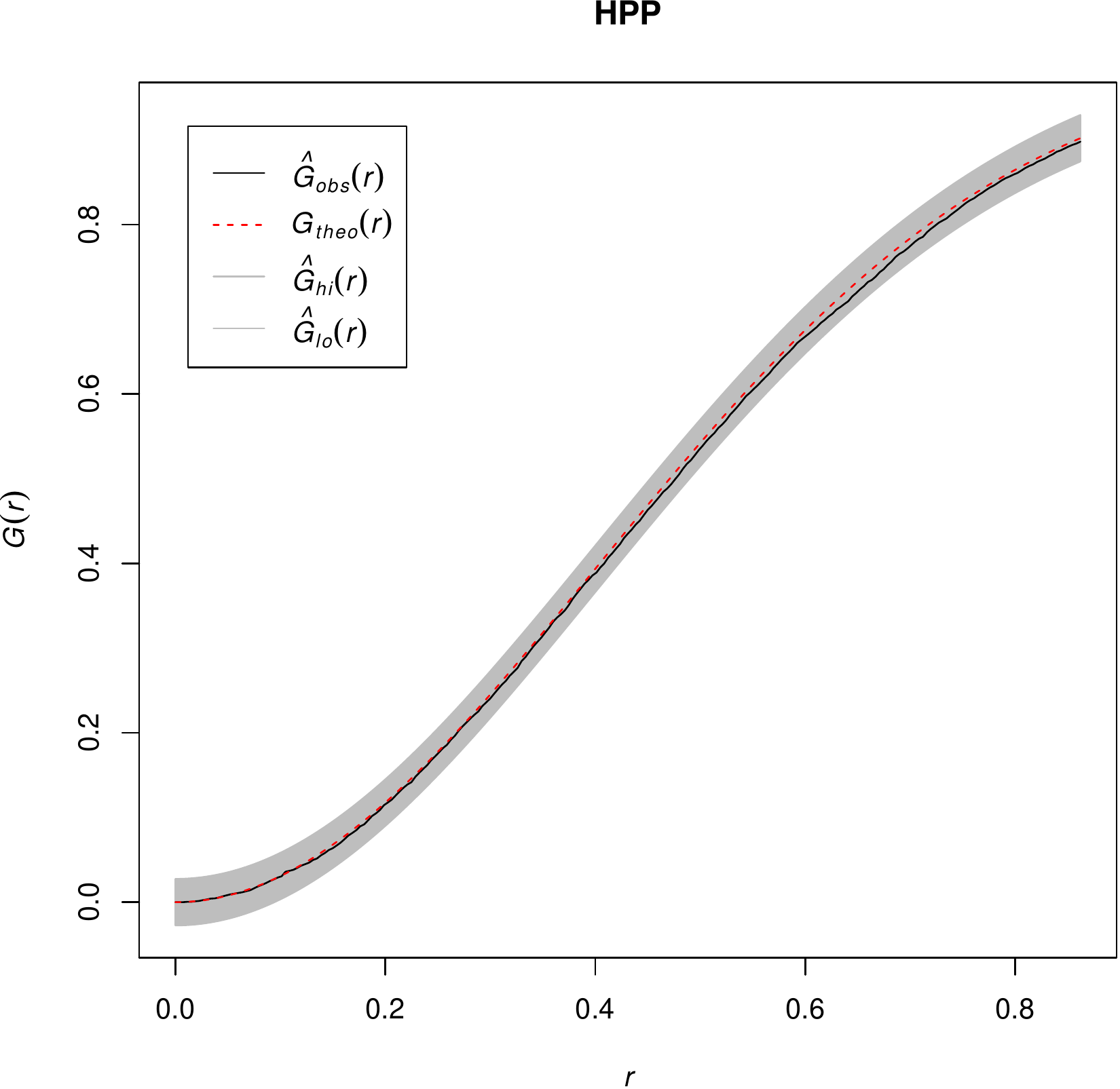}}\\
\vspace{2mm}
\subfigure [IHPP detection with Bayesian method.]{ \label{fig:ihpp_bayesian}
\includegraphics[width=5.5cm,height=5.5cm]{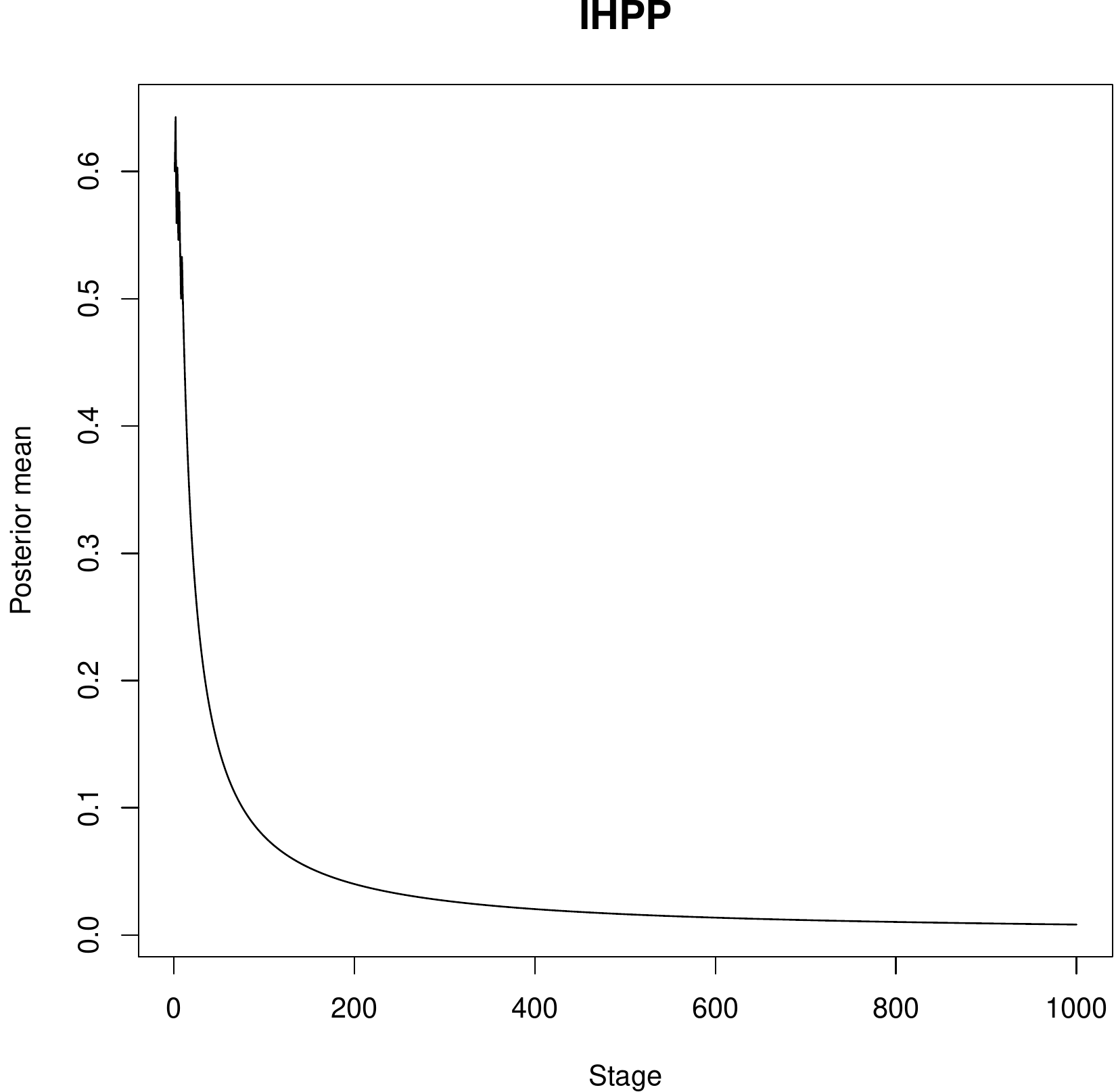}}
\hspace{2mm}
\subfigure [IHPP detection with classical method.]{ \label{fig:ihpp_classical}
\includegraphics[width=5.5cm,height=5.5cm]{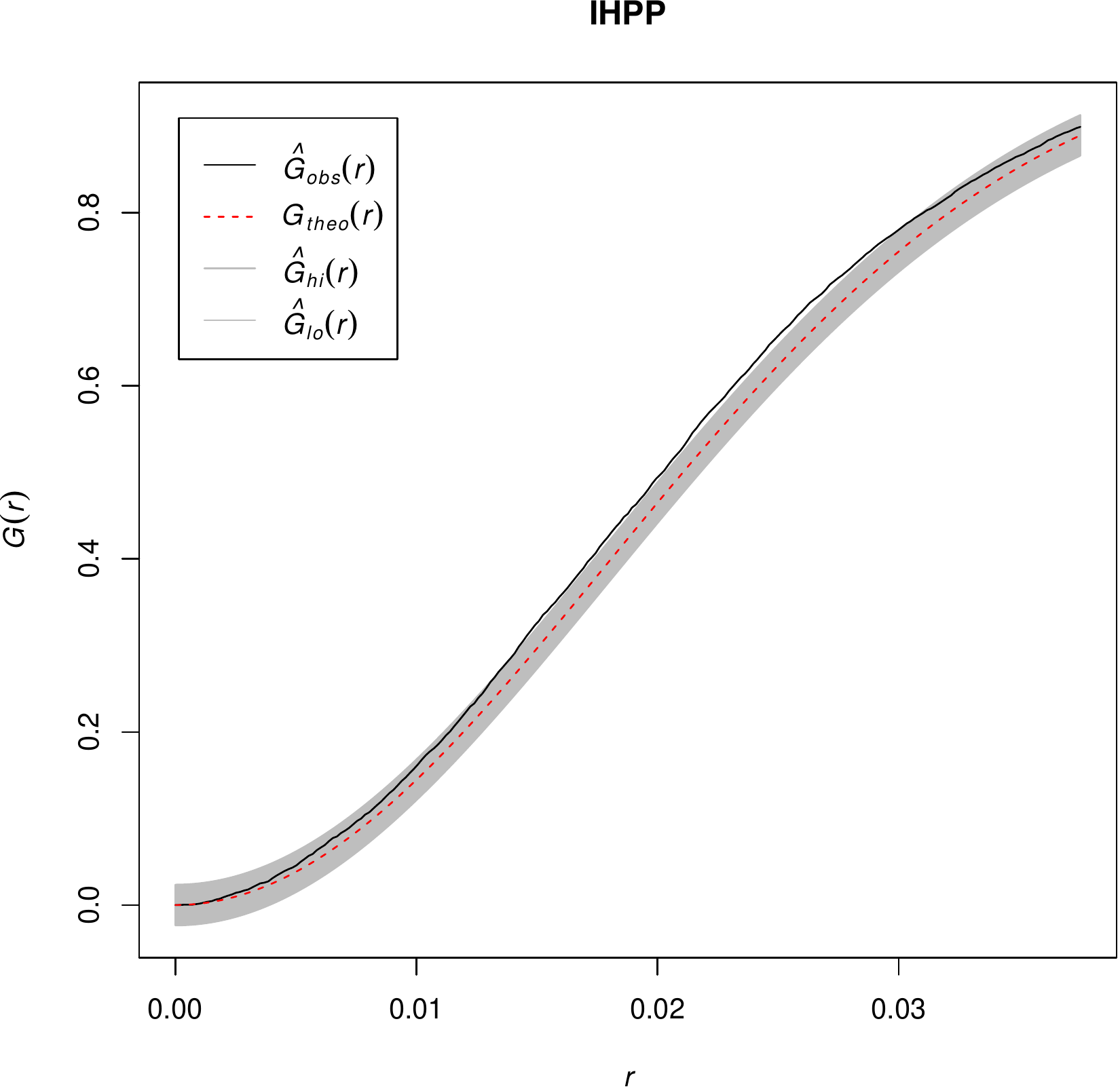}}
\caption{Detection of CSR with our Bayesian method and traditional classical method.} 
\label{fig:pp1}
\end{figure}

\subsubsection{Stationarity detection}
\label{subsubsec:stationarity}
The traditional tests of CSR tests for HPP only. But inhomogeneity neither rejects the Poisson assumption, nor either of stationarity and nonstationarity.
In this regard, we first address the question of stationarity and nonstationarity with our Bayesian method in our current examples of HPP and IHPP. 
Recall that for point processes, we regard the minimum distances $d_i$; $i=1,\ldots,n$, as the spatial data, along with their corresponding locations.
Indeed, with this, we obtain the correct results with $K=1000$ clusters, bound (\ref{eq:ar1_bound3}) with $\hat C_1=0.06$, the minimum value for which convergence to $1$
is obtained under the HPP example. The results presented in Figure \ref{fig:pp1_stationarity}, correctly identifies HPP and IHPP as stationary and nonstationary,
respectively. Larger values of $\hat C_1$, such as $\hat C_1=0.1$ led to the same result.

\begin{figure}
\centering
\subfigure [Stationary point process (HPP).]{ \label{fig:hpp_bayesian_stationary}
\includegraphics[width=5.5cm,height=5.5cm]{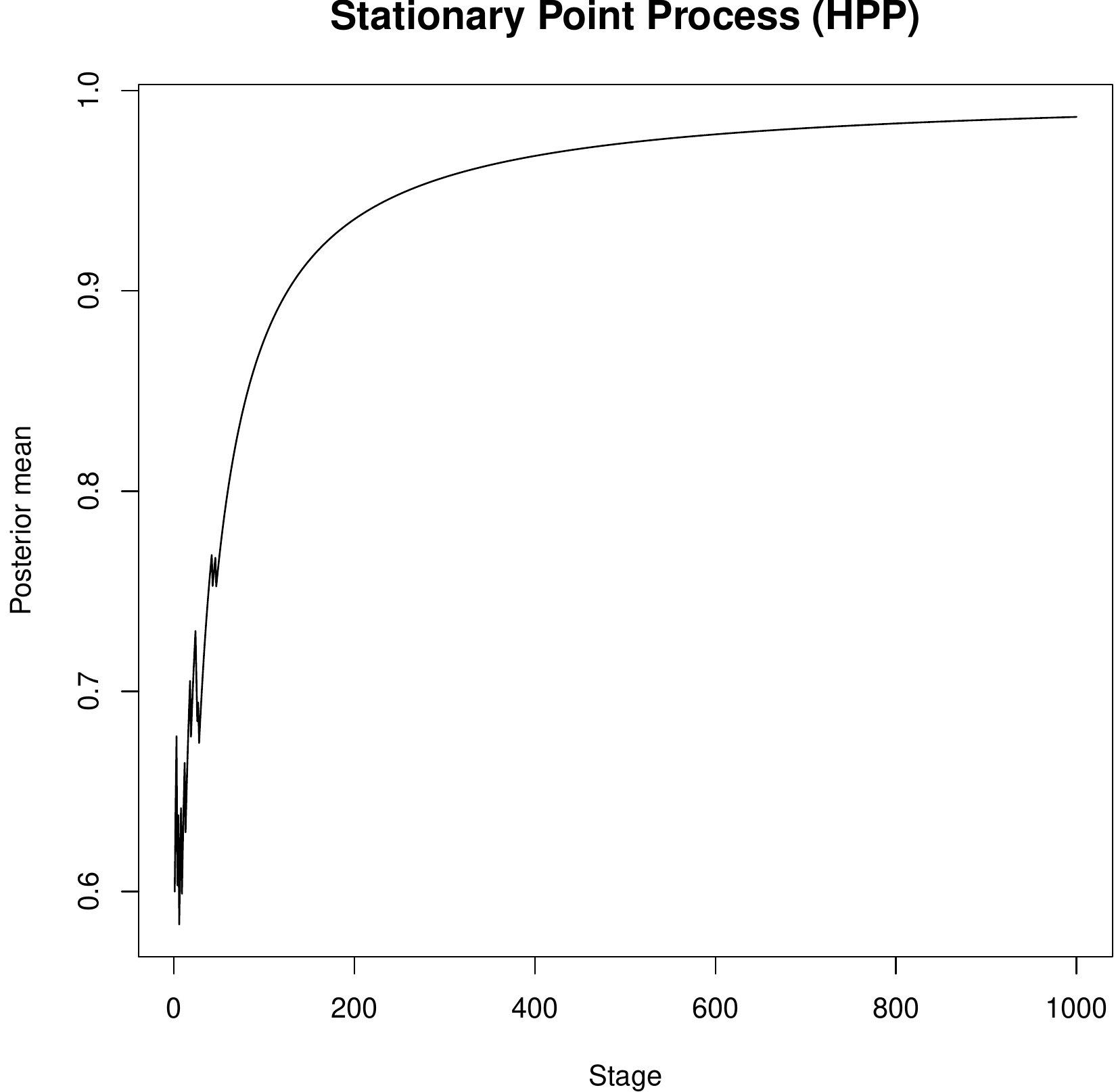}}
\hspace{2mm}
\subfigure [Nonstationary point process (IHPP).]{ \label{fig:ihpp_bayesian_nonstationary}
\includegraphics[width=5.5cm,height=5.5cm]{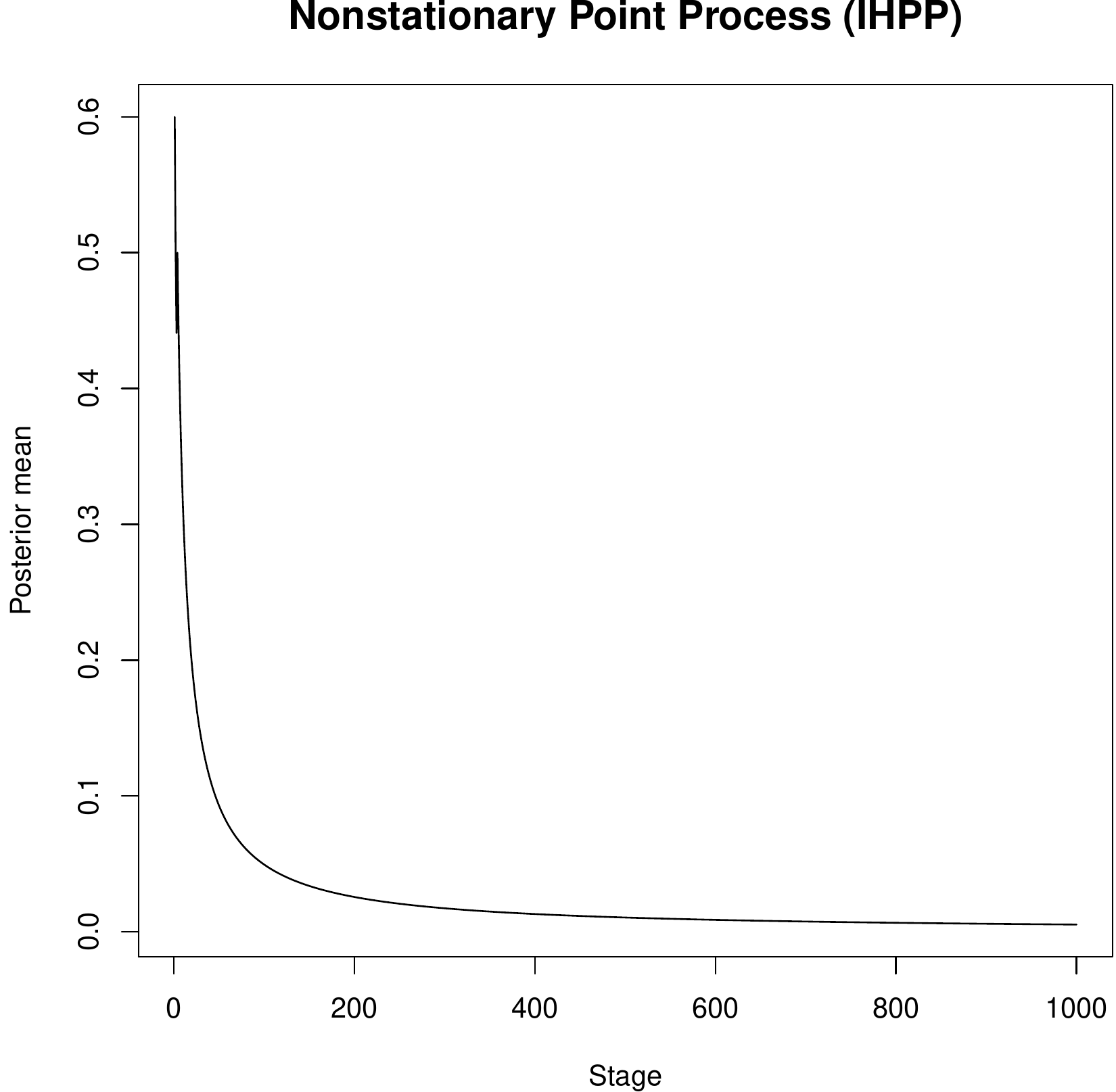}}
\caption{Detection of stationarity and nonstationarity of point processes (here HPP and IHPP) with our Bayesian method.} 
\label{fig:pp1_stationarity}
\end{figure}

\subsubsection{Validation of Poisson assumption}
\label{subsubsec:poisson}
We finally examine, with our recursive Bayesian method for characterizing mutual independence, if the two point patterns that we
generated can be safely assumed to be Poisson point patterns. Note that Poisson point process is equivalent to mutual independence of the points 
in disjoint subsets of $W$. In this regard, for $i=1,\ldots,K$, let $\bX_{C_i}$ denote the points in cluster $C_i$. If $\bX_{C_i}$ are mutually independent
for all possible clusters $C_i$ and $K$, then $\bX$ can be regarded as Poisson point process. For practical purposes, we restrict attention to a single set
of clusters $C_1,\ldots,C_K$. For numerical stability of the computations, we set $K=50$, so that in most cases we investigate mutual independence among $K=50$ variables,
where each variable is considered to take values in one and only one of the clusters. 
We set the strength parameter $\alpha$ of the Dirichlet process to $1$,
which is quite standard, and use the `emcdf' function of the `Emcdf' package in R to parallelise the computations of the joint empirical distribution functions
required for our Bayesian method. Here the joint distribution functions are those of the log-distances associated with the clusters. For the base distribution
$G_0$ of the Dirichlet process, we considered the multivariate normal distribution with mean vector and covariance matrices obtained empirically from the log-distances
associated with $\bX_{C_i}$'s. Specifically, for $K$ dimensions, $G_0$ is a $K$-variate normal distribution with mean vector being the $K$-component vector
obtained by taking the means of the log-distances in $\bX_{C_i}$; $i=1,\ldots,K$ and the covariance matrix being the empirical covariance obtained from
the log-distances in the $K$ clusters. The lower-dimensional distributions are then simply the marginalized versions of the higher-dimensional cases.

The entire exercise beginning from clustering the observed point pattern to yielding the maximum absolute differences between the conditional distribution functions
and the marginal distribution functions, takes about $20$ minutes in a 4-core laptop. The results of our Bayesian analyses 
with the bound (\ref{eq:ar1_bound3}) and $\hat C_1=0.5$, the minimum value for convergence in the HPP case,  are provided in Figure \ref{fig:pp1_independence}.
Indeed, both the panels indicate convergence, and hence independence. Hence, both the point processes can be safely assumed to be Poisson point processes.
\begin{figure}
\centering
\subfigure [Independence (HPP).]{ \label{fig:hpp_bayesian_independence}
\includegraphics[width=5.5cm,height=5.5cm]{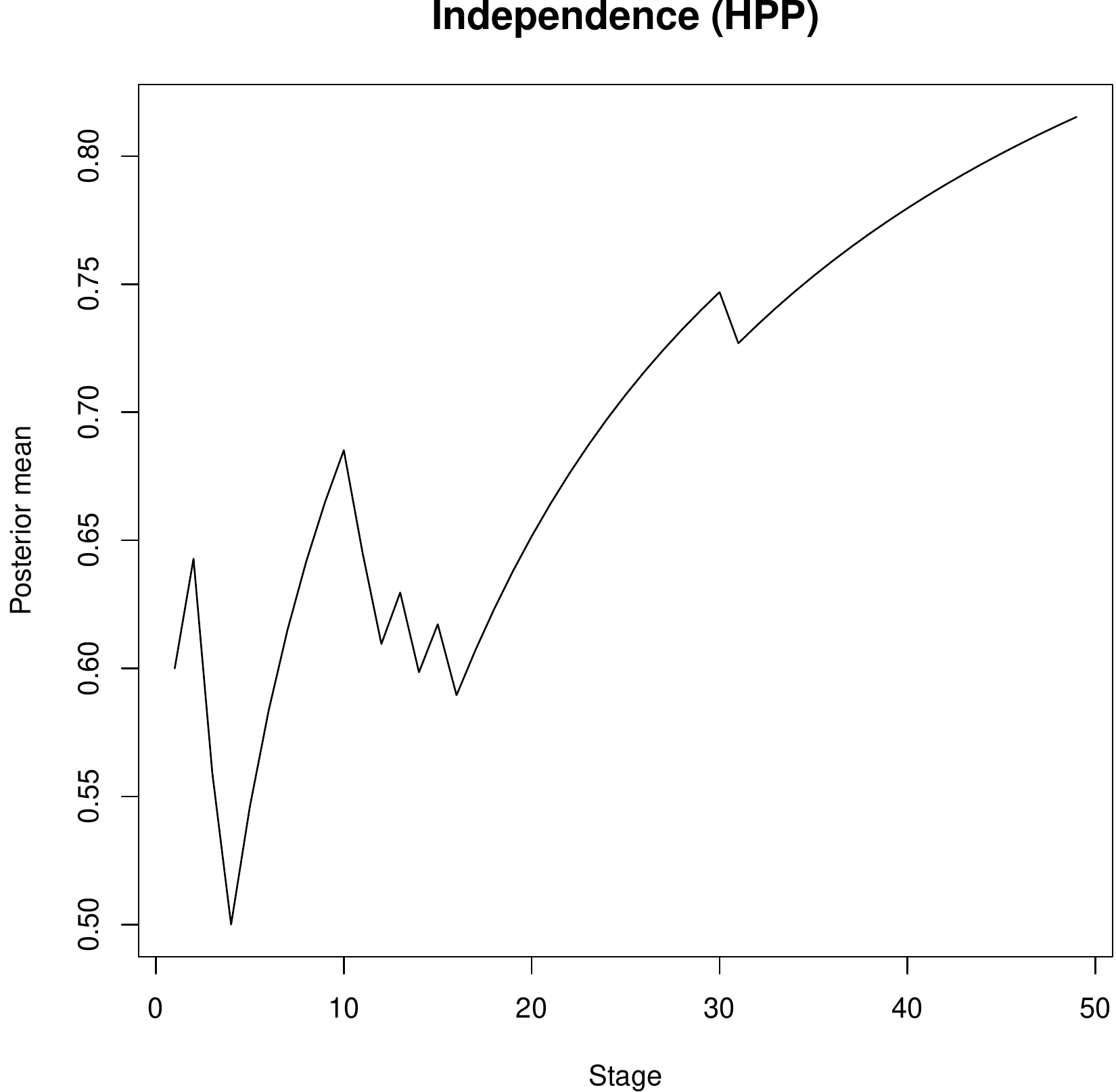}}
\hspace{2mm}
\subfigure [Independence (IHPP).]{ \label{fig:ihpp_bayesian_independence}
\includegraphics[width=5.5cm,height=5.5cm]{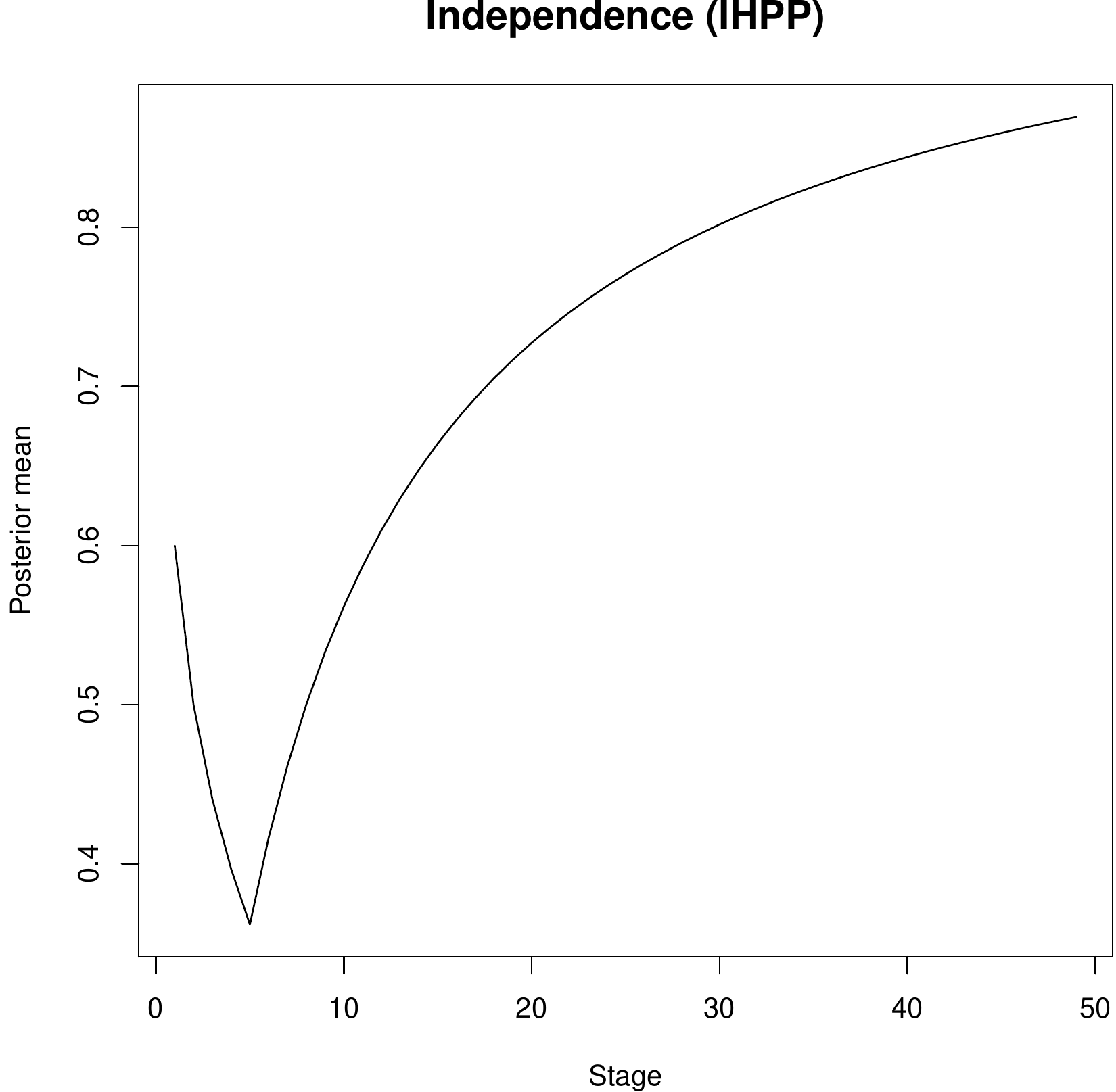}}
\caption{Detection of independence in point patterns (here HPP and IHPP) with our Bayesian method, suggesting that both the point processes are Poisson
point processes.} 
\label{fig:pp1_independence}
\end{figure}

\subsection{Example 2: Homogeneous log-Gaussian Cox process}
\label{subsec:lgcp1}
We now consider analyses of simulated data obtained from log-Gaussian Cox process. $\bX$ is a Cox process if conditional on a non-negative process 
$\left\{\Lambda(u):u\in\bS\right\}$, $\bX$ is a Poisson process with intensity function $\Lambda$ (see, for example, \ctn{Daley03}), and $\bX$
is a log-Gaussian Cox process if $\log\Lambda$ is a Gaussian process. In this example, let us consider a log-Gaussian Cox process with 
mean function $E\left[\log\Lambda(u)\right]=\mu(u)=3$ for all $u$, and exponential covariance function given by 
$Cov\left(\log\Lambda(u),\log\Lambda(v)\right)=\sigma^2\times\exp\left(a\left\|u-v\right\|\right)$, where $\|\cdot\|$ denotes Euclidean distance, $\sigma^2=0.2$ and
$a=10$. This is a stationary non-Poisson point process, and homogeneous in the sense that the marginalized intensity $E\left[\log\Lambda(u)\right]$, integrating
out the random function $\Lambda$, is constant. 

We choose $W=[0,15]\times[0,20]$ and obtain $6553$ observations from this point process using spatstat, which are displayed in Figure \ref{fig:pp11_lgcp}. 

\begin{figure}
\centering
\includegraphics[width=5.5cm,height=5.5cm]{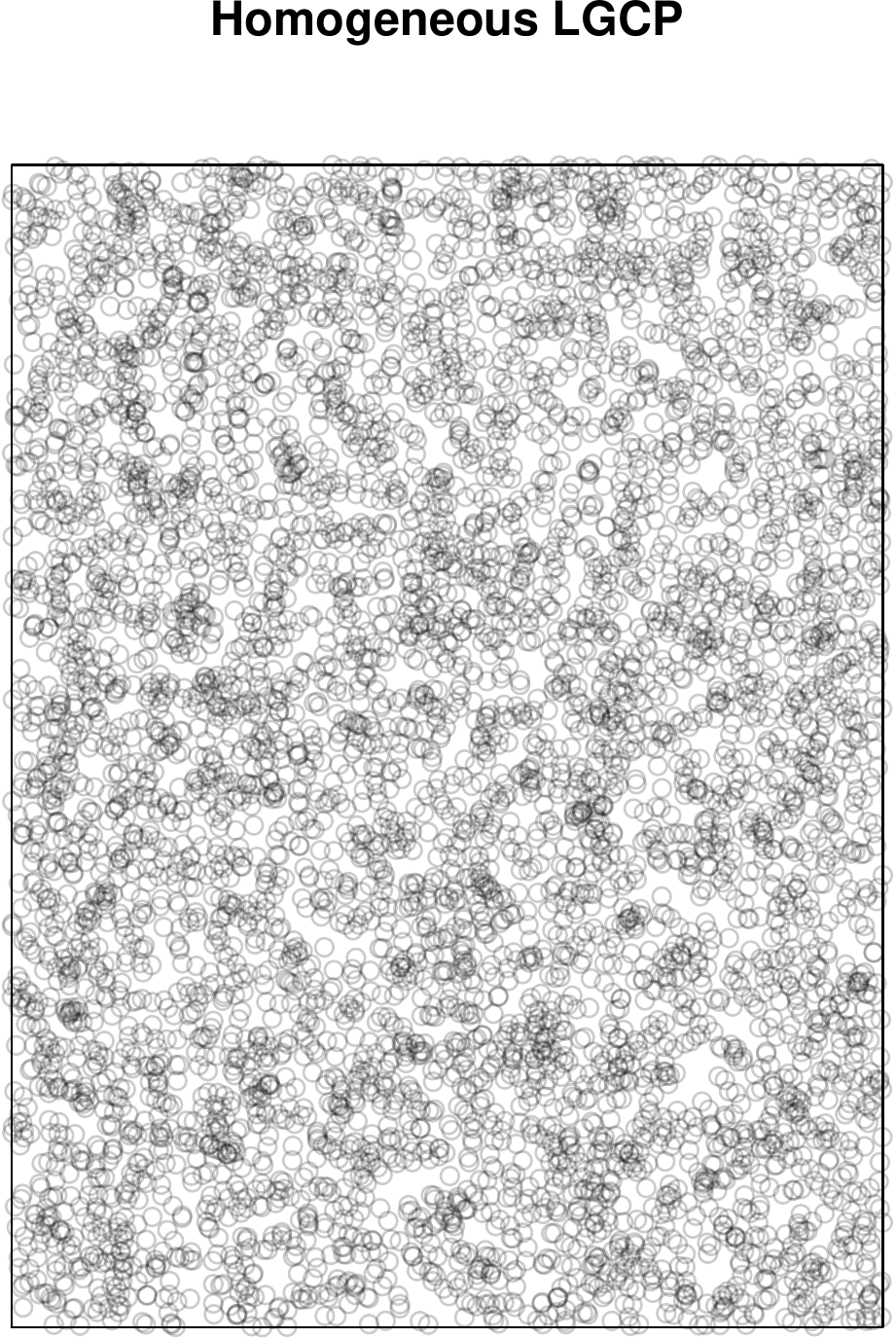}
\caption{Homogeneous LGCP.} 
\label{fig:pp11_lgcp}
\end{figure}

We consider $K=800$ and algorithm (\ref{eq:ar1_bound3}) with $\hat C_1=0.24$ for our Bayesian method.
Figure \ref{fig:pp11} compares our Bayesian method with the classical method regarding CSR detection. Observe that the Bayesian method correctly identifies
that the point process is not CSR, while the classical method fails to correctly recognize the process.
\begin{figure}
\centering
\subfigure [HPP detection with Bayesian method for LGCP.]{ \label{fig:hpp_bayesian11}
\includegraphics[width=5.5cm,height=5.5cm]{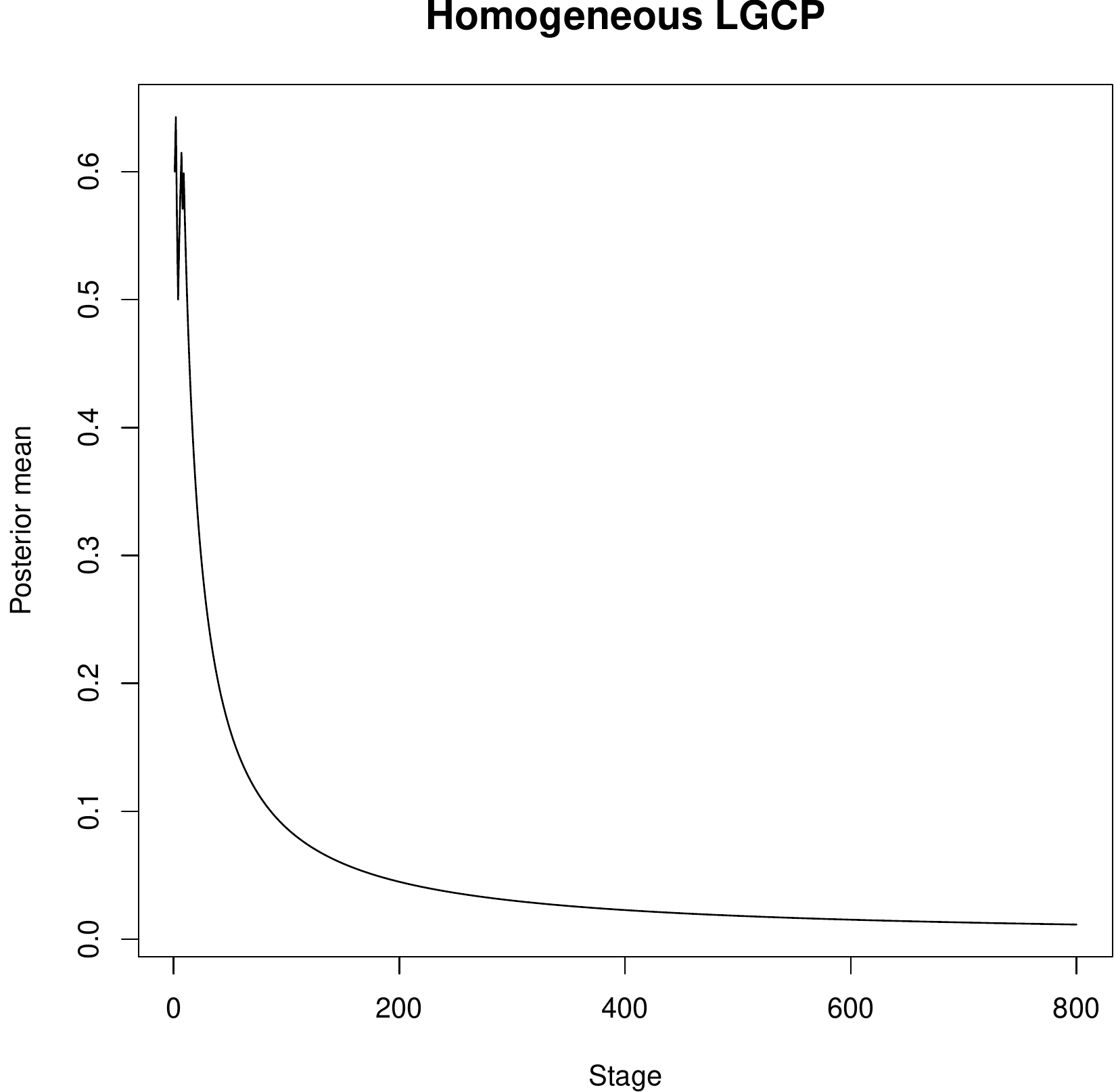}}
\hspace{2mm}
\subfigure [HPP detection with classical method for LGCP.]{ \label{fig:hpp_classical11}
\includegraphics[width=5.5cm,height=5.5cm]{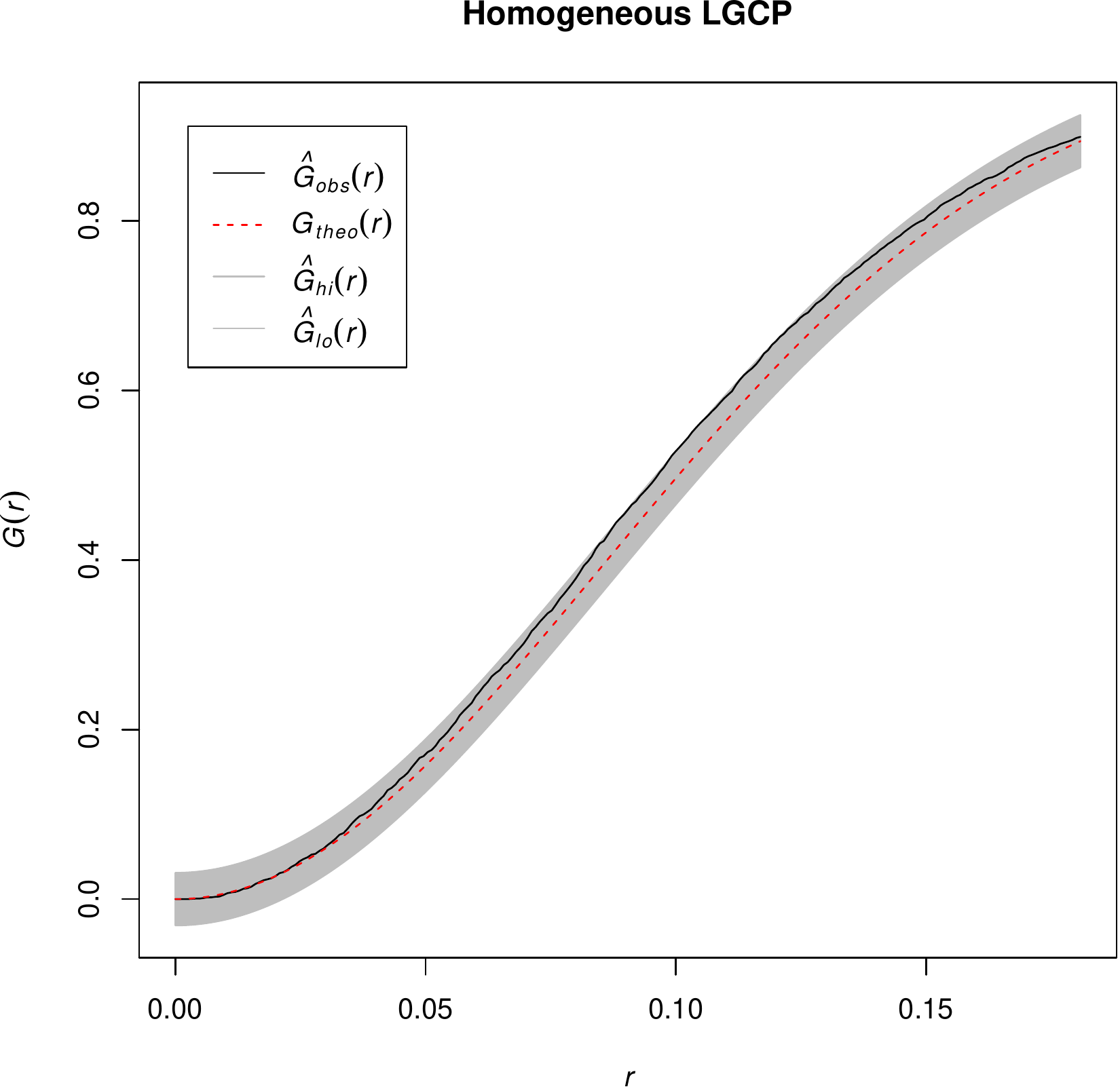}}
\caption{Detection of CSR with our Bayesian method and traditional classical method for LGCP. The Bayesian method correctly
identifies that the underlying point process is not CSR, but the classical method falsely indicates CSR.} 
\label{fig:pp11}
\end{figure}

For addressing stationarity, we set $K=800$ and $\hat C_1=0.15$. Panel (a) of Figure \ref{fig:pp11_stationarity_indep} shows that stationarity is clearly
indicated by our Bayesian approach.

For testing if the underlying point process is Poisson process, we test independence as before, among $K=70$ random variables $\bX_{C_i}$; $i=1,\ldots,K$.
With $\hat C_1=0.5$, panel (b) of Figure \ref{fig:pp11_stationarity_indep} indicates independence, validating the Poisson assumption, given $\Lambda$, as mentioned above.
\begin{figure}
\centering
\subfigure [Stationary LGCP.]{ \label{fig:lgcp1_bayesian_stationary}
\includegraphics[width=5.5cm,height=5.5cm]{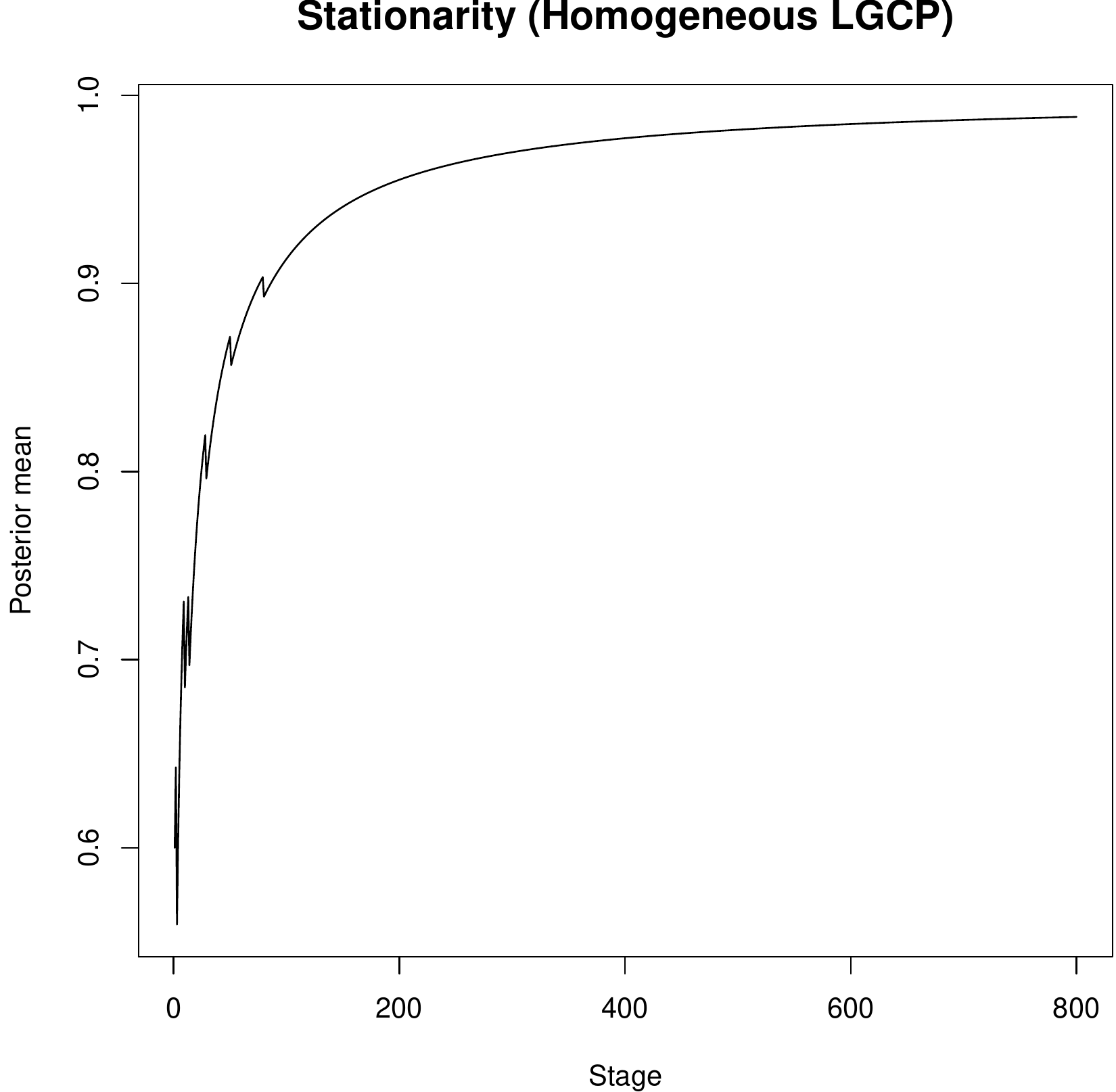}}
\hspace{2mm}
\subfigure [Dependent point process (LGCP).]{ \label{fig:lgcp1_bayesian_dependent}
\includegraphics[width=5.5cm,height=5.5cm]{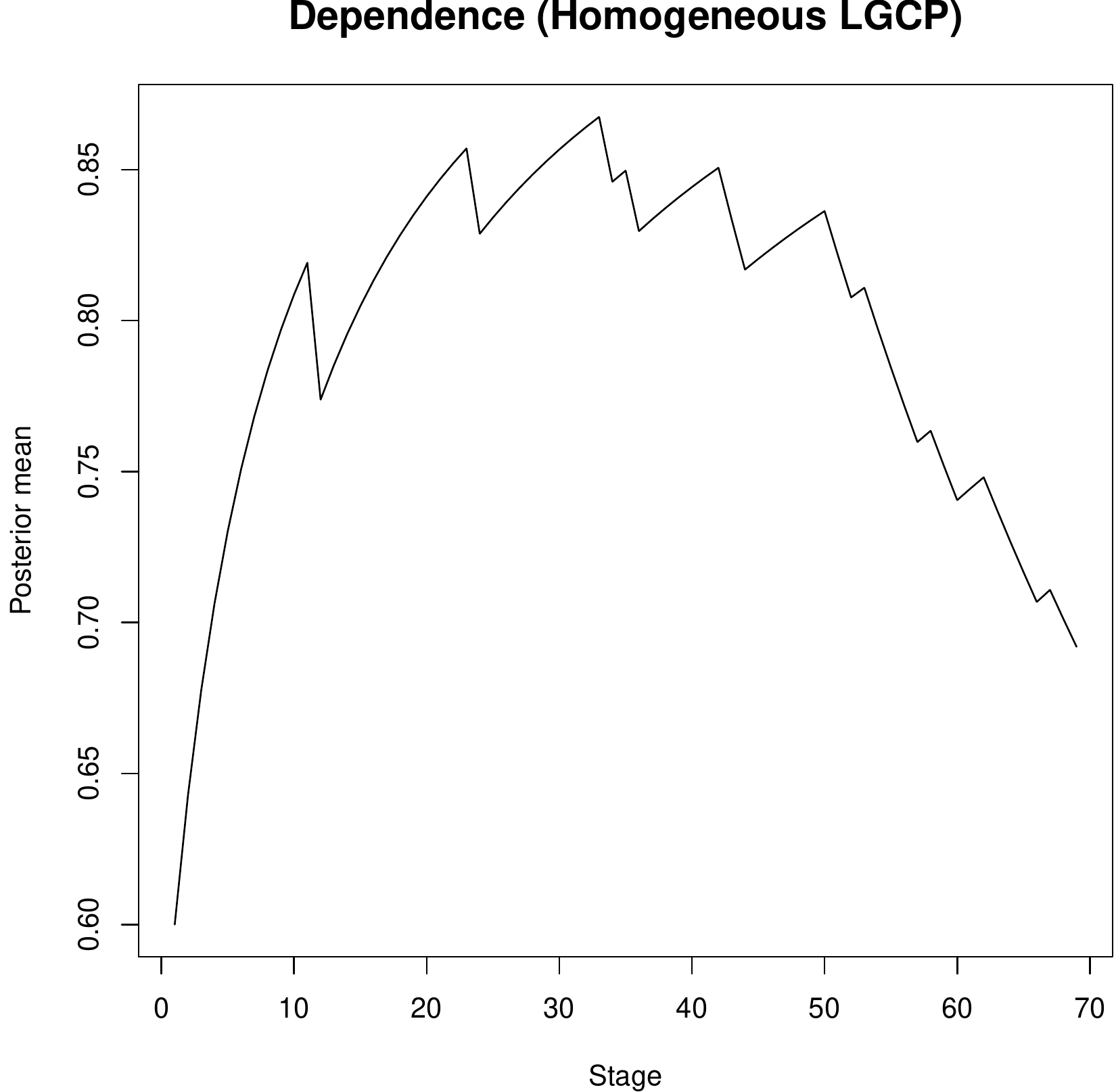}}
\caption{Detection of stationarity and dependence of homogeneous LGCP with our Bayesian method.} 
\label{fig:pp11_stationarity_indep}
\end{figure}

\subsection{Example 3: Inhomogeneous log-Gaussian Cox process}
\label{subsec:lgcp2}

We now consider a log-Gaussian Cox process where the covariance is now of the Mat\'{e}rn form
\begin{equation}
Cov\left(\log\Lambda(u),\log\Lambda(v)\right)=\sigma^2\frac{2^{1-\nu}}{\Gamma(\nu)}\left(\sqrt{2\nu}\frac{\|u-v\|}{\rho}\right)^{\nu}
\mathcal K_{\nu}\left(\sqrt{2\nu}\frac{\|u-v\|}{\rho}\right),
\label{eq:lgcp_matern1}
\end{equation}
where $\Gamma$ is the gamma function, $\mathcal K_{\nu}$ is the modified Bessel function of the second kind of the order $\nu$, and $\rho^{-1}$ is the scale parameter.
We chose $\sigma^2=2$, $\rho^{-1}=0.7$ and $\nu=0.5$. For the mean function, we chose $\mu(u_1,u_2)=5 - 1.5(u_1 - 0.5)^2 + 2 (u_2 - 0.5)^2$. Thus, the underlying
LGCP is nonstationary. Since the expected intensity is not constant, the point process is inhomogeneous from this perspective.

Using spatstat, we obtained $8814$ observations on $W=[0,3]\times[0,2.2]$, displayed in Figure \ref{fig:pp12_lgcp}.
\begin{figure}
\centering
\includegraphics[width=5.5cm,height=5.5cm]{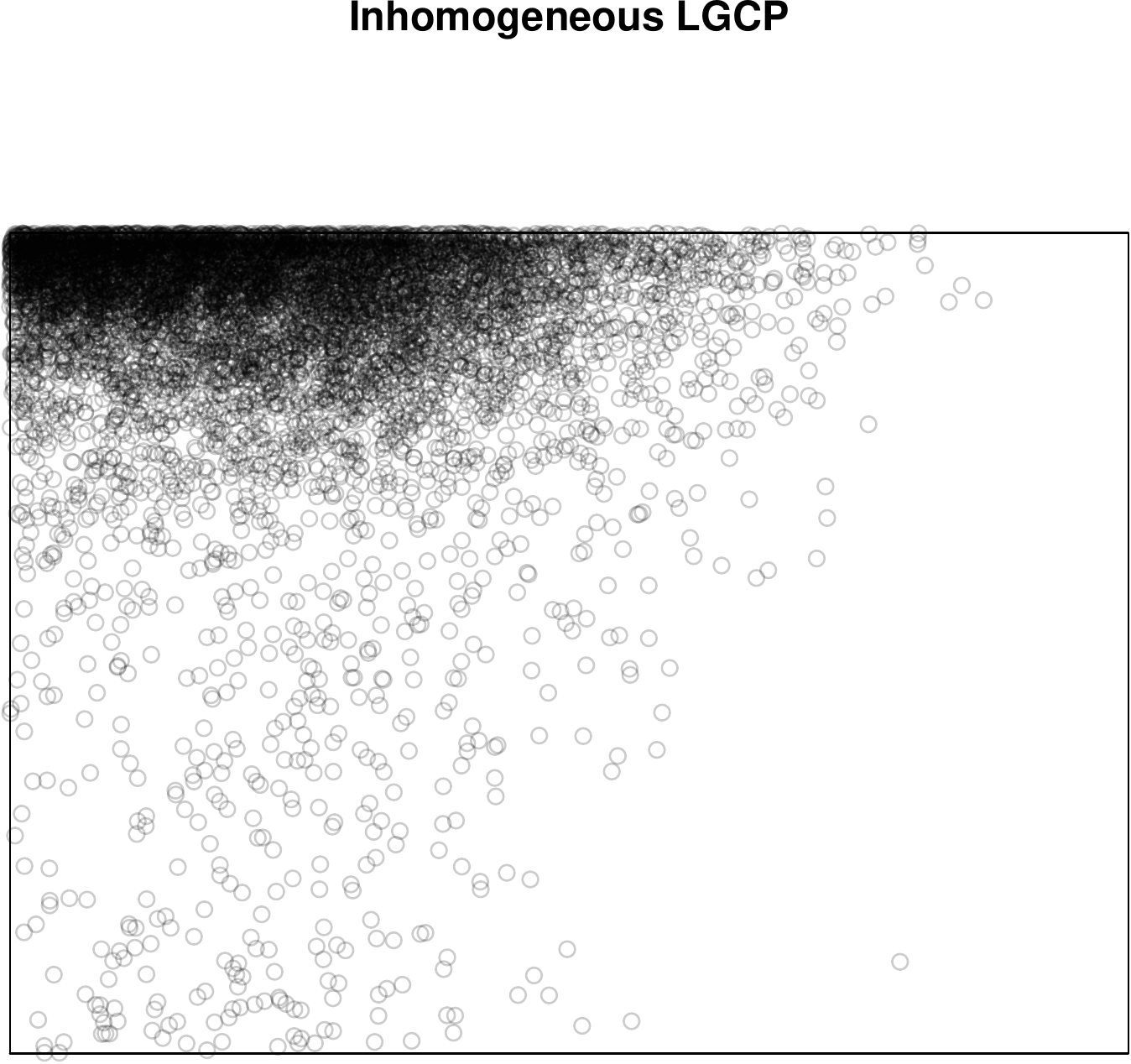}
\caption{Inhomogeneous LGCP.} 
\label{fig:pp12_lgcp}
\end{figure}

Panel (a) of Figure \ref{fig:pp12} shows the result of our Bayesian approach to CSR detection With $K=800$ and $\hat C_1=0.24$, while
panel (b) shows the result of the classical method. Both the methods successfully identify that the underlying point process is not CSR. 
\begin{figure}
\centering
\subfigure [HPP detection with Bayesian method for inhomogeneous LGCP.]{ \label{fig:hpp_bayesian12}
\includegraphics[width=5.5cm,height=5.5cm]{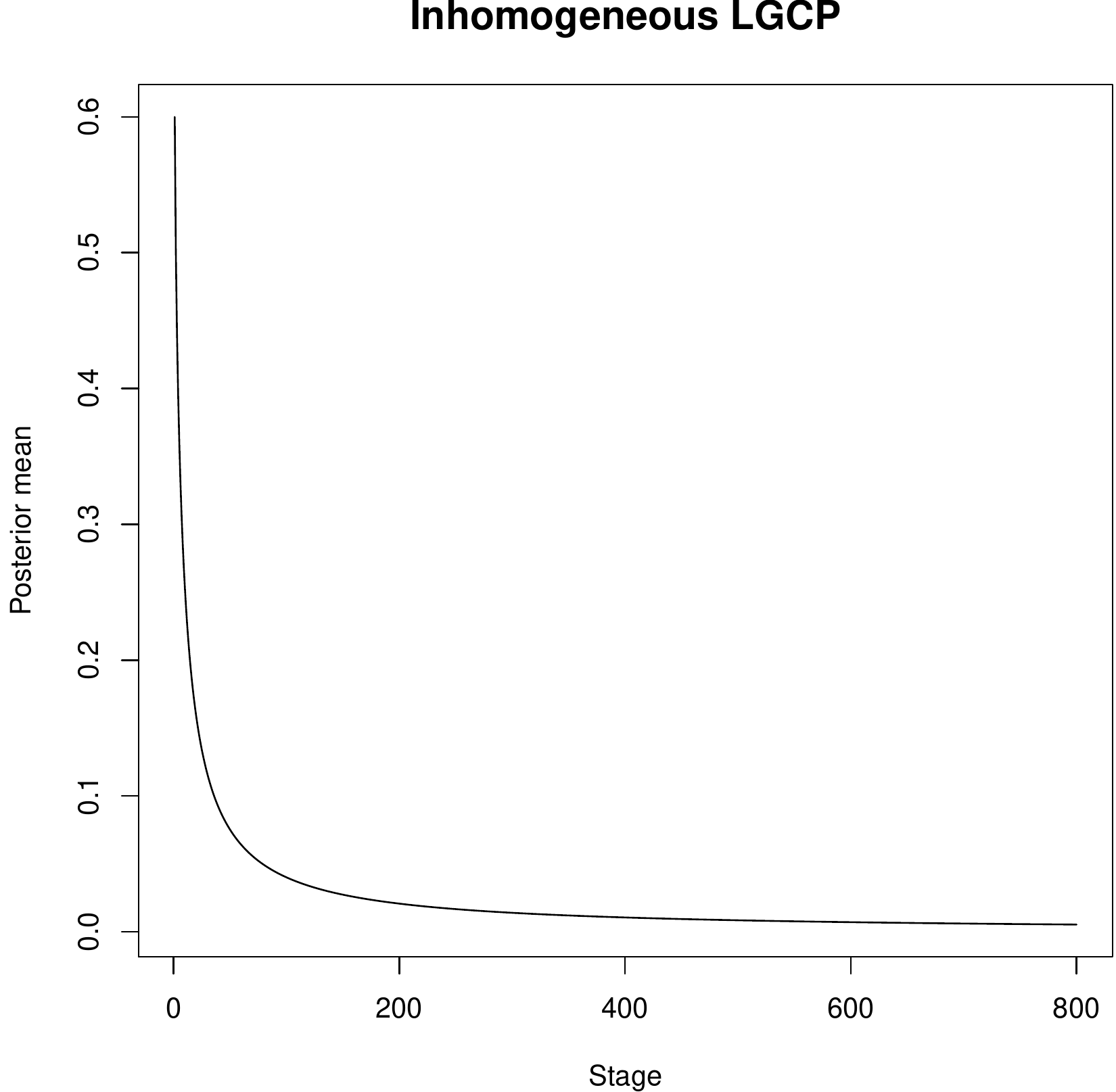}}
\hspace{2mm}
\subfigure [HPP detection with classical method for inhomogeneous LGCP.]{ \label{fig:hpp_classical12}
\includegraphics[width=5.5cm,height=5.5cm]{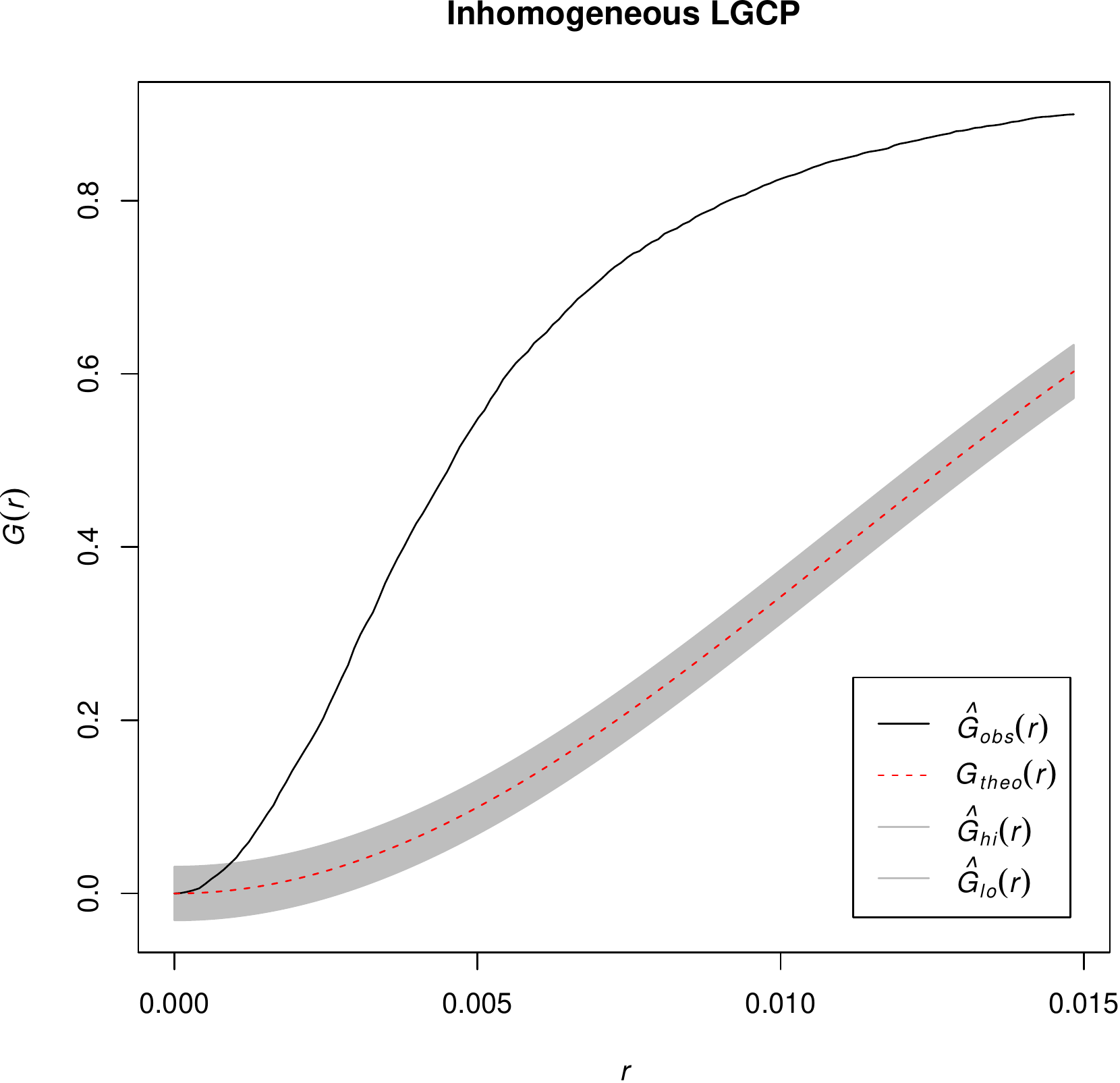}}
\caption{Detection of CSR with our Bayesian method and traditional classical method for LGCP. Both the methods correctly
identify that the underlying point process is not CSR.} 
\label{fig:pp12}
\end{figure}

As shown in panel (a) of Figure \ref{fig:pp12_stationarity_indep}, our Bayesian approach captures nonstationarity of the point process. 
As before, for detection of nonstationarity, we set $K=800$ and $\hat C_1=0.15$.

To test mutual independence among $\bX_{C_i}$, for $i=1,\ldots,K$, we set $K=45$ (due to reasons of numerical stability) and $\hat C_1=0.5$,
as before. Panel (b) of Figure \ref{fig:pp12_stationarity_indep} shows approximately stable behaviour around $0.6$ till the last few points, where steady decrease
is noticed. The stability around the relatively large value $0.6$ for most part of the series indicates mutual independence among most of the random variables
$\bX_{C_i}$, but the last few values of the series suggest that the entire set of random variables $\bX_{C_i}$; $i=1,\ldots,45$, are perhaps not mutually independent.
Hence, the entire set of random variables can not be regarded as mutually independent, leading to non-Poisson conclusion.

\begin{figure}
\centering
\subfigure [Nonstationary LGCP.]{ \label{fig:lgcp2_bayesian_stationary}
\includegraphics[width=5.5cm,height=5.5cm]{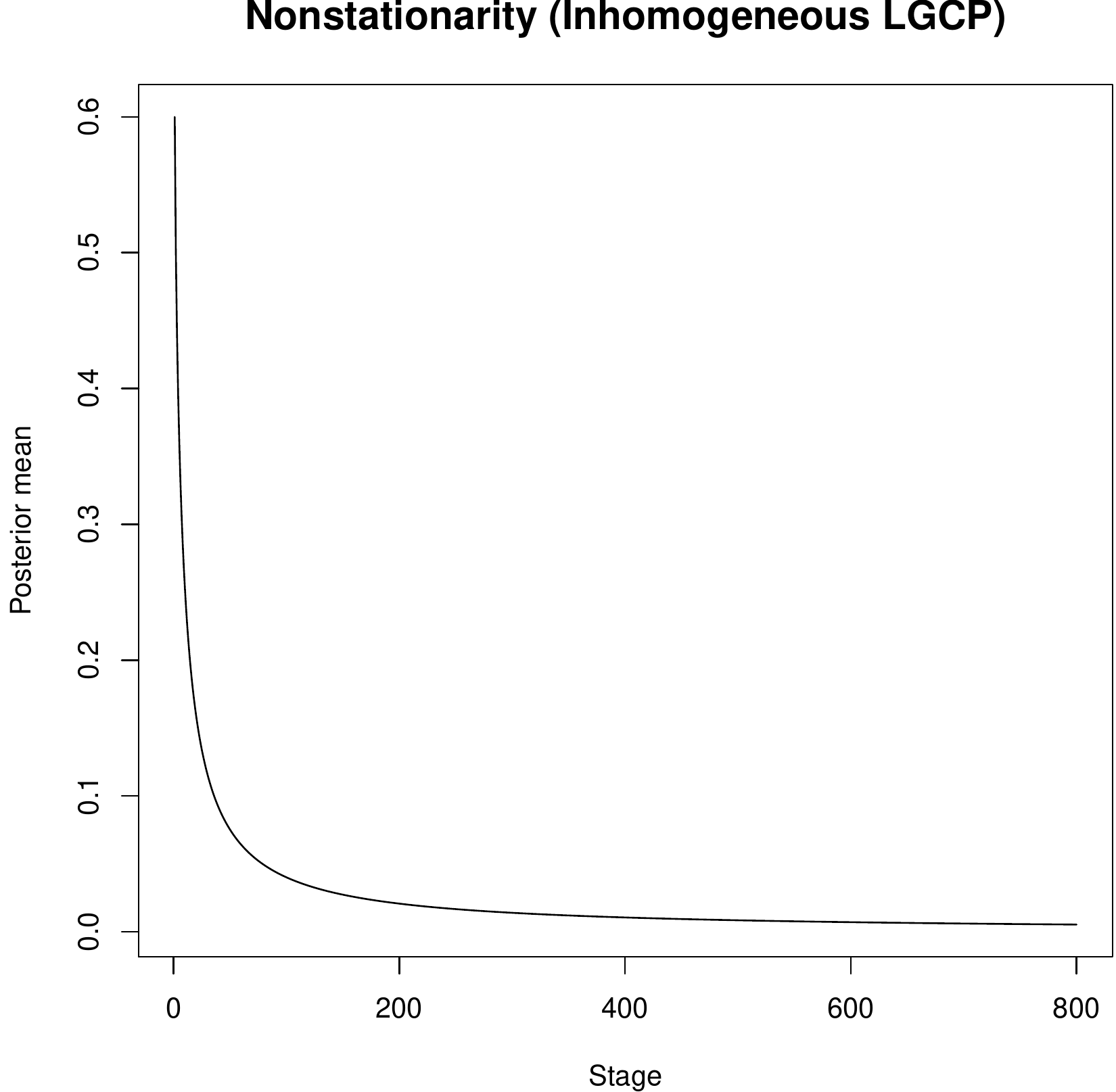}}
\hspace{2mm}
\subfigure [Dependent point process (inhomogeneous LGCP).]{ \label{fig:lgcp2_bayesian_dependent}
\includegraphics[width=5.5cm,height=5.5cm]{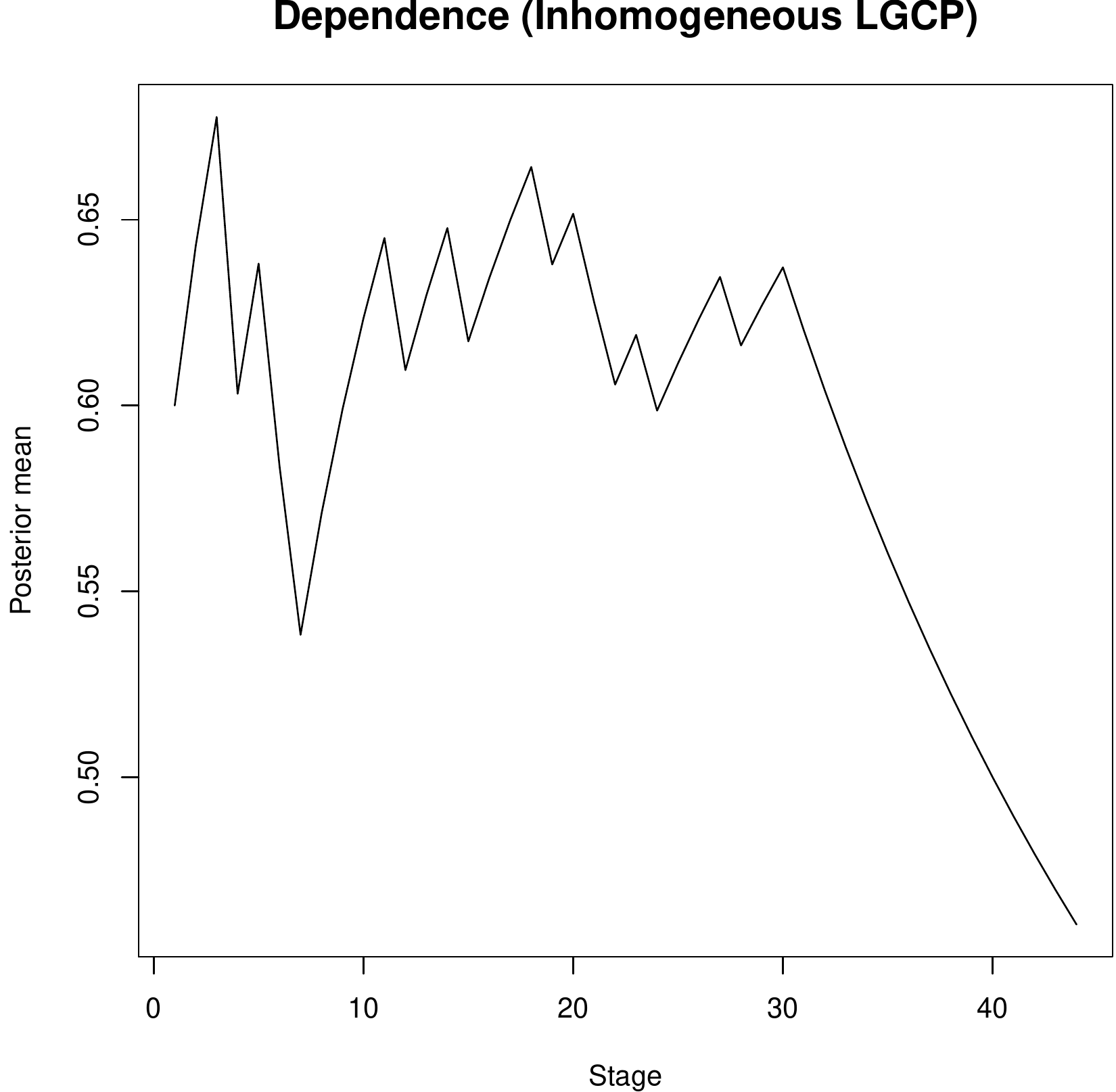}}
\caption{Detection of nonstationarity and dependence of inhomogeneous LGCP with our Bayesian method.} 
\label{fig:pp12_stationarity_indep}
\end{figure}

\subsection{Example 4: Inhomogeneous log-Gaussian Cox process}
\label{subsec:lgcp3}

In this example, we choose the same Mat\'{e}rn covariance function (\ref{eq:lgcp_matern1}), with the same values of $\sigma^2$, $\rho$ and $\nu$ as before, but
now we set $\mu(u_1,u_2)=1-0.4u_1$. The resulting inhomogeneous LGCP obtained using spatstat, consisting of $7245$ points, is depicted in Figure \ref{fig:pp13_lgcp}. 
\begin{figure}
\centering
\includegraphics[width=5.5cm,height=5.5cm]{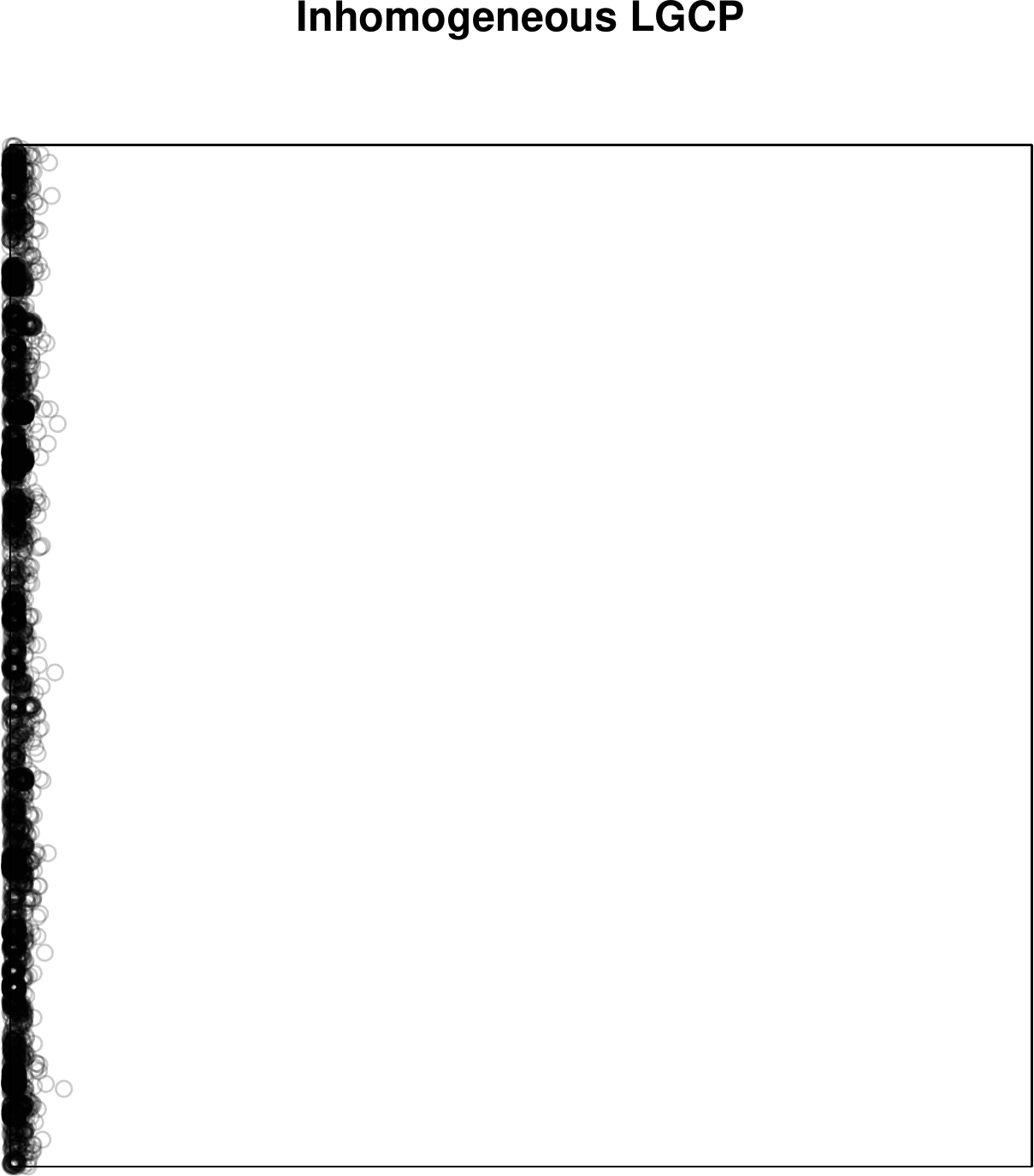}
\caption{Inhomogeneous LGCP.} 
\label{fig:pp13_lgcp}
\end{figure}

With $K=800$ and $\hat C_1=0.24$, our Bayesian method successfully identifies the process as not CSR. The classical method is also successful in this regard.
The results are shown in Figure \ref{fig:pp13}.
\begin{figure}
\centering
\subfigure [HPP detection with Bayesian method for inhomogeneous LGCP.]{ \label{fig:hpp_bayesian13}
\includegraphics[width=5.5cm,height=5.5cm]{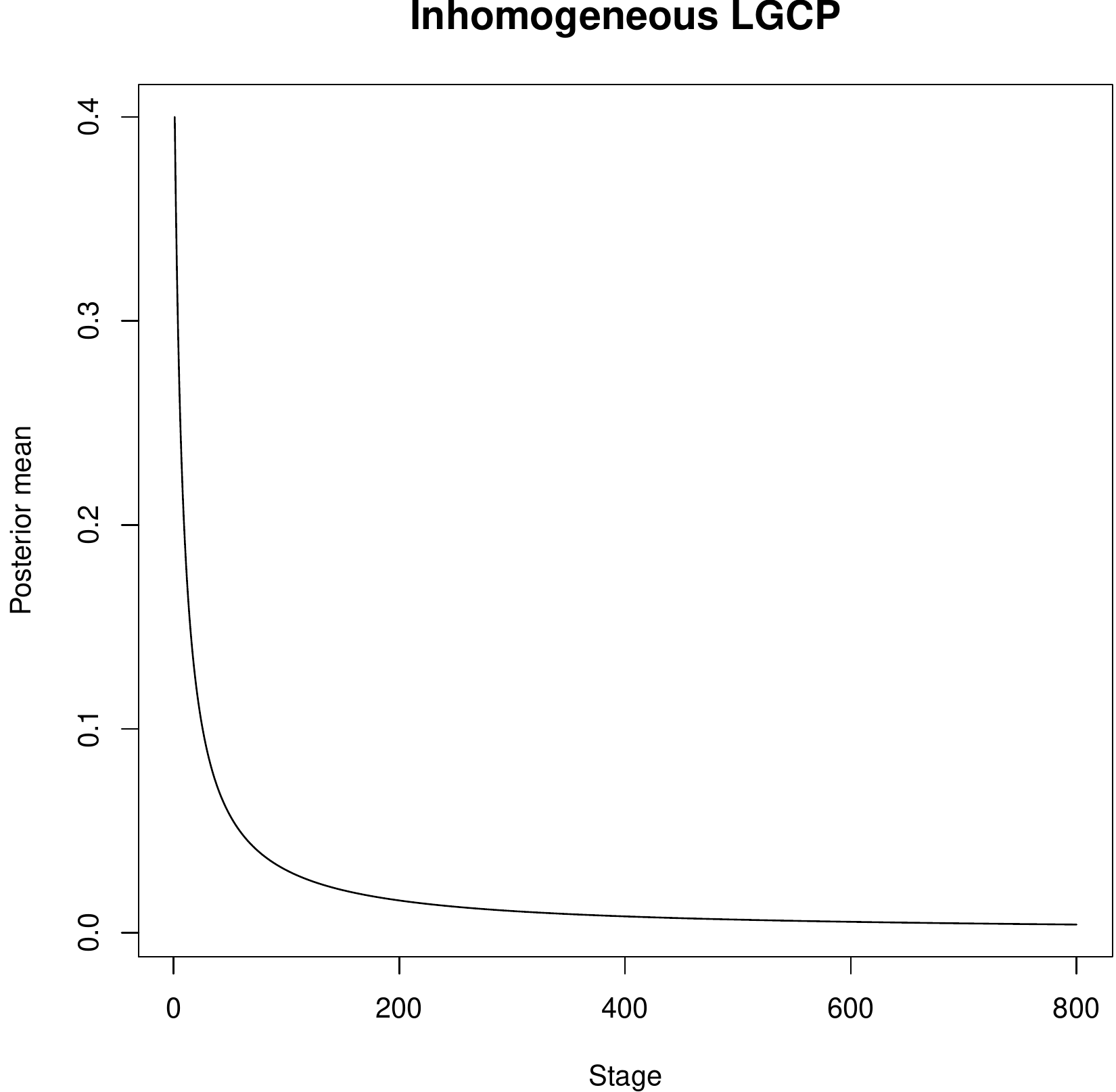}}
\hspace{2mm}
\subfigure [HPP detection with classical method for inhomogeneous LGCP.]{ \label{fig:hpp_classical13}
\includegraphics[width=5.5cm,height=5.5cm]{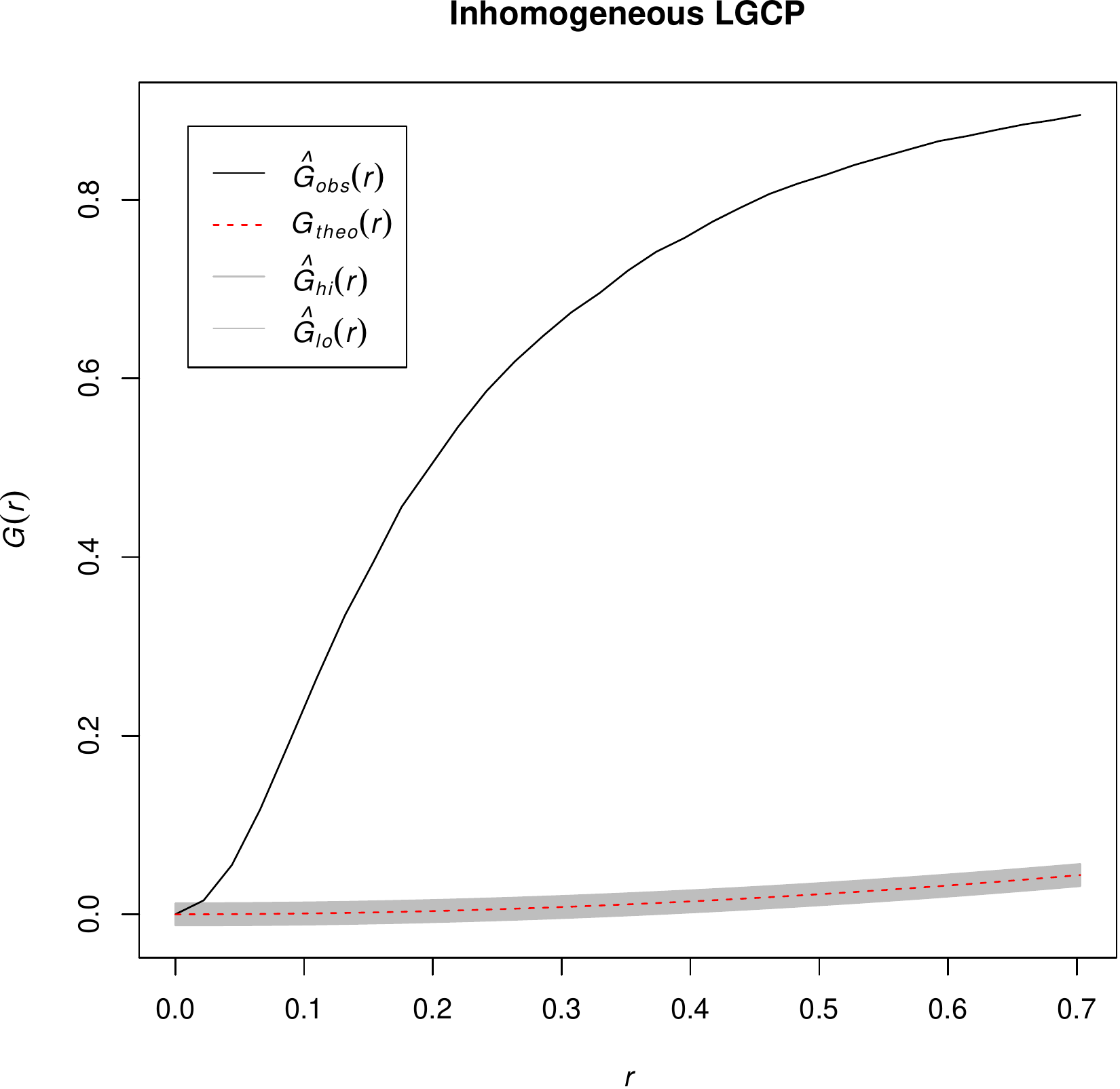}}
\caption{Detection of CSR with our Bayesian method and traditional classical method for LGCP. Both the methods correctly
identify that the underlying point process is not CSR.} 
\label{fig:pp13}
\end{figure}

Again with $K=800$ and $\hat C_1=0.15$, our Bayesian method detects nonstationarity of the underlying point process. Also, with $K=40$ and $\hat C_1=0.5$ as before,
our method correctly detects dependence among $\bX_{C_i}$; $i=1,\ldots,K$.
\begin{figure}
\centering
\subfigure [Nonstationary LGCP.]{ \label{fig:lgcp3_bayesian_stationary}
\includegraphics[width=5.5cm,height=5.5cm]{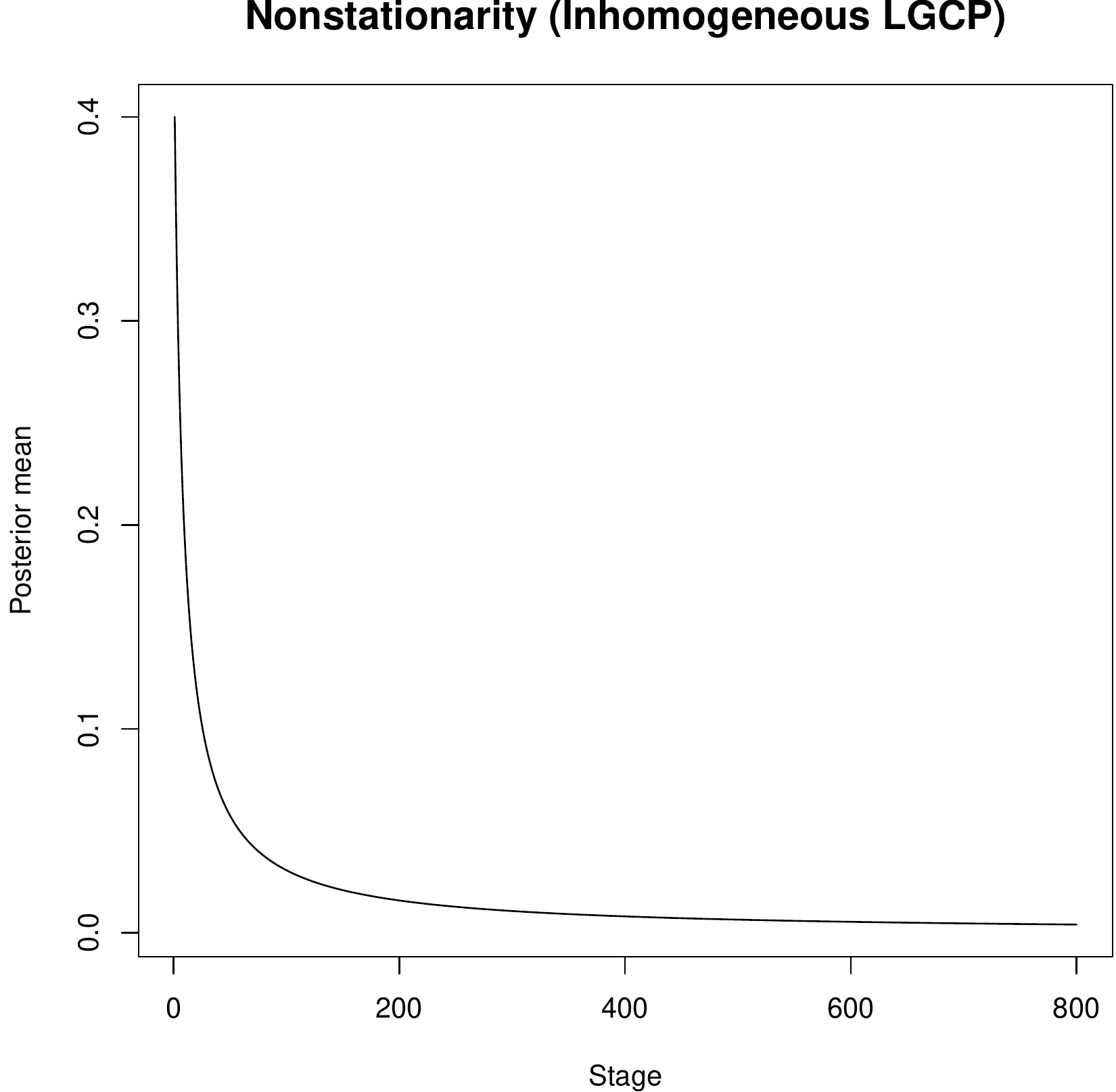}}
\hspace{2mm}
\subfigure [Dependent point process (inhomogeneous LGCP).]{ \label{fig:lgcp3_bayesian_dependent}
\includegraphics[width=5.5cm,height=5.5cm]{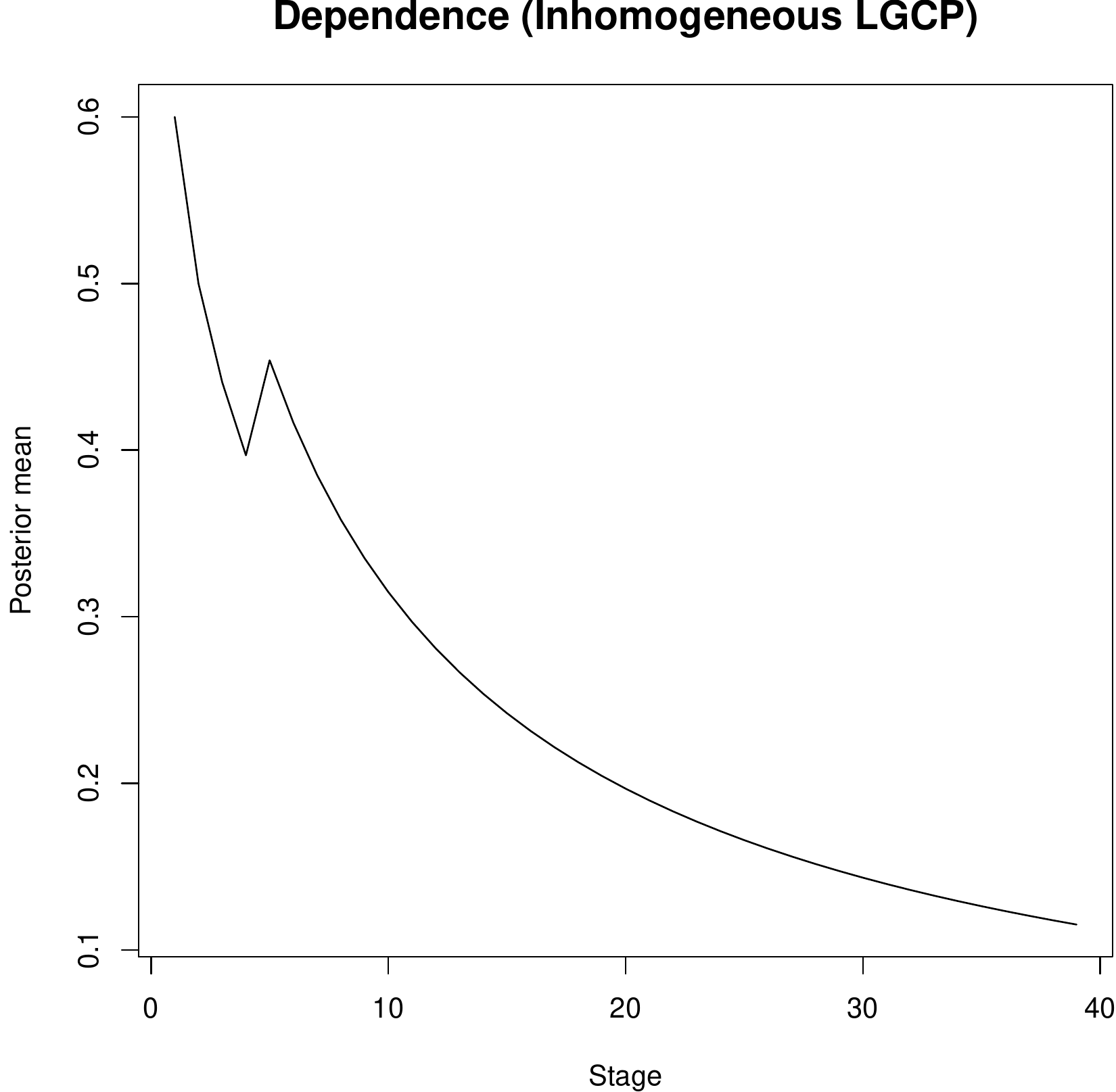}}
\caption{Detection of nonstationarity and dependence of inhomogeneous LGCP with our Bayesian method.} 
\label{fig:pp13_stationarity_indep}
\end{figure}

\subsection{Example 5: Homogeneous Mat\'{e}rn cluster process}
\label{subsec:matern}
The Mat\'{e}rn cluster process is a special case of shot-noise Cox process 
where the offspring points are distributed uniformly inside a disc around the cluster center.
To clarify, first consider a Poisson point process with intensity $\kappa$. Then each `parent' point of this Poisson point process is replaced with a 
random cluster of `offspring' points,
where the number of points per cluster is distributed as Poisson with intensity $\mu$ on a disc with center being the parent point. This point process is non-Poisson. 
Mathematically, consider 
\begin{equation}
	\Lambda(u)=\sum_{(c,\gamma)\in\bPhi}\gamma k(c,u),
	\label{eq:shotnoise1}
\end{equation}
where $c\in\mathbb R^2$, $\gamma>0$, $\bPhi$ is a Poisson process on $\mathbb R^2\times(0,\infty)$, and $k(c,\cdot)$ is a density for a two-dimensional
continuous random variable. Then $\bX$ is a shot noise Cox process if given $\blambda$ defined by (\ref{eq:shotnoise1}), $\bX$ is a Poisson process with intensity
function $\Lambda$. It follows that $\bX$ is the superposition (union) of independent Poisson processes $\bX_{(c,\gamma)}$ with intensity functions $\gamma k(c,\cdot)$,
where $(c,\gamma)\in\bPhi$. If $\gamma$ is a variable (either random or non-random), then $\bX_{(c,\gamma)}$ can be thought of as a cluster with center $c$ 
and mean number of points $\gamma$. In this sense, $\bX$ is a Poisson cluster process. 

The Mat\'{e}rn cluster process is a special case of the above process, where the centre points $c$ arise from a Poisson process with intensity function $\kappa$
and $\gamma\equiv\mu$, a positive non-random function, and $k(c,\cdot)$ is the density of the uniform distribution on a disc of radius $r$, with center $c$. 

In this example, we simulate a Mat\'{e}rn cluster process on a window $W=[0,10]\times[0,10]$, $\kappa=10$, $\mu=5$, and disc radius $r=0.1$, and obtain $4882$ points,
shown in Figure \ref{fig:pp2_plots}. As can be easily verified from (\ref{eq:shotnoise1}) and the following expositions, 
the random intensity function $\Lambda$ in this case is stationary, and hence, $\bX$ is stationary. 

\begin{figure}
\centering
\includegraphics[width=5.5cm,height=5.5cm]{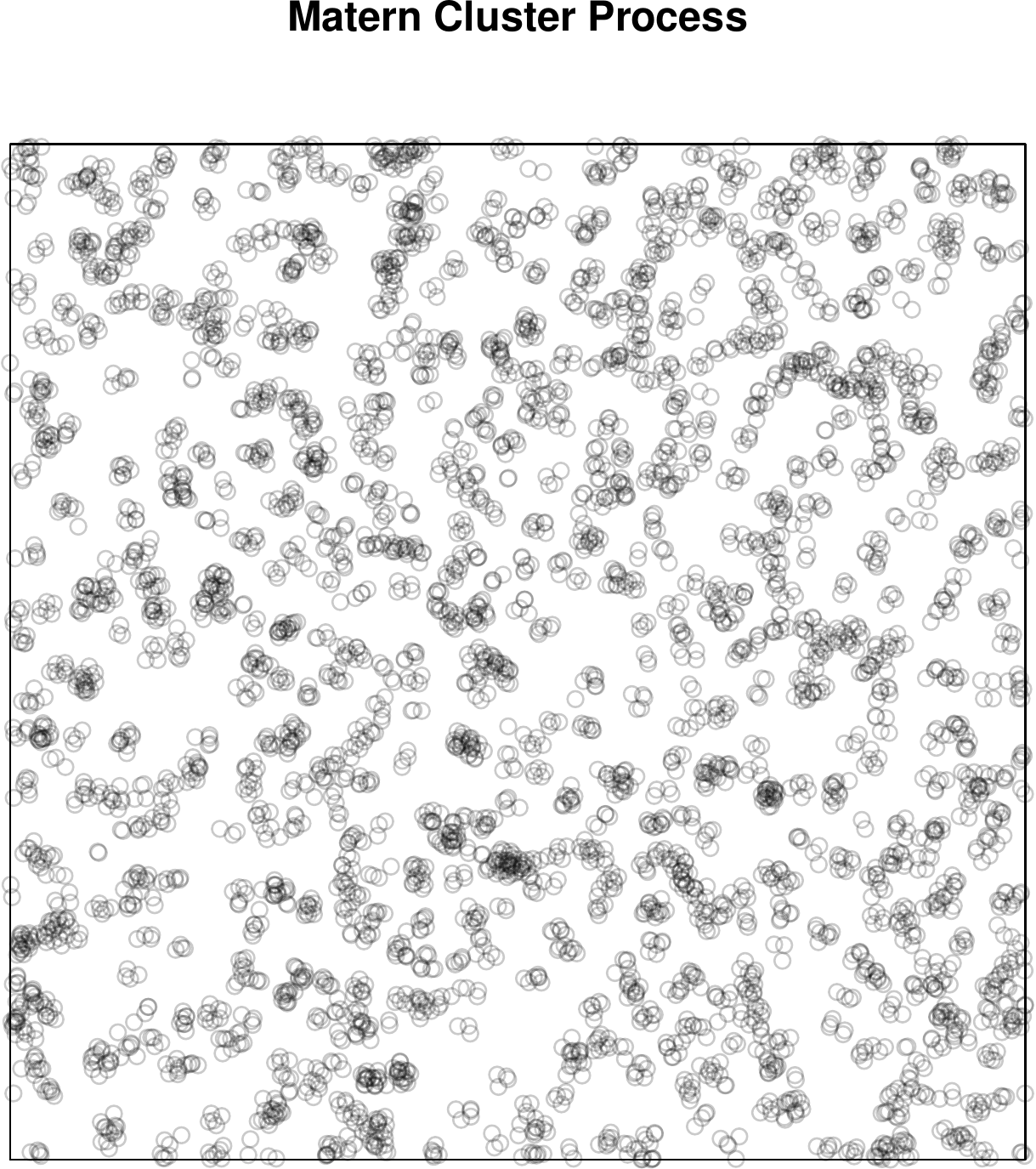}
	\caption{Mat\'{e}rn cluster point process pattern.} 
\label{fig:pp2_plots}
\end{figure}

Figure \ref{fig:pp2} shows the results of our Bayesian method and the classical method for detecting CSR. Both the methods correctly point out that the
underlying point process is not CSR. Here, for the Bayesian method, we set $K=500$ and $\hat C_1=0.25$, the maximum value leading to the conclusion
of not CSR. 

\begin{figure}
\centering
	\subfigure [HPP detection with Bayesian method for Mat\'{e}rn cluster process.]{ \label{fig:hpp_bayesian2}
\includegraphics[width=5.5cm,height=5.5cm]{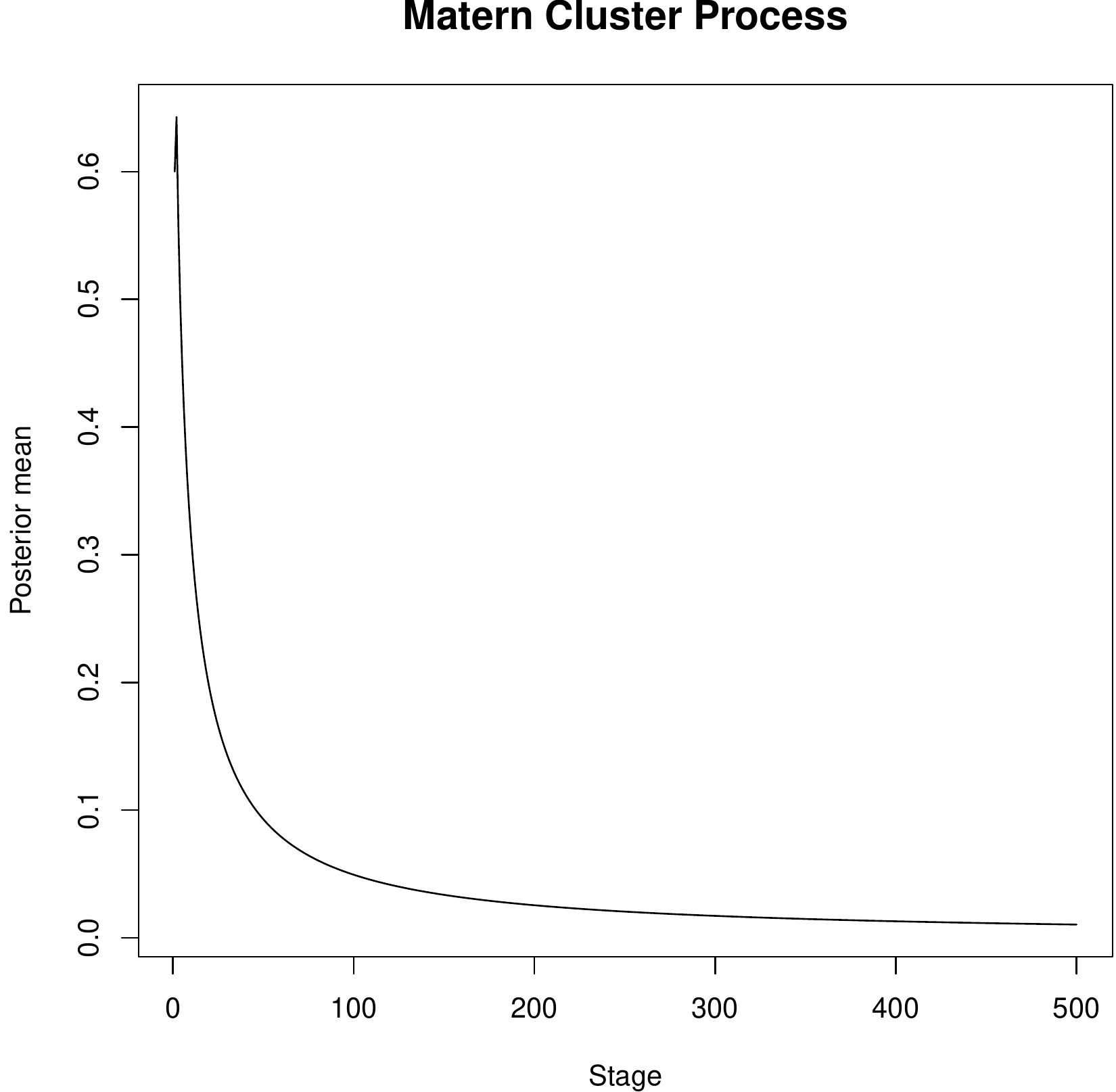}}
\hspace{2mm}
	\subfigure [HPP detection with classical method for Mat\'{e}rn cluster process.]{ \label{fig:hpp_classical2}
\includegraphics[width=5.5cm,height=5.5cm]{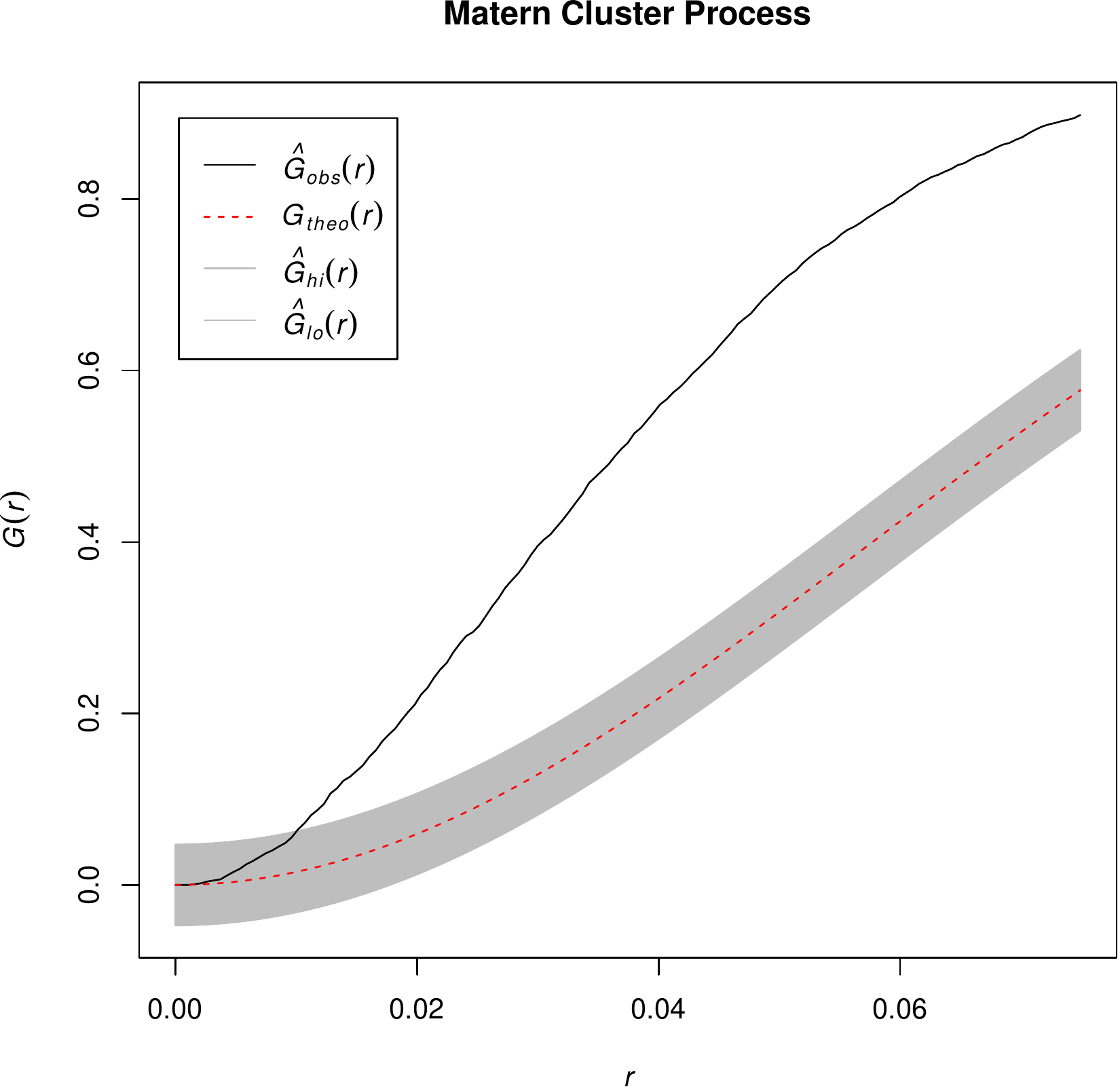}}
	\caption{Detection of CSR with our Bayesian method and traditional classical method for Mat\'{e}rn cluster process. Both the methods correctly
identify that the underlying point process is not CSR.} 
\label{fig:pp2}
\end{figure}

Panel (a) of Figure \ref{fig:pp2_stationarity_indep} shows that stationarity of the point process has been correctly captured by our Bayesian procedure,
with $K=500$ and $\hat C_1=0.06$, the minimum value of $\hat C_1$ leading to stationarity. 

The result of our test for independence is depicted by panel (b) of Figure \ref{fig:pp2_stationarity_indep}, for $K=50$ and $\hat C_1=0.5$ as usual.
Dependence is indicated, correctly leading to the non-Poisson conclusion.
\begin{figure}
\centering
	\subfigure [Stationary point process (Mat\'{e}rn cluster process).]{ \label{fig:matern_bayesian_stationary}
\includegraphics[width=5.5cm,height=5.5cm]{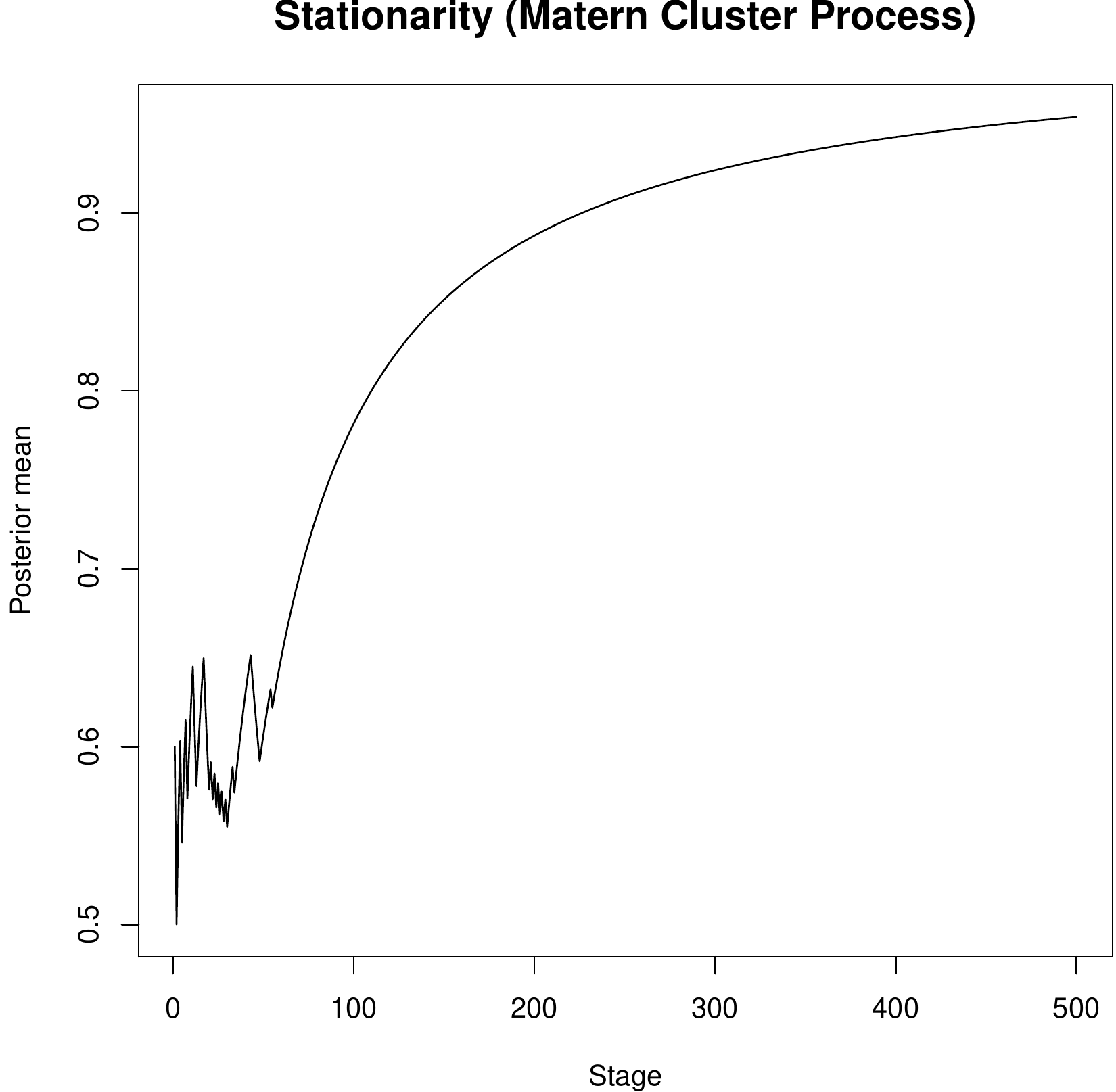}}
\hspace{2mm}
	\subfigure [Dependent point process (Mat\'{e}rn cluster process).]{ \label{fig:matern_bayesian_dependent}
\includegraphics[width=5.5cm,height=5.5cm]{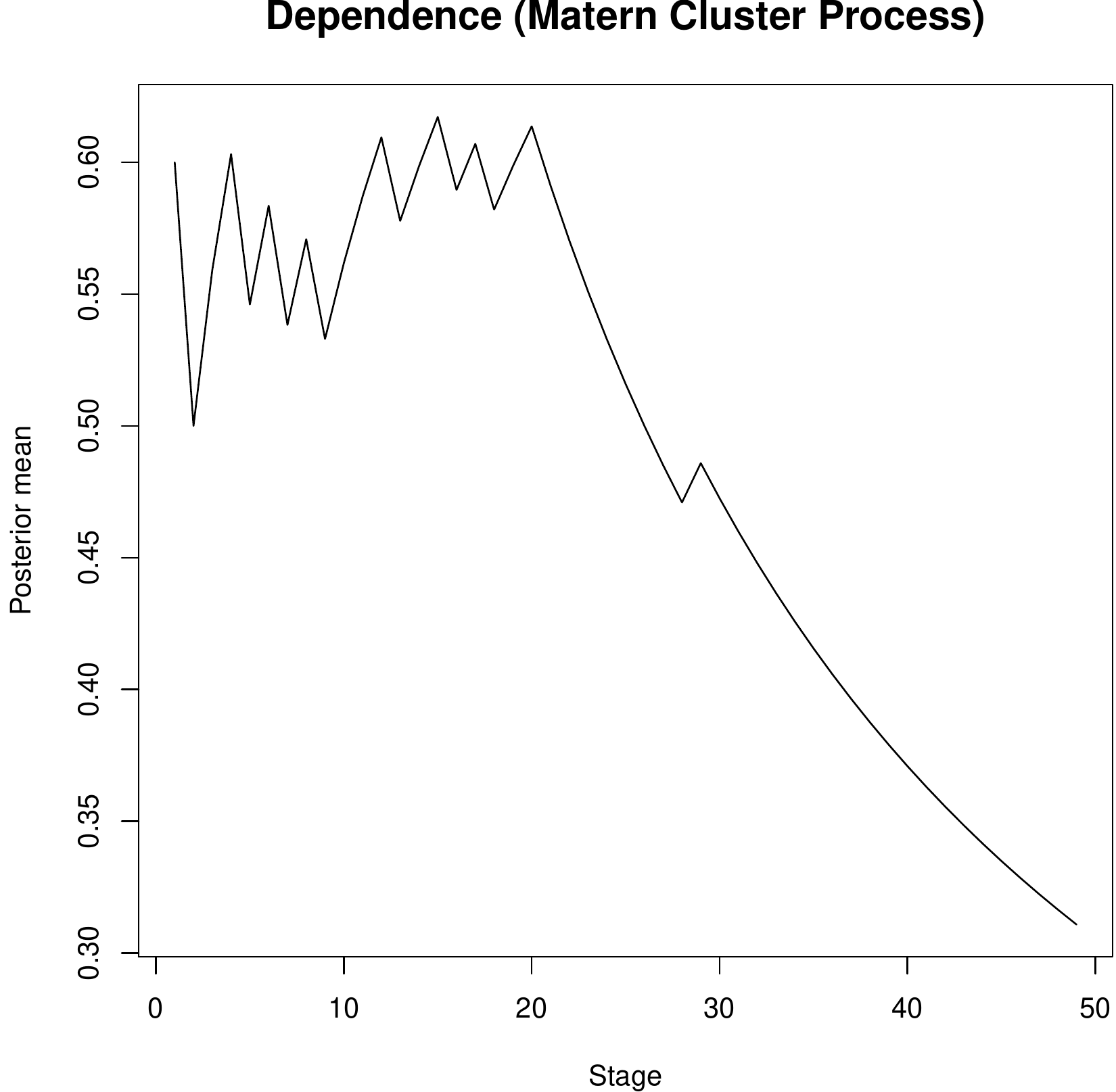}}
	\caption{Detection of stationarity and dependence of Mat\'{e}rn cluster process with our Bayesian method.} 
\label{fig:pp2_stationarity_indep}
\end{figure}

\subsection{Example 6: Inhomogeneous Mat\'{e}rn cluster process with $\mu$ inhomogeneous}
\label{subsec:matern2}

We now consider an inhomogeneous Mat\'{e}rn cluster process with $\kappa=10$, disc radius $r=0.05$, and $\mu(u_1,u_2)=2\exp\left(2|u_1|-1\right)$, an obtain
$8606$ points in $W=[0,3]\times[0,3]$. The points are plotted in Figure \ref{fig:pp2_plots2}.
\begin{figure}
\centering
\includegraphics[width=5.5cm,height=5.5cm]{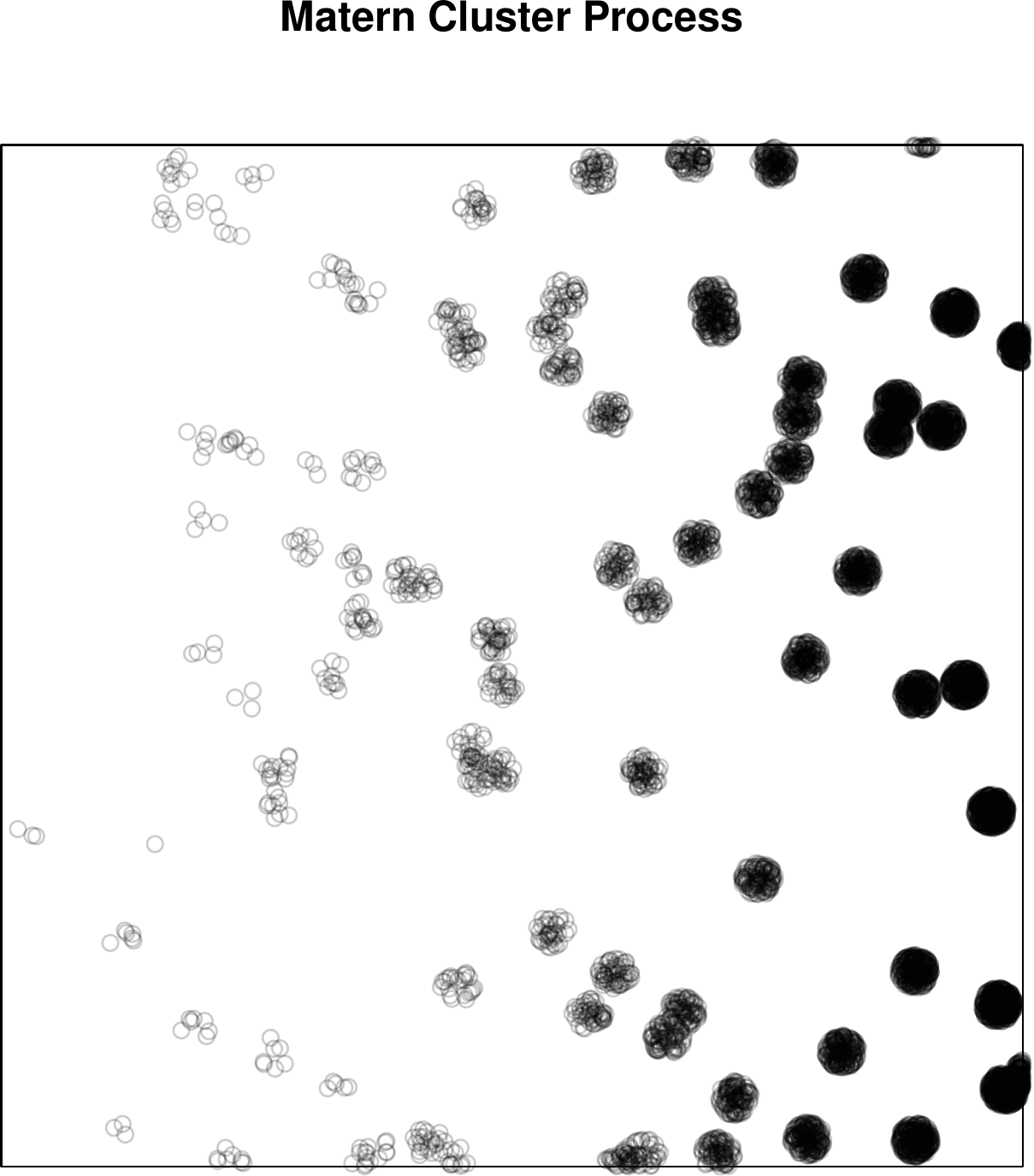}
	\caption{Inhomogeneous Mat\'{e}rn cluster point process pattern.} 
\label{fig:pp2_plots2}
\end{figure}

Figure \ref{fig:pp22} shows that both the methods for detecting CSR correctly detect non-CSR.
For the Bayesian method, we set $K=800$ and $\hat C_1=0.6$, the maximum value leading to the conclusion
of not CSR. 

\begin{figure}
\centering
	\subfigure [HPP detection with Bayesian method for Mat\'{e}rn cluster process.]{ \label{fig:hpp_bayesian3_mat}
\includegraphics[width=5.5cm,height=5.5cm]{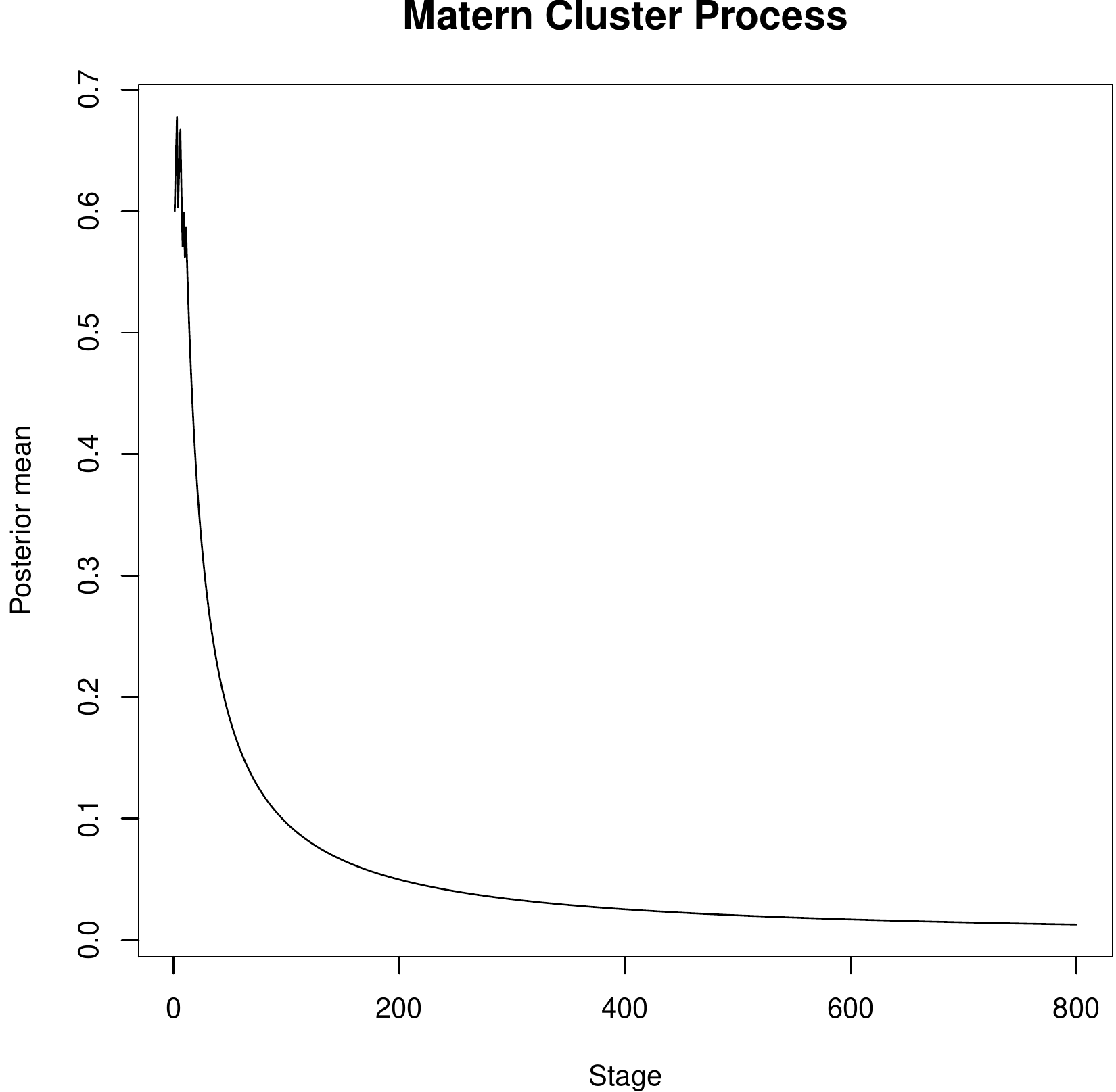}}
\hspace{2mm}
	\subfigure [HPP detection with classical method for Mat\'{e}rn cluster process.]{ \label{fig:hpp_classical3_mat}
\includegraphics[width=5.5cm,height=5.5cm]{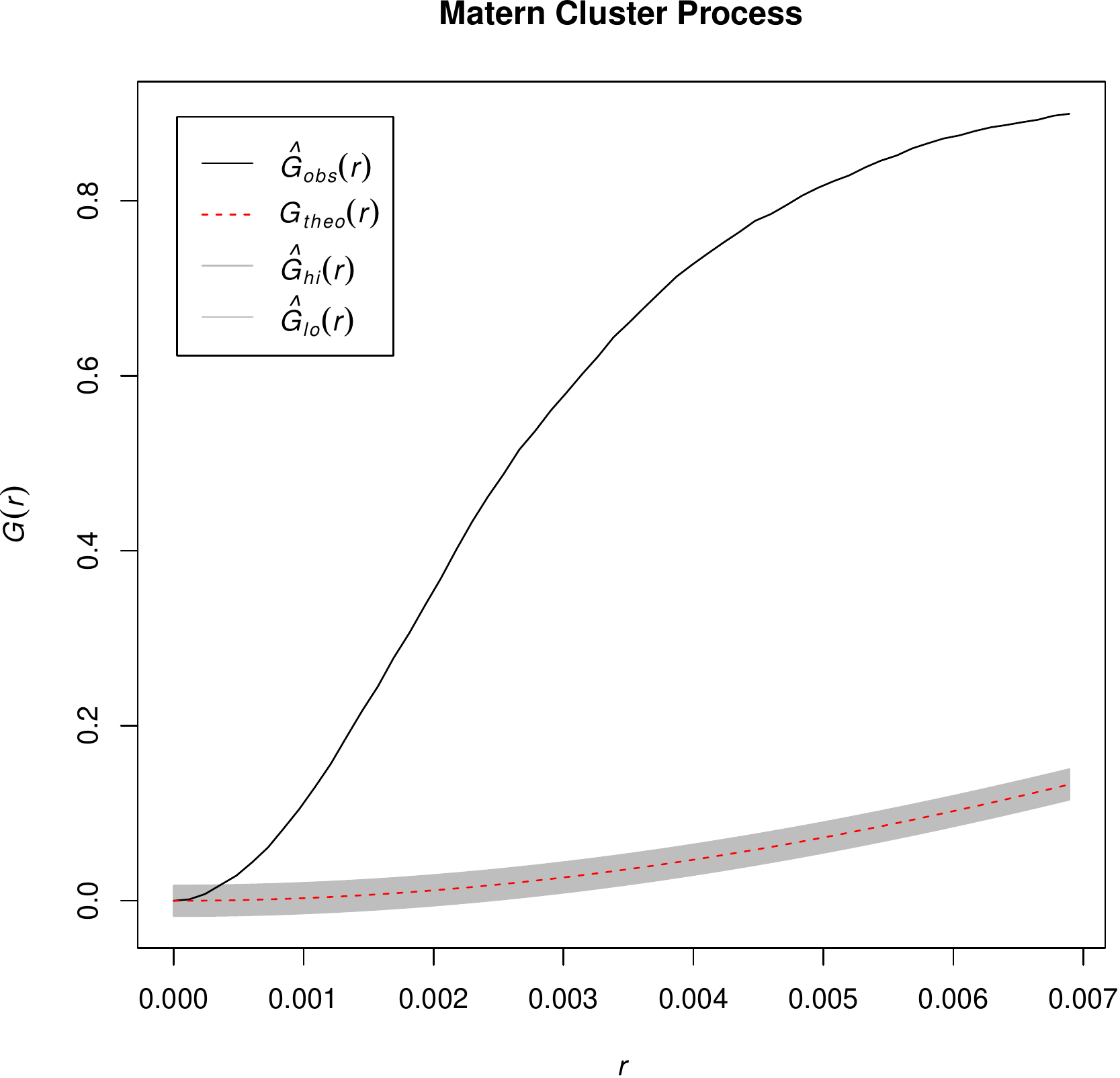}}
	\caption{Detection of CSR with our Bayesian method and traditional classical method for inhomogeneous Mat\'{e}rn cluster process. Both the methods correctly
identify that the underlying point process is not CSR.} 
\label{fig:pp22}
\end{figure}

With $K=800$ and $\hat C_1=0.27$, our Bayesian method correct points out nonstationarity. This value of $\hat C_1$ is the maximum value leading to nonstationarity.
As before, the Bayesian method correctly detects dependence with $K=50$ and $\hat C_1=0.5$. The results are depicted in Figure \ref{fig:pp2_stationarity_indep2}.
\begin{figure}
\centering
	\subfigure [Nonstationary point process (Mat\'{e}rn cluster process).]{ \label{fig:matern_bayesian_stationary3_mat}
\includegraphics[width=5.5cm,height=5.5cm]{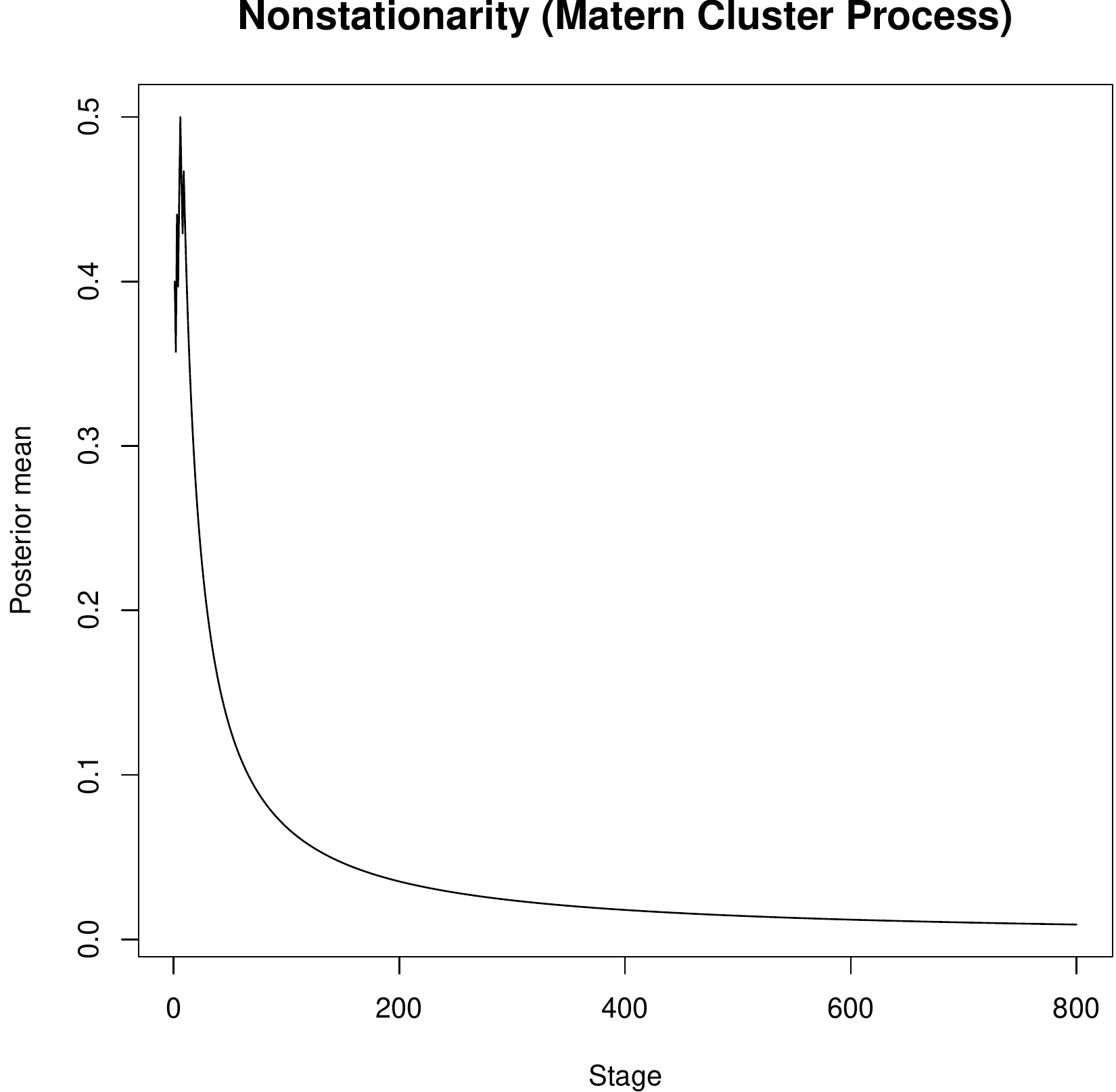}}
\hspace{2mm}
	\subfigure [Dependent point process (Mat\'{e}rn cluster process).]{ \label{fig:matern_bayesian_dependent3_mat}
\includegraphics[width=5.5cm,height=5.5cm]{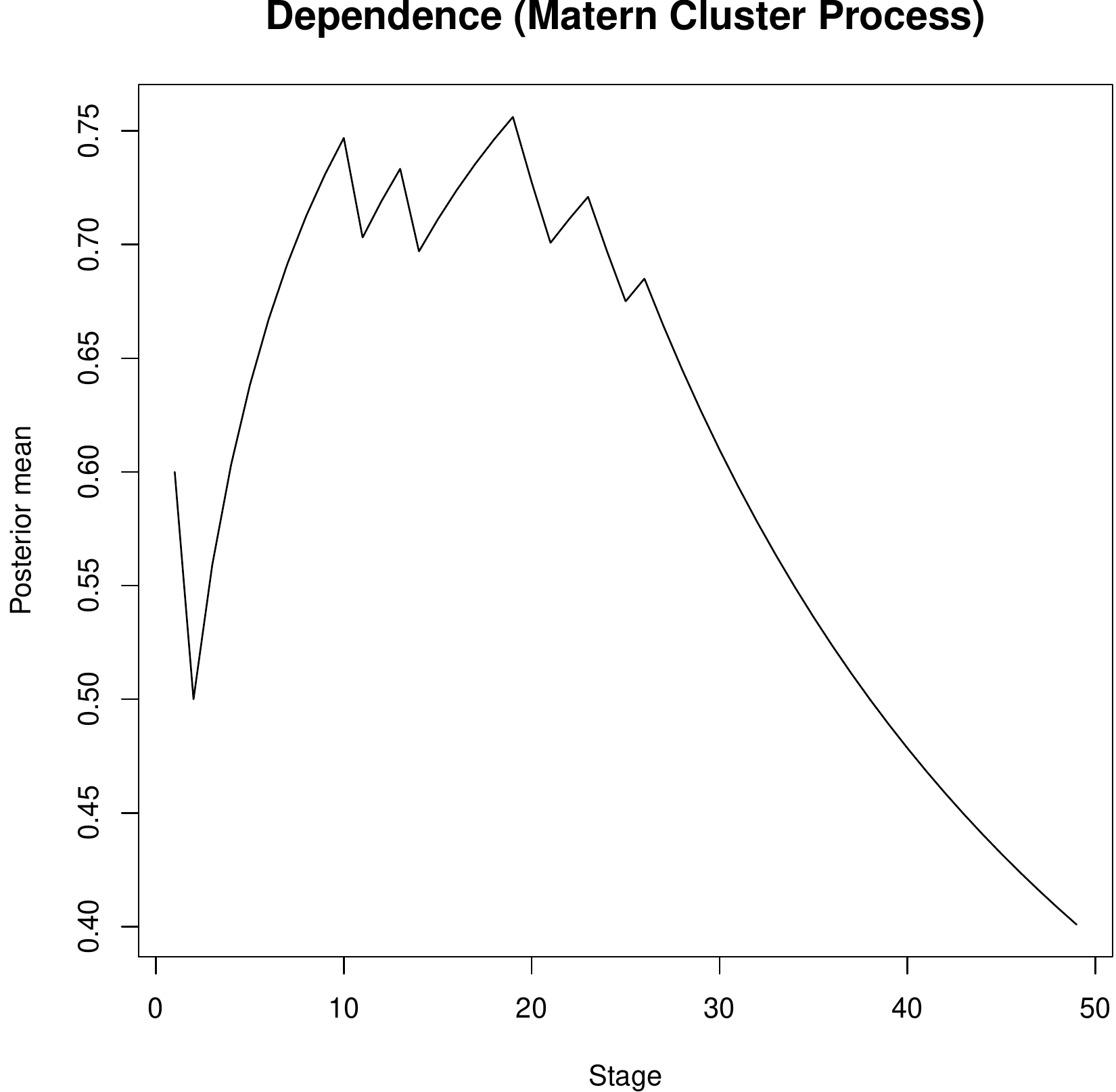}}
	\caption{Detection of nonstationarity and dependence of Mat\'{e}rn cluster process with our Bayesian method.} 
\label{fig:pp2_stationarity_indep2}
\end{figure}

\subsection{Example 7: Mat\'{e}rn cluster process with $\kappa$ Inhomogeneous}
\label{subsec:matern3}

We consider another inhomogeneous Mat\'{e}rn cluster process with $\kappa(u_1,u_2)=2\exp\left(2|u_1|-1\right)$, disc radius $r=0.05$, and $\mu=3$. 
The $2625$ points that we obtained in $W=[0,3]\times[0,3]$ are displayed in Figure \ref{fig:pp2_plots3}.
\begin{figure}
\centering
\includegraphics[width=5.5cm,height=5.5cm]{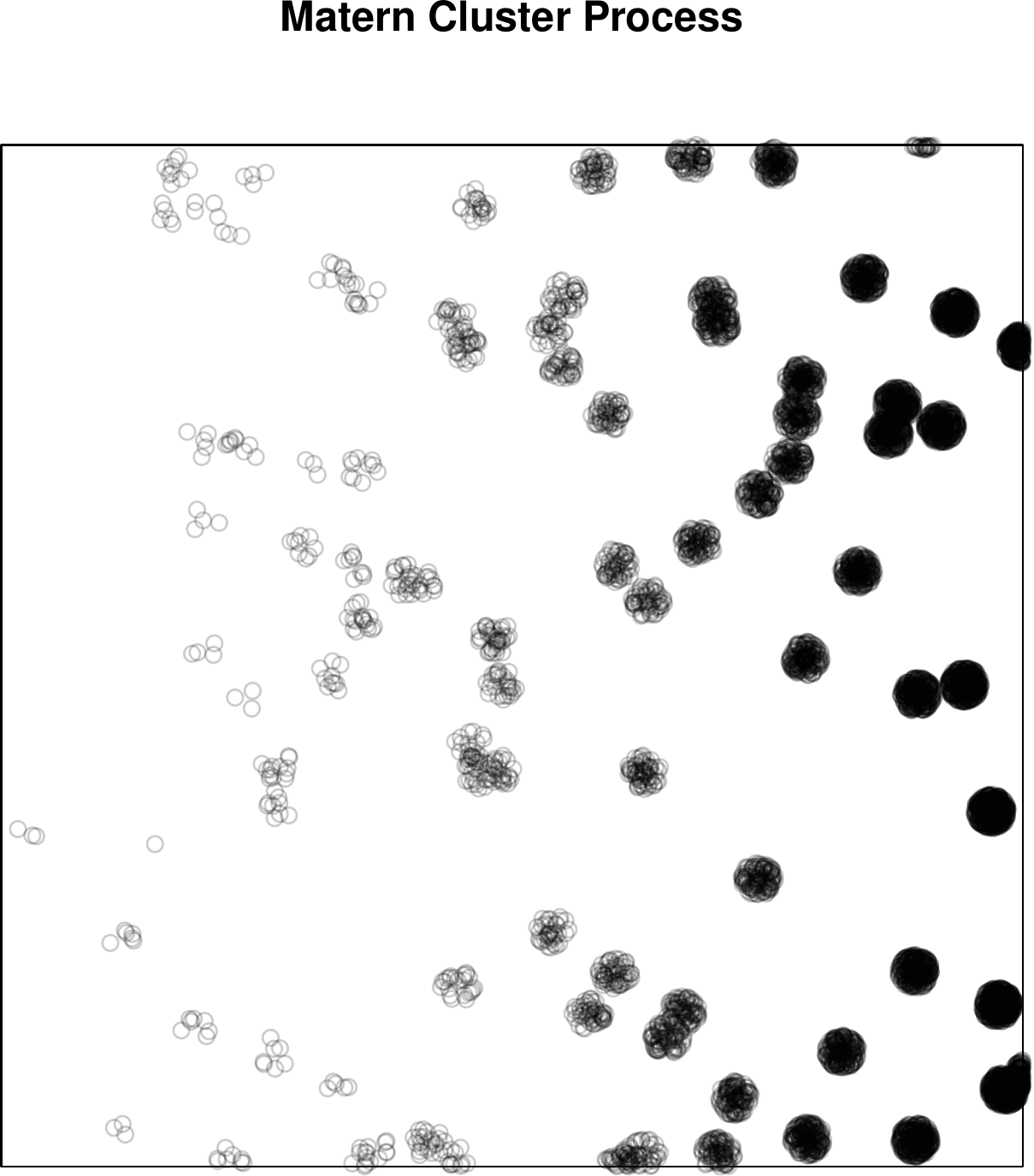}
	\caption{Inhomogeneous Mat\'{e}rn cluster point process pattern.} 
\label{fig:pp2_plots3}
\end{figure}

With $K=300$ and $\hat C_1=0.4$, the Bayesian algorithm correctly detects non-CSR. The classical method also performs adequately.
Figure \ref{fig:pp23} shows that both the methods for detecting CSR correctly detect non-CSR. 

\begin{figure}
\centering
	\subfigure [HPP detection with Bayesian method for Mat\'{e}rn cluster process.]{ \label{fig:hpp_bayesian4_mat}
\includegraphics[width=5.5cm,height=5.5cm]{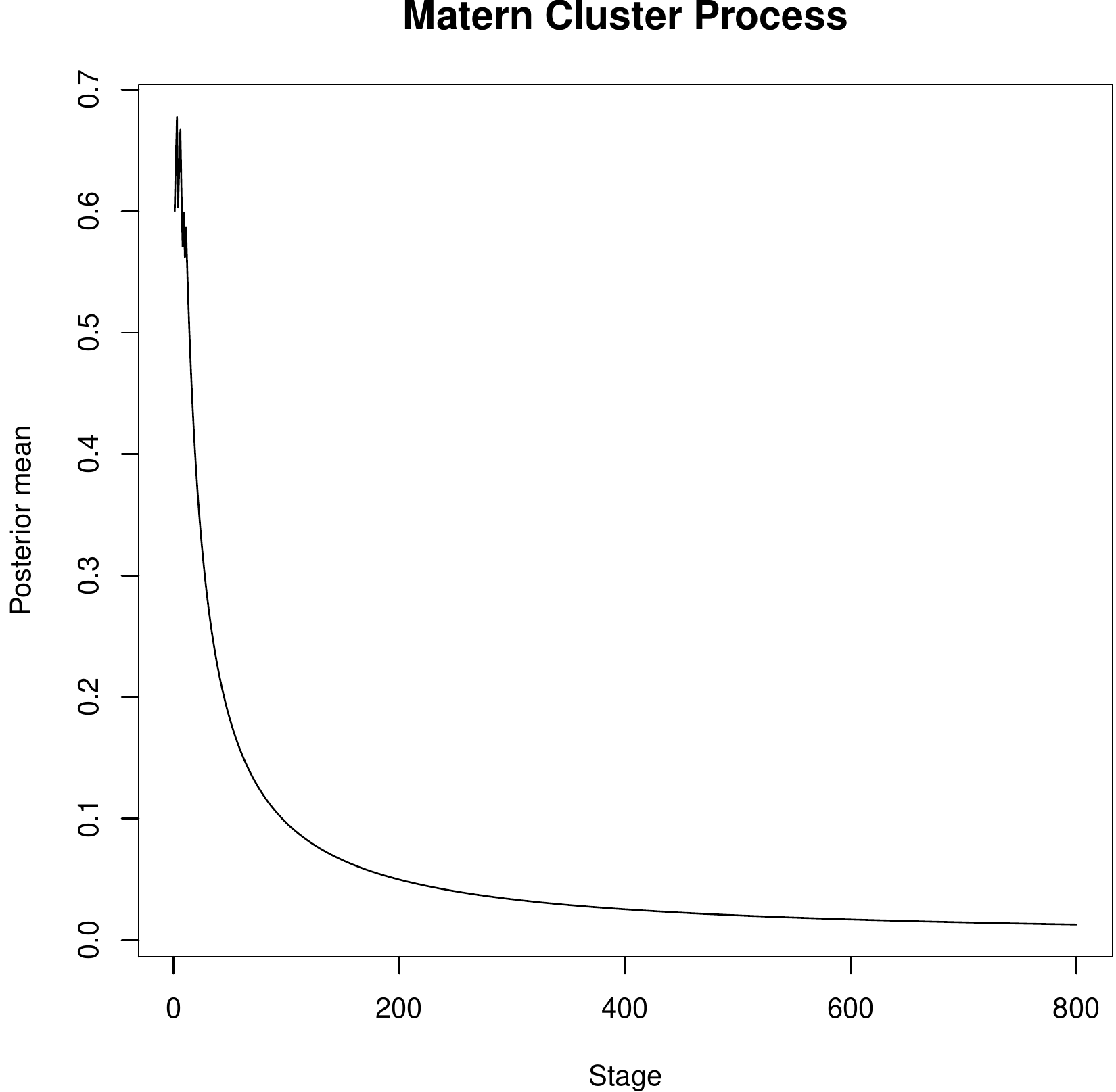}}
\hspace{2mm}
	\subfigure [HPP detection with classical method for Mat\'{e}rn cluster process.]{ \label{fig:hpp_classical4_mat}
\includegraphics[width=5.5cm,height=5.5cm]{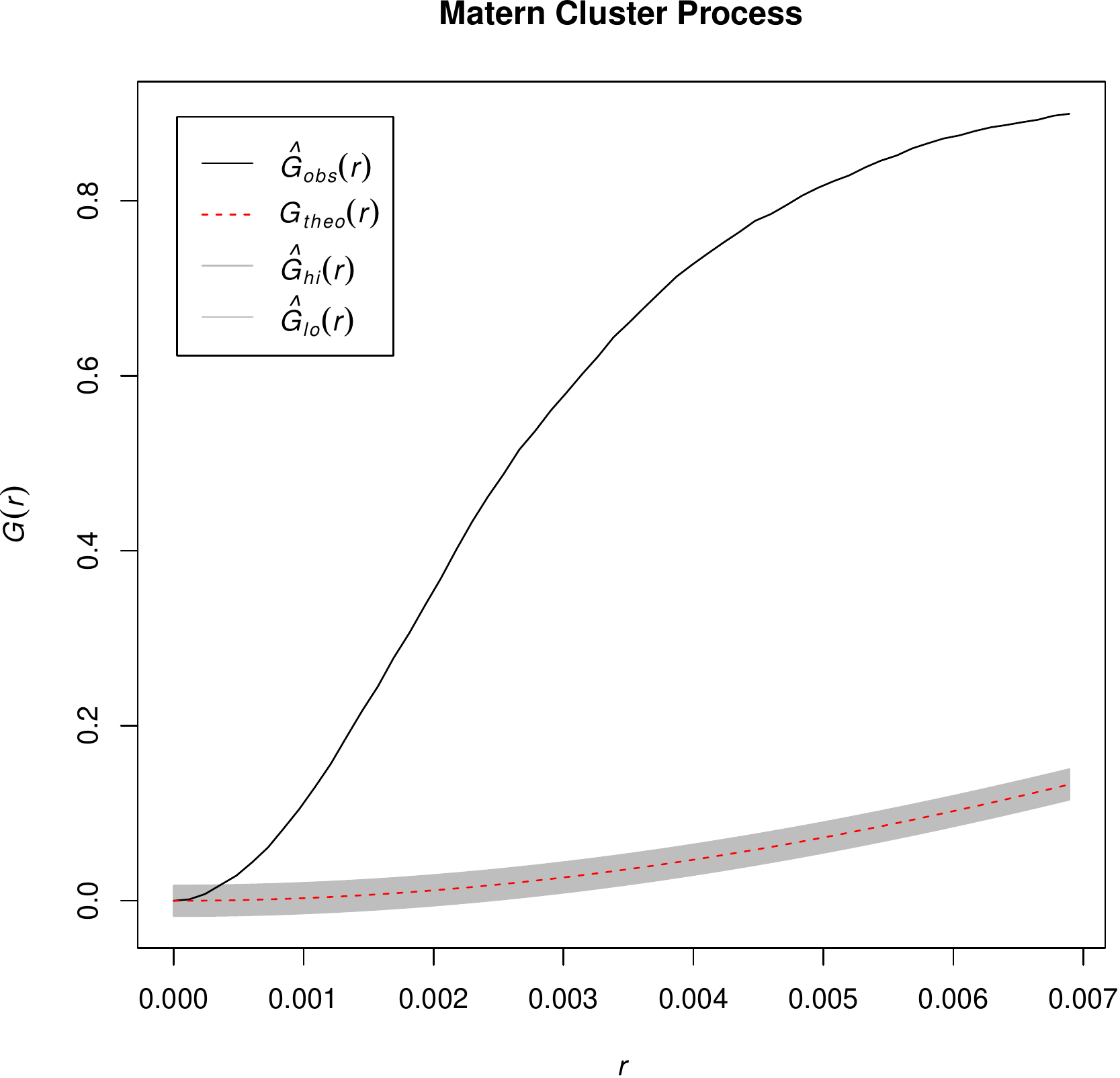}}
	\caption{Detection of CSR with our Bayesian method and traditional classical method for inhomogeneous Mat\'{e}rn cluster process. Both the methods correctly
identify that the underlying point process is not CSR.} 
\label{fig:pp23}
\end{figure}

Nonstationarity is also correctly detected by the Bayesian method with $K=300$ and $\hat C_1=0.26$, the maximum value leading to nonstationarity. Correct detection
of dependence among $\bX_{C_i}$; $i=1,\ldots,50$, has also been possible with the Bayesian algorithm with $\hat C_1=0.5$.
Figure \ref{fig:pp2_stationarity_indep3} presents the relevant results.
\begin{figure}
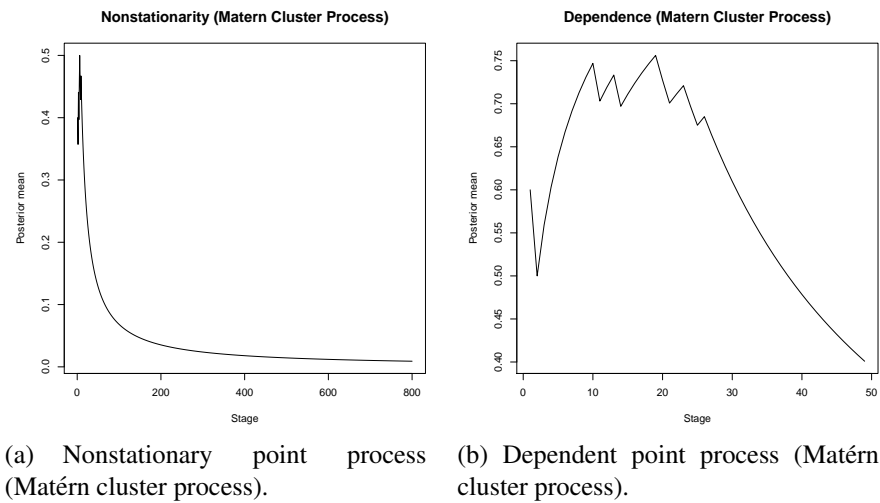

\centering
	\subfigure [Nonstationary point process (Mat\'{e}rn cluster process).]{ \label{fig:matern_bayesian_stationary4_mat}
\includegraphics[width=5.5cm,height=5.5cm]{figures_matern2/stationary-crop.pdf}}
\hspace{2mm}
	\subfigure [Dependent point process (Mat\'{e}rn cluster process).]{ \label{fig:matern_bayesian_dependent4_mat}
\includegraphics[width=5.5cm,height=5.5cm]{figures_matern2/indep-crop.pdf}}
	\caption{Detection of nonstationarity and dependence of Mat\'{e}rn cluster process with our Bayesian method.} 
\label{fig:pp2_stationarity_indep3}
\end{figure}

\subsection{Example 8: Homogeneous Thomas process}
\label{subsec:thomas_hom}

The (modified) Thomas process is a special case of the general shot-noise Cox process in the same way as Mat\'{e}rn cluster process, but 
where $k(c,\cdot)$ is the bivariate normal density with mean $c$ and covariance $\sigma^2I$. From (\ref{eq:shotnoise1}) it is seen that a
stationary process $\bX$ results provided $\kappa$ and $\mu$ are constants. The intensity after integrating out $\Lambda$ is constant in this case,
leading to homogeneous Thomas process.

In this example, we first simulate a Thomas process with $\kappa=10$, $\mu=5$, $\sigma^2=10$, on the window $W=[0,10]\times[0,10]$, and obtained $4858$ points.
The point pattern for this homogeneous Thomas process is displayed in Figure \ref{fig:pp3_plots}.
\begin{figure}
\centering
\includegraphics[width=5.5cm,height=5.5cm]{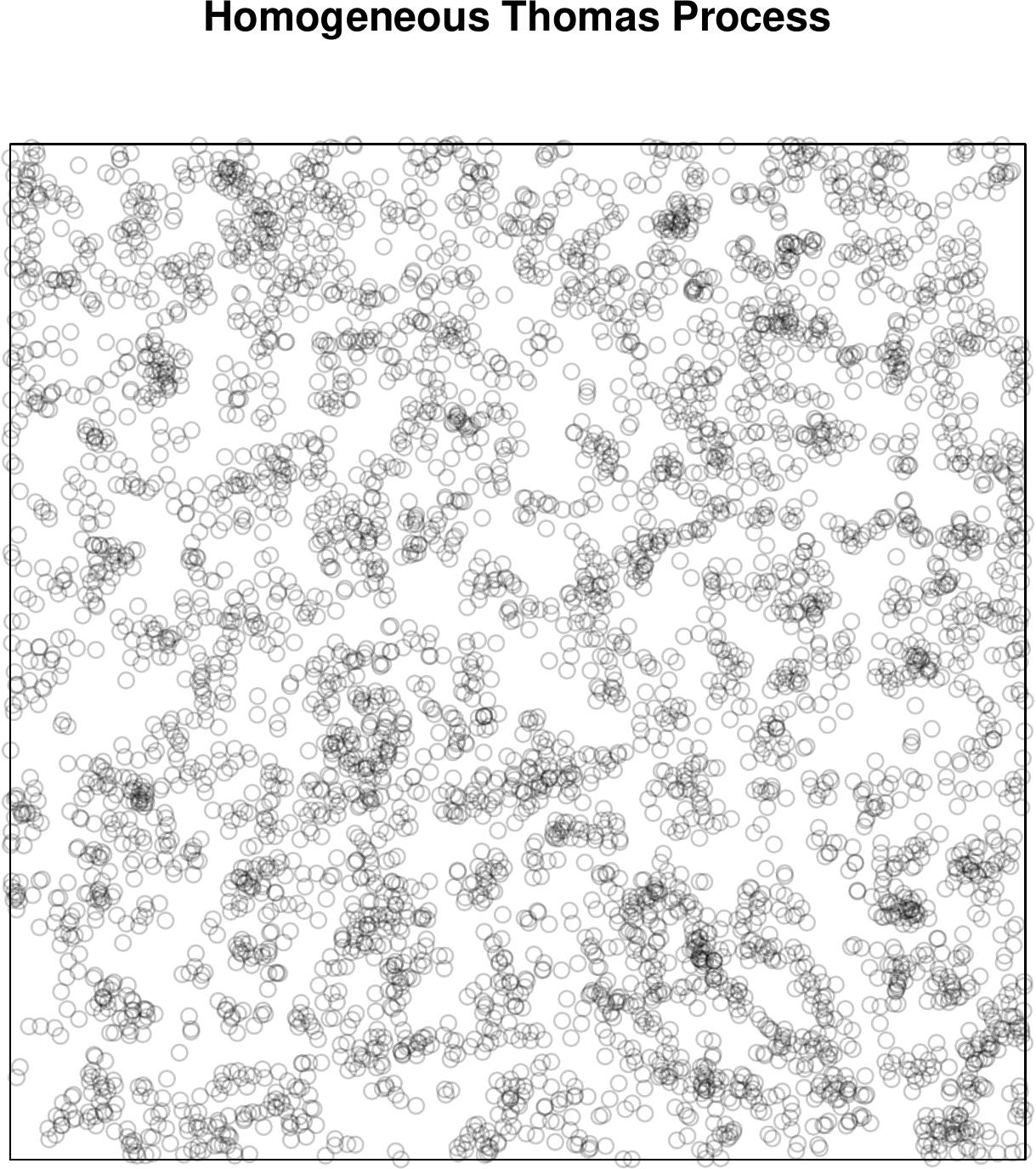}
\caption{Homogeneous Thomas point process pattern.} 
\label{fig:pp3_plots}
\end{figure}

To test CSR, here we set $K=500$ and $\hat C_1=0.23$ for the Bayesian method. The Bayesian method, as well as the classical method, correctly indicate
that the underlying point process is not CSR. The results are displayed in Figure \ref{fig:pp3}.
\begin{figure}
\centering
\subfigure [HPP detection with Bayesian method for homogeneous Thomas point process.]{ \label{fig:hpp_bayesian3}
\includegraphics[width=5.5cm,height=5.5cm]{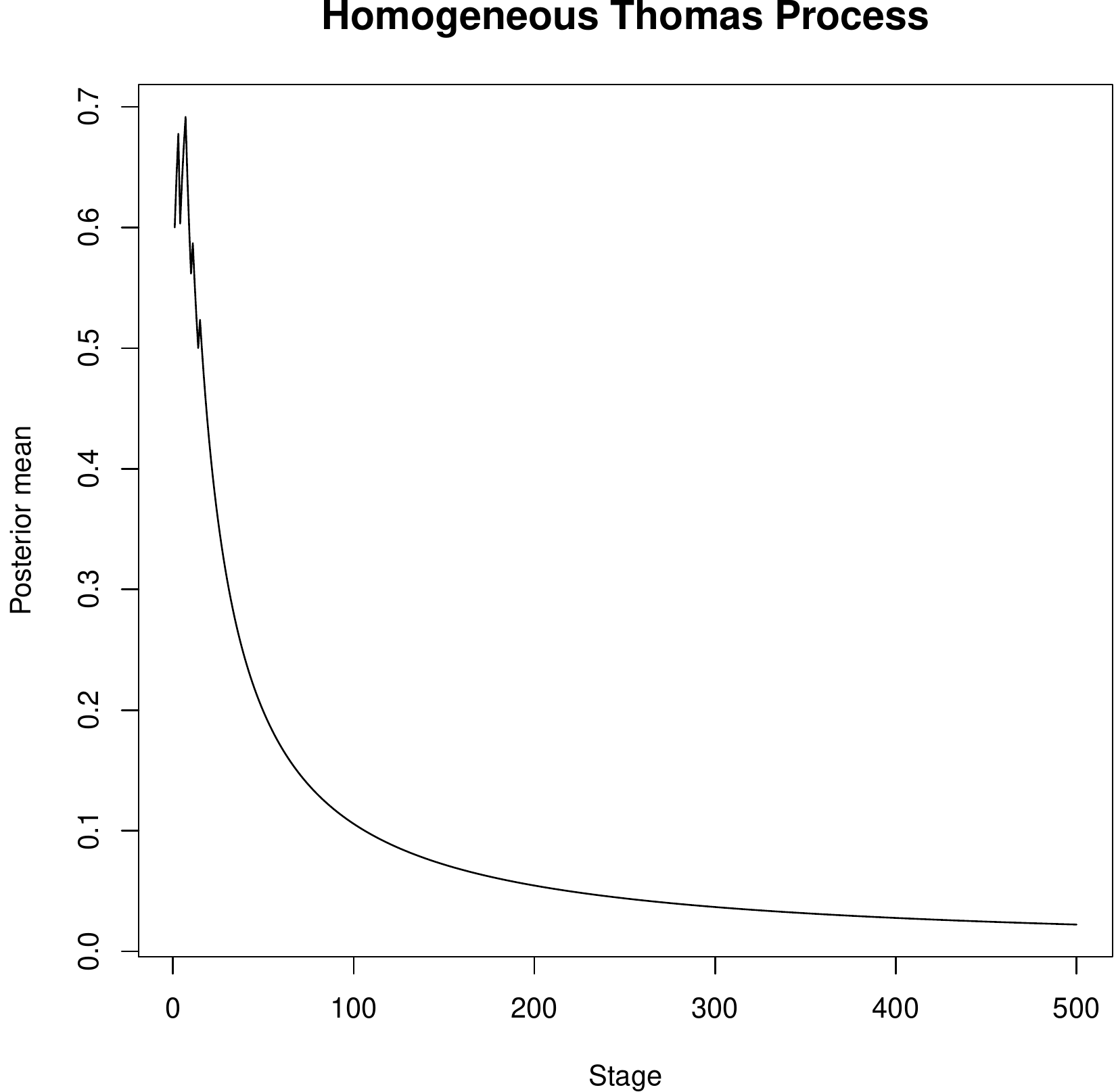}}
\hspace{2mm}
\subfigure [HPP detection with classical method for homogeneous Thomas point process.]{ \label{fig:hpp_classical3}
\includegraphics[width=5.5cm,height=5.5cm]{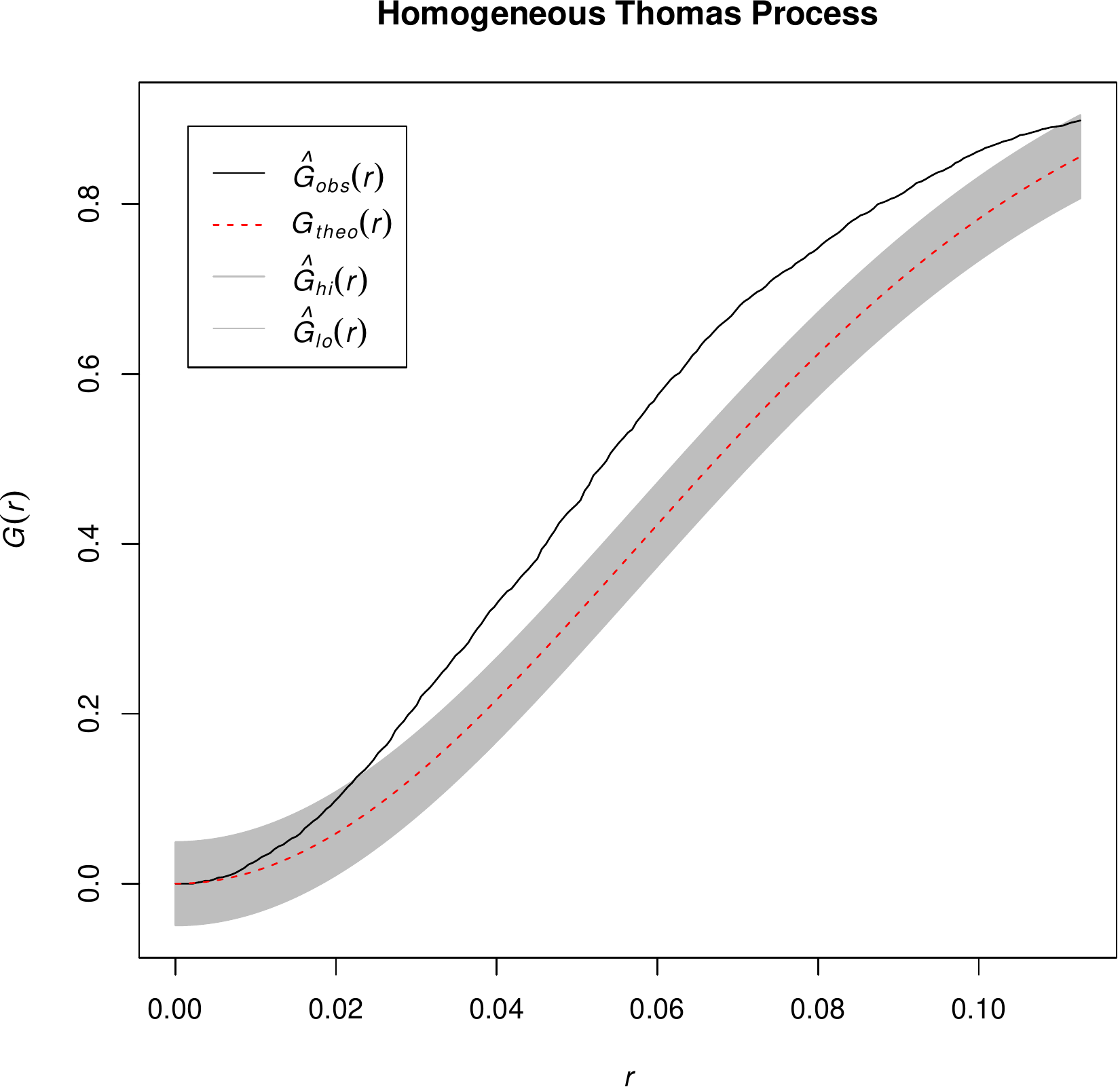}}
\caption{Detection of CSR with our Bayesian method and traditional classical method for homogeneous Thomas point process. Both the methods correctly
identify that the underlying point process is not CSR.} 
\label{fig:pp3}
\end{figure}

With $K=500$ and $\hat C_1=0.18$, we are able to identify stationarity of the underlying homogeneous Thomas point process using our Bayesian method. 
Also, with $K=500$ and $\hat C_1=0.5$, our Bayesian procedure suggests dependence among $\bX_{C_i}$; $i=1,\ldots,50$,leading us to correctly conclude
that the point process is not Poisson. 
\begin{figure}
\centering
\subfigure [Stationary point process (homogeneous Thomas process).]{ \label{fig:thomas_bayesian_stationary}
\includegraphics[width=5.5cm,height=5.5cm]{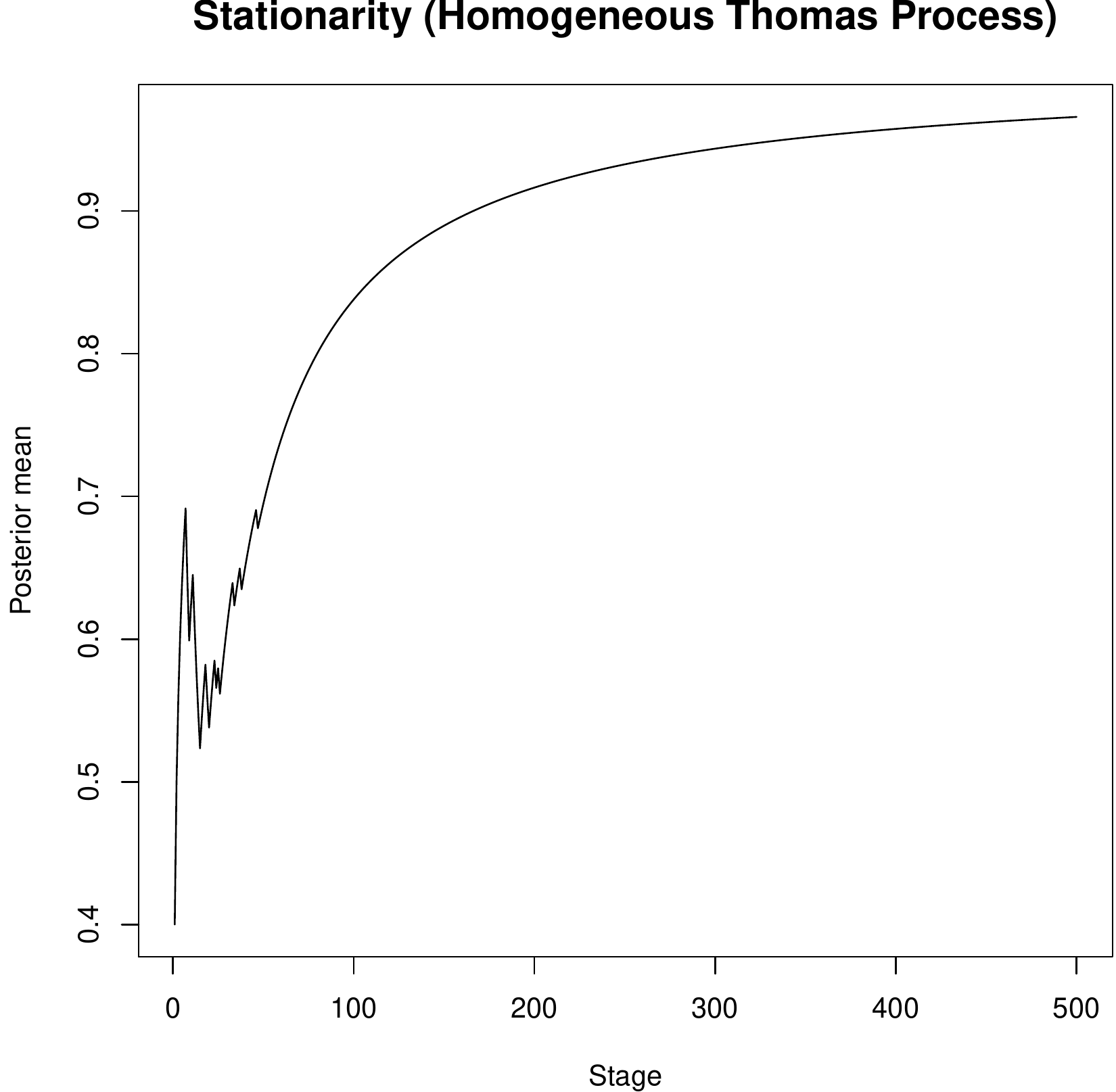}}
\hspace{2mm}
\subfigure [Dependent point process (homogeneous Thomas process).]{ \label{fig:thomas_bayesian_dependent}
\includegraphics[width=5.5cm,height=5.5cm]{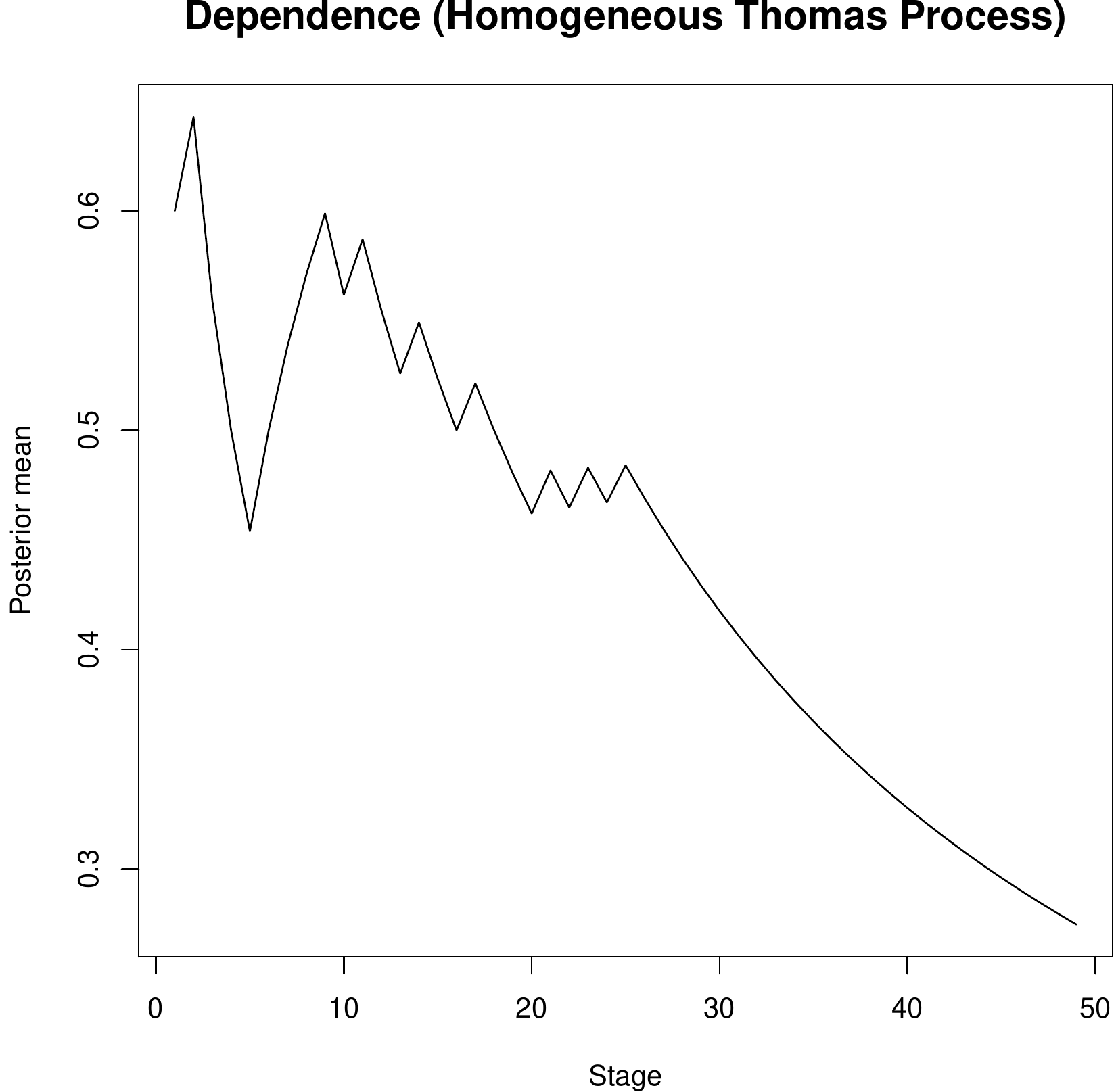}}
\caption{Detection of stationarity and dependence of homogeneous Thomas process with our Bayesian method.} 
\label{fig:pp3_stationarity_indep}
\end{figure}

\subsection{Example 9: Inhomogeneous Thomas process with $\mu$ inhomogeneous}
\label{subsec:thomas12}

We now test our methods on an inhomogeneous Thomas process in $W=[0,3]\times[0,3]$ with $\kappa=10$, $\sigma^2=10$, but $\mu(u_1,u_2)=5\exp\left(2u_1-1\right)$.
That this process is also nonstatioanry follows from (\ref{eq:shotnoise1}), since $\Lambda$ is nonstationary in this case.
The $10735$ points we obtained using spatstat are shown in Figure \ref{fig:pp4_plots}.
\begin{figure}
\centering
\includegraphics[width=5.5cm,height=5.5cm]{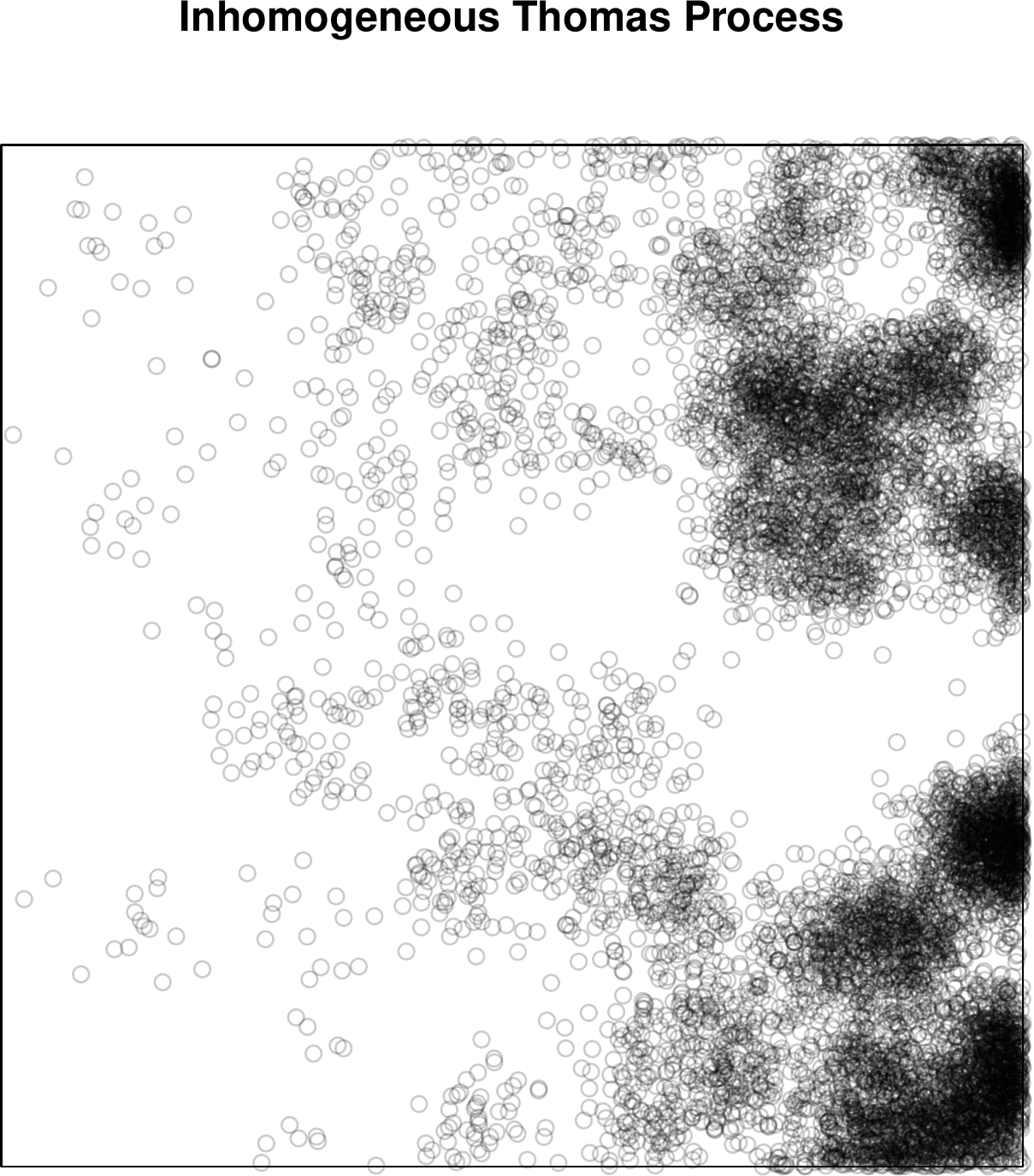}
\caption{Inhomogeneous Thomas point process pattern.} 
\label{fig:pp4_plots}
\end{figure}

With $K=1000$ and $\hat C_1=0.23$, our Bayesian method correctly identifies non-CSR. The classical method also does as well. The results of both these methods
are shown in Figure \ref{fig:pp4}.
\begin{figure}
\centering
\subfigure [HPP detection with Bayesian method for inhomogeneous Thomas point process.]{ \label{fig:hpp_bayesian4}
\includegraphics[width=5.5cm,height=5.5cm]{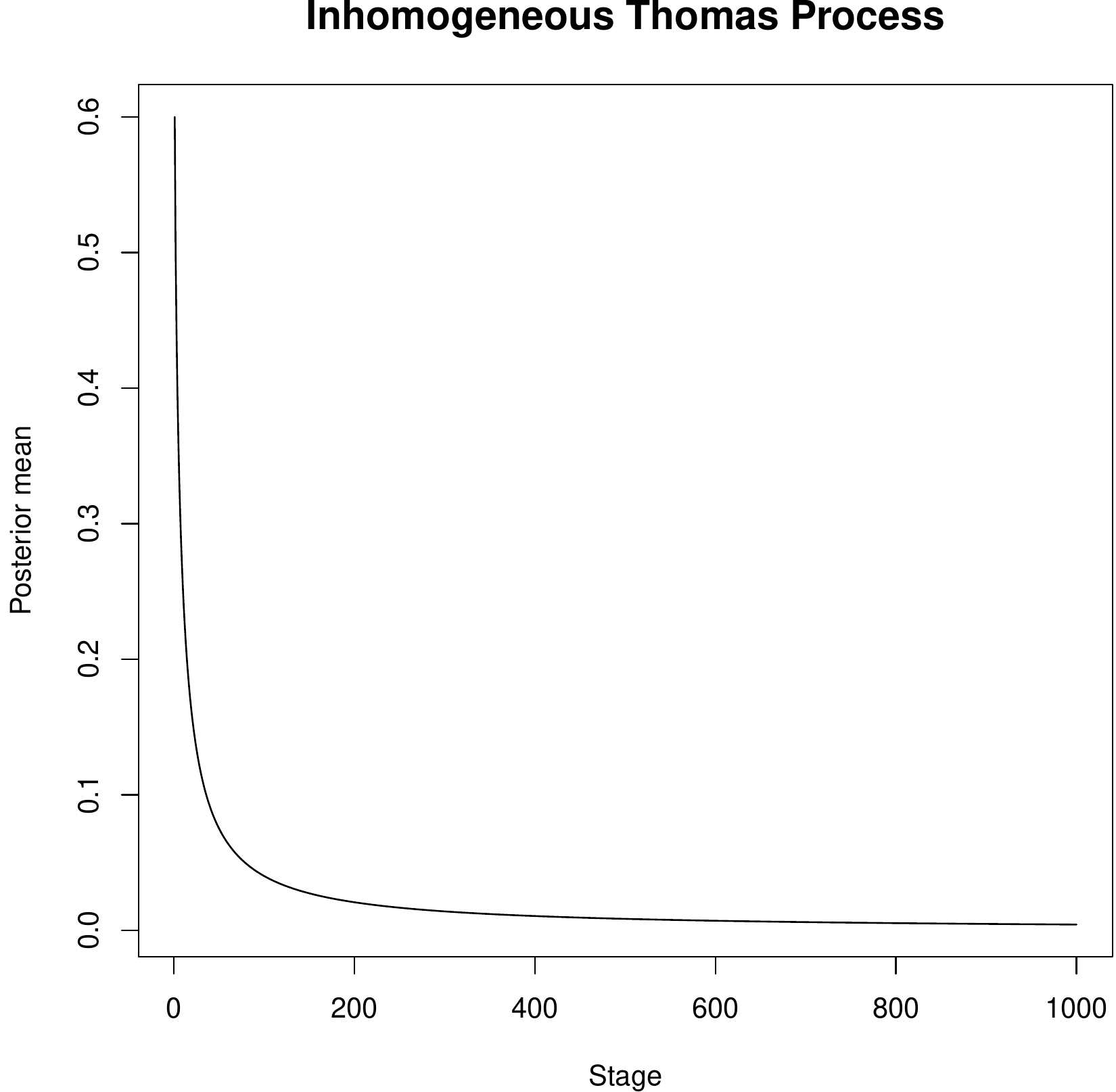}}
\hspace{2mm}
\subfigure [HPP detection with classical method for inhomogeneous Thomas point process.]{ \label{fig:hpp_classical4}
\includegraphics[width=5.5cm,height=5.5cm]{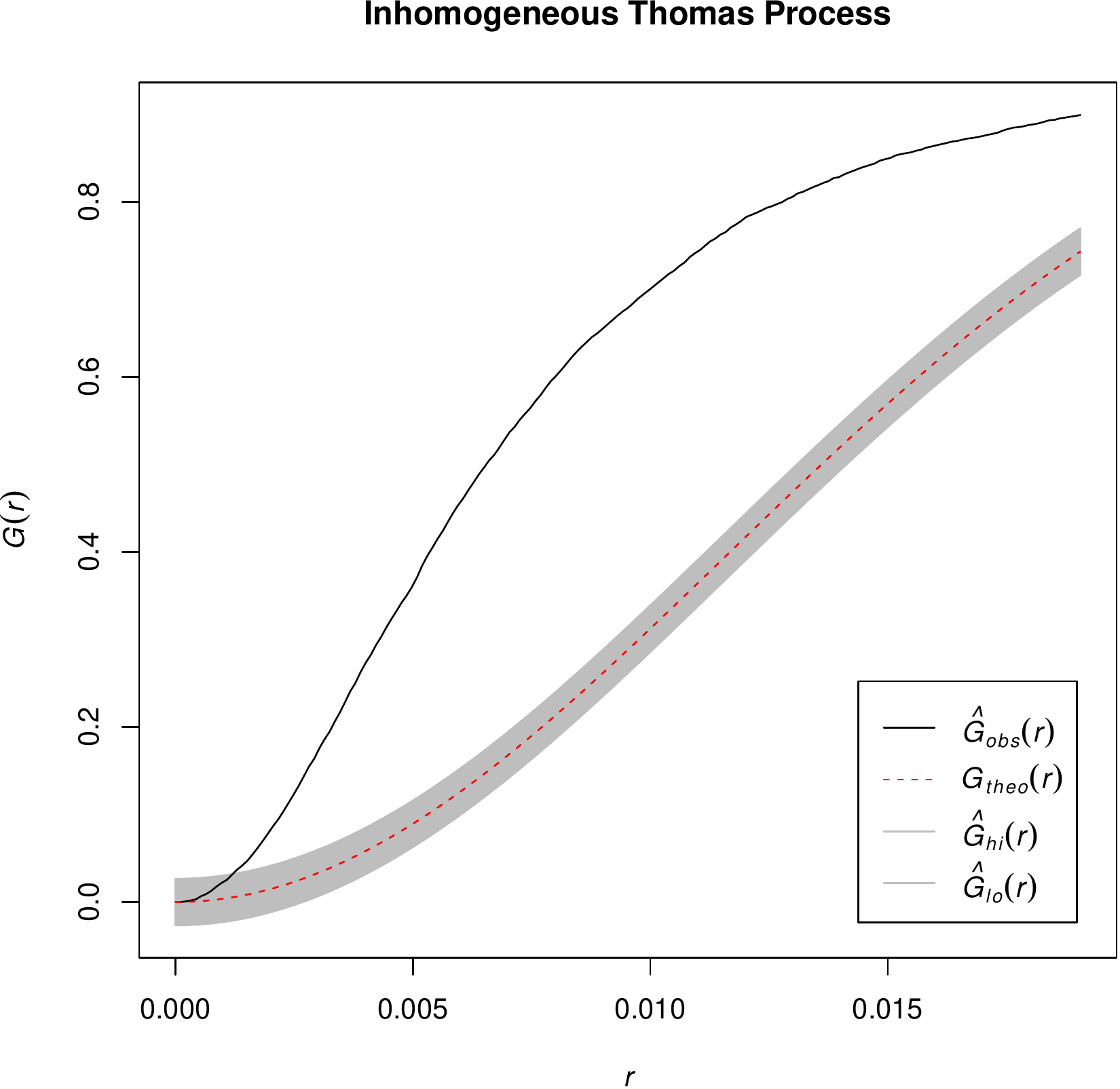}}
\caption{Detection of CSR with our Bayesian method and traditional classical method for Inhomogeneous Thomas point process. Both the methods correctly
identify that the underlying point process is not CSR.} 
\label{fig:pp4}
\end{figure}

Our Bayesian algorithm correctly captures nonstationarity with $K=1000$ and $\hat C_1=0.18$, the maximum value of $\hat C_1$ leading to nonstationarity.
Dependence among $\bX_{C_i}$; $i=1,\ldots,50$ is borne out by our Bayesian strategy with $\hat C_1=0.5$. The results are presented in 
Figure \ref{fig:pp4_stationarity_indep}. 
\begin{figure}
\centering
\subfigure [Nontationary point process (inhomogeneous Thomas Process).]{ \label{fig:thomas_bayesian_nonstationary}
\includegraphics[width=5.5cm,height=5.5cm]{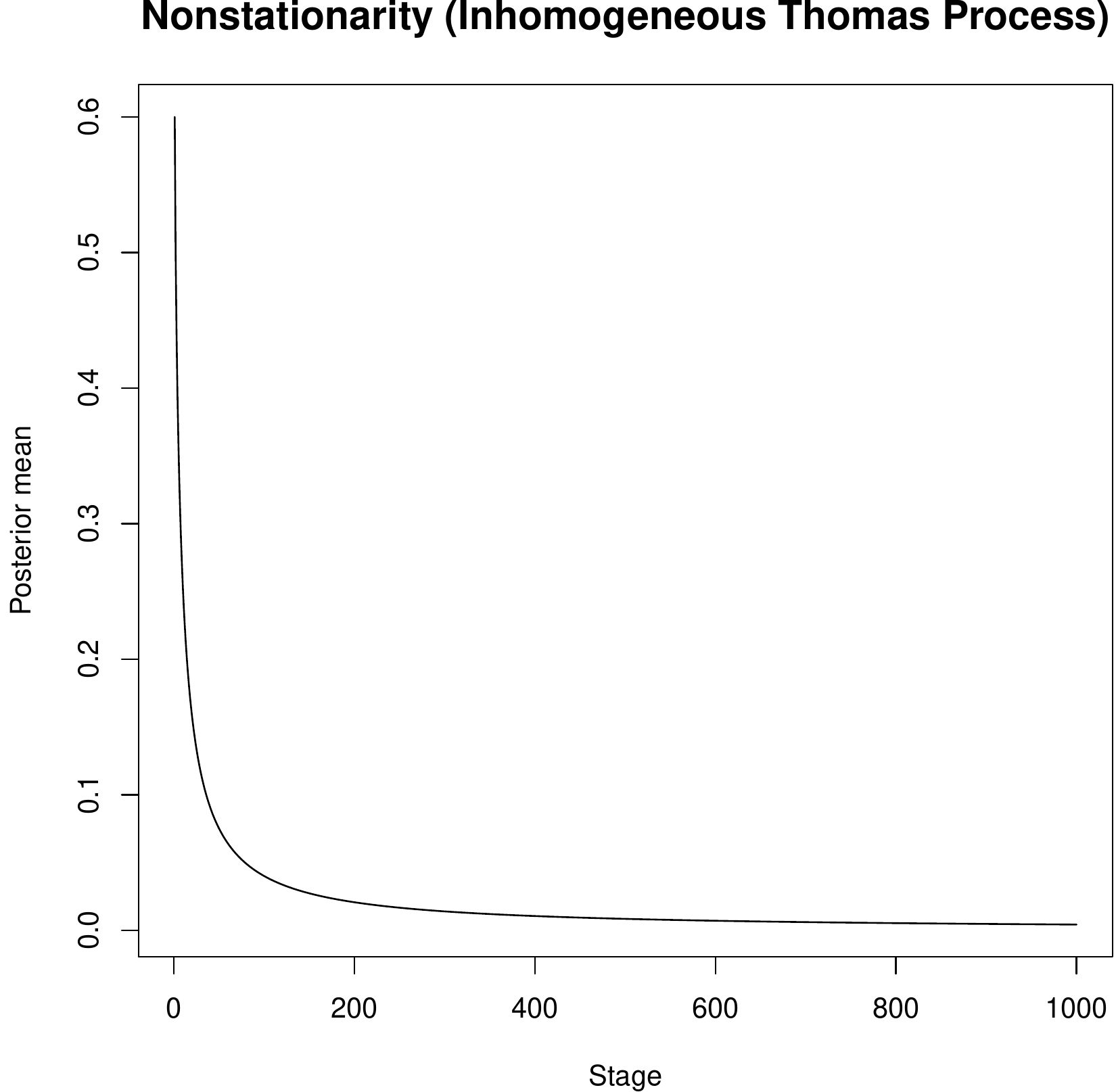}}
\hspace{2mm}
\subfigure [Dependent point process (inhomogeneous Thomas Process).]{ \label{fig:thomas12_bayesian_dependent}
\includegraphics[width=5.5cm,height=5.5cm]{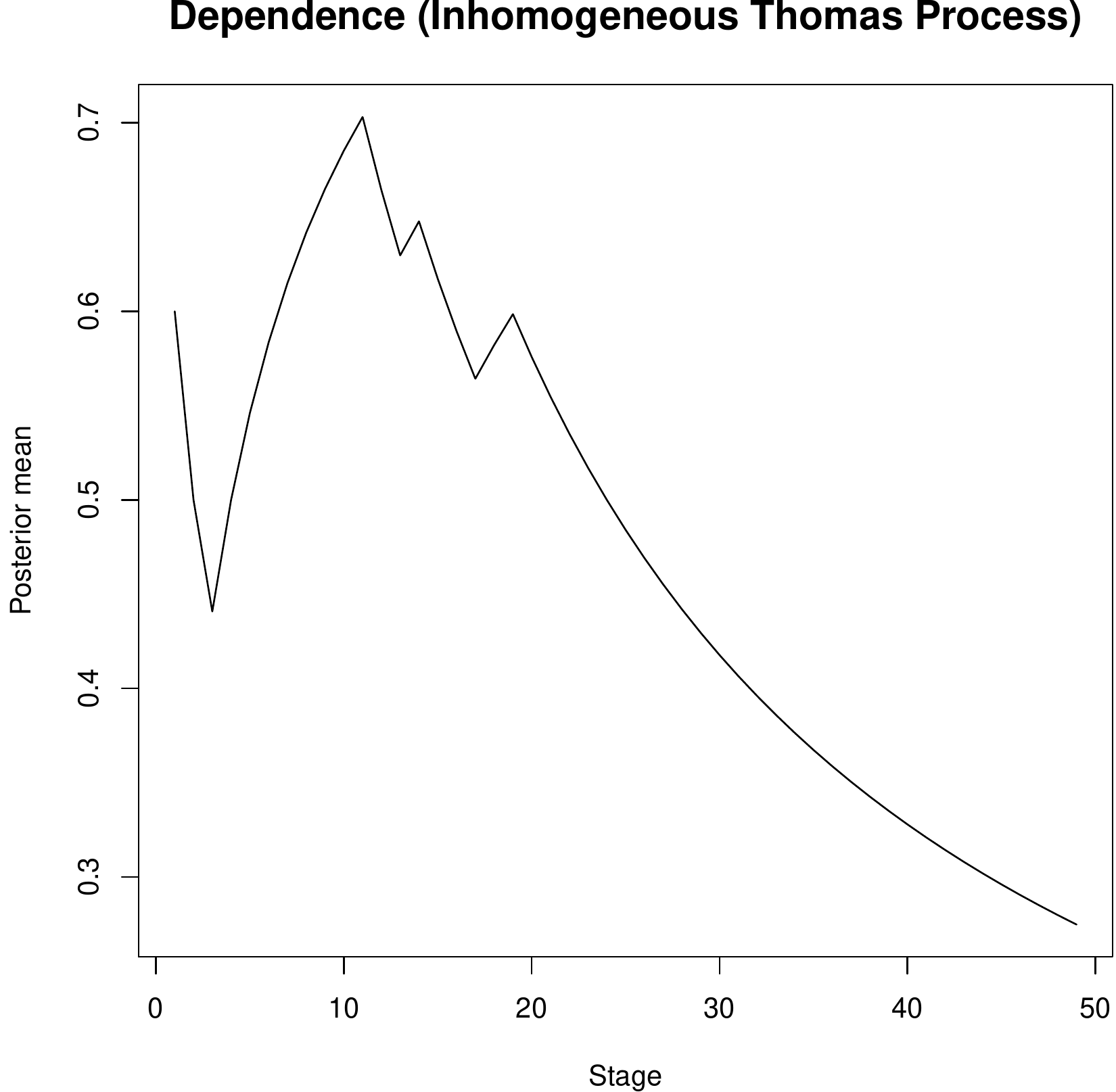}}
\caption{Detection of nonstationarity and dependence of inhomogeneous Thomas process with our Bayesian method.} 
\label{fig:pp4_stationarity_indep}
\end{figure}

\subsection{Example 10: Inhomogeneous Thomas process with $\kappa$ inhomogeneous}
\label{subsec:thomas21}

We now consider another inhomogeneous Thomas process on $W=[0,3]\times[0,3]$ with $\mu=5$, $\sigma^2=10$ but $\kappa(u_1,u_2)=5\exp\left(2x-1\right)$.
This is also a nonstationary, non-Poisson, non-homogeneous point process.
Figure \ref{fig:pp5_plots} displays the $5608$ points that we obtained from this process.
\begin{figure}
\centering
\includegraphics[width=5.5cm,height=5.5cm]{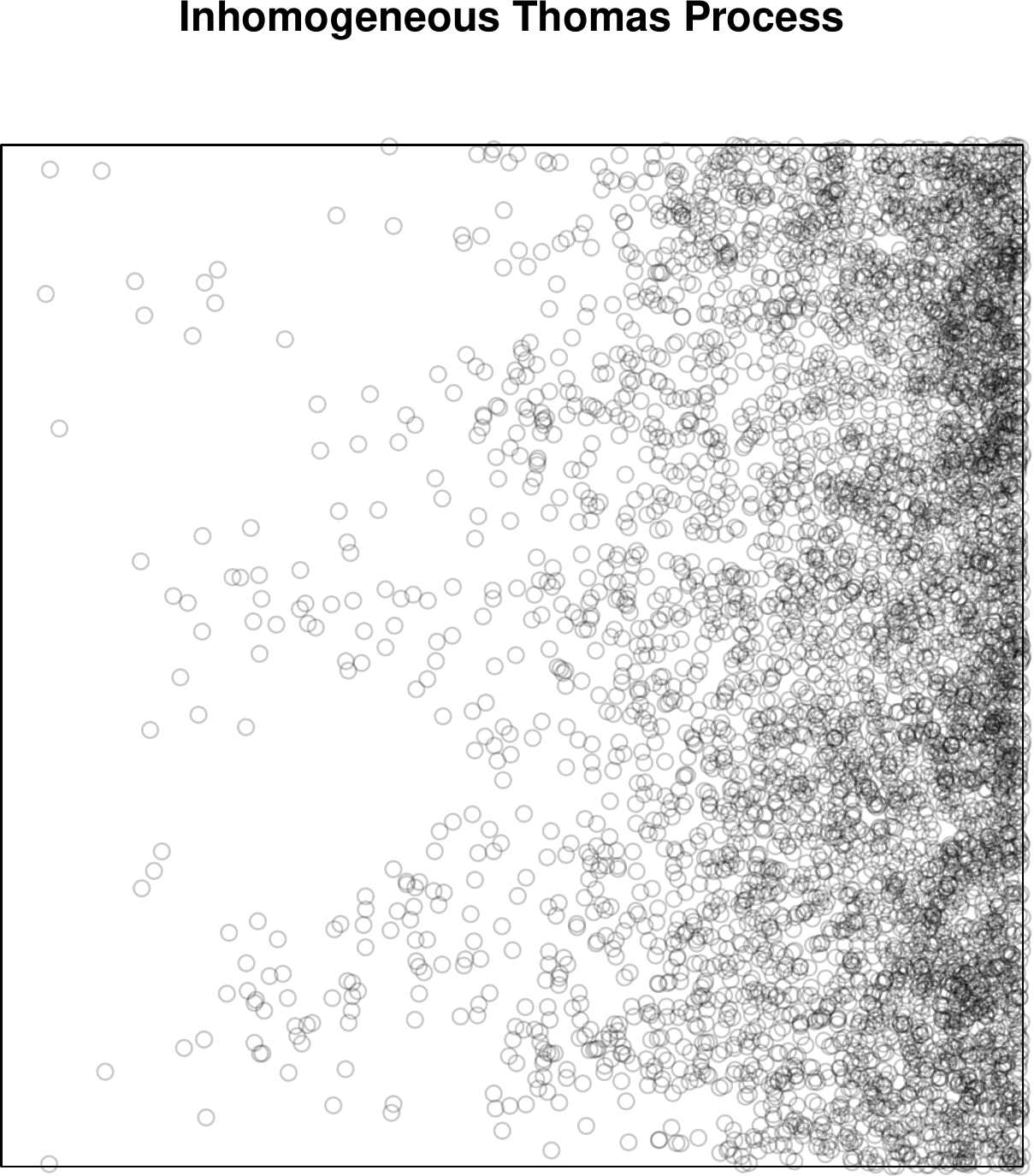}
\caption{Inhomogeneous Thomas point process pattern.} 
\label{fig:pp5_plots}
\end{figure}

With $K=500$ and $\hat C_1=0.23$, our Bayesian correctly detected non-CSR. The classical method also performed adequately in this case.
The results are shown in Figure \ref{fig:pp5}.
\begin{figure}
\centering
\subfigure [HPP detection with Bayesian method for inhomogeneous Thomas point process.]{ \label{fig:hpp_bayesian5}
\includegraphics[width=5.5cm,height=5.5cm]{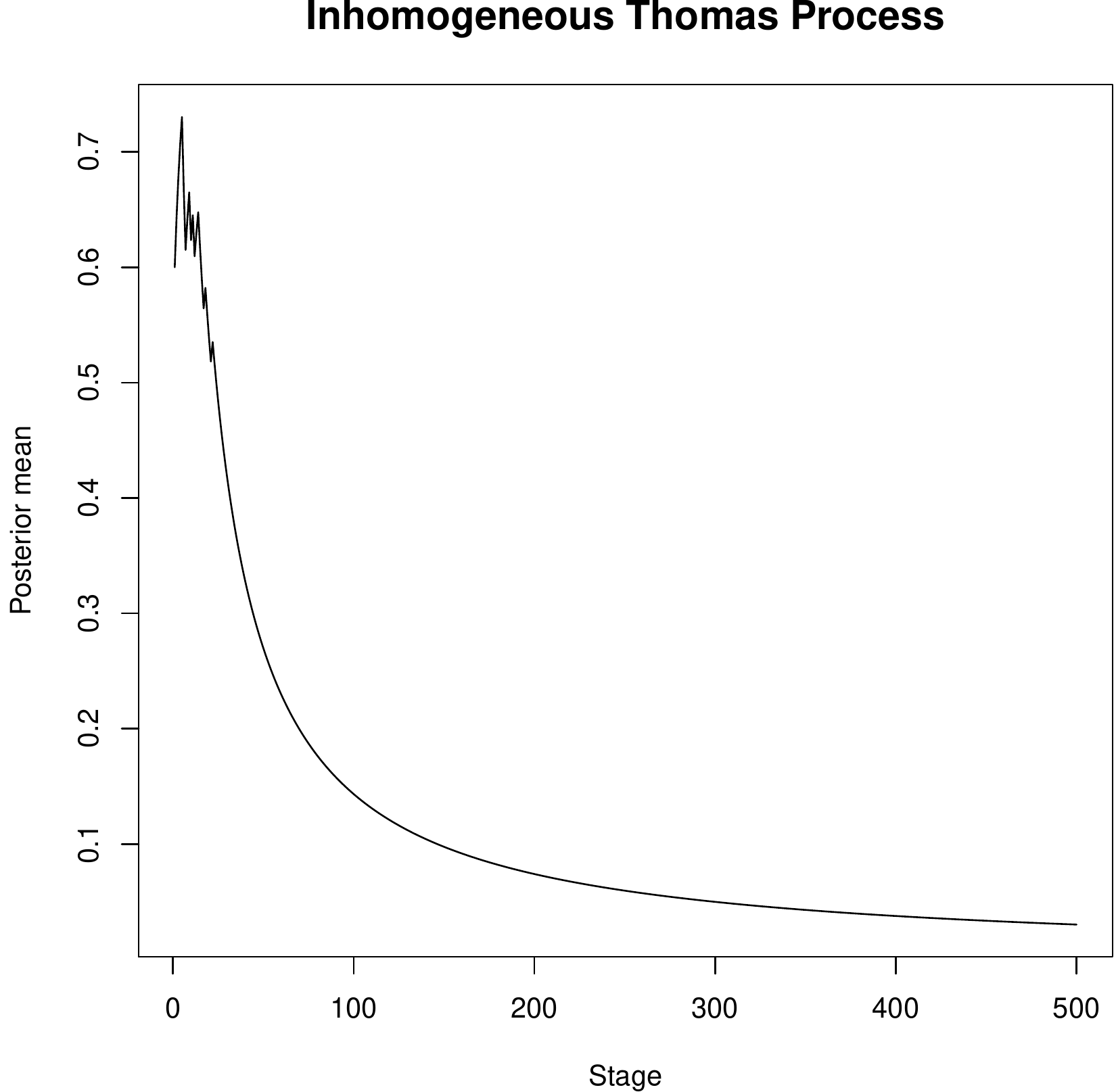}}
\hspace{2mm}
\subfigure [HPP detection with classical method for inhomogeneous Thomas point process.]{ \label{fig:hpp_classical5}
\includegraphics[width=5.5cm,height=5.5cm]{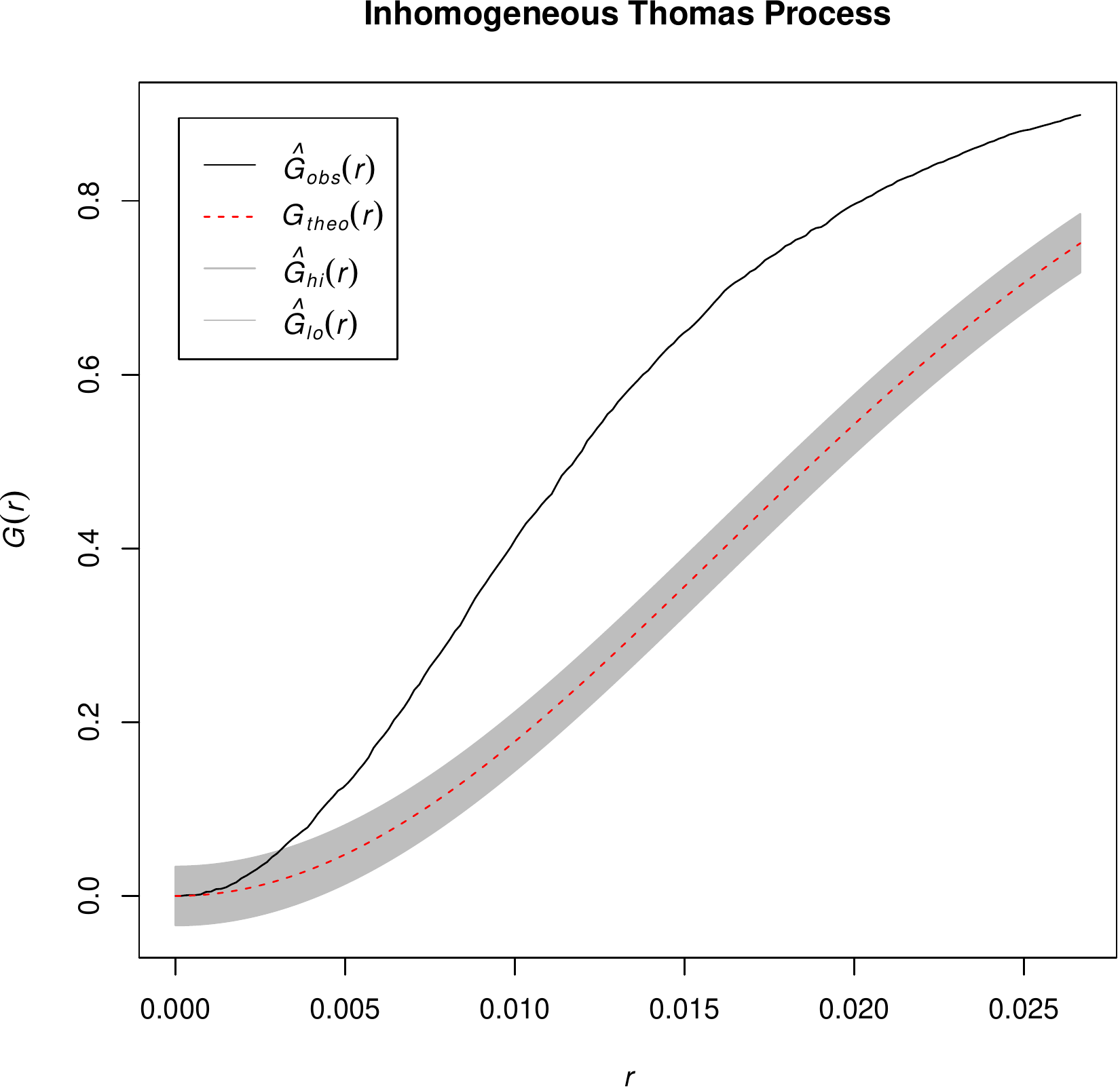}}
\caption{Detection of CSR with our Bayesian method and traditional classical method for inhomogeneous Thomas point process. Both the methods correctly
identify that the underlying point process is not CSR.} 
\label{fig:pp5}
\end{figure}

As before, our Bayesian method correctly detected nonstationarity with $K=500$ and $\hat C_1=0.18$. Also, as before, dependence among
$\bX_{C_i}$; $i=1,\ldots,50$, is correctly indicated by our Bayesian method, with $\hat C_1=0.5$.
\begin{figure}
\centering
\subfigure [Nontationary point process (inhomogeneous Thomas process).]{ \label{fig:thomas21_bayesian_nonstationary}
\includegraphics[width=5.5cm,height=5.5cm]{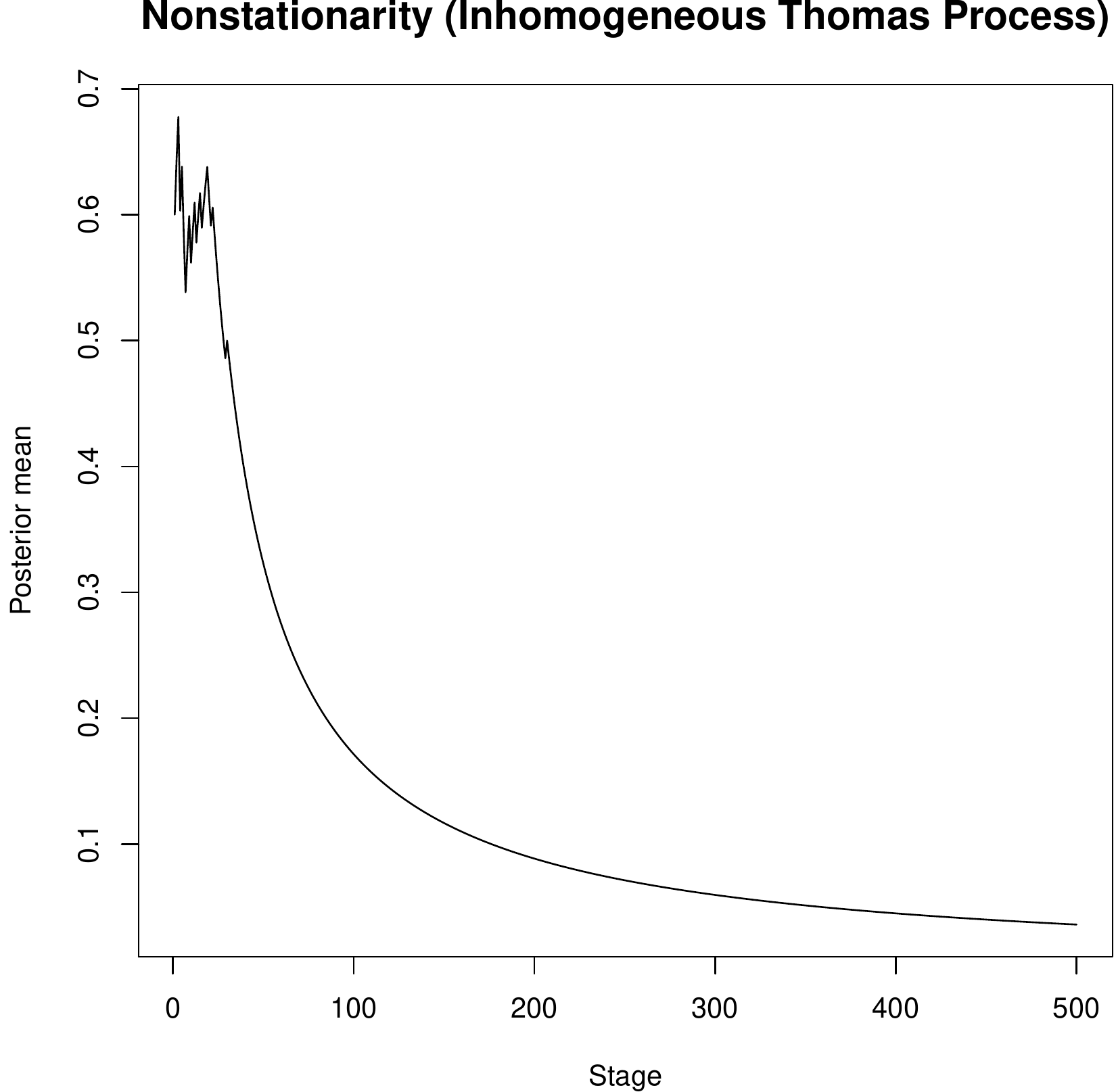}}
\hspace{2mm}
\subfigure [Dependent point process (inhomogeneous Thomas process).]{ \label{fig:thomas21_bayesian_dependent}
\includegraphics[width=5.5cm,height=5.5cm]{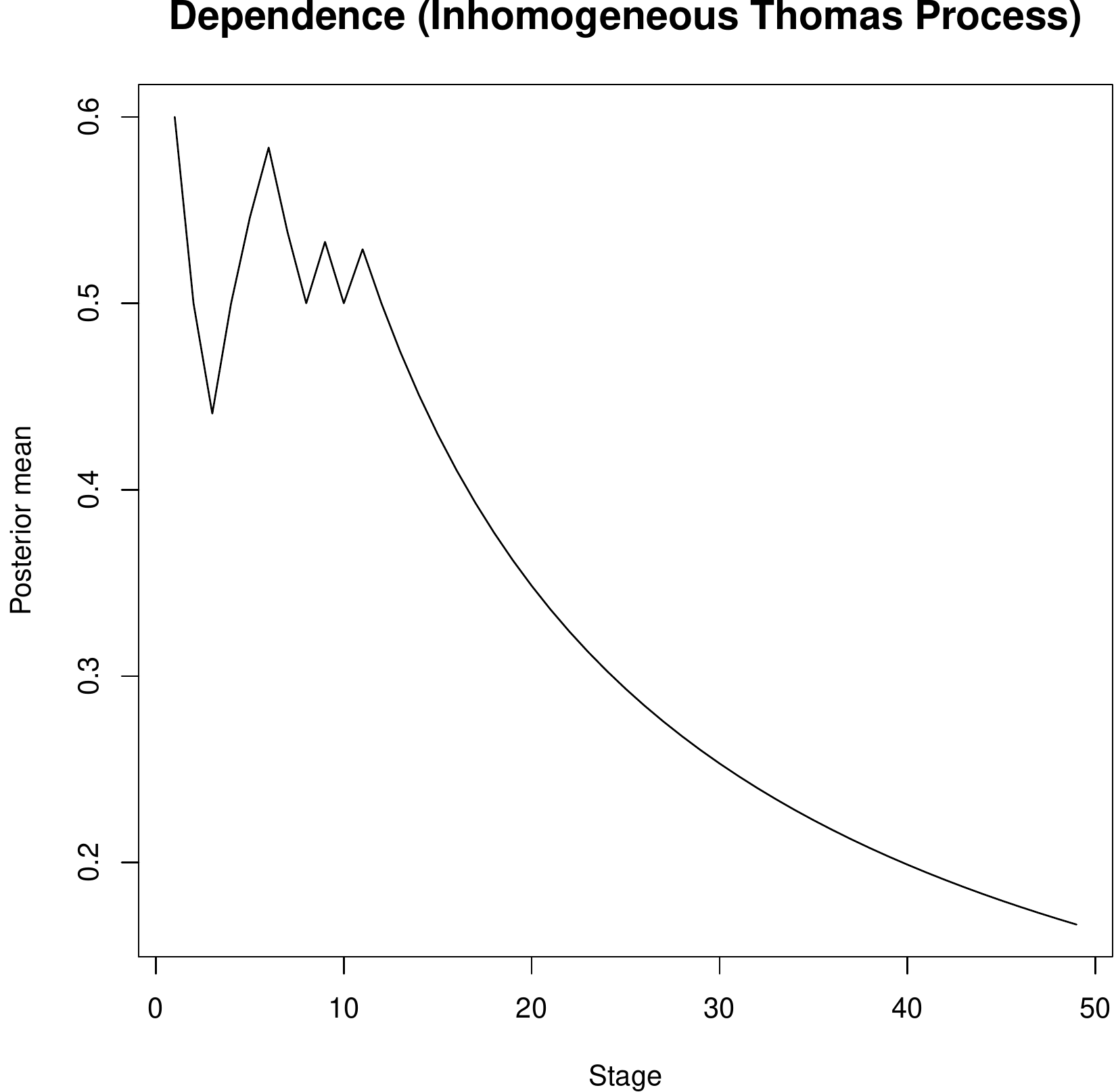}}
\caption{Detection of nonstationarity and dependence of inhomogeneous Thomas process with our Bayesian method.} 
\label{fig:pp5_stationarity_indep}
\end{figure}

\subsection{Example 11: Inhomogeneous Thomas process with $\kappa$ and $\mu$ the same inhomogeneous function}
\label{subsec:thomas22}

Let us consider simulation from another inhomogeneous Thomas process where $\kappa(u_1,u_2)=\mu(u_1,u_2)=5\exp\left(2u_1-1\right)$. With $\sigma^2=10$,
we obtained $5302$ points on the window $W=[0,2]\times[0,2]$, displayed in Figure \ref{fig:pp6_plots}.
\begin{figure}
\centering
\includegraphics[width=5.5cm,height=5.5cm]{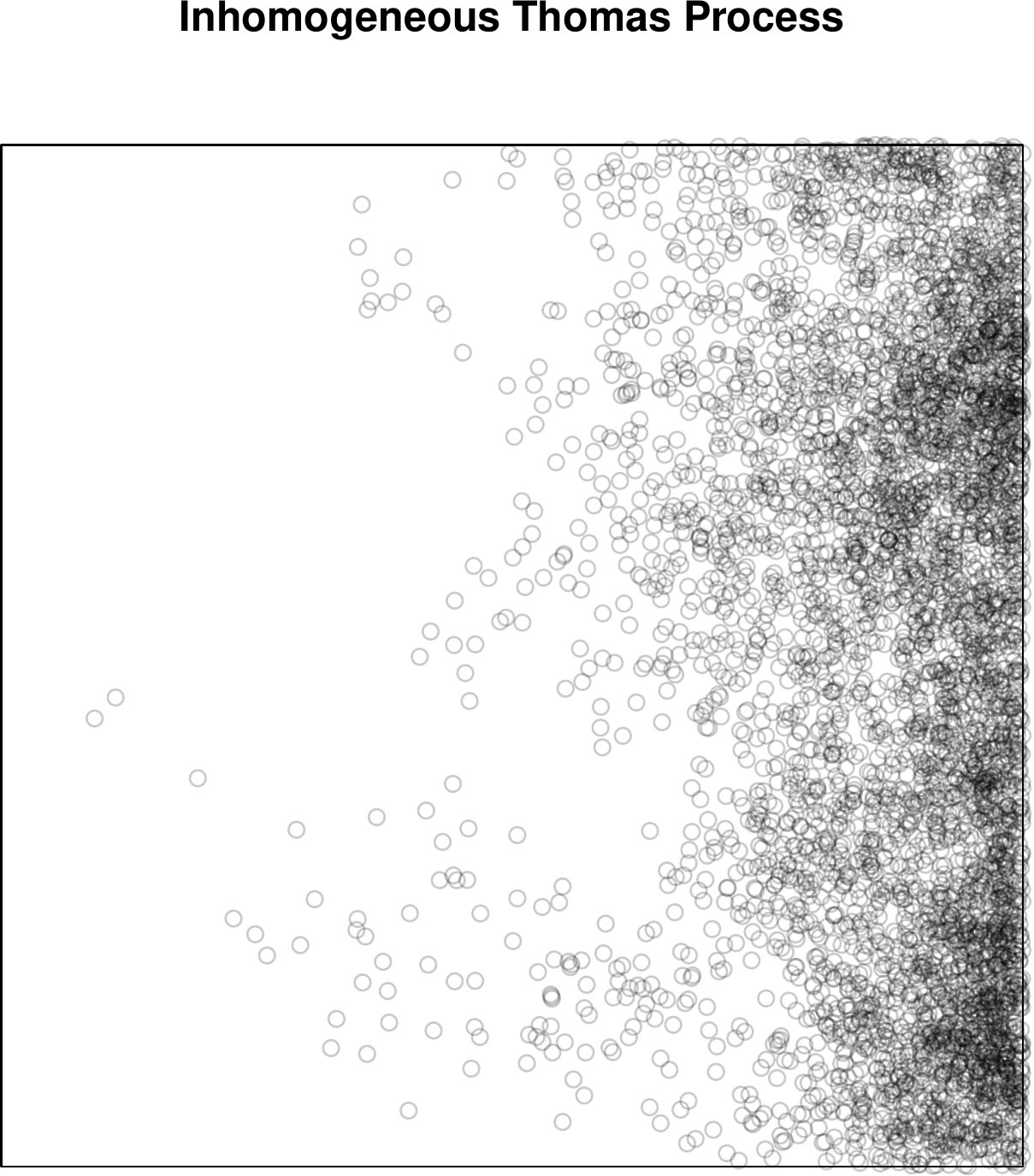}
\caption{Inhomogeneous Thomas point process pattern.} 
\label{fig:pp6_plots}
\end{figure}

Figure \ref{fig:pp6} shows the results of Bayesian and classical CSR detection methods; both the methods performed adequately, correctly identifying non-CSR. 
For the Bayesian method we set $K=500$ and $\hat C_1=0.23$.
\begin{figure}
\centering
\subfigure [HPP detection with Bayesian method for inhomogeneous Thomas point process.]{ \label{fig:hpp_bayesian6}
\includegraphics[width=5.5cm,height=5.5cm]{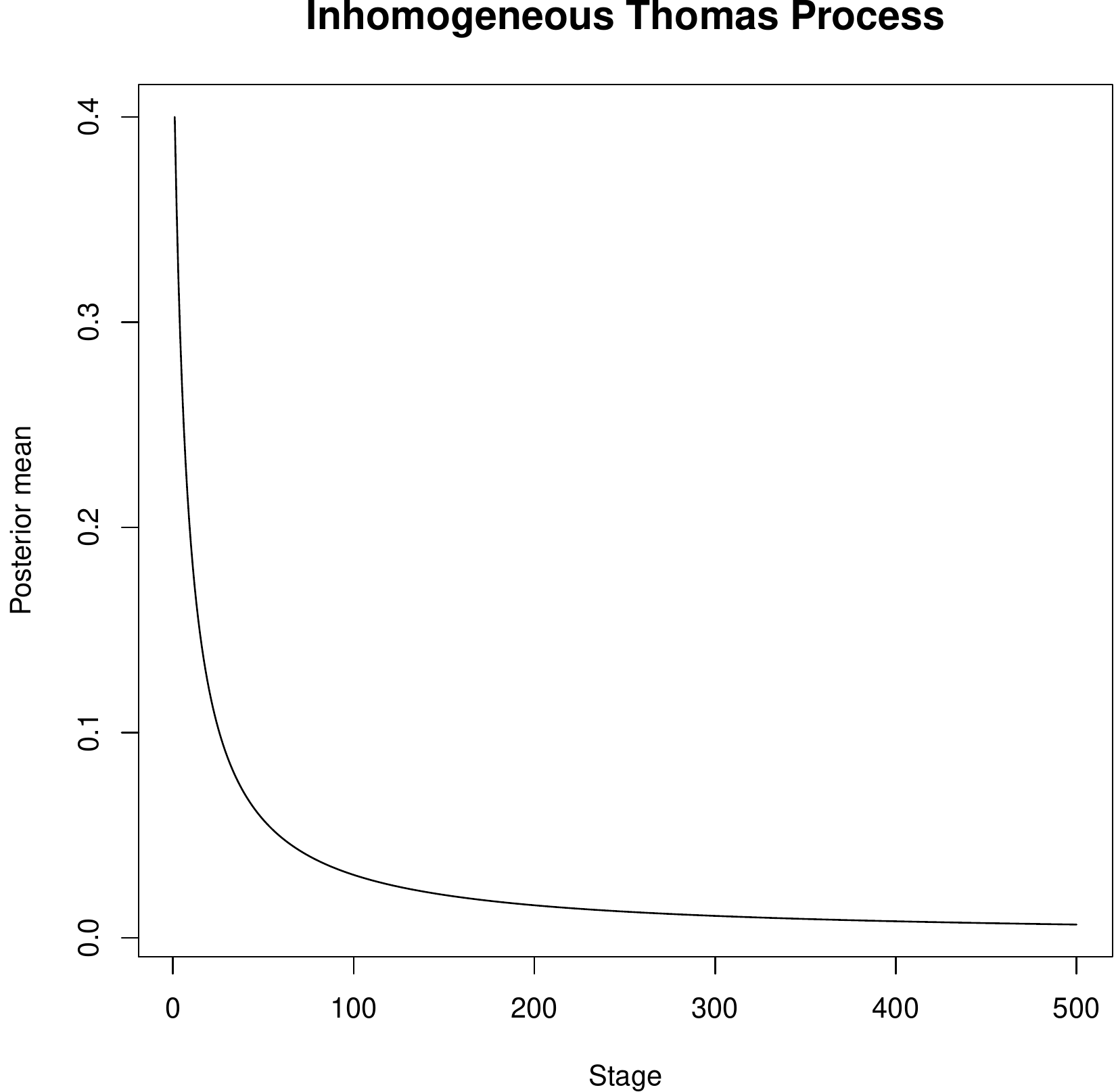}}
\hspace{2mm}
\subfigure [HPP detection with classical method for inhomogeneous Thomas point process.]{ \label{fig:hpp_classical6}
\includegraphics[width=5.5cm,height=5.5cm]{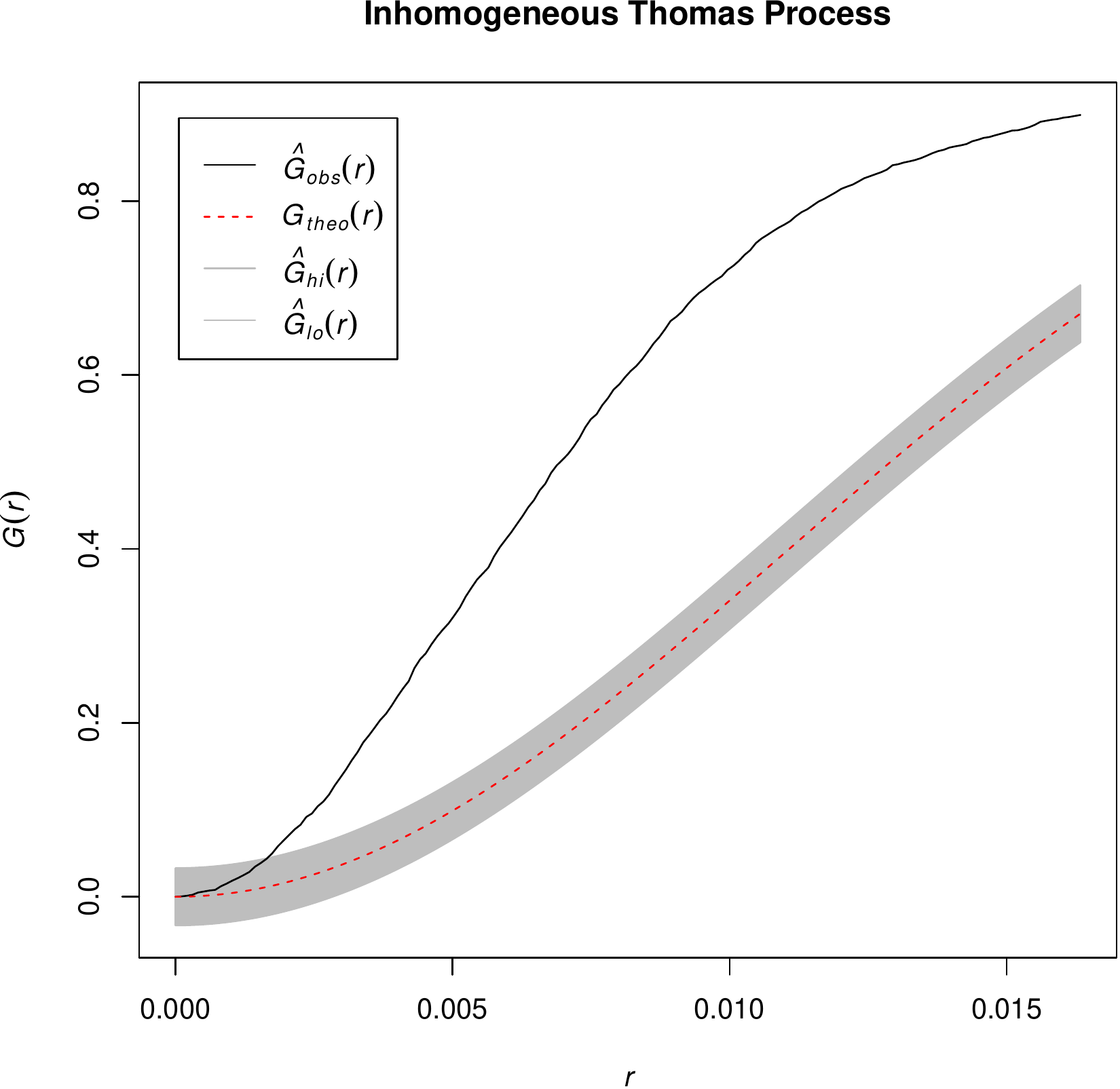}}
\caption{Detection of CSR with our Bayesian method and traditional classical method for inhomogeneous Thomas point process. Both the methods correctly
identify that the underlying point process is not CSR.} 
\label{fig:pp6}
\end{figure}

Nonstationarity of this point process has been correctly detected by our Bayesian method with $K=810$ and $\hat C_1=0.18$. As regards our Bayesian test
for mutual independence, we correctly obtained dependence with $K=50$ and $\hat C_1=0.5$. The results are presented in Figure \ref{fig:pp6_stationarity_indep}.
\begin{figure}
\centering
\subfigure [Nontationary point process (inhomogeneous Thomas process).]{ \label{fig:thomas22_bayesian_nonstationary}
\includegraphics[width=5.5cm,height=5.5cm]{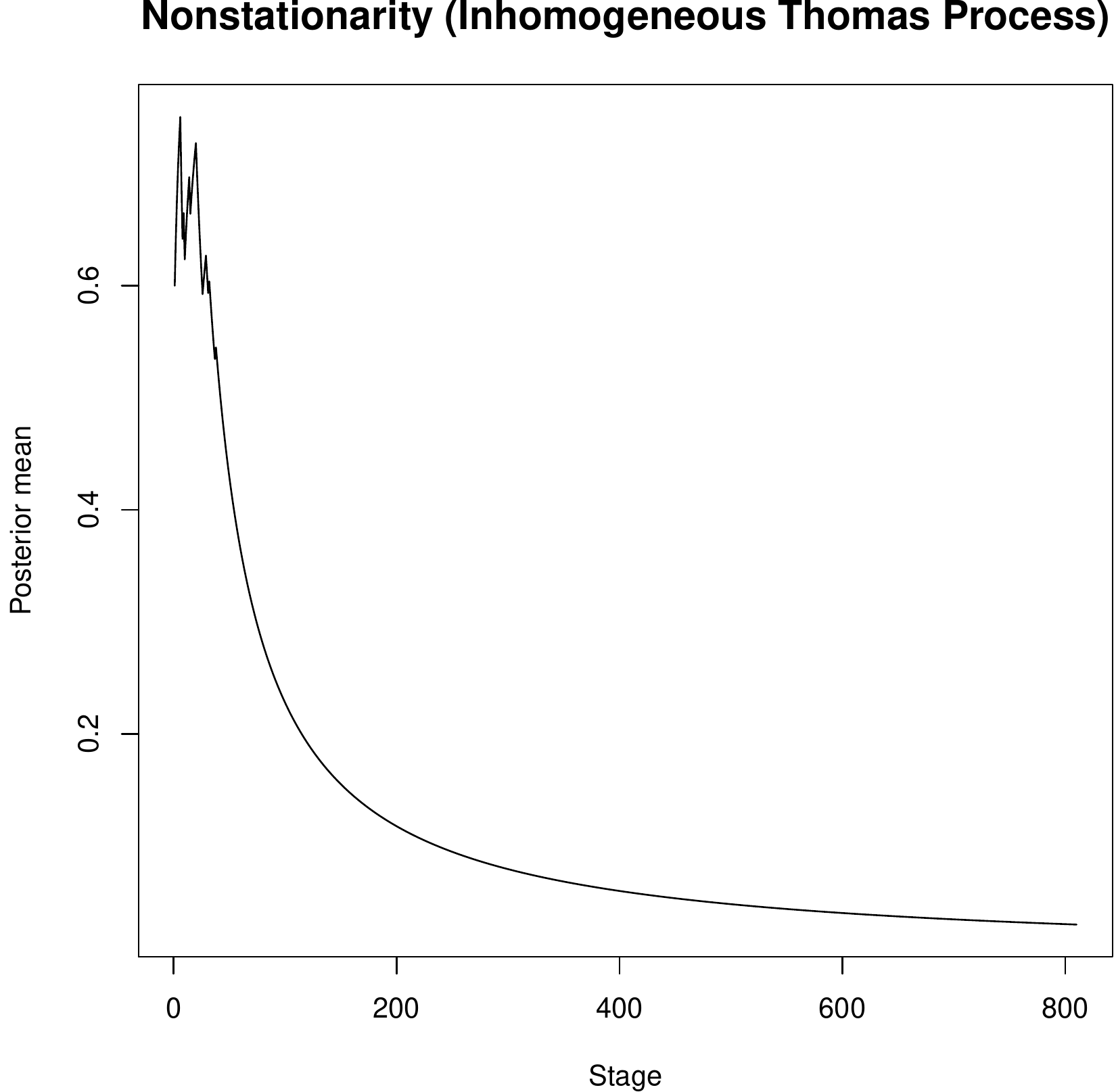}}
\hspace{2mm}
\subfigure [Dependent point process (inhomogeneous Thomas process).]{ \label{fig:thomas22_bayesian_dependent}
\includegraphics[width=5.5cm,height=5.5cm]{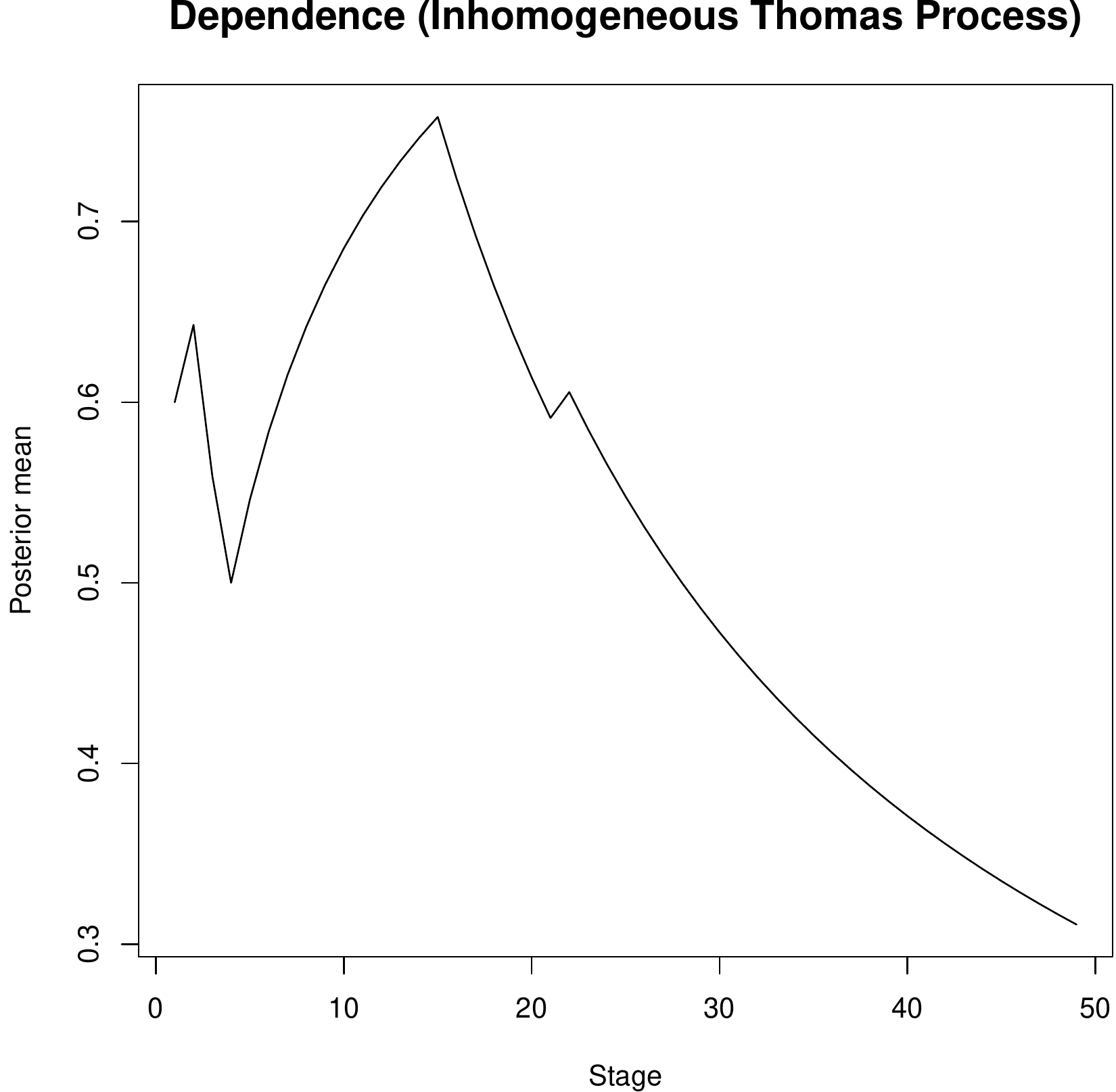}}
\caption{Detection of nonstationarity and dependence of inhomogeneous Thomas process with our Bayesian method.} 
\label{fig:pp6_stationarity_indep}
\end{figure}

\subsection{Example 12: Inhomogeneous Thomas process with $\kappa$ and $\mu$ different inhomogeneous functions}
\label{subsec:thomas31}

Let us now consider another inhomogeneous Thomas process, where $\mu(u_1,u_2)=5\exp\left(2u_1-1\right)$ and $\kappa(u_1,u_2)=10(u^2_1+u^2_2)$. We obtained
$3573$ observations with $\sigma^2=10$ on the window $W=[0,2]\times[0,2]$. The data are displayed in Figure \ref{fig:pp7_plots}.
\begin{figure}
\centering
\includegraphics[width=5.5cm,height=5.5cm]{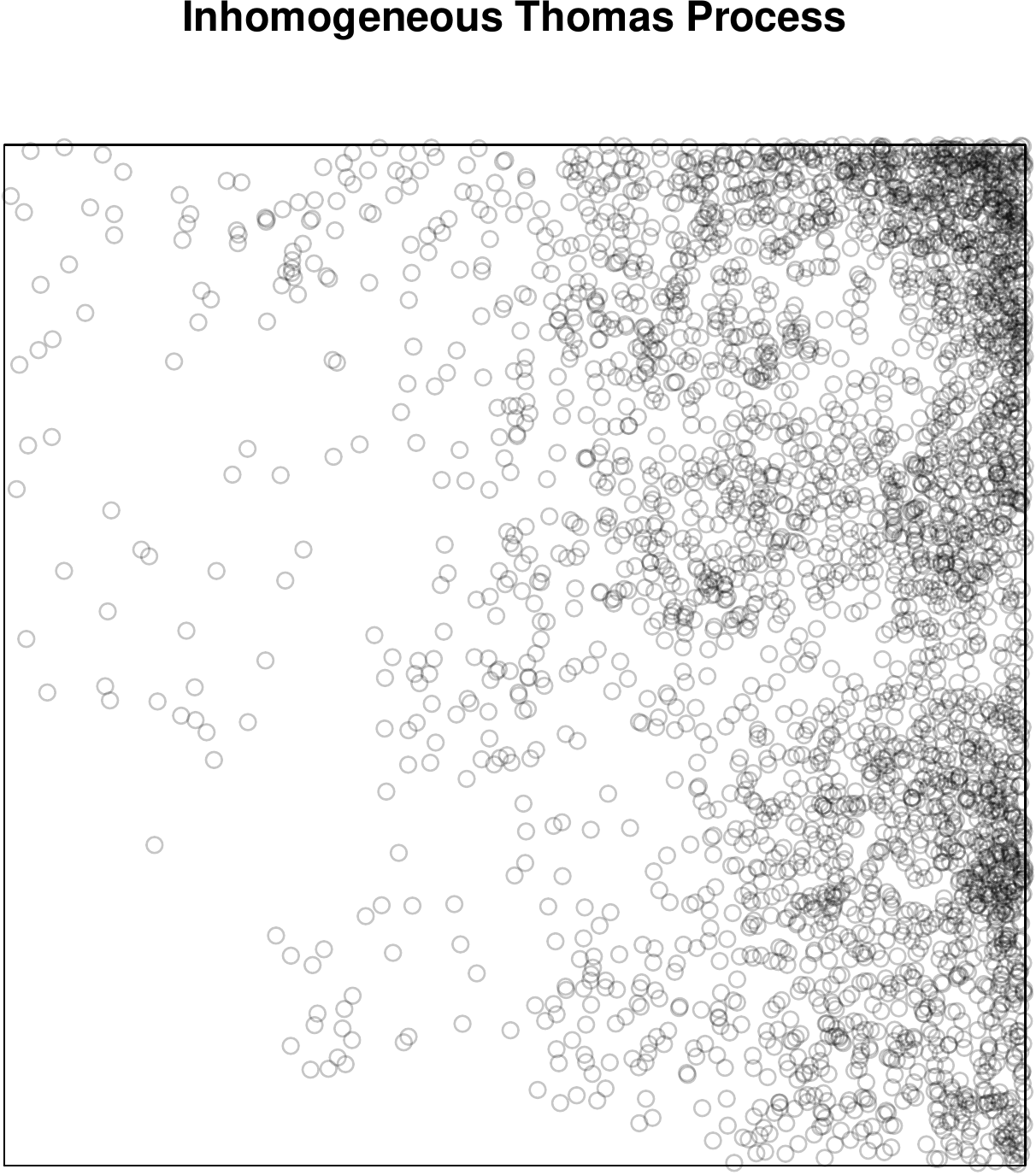}
\caption{Inhomogeneous Thomas point process pattern.} 
\label{fig:pp7_plots}
\end{figure}

With $K=500$ and $\hat C_1=0.23$, we correctly obtained non-CSR with our Bayesian method. The classical method also correctly detected non-CSR.
The results are presented in Figure \ref{fig:pp7}.
\begin{figure}
\centering
\subfigure [HPP detection with Bayesian method for inhomogeneous Thomas point process.]{ \label{fig:hpp_bayesian7}
\includegraphics[width=5.5cm,height=5.5cm]{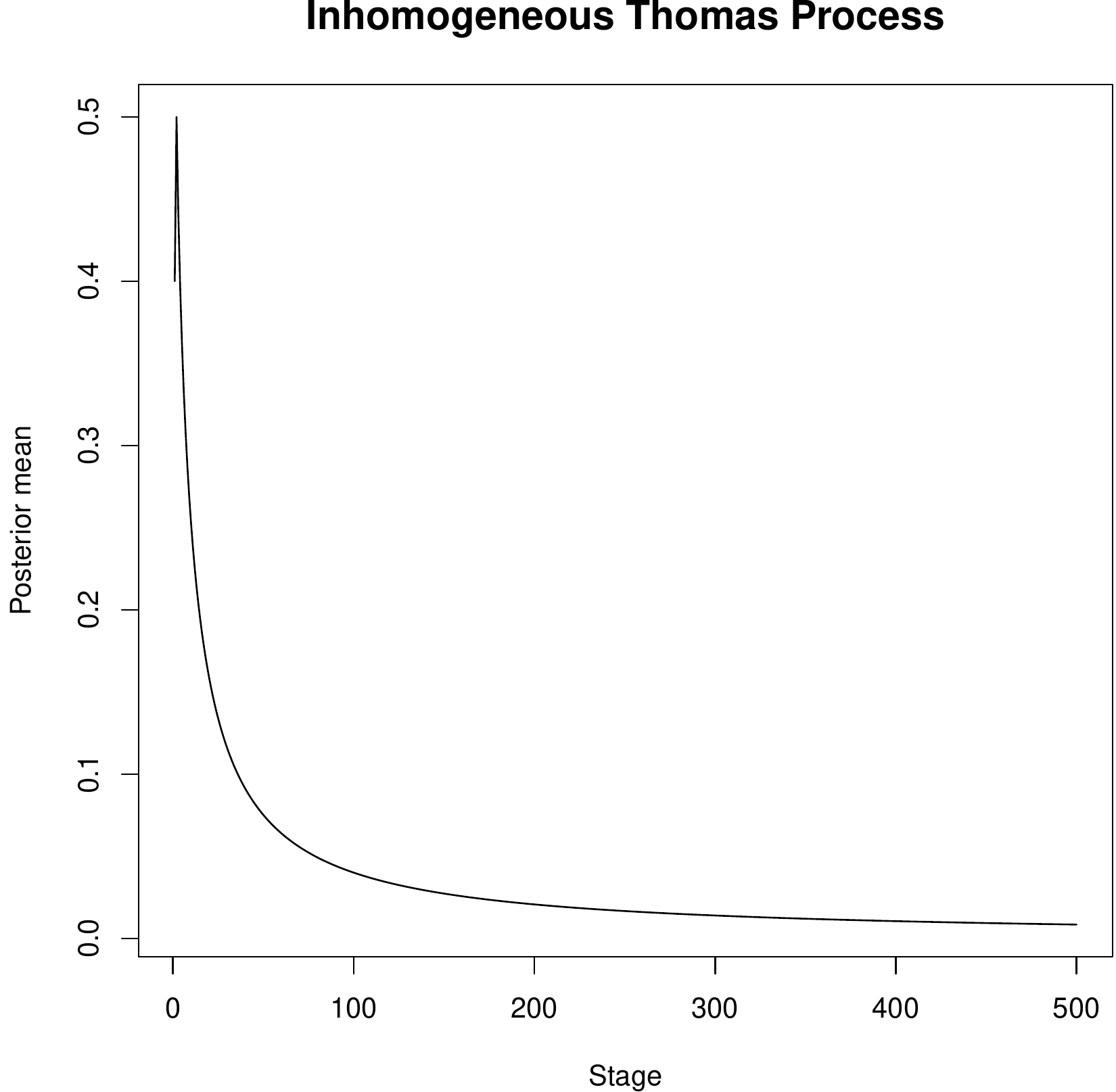}}
\hspace{2mm}
\subfigure [HPP detection with classical method for inhomogeneous Thomas point process.]{ \label{fig:hpp_classical7}
\includegraphics[width=5.5cm,height=5.5cm]{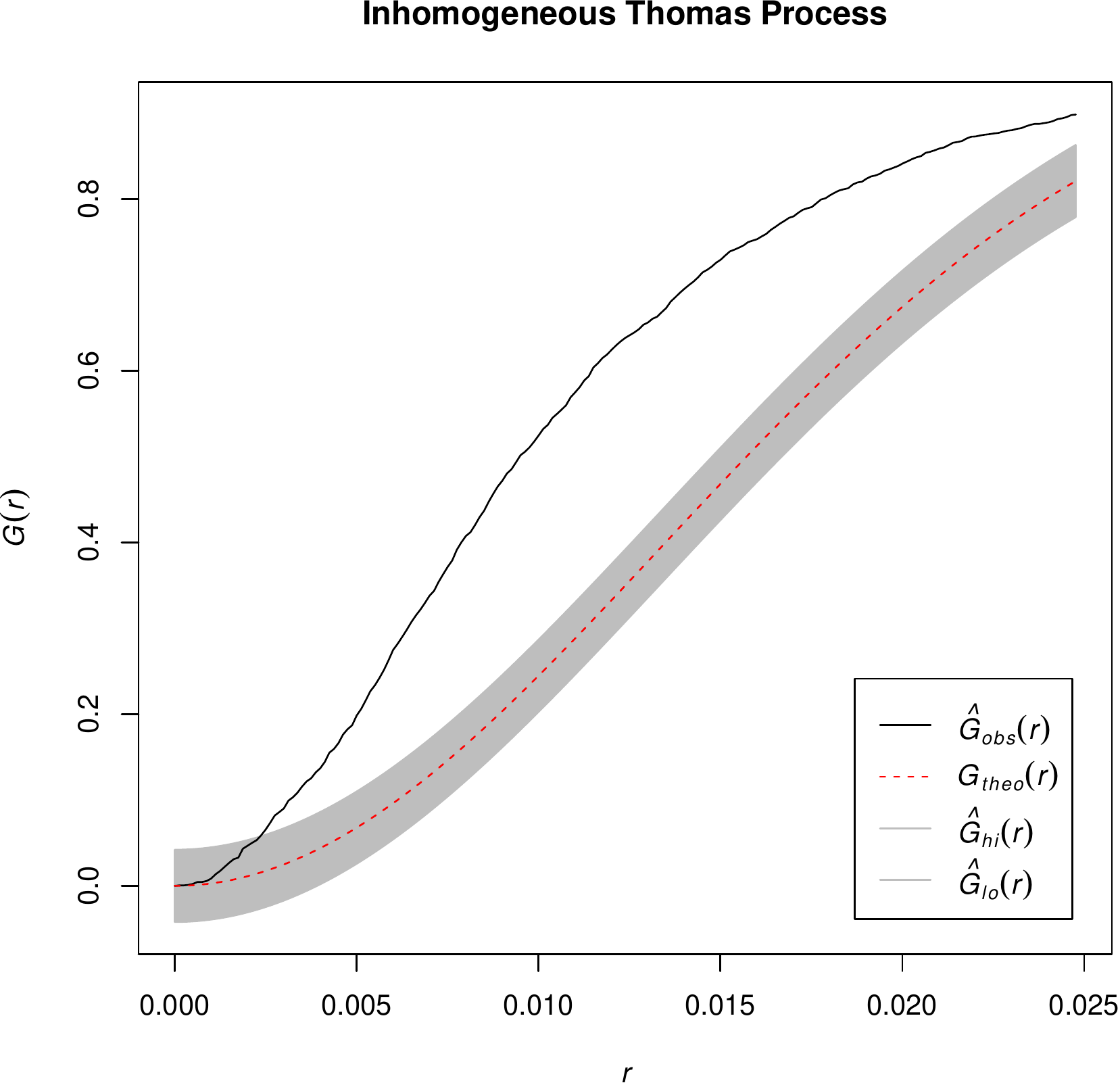}}
\caption{Detection of CSR with our Bayesian method and traditional classical method for inhomogeneous Thomas point process. Both the methods correctly
identify that the underlying point process is not CSR.} 
\label{fig:pp7}
\end{figure}

Our Bayesian algorithm correctly detected nonstationarity with $K=500$ and $\hat C_1=0.18$. The Bayesian test for independence also correctly detected dependence
with $K=50$ and $\hat C_1=0.5$. Both these results are presented in Figure \ref{fig:pp7_stationarity_indep}.
\begin{figure}
\centering
\subfigure [Nontationary point process (inhomogeneous Thomas process).]{ \label{fig:thomas31_bayesian_nonstationary}
\includegraphics[width=5.5cm,height=5.5cm]{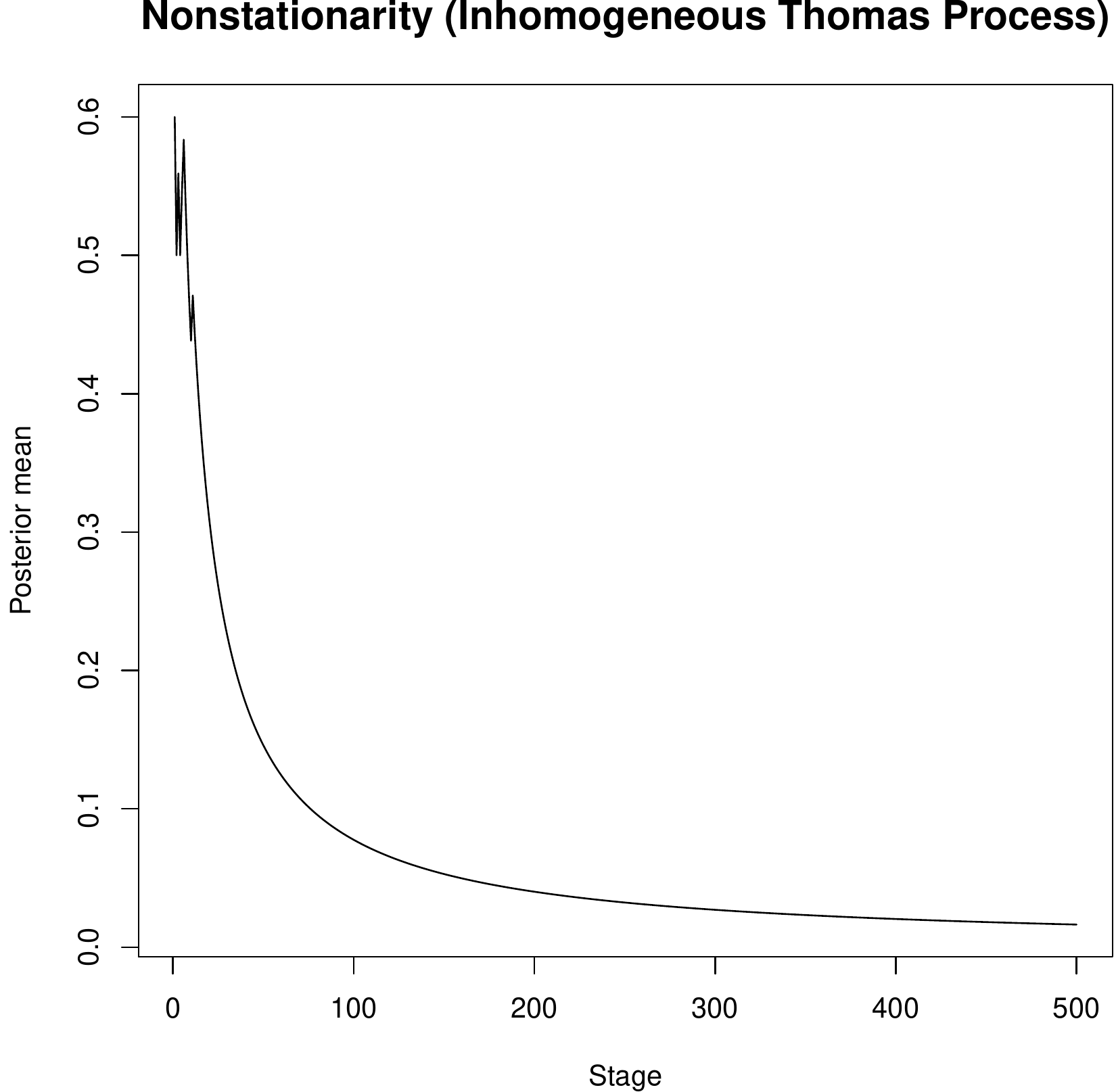}}
\hspace{2mm}
\subfigure [Dependent point process (inhomogeneous Thomas process).]{ \label{fig:thomas31_bayesian_dependent}
\includegraphics[width=5.5cm,height=5.5cm]{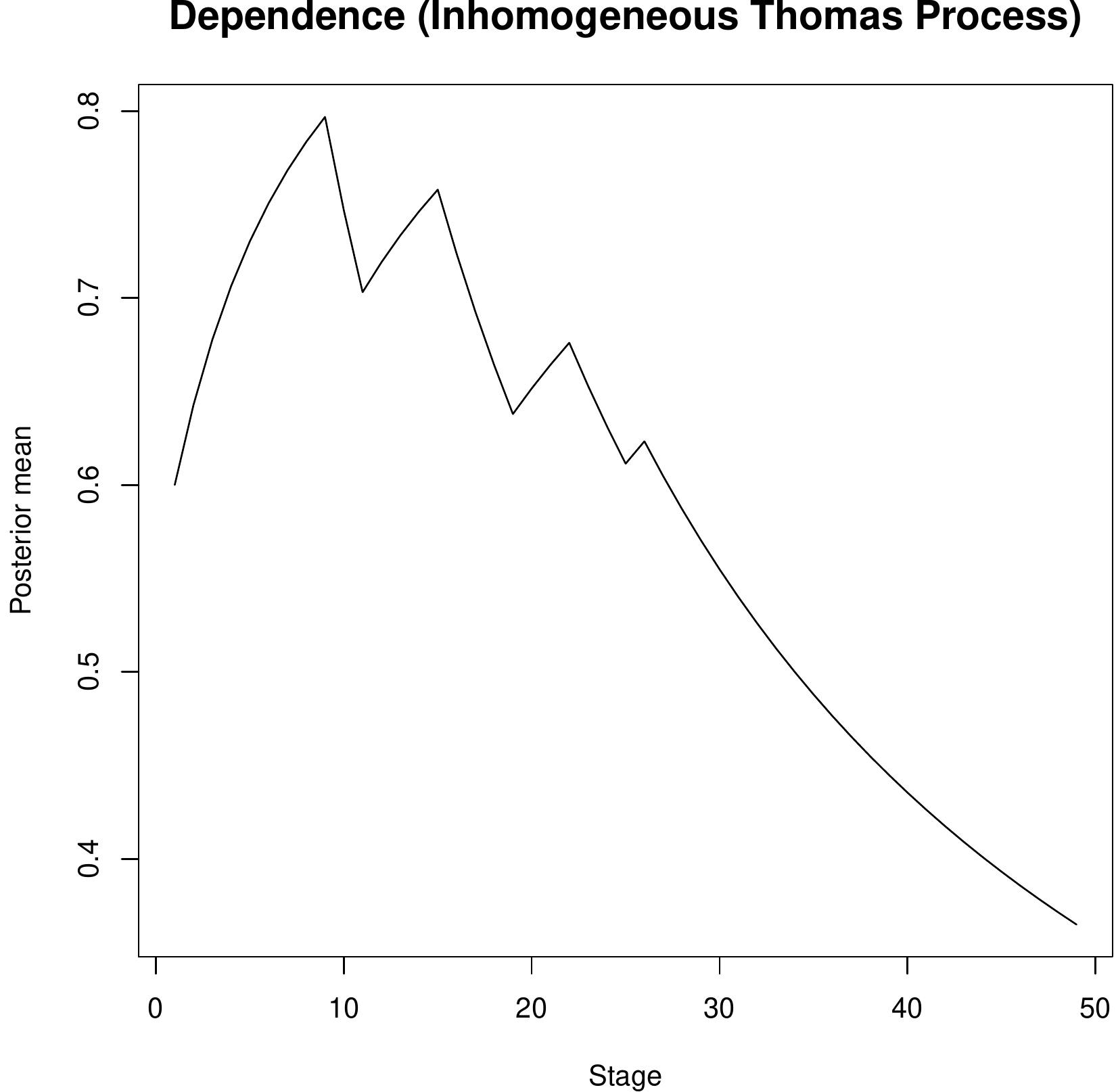}}
\caption{Detection of nonstationarity and dependence of inhomogeneous Thomas process with our Bayesian method.} 
\label{fig:pp7_stationarity_indep}
\end{figure}

\subsection{Example 13: Inhomogeneous Thomas Process with interchanged inhomogeneous $\kappa$ and $\mu$}
\label{subsec:thomas32}

We consider a final inhomogeneous Thomas process with $\mu(u_1,u_2)=10(u^2_1+u^2_2)$ and $\kappa(u_1,u_2)=5\exp\left(2u_1-1\right)$. In this case, we obtained
$4008$ observations on the window $W=[0,2]\times[0,2]$, which we display in Figure \ref{fig:pp8_plots}.
\begin{figure}
\centering
\includegraphics[width=5.5cm,height=5.5cm]{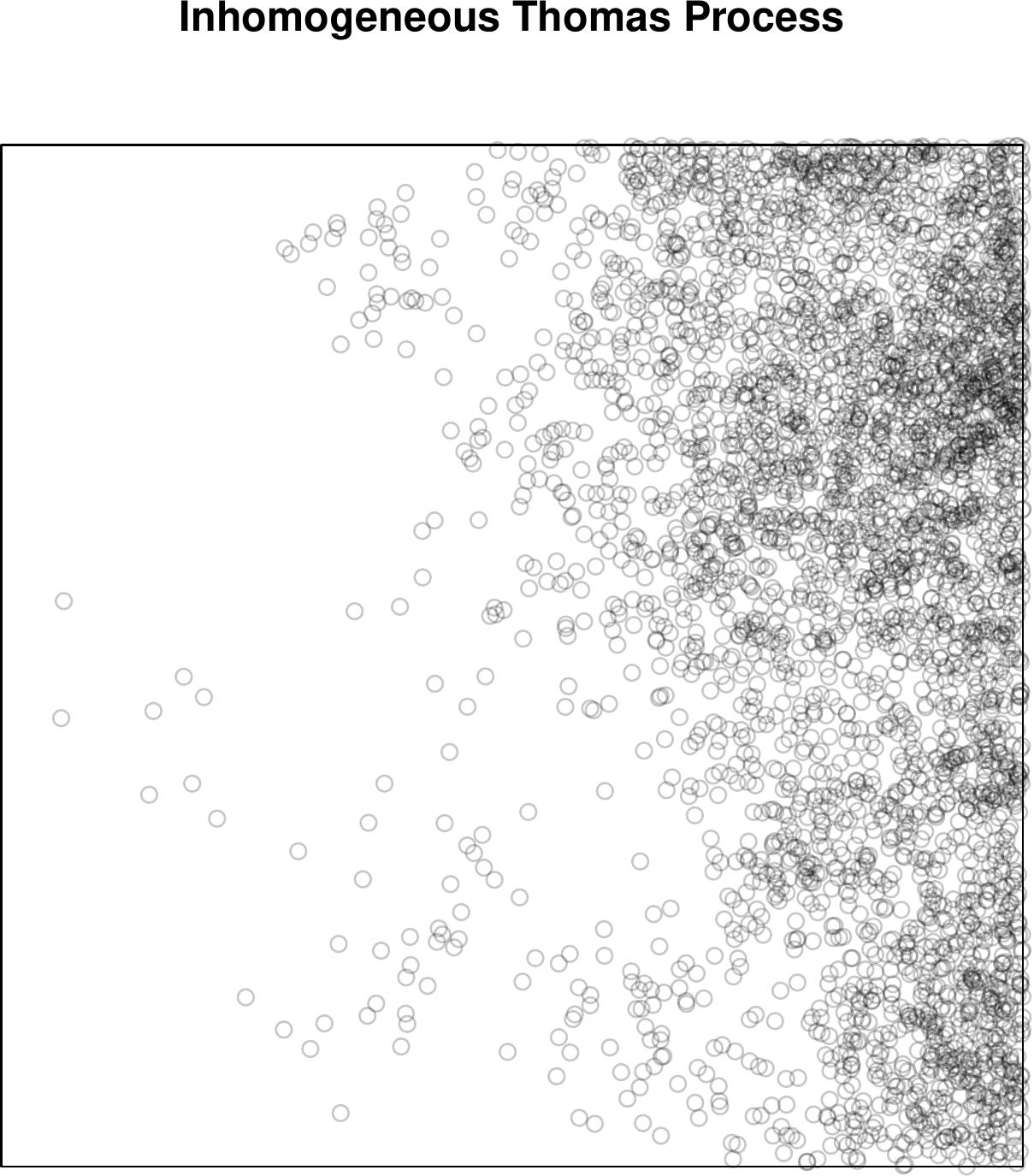}
\caption{Inhomogeneous Thomas point process pattern.} 
\label{fig:pp8_plots}
\end{figure}

For CSR detection, we set $K=500$ and $\hat C_1=0.23$ for the Bayesian method. As shown by Figure \ref{fig:pp8}, both the Bayesian and the classical method successfully
detect non-CSR.
\begin{figure}
\centering
\subfigure [HPP detection with Bayesian method for inhomogeneous Thomas point process.]{ \label{fig:hpp_bayesian8}
\includegraphics[width=5.5cm,height=5.5cm]{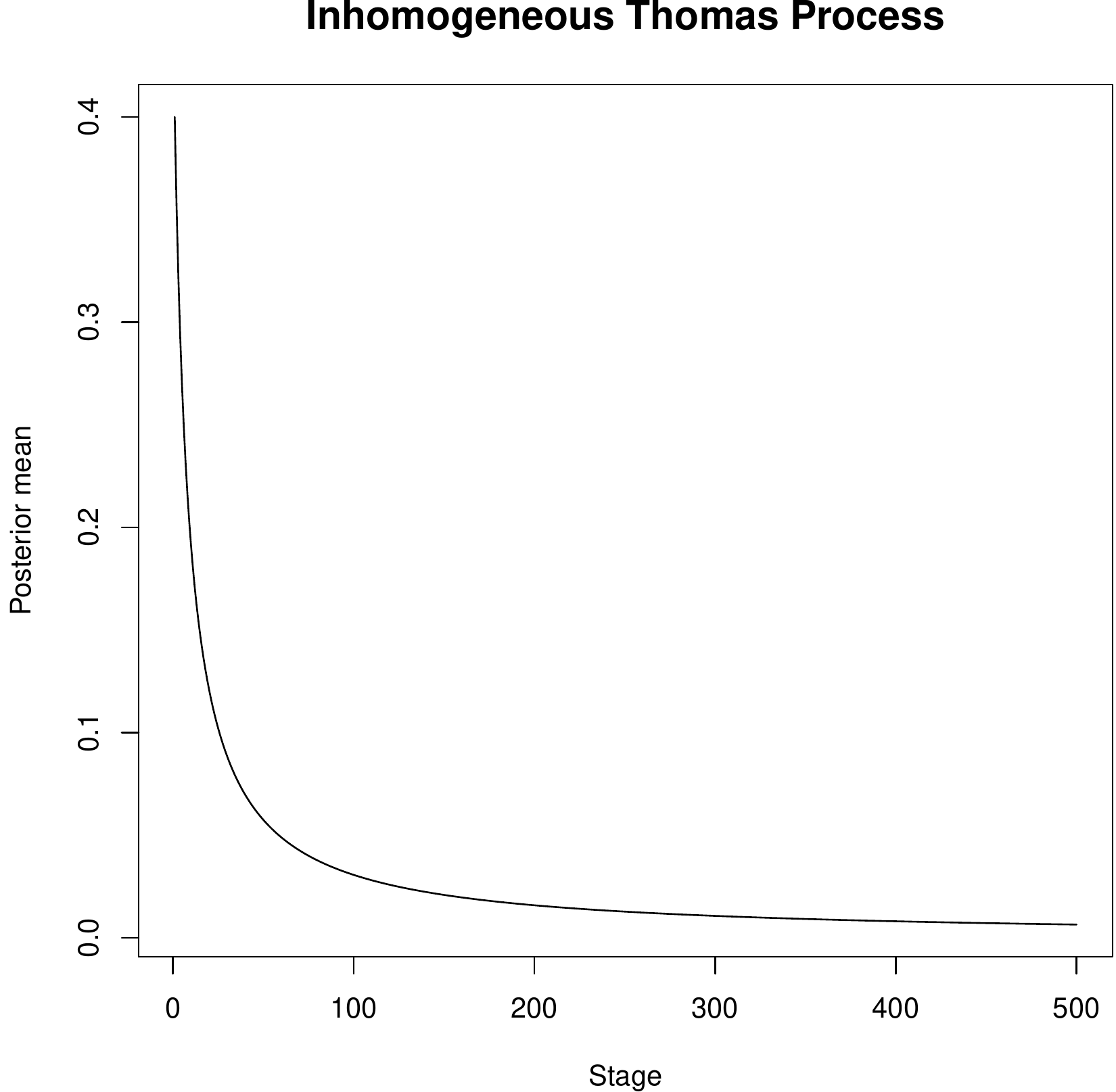}}
\hspace{2mm}
\subfigure [HPP detection with classical method for Inhomogeneous Thomas point process.]{ \label{fig:hpp_classical8}
\includegraphics[width=5.5cm,height=5.5cm]{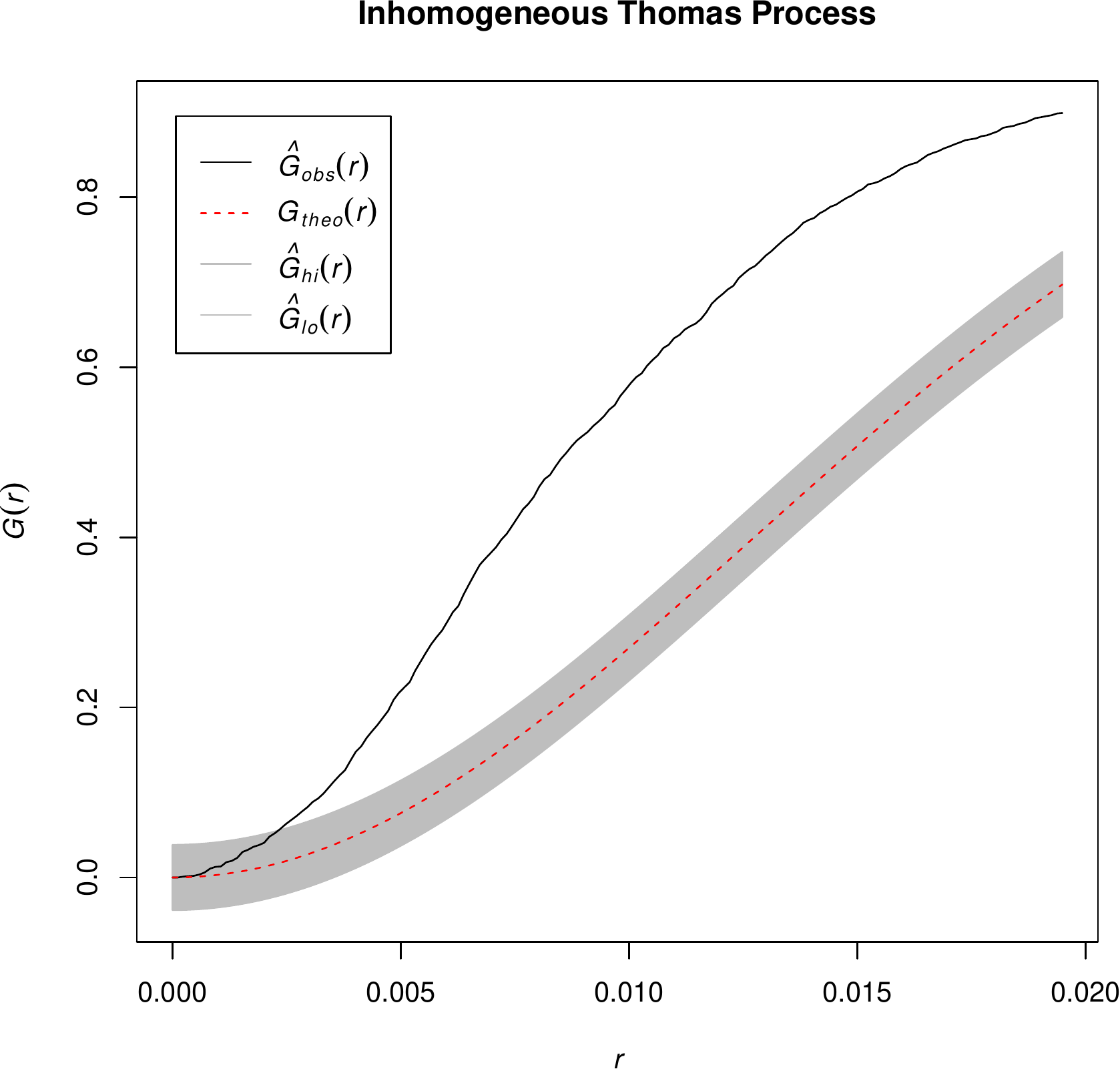}}
\caption{Detection of CSR with our Bayesian method and traditional classical method for inhomogeneous Thomas point process. Both the methods correctly
identify that the underlying point process is not CSR.} 
\label{fig:pp8}
\end{figure}

Our Bayesian method also successfully detected nonstationarity with $K=500$ and $\hat C_1=0.18$, and dependence, with $K=27$ (smaller value chosen to ensure numerical stability) 
and $\hat C_1=0.5$. These results are
depicted in Figure \ref{fig:pp8_stationarity_indep}.
\begin{figure}
\centering
\subfigure [Nonstationary point process (inhomogeneous Thomas process).]{ \label{fig:thomas32_bayesian_nonstationary}
\includegraphics[width=5.5cm,height=5.5cm]{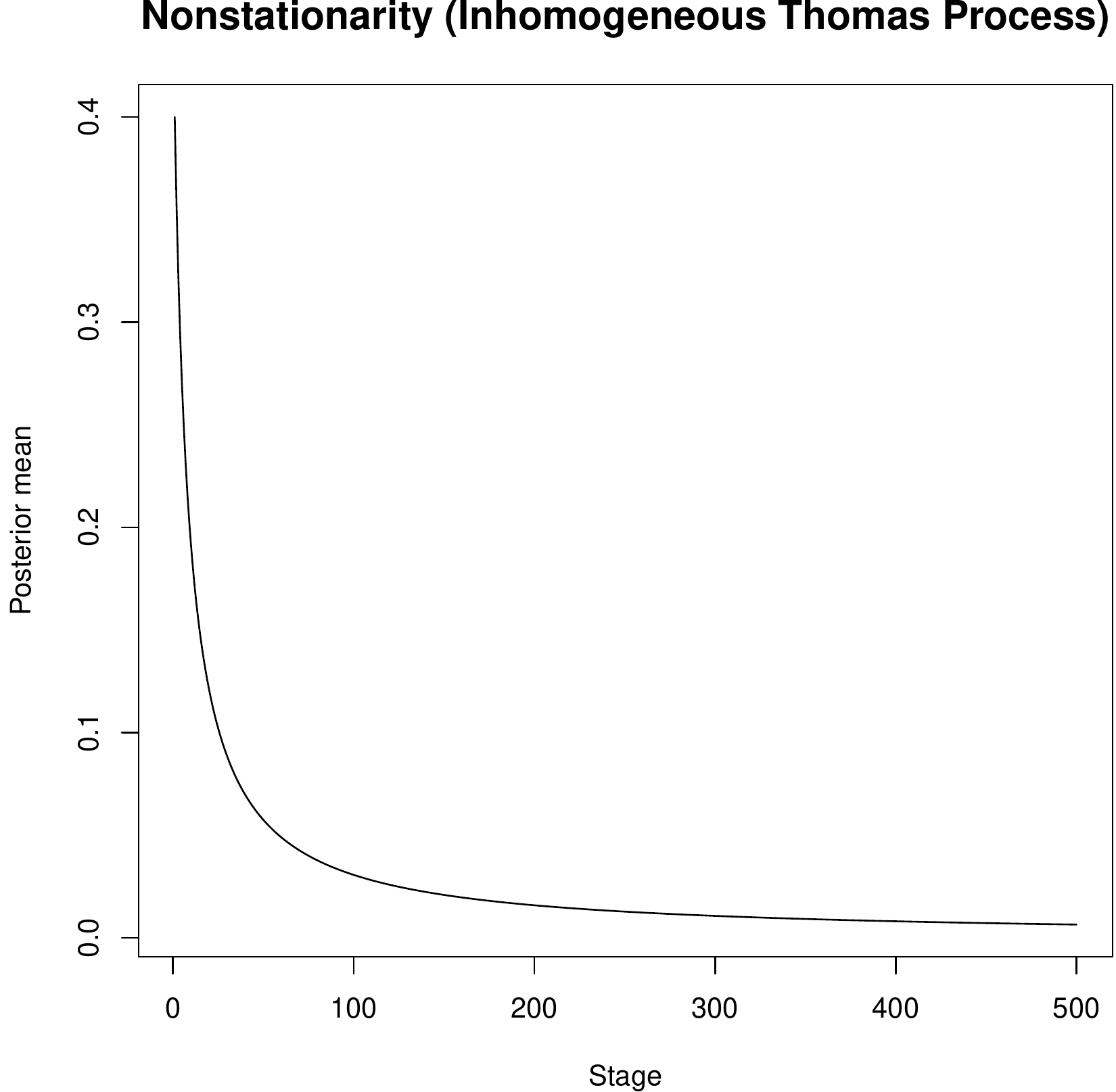}}
\hspace{2mm}
\subfigure [Dependent point process (inhomogeneous Thomas process).]{ \label{fig:thomas32_bayesian_dependent}
\includegraphics[width=5.5cm,height=5.5cm]{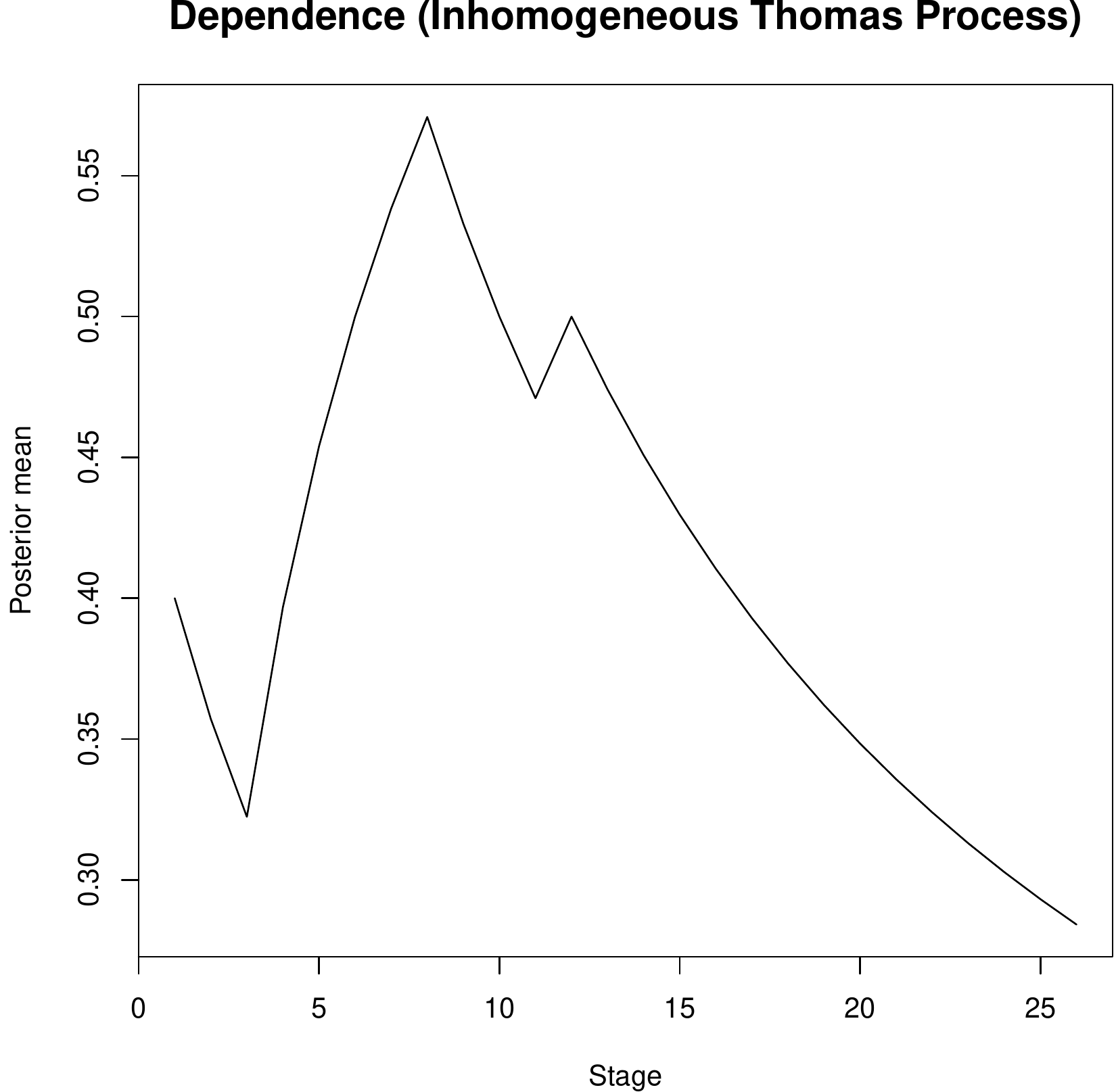}}
\caption{Detection of nonstationarity and dependence of Inhomogeneous Thomas process with our Bayesian method.} 
\label{fig:pp8_stationarity_indep}
\end{figure}

\subsection{Example 14: Homogeneous Neyman-Scott process}
\label{subsec:ns1}

A Neyman-Scott process is a Cox process where the centers $c$ in (\ref{eq:shotnoise1}) arising from a Poisson process with intensity function
and $\kappa$ and $\gamma\equiv \mu$, where $\mu$ is some deterministic function. Note that the Neyman-Scott process is more general than the Thomas process
in the sense that the density function $k(c,\cdot)$ is left unspecified in the Neyman-Scott case, whereas for the Thomas process, this is a specific 
bivariate normal density. 

More generally, the Neyman-Scott process allows a fixed number of offsprings on a disc with the parent point being the center of the disc. Here even though
the centers arise from a Poisson process with intensity $\kappa$, the offsprings no longer follow the Poisson process, since given the parent points,
the number of offsprings given each parent, is non-random.
In such a case, the Neyman-Scott process is no longer a Cox process.

In order to test our methods on Neyman-Scott process, we first consider a homogeneous general Neyman-Scott process with $\kappa=10$, with $5$ points generated uniformly
on each disc of radius $0.2$ around the parent centers. The point pattern, simulated on $W=[0,10]\times[0,10]$, consisting of $4867$
observations, is shown in Figure \ref{fig:pp9_plots}.

\begin{figure}
\centering
\includegraphics[width=5.5cm,height=5.5cm]{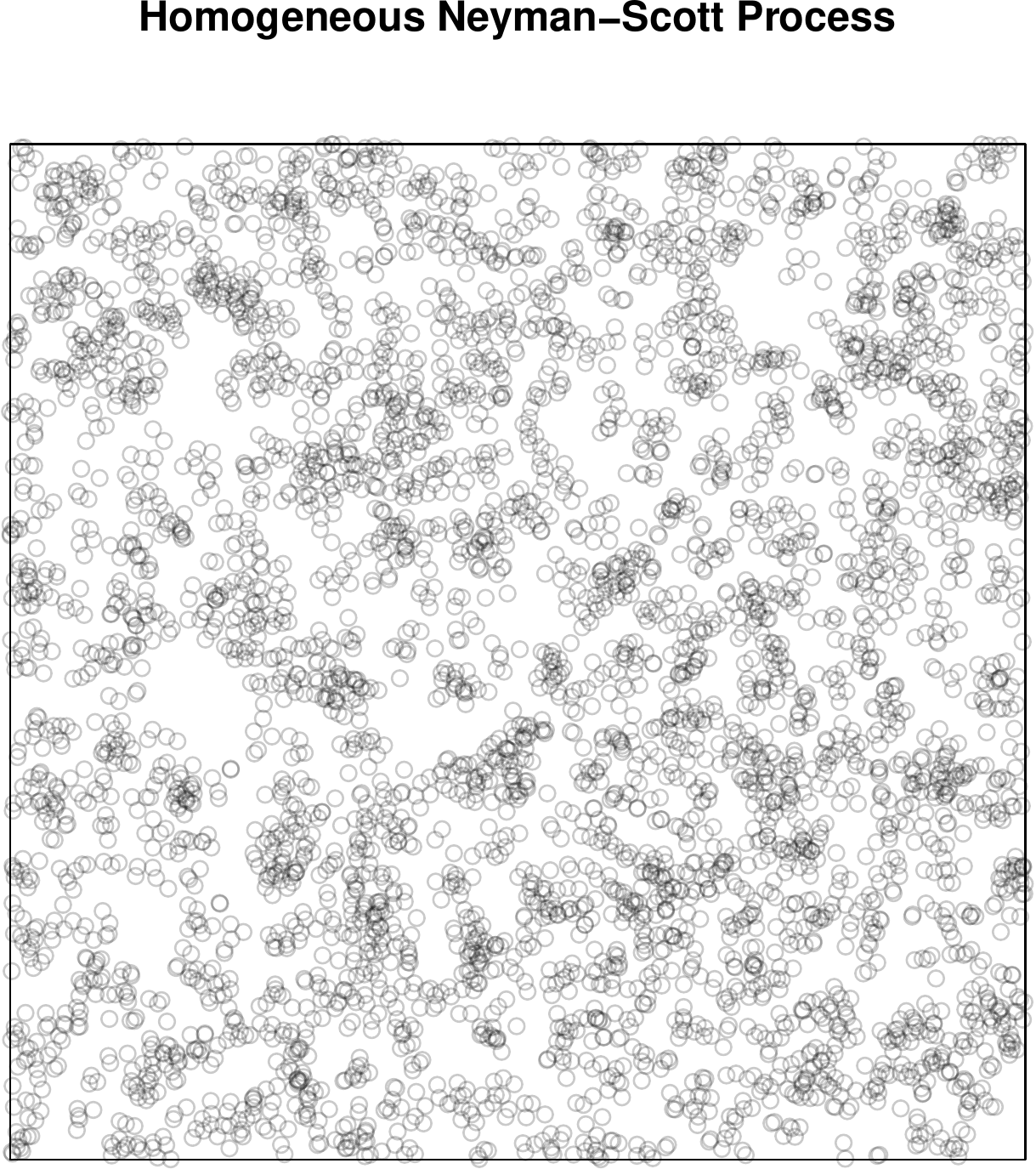}
\caption{Homogeneous Neyman-Scott point process pattern.} 
\label{fig:pp9_plots}
\end{figure}

Both the Bayesian and the traditional method of checking CSR correctly indicate that the underlying process is not CSR. In the Bayesian case, we set $K=500$
and $\hat C_1=0.20$. The results are displayed in Figure \ref{fig:pp9}. 
\begin{figure}
\centering
\subfigure [HPP detection with Bayesian method for homogeneous Neyman-Scott point process.]{ \label{fig:hpp_bayesian9}
\includegraphics[width=5.5cm,height=5.5cm]{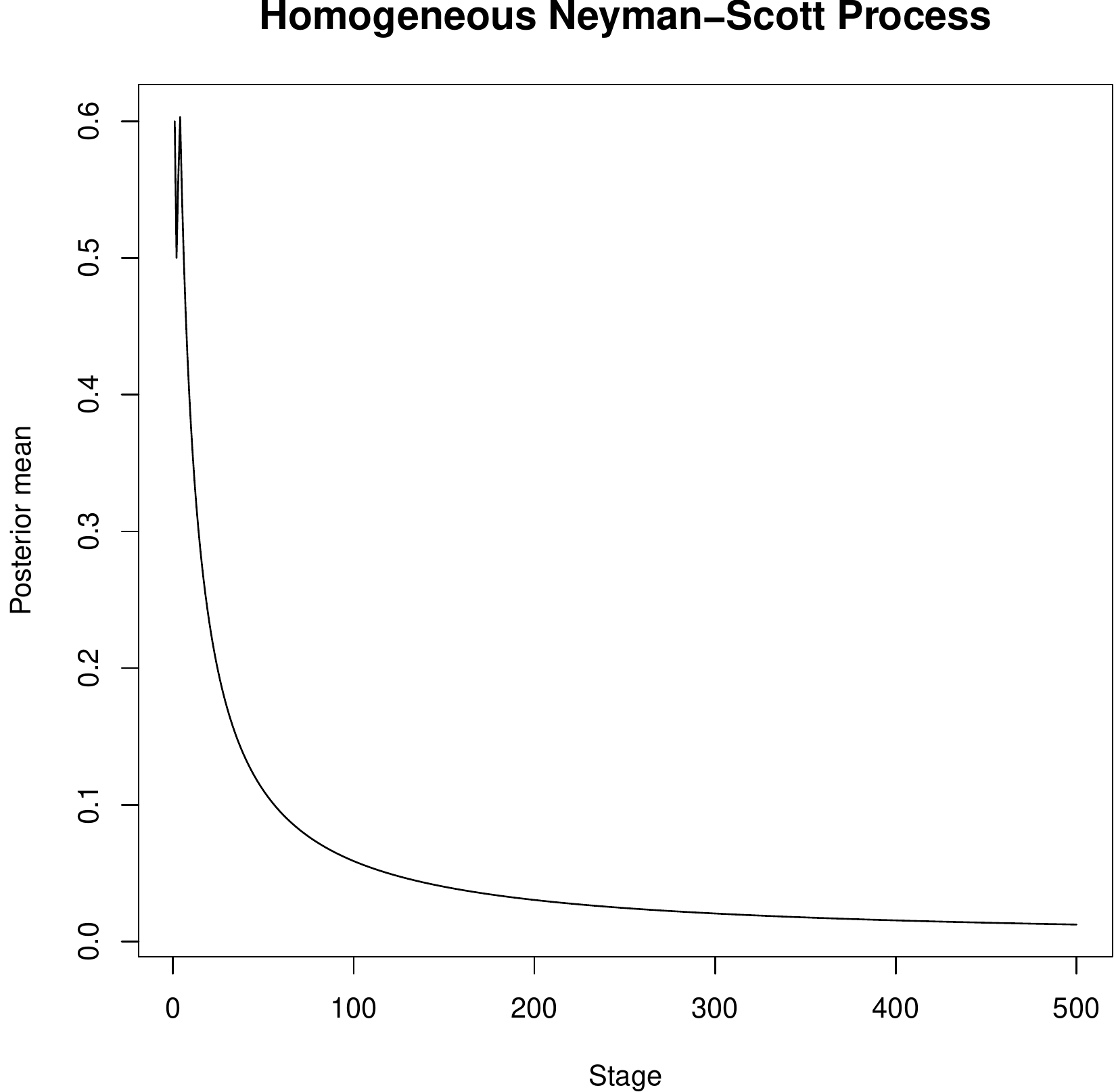}}
\hspace{2mm}
\subfigure [HPP detection with classical method for homogeneous Neyman-Scott point process.]{ \label{fig:hpp_classical9}
\includegraphics[width=5.5cm,height=5.5cm]{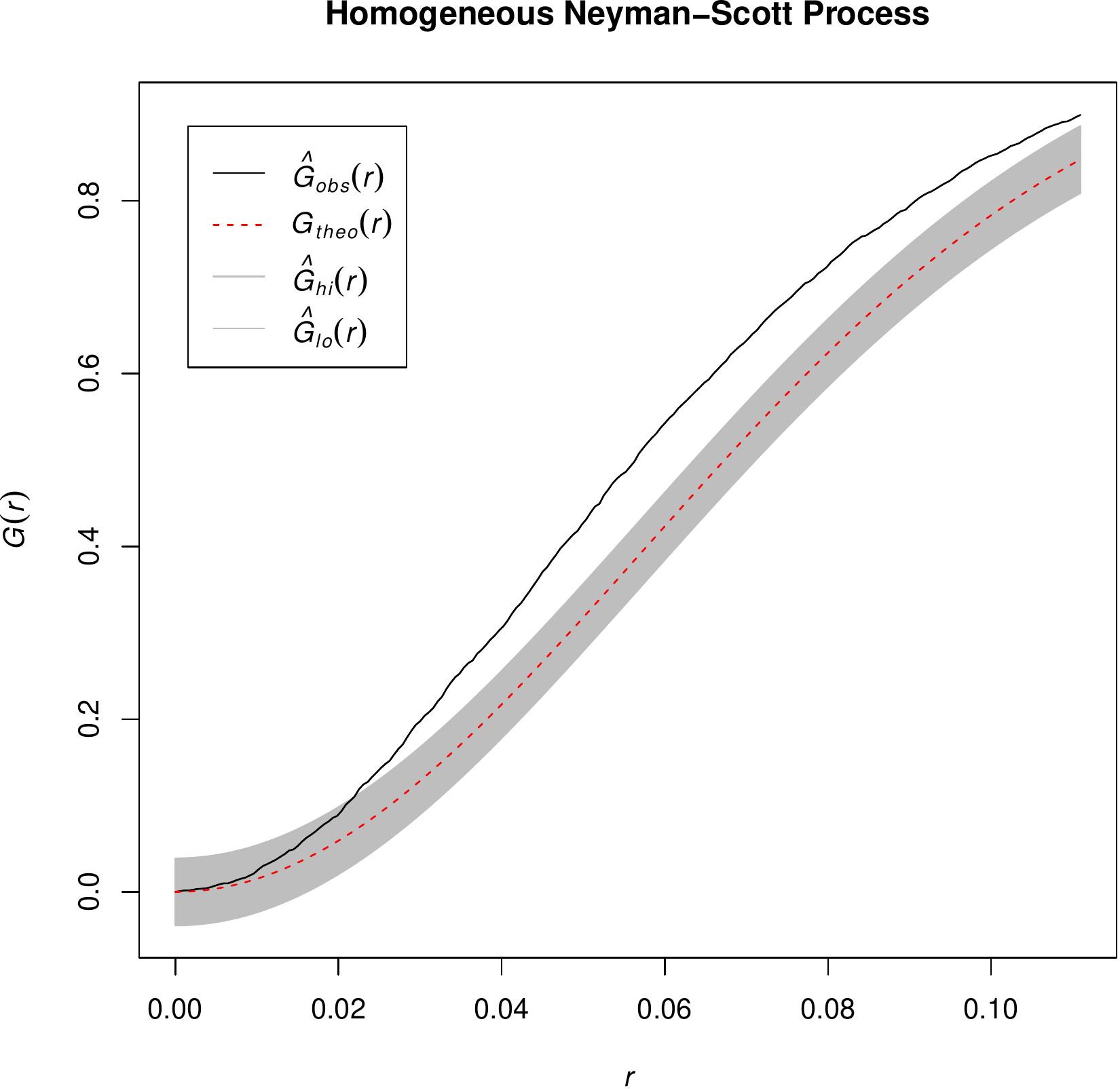}}
\caption{Detection of CSR with our Bayesian method and traditional classical method for homogeneous Neyman-Scott point process. Both the methods correctly
identify that the underlying point process is not CSR.} 
\label{fig:pp9}
\end{figure}

Stationarity is correctly detected by our Bayesian method with $K=500$ and $\hat C_1=0.23$. Also, with $K=50$ and $\hat C_1=0.5$, Poisson process is correctly ruled out.
The results are depicted in Figure \ref{fig:pp9_stationarity_indep}.
\begin{figure}
\centering
\subfigure [Stationary point process (homogeneous Neyman-Scott process).]{ \label{fig:ns1_bayesian_nonstationary}
\includegraphics[width=5.5cm,height=5.5cm]{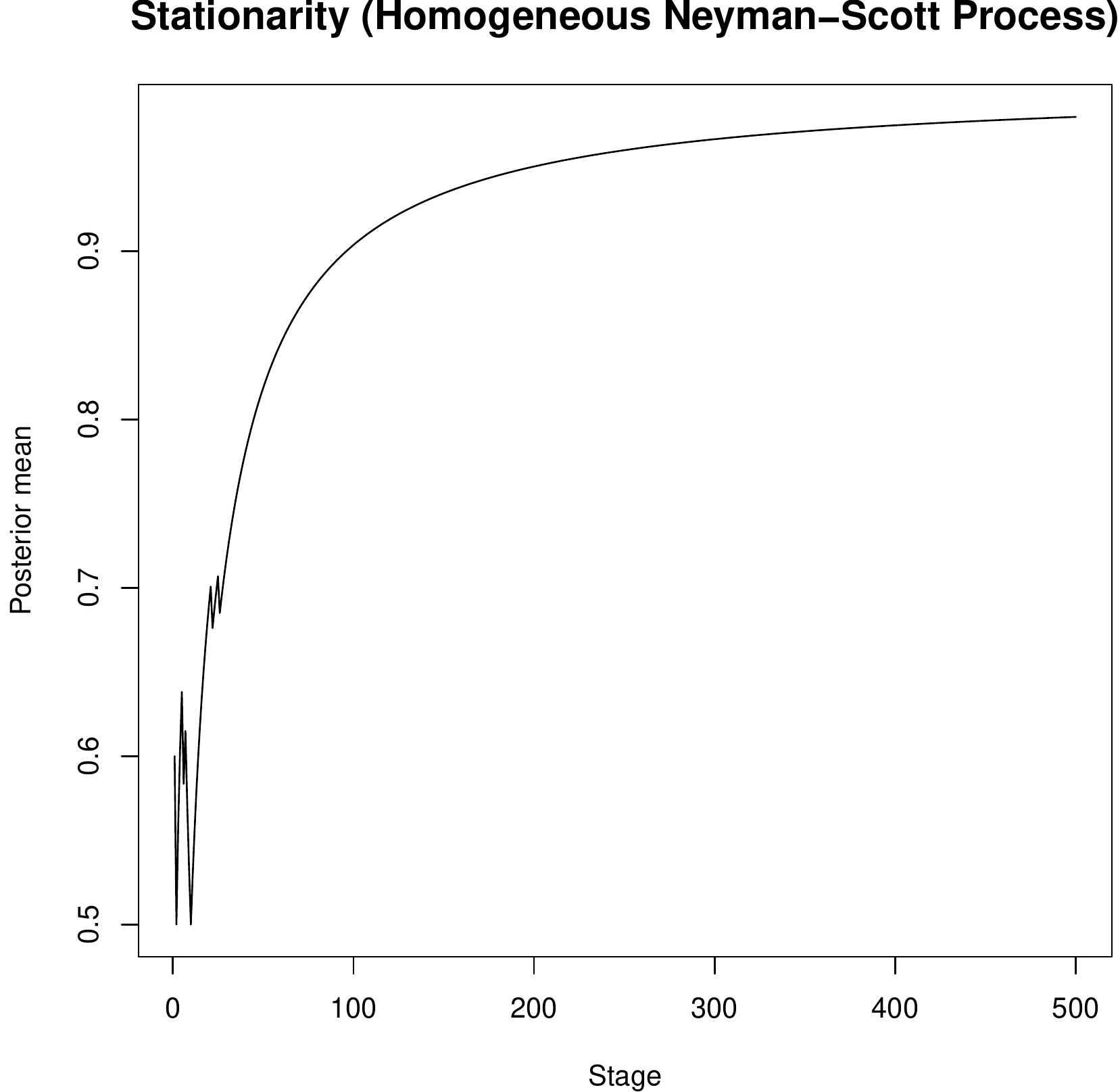}}
\hspace{2mm}
\subfigure [Dependent point process (homogeneous Neyman-Scott process).]{ \label{fig:ns1_bayesian_dependent}
\includegraphics[width=5.5cm,height=5.5cm]{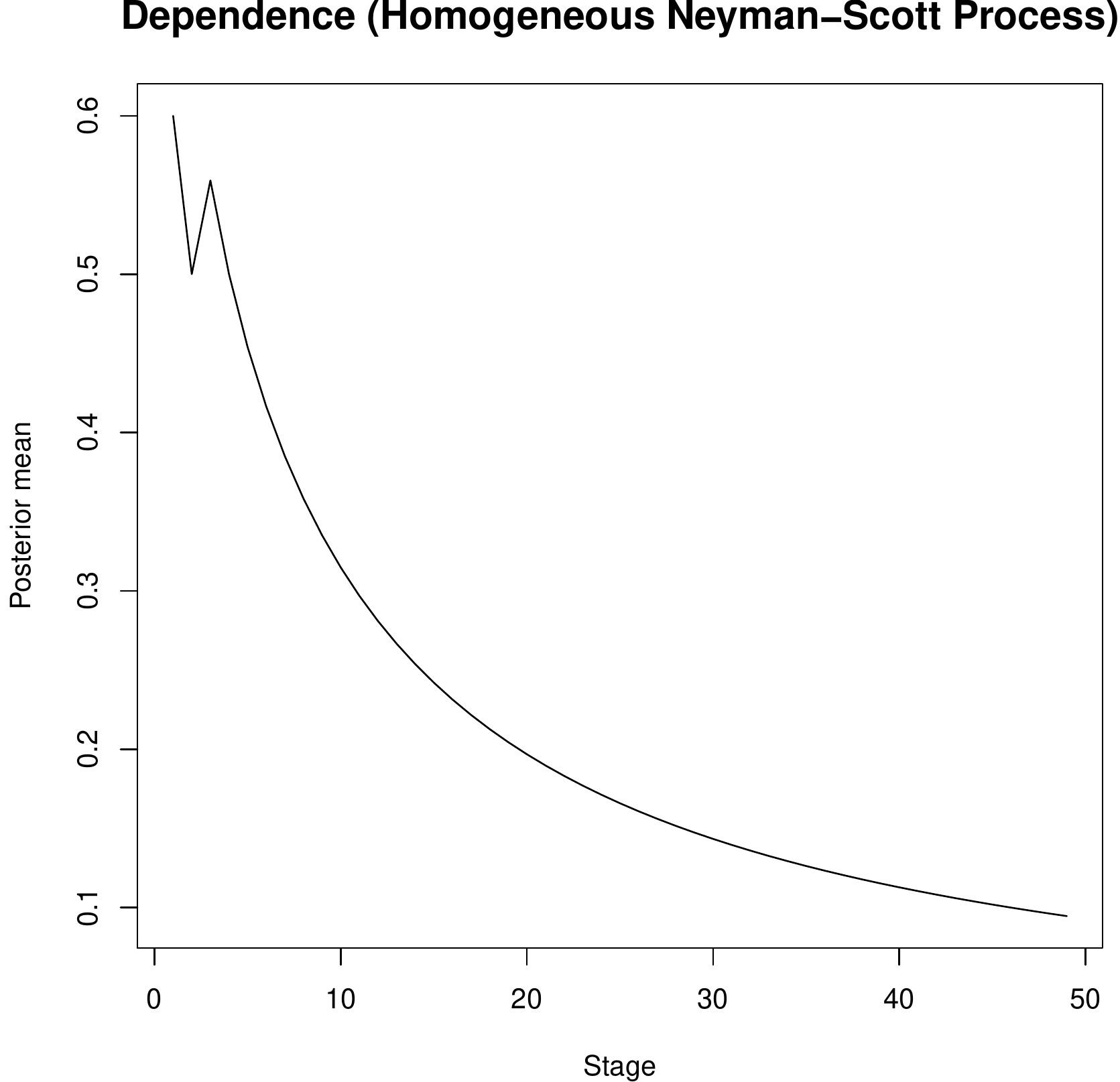}}
\caption{Detection of stationarity and dependence of homogeneous Neyman-Scott process with our Bayesian method.} 
\label{fig:pp9_stationarity_indep}
\end{figure}

\subsection{Example 15: Inhomogeneous Neyman-Scott process}
\label{subsec:ns2}

In this case, we generate a sample of size $8358$ on $W=[0,4]\times[0,4]$ from a Neyman-Scott process with the same setup as above, but with
$\kappa(u_1,u_2)=10(u^2_1+u^2_2)$. The point pattern thus generated from this inhomogeneous Neyman-Scott process is shown in Figure \ref{fig:pp10_plots}.
\begin{figure}
\centering
\includegraphics[width=5.5cm,height=5.5cm]{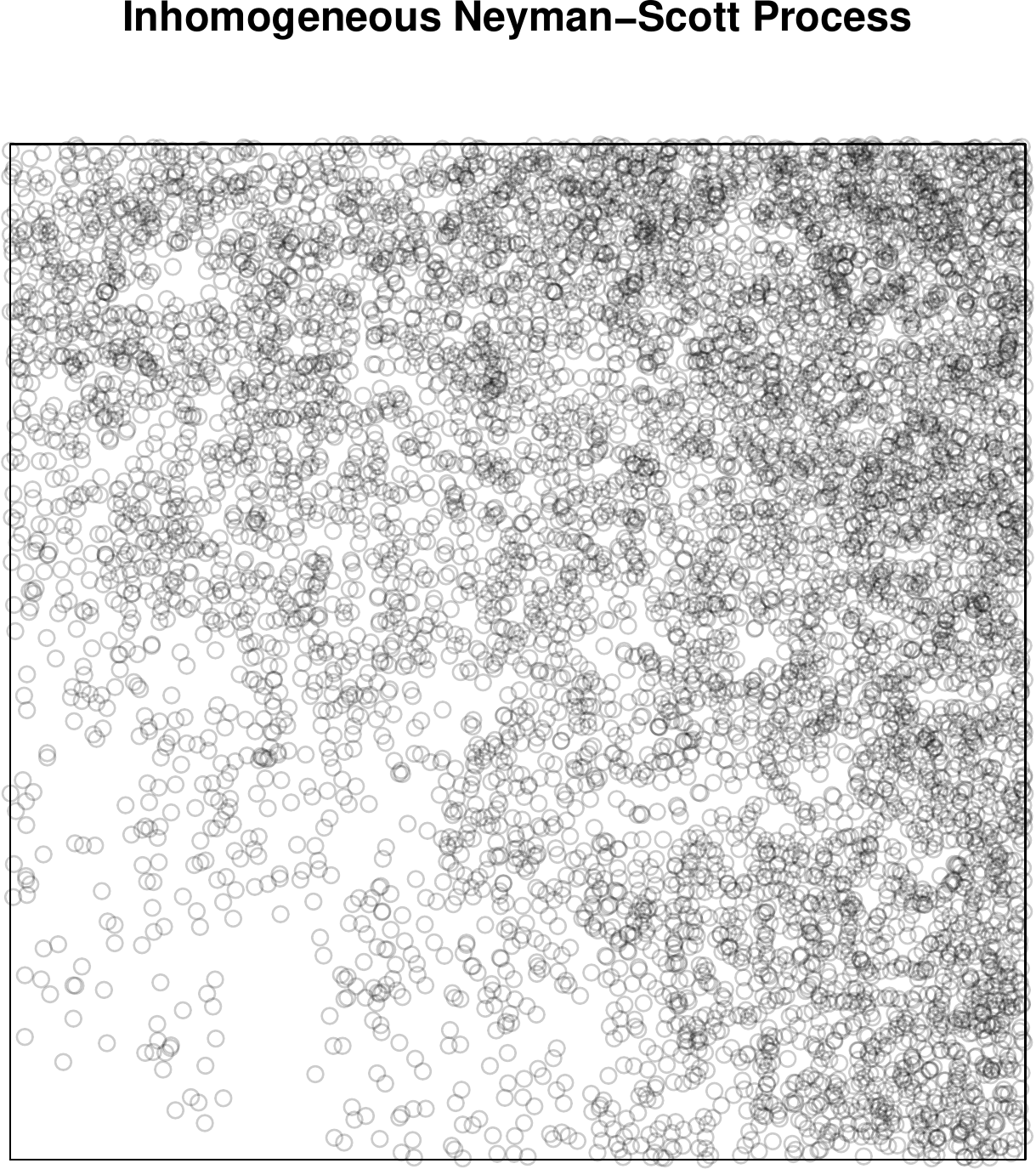}
\caption{Inhomogeneous Neyman-Scott point process pattern.} 
\label{fig:pp10_plots}
\end{figure}

With $K=800$ and $\hat C_1=0.19$, we obtain the correct non-CSR conclusion with the Bayesian method. The correct result is also identified by the classical method.
Both the results are depicted in Figure \ref{fig:pp10}. 
\begin{figure}
\centering
\subfigure [HPP detection with Bayesian method for inhomogeneous Neyman-Scott point process.]{ \label{fig:hpp_bayesian10}
\includegraphics[width=5.5cm,height=5.5cm]{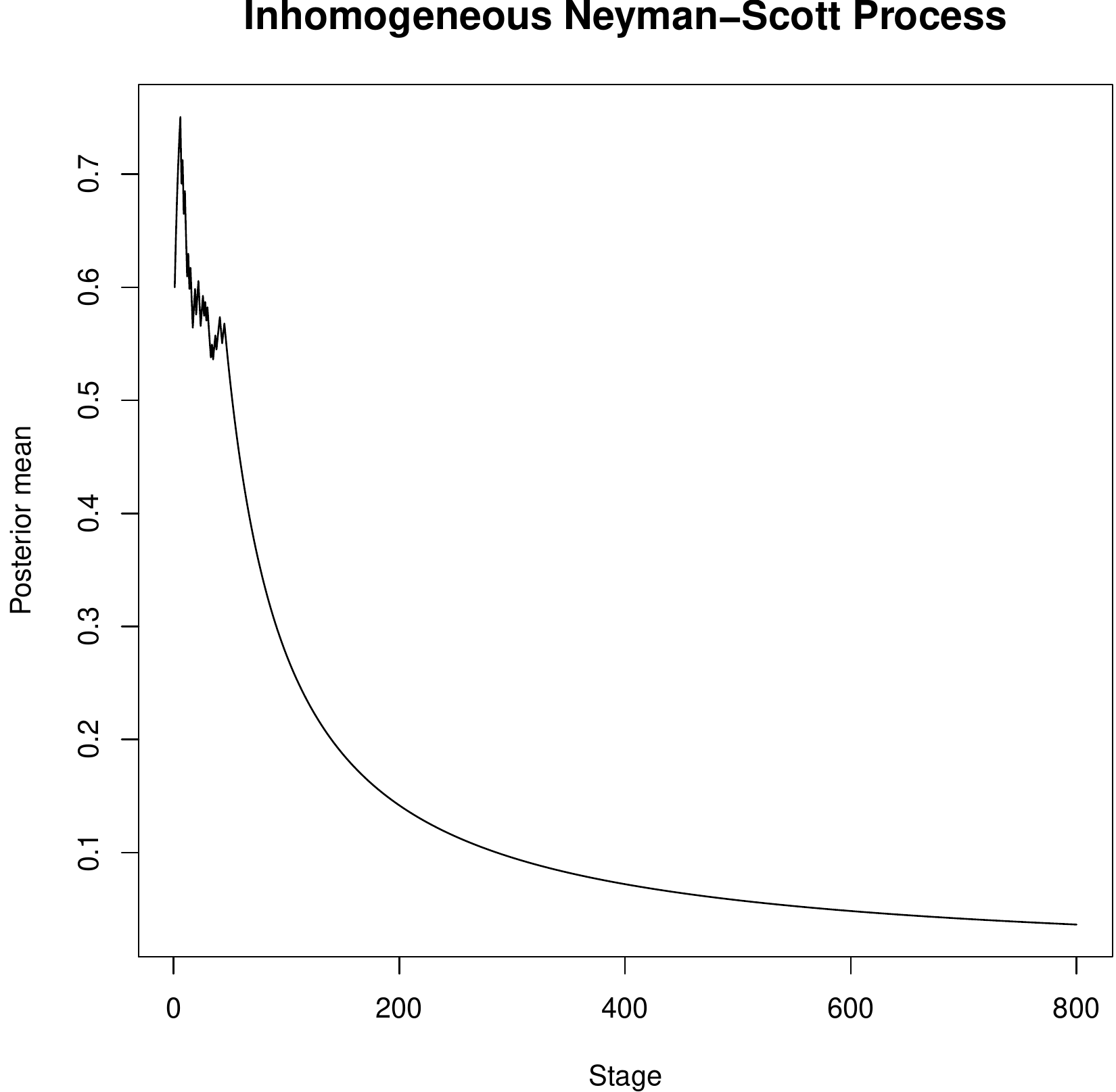}}
\hspace{2mm}
\subfigure [HPP detection with classical method for homogeneous Neyman-Scott point process.]{ \label{fig:hpp_classical10}
\includegraphics[width=5.5cm,height=5.5cm]{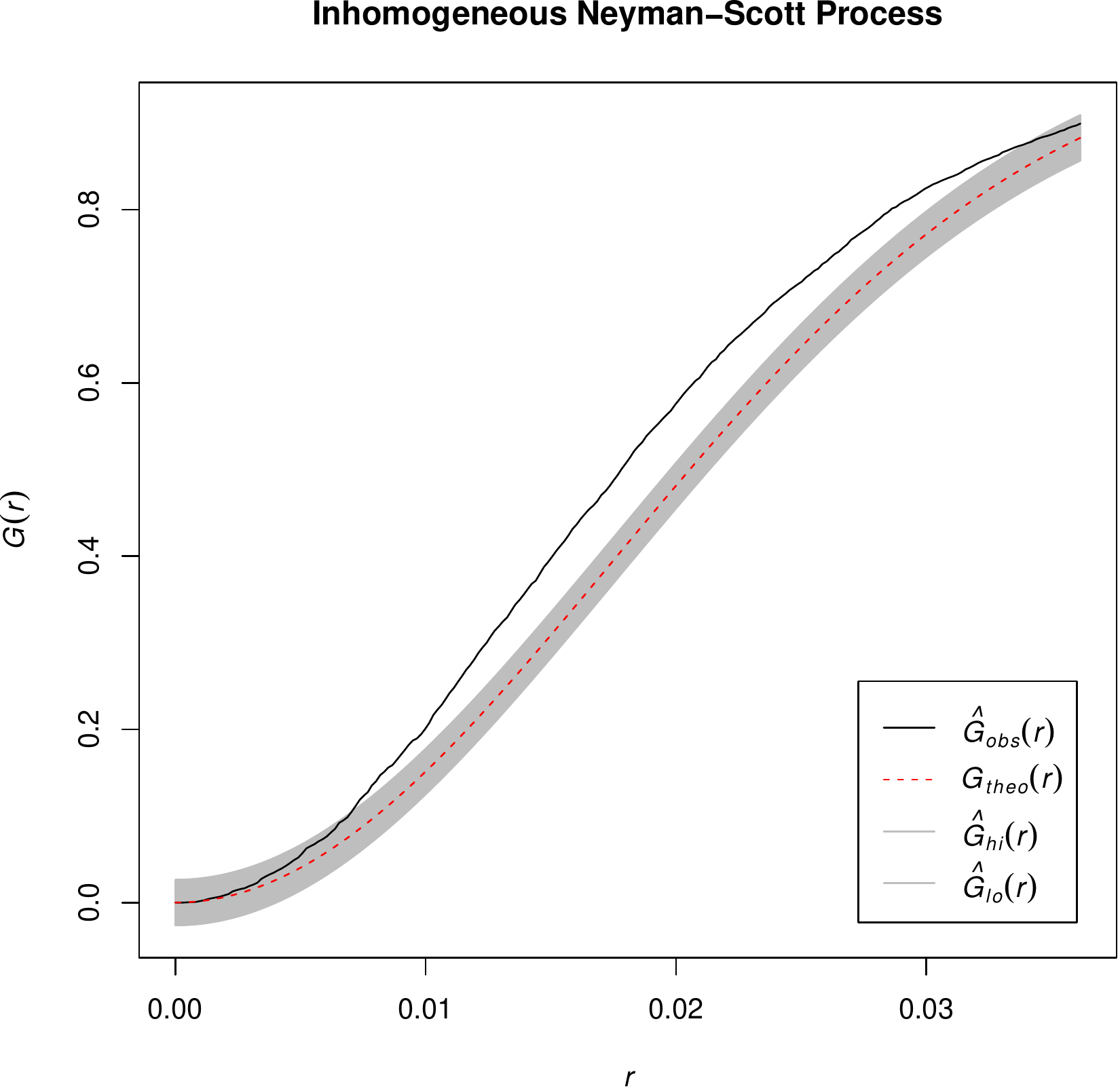}}
\caption{Detection of CSR with our Bayesian method and traditional classical method for homogeneous Neyman-Scott point process. Both the methods correctly
identify that the underlying point process is not CSR.} 
\label{fig:pp10}
\end{figure}

Nonstationarity of this process is correctly detected by the Bayesian method with $K=1000$ and $\hat C_1=0.23$; this is shown in panel (a) of 
Figure \ref{fig:pp10_stationarity_indep}. For $K=50$ and $\hat C_1=0.5$.
panel (b) of Figure \ref{fig:pp10_stationarity_indep} shows steady increase
for about the first $35$ stages, but sharply decreases thenceforward, indicating dependence.
\begin{figure}
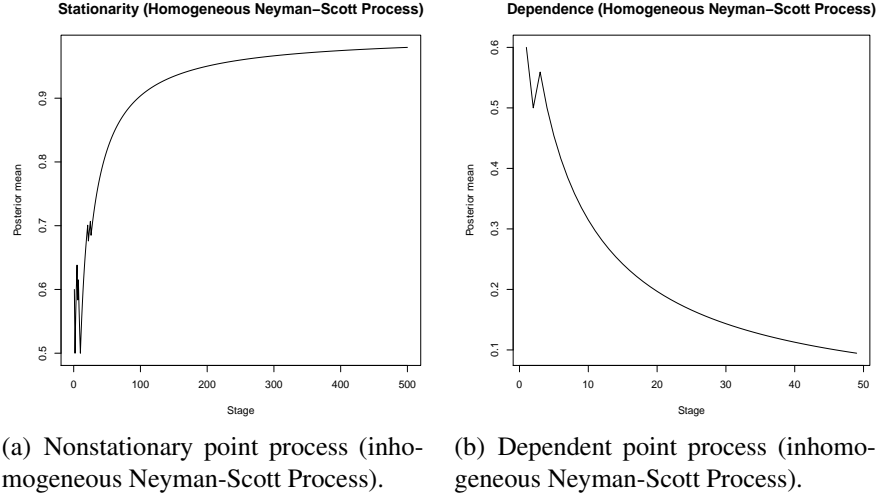

\centering
\subfigure [Nonstationary point process (inhomogeneous Neyman-Scott Process).]{ \label{fig:ns2_bayesian_nonstationary}
\includegraphics[width=5.5cm,height=5.5cm]{figures_neyscott1/stationary-crop.pdf}}
\hspace{2mm}
\subfigure [Dependent point process (inhomogeneous Neyman-Scott Process).]{ \label{fig:ns2_bayesian_dependent}
\includegraphics[width=5.5cm,height=5.5cm]{figures_neyscott1/indep-crop.pdf}}
\caption{Detection of nonstationarity and dependence of homogeneous Neyman-Scott process with our Bayesian method.} 
\label{fig:pp10_stationarity_indep}
\end{figure}

\subsection{Example 16: Strauss process}
\label{subsec:strauss1}

The Strauss process (\ctn{Strauss75}; see also \ctn{Moller04}) is an instance of pairwise interaction point process with density (with respect to unit intensity Poisson process)
\begin{equation}
	f(x)\propto\beta^{n(x)}\gamma^{s_R(x)},
	\label{eq:strauss1}
\end{equation}
where $\beta>0$, $n(x)$ is the number of points in $x$ and $s_R(x)=\sum_{(\xi,\eta)\subseteq x}I\left\{\|\xi-\eta\|\leq R\right\}$ is the number of 
$R$-close pairs of points in $x$.
Note that if $\gamma=1$, we obtain Poisson process on $\bS$ with intensity $\beta$, and if $\gamma<1$, there is repulsion between the $R$-close points pairs of
points in $\bX$.

Using spatstat, we generate $9790$ points from a Strauss process with $\beta=0.05$, $\gamma=0.2$ and $R=1.5$ on $W=[0,500]\times[0,500]$. The points are
displayed in Figure \ref{fig:pp14_strauss}.
\begin{figure}
\centering
\includegraphics[width=5.5cm,height=5.5cm]{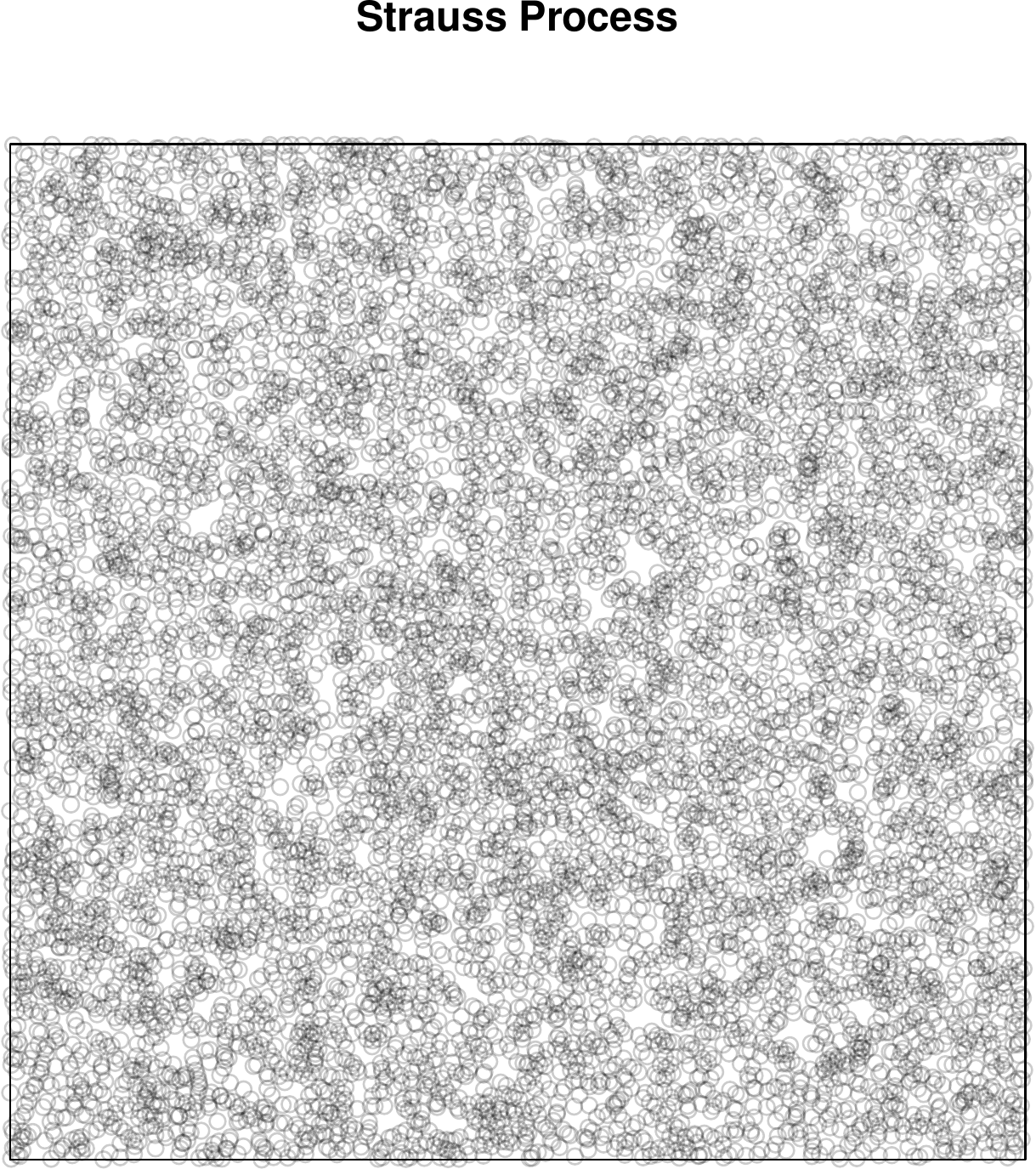}
\caption{Strauss point process pattern.} 
\label{fig:pp14_strauss}
\end{figure}

To detect CSR, we set $K=800$ and $\hat C_1=0.15$ for the Bayesian algorithm. As Figure \ref{fig:pp14} shows, both the classical and the Bayesian methods correctly
identify that the underlying process is not CSR.
\begin{figure}
\centering
\subfigure [HPP detection with Bayesian method for Strauss process.]{ \label{fig:hpp_bayesian14}
\includegraphics[width=5.5cm,height=5.5cm]{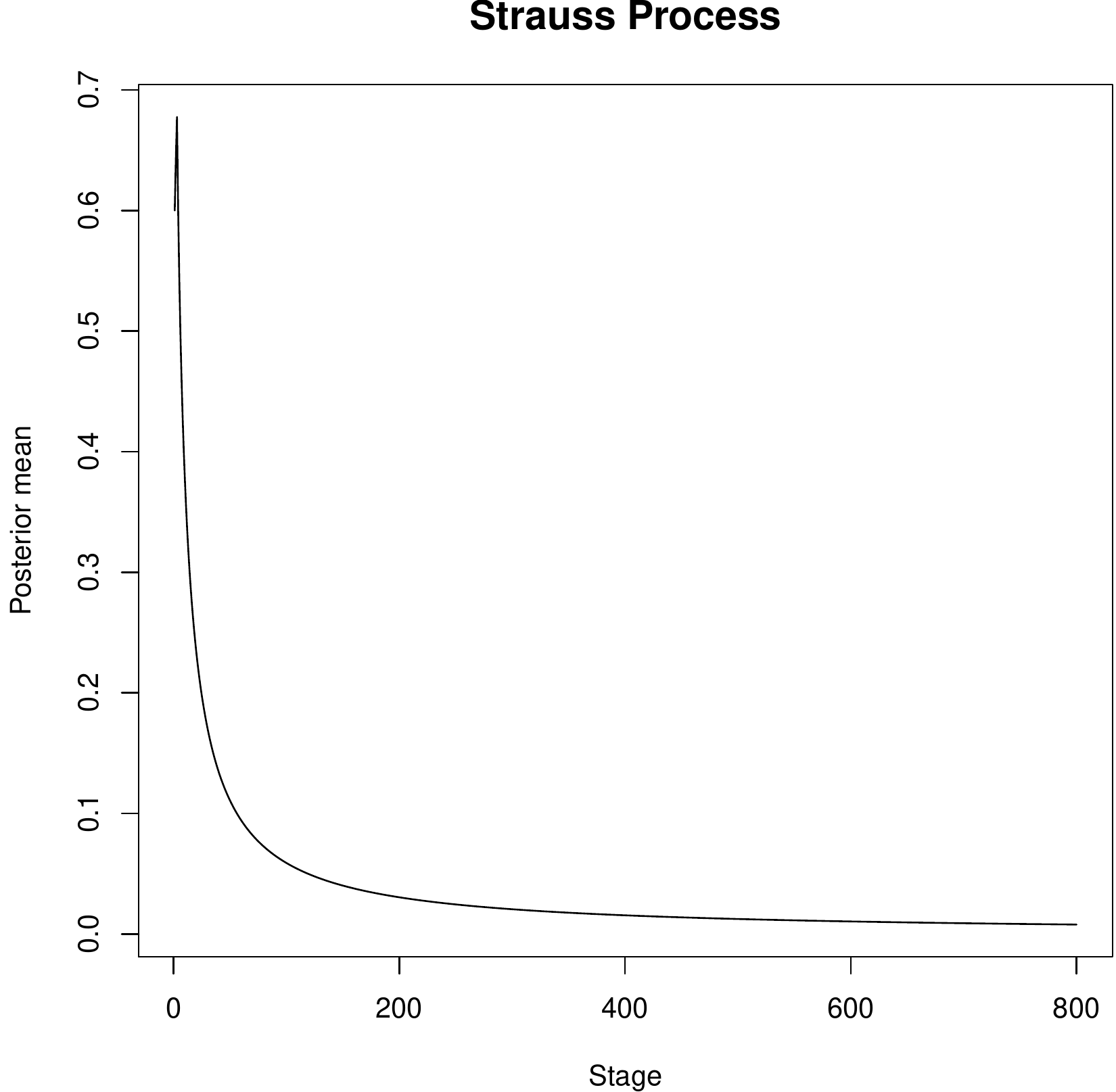}}
\hspace{2mm}
\subfigure [HPP detection with classical method for Strauss process.]{ \label{fig:hpp_classical14}
\includegraphics[width=5.5cm,height=5.5cm]{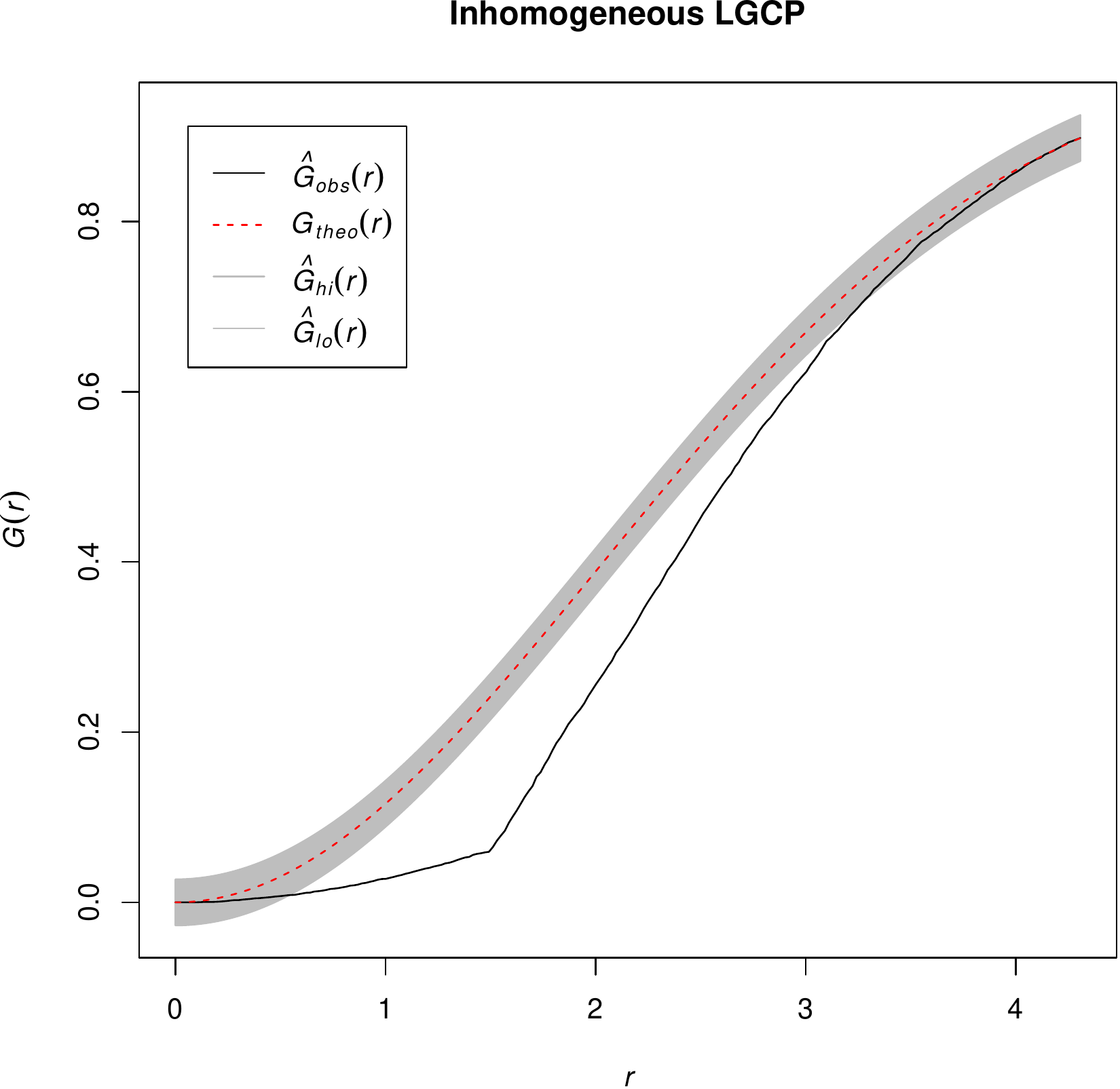}}
\caption{Detection of CSR with our Bayesian method and traditional classical method for Strauss process. Both the methods correctly
identify that the underlying point process is not CSR.} 
\label{fig:pp14}
\end{figure}

The left panel of Figure \ref{fig:pp14_stationarity_indep} captures the stationarity property of the Strauss process with $K=800$ and $\hat C_1=0.15$. 
As before, larger values of $\hat C_1$ also lead to stationarity. The right panel of Figure \ref{fig:pp14_stationarity_indep} correctly indicates dependence among
$\bX_{C_i}$, for $i=1,\ldots,100$, with $\hat C_1=0.5$.
\begin{figure}
\centering
\subfigure [Stationarity (Strauss process).]{ \label{fig:strauss1_bayesian_stationary}
\includegraphics[width=5.5cm,height=5.5cm]{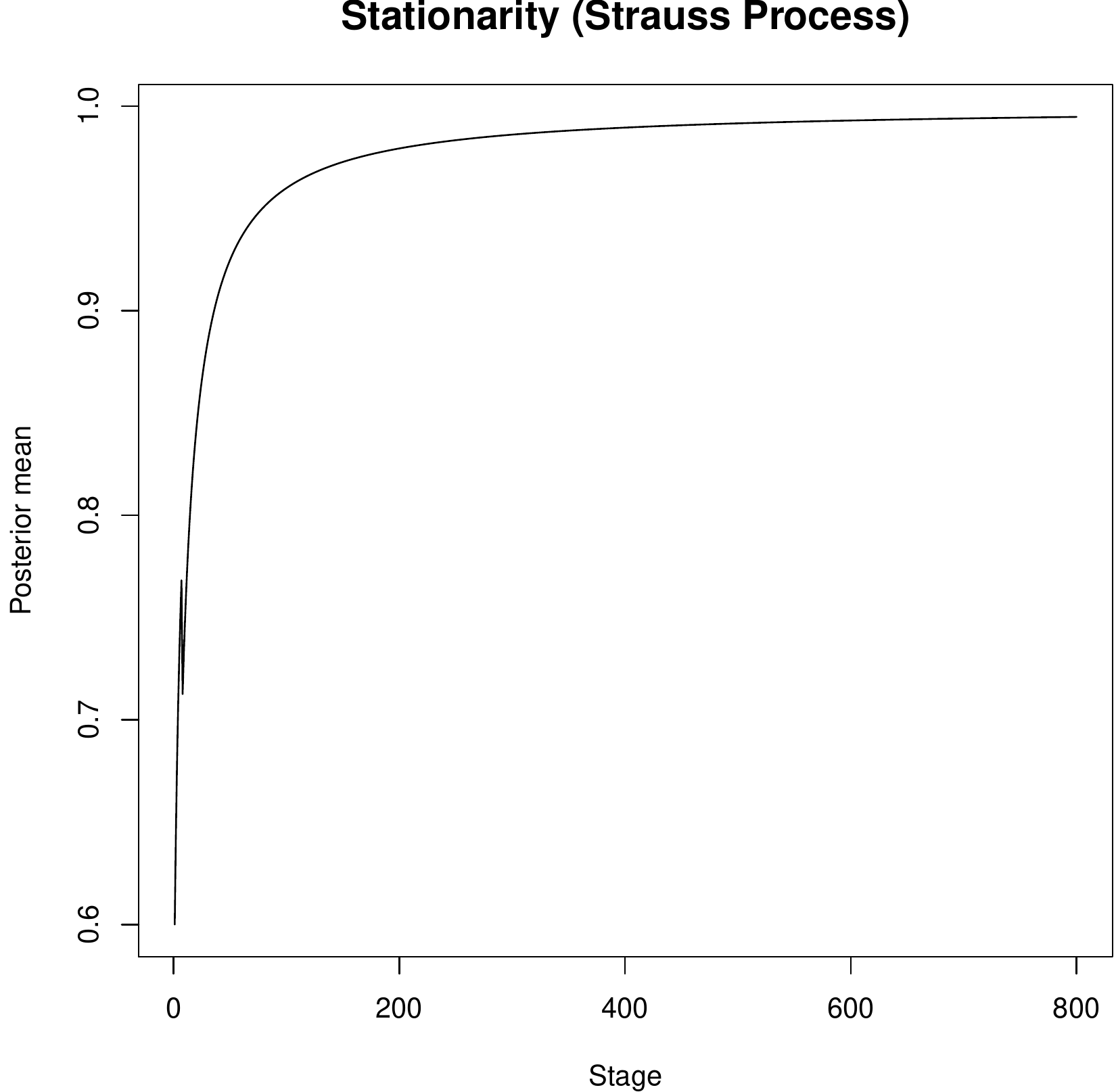}}
\hspace{2mm}
\subfigure [Dependent point process (Strauss process).]{ \label{fig:strauss1_bayesian_dependent}
\includegraphics[width=5.5cm,height=5.5cm]{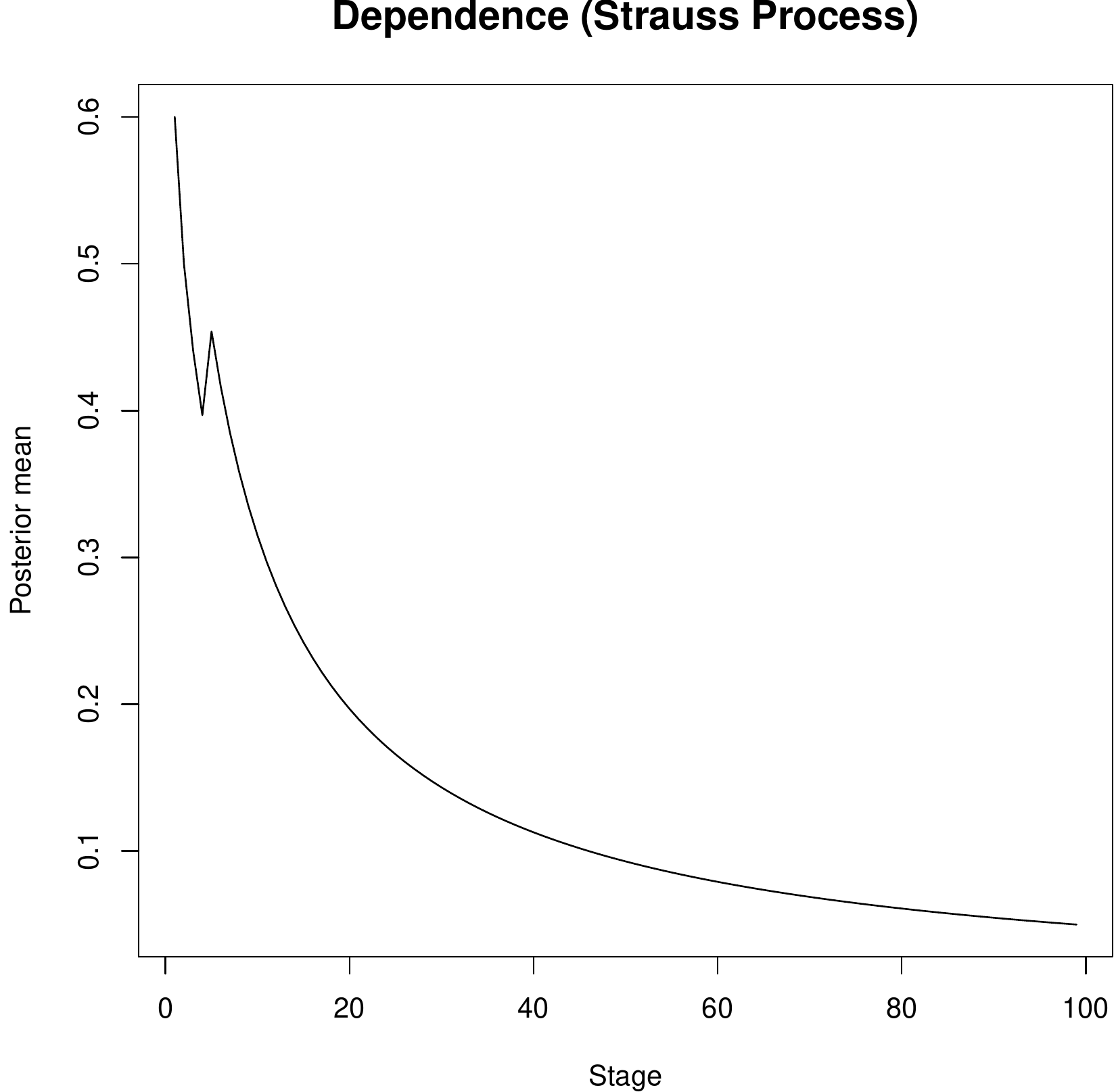}}
\caption{Detection of stationarity and dependence of Strauss process with our Bayesian method.} 
\label{fig:pp14_stationarity_indep}
\end{figure}

\subsection{Example 17: Another Strauss process}
\label{subsec:strauss2}

We now consider simulation from another homogeneous Strauss process with $\beta=100$, $\gamma=0.7$ and $R=0.05$ on $W=[0,8]\times[0,8]$. 
The $5168$ points that we obtained, are plotted in Figure \ref{fig:pp15_strauss}.
\begin{figure}
\centering
\includegraphics[width=5.5cm,height=5.5cm]{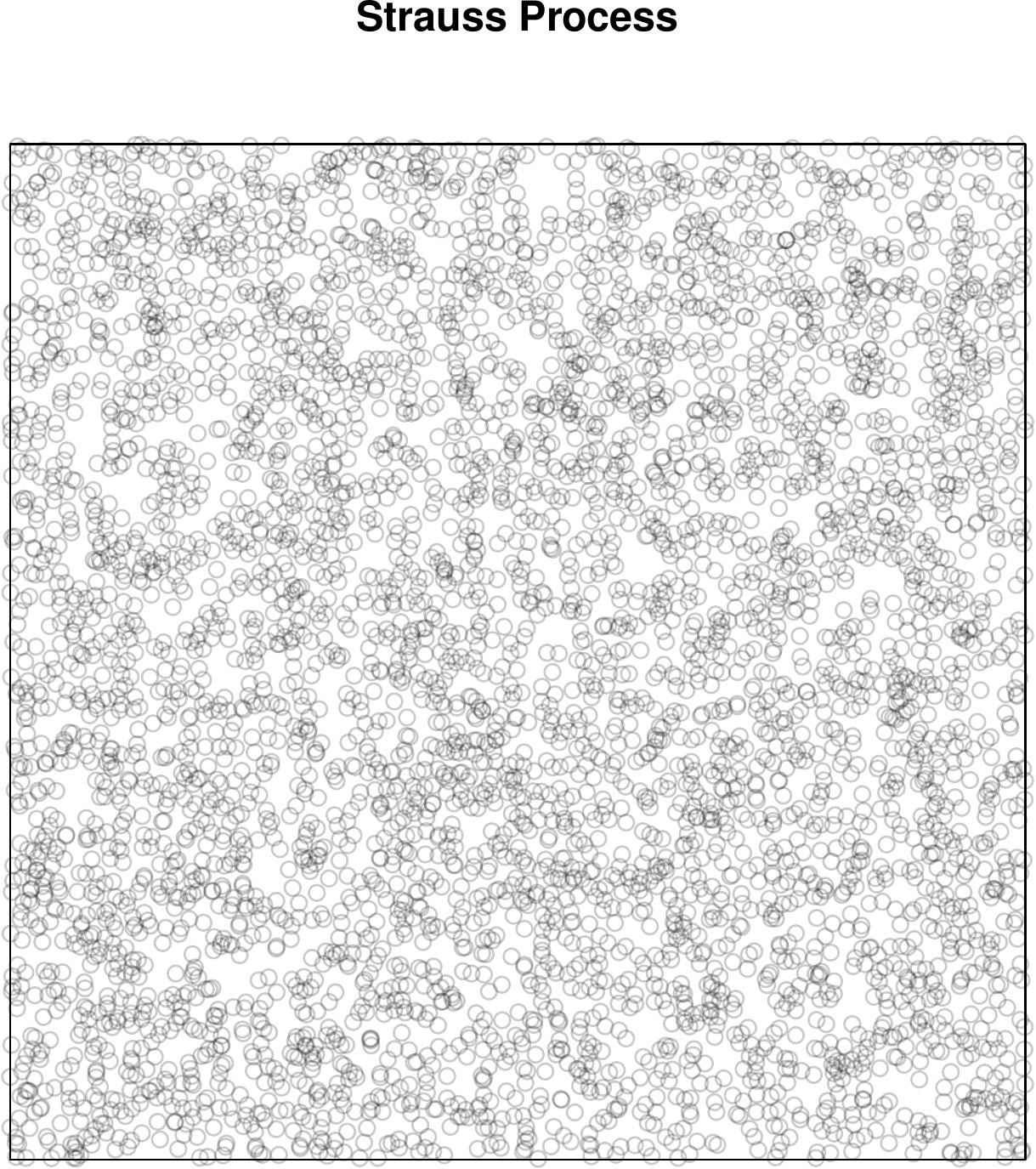}
\caption{Strauss Process.} 
\label{fig:pp15_strauss}
\end{figure}

Again, both the Bayesian and classical method correctly detects non-CSR, as shown by Figure \ref{fig:pp15}. For the Bayesian method, we set $K=500$ and $\hat C_1=0.15$.
\begin{figure}
\centering
\subfigure [HPP detection with Bayesian method for Strauss process.]{ \label{fig:hpp_bayesian15}
\includegraphics[width=5.5cm,height=5.5cm]{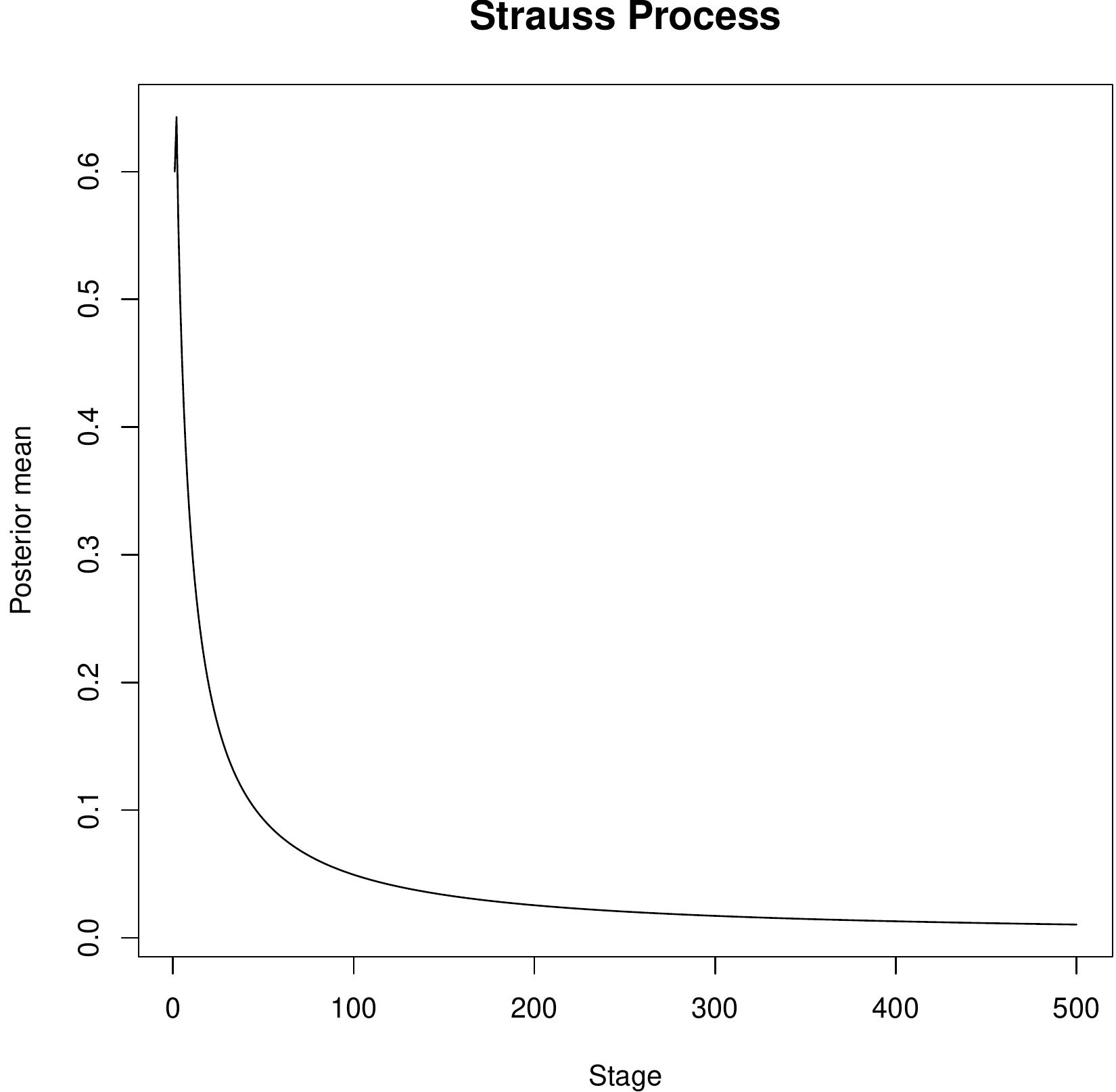}}
\hspace{2mm}
\subfigure [HPP detection with classical method for Strauss process.]{ \label{fig:hpp_classical15}
\includegraphics[width=5.5cm,height=5.5cm]{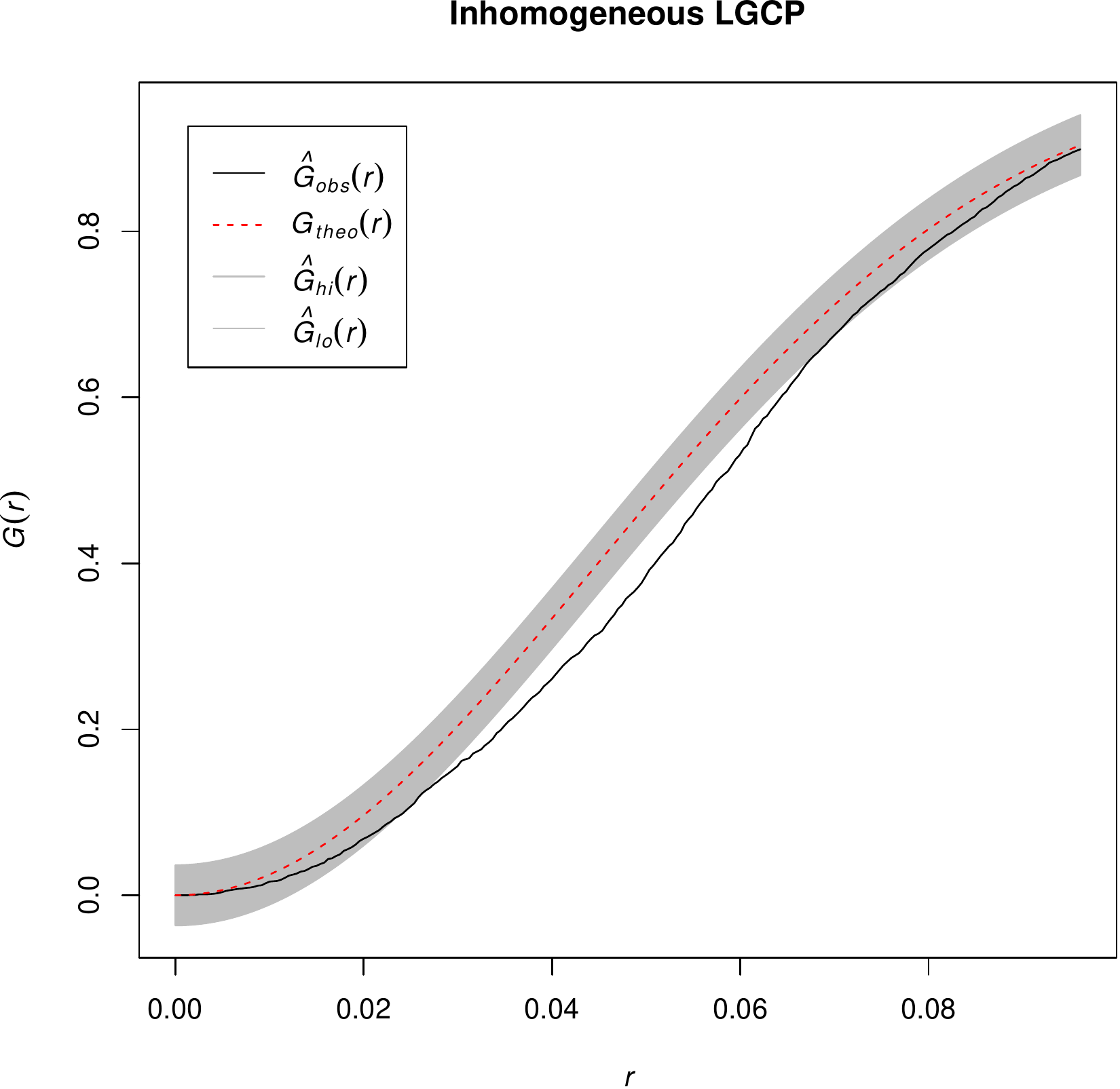}}
\caption{Detection of CSR with our Bayesian method and traditional classical method for Strauss process. Both the methods correctly
identify that the underlying point process is not CSR.} 
\label{fig:pp15}
\end{figure}

Again, stationarity of the process is clearly indicated by panel (a) of Figure \ref{fig:pp15_stationarity_indep}; here $K=500$ and $\hat C_1=0.15$.
Panel (b) shows dependence with $K=50$ and $\hat C_1=0.5$. 
\begin{figure}
\centering
\subfigure [Stationarity (Strauss process).]{ \label{fig:strauss2_bayesian_stationary}
\includegraphics[width=5.5cm,height=5.5cm]{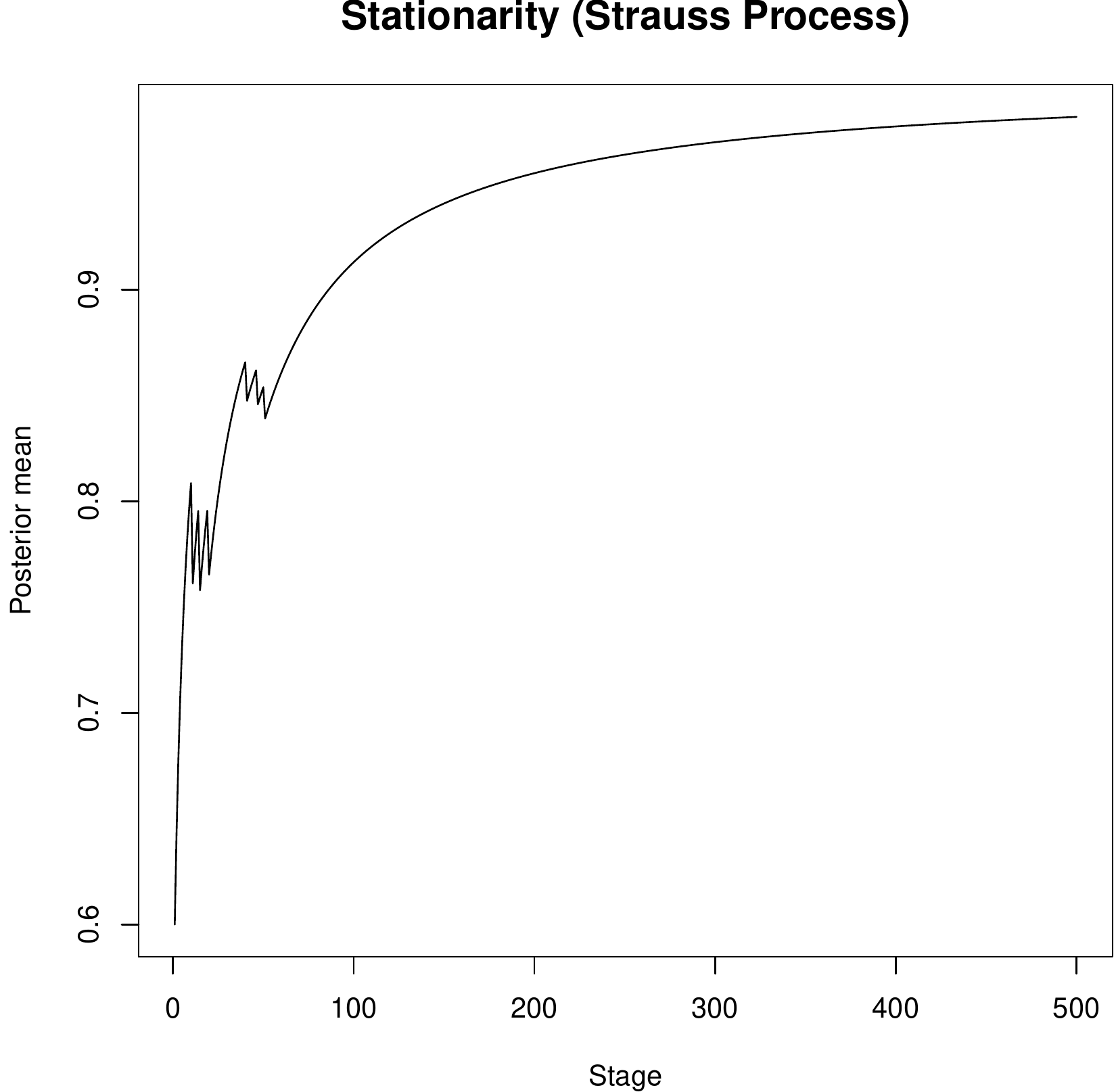}}
\hspace{2mm}
\subfigure [Dependent point process (Strauss process).]{ \label{fig:strauss2_bayesian_dependent}
\includegraphics[width=5.5cm,height=5.5cm]{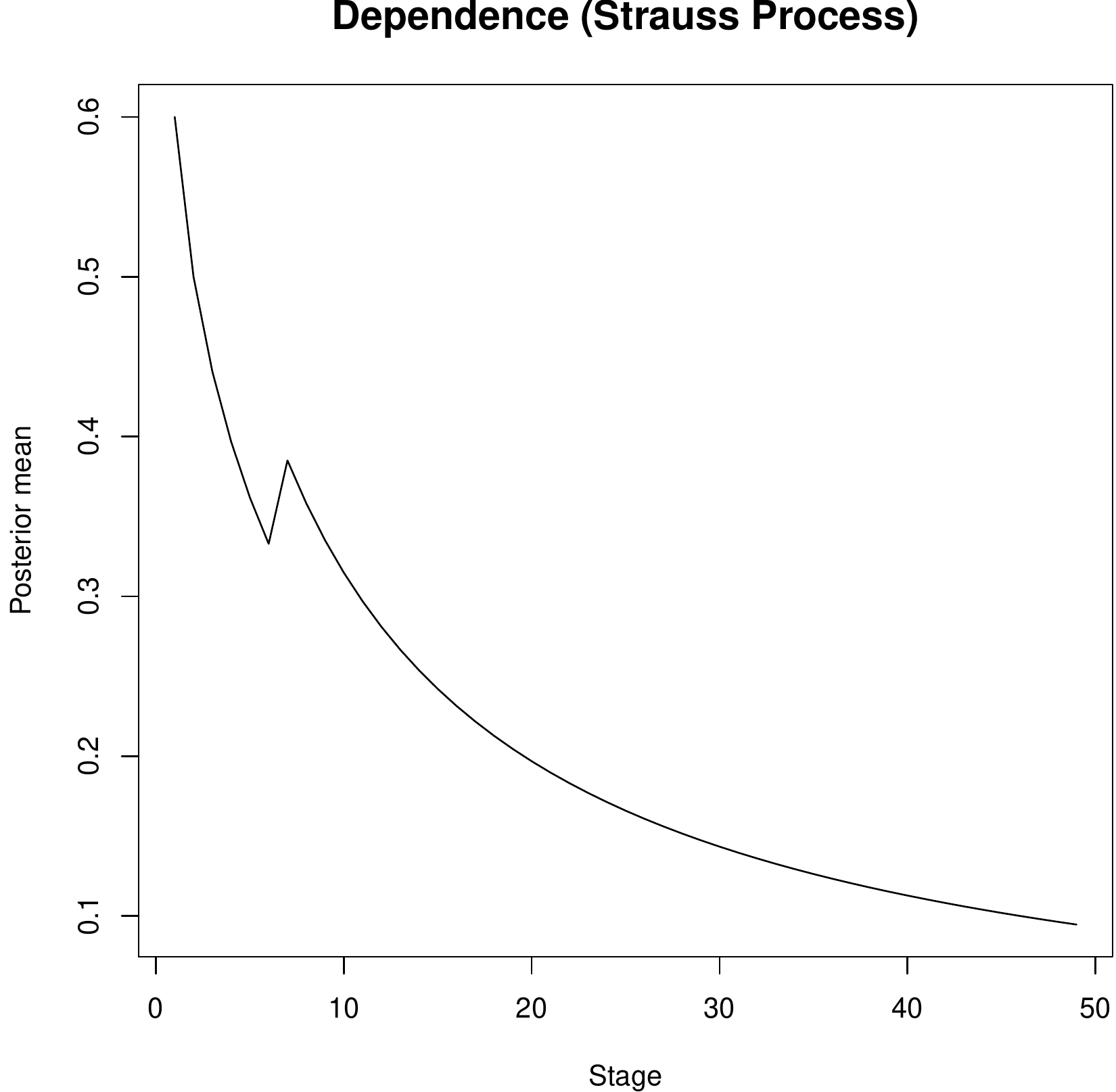}}
\caption{Detection of stationarity and dependence of Strauss process with our Bayesian method.} 
\label{fig:pp15_stationarity_indep}
\end{figure}

\section{Bayesian determination of frequencies of oscillatory stochastic processes}
\label{sec:osc_mult}

In this section we assume that the underlying stochastic process has multiple
frequencies of oscillations almost surely, including the possibility that the number of such frequencies is countably infinite.

\subsection{The key idea for Bayesian frequency determination}
\label{subsec:key_frequency}
Let us assume that there are $N~(\geq 1)$ frequencies of oscillations of the stochastic process $\bX=\left\{X_1,X_2,\ldots\right\}$. 
Here $N$ may even be countably infinite. Consider the transformed process $\bZ=\left\{Z_1,Z_2,\ldots\right\}$,
with $Z_j=\frac{\exp\left(X_j\right)}{1+\exp\left(X_j\right)}$; $j\geq 1$.
Hence, $Z_j\in[0,1]$. Now consider dividing up the interval $[0,1]$ into $\cup_{m=1}^M[\tilde p_{m-1},\tilde p_m]$, for $M>1$, such that 
$\tilde p_0=0$, $\tilde p_m=\tilde p_{m-1}+q_m$, where $\left\{q_m:m=1,\ldots,M\right\}$ is some probability distribution satisfying $0\leq q_m\leq 1$ for $m=1,\ldots,M$,
and $\sum_{m=0}^Mq_m=1$. Here $M$ can be even be infinite. 

For oscillating stochastic process $\bX$, for any $r>0$, $\bZ^r=\left\{Z^r_1,Z^r_2,\ldots\right\}$ is also an oscillating stochastic process 
taking values in $[0,1]$. Crucially, when raised to some sufficiently large positive power $r$, the originally smaller values of 
$\bZ$ tend to be much smaller compared to the originally larger values. These larger values of $\bZ^r$ will be contained in $[\tilde p_{m-1},\tilde p_m]$,
for large values of $m$. In particular, the largest values of $\bZ^r$ are expected to be contained in $(\tilde p_{M-1},1]$, or in $[\tilde p_{m_0-1},\tilde p_{m_0}]$
for $1\leq M_0<m_0<M$. Here $M_0$ is expected to be reasonably close to $M$. In the latter case, intervals of the form $[\tilde p_{m-1},\tilde p_m]$ 
will remain empty for $m>m_0$. The next largest values of $\bZ^r$ will be concentrated in $[\tilde p_{m_1-1},\tilde p_{m_1}]$ for some $1\leq M_1<m_1<m_0$. In this case,
$[\tilde p_{m-1},\tilde p_m]$ will remain empty for $m_1+1<m<m_0-1$, and so on. 

Note that the proportions of the values contained in the intervals constitute the frequencies of oscillations of the original process $\bX$.
We formalize this key idea into a Bayesian theory, treating $M$ as finite as well as infinite.

\subsection{Bayesian theory for finite $M$}
\label{subsec:finite_M}

To fix ideas, let us define
\begin{equation}
	Y_{j}=m~~\mbox{if}~~\tilde p_{m-1}<Z^r_j\leq\tilde p_m;~m=1,2,\ldots,M.
\label{eq:y_finite}
\end{equation}



We assume that
\begin{equation}
\left(\mathbb I(Y_{j}=1),\ldots,\mathbb I(Y_{j}=M)\right)
\sim Multinomial\left(1,p_{1,j},\ldots,p_{M,j}\right),
\label{eq:multinomial}
\end{equation}
where $p_{m,j}$ can be interpreted as the probability that $Z^r_j\in (\tilde p_{m-1},\tilde p_m]$.

Now note that for large $M$, the intervals $(\tilde p_{m-1},\tilde p_m]$ correspond to small regions of the index set of the stochastic process $\bX$,
and hence, the part of the process $\bZ^r$ falling in $(\tilde p_{m-1},\tilde p_m]$ can be safely regarded as stationary.
Further, assuming ergodicity of the process falling in the interval, 
it is expected that 
$p_{m,j}$ will tend to the correct proportion of the process $\bZ^r$ falling in $(\tilde p_{m-1},\tilde p_m]$, as $j\rightarrow\infty$.  
Notationally, we let $\left\{p_{m,0};~m=1,\ldots,M\right\}$ denote the actual proportions
of the process $\bZ^r$ falling in $(\tilde p_{m-1},\tilde p_m]$; $m=1,\ldots,M$.

Following the same principle discussed in Section \ref{sec:recursive1}, 
and extending the Beta prior to the Dirichlet prior, at the $k$-th stage we arrive at the
following posterior of $\left\{p_{m,k}:m=1,\ldots,M\right\}$:
\begin{equation}
\pi\left(p_{1,k},\ldots,p_{M,k}|y_{k}\right)
\equiv Dirichlet\left(\sum_{j=1}^k\frac{1}{j^2}+\sum_{j=1}^k\mathbb I\left(y_{j}=1\right),\ldots,
\sum_{j=1}^k\frac{1}{j^2}+\sum_{j=1}^k\mathbb I\left(y_{j}=M\right)\right).
\label{eq:posterior_dirichlet}
\end{equation}
The posterior mean and posterior variance of $p_{m,k}$, for $m=1,\ldots,M$, are given by:
\begin{align}
E\left(p_{m,k}|y_{k}\right)&=
\frac{\sum_{j=1}^k\frac{1}{j^2}+\sum_{j=1}^k\mathbb I\left(y_{j}=m\right)}
{M\sum_{j=1}^k\frac{1}{j^2}+k};
\label{eq:mean_dirichlet}\\
Var\left(p_{m,k}|y_{k}\right)&=
\frac{\left(\sum_{j=1}^k\frac{1}{j^2}+\sum_{j=1}^k\mathbb I\left(y_{j}=m\right)\right)
\left((M-1)\sum_{j=1}^k\frac{1}{j^2}+k-\sum_{j=1}^k\mathbb I\left(y_{j}=m\right)\right)}
{\left(M\sum_{j=1}^k\frac{1}{j^2}+k\right)^2\left(M\sum_{j=1}^k\frac{1}{j^2}+k+1\right)}.
\label{eq:var_dirichlet}
\end{align}
Since the process $\bZ^r$ falling in $(\tilde p_{m-1},\tilde p_m]$ is stationary and ergodic, it follows 
from (\ref{eq:mean_dirichlet}) and (\ref{eq:var_dirichlet}) it is easily seen, using 
$\frac{\sum_{j=1}^k\mathbb I\left(y_{j}=m\right)}{k}\rightarrow p_{m,0}$, almost surely, as $k\rightarrow\infty$,
that, almost surely,
\begin{align}
E\left(p_{m,k}|y_{k}\right)&\rightarrow p_{m,0},~~\mbox{and}
\label{eq:mean_dirichlet_convergence}\\
Var\left(p_{m,k}|y_{k}\right)&= O\left(\frac{1}{k}\right)\rightarrow 0,
\label{eq:var_dirichlet_convergence}
\end{align}
as $k\rightarrow\infty$.

Theorem \ref{theorem:finite_limit_points} formalizes the above arguments in terms of the limits of the marginal posterior
probabilities of $p_{m,k}$, denoted by $\pi_m\left(\cdot|y_k\right)$, as $k\rightarrow\infty$.
\begin{theorem}
\label{theorem:finite_limit_points}
Assume that $M$ is so large that $\bZ^r$ falling in the intervals $(\tilde p_{m-1},\tilde p_m]$; $m=1,\ldots,M$,
constitute stationary processes, and that such stationary processes are also ergodic.

Let $\mathcal N_{p_{m,0}}$ be any neighborhood of $p_{m,0}$, with $p_{m,0}$ satisfying
$0<p_{m,0}<1$ for $m=1,\ldots,M$ such that $\sum_{m=1}^Mp_{m,0}=1$.
Then
\begin{equation}
\pi_m\left(\mathcal N_{p_{m,0}}|y_{k}\right)\rightarrow 1,
\label{eq:consistency_at_p_m_0}
\end{equation}
almost surely
as $k\rightarrow\infty$. 
\end{theorem}
\begin{proof}
For any neighborhood of $p_{m,0}$, denoted by $\mathcal N_{p_{m,0}}$, let $\epsilon>0$ be sufficiently small so that 
$\mathcal N_{p_{m,0}}\supseteq\left\{|p_{m,k}-p_{m,0}|<\epsilon\right\}$. Then by Chebychev's inequality,
using (\ref{eq:mean_dirichlet_convergence}) and (\ref{eq:var_dirichlet_convergence}), it is seen that
$\pi_m\left(\mathcal N_{p_{m,0}}|y_k\right)\rightarrow 1$, almost surely, as $k\rightarrow\infty$.
\end{proof}

\begin{corollary}
\label{cor:cor_finite_1}
For adequate choices of $r$ and $M$, the non-zero distinct elements of $\left\{p_{m,0};~m=2,\ldots,M\right\}$ are the desired frequencies of 
the oscillating stochastic process $\bX$. Note that for adequately large $M$, $p_{1,0}$ is associated with the small values of $Z^r$, and hence
does not correspond to any frequency of the original stochastic process.
\end{corollary}

\subsection{Choice of $r$, $M$ and $\{q_1,\ldots,q_M\}$}
\label{subsec:choice_c}

In principle, the probability distribution $\{q_1,\ldots,q_M\}$ should be chosen based on prior information regarding which intervals contain the 
desired frequencies. Given sufficiently large $M$, the values of $q_m$ can then be chosen to shorten or widen any given interval.  
Short intervals are preferable when there is strong prior information of some frequency falling in the vicinity of some point. On the other hand,
larger intervals are appropriate in the case of weak prior information. Such prior knowledge may be obtained, say, by periodogram analysis of the 
underlying time series.

However, in our experiments, the uniform distribution $q_m=1/M$, for $m=1,\ldots,M$, yielded excellent results. For the choice of $r$, we recommend
that value for which the oscillations of $\bZ^r$ as distinctly visible as possible. The choice of $M$ should be such that 
$\left\{(\tilde p_{m-1},\tilde p_m];m=1,\ldots,M\right\}$
covers the range of $\bZ^r$ with adequately fine intervals. We discuss these issues in details with simulation studies and real data examples.

\subsection{Infinite number of frequencies}
\label{subsec:infinite_limit_points}

We now assume that the number of frequencies, $m$, is countably infinite,
and that $\left\{p_{m,0};m=1,2,3,\ldots\right\}$, where $0\leq p_{m,0}\leq 1$ and $\sum_{m=1}^{\infty}p_{m,0}=1$,
are the true proportions of the process $\bZ^r$ falling in the intervals $(\tilde p_{m-1},\tilde p_m]$; $m=1,2,\ldots$.

Now we define
\begin{equation}
Y_{j}=m~~\mbox{if}~~\tilde p_{m-1}<Z^r_j\leq \tilde p_{m};~m=1,2,\ldots,\infty.
\label{eq:y_infinite}
\end{equation}

Let $\mathcal X=\left\{1,2,\ldots\right\}$ and let $\mathcal B\left(\mathcal X\right)$ denote the Borel
$\sigma$-field on $\mathcal X$ (assuming every singleton of $\mathcal X$ is an open set). Let $\mathcal P$
denote the set of probability measures on $\mathcal X$. Then, at the $j$-th stage,
\begin{equation}
[Y_j|P_j]\sim P_j,
\label{eq:Y_DP}
\end{equation}
where $P_j\in\mathcal P$. We assume that
$P_j$ is the following Dirichlet process (see \ctn{Ferguson73}):
\begin{equation}
P_j\sim DP\left(\frac{1}{j^2}G\right),
\label{eq:DP}
\end{equation}
where, the probability measure $G$ is such that, for every $j\geq 1$,
\begin{equation}
G\left(Y_j=m\right)=\frac{1}{2^m}.
\label{eq:G}
\end{equation}
It then follows using the same previous principles that, at the $k$-th stage, 
the posterior of $P_k$ is again a Dirichlet process, given by
\begin{equation}
[P_k|y_k]\sim DP\left(\sum_{j=1}^k\frac{1}{j^2}G+\sum_{j=1}^k\delta_{y_j}\right),
\label{eq:posterior_DP}
\end{equation}
where $\delta_{y_j}$ denotes point mass at $y_j$.
It follows from (\ref{eq:posterior_DP}) that
\begin{align}
E\left(p_{m,k}|y_{k}\right)&=
\frac{\frac{1}{2^m}\sum_{j=1}^k\frac{1}{j^2}+\sum_{j=1}^k\mathbb I\left(y_{j}=m\right)}
{\sum_{j=1}^k\frac{1}{j^2}+k};
\label{eq:mean_DP}\\
Var\left(p_{m,k}|y_{k}\right)&=
\frac{\left(\sum_{j=1}^k\frac{1}{j^2}+\sum_{j=1}^k\mathbb I\left(y_{j}=m\right)\right)
\left((1-\frac{1}{2^m})\sum_{j=1}^k\frac{1}{j^2}+k-\sum_{j=1}^k\mathbb I\left(y_{j}=m\right)\right)}
{\left(\sum_{j=1}^k\frac{1}{j^2}+k\right)^2\left(\sum_{j=1}^k\frac{1}{j^2}+k+1\right)}.
\label{eq:var_DP}
\end{align}
As before, it easily follows from (\ref{eq:mean_DP}) and (\ref{eq:var_DP}) that
for $m=1,2,3,\ldots$, 
\begin{align}
E\left(p_{m,k}|y_{k}\right)&\rightarrow p_{m,0},~~\mbox{and}
\label{eq:mean_DP_convergence}\\
Var\left(p_{m,k}|y_{k}\right)&= O\left(\frac{1}{k}\right)\rightarrow 0,
\label{eq:var_DP_convergence}
\end{align}
almost surely, as $k\rightarrow\infty$.

The theorem below formalizes the above arguments in the infinite number of frequency situation 
in terms of the limit of the marginal posterior
probabilities of $p_{m,k}$, 
as $k\rightarrow\infty$. 
\begin{theorem}
\label{theorem:infinite_limit_points}
Assume that $\bZ^r$ falling in the intervals $(\tilde p_{m-1},\tilde p_m]$; $m=1,2,\ldots$,
constitute stationary processes, and that such stationary processes are also ergodic.
	
Let $\mathcal N_{p_{m,0}}$ be any neighborhood of $p_{m,0}$, with $p_{m,0}$ satisfying		
$0\leq p_{m,0}\leq 1$ for $m=1,2,\ldots$ such that $\sum_{m=1}^{\infty}p_{m,0}=1$, with at most finite number
of $m$ such that $p_{m,0}=0$.
Then with $Y_j$ defined as in (\ref{eq:y_infinite}),
\begin{equation}
\pi_m\left(\mathcal N_{p_{m,0}}|y_{k}\right)\rightarrow 1,
\label{eq:consistency_at_p_m_0_DP}
\end{equation}
almost surely, as $k\rightarrow\infty$. 
\end{theorem}
\begin{proof}
Follows using the same ideas as the proof of Theorem \ref{theorem:finite_limit_points}.
\end{proof}

\begin{corollary}
\label{cor:cor_infinite_1}
The non-zero distinct elements of $\left\{p_{m,0};~m=1,2,\ldots\right\}$ are the desired frequencies of 
the oscillating stochastic process $\bX$. Again, $p_{1,0}$ does not correspond to any frequency of the original stochastic process.
\end{corollary}

\begin{remark}
\label{remark:remark_infinite_1}
As regards the choice of the quantities $q_m$, we suggest setting $q_m=2^{-m}$, for $m\geq 1$, which is the same as the base measure
for the Dirichlet process prior. For countably infinite number of frequencies, the choice of $r$ is difficult to decide. But we hope that
selecting $r$ such that most of the oscillations are visible as much as possible, will work even in this situation.
\end{remark}

\begin{remark}
\label{remark:remark_infinite_2}
It is useful to remark that our theory with countably infinite number of frequencies is readily
applicable to 
situations where the number of frequencies
is finite but unknown. In such cases, only a finite number of the probabilities $\left\{p_{m,j};~m=2,3\ldots\right\}$ 
will have posterior probabilities around positive quantities, while the rest will concentrate around zero.
For known finite number of limit points, it is only required to specify $G$ such that it 
gives positive mass to only a specific finite set.
\end{remark}

We now illustrate our Bayesian theory for detecting frequencies using simulation studies.

\subsection{Simulation study with a single frequency}
\label{subsec:single_frequency}

Following Example 2.8 of \ctn{Shumway06}, we generate $T = 500$ observations from the model
\begin{equation}
x_t = A \cos(2\pi\omega t + \varphi) + \epsilon_t,
\label{eq:single_freq}
\end{equation}
where $\omega = 1/50$, $A = 2$, $\varphi = 0.6\pi$, and $\epsilon_t\stackrel{iid}{\sim}N\left(0,\sigma^2\right)$, with $\sigma = 5$.
Figure \ref{fig:osc_series1} displays the generated time series. Observe that due to the relatively large $\sigma$, the true frequency is blurred
in the observed time series.
Our goal is to recover the frequency $\omega=1/50$ using our Bayesian method, pretending
that the true frequency is unknown.
\begin{figure}
\centering
\includegraphics[width=10cm,height=6cm]{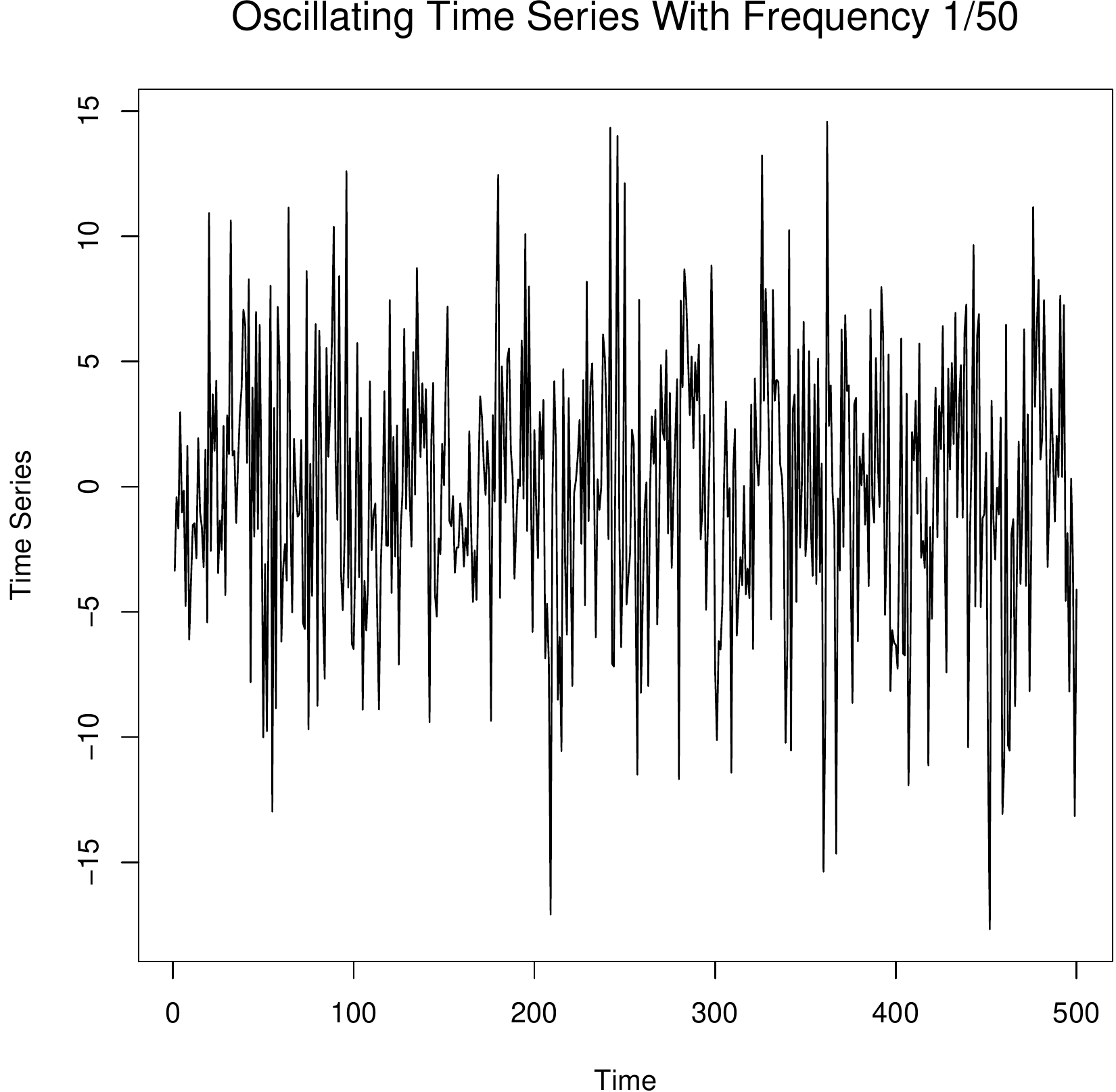}
\caption{Simulated oscillating time series with true frequency $0.02$. }
\label{fig:osc_series1}
\end{figure}

We apply our Bayesian technique based on Dirichlet process, but with the base measure $G_0$ giving probability $1/M$ to each of the values
$1,\ldots,M$. Since our method depends crucially on the choices of $r$ and $M$, it is important to carefully choose these quantities. 
As we had already prescribed, $r$ should be so chosen that the oscillations of $\bZ^r$ are easy to visualize. Figure \ref{fig:osc_example1_r}
shows the transformed time series $\bZ^r$ for different values of $r$. 
In this example we see that as $r$ is increased, the oscillations tend to be more and more explicit.
Thus, it seems that $r=1000$ is the best choice among those experimented with.
\begin{figure}
\centering
\subfigure [Transformed series $\bZ^{10}$.]{ \label{fig:r1}
\includegraphics[width=6.5cm,height=5cm]{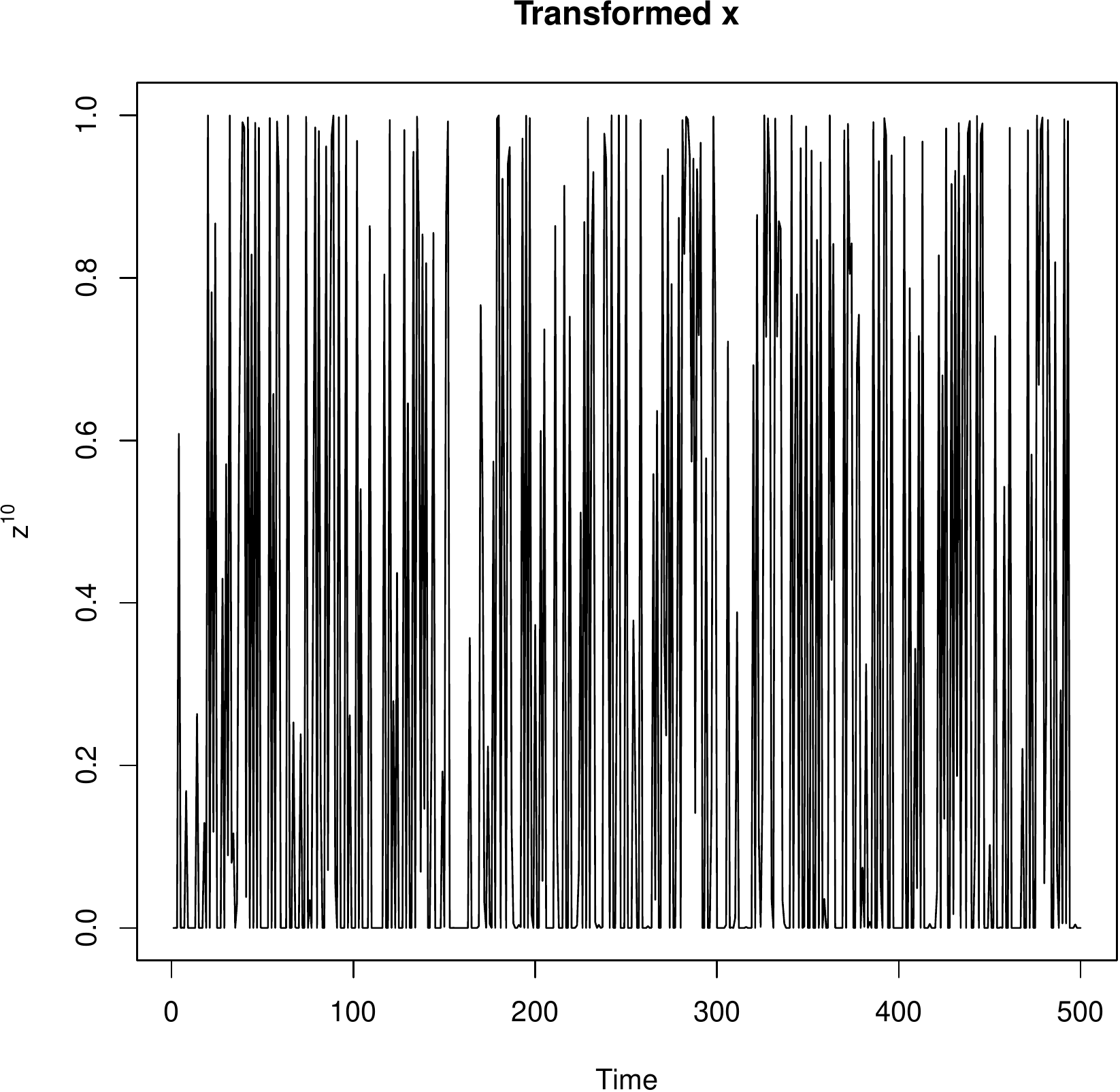}}
\hspace{2mm}
\subfigure [Transformed series $\bZ^{50}$.]{ \label{fig:r2}
\includegraphics[width=6.5cm,height=5cm]{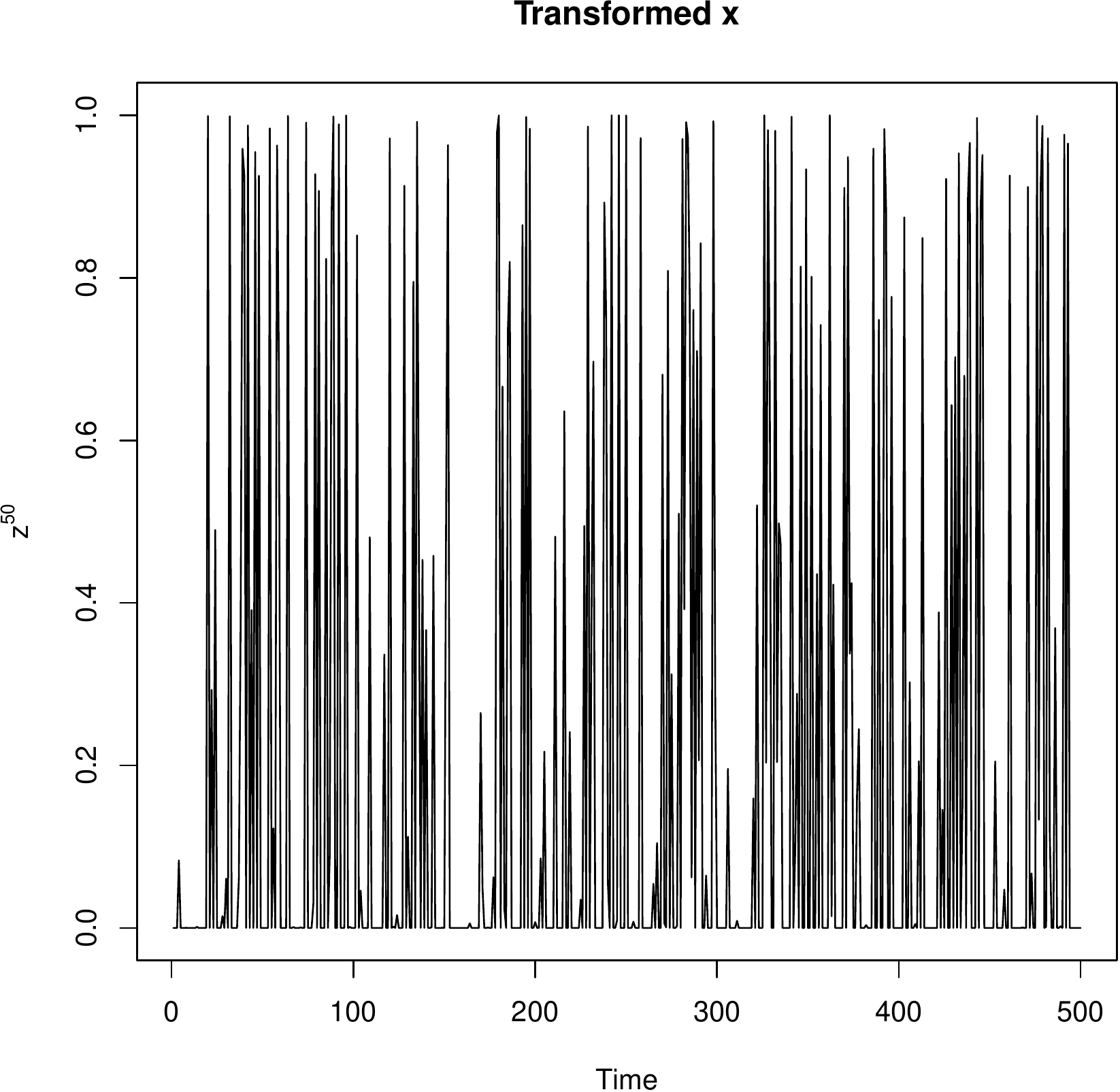}}\\
\vspace{2mm}
\subfigure [Transformed series $\bZ^{100}$.]{ \label{fig:r3}
\includegraphics[width=6.5cm,height=5cm]{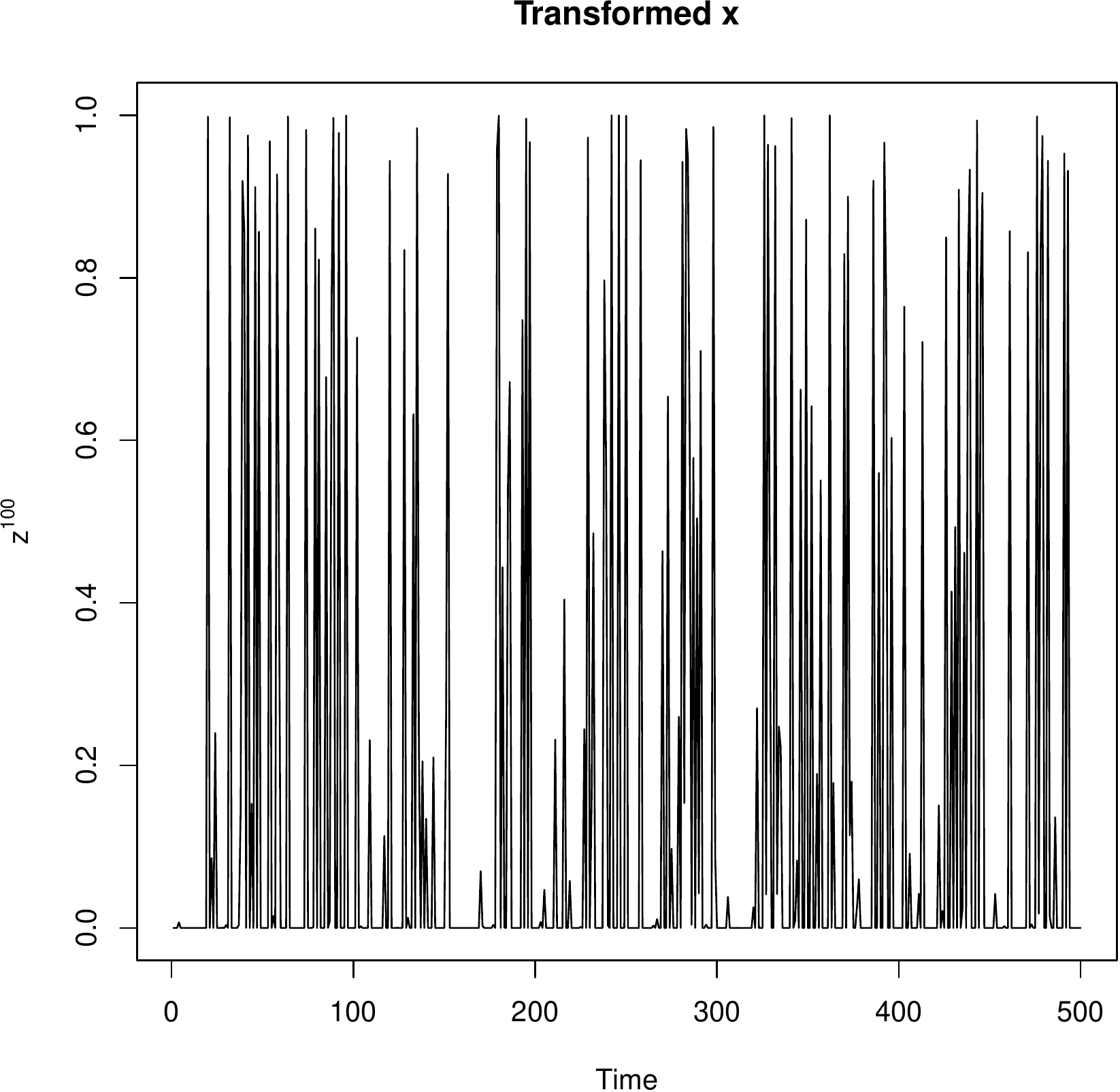}}
\hspace{2mm}
\subfigure [Transformed series $\bZ^{500}$.]{ \label{fig:r4}
\includegraphics[width=6.5cm,height=5cm]{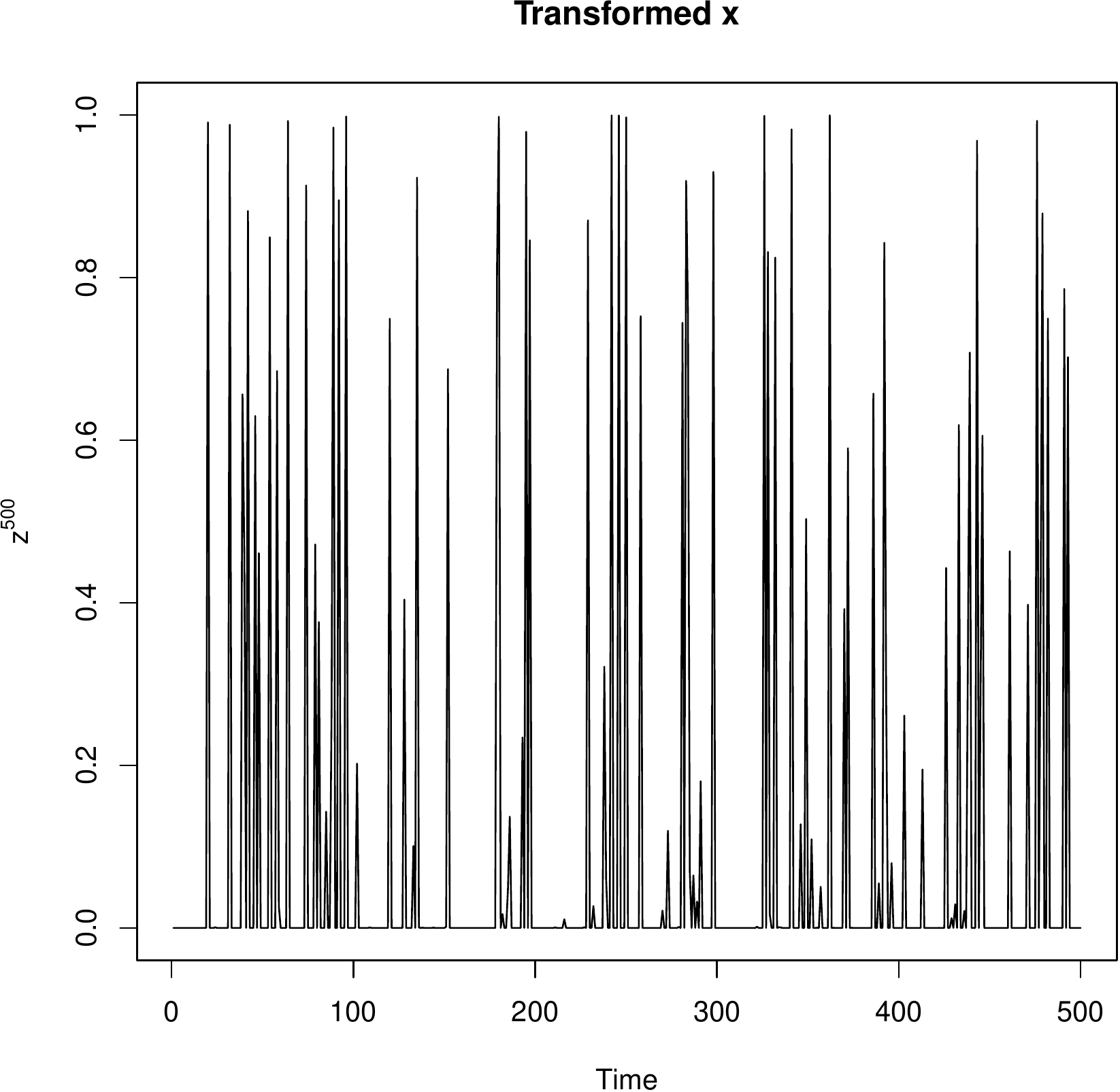}}\\
\vspace{2mm}
\subfigure [Transformed series $\bZ^{1000}$.]{ \label{fig:r5}
\includegraphics[width=6.5cm,height=5cm]{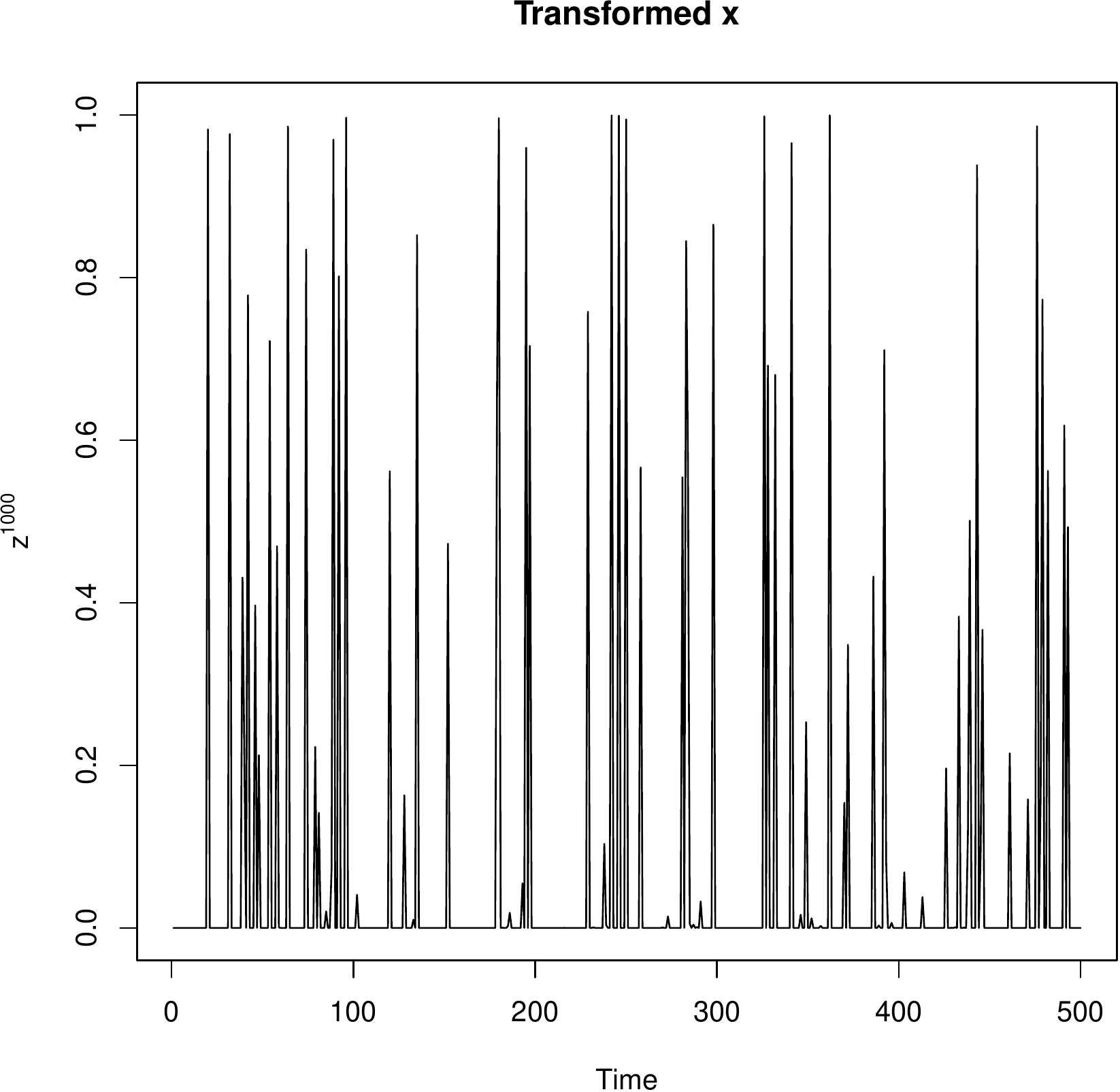}}
\caption{Illustration of effects of $r$ in $\bZ^r$ in determining single frequency in (\ref{eq:single_freq}). Here the true frequency is $0.02$. }
\label{fig:osc_example1_r}
\end{figure}

For the choice of $M$ we need to select a large enough value such that the range of $\bZ^r$ gets adequately partitioned within 
$\left\{(\tilde p_{m-1},\tilde p_m];m=1,\ldots,M\right\}$. In other words, relatively large values of $r$ and $M$ are expected to yield good Bayesian results.
We investigate this by implementing our Bayesian method for different values of $r$ and $M$ and comparing the results.

Figures \ref{fig:osc_example1} and \ref{fig:osc_example2} depict the results of our Bayesian method for various choices of $r$ and $M$.
As shown by the figures, for increasing values of $r=10,50,100,500,1000$, and $M=10,50,100$, the posterior of $p_{M,j}$ 
associated with the interval $(\tilde p_{M-1},\tilde p_M]$, increasingly converges to the true frequency $0.02$. Note that for relatively small values of either
$r$ or $M$, the relevant posteriors fail to converge. 
Thus, the results are in keeping with our expectation of obtaining superior results for large values of $r$ and $M$.
Note that the rate of convergence of the posterior seems to be faster with respect to increasing values of $r$ compared to increasing values of $M$.
Thus, appropriate choice of $r$ seems to be more important than $M$.
\begin{figure}
\centering
\subfigure [$r=10,M=10$.]{ \label{fig:osc_1}
\includegraphics[width=4.5cm,height=4.5cm]{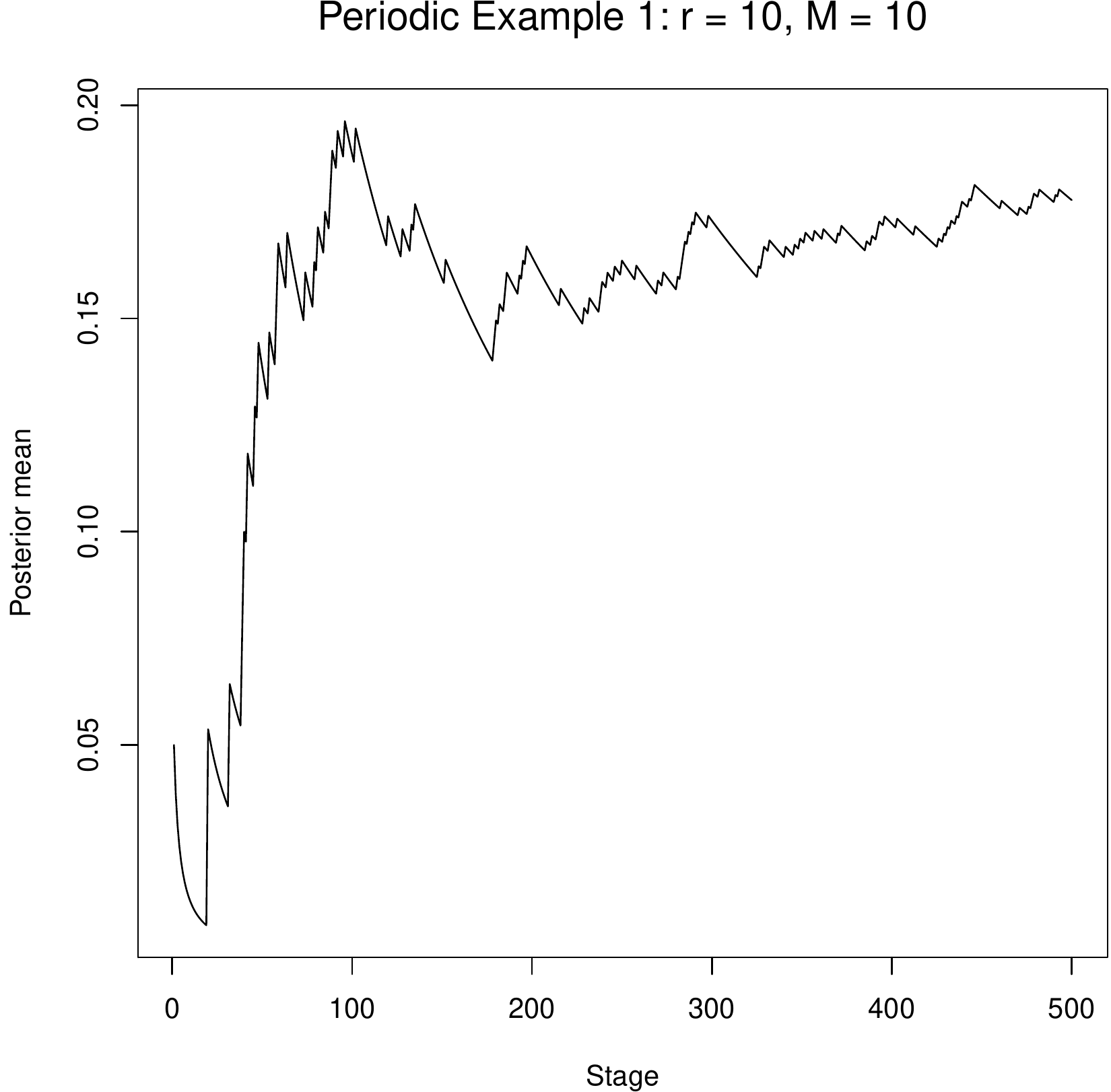}}
\hspace{2mm}
\subfigure [$r=10,M=50$.]{ \label{fig:osc_2}
\includegraphics[width=4.5cm,height=4.5cm]{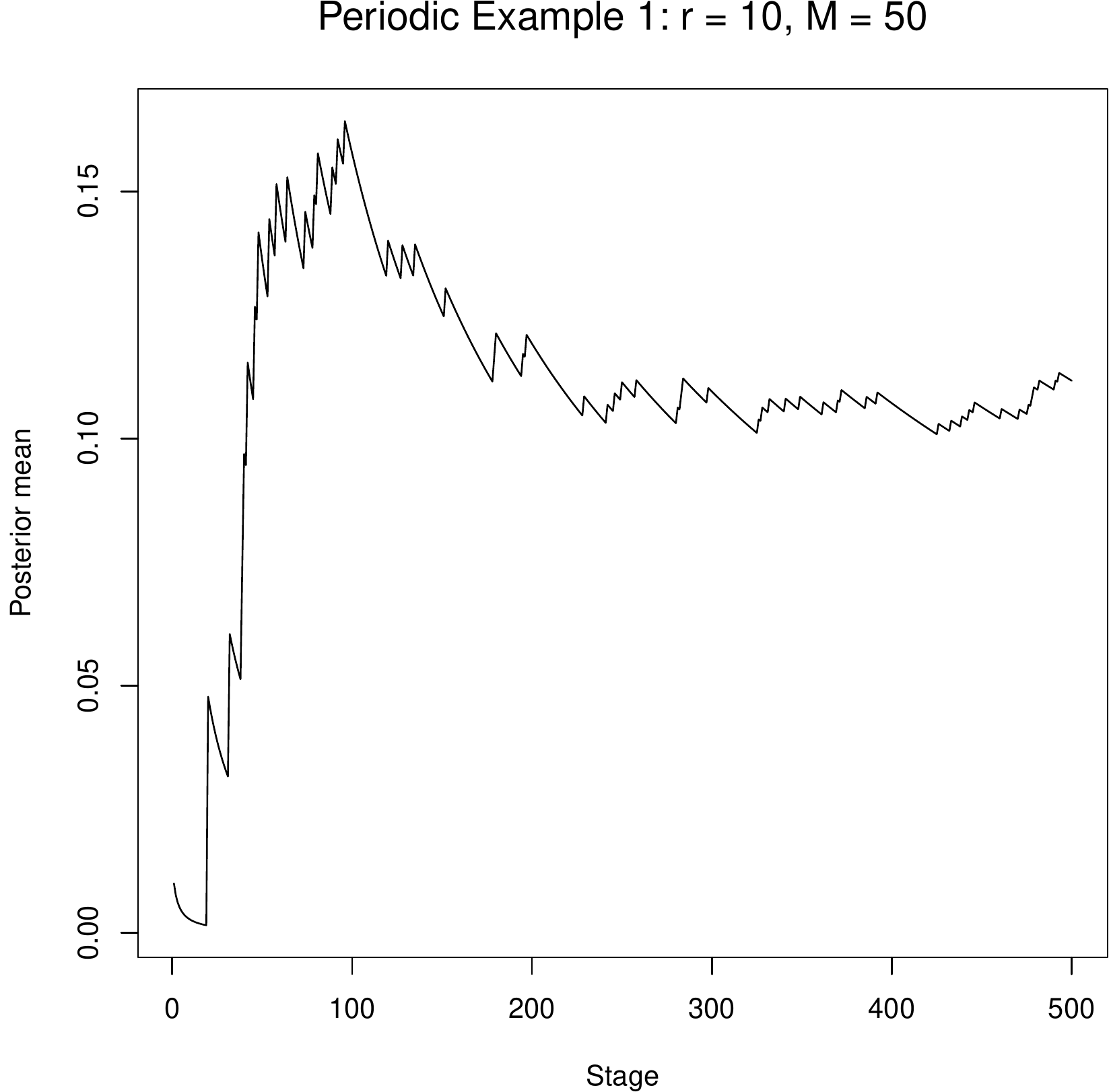}}
\hspace{2mm}
\subfigure [$r=10,M=100$.]{ \label{fig:osc_3}
\includegraphics[width=4.5cm,height=4.5cm]{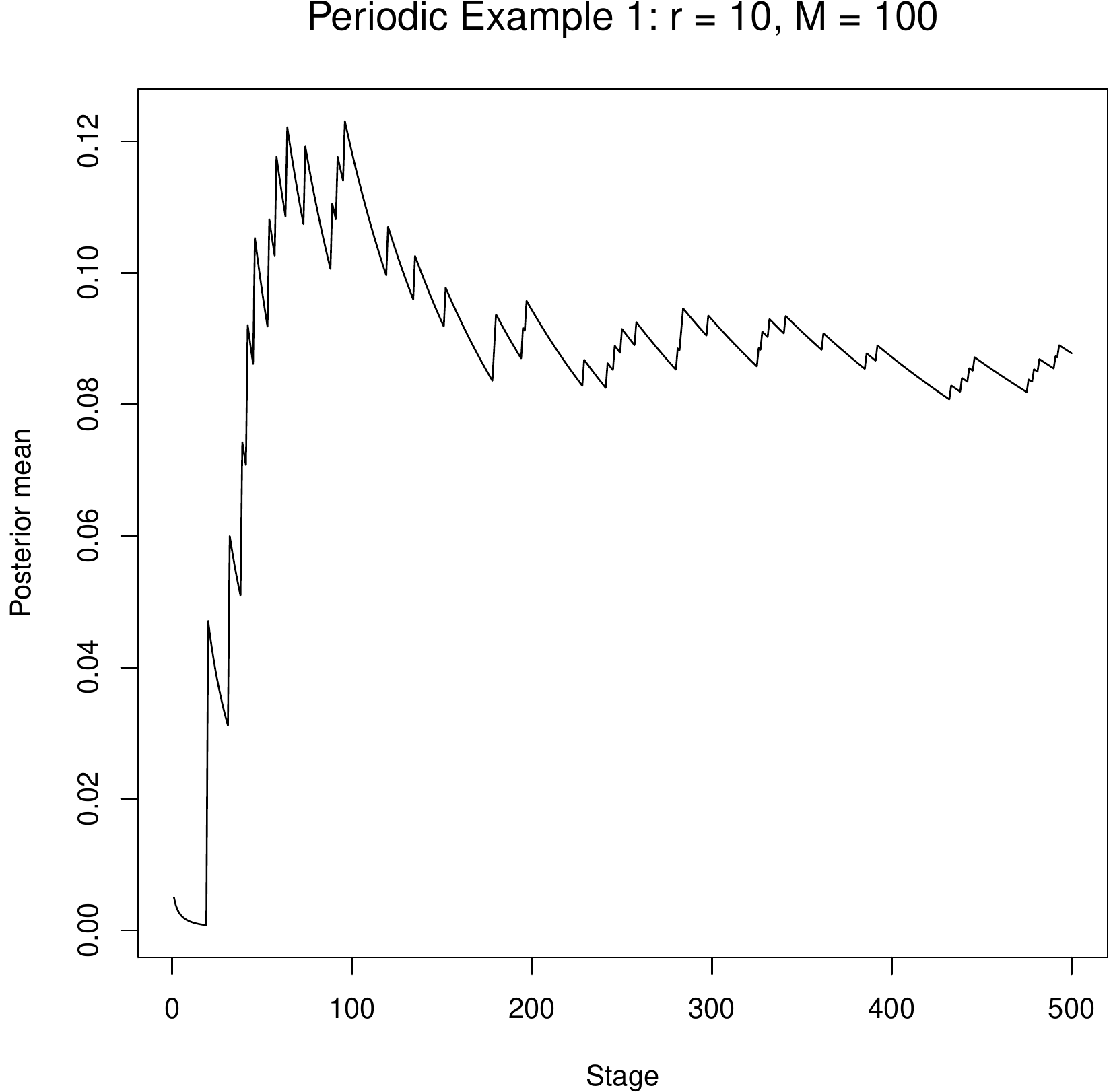}}\\
\vspace{2mm}
\subfigure [$r=50,M=10$.]{ \label{fig:osc_4}
\includegraphics[width=4.5cm,height=4.5cm]{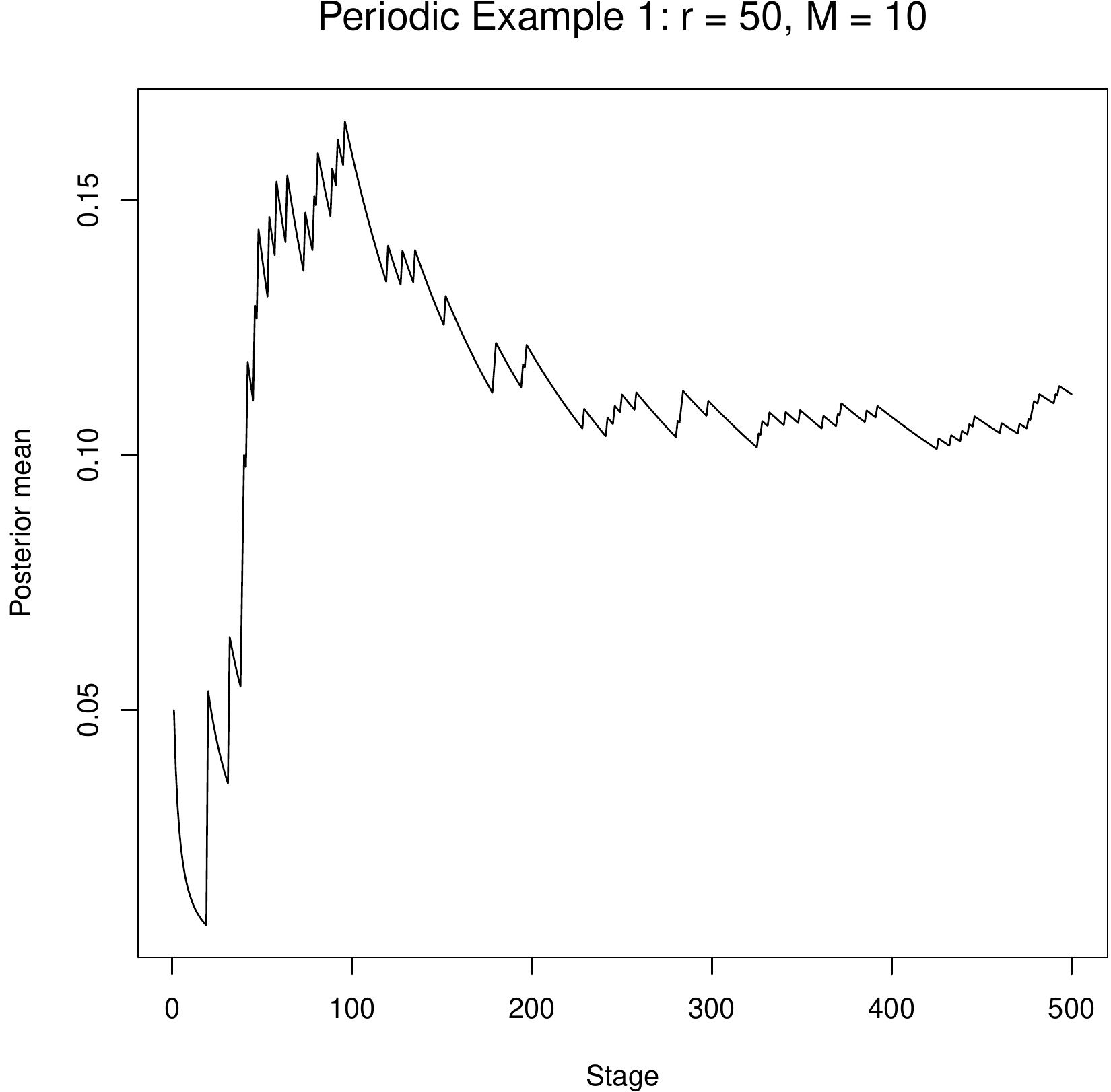}}
\hspace{2mm}
\subfigure [$r=50,M=50$.]{ \label{fig:osc_5}
\includegraphics[width=4.5cm,height=4.5cm]{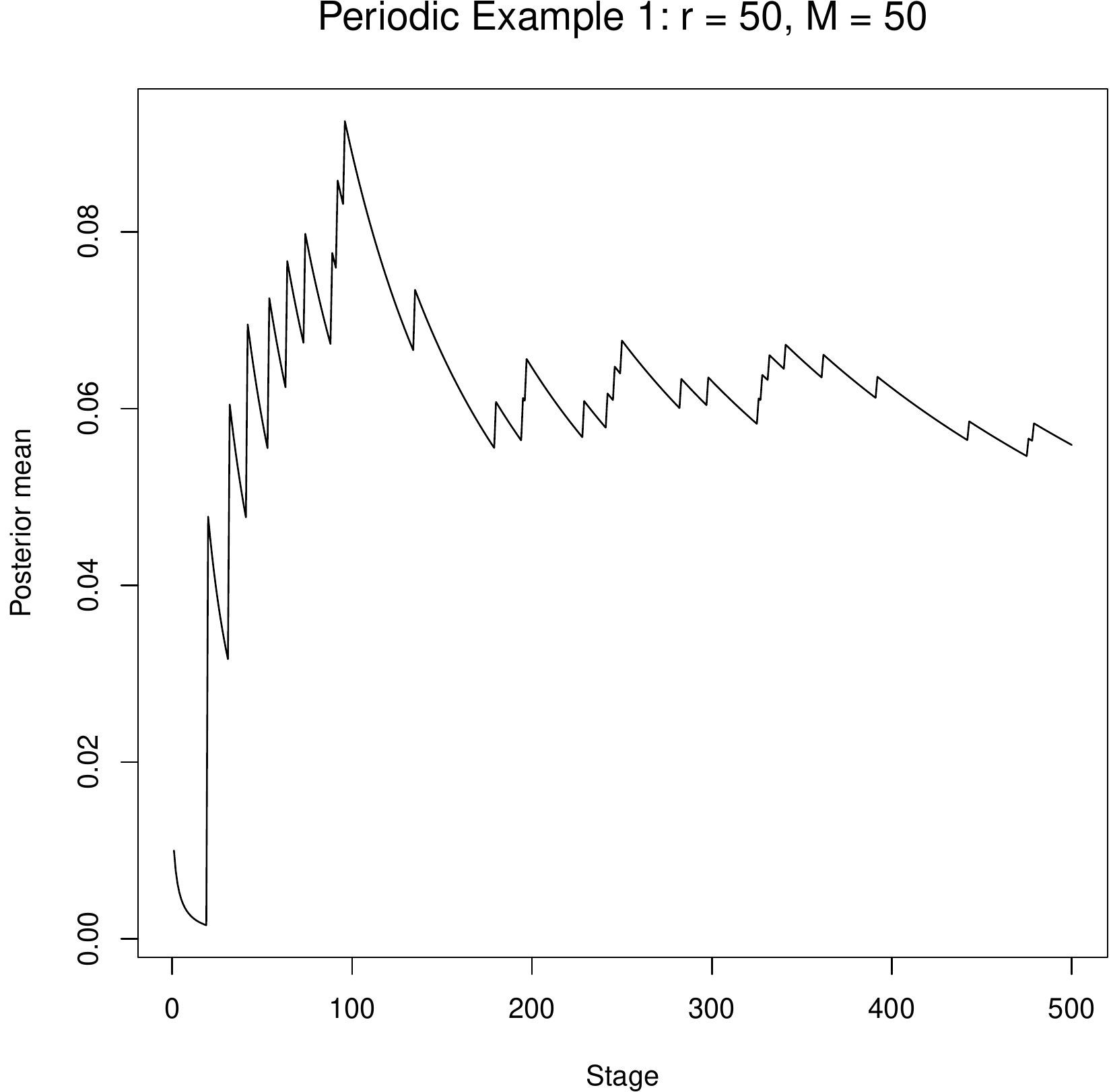}}
\hspace{2mm}
\subfigure [$r=50,M=100$.]{ \label{fig:osc_6}
\includegraphics[width=4.5cm,height=4.5cm]{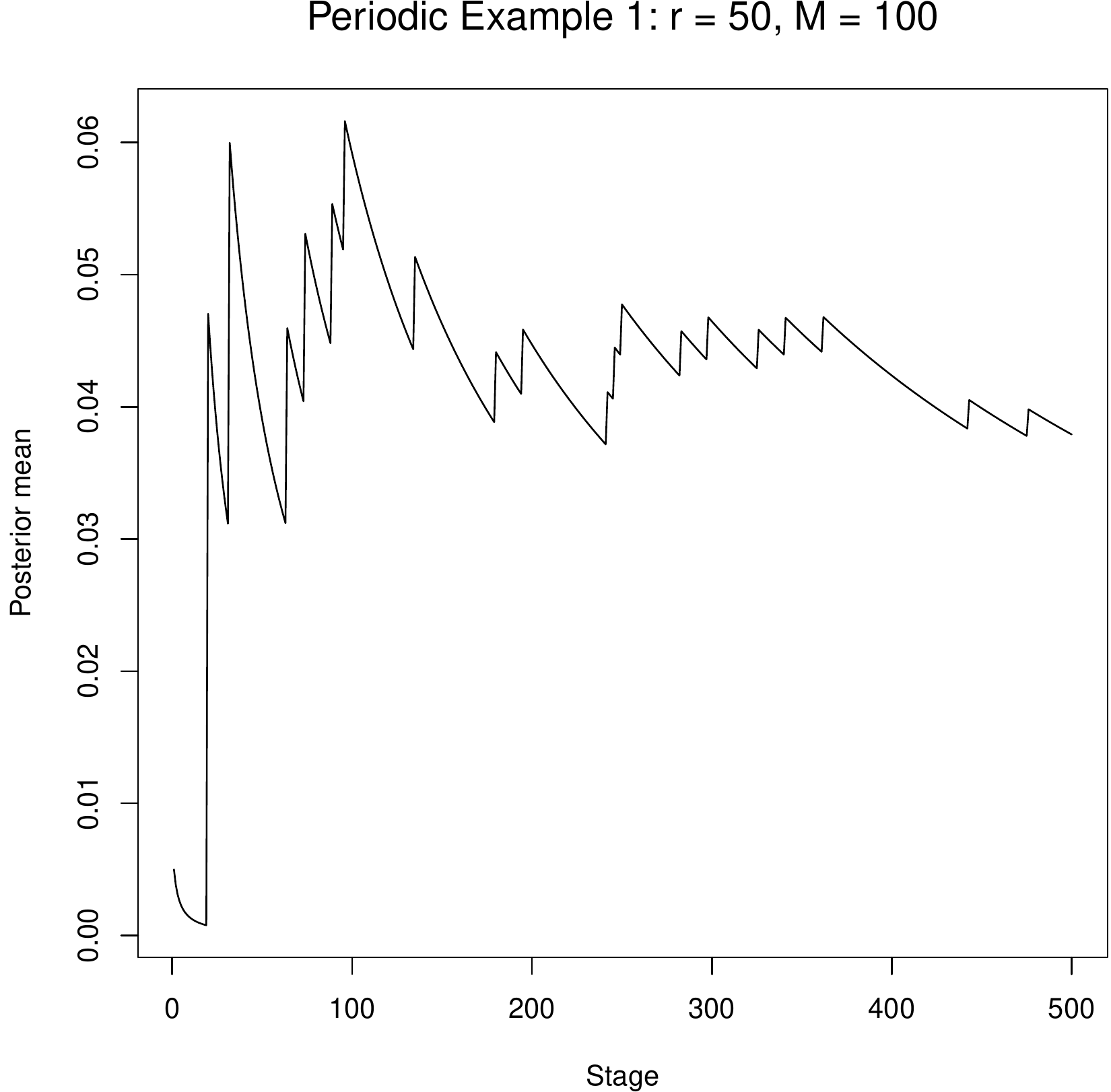}}\\
\vspace{2mm}
\subfigure [$r=100,M=10$.]{ \label{fig:osc_7}
\includegraphics[width=4.5cm,height=4.5cm]{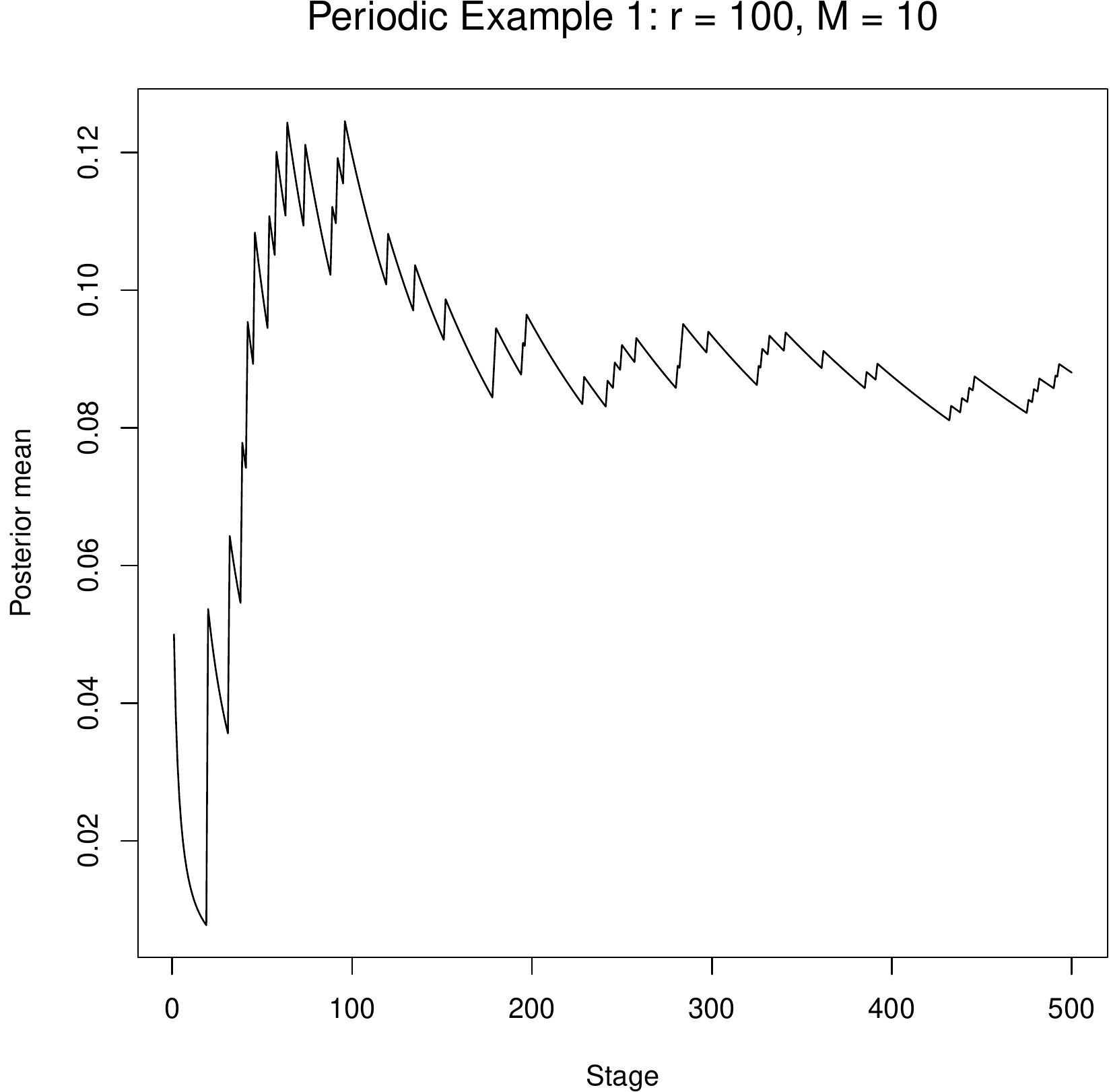}}
\hspace{2mm}
\subfigure [$r=100,M=50$.]{ \label{fig:osc_8}
\includegraphics[width=4.5cm,height=4.5cm]{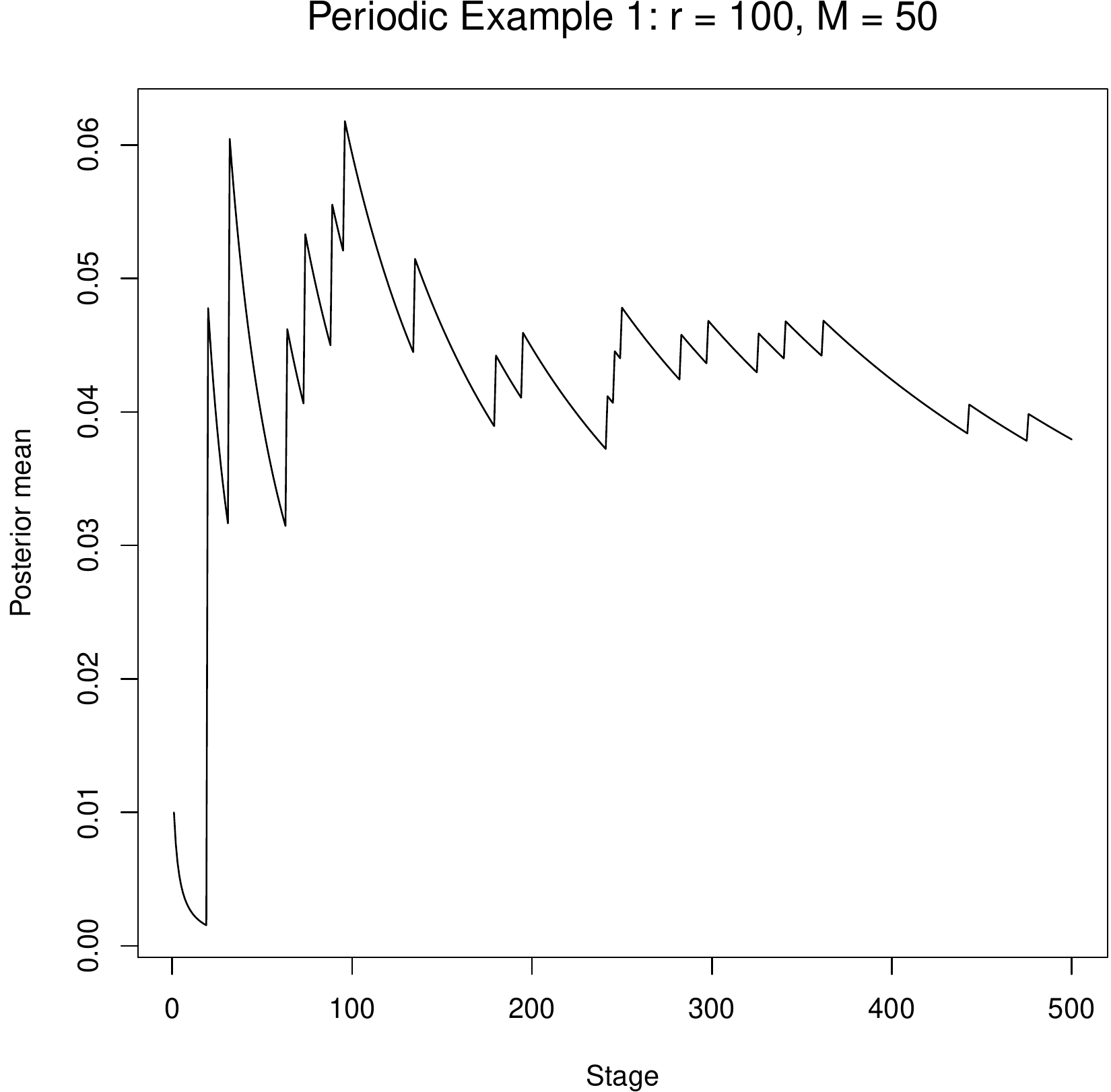}}
\hspace{2mm}
\subfigure [$r=100,M=100$.]{ \label{fig:osc_9}
\includegraphics[width=4.5cm,height=4.5cm]{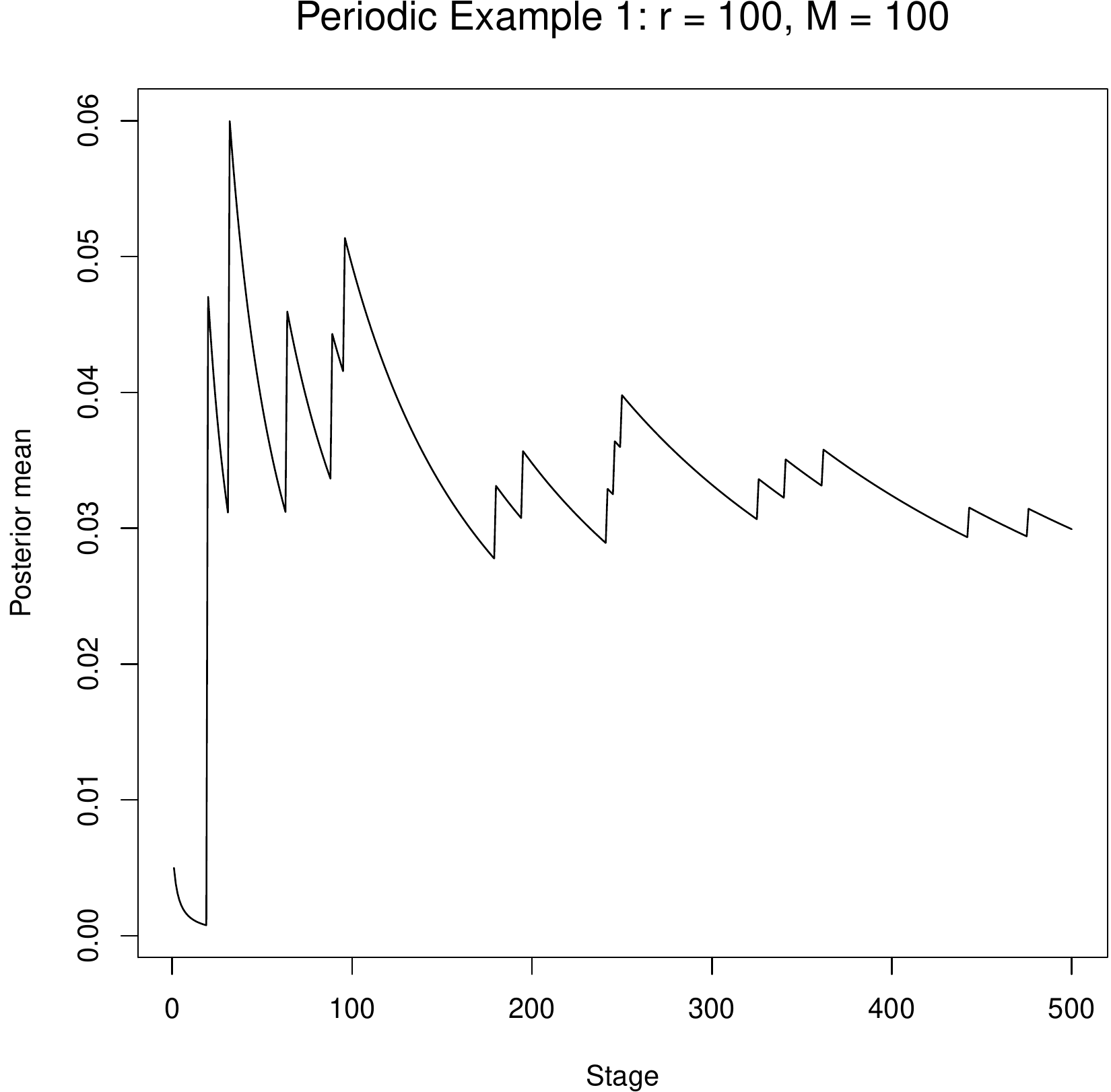}}
\caption{Illustration of our Bayesian method for determining single frequency. Here the true frequency is $0.02$. }
\label{fig:osc_example1}
\end{figure}

\begin{figure}
\centering
\subfigure [$r=500,M=10$.]{ \label{fig:osc_10}
\includegraphics[width=4.5cm,height=4.5cm]{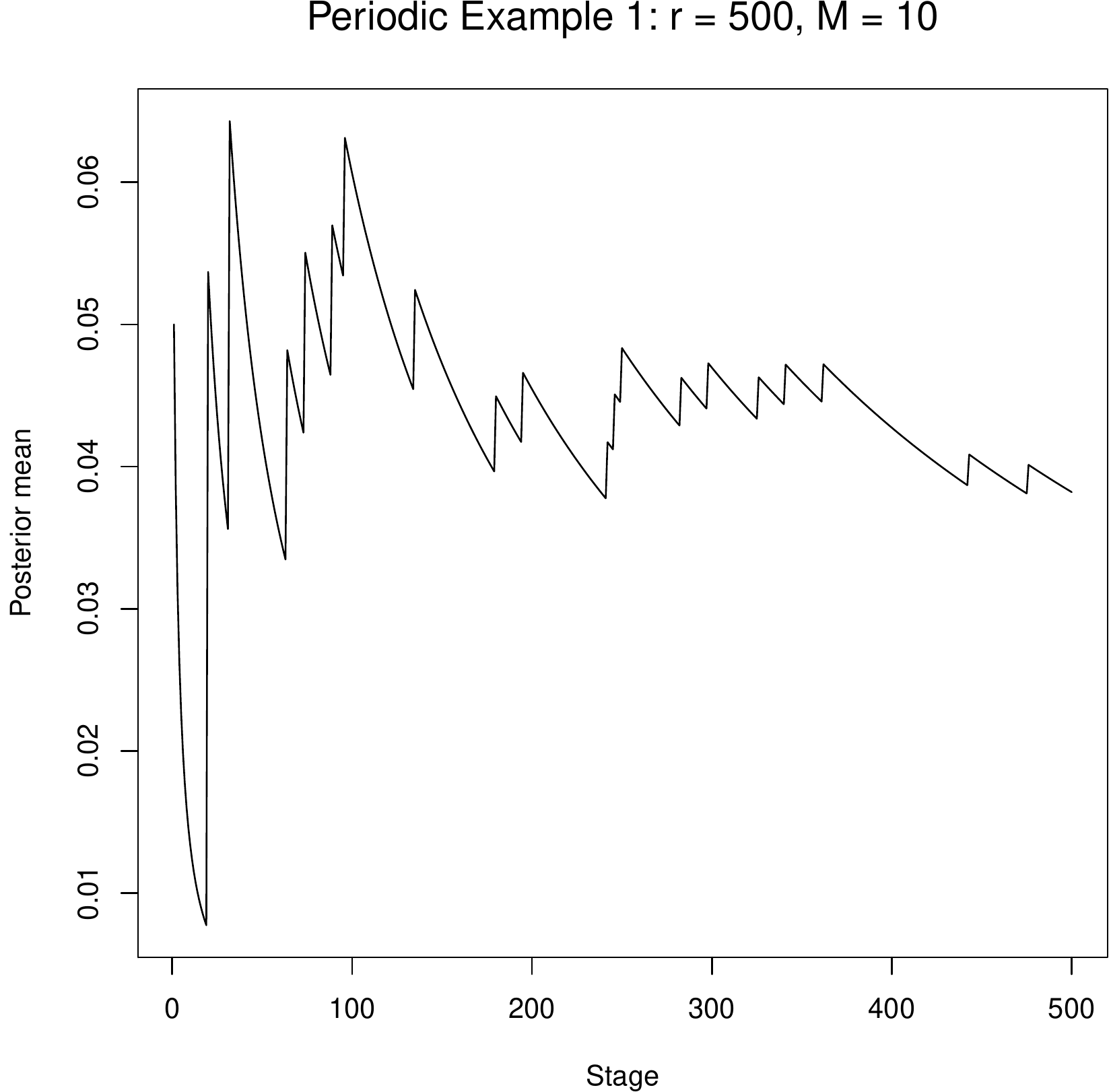}}
\hspace{2mm}
\subfigure [$r=500,M=50$.]{ \label{fig:osc_11}
\includegraphics[width=4.5cm,height=4.5cm]{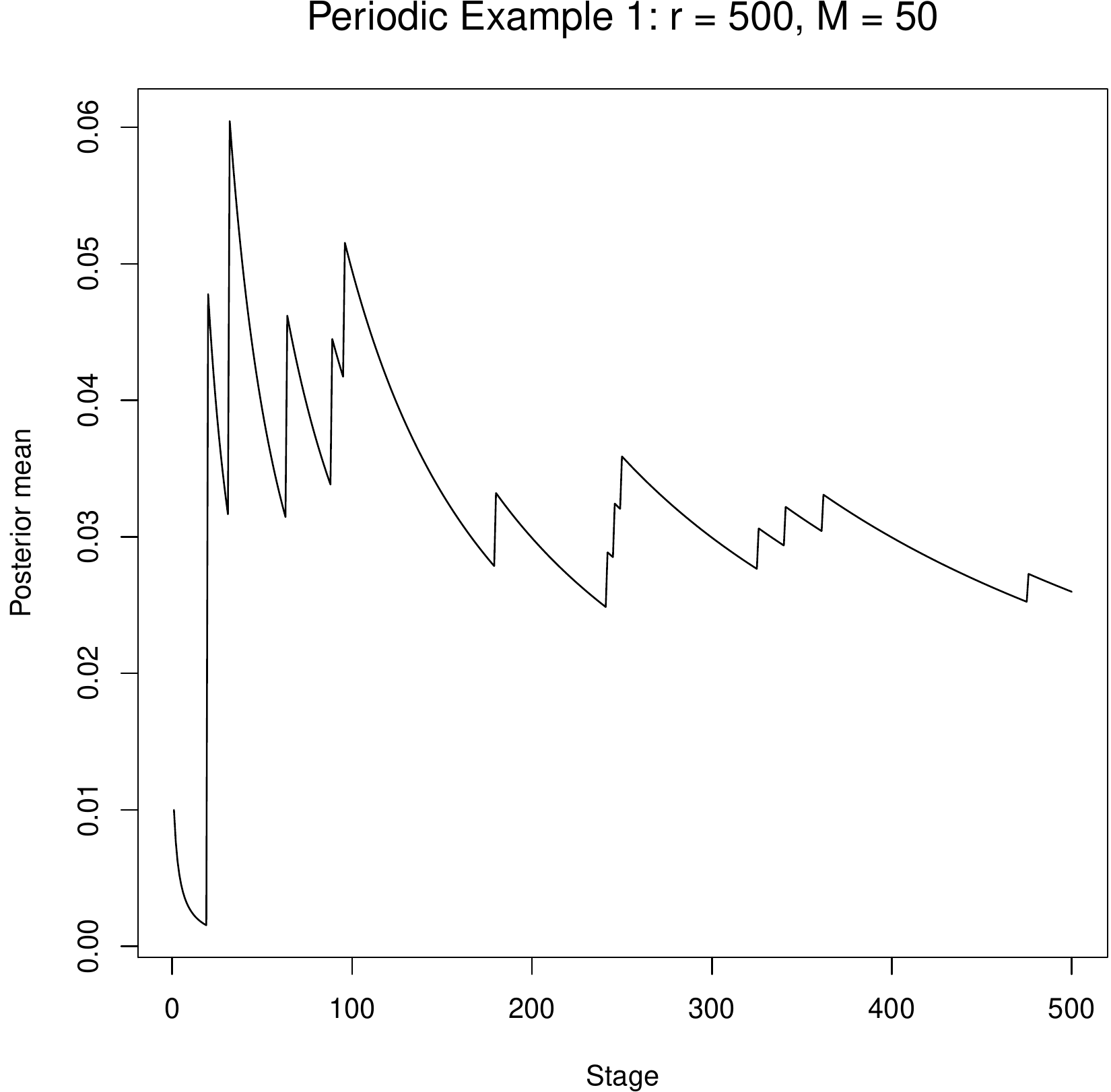}}
\hspace{2mm}
\subfigure [$r=500,M=100$.]{ \label{fig:osc_12}
\includegraphics[width=4.5cm,height=4.5cm]{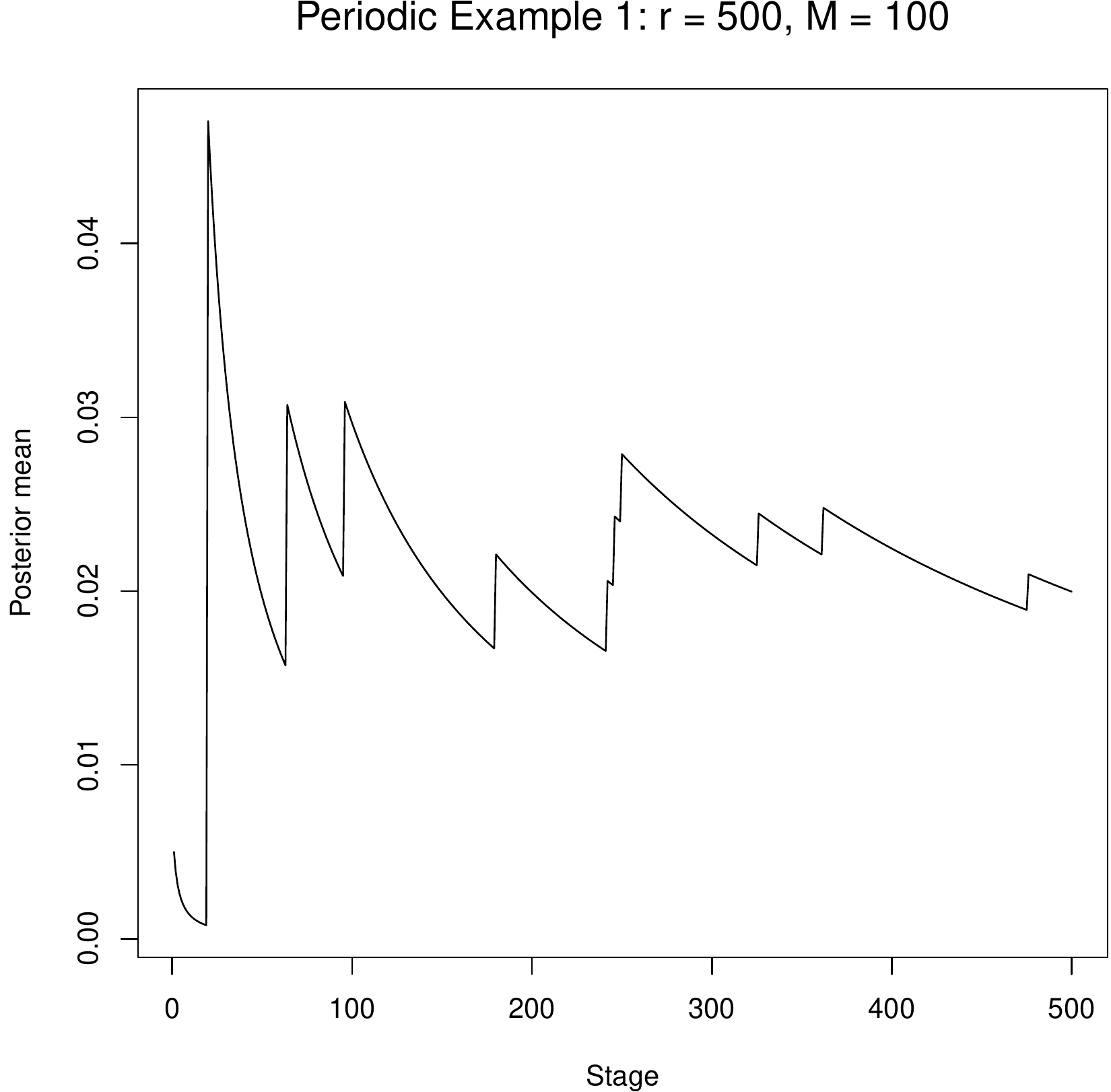}}\\
\vspace{2mm}
\subfigure [$r=1000,M=10$.]{ \label{fig:osc_13}
\includegraphics[width=4.5cm,height=4.5cm]{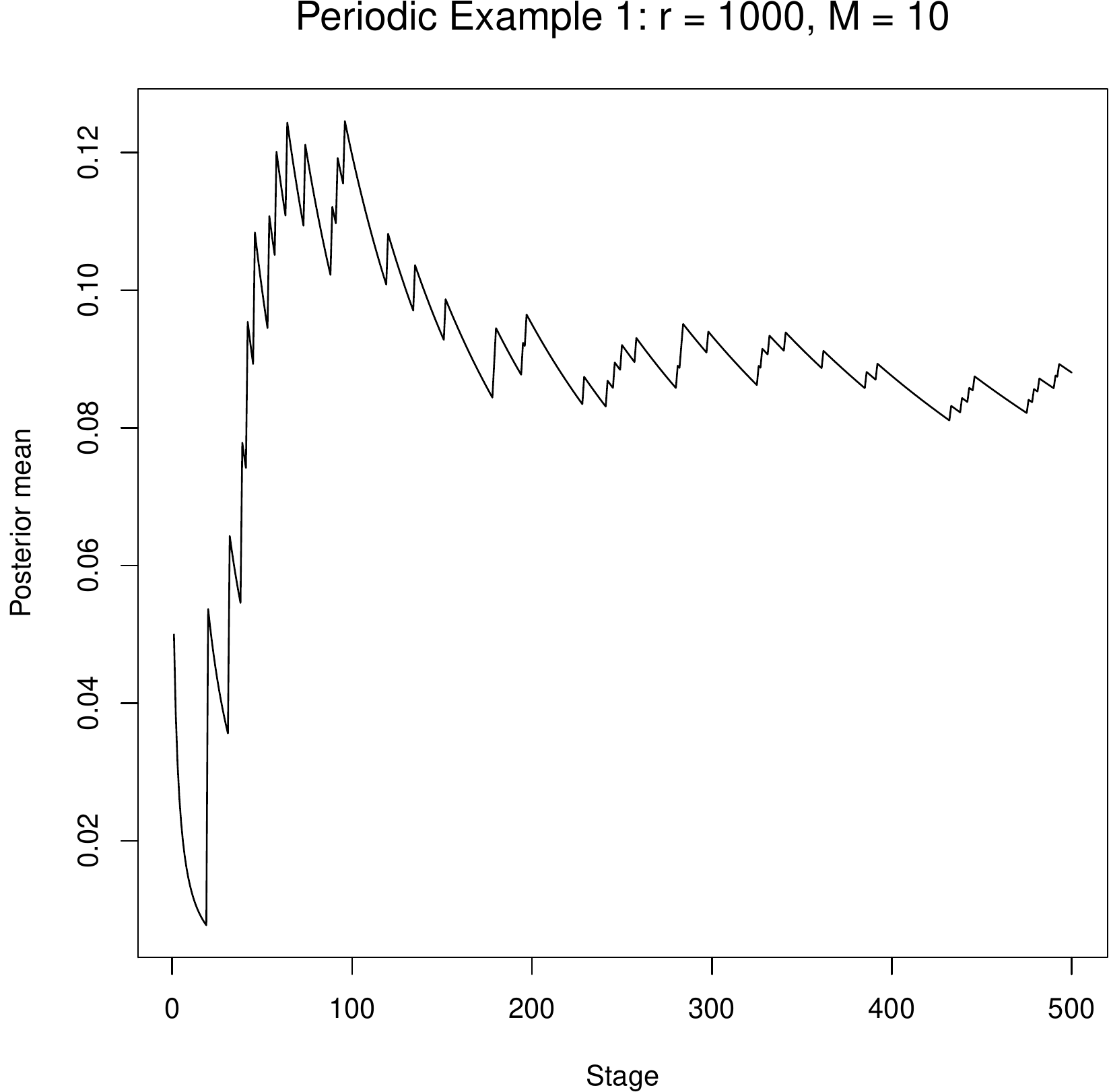}}
\hspace{2mm}
\subfigure [$r=1000,M=50$.]{ \label{fig:osc_14}
\includegraphics[width=4.5cm,height=4.5cm]{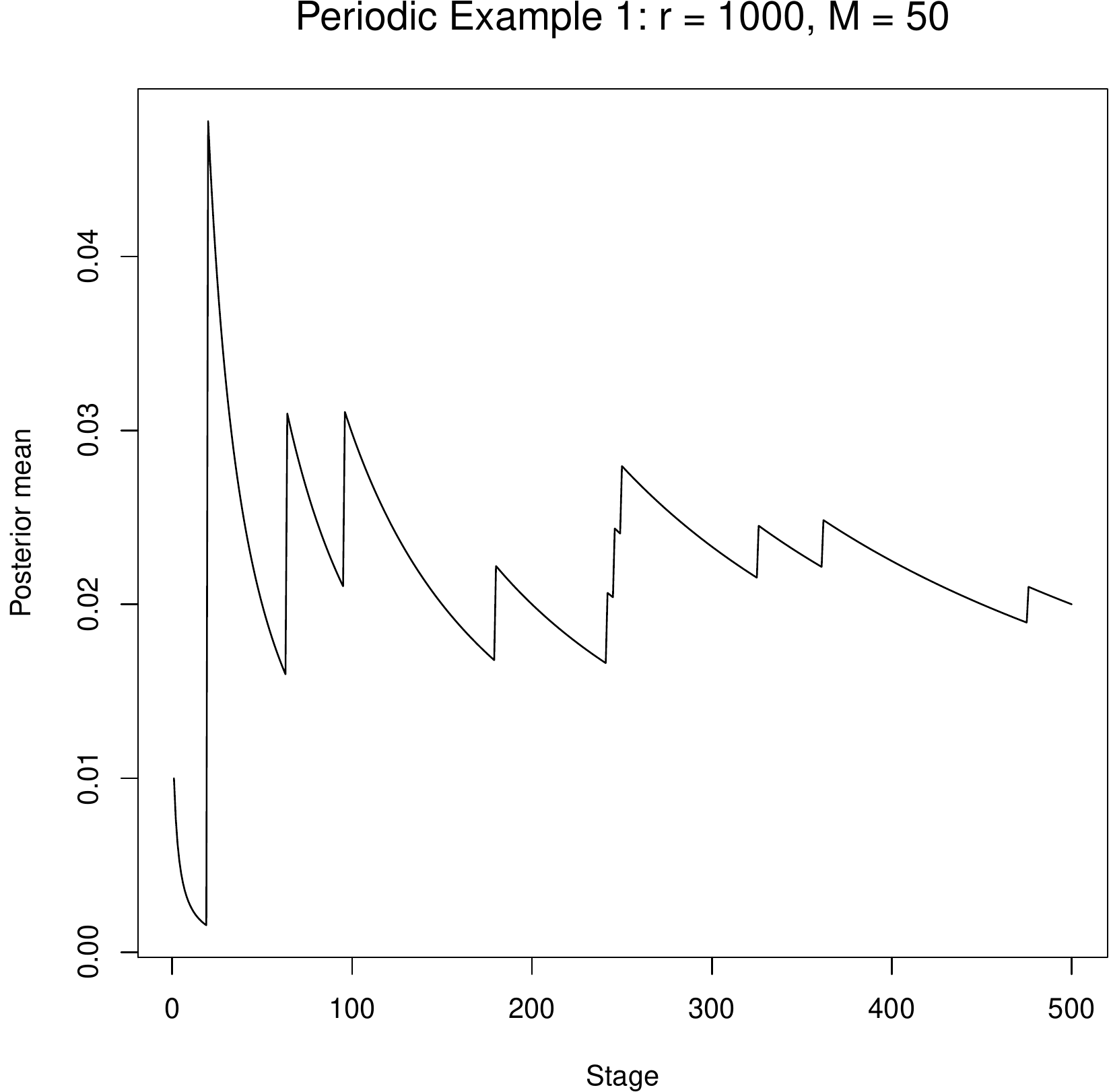}}
\hspace{2mm}
\subfigure [$r=1000,M=100$.]{ \label{fig:osc_15}
\includegraphics[width=4.5cm,height=4.5cm]{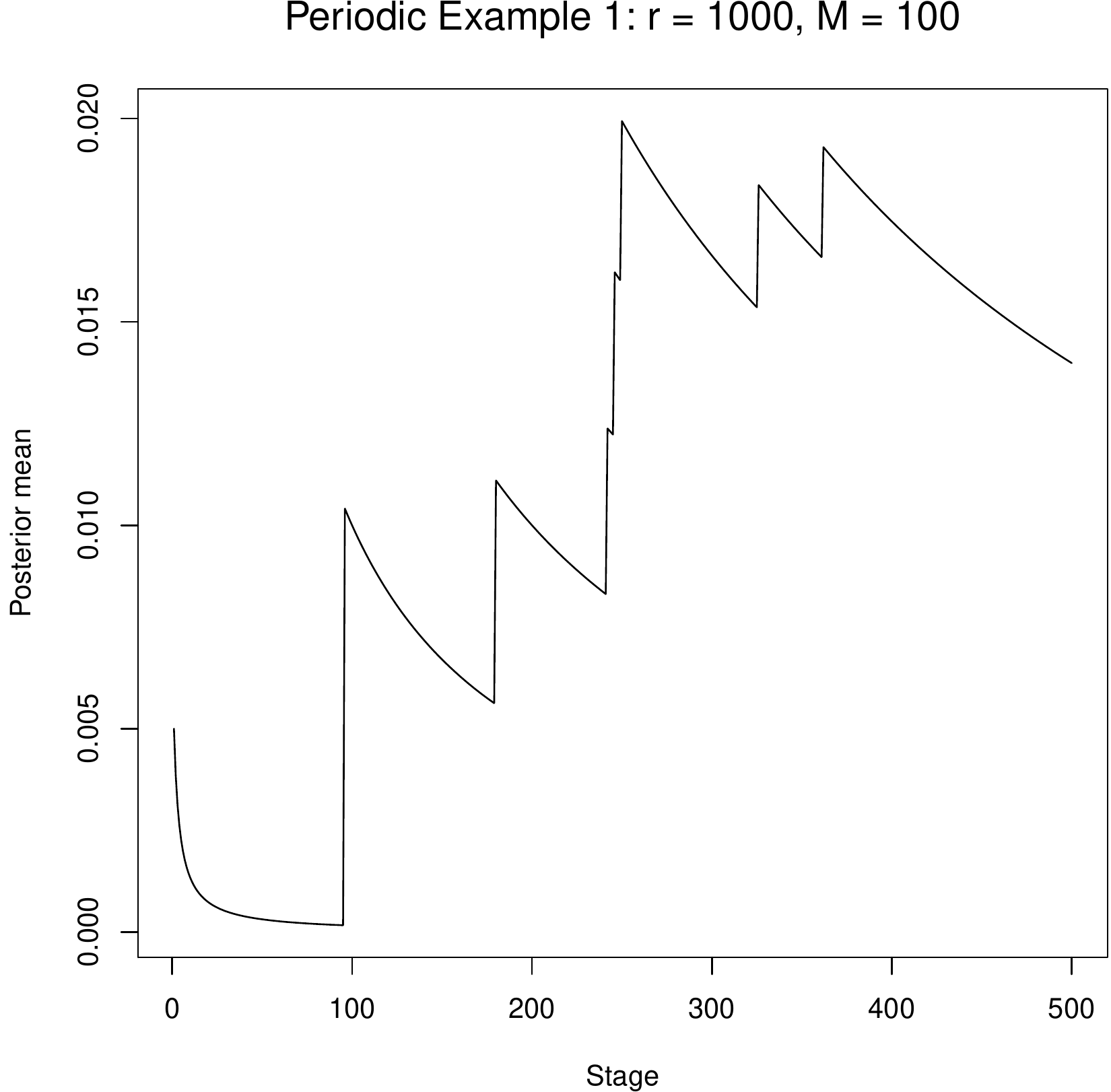}}
\caption{Illustration of our Bayesian method for determining single frequency. Here the true frequency is $0.02$. }
\label{fig:osc_example2}
\end{figure}

Following \ctn{Shumway06} we have generated only $500$ observations from (\ref{eq:single_freq}) for inference, due to reasons of comparability
with the results obtained by \ctn{Shumway06}. If large enough datasets are not available in reality, our Bayesian inference needs to be as accurate as possible
based on the available data, and our analyses indeed provide glimpses of such reliable Bayesian inference. But in the current ``big data"
era large datasets are making their appearances, and it is important to weigh our inference with respect to large datasets, which also provide opportunities
to properly validate our convergence theory, which is usually not viable for small datasets. 

We thus generate a dataset from (\ref{eq:single_freq}) with $T=5\times 10^5$, and apply our Bayesian procedure with $r=1000$ and $M=10,50,100$, in order
to detect the true frequency $0.02$. The results are displayed in Figure \ref{fig:osc_example2_long}. Observe that for $M=10$, the true frequency is overestimated,
as shown in panel (a) associated with convergence of $p_{10,j}$ as $j\rightarrow\infty$, and for $M=100$, underestimation occurs, 
as captured by panel (c) associated with convergence of $p_{100,j}$ as $j\rightarrow\infty$. Panel (b) shows convergence of $p_{50,j}$
as $j\rightarrow\infty$, where convergence occurs around $0.019$,
quite close to the truth. Panel (d) displays the result of convergence of $p_{100,j}+p_{99,j}$, as $j\rightarrow\infty$. This sum converges around $0.019$.
The reason for over and under estimation for $M=10$ and $100$ can be attributed to too coarse and too fine partitions of $[0,1]$ via the choice of $M$, while
for $M=50$, the partitioning seems more reasonable in comparison. Adding up $p_{100,j}$ and $p_{99,j}$ compensates for the too fine partitioning of $[0,1]$ in this case.

The effects of partitioning also points towards another issue -- even $p_{50,j}$ and $p_{100,j}+p_{99,j}$ fail to capture the true frequency as $j\rightarrow\infty$,
since the posterior variance becomes negligibly small as $j\rightarrow\infty$. In principle, it is possible to partition $[0,1]$ appropriately (perhaps, using 
good choices of $q_m$), such that convergence to the exact true frequency is achieved. In this example, setting $M=40$ is enough, as depicted in 
Figure \ref{fig:osc_correct_convergence}. Note that such subtle issues can not be detected or analyzed for sample size as small as $500$. Nevertheless,
our final Bayesian results do convey very reliable analysis even for such small dataset.
\begin{figure}
\centering
\subfigure [$r=1000,M=10$.]{ \label{fig:osc_1_long}
\includegraphics[width=6.0cm,height=6.0cm]{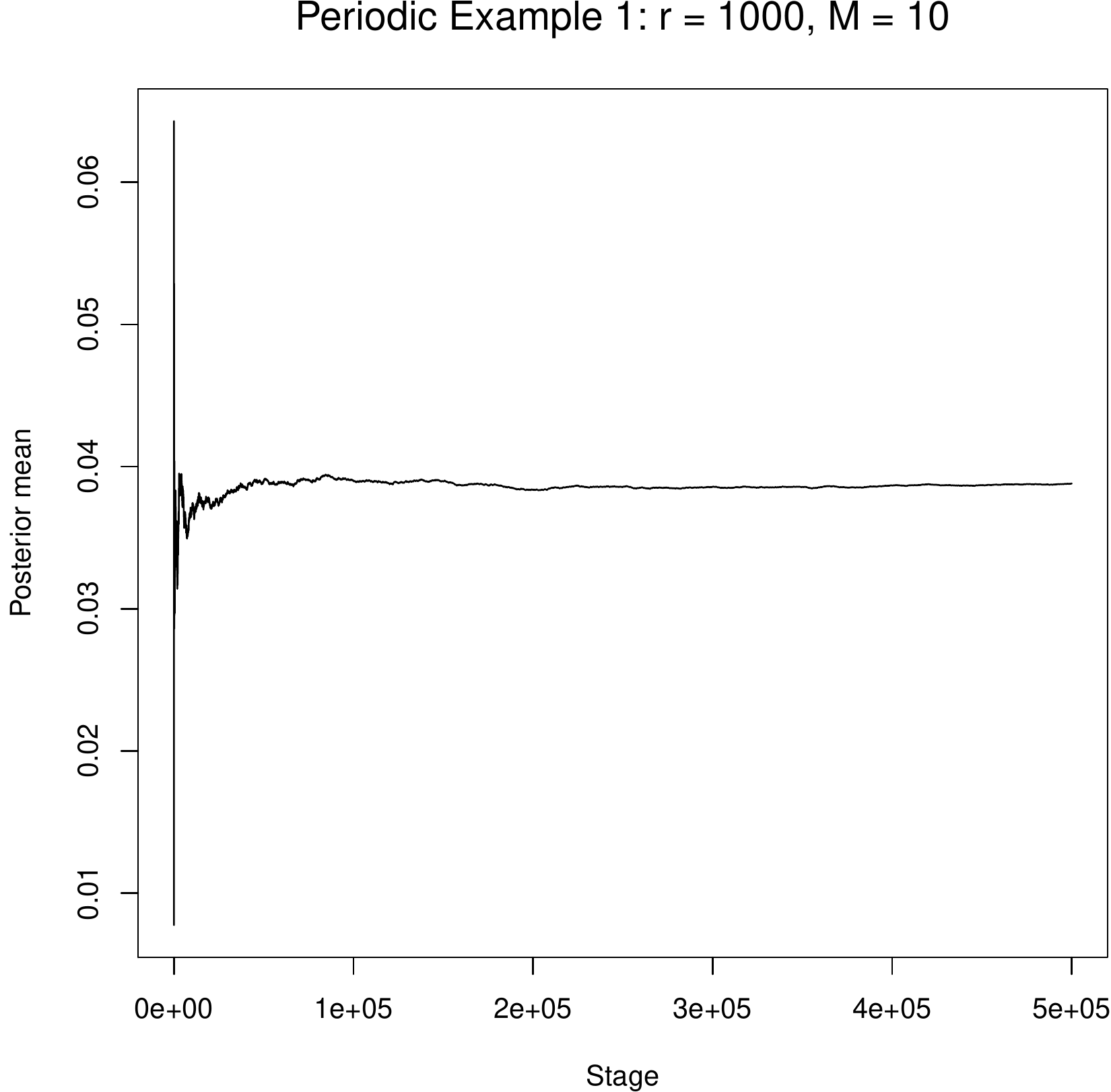}}
\hspace{2mm}
\subfigure [$r=1000,M=50$.]{ \label{fig:osc_2_long}
\includegraphics[width=6.0cm,height=6.0cm]{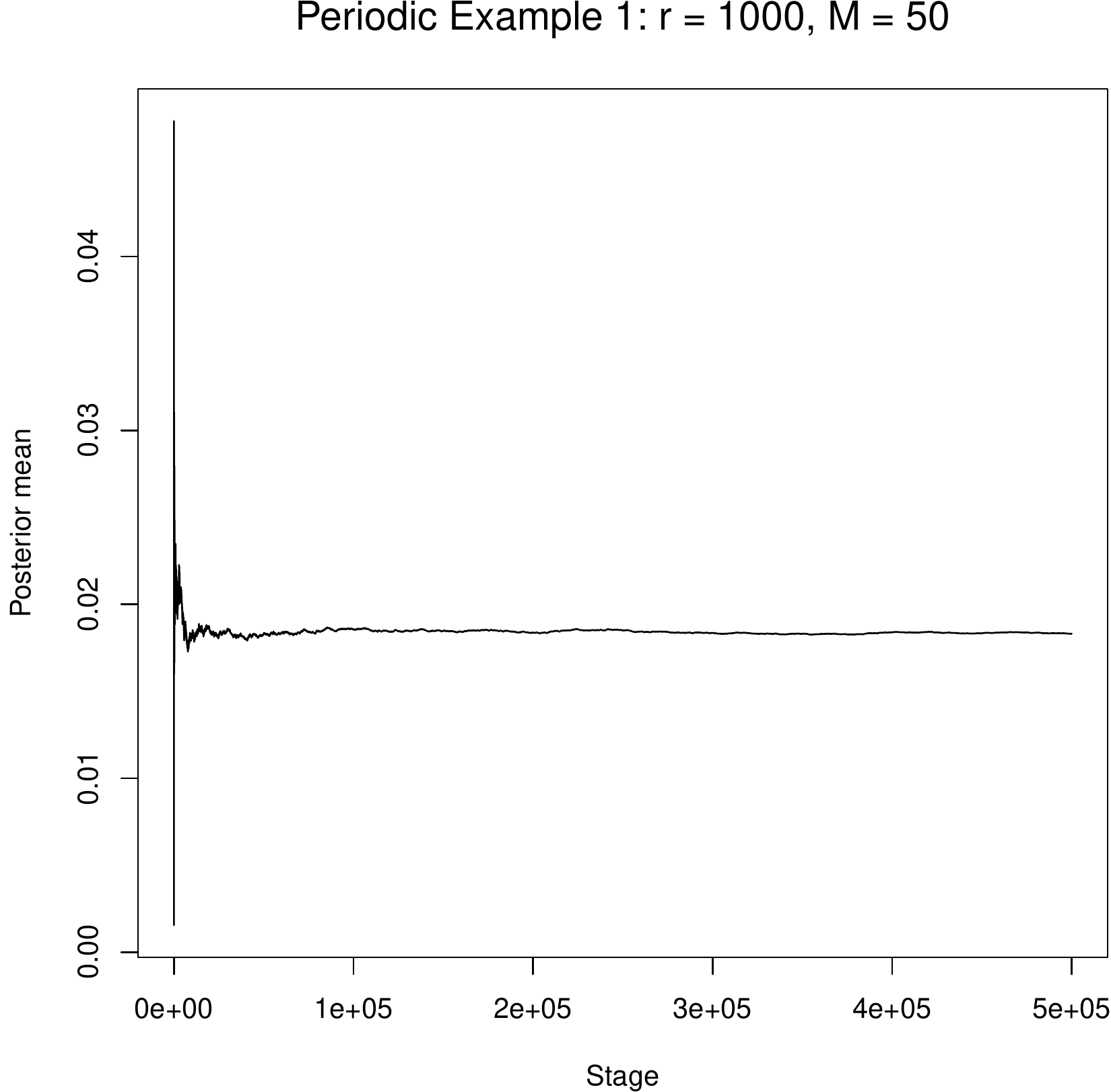}}\\
\vspace{2mm}
\subfigure [$r=1000,M=100$.]{ \label{fig:osc_3_long}
\includegraphics[width=6.0cm,height=6.0cm]{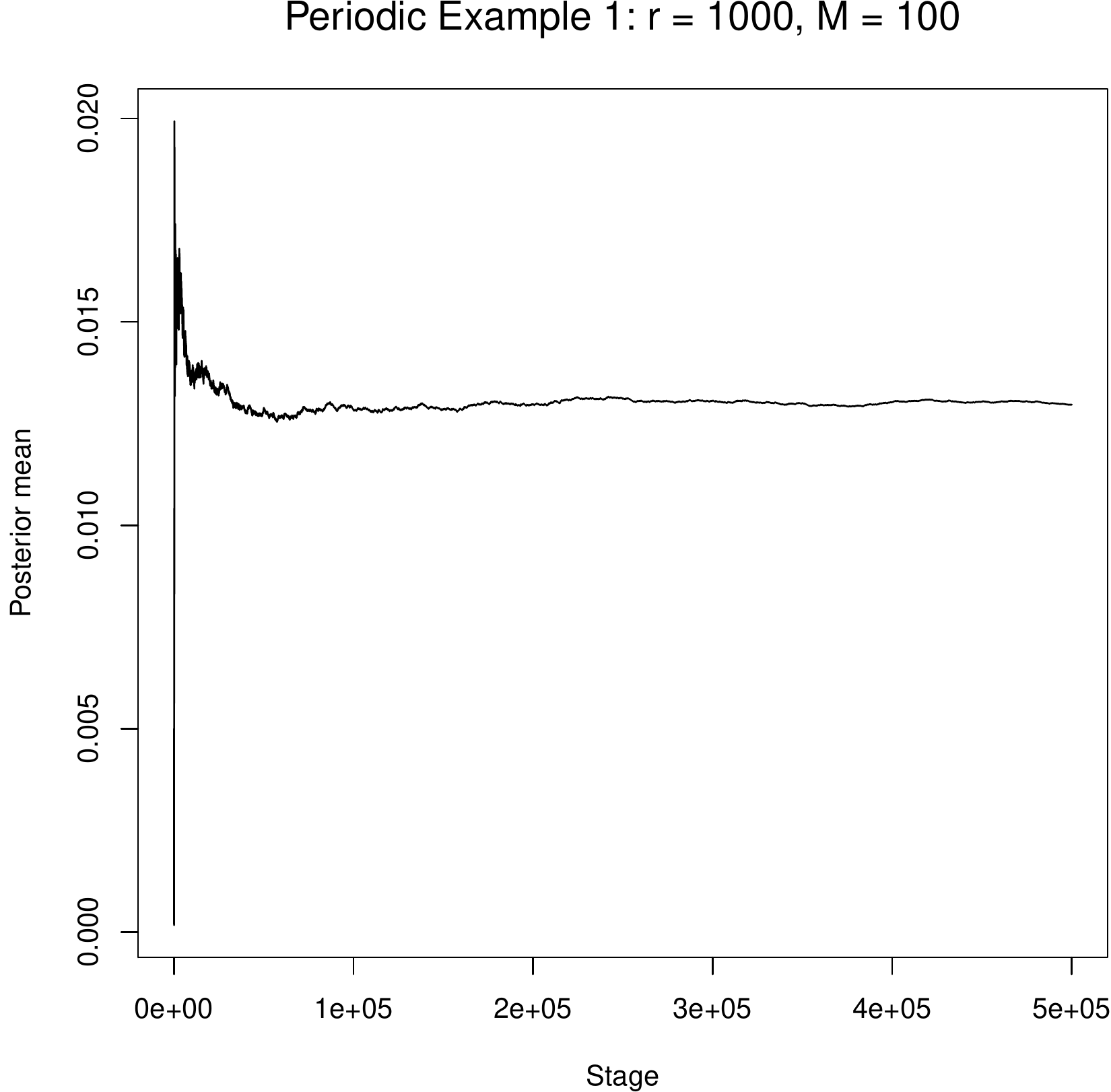}}
\hspace{2mm}
\subfigure [$r=1000,M=100$, with addition of $100$-th and $99$-th co-ordinates.]{ \label{fig:osc_4_long}
\includegraphics[width=6.0cm,height=6.0cm]{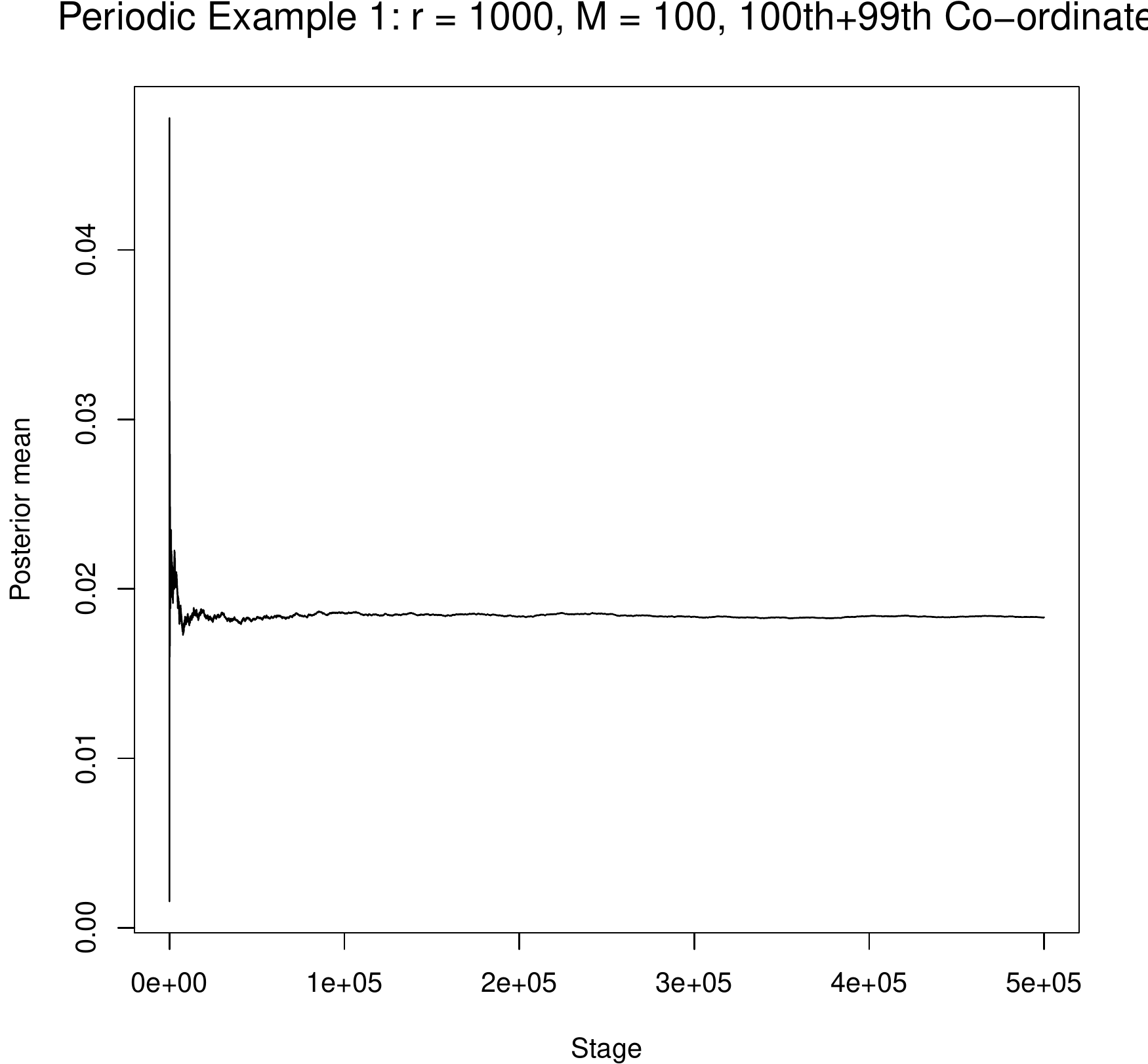}}
\caption{Illustration of our Bayesian method for determining single frequency for long enough time series. Here the true frequency is $0.02$. }
\label{fig:osc_example2_long}
\end{figure}

\begin{figure}
\centering
\includegraphics[width=9.0cm,height=6.0cm]{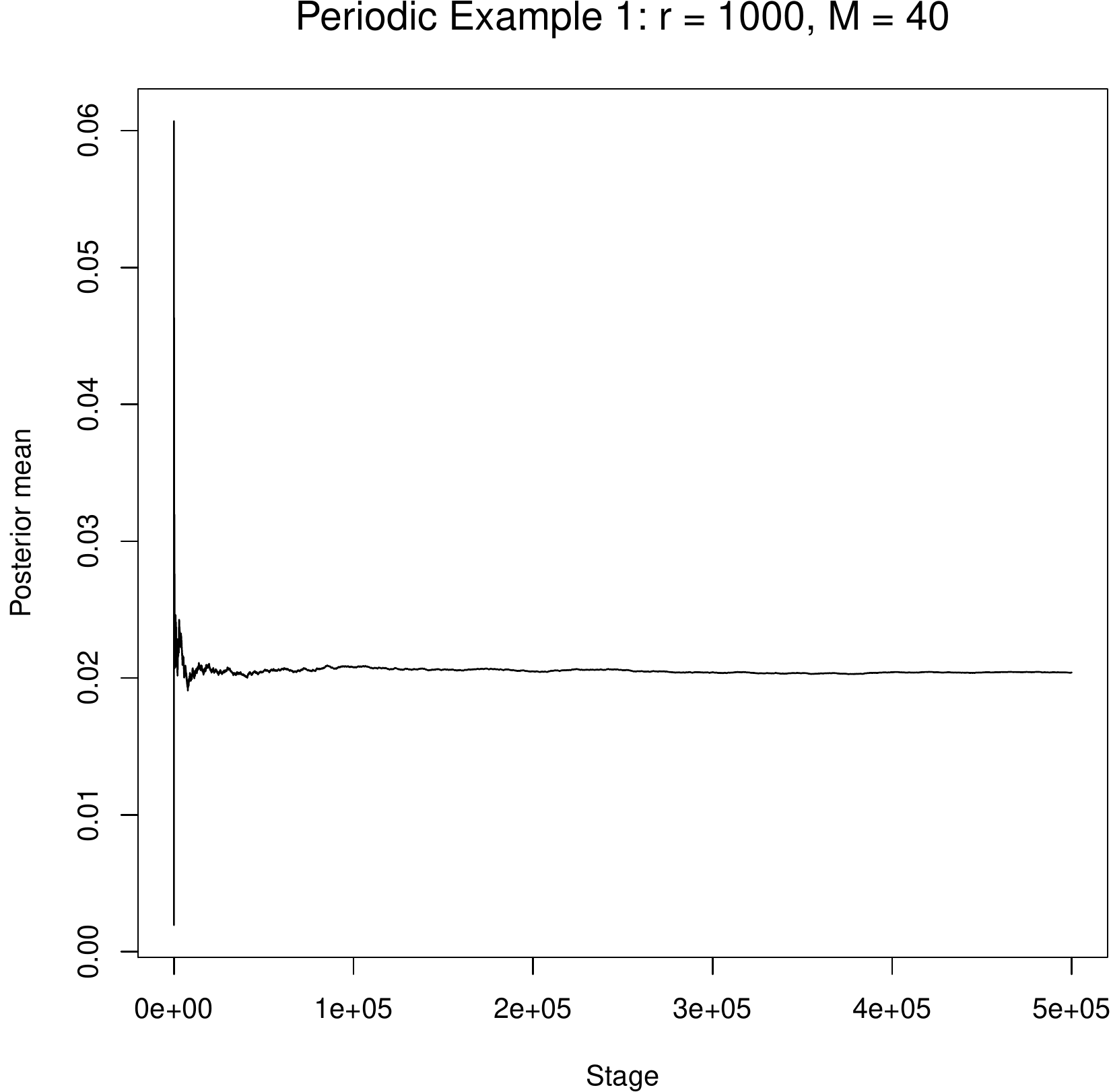}
\caption{Convergence of our Bayesian method to the true frequency $0.02$ for long enough time series with $r=1000$ and $M=40$.}
\label{fig:osc_correct_convergence}
\end{figure}

\subsection{Simulation study with multiple frequencies}
\label{subsec:multiple_frequency}
As in Example 4.1 of \ctn{Shumway06}, for $t=1,\ldots,100$, first we generate the following three series: 
\begin{align}
x_{t_1} & = 2 \cos(2\pi t 6/100) + 3 \sin(2\pi t 6/100);\notag\\
x_{t_2} &= 4 \cos(2\pi t 10/100) + 5 \sin(2\pi t 10/100);\notag\\
x_{t_3} &= 6 \cos(2\pi t 40/100) + 7 \sin(2\pi t 40/100),\notag
\end{align}
and set 
\begin{equation}
x_t=x_{t_1}+x_{t_2}+x_{t_3}. 
\label{eq:mult_freq}	
\end{equation}
The series $x_t$, which consists of the three frequencies $0.4$, $0.1$ and $0.06$, is shown in Figure \ref{fig:osc_series2}.
\begin{figure}
\centering
\includegraphics[width=10cm,height=6cm]{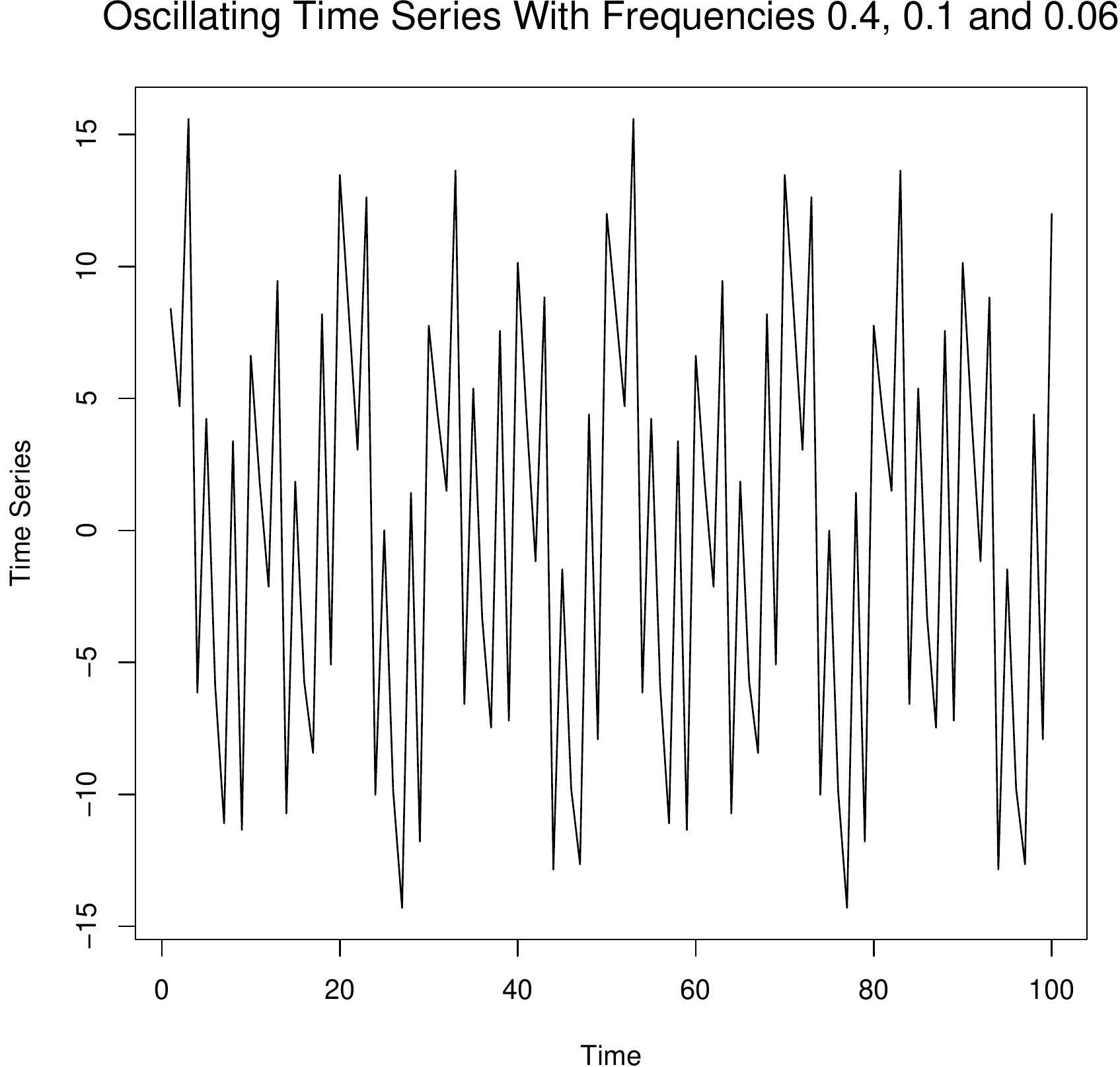}
\caption{Simulated oscillating time series with true frequencies $0.4$, $0.1$ and $0.06$. }
\label{fig:osc_series2}
\end{figure}

Before applying our Bayesian method based on Dirichlet process to this example, we again need to choose $r$ and $M$ properly. Regarding the choice of $r$,
Figure \ref{fig:osc_example2_r} depicts the process $\bZ^r$ for $r=1,5,10,50,100$. Here although it seems at first glance that increasing $r$ leads to increasing isolation
of the oscillations, actually, it is evident from closer look that increasing the power here has the effect of reducing the peaks of many relevant 
oscillations quite close to the highest peaks that are present in panel (a) of the figure, corresponding to $r=1$. 
Thus, in this example, large values of $r$ are inappropriate, unlike in the first example on single frequency. Here $r=1$ seems more appropriate compared to the
other values of $r$.
\begin{figure}
\centering
\subfigure [Transformed series $\bZ^{1}$.]{ \label{fig:r6}
\includegraphics[width=6.5cm,height=5cm]{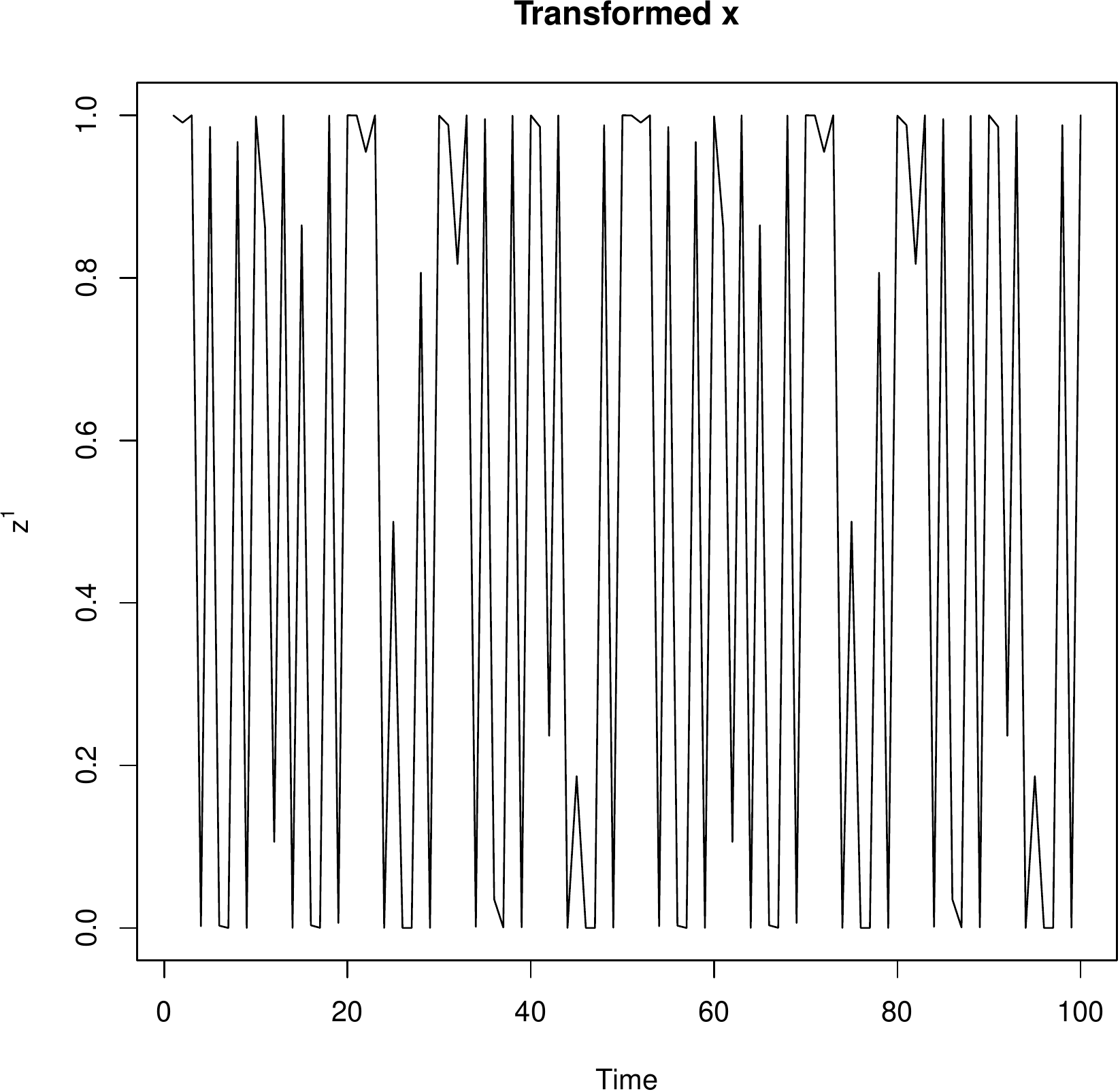}}
\hspace{2mm}
\subfigure [Transformed series $\bZ^{5}$.]{ \label{fig:r7}
\includegraphics[width=6.5cm,height=5cm]{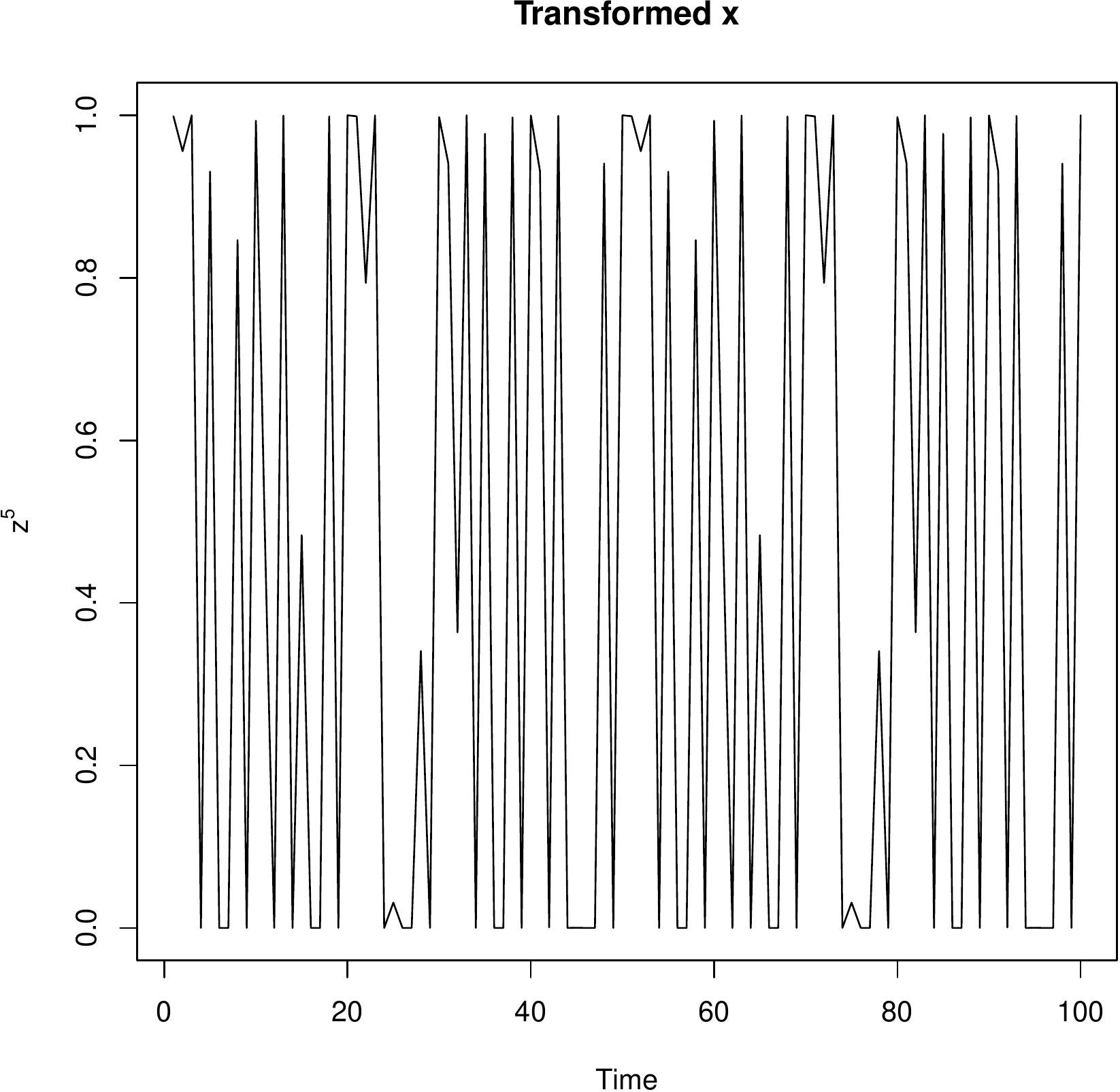}}\\
\vspace{2mm}
\subfigure [Transformed series $\bZ^{10}$.]{ \label{fig:r8}
\includegraphics[width=6.5cm,height=5cm]{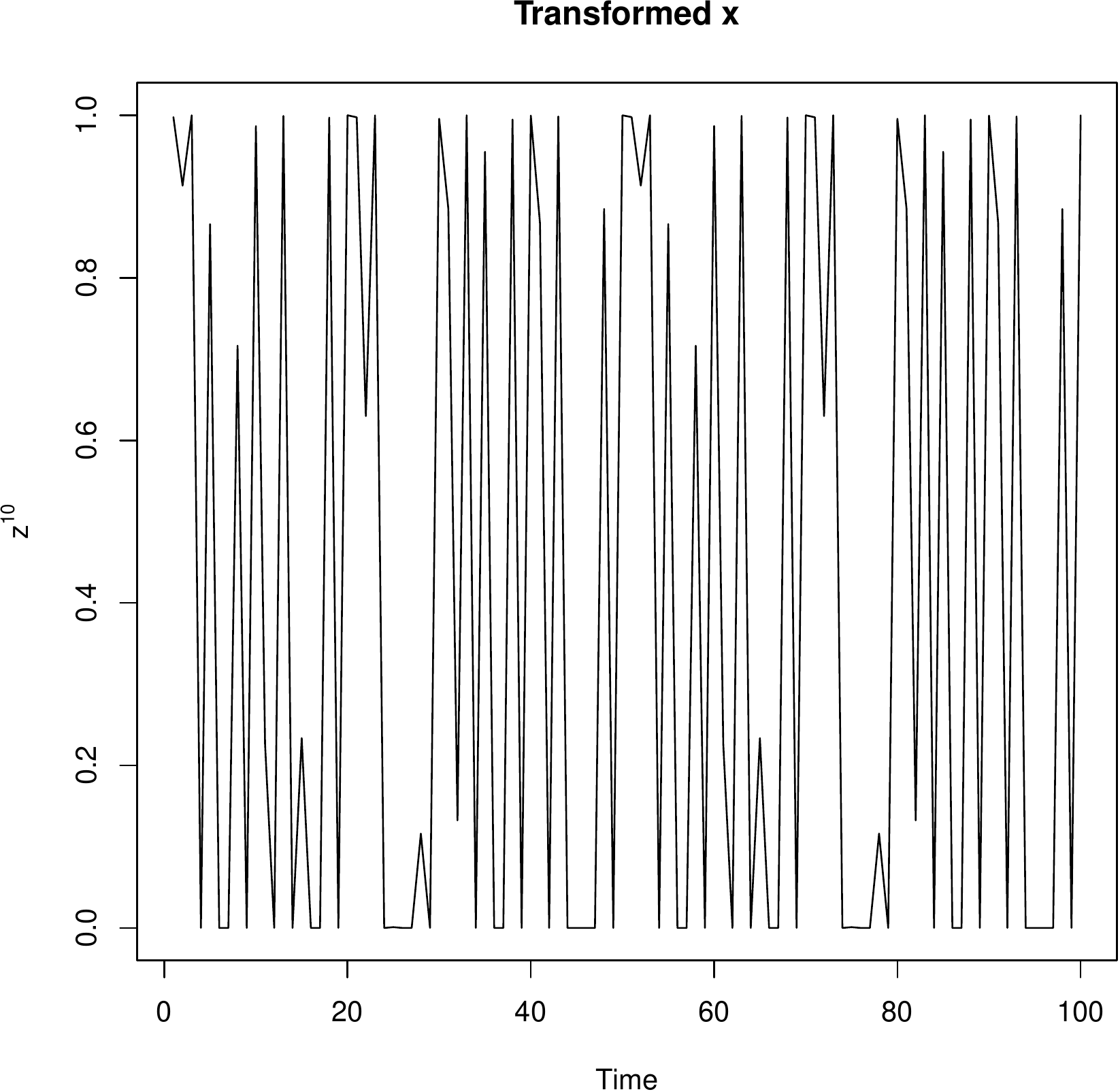}}
\hspace{2mm}
\subfigure [Transformed series $\bZ^{50}$.]{ \label{fig:r9}
\includegraphics[width=6.5cm,height=5cm]{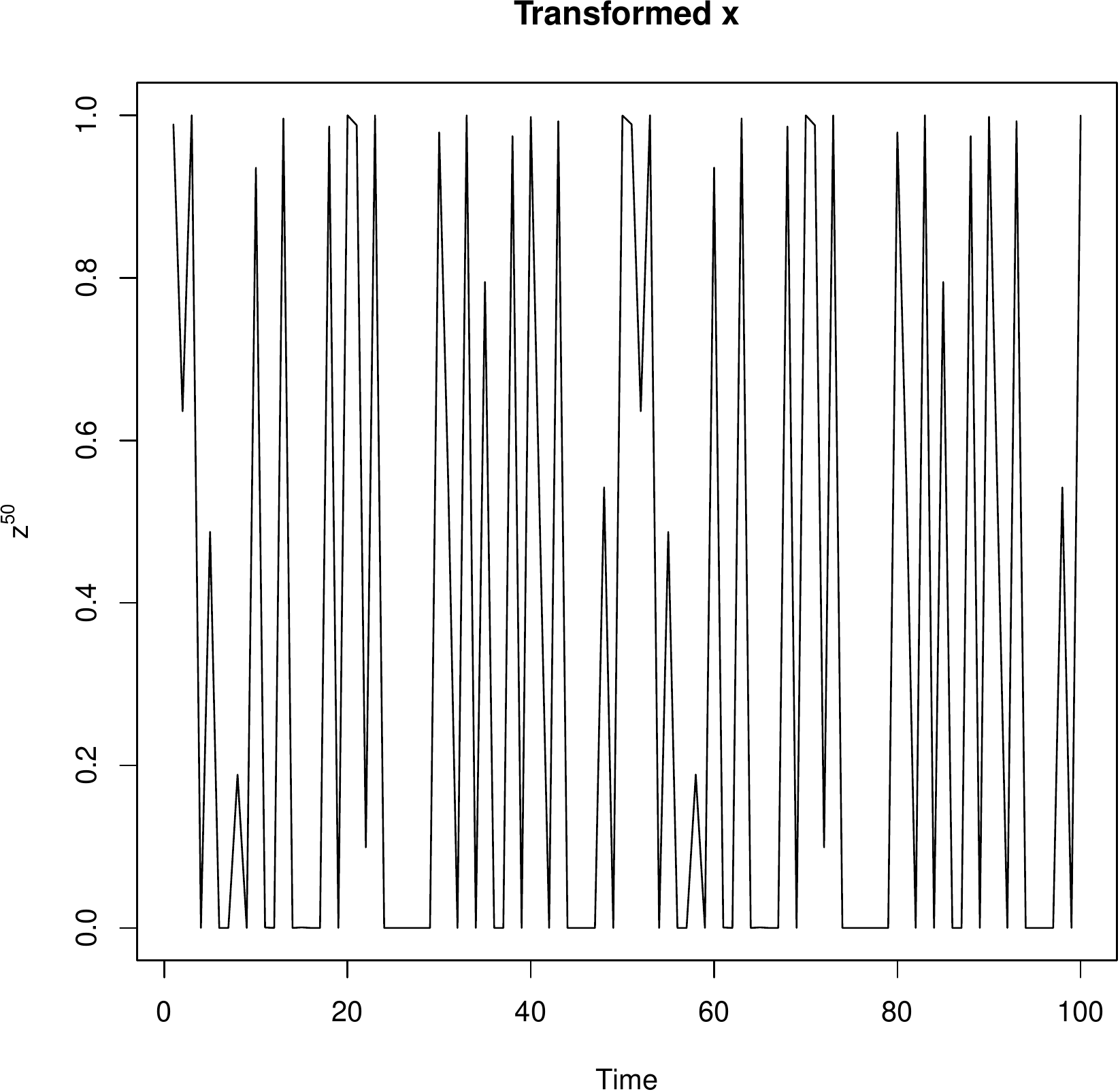}}\\
\vspace{2mm}
\subfigure [Transformed series $\bZ^{100}$.]{ \label{fig:r10}
\includegraphics[width=6.5cm,height=5cm]{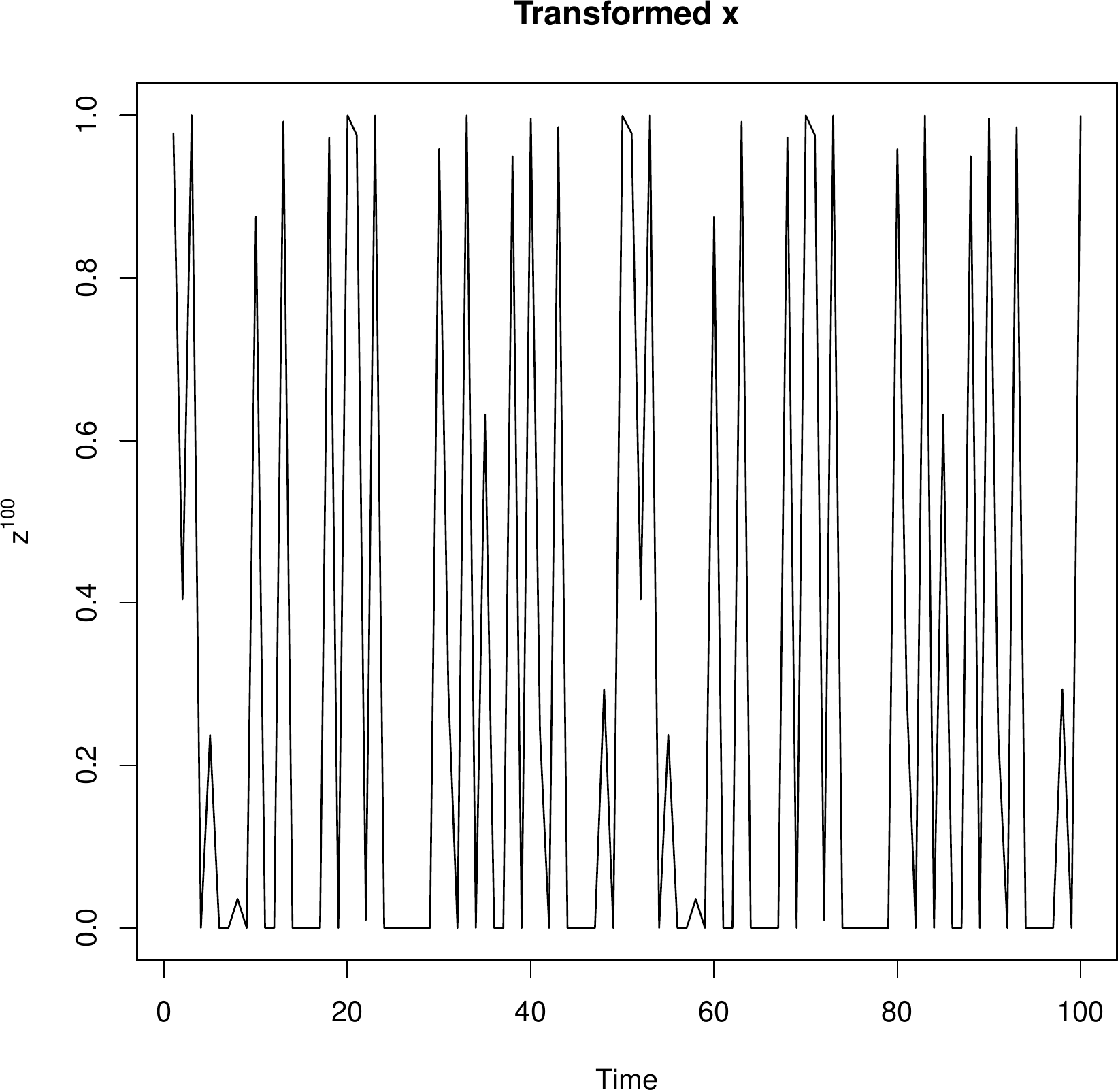}}
\caption{Illustration of effects of $r$ in $\bZ^r$ in determining multiple frequencies in (\ref{eq:mult_freq}). Here the true frequencies are $0.4$, $0.1$ and $0.06$. }
\label{fig:osc_example2_r}
\end{figure}

Regarding adequacy of the choice of $r$ and $M$, a detailed analysis of our Bayesian results for this multiple frequency example 
is provided by Figures \ref{fig:mult_osc_example1}, \ref{fig:mult_osc_example2}, \ref{fig:mult_osc_example3},
\ref{fig:mult_osc_example4} and \ref{fig:mult_osc_example5}.
Most of these diagrams, for given $r$ and $M$, are obtained by summing up the $p_{m,j}$ for nearby values of $m$. 
These yielded the three frequencies associated with our Bayesian technique. 
The values of $m$ that are summed up, are provided on the top
of each panel. Indeed, for relatively larger values of $M$, the frequencies are divided up into several nearby intervals $(\tilde p_{m-1},\tilde p_m]$.

Recall that we do not consider the first interval $(\tilde p_0,\tilde p_1]$ at all as it is a small interval around zero
for relatively large $M$ and hence not associated with any true frequency significantly different from zero. 
The proportions of the intervals that converged to zero, are not considered either.

Figures \ref{fig:mult_osc_example1}, \ref{fig:mult_osc_example2} and \ref{fig:mult_osc_example3} depict the details of our results for $r=1,5,10$
and $M=10,50,100$. Observe that $r=1$ gives the best performance, while the performance deteriorates for $r=5$ is also close. 
But observe that for $r=5,M=10$, the frequency $0.06$ seems to been somewhat underestimated. 
However importantly, for $r=10$, while the frequencies
$0.4$ and $0.1$ are correctly converged to for these values of $r$, the frequency $0.06$ seems to be significantly underestimated, for $M=10,50,100$.

As seen in Figures \ref{fig:mult_osc_example4} and \ref{fig:mult_osc_example5}, for $r=50$ and $100$, 
although the frequency $0.06$ is underestimated in some cases, the most conspicuous is the case of underestimation of the highest frequency $0.4$.
This is due to the fact that for relatively large values of $r$, about half of the peaks of the original process close to the highest peaks, die down.
Since half of these peaks close to the highest peaks contribute half of the total frequency $0.4$ (obvious from direct counting of the highest
and second highest peaks in Figure \ref{fig:osc_series2}, this results in significant underestimation of the highest frequency. 

Hence, consistent from the insight gained from Figure \ref{fig:osc_example2_r}, $r=1$ yields the best performance
The choice of $M$ seems to be less important compared to that of $r$, as in the previous example with single frequency.
\begin{figure}
\centering
\subfigure [$r=1,M=10$. True frequency $=0.4$.]{ \label{fig:mult_osc_1}
\includegraphics[width=4.5cm,height=4.5cm]{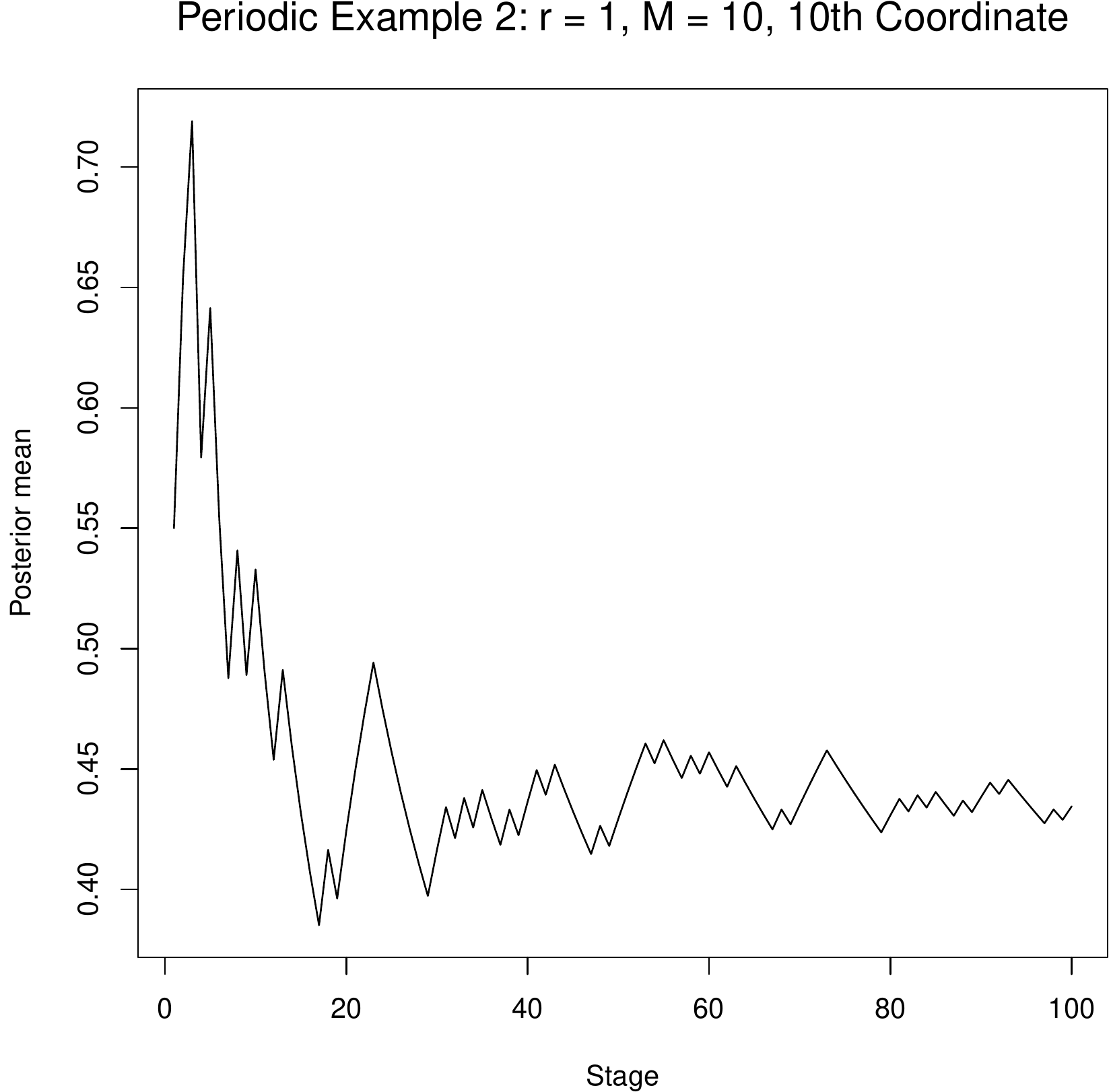}}
\hspace{2mm}
\subfigure [$r=1,M=10$. True frequency $=0.1$.]{ \label{fig:mult_osc_2}
\includegraphics[width=4.5cm,height=4.5cm]{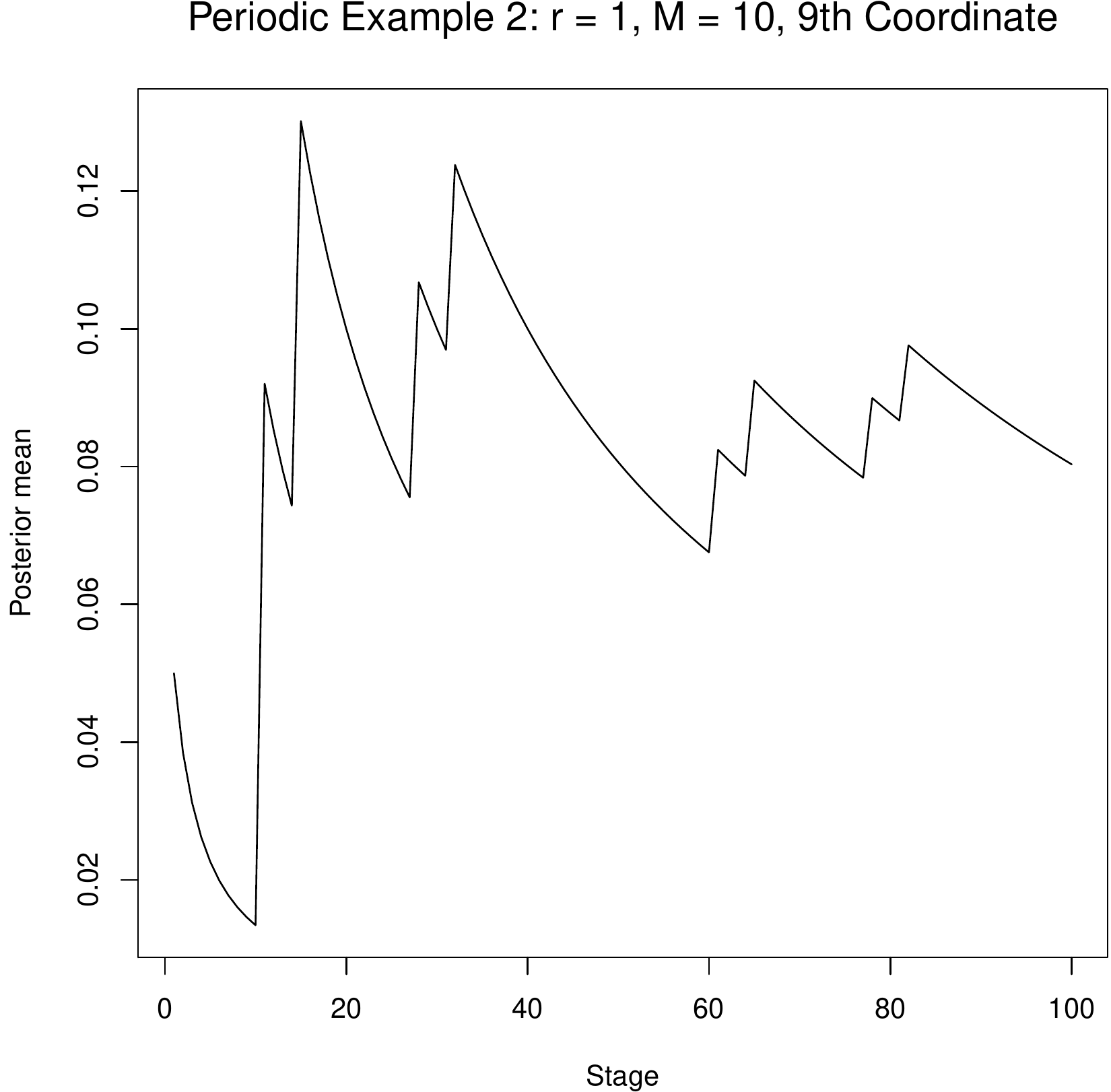}}
\hspace{2mm}
\subfigure [$r=1,M=10$. True frequency $=0.06$.]{ \label{fig:mult_osc_3}
\includegraphics[width=4.5cm,height=4.5cm]{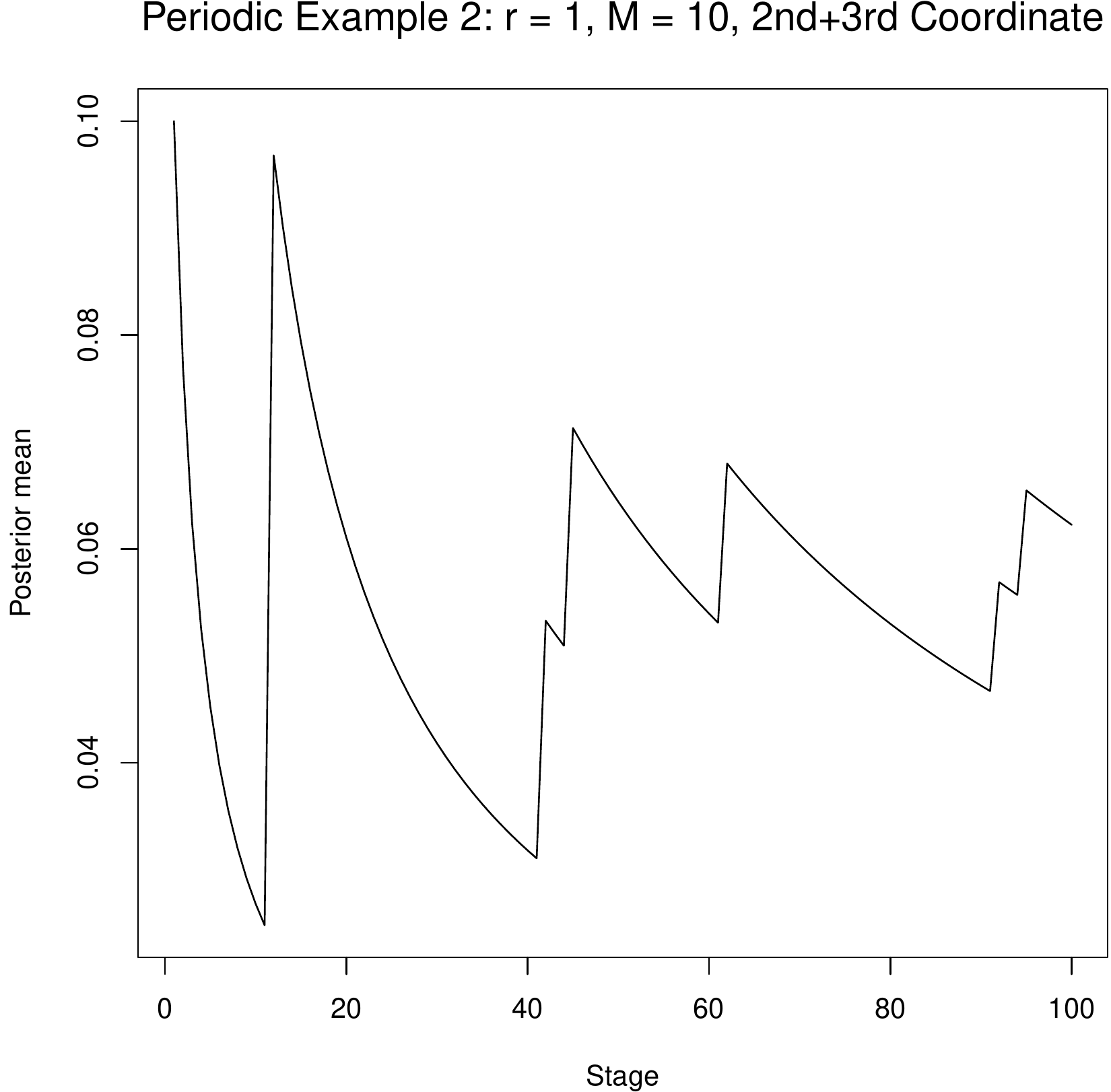}}\\
\vspace{2mm}
\subfigure [$r=1,M=50$. True frequency $=0.4$.]{ \label{fig:mult_osc_4}
\includegraphics[width=4.5cm,height=4.5cm]{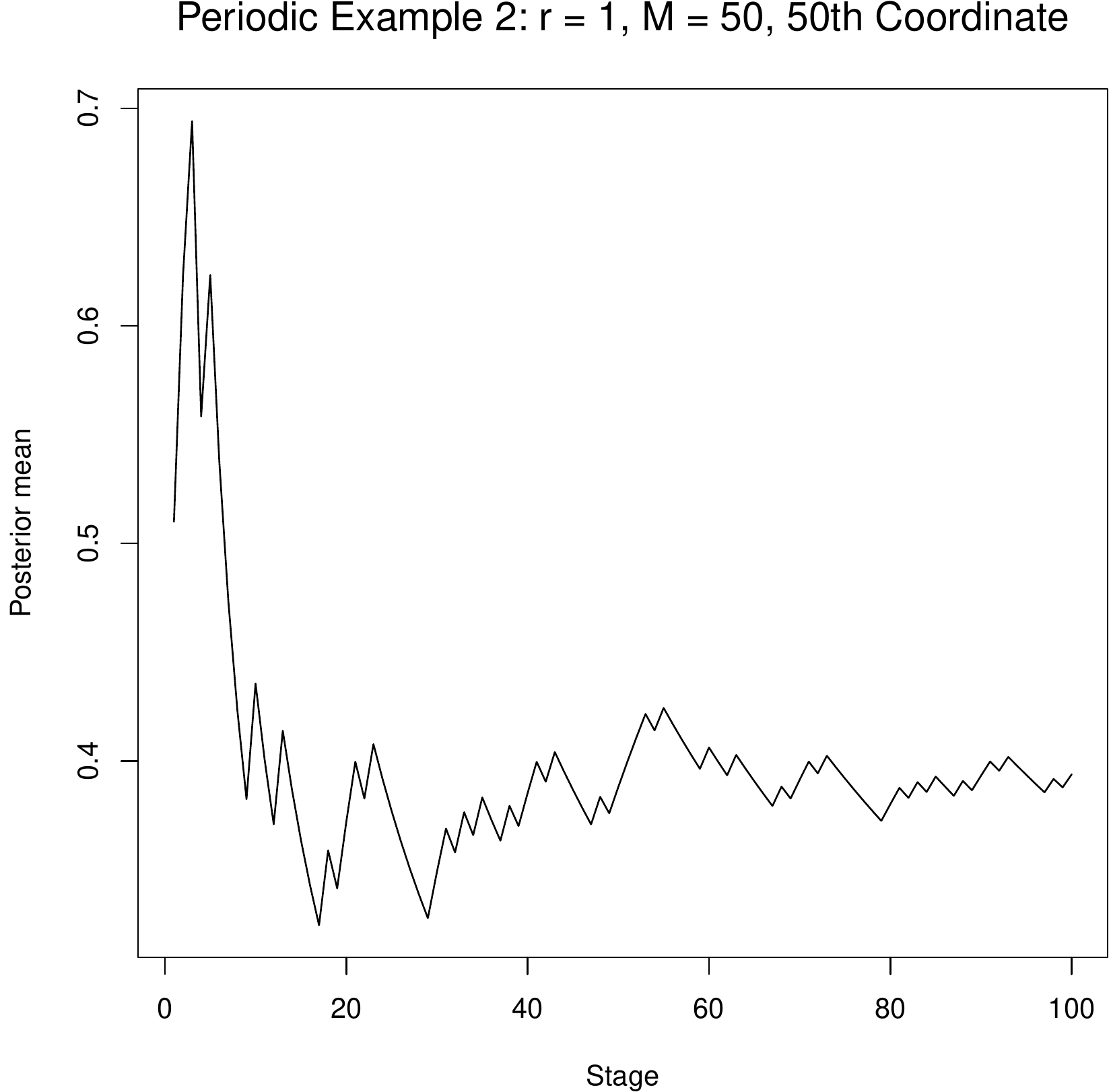}}
\hspace{2mm}
\subfigure [$r=1,M=50$. True frequency $=0.1$.]{ \label{fig:mult_osc_5}
\includegraphics[width=4.5cm,height=4.5cm]{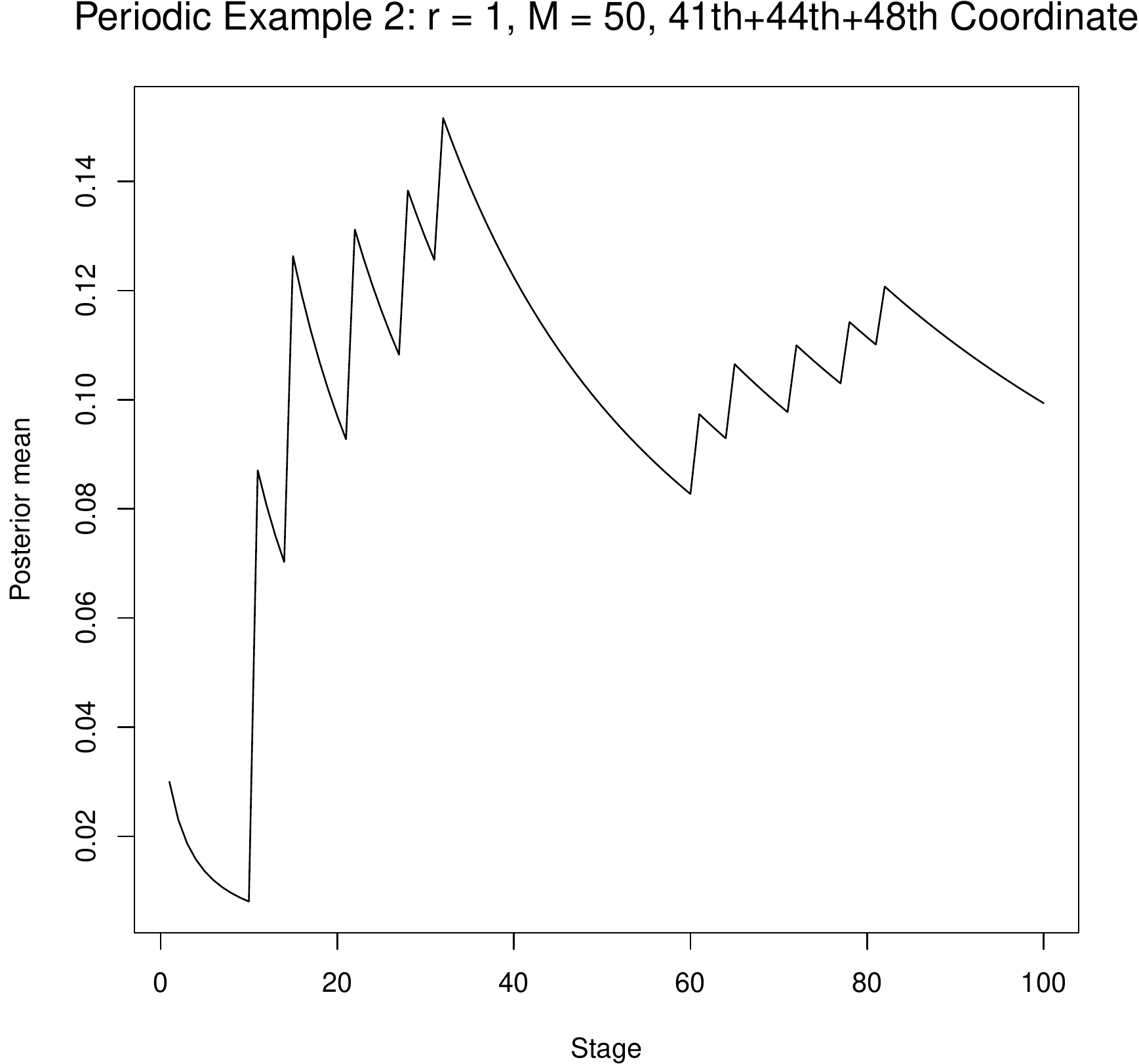}}
\hspace{2mm}
\subfigure [$r=1,M=50$. True frequency $=0.06$.]{ \label{fig:mult_osc_6}
\includegraphics[width=4.5cm,height=4.5cm]{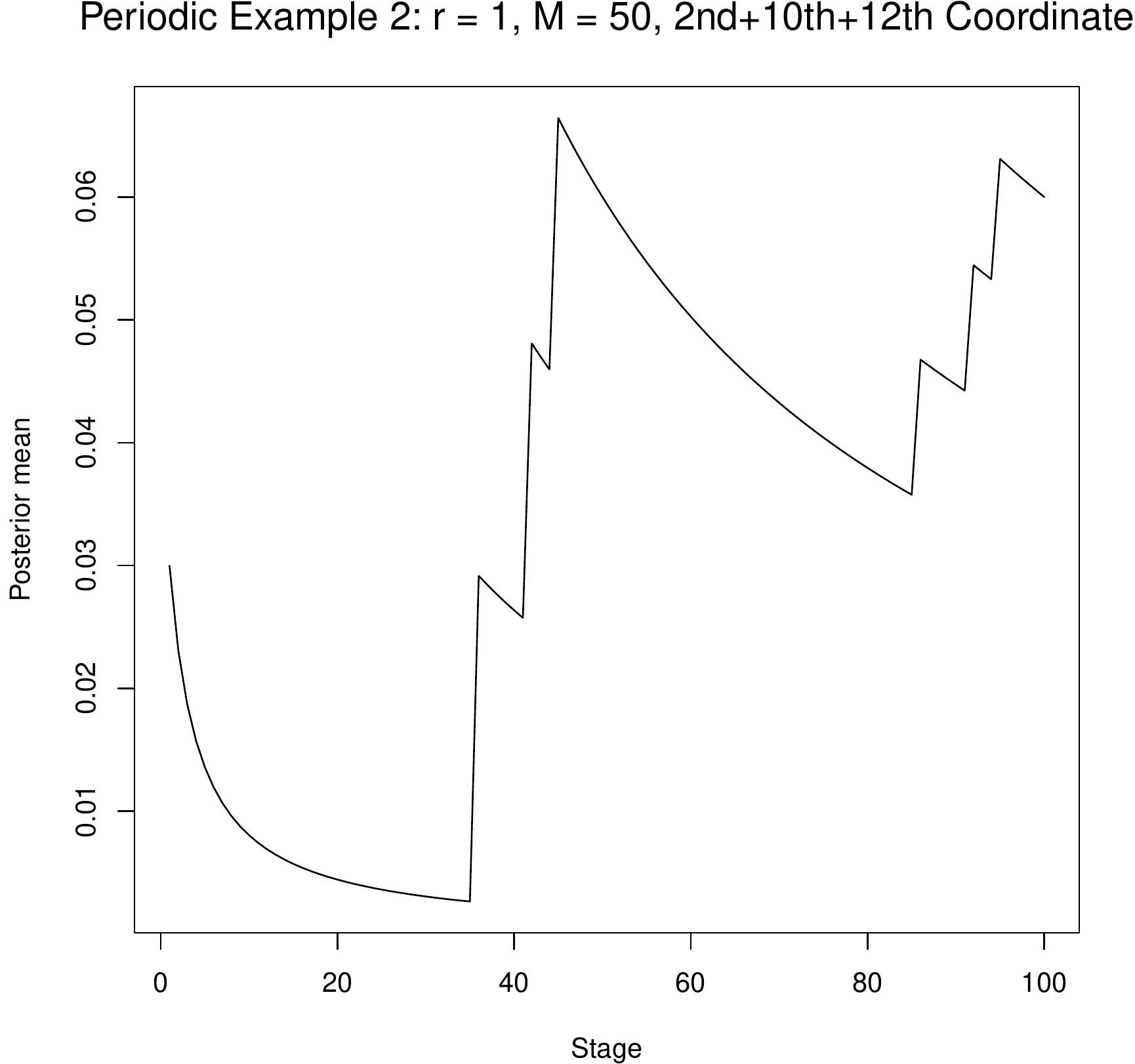}}\\
\vspace{2mm}
\subfigure [$r=1,M=100$. True frequency $=0.4$.]{ \label{fig:mult_osc_7}
\includegraphics[width=4.5cm,height=4.5cm]{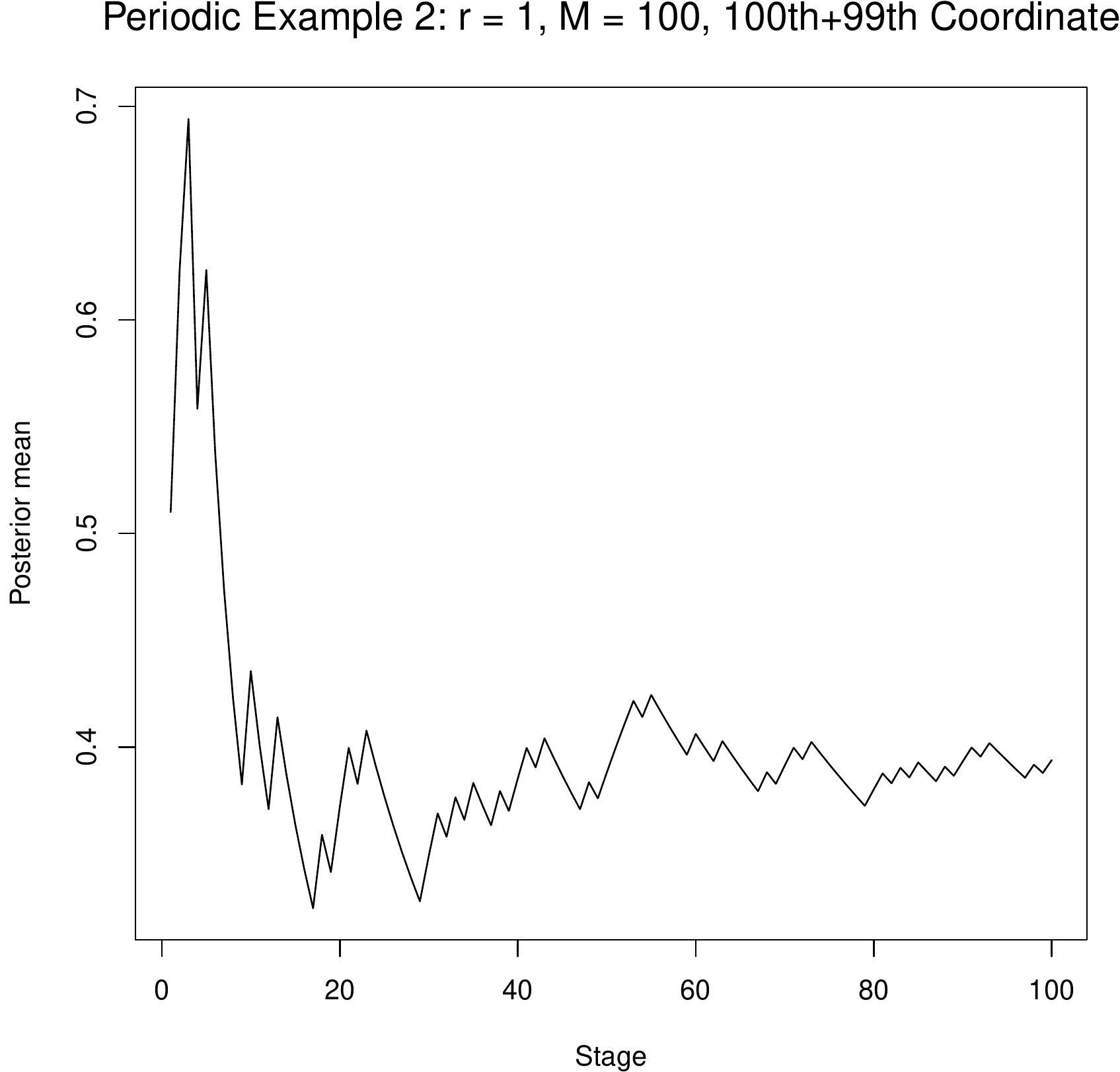}}
\hspace{2mm}
\subfigure [$r=1,M=100$. True frequency $=0.1$.]{ \label{fig:mult_osc_8}
\includegraphics[width=4.5cm,height=4.5cm]{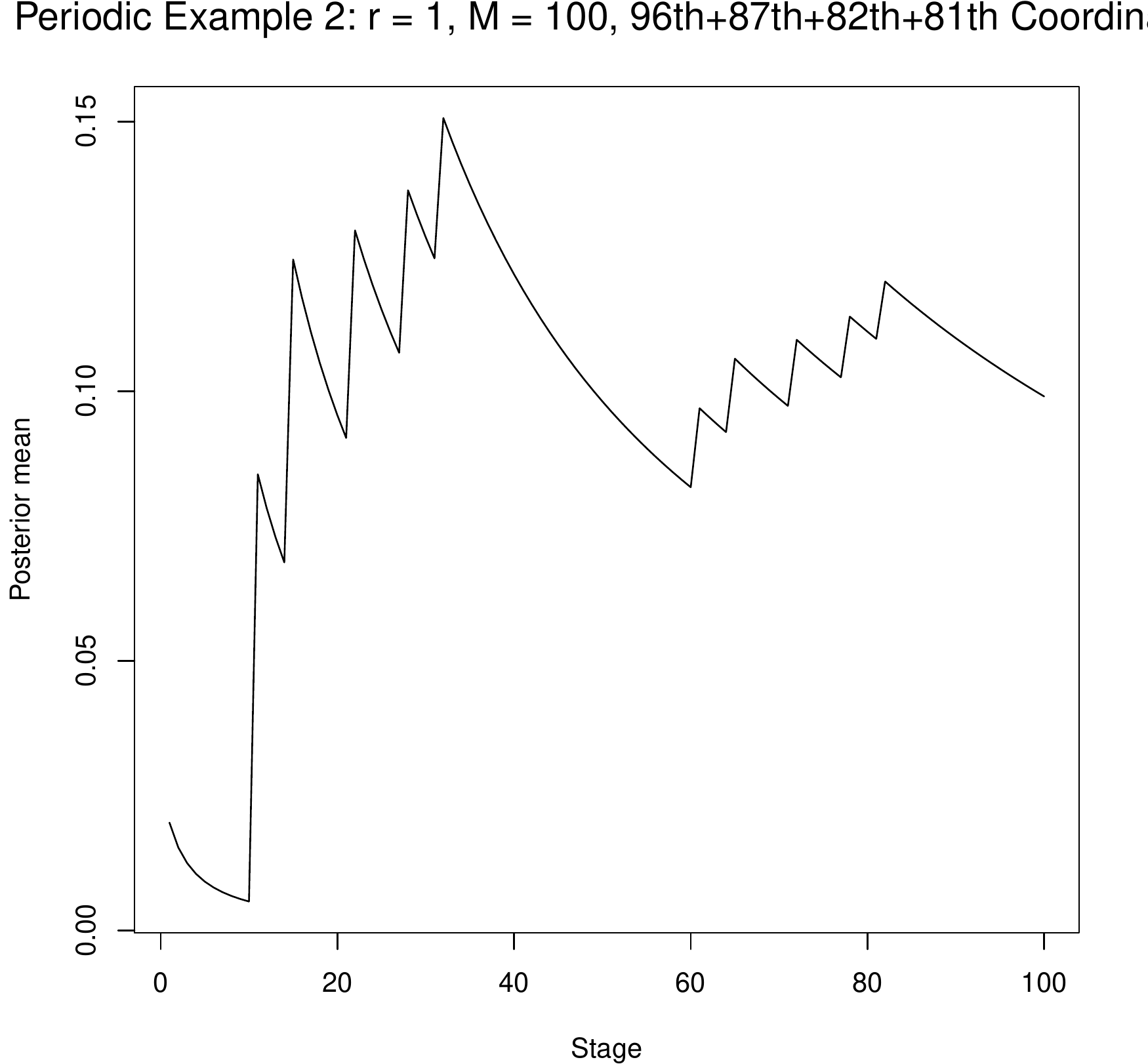}}
\hspace{2mm}
\subfigure [$r=1,M=100$. True frequency $=0.06$.]{ \label{fig:mult_osc_9}
\includegraphics[width=4.5cm,height=4.5cm]{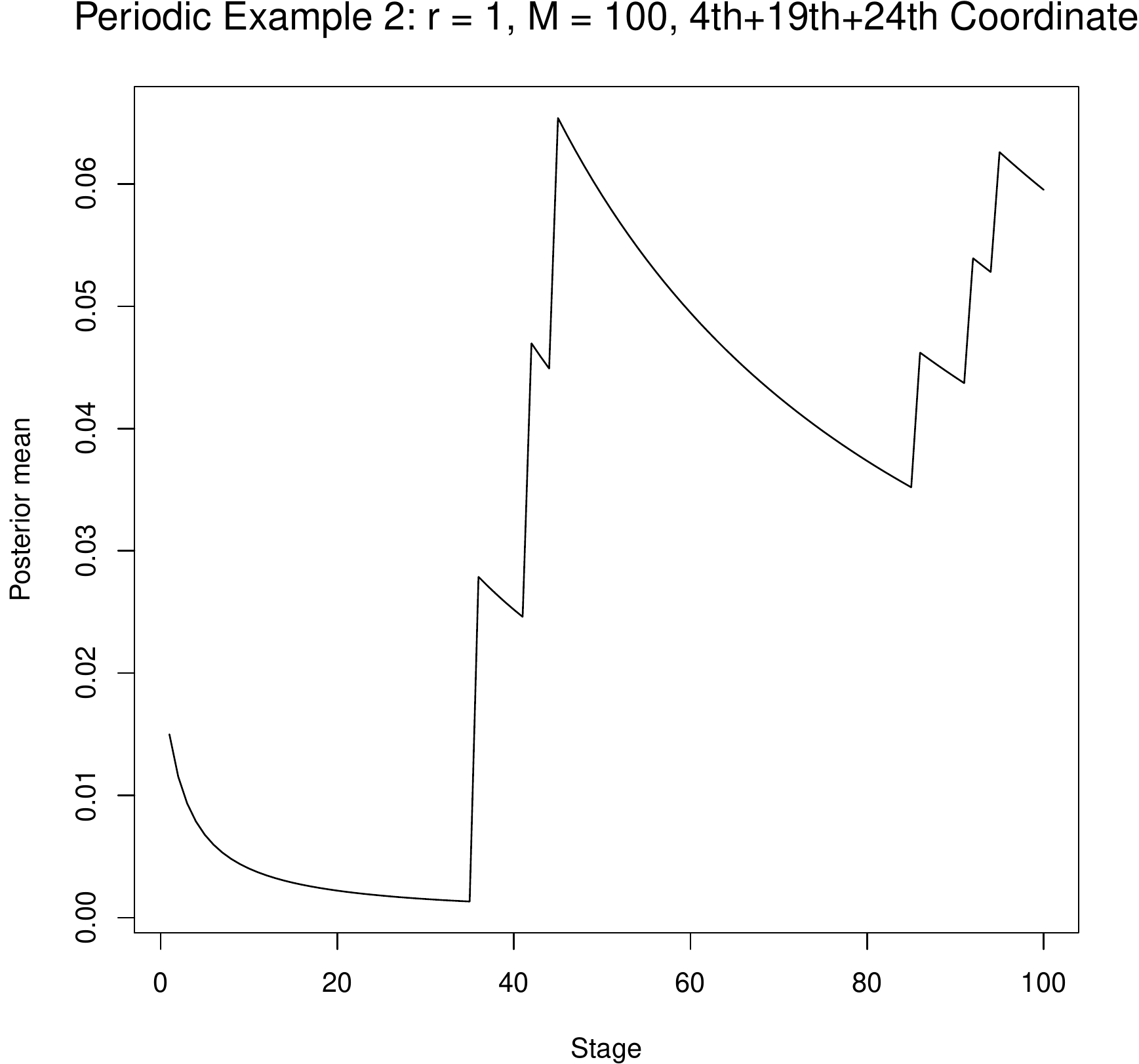}}
\caption{Illustration of our Bayesian method for determining multiple frequencies. Here the true frequencies are $0.4$, $0.1$ and $0.06$. }
\label{fig:mult_osc_example1}
\end{figure}

\begin{figure}
\centering
\subfigure [$r=5,M=10$. True frequency $=0.4$.]{ \label{fig:mult_osc_10}
\includegraphics[width=4.5cm,height=4.5cm]{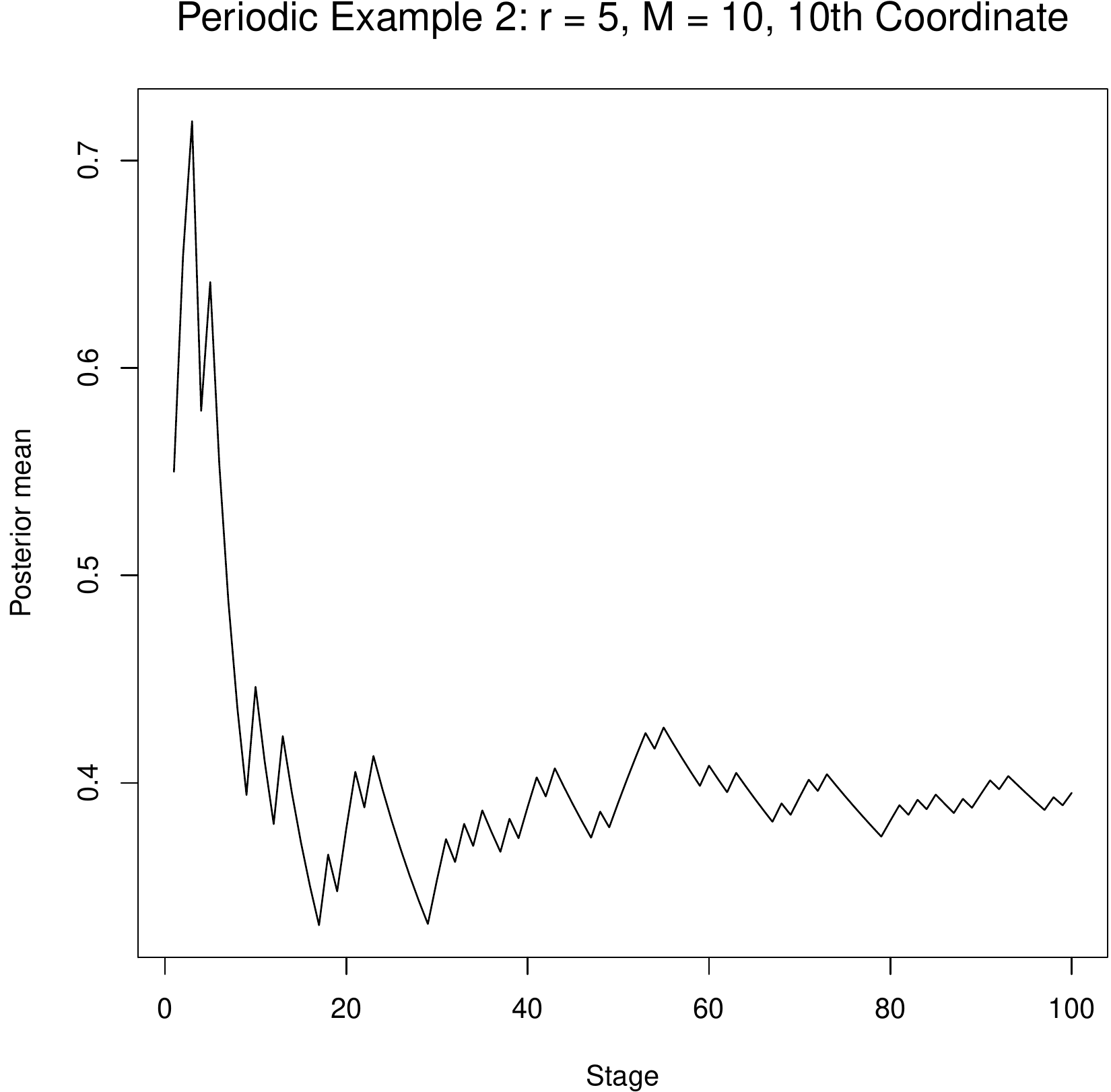}}
\hspace{2mm}
\subfigure [$r=5,M=10$. True frequency $=0.1$.]{ \label{fig:mult_osc_11}
\includegraphics[width=4.5cm,height=4.5cm]{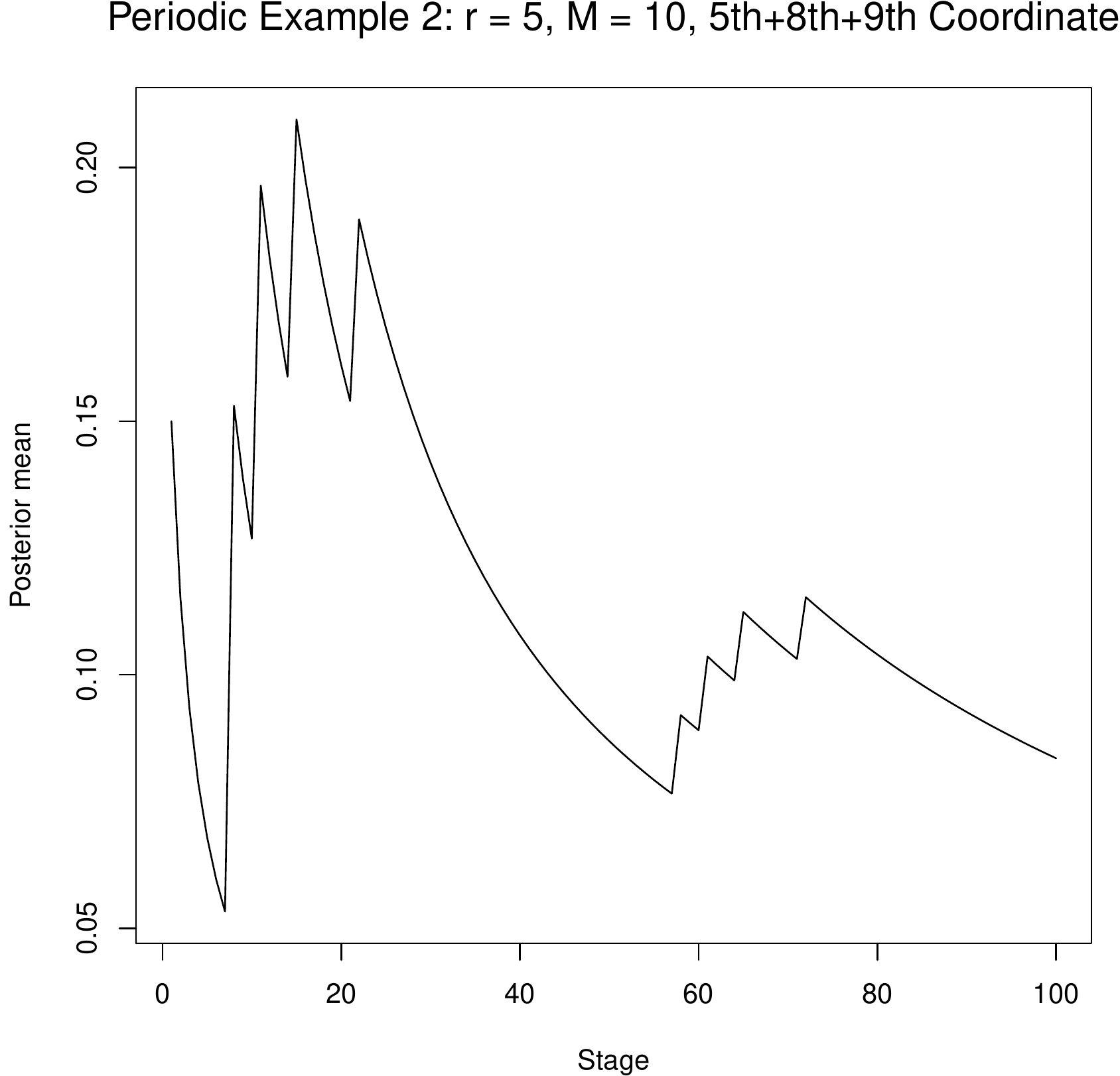}}
\hspace{2mm}
\subfigure [$r=5,M=10$. True frequency $=0.06$.]{ \label{fig:mult_osc_12}
\includegraphics[width=4.5cm,height=4.5cm]{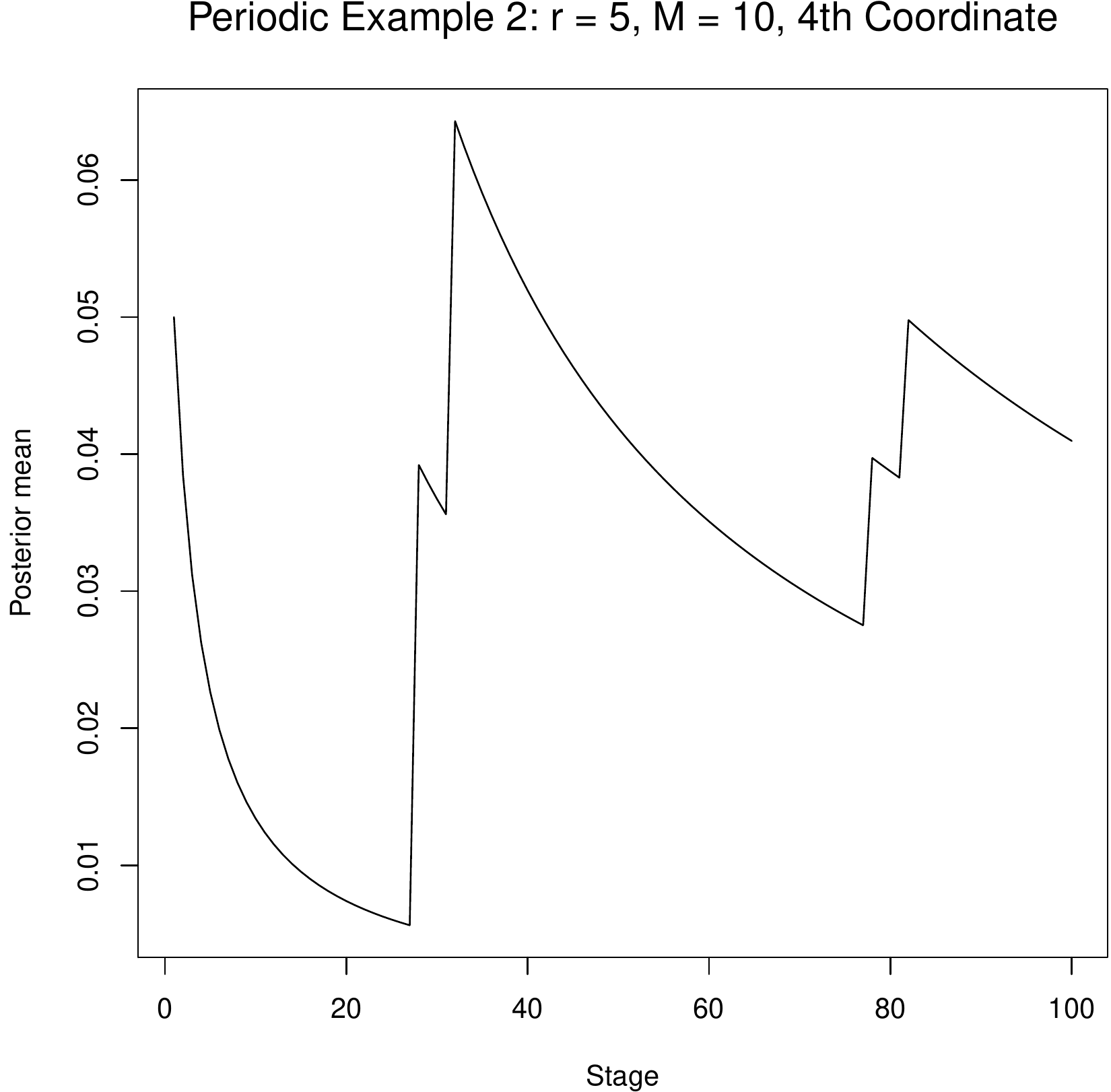}}\\
\vspace{2mm}
\subfigure [$r=5,M=50$. True frequency $=0.4$.]{ \label{fig:mult_osc_13}
\includegraphics[width=4.5cm,height=4.5cm]{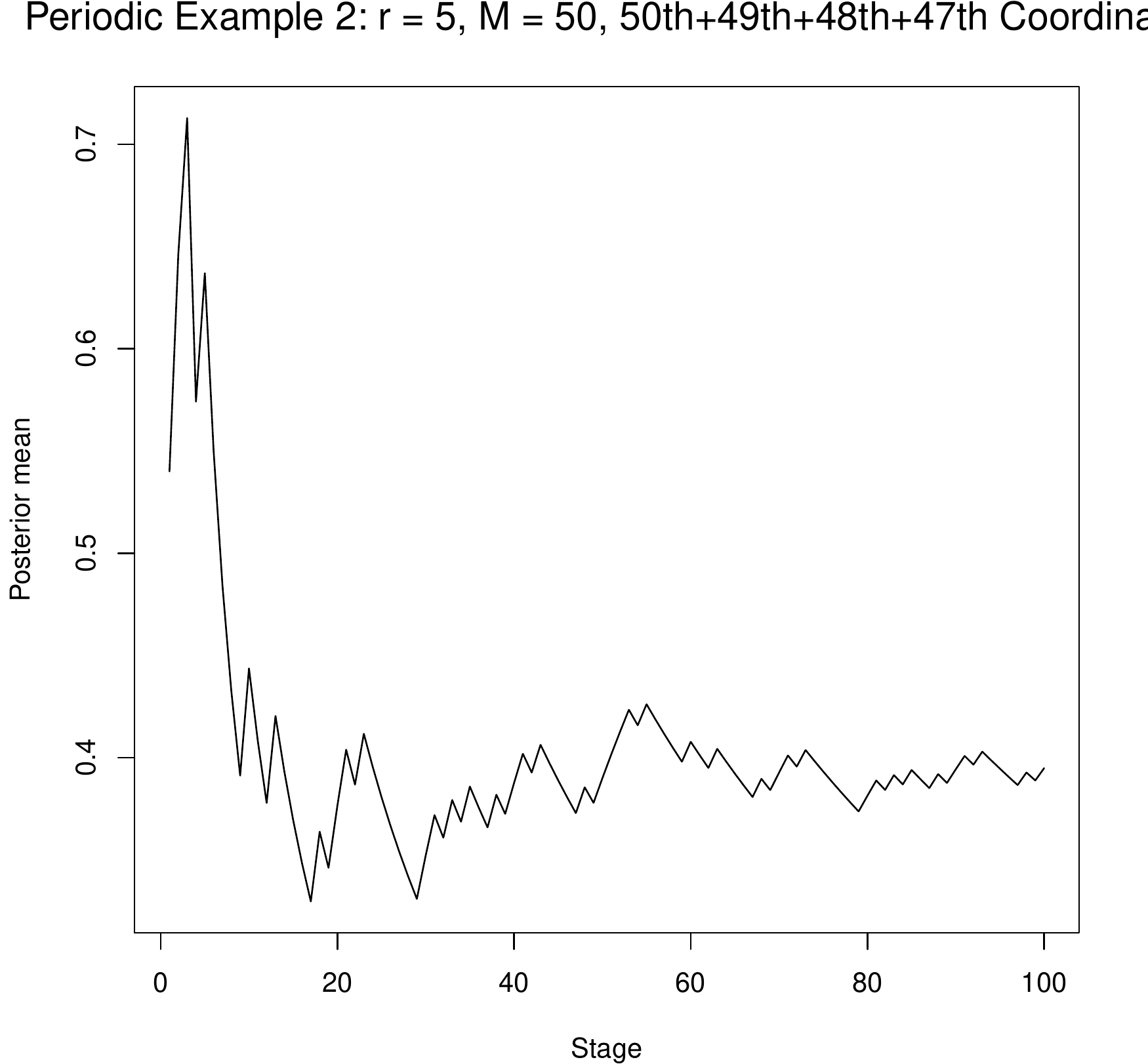}}
\hspace{2mm}
\subfigure [$r=5,M=50$. True frequency $=0.1$.]{ \label{fig:mult_osc_14}
\includegraphics[width=4.5cm,height=4.5cm]{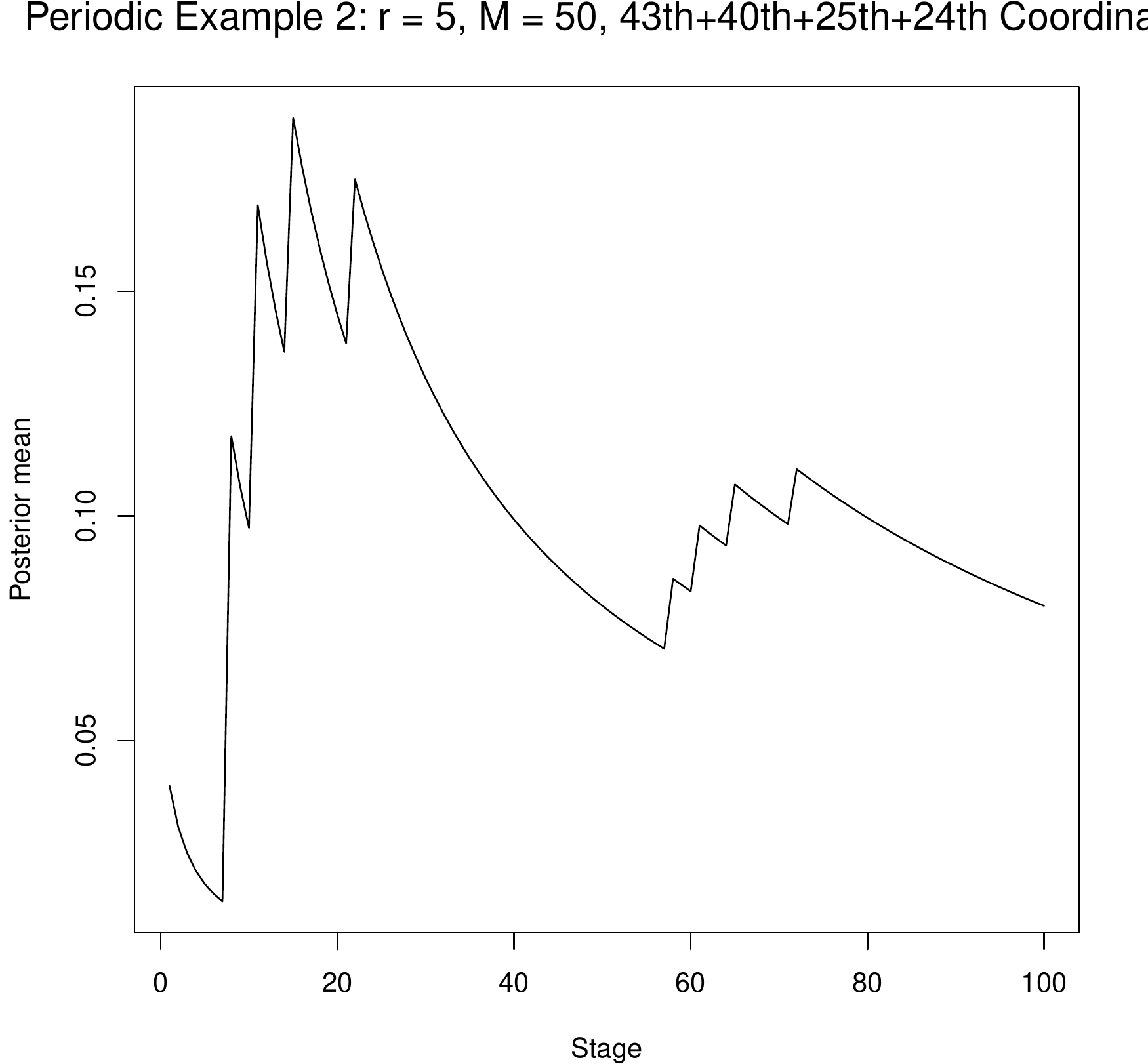}}
\hspace{2mm}
\subfigure [$r=5,M=50$. True frequency $=0.06$.]{ \label{fig:mult_osc_15}
\includegraphics[width=4.5cm,height=4.5cm]{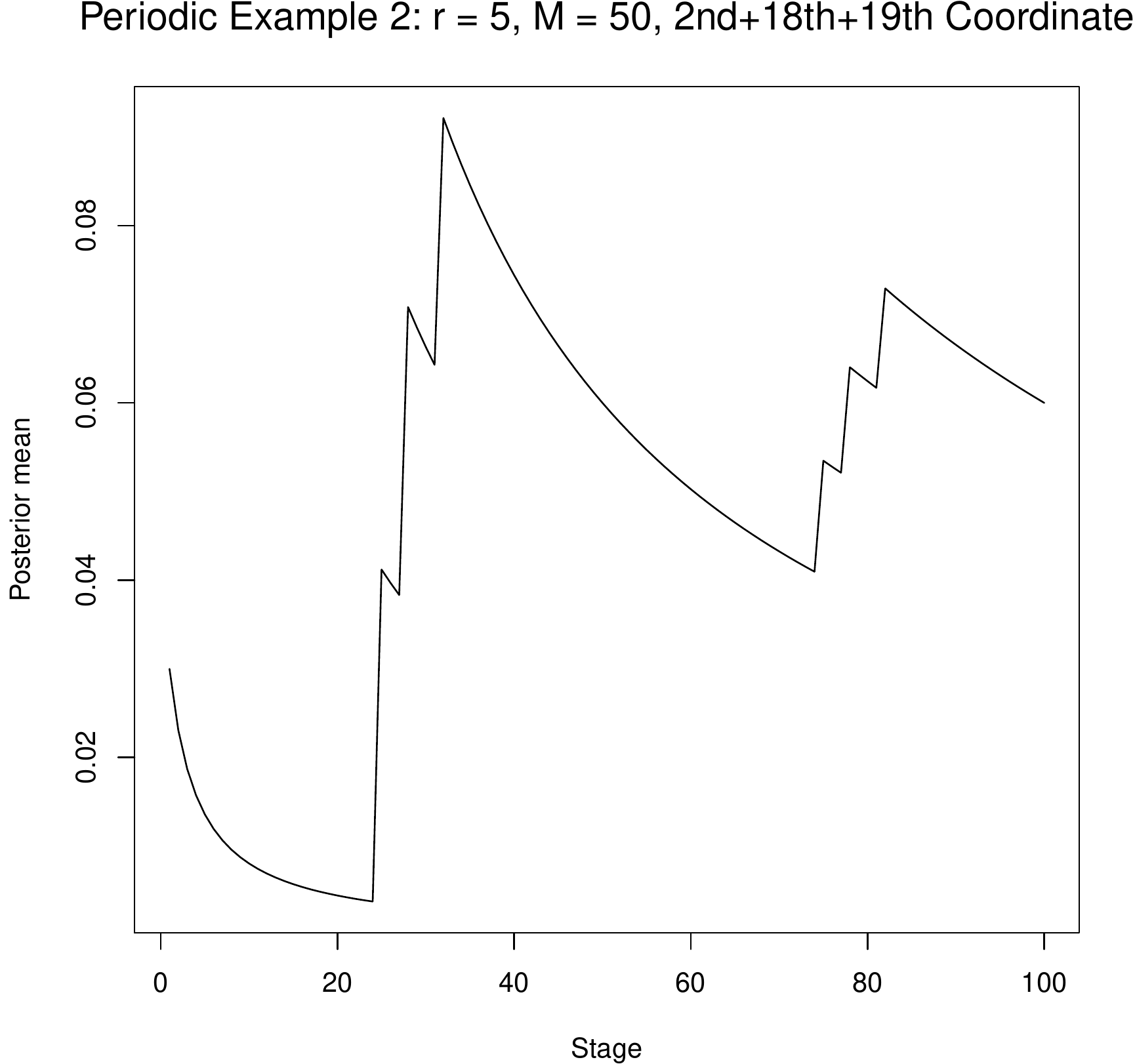}}\\
\vspace{2mm}
\subfigure [$r=5,M=100$. True frequency $=0.4$.]{ \label{fig:mult_osc_16}
\includegraphics[width=4.5cm,height=4.5cm]{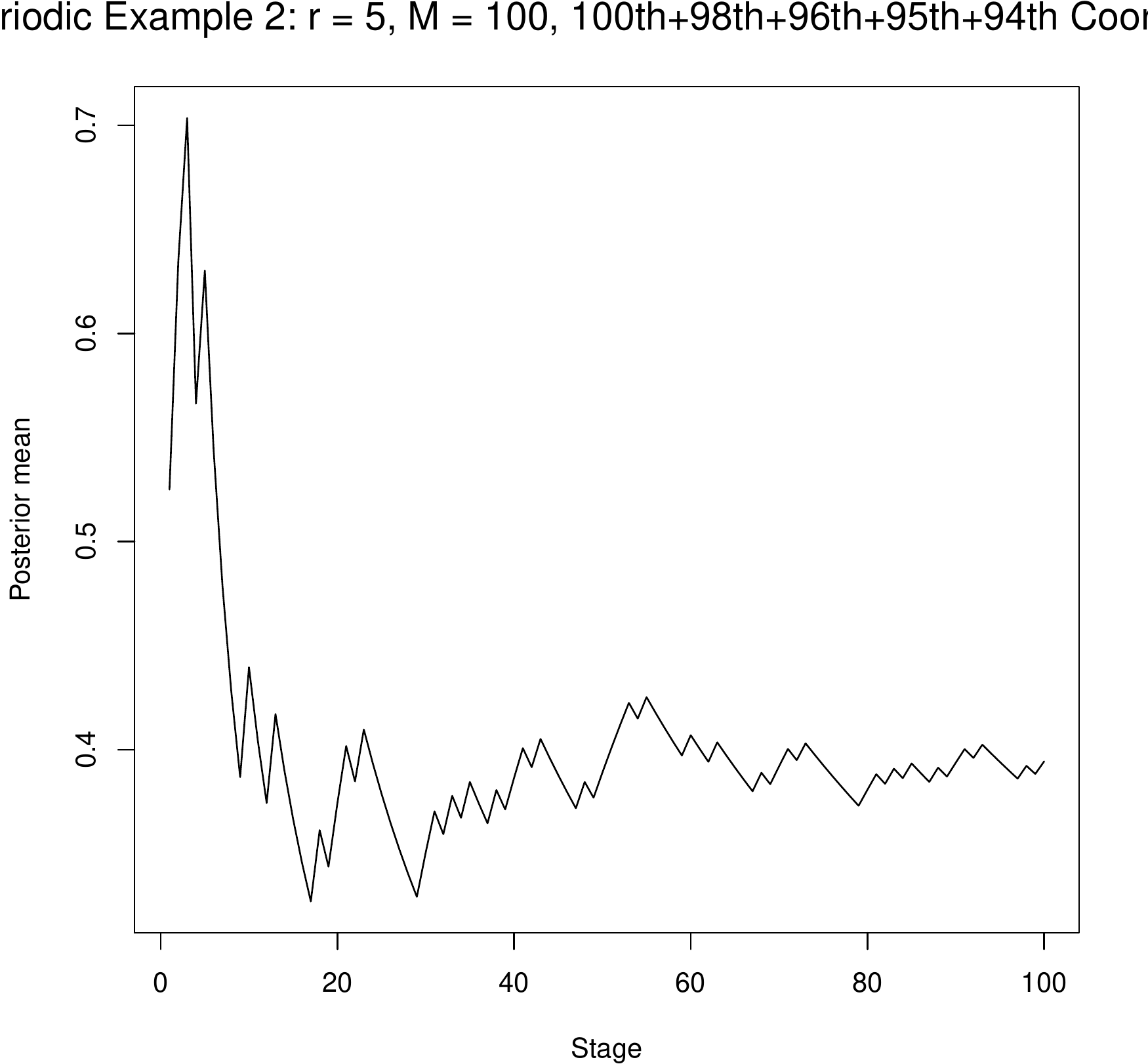}}
\hspace{2mm}
\subfigure [$r=5,M=100$. True frequency $=0.1$.]{ \label{fig:mult_osc_17}
\includegraphics[width=4.5cm,height=4.5cm]{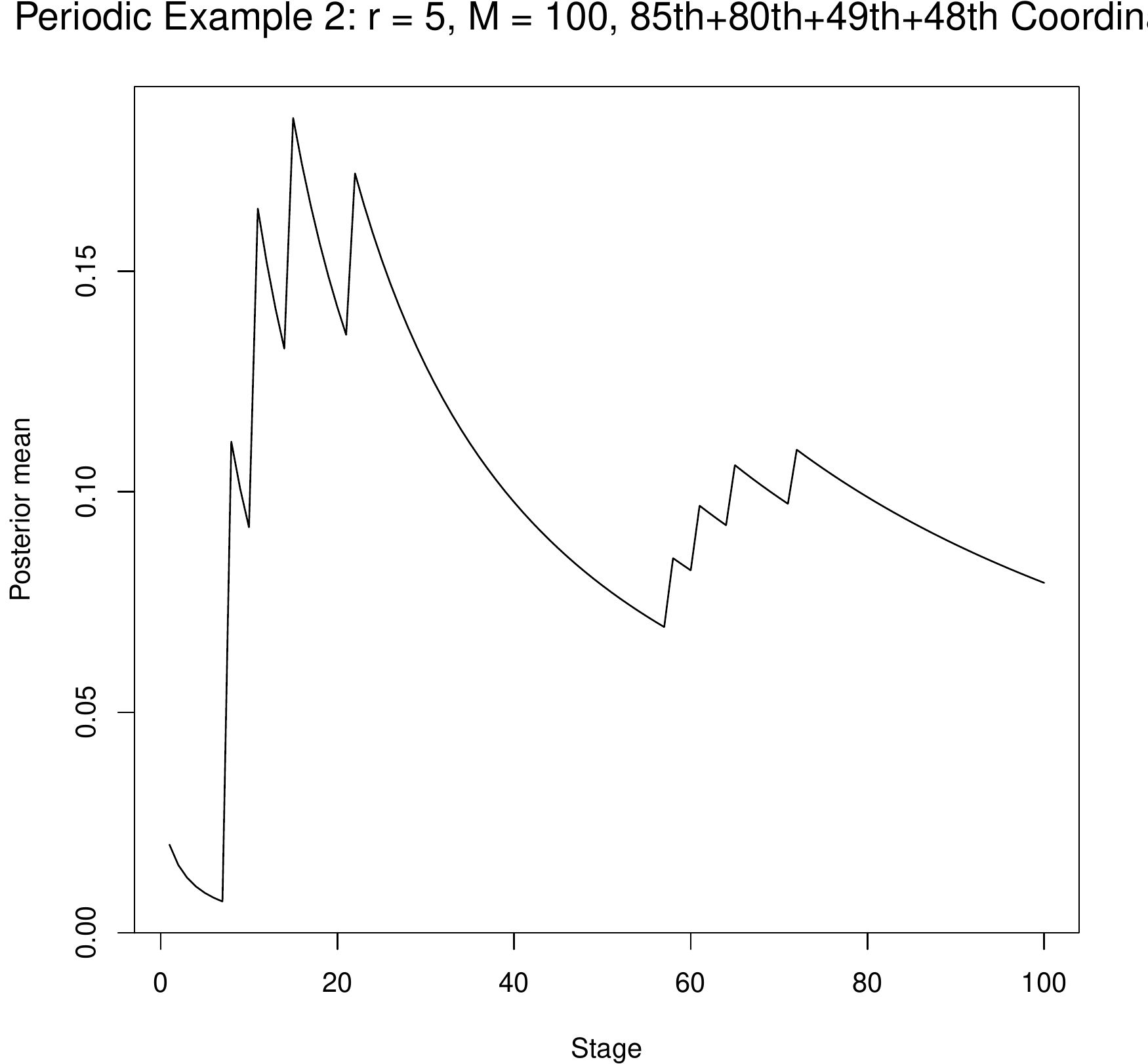}}
\hspace{2mm}
\subfigure [$r=5,M=100$. True frequency $=0.06$.]{ \label{fig:mult_osc_18}
\includegraphics[width=4.5cm,height=4.5cm]{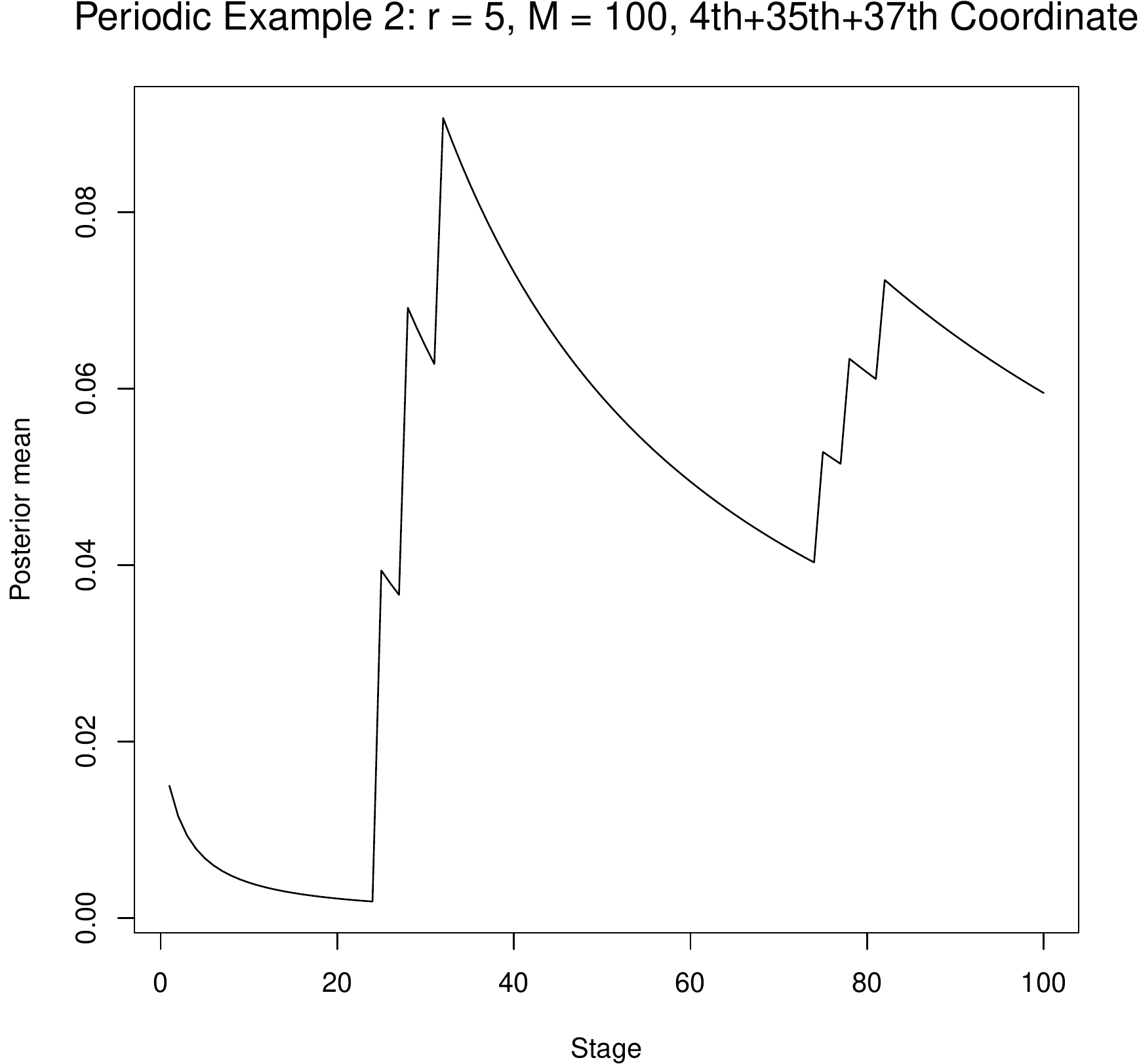}}
\caption{Illustration of our Bayesian method for determining multiple frequencies. Here the true frequencies are $0.4$, $0.1$ and $0.06$. }
\label{fig:mult_osc_example2}
\end{figure}

\begin{figure}
\centering
\subfigure [$r=10,M=10$. True frequency $=0.4$.]{ \label{fig:mult_osc_19}
\includegraphics[width=4.5cm,height=4.5cm]{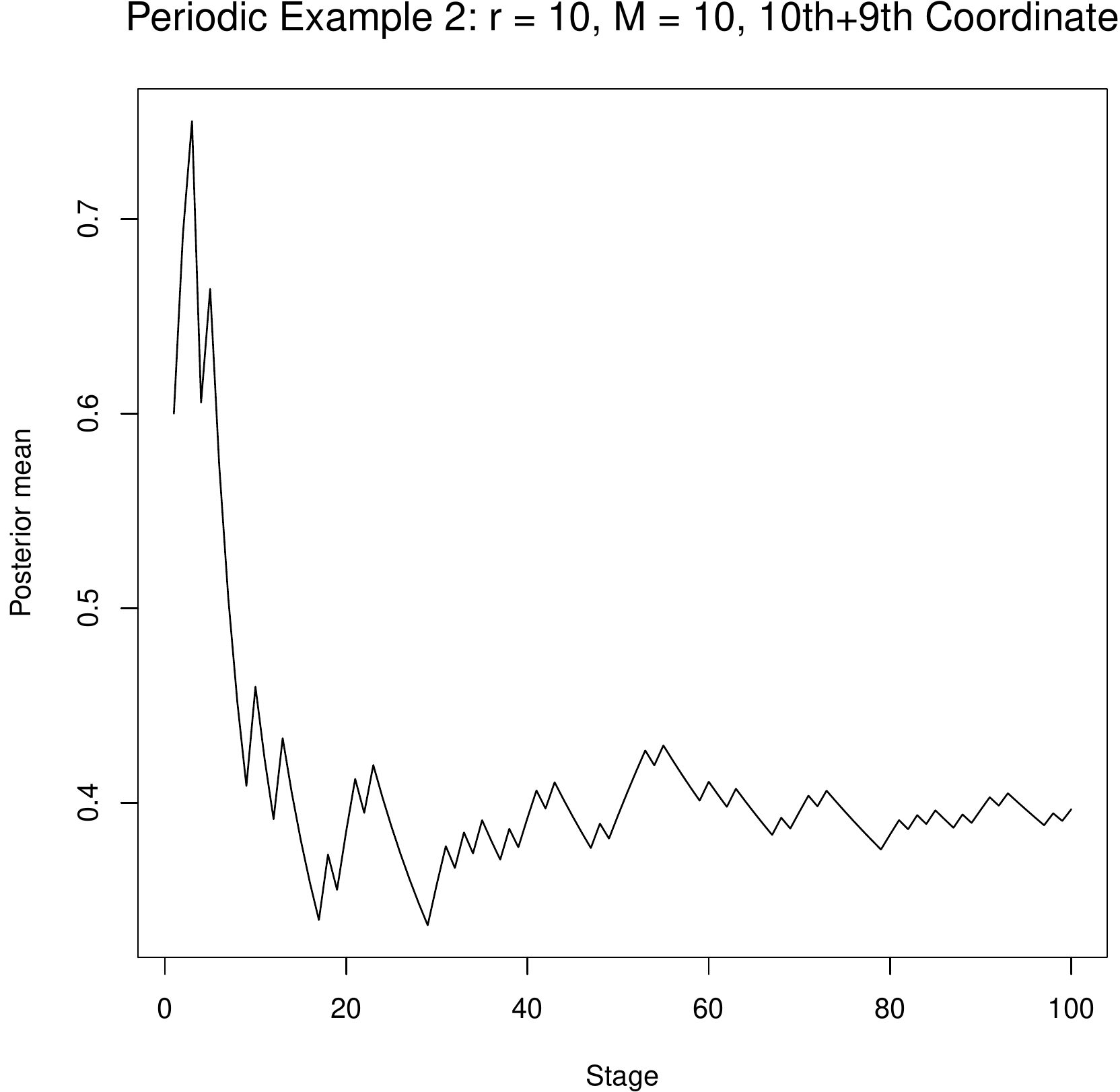}}
\hspace{2mm}
\subfigure [$r=10,M=10$. True frequency $=0.1$.]{ \label{fig:mult_osc_20}
\includegraphics[width=4.5cm,height=4.5cm]{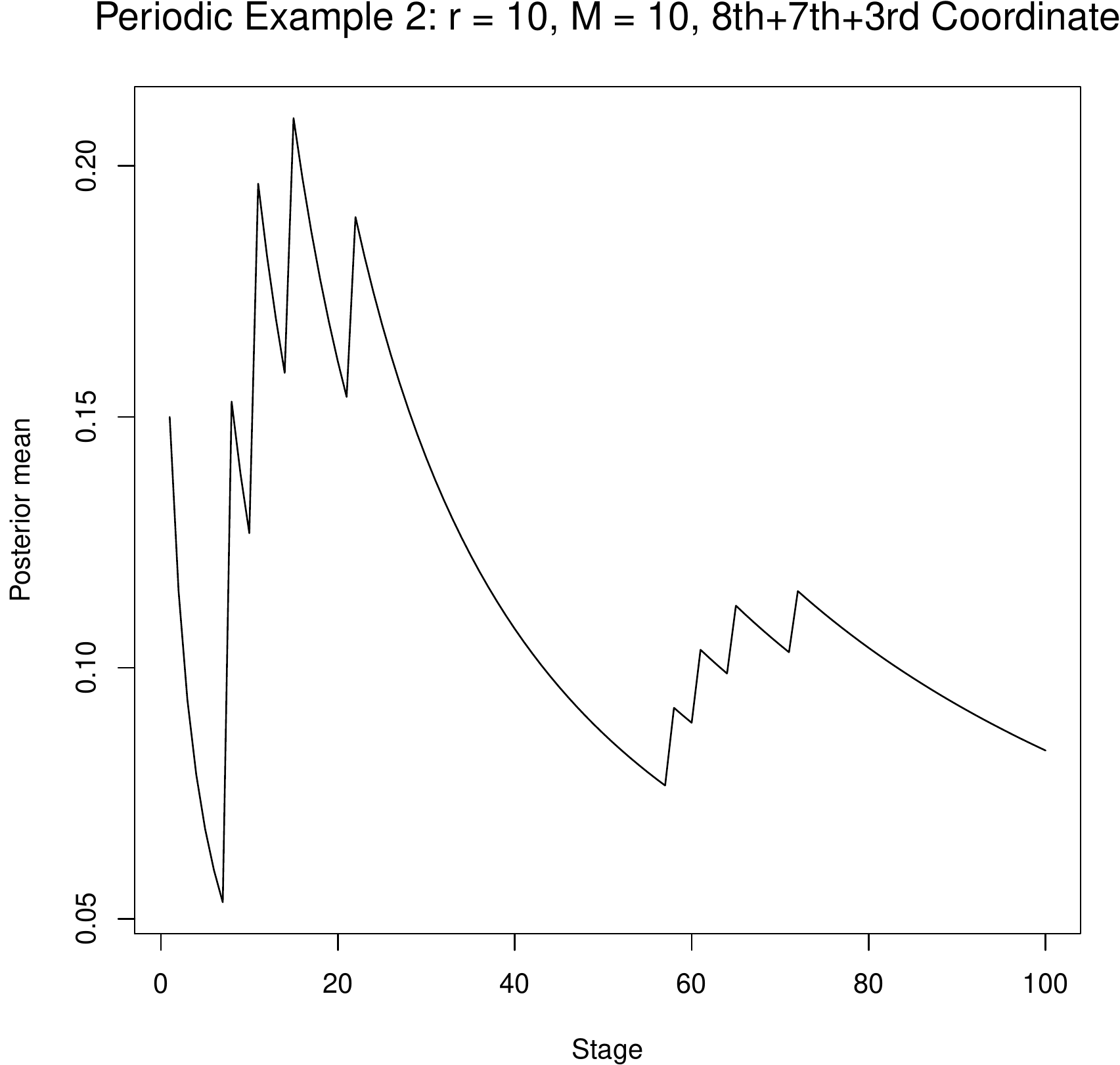}}
\hspace{2mm}
\subfigure [$r=10,M=10$. True frequency $=0.06$.]{ \label{fig:mult_osc_21}
\includegraphics[width=4.5cm,height=4.5cm]{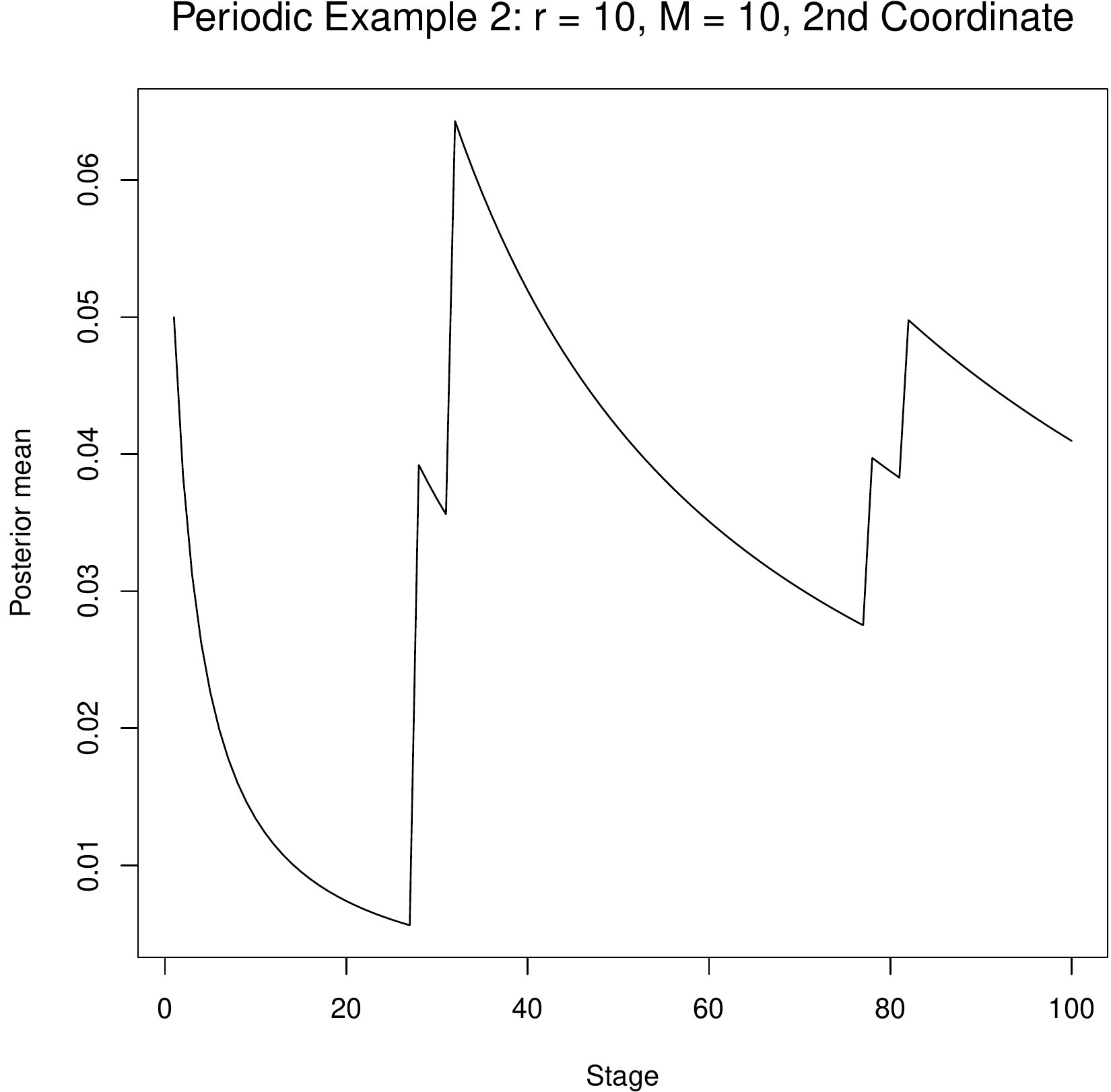}}\\
\vspace{2mm}
\subfigure [$r=10,M=50$. True frequency $=0.4$.]{ \label{fig:mult_osc_22}
\includegraphics[width=4.5cm,height=4.5cm]{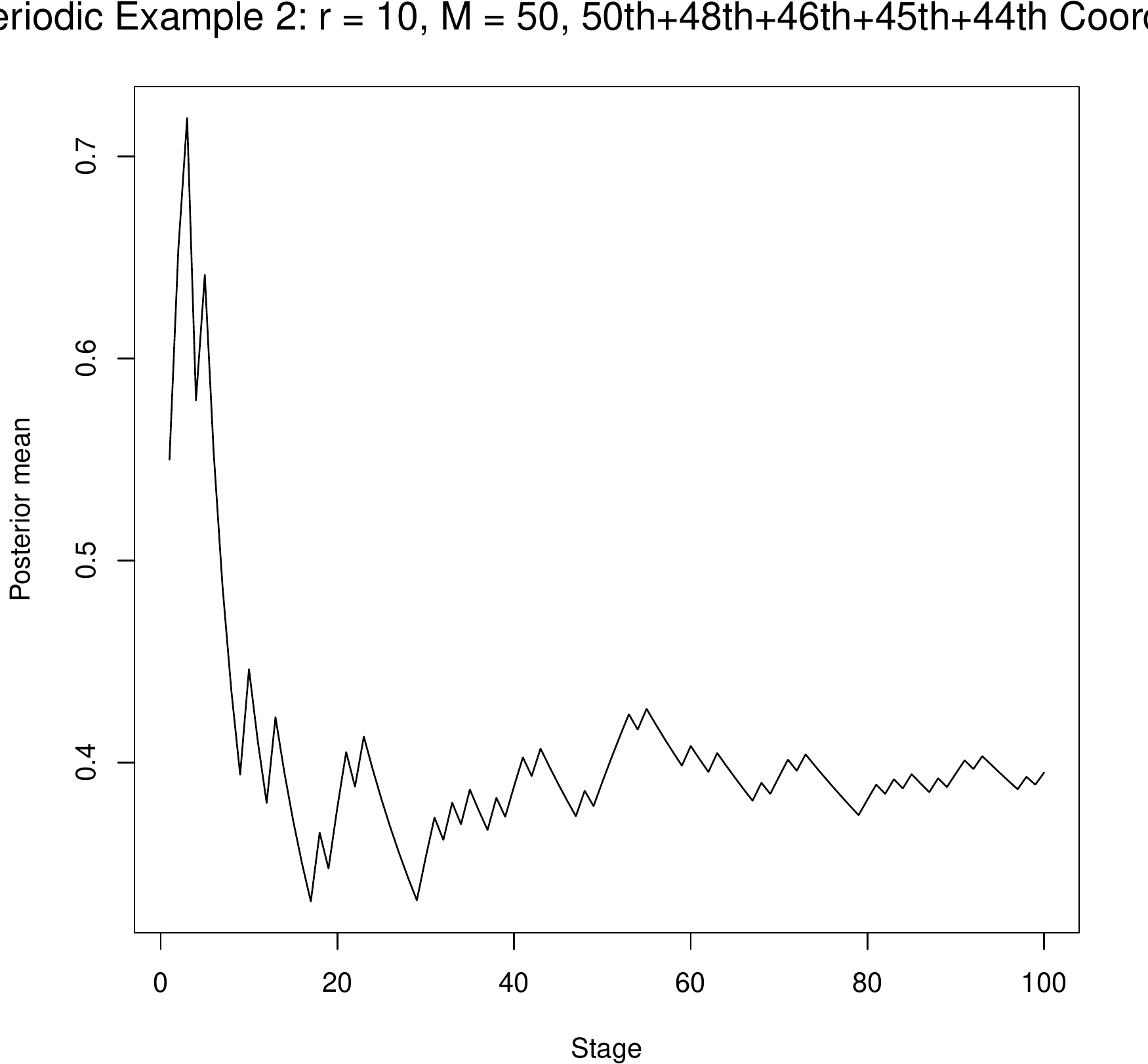}}
\hspace{2mm}
\subfigure [$r=10,M=50$. True frequency $=0.1$.]{ \label{fig:mult_osc_23}
\includegraphics[width=4.5cm,height=4.5cm]{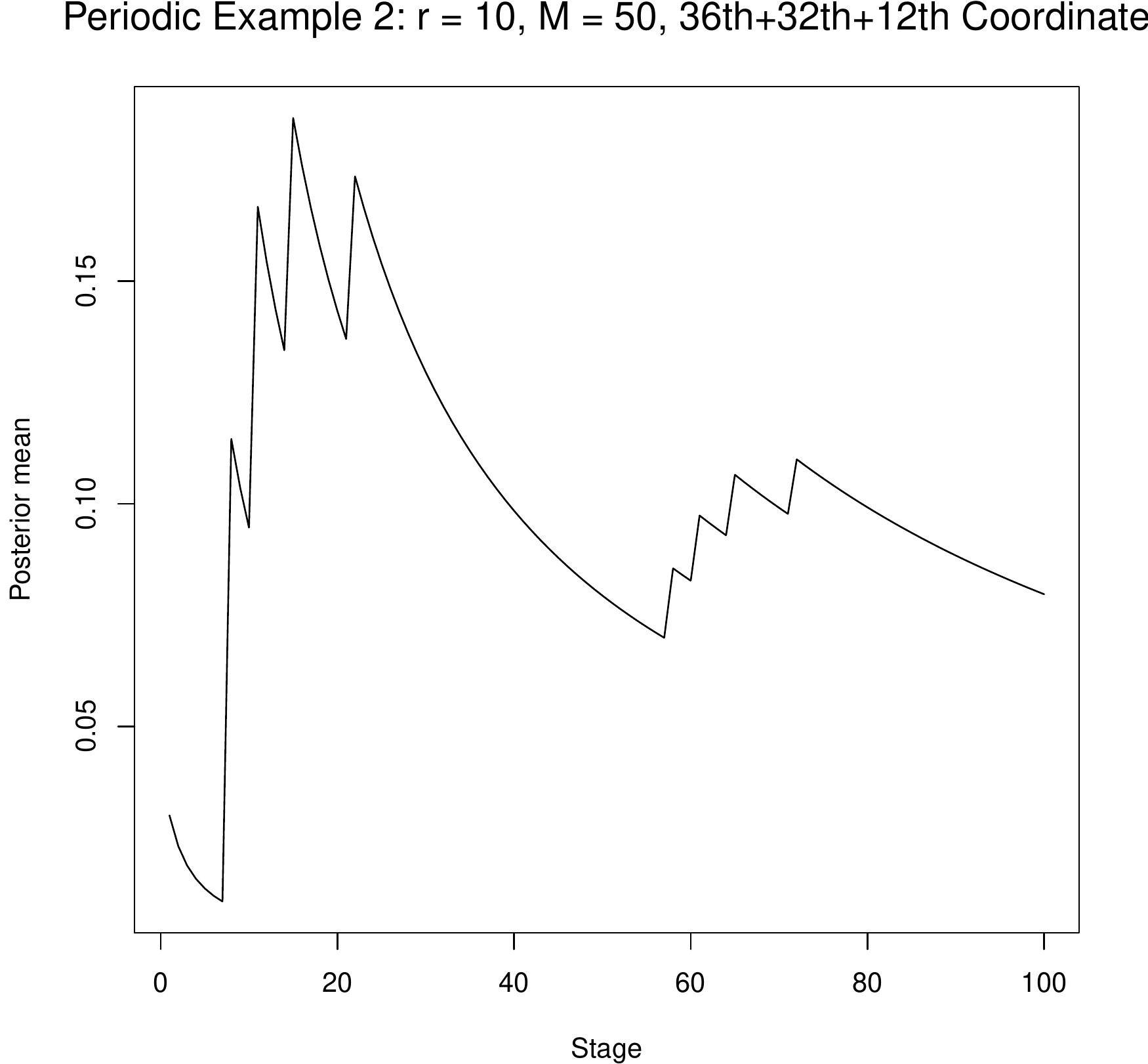}}
\hspace{2mm}
\subfigure [$r=10,M=50$. True frequency $=0.06$.]{ \label{fig:mult_osc_24}
\includegraphics[width=4.5cm,height=4.5cm]{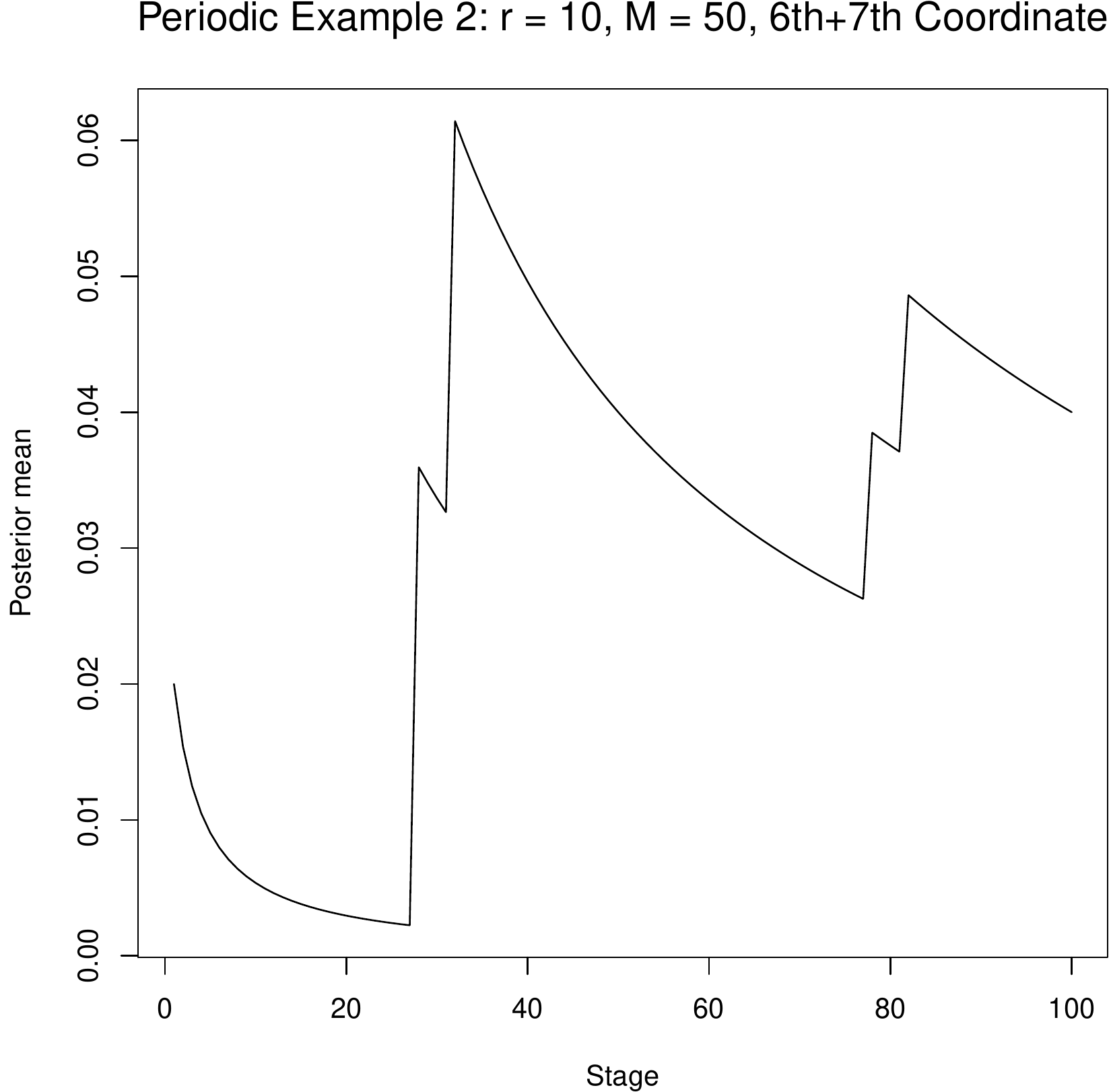}}\\
\vspace{2mm}
\subfigure [$r=10,M=100$. True frequency $=0.4$.]{ \label{fig:mult_osc_25}
\includegraphics[width=4.5cm,height=4.5cm]{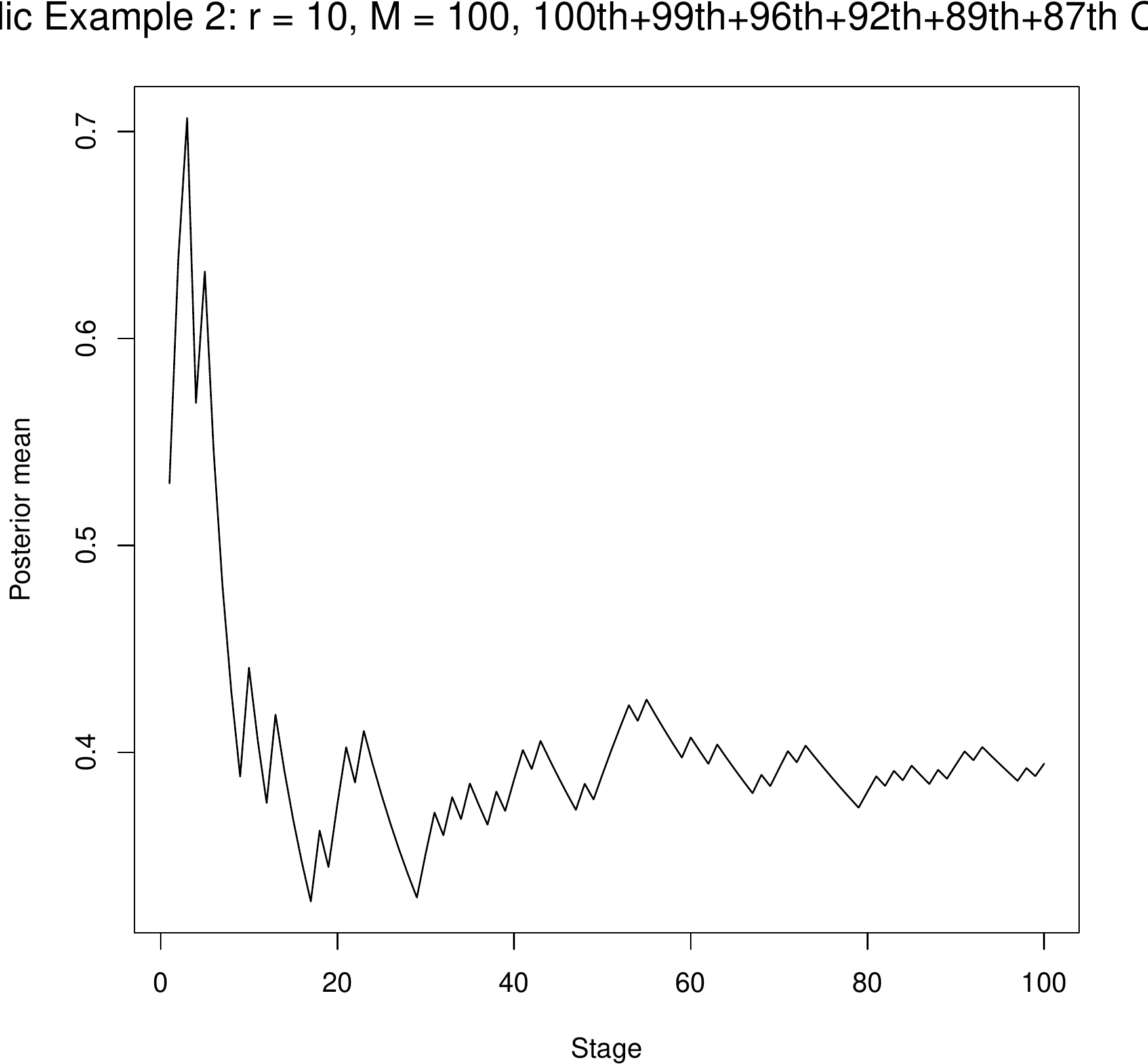}}
\hspace{2mm}
\subfigure [$r=10,M=100$. True frequency $=0.1$.]{ \label{fig:mult_osc_26}
\includegraphics[width=4.5cm,height=4.5cm]{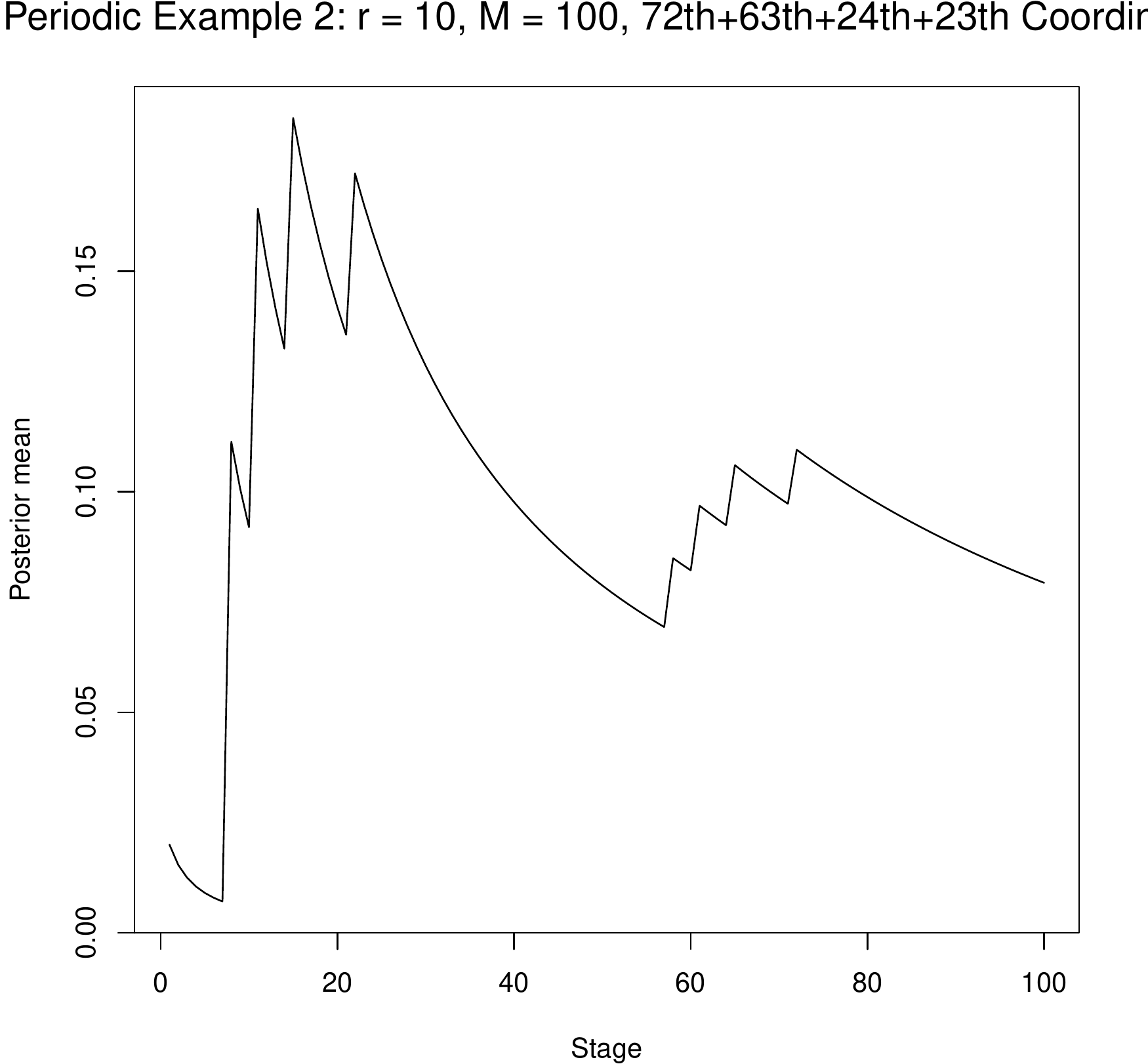}}
\hspace{2mm}
\subfigure [$r=10,M=100$. True frequency $=0.06$.]{ \label{fig:mult_osc_27}
\includegraphics[width=4.5cm,height=4.5cm]{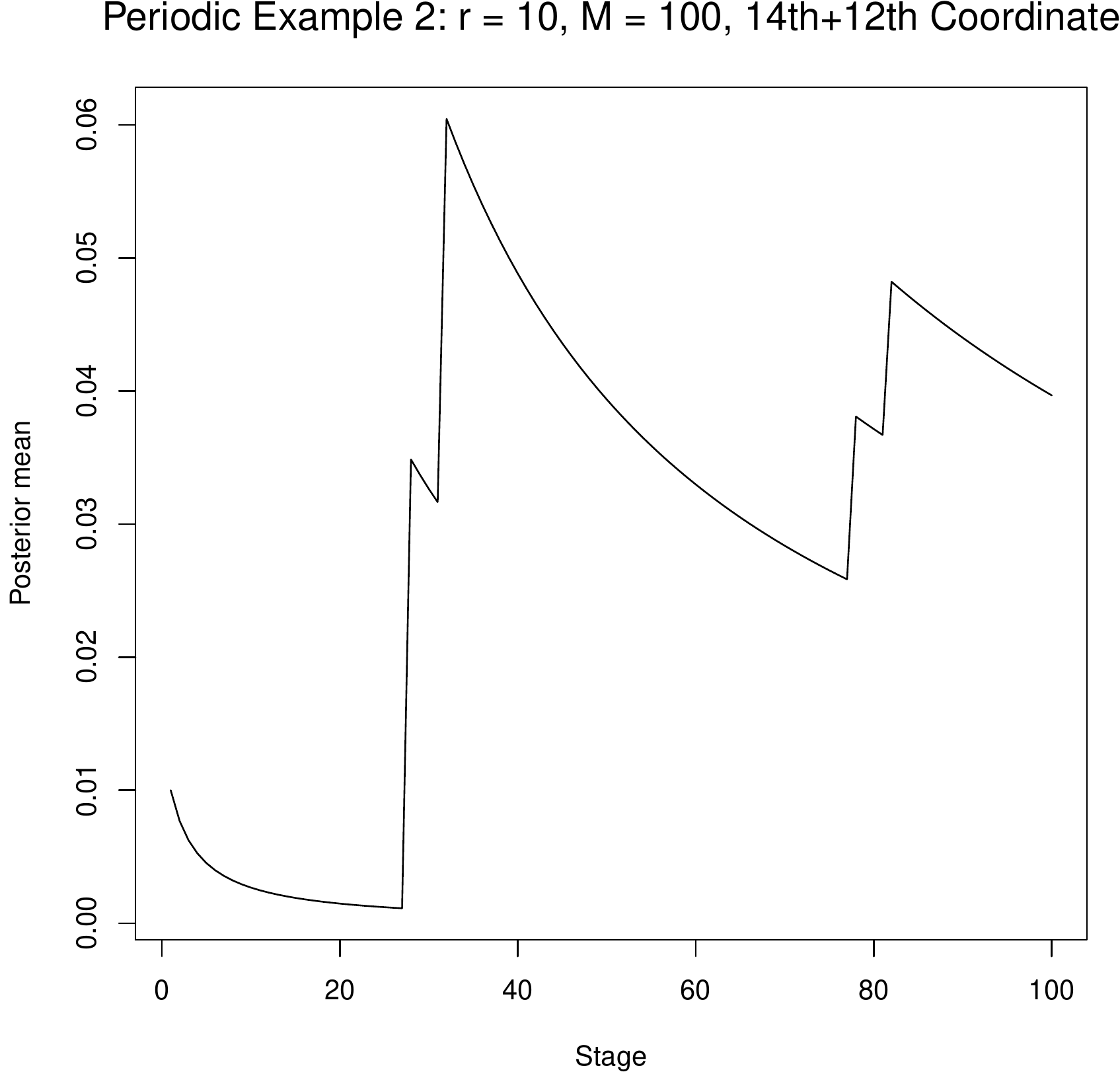}}
\caption{Illustration of our Bayesian method for determining multiple frequencies. Here the true frequencies are $0.4$, $0.1$ and $0.06$. }
\label{fig:mult_osc_example3}
\end{figure}

\begin{figure}
\centering
\subfigure [$r=50,M=10$. True frequency $=0.4$.]{ \label{fig:mult_osc_28}
\includegraphics[width=4.5cm,height=4.5cm]{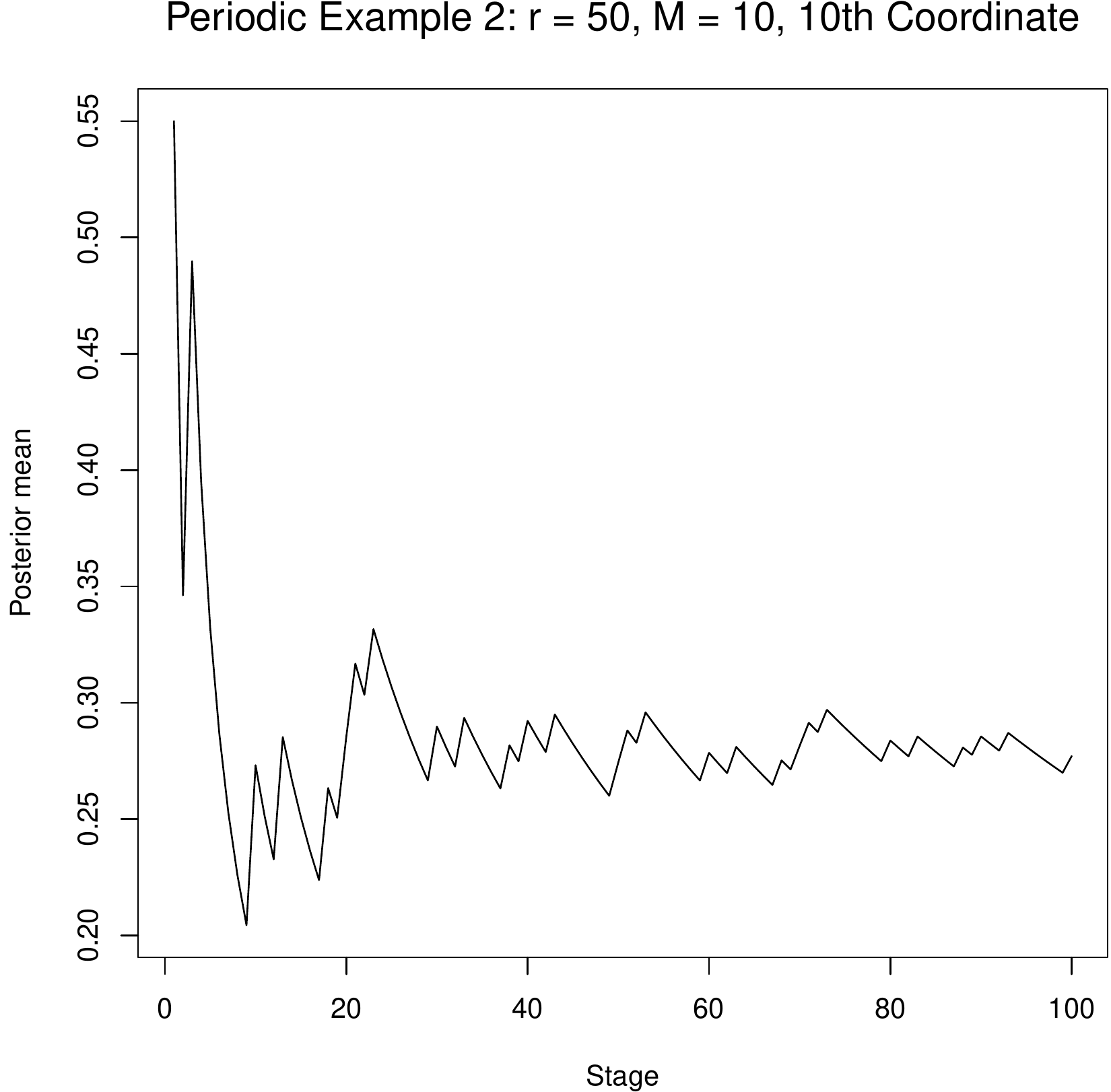}}
\hspace{2mm}
\subfigure [$r=50,M=10$. True frequency $=0.1$.]{ \label{fig:mult_osc_29}
\includegraphics[width=4.5cm,height=4.5cm]{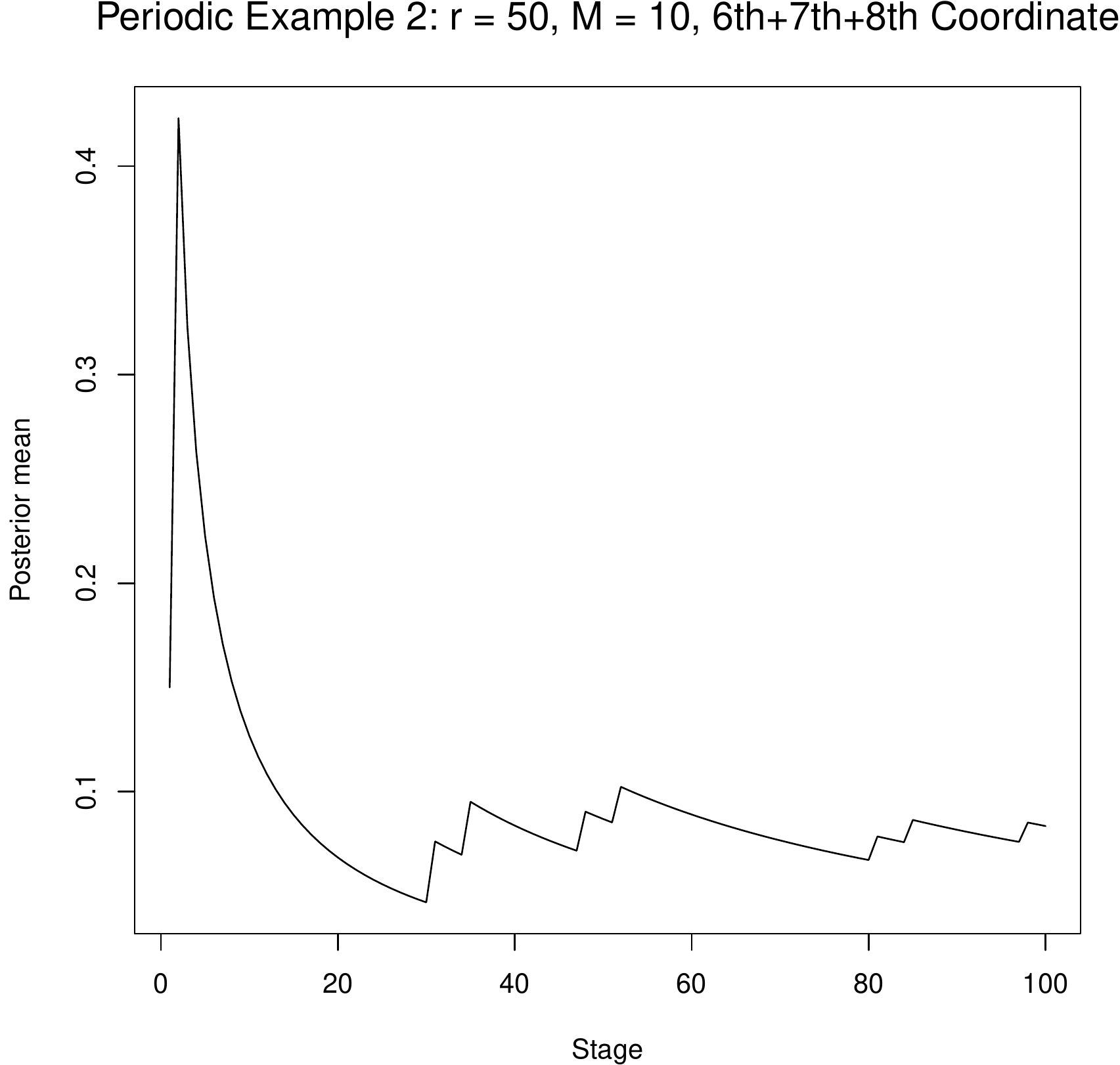}}
\hspace{2mm}
\subfigure [$r=50,M=10$. True frequency $=0.06$.]{ \label{fig:mult_osc_30}
\includegraphics[width=4.5cm,height=4.5cm]{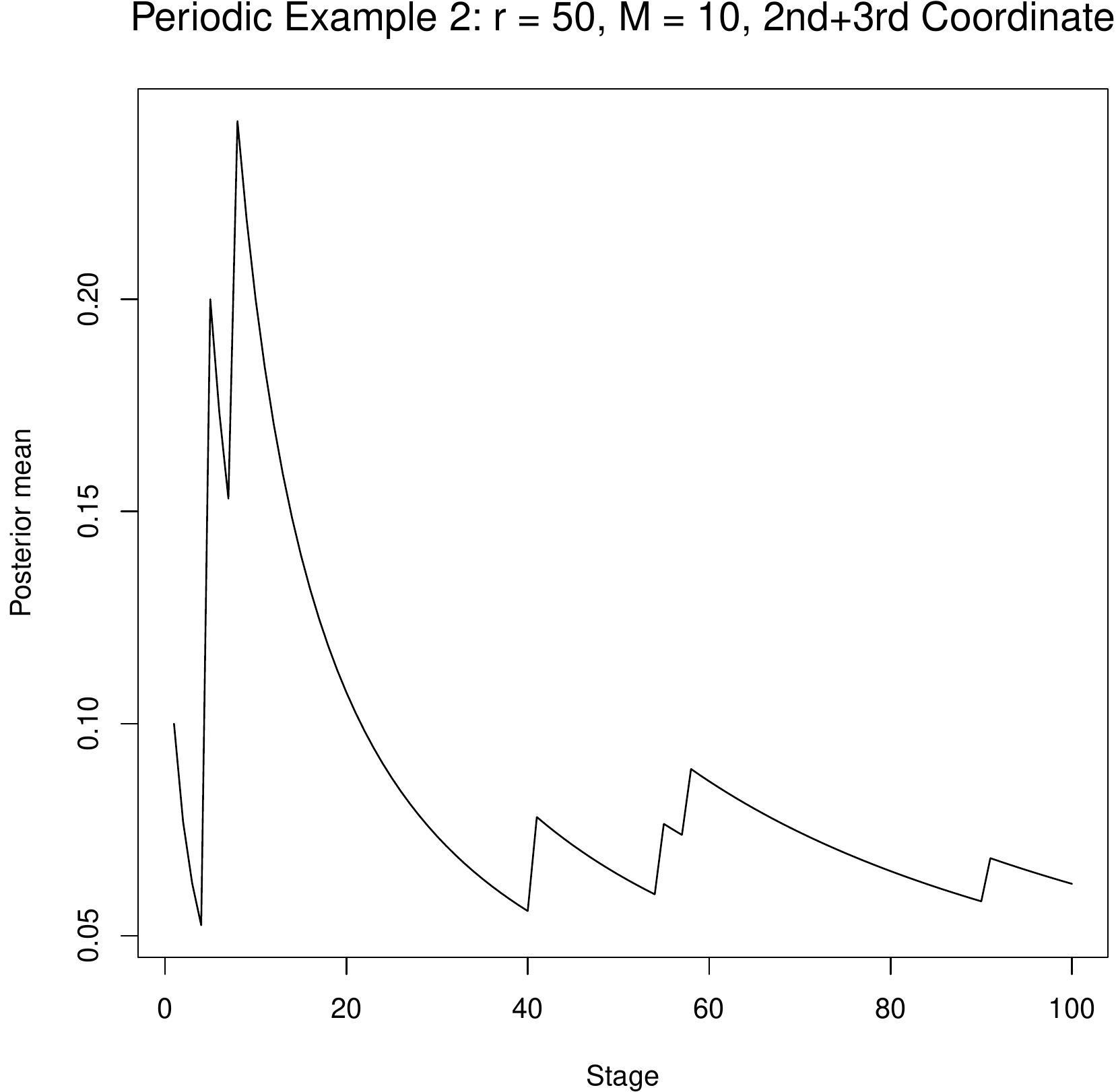}}\\
\vspace{2mm}
\subfigure [$r=50,M=50$. True frequency $=0.4$.]{ \label{fig:mult_osc_31}
\includegraphics[width=4.5cm,height=4.5cm]{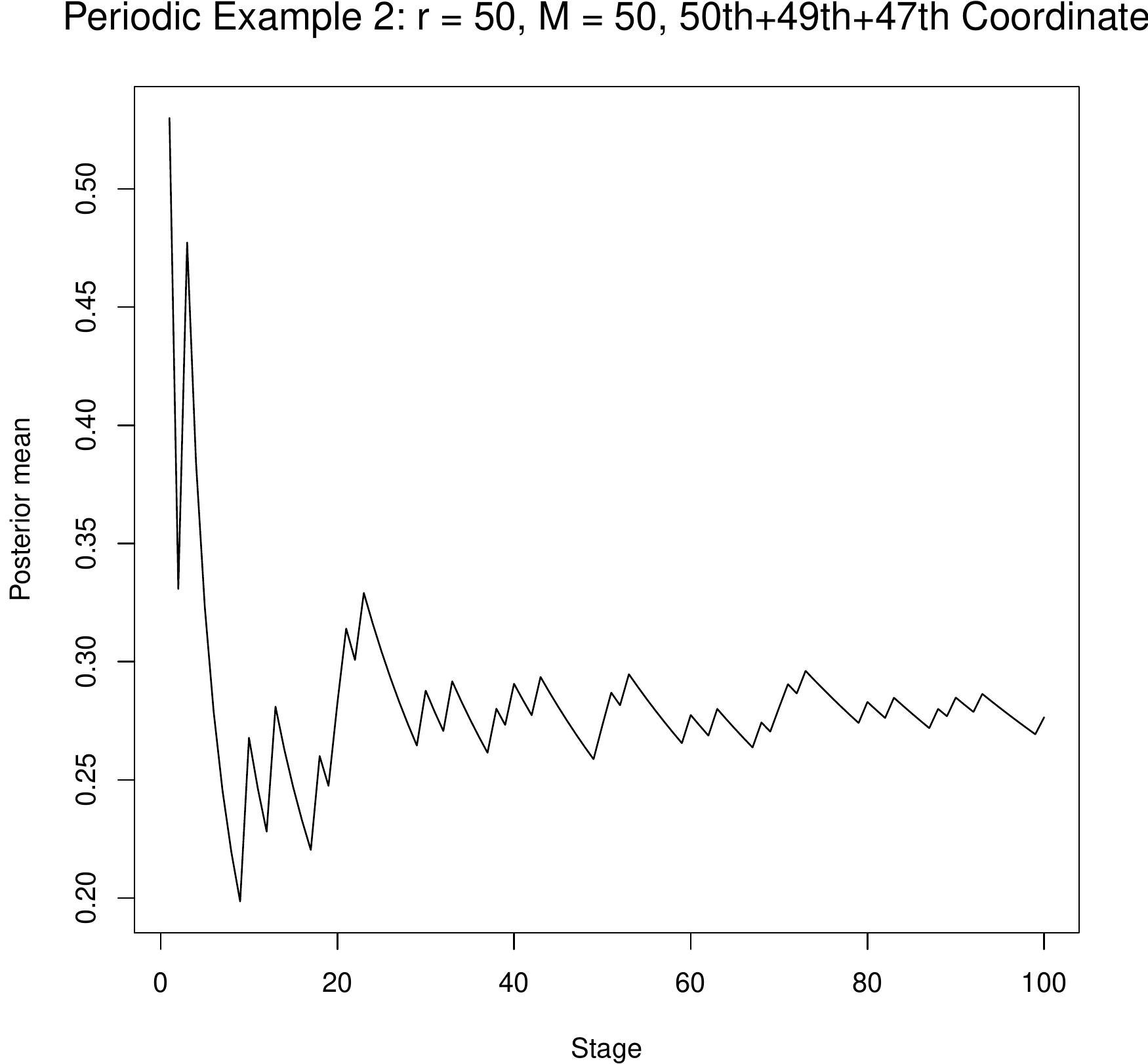}}
\hspace{2mm}
\subfigure [$r=50,M=50$. True frequency $=0.1$.]{ \label{fig:mult_osc_32}
\includegraphics[width=4.5cm,height=4.5cm]{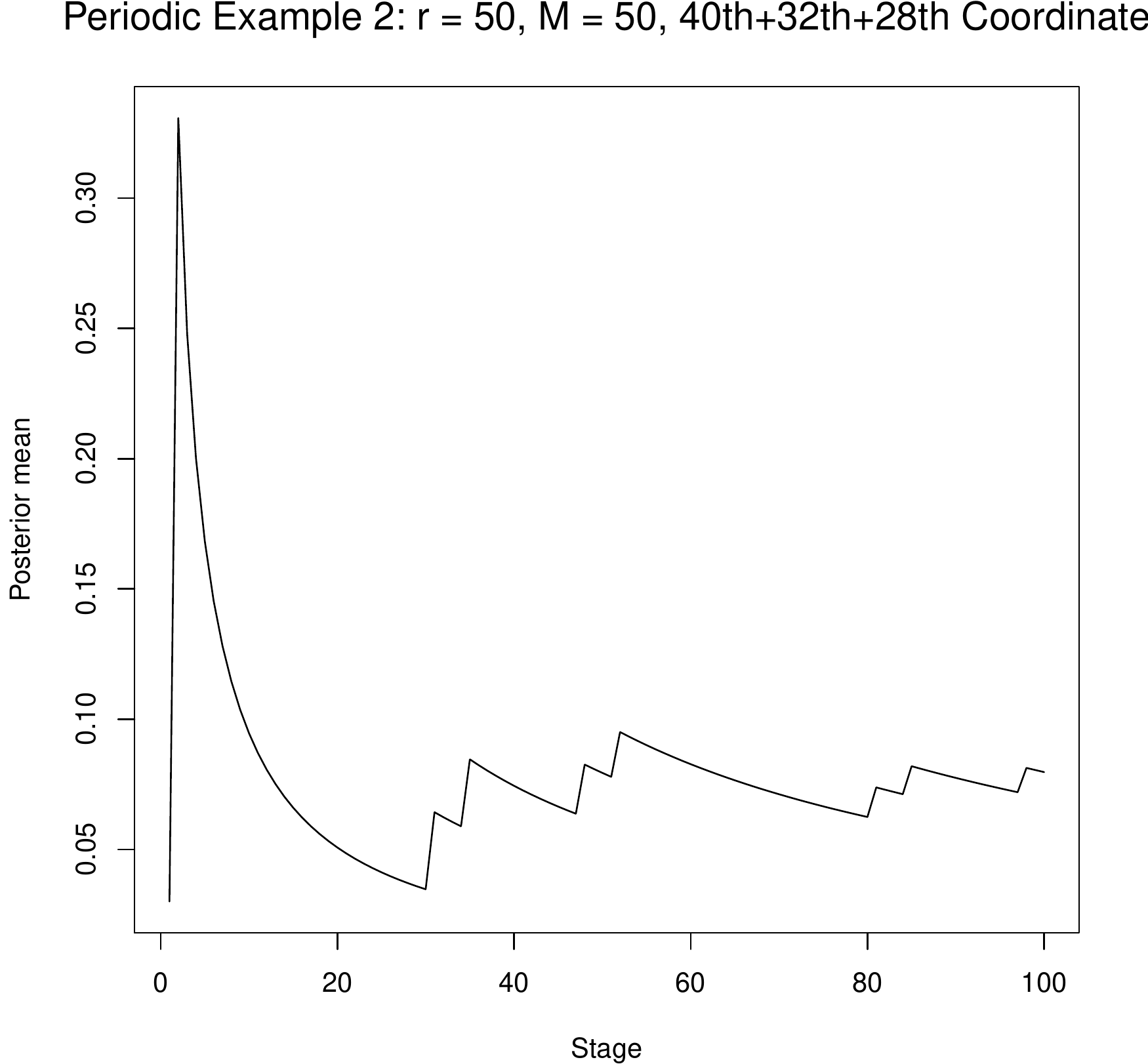}}
\hspace{2mm}
\subfigure [$r=50,M=50$. True frequency $=0.06$.]{ \label{fig:mult_osc_33}
\includegraphics[width=4.5cm,height=4.5cm]{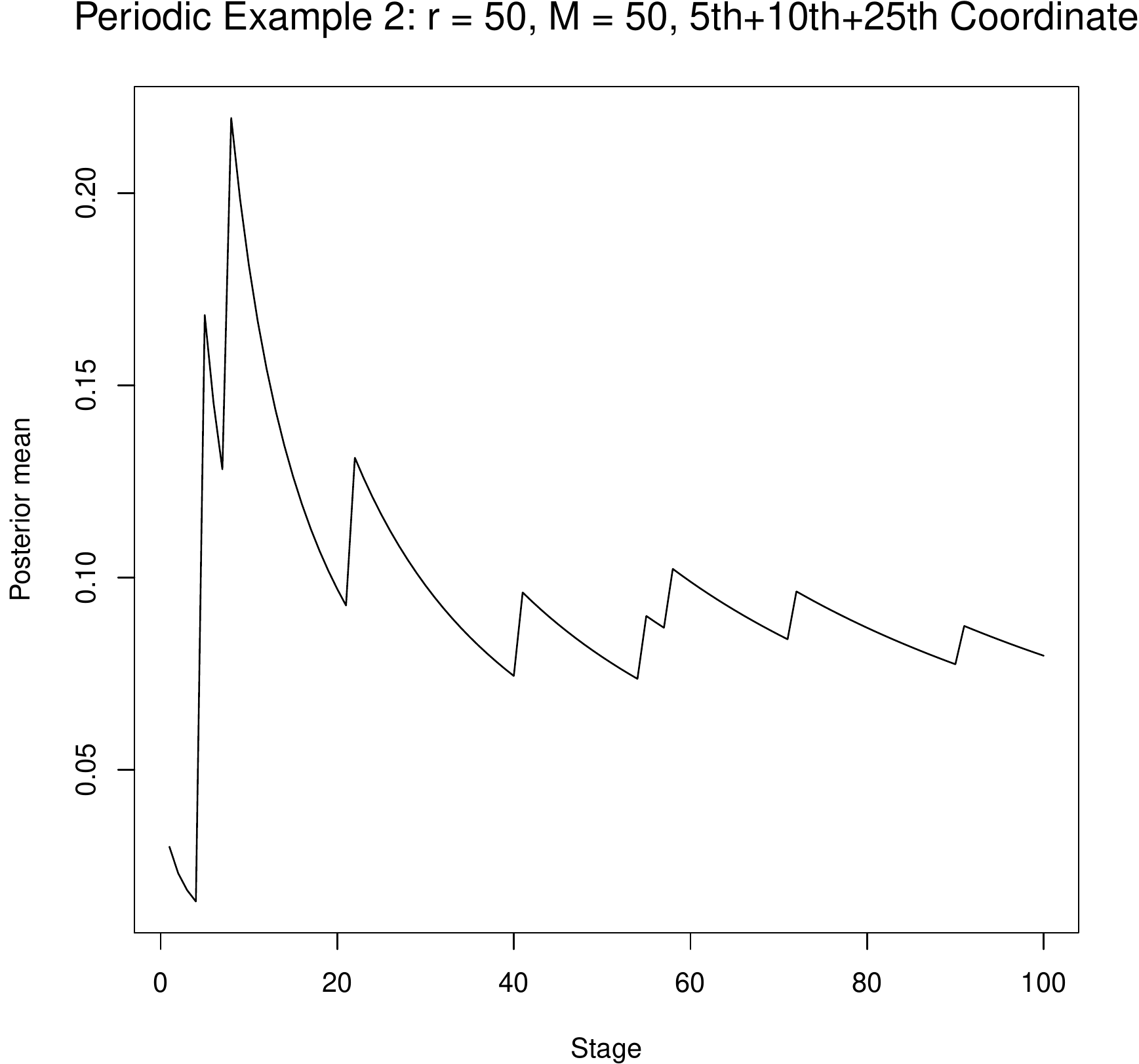}}\\
\vspace{2mm}
\subfigure [$r=50,M=100$. True frequency $=0.4$.]{ \label{fig:mult_osc_34}
\includegraphics[width=4.5cm,height=4.5cm]{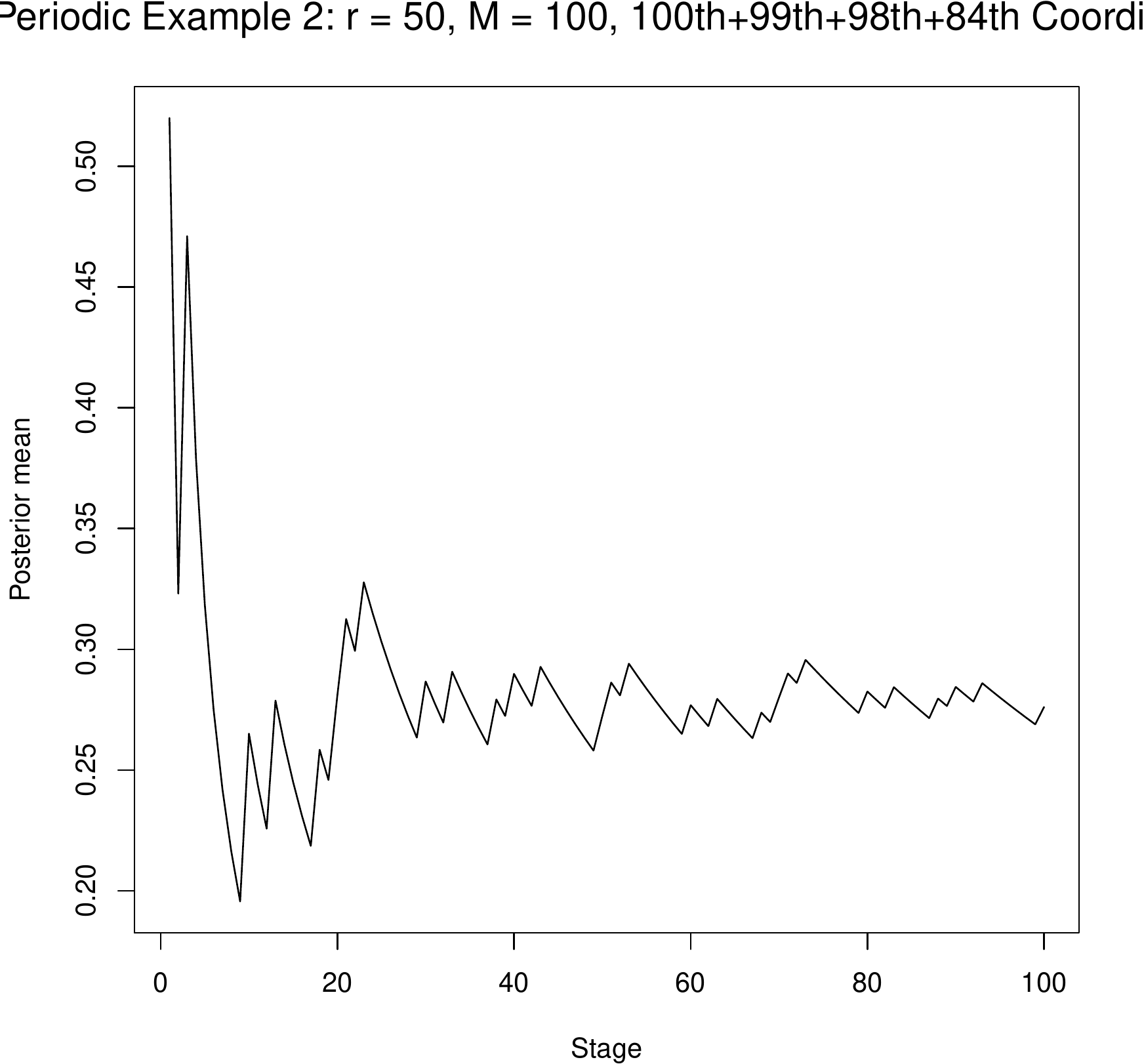}}
\hspace{2mm}
\subfigure [$r=50,M=100$. True frequency $=0.1$.]{ \label{fig:mult_osc_35}
\includegraphics[width=4.5cm,height=4.5cm]{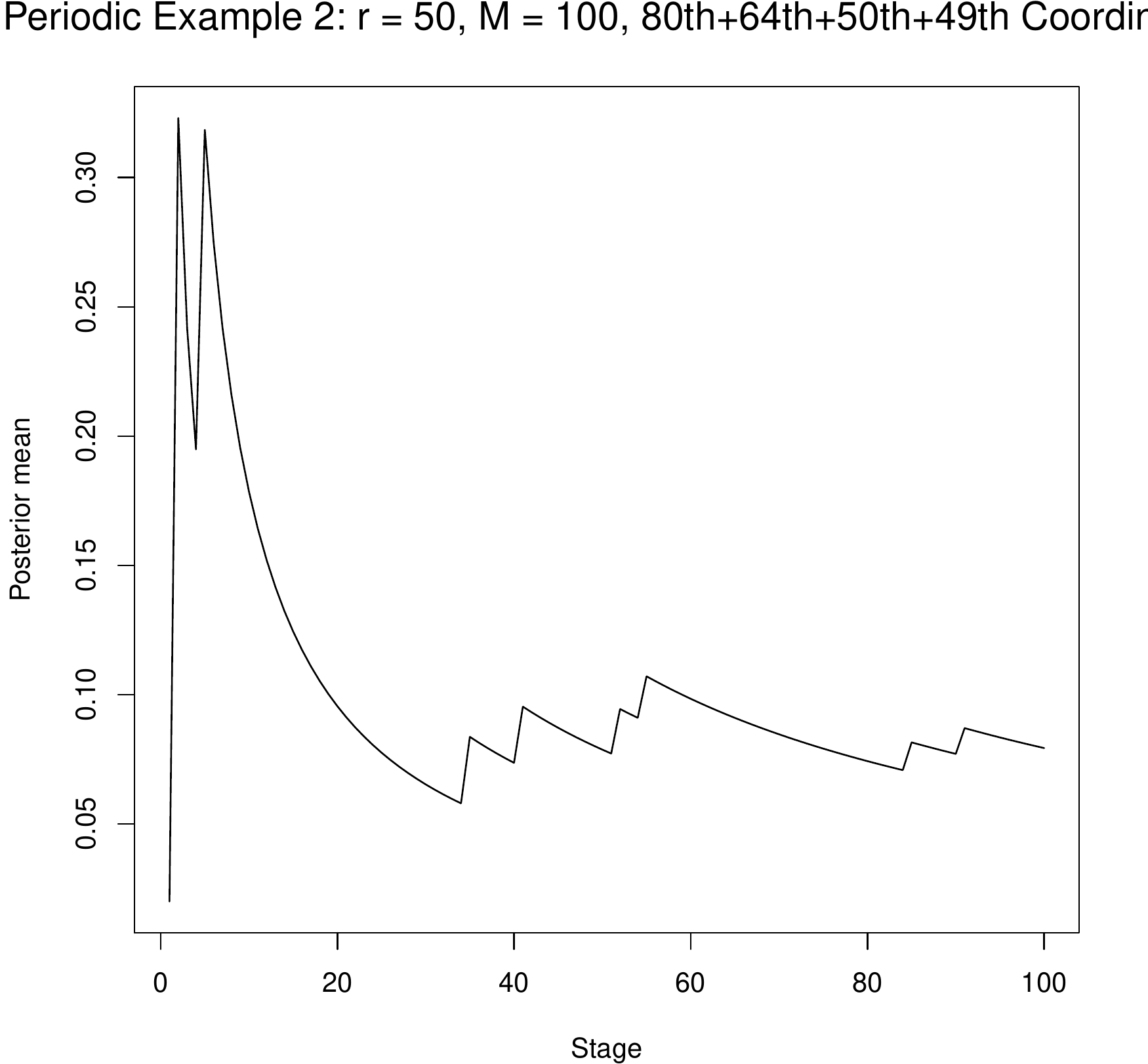}}
\hspace{2mm}
\subfigure [$r=50,M=100$. True frequency $=0.06$.]{ \label{fig:mult_osc_36}
\includegraphics[width=4.5cm,height=4.5cm]{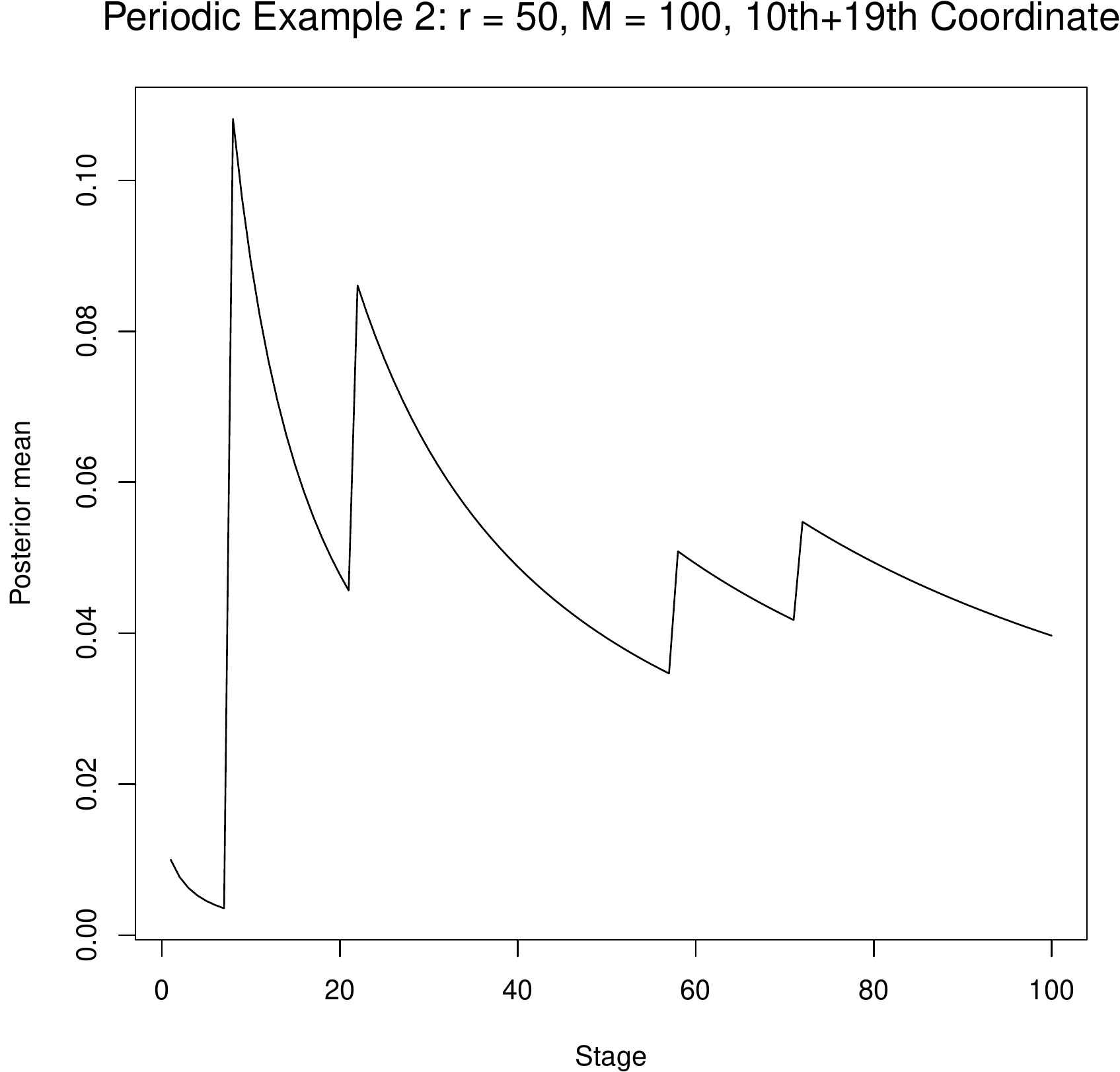}}
\caption{Illustration of our Bayesian method for determining multiple frequencies. Here the true frequencies are $0.4$, $0.1$ and $0.06$. }
\label{fig:mult_osc_example4}
\end{figure}

\begin{figure}
\centering
\subfigure [$r=100,M=10$. True frequency $=0.4$.]{ \label{fig:mult_osc_37}
\includegraphics[width=4.5cm,height=4.5cm]{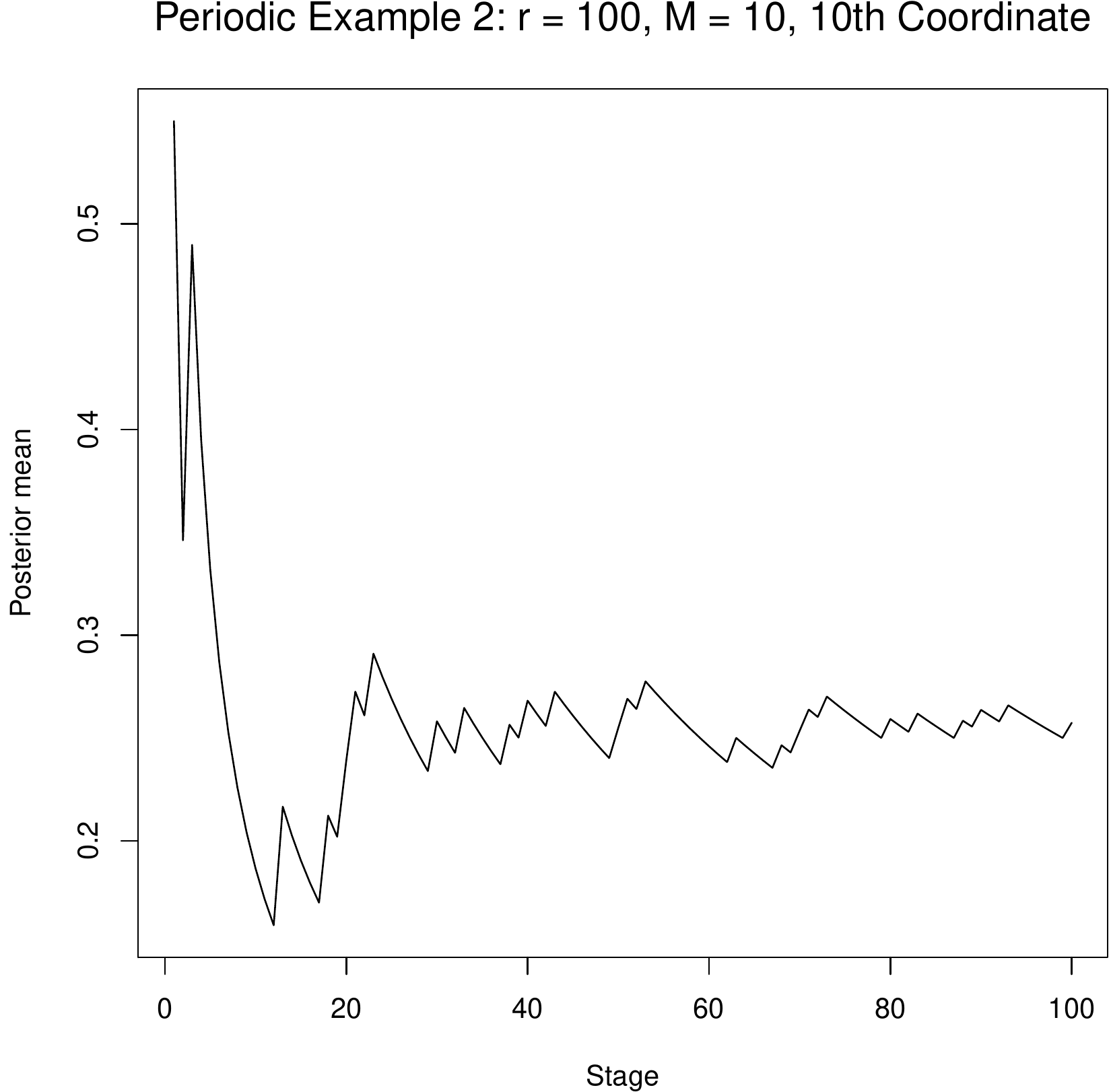}}
\hspace{2mm}
\subfigure [$r=100,M=10$. True frequency $=0.1$.]{ \label{fig:mult_osc_38}
\includegraphics[width=4.5cm,height=4.5cm]{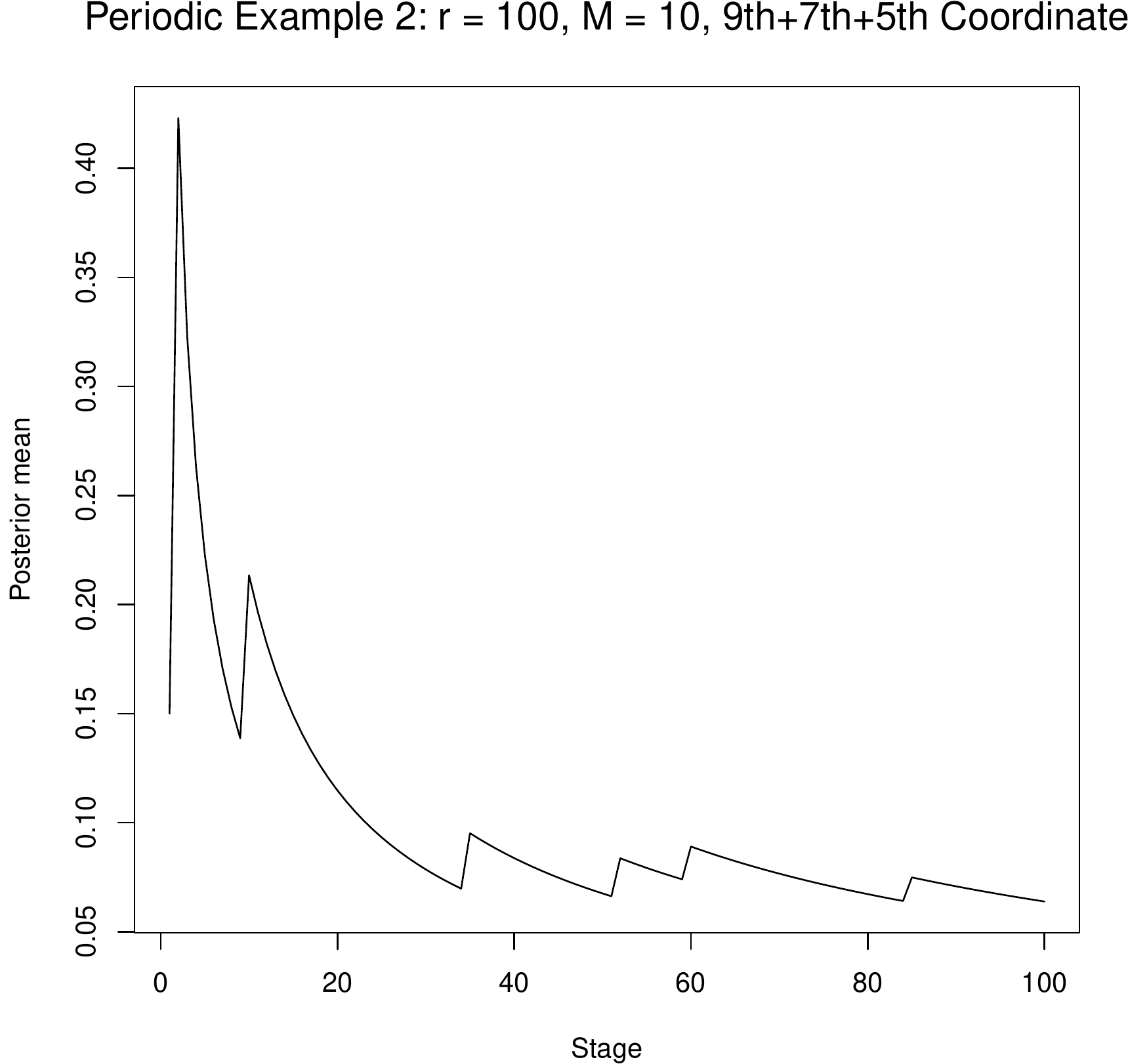}}
\hspace{2mm}
\subfigure [$r=100,M=10$. True frequency $=0.06$.]{ \label{fig:mult_osc_39}
\includegraphics[width=4.5cm,height=4.5cm]{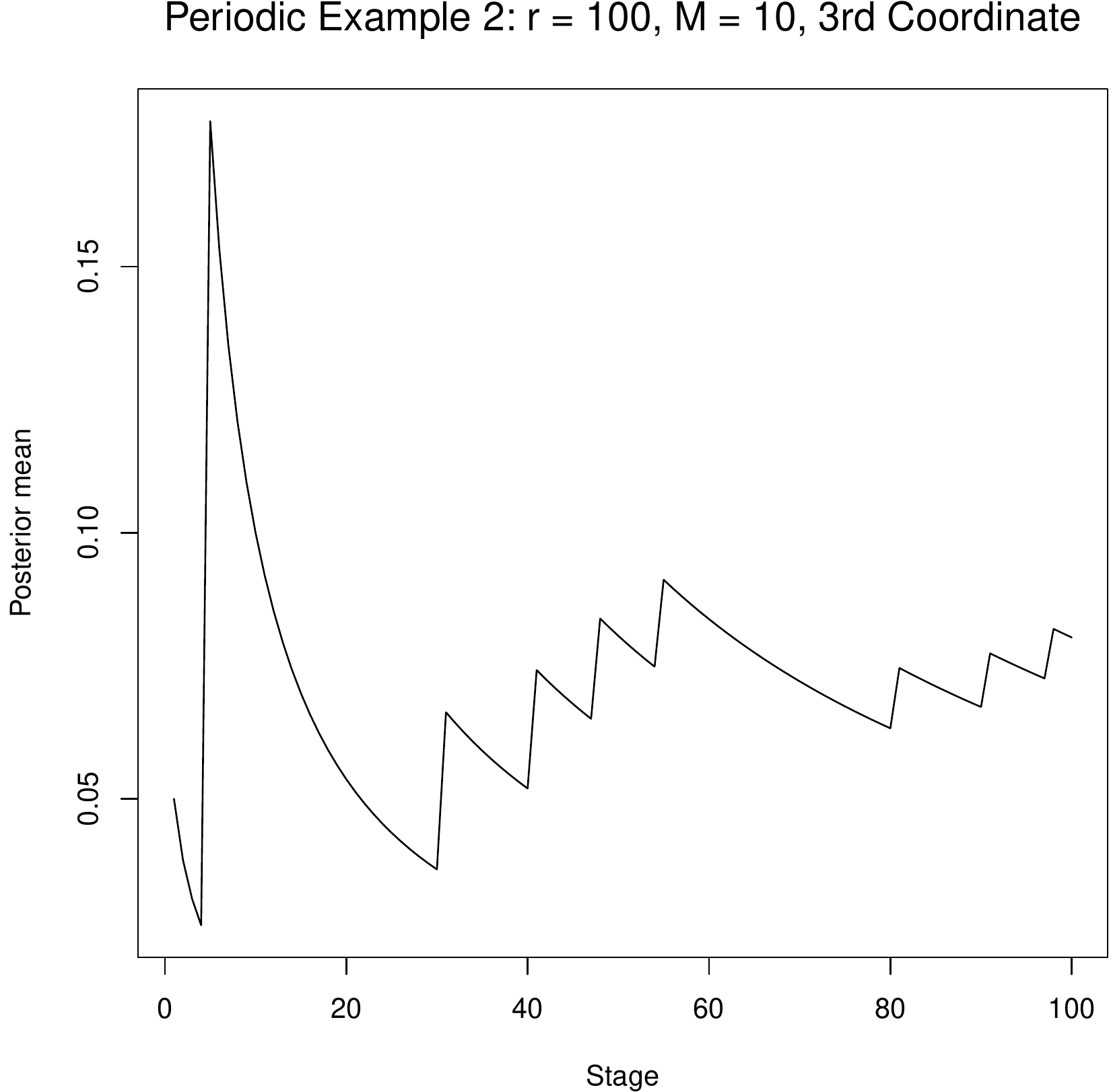}}\\
\vspace{2mm}
\subfigure [$r=100,M=50$. True frequency $=0.4$.]{ \label{fig:mult_osc_40}
\includegraphics[width=4.5cm,height=4.5cm]{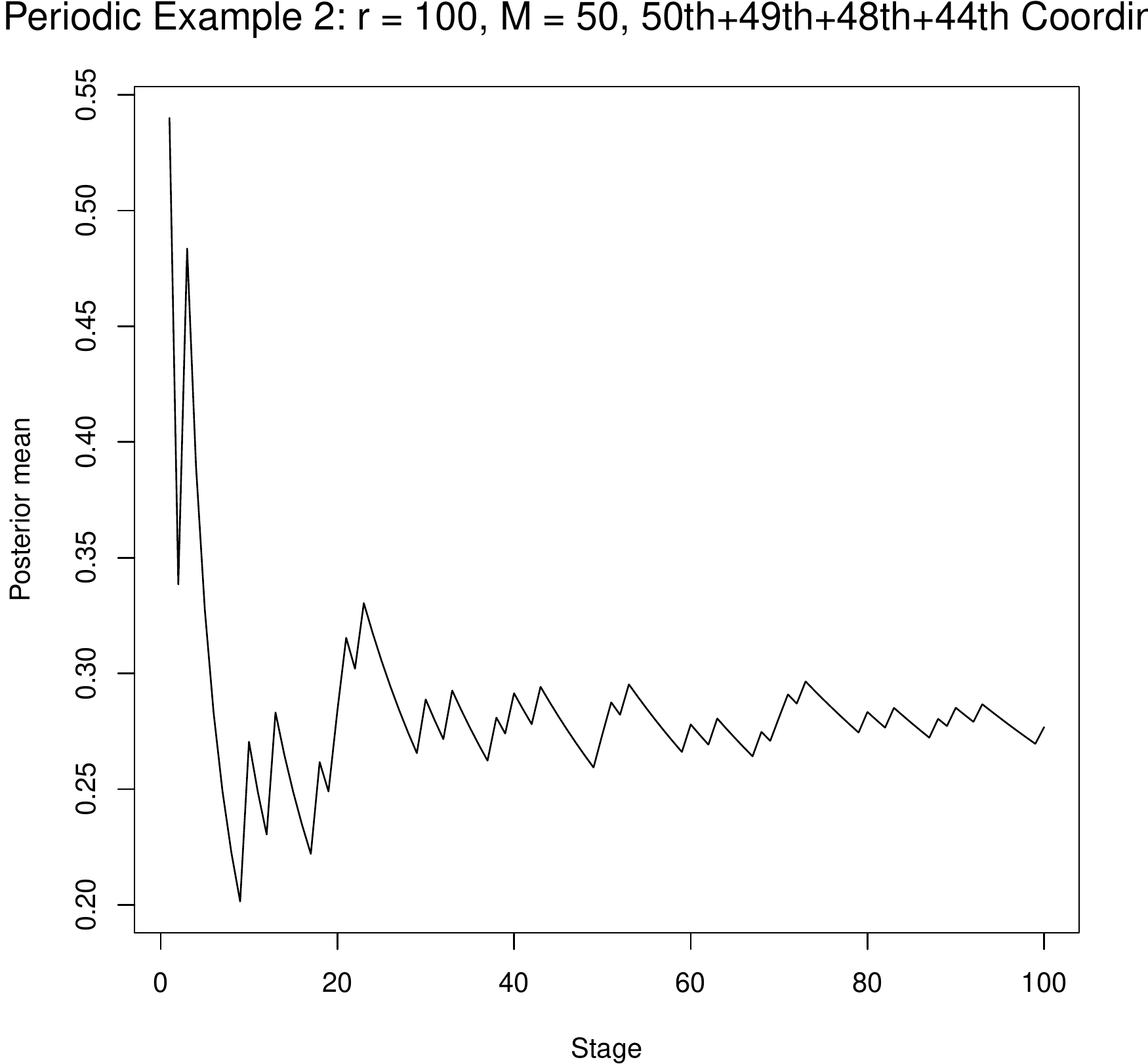}}
\hspace{2mm}
\subfigure [$r=100,M=50$. True frequency $=0.1$.]{ \label{fig:mult_osc_41}
\includegraphics[width=4.5cm,height=4.5cm]{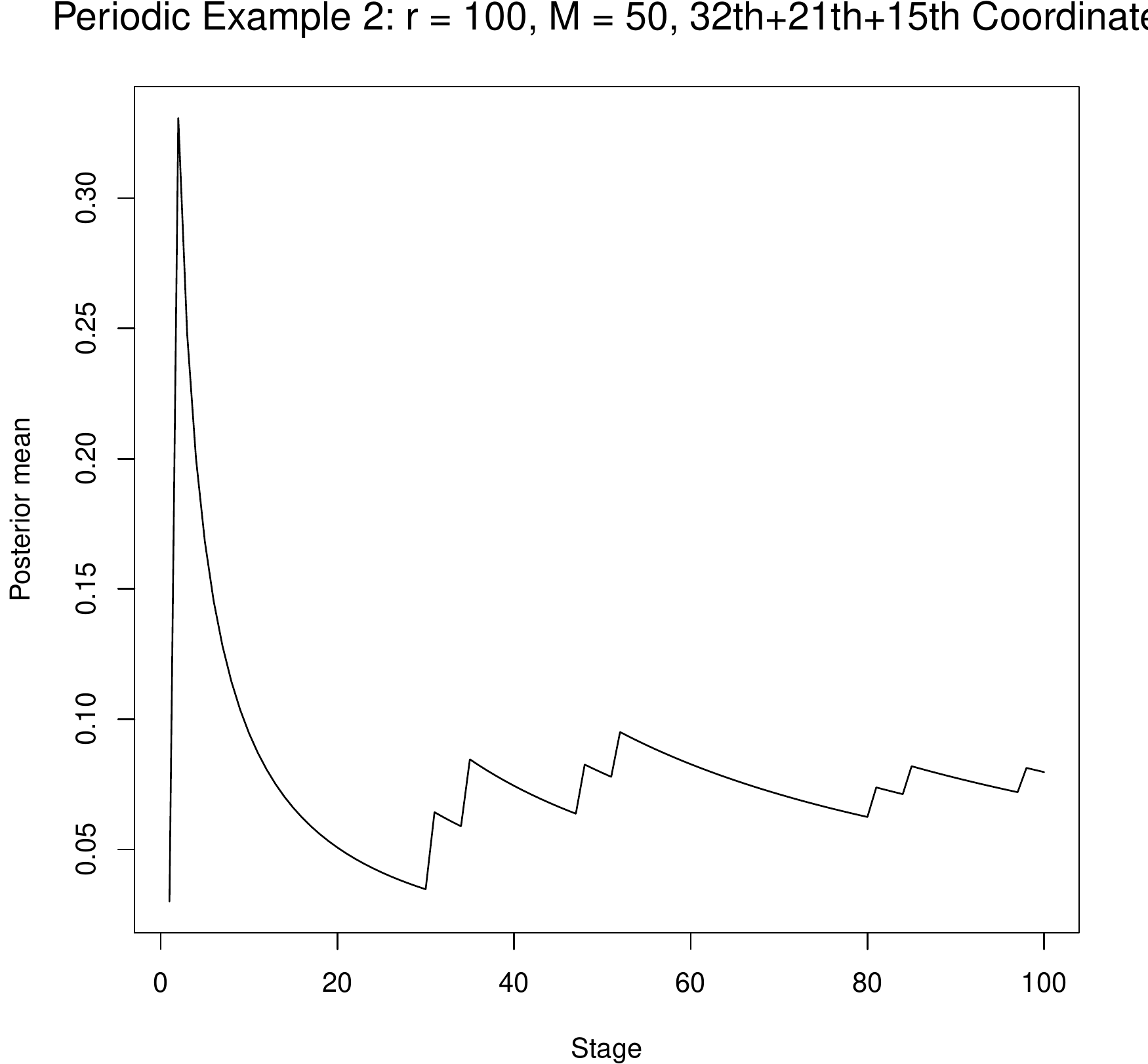}}
\hspace{2mm}
\subfigure [$r=100,M=50$. True frequency $=0.06$.]{ \label{fig:mult_osc_42}
\includegraphics[width=4.5cm,height=4.5cm]{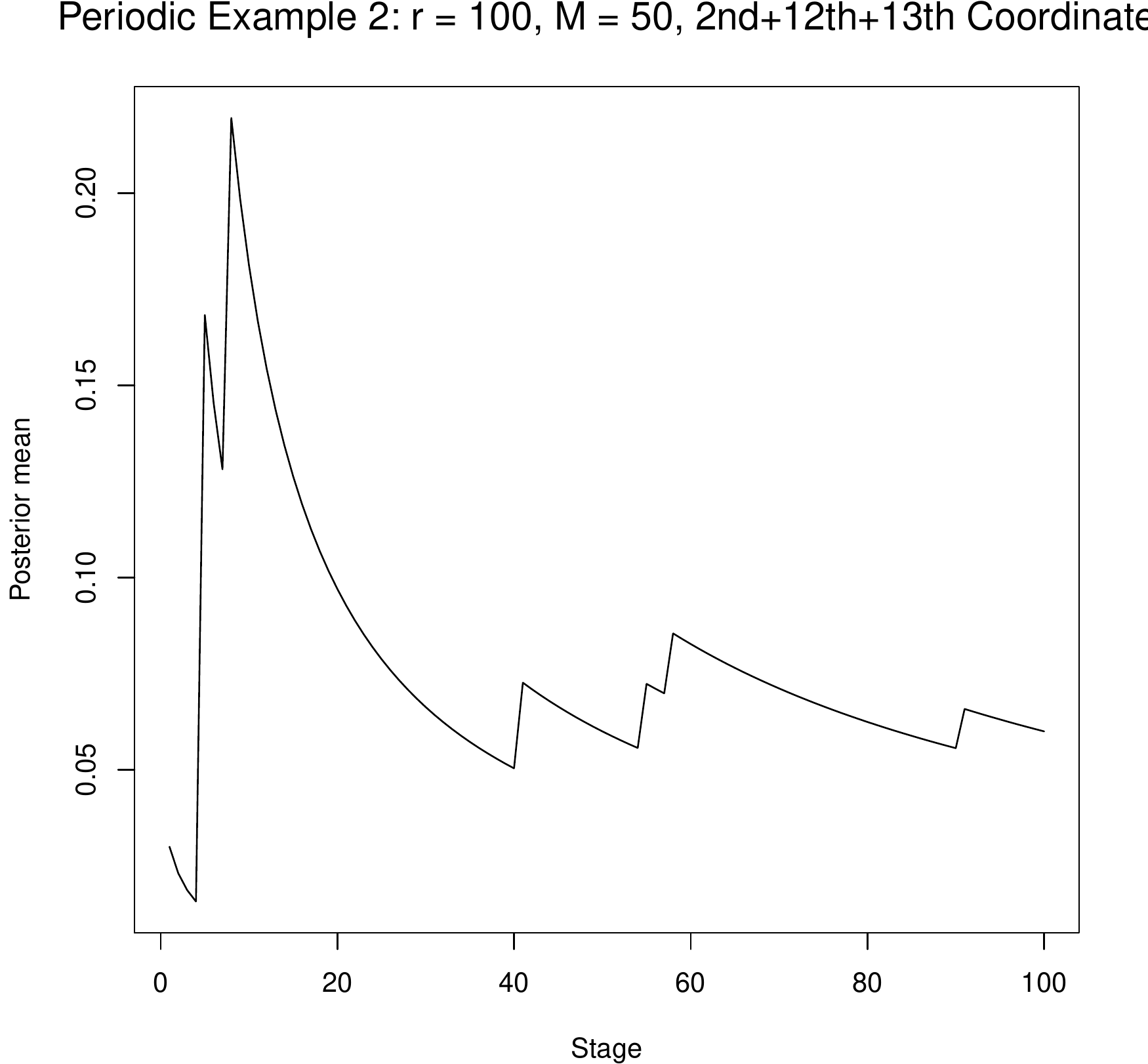}}\\
\vspace{2mm}
\subfigure [$r=100,M=100$. True frequency $=0.4$.]{ \label{fig:mult_osc_43}
\includegraphics[width=4.5cm,height=4.5cm]{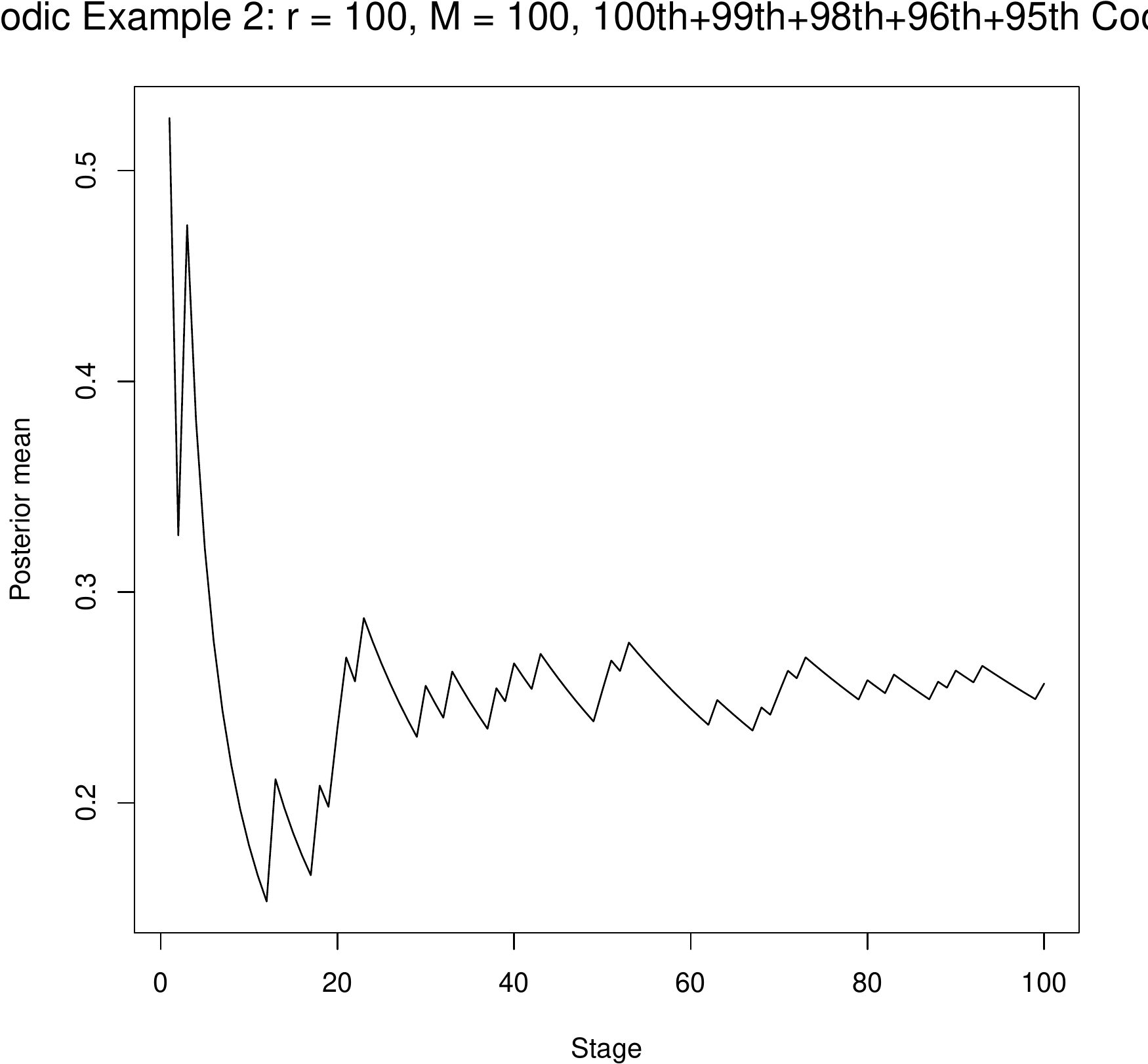}}
\hspace{2mm}
\subfigure [$r=100,M=100$. True frequency $=0.1$.]{ \label{fig:mult_osc_44}
\includegraphics[width=4.5cm,height=4.5cm]{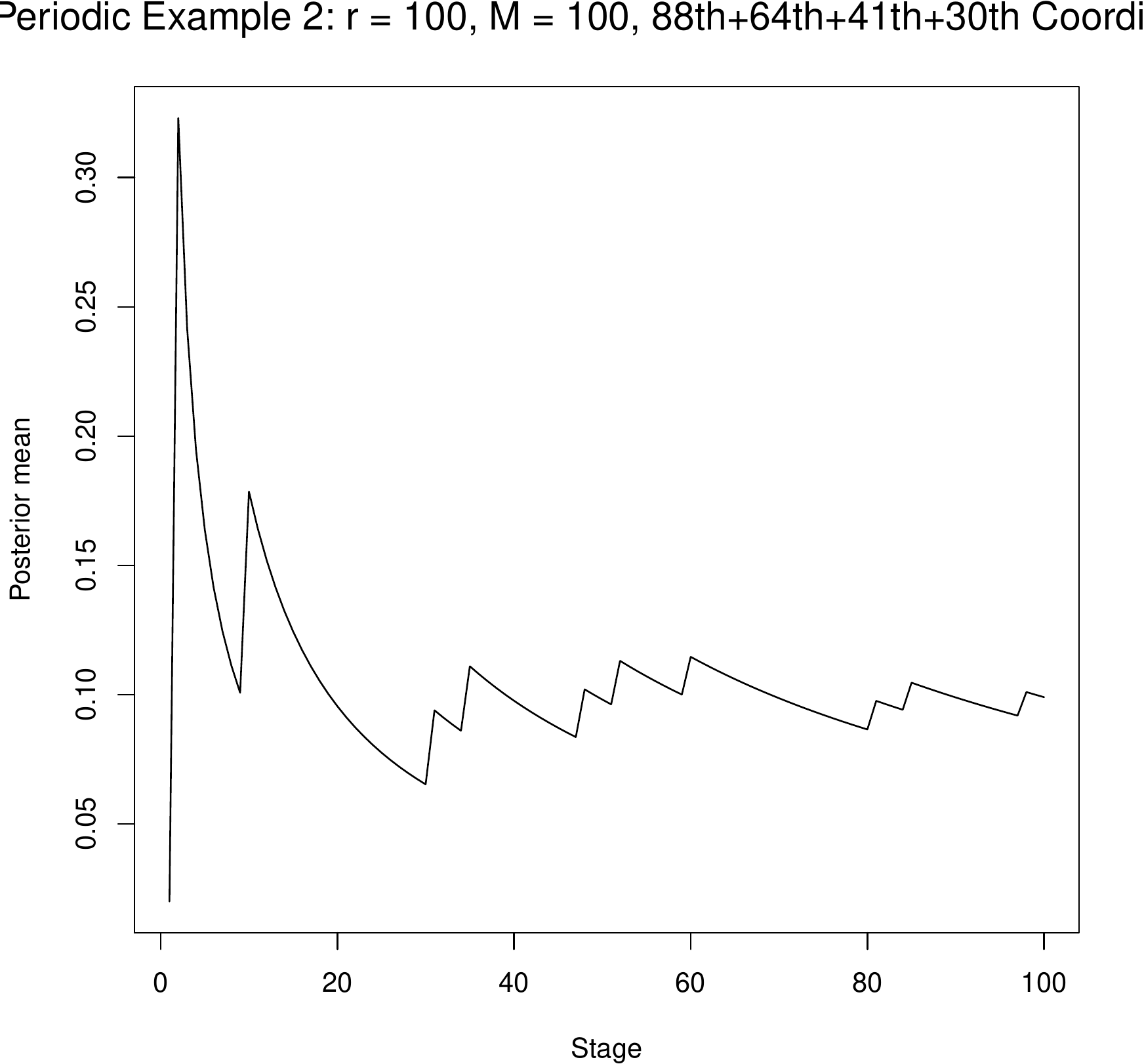}}
\hspace{2mm}
\subfigure [$r=100,M=100$. True frequency $=0.06$.]{ \label{fig:mult_osc_45}
\includegraphics[width=4.5cm,height=4.5cm]{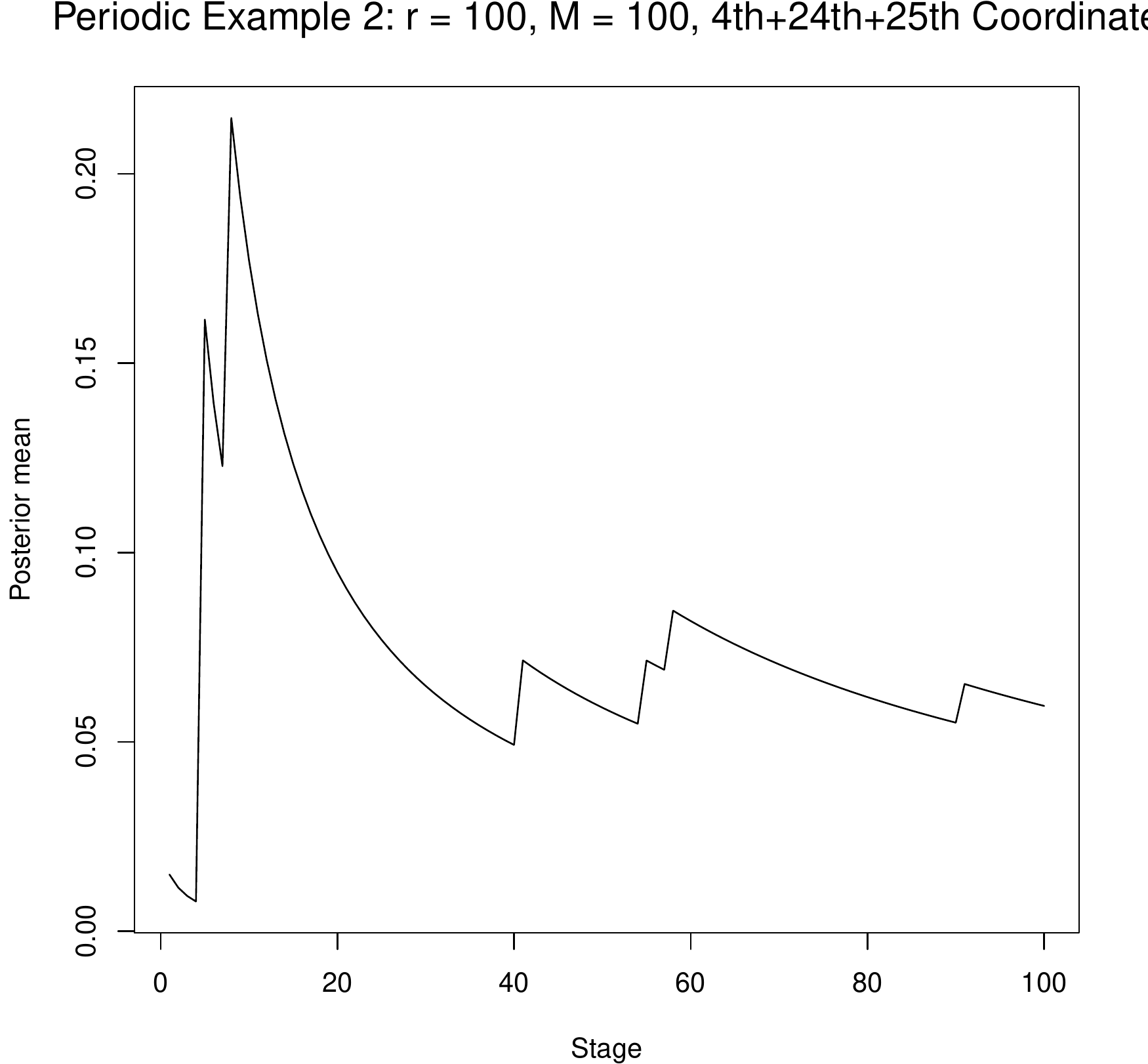}}
\caption{Illustration of our Bayesian method for determining multiple frequencies. Here the true frequencies are $0.4$, $0.1$ and $0.06$. }
\label{fig:mult_osc_example5}
\end{figure}

\subsection{Real data example: El Ni\~{n}o and fish population}
\label{subsec:soi_rec}
Based on data provided by Dr. Roy Mendelssohn of the Pacific Environmental Fisheries Group, \ctn{Shumway06} analyse two oscillating time series on
monthly values of an environmental series called the Southern Oscillation Index (SOI) and associated Recruitment (number of new fish), available
for a period of 453 months, ranging over the years 1950--1987. The plots are provided in \ctn{Shumway06}; see also panel (a) of Figure \ref{fig:soi1}
and panel (a) of Figure \ref{fig:rec1}. The quantity SOI is a measurement of air pressure change associated with
sea surface temperatures in the central Pacific Ocean. The El Ni\~{n}o effect is considered to cause warming of the central Pacific every three to seven years,
which is turn, is presumed to be responsible for causing floods in the midwestern portions of the United States in the year 1997. It is thus important to
identify the frequency of oscillation of the SOI series and the associated dependent Recruitment series, which seem to have slightly slower frequency of
oscillation in comparison to the SOI series. At first glance, both the series seem to have two significant frequencies of oscillations. For instance,
the Recruitment series seems to oscillate once in every 12 months and also once in every 50 months. Slightly faster frequencies can be expected
of the SOI series. The periodogram analyses provided in \ctn{Shumway06} indeed give weight to these frequencies.

We now apply our Bayesian method to investigate the frequencies hidden in the two underlying time series.
Although the two series seem to be dependent, we consider their analyses one by one. In the case of dependence, the frequencies in this situation are
expected to be close. 

\subsubsection{SOI series}
\label{subsubsec:soi}
We first take up the case of the SOI series, centering it first to remove any possible trend. 
Denoting the centered series by $X_t$, for our purpose, we need to consider a transformation
of the series to $Z^r_t$, with $Z^t=\exp\left(X_t\right)/\left(1+\exp\left(X_t\right)\right)$. We choose $r~(>0)$ such that the 
oscillations in the process $\bZ^r=\left\{Z^r_t\right\}$ become as explicit as possible. With $r=10$, this goal seems to be achieved. The original
SOI time series and the transformed time series $\bZ^{10}$ are shown in Figure \ref{fig:soi1}.

\begin{figure}
\centering
\subfigure [The original SOI time series.]{ \label{fig:soi_original1}
\includegraphics[width=10cm,height=6cm]{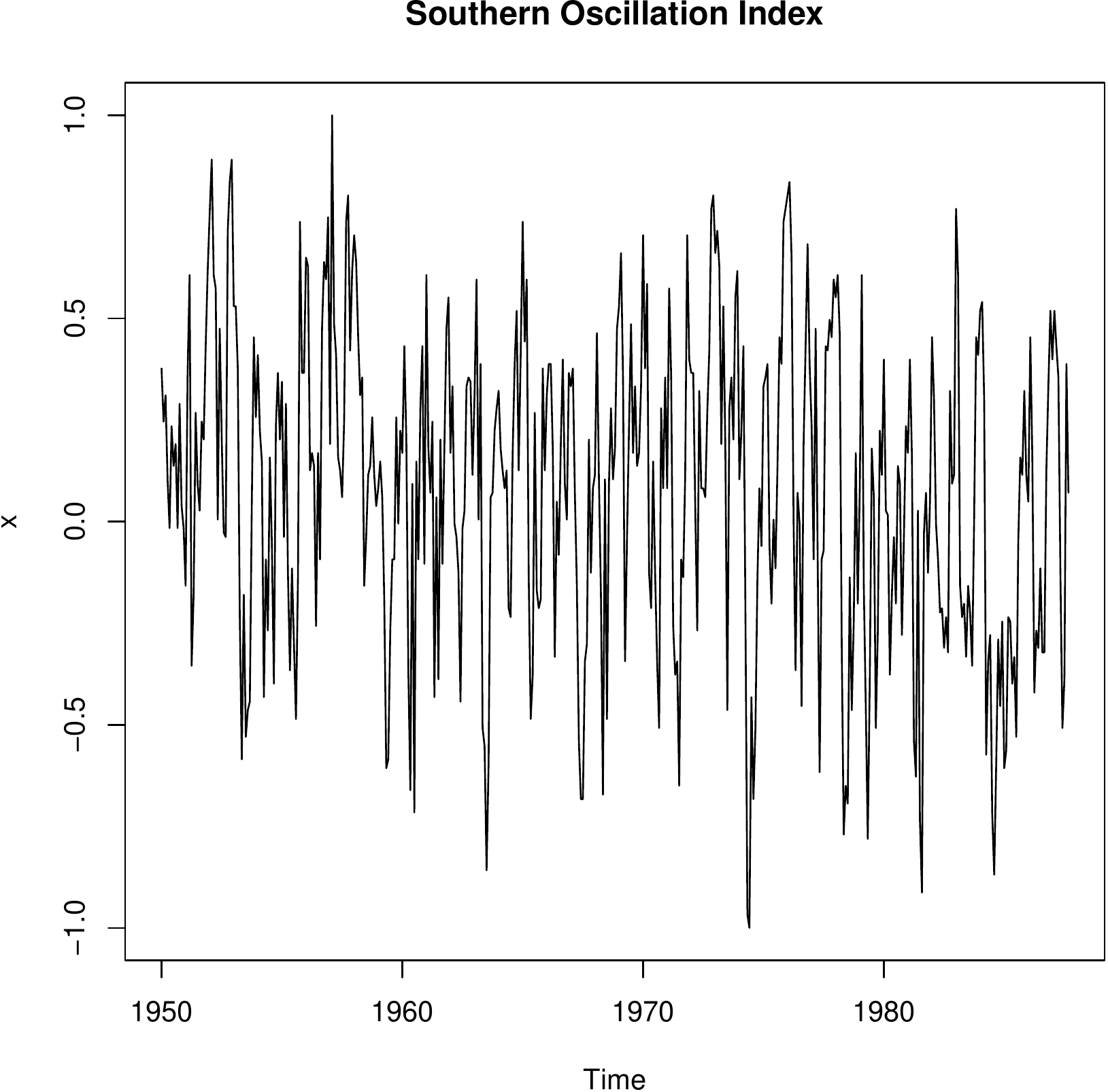}}
\vspace{2mm}
\subfigure [The transformed SOI time series.]{ \label{fig:soi_transf1}
\includegraphics[width=10cm,height=6cm]{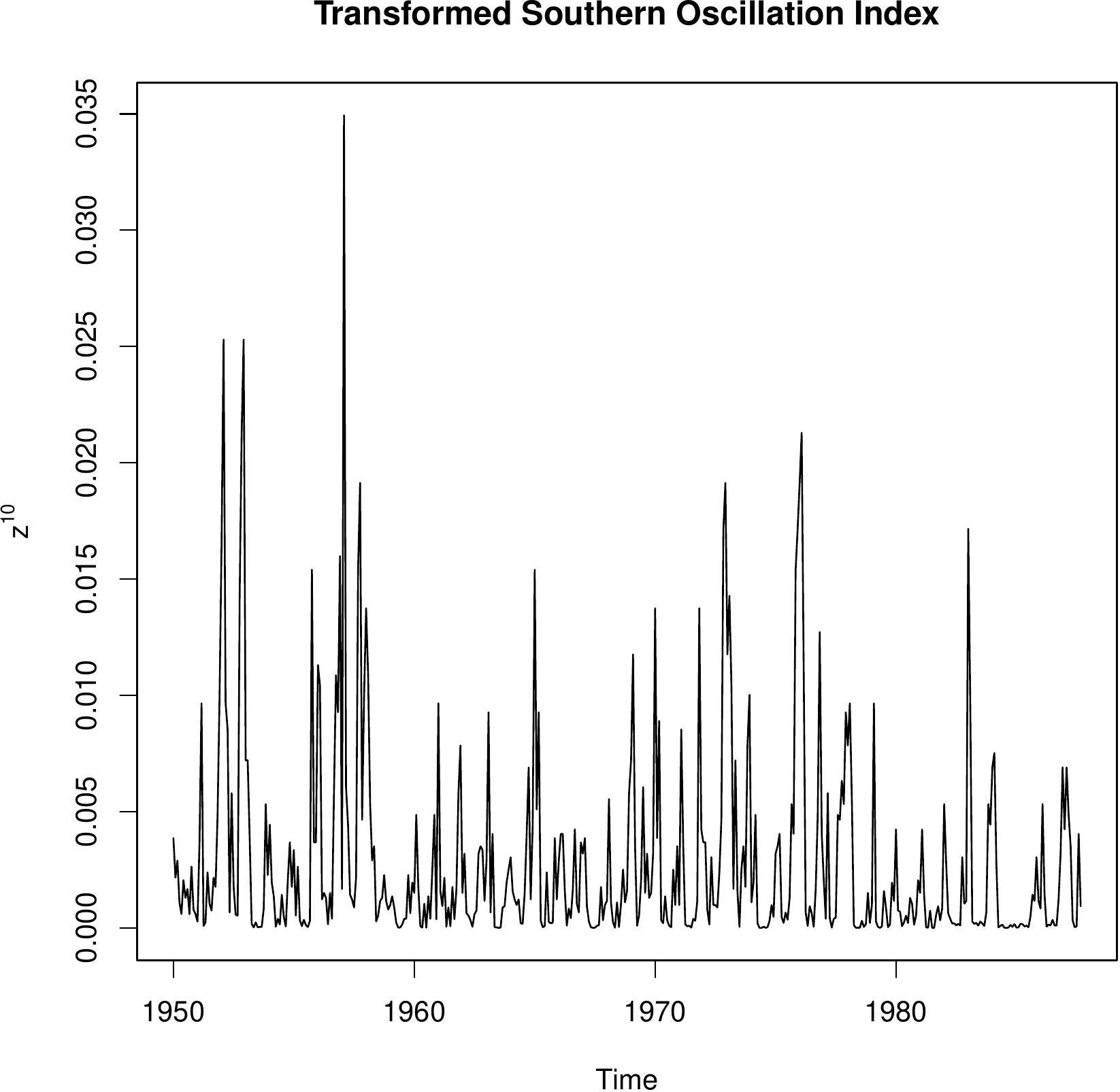}}
\caption{The original and the transformed SOI time series.}
\label{fig:soi1}
\end{figure}

The range of Figure \ref{fig:soi_transf1} reveals that a very fine partition of the interval $[0,1]$ is necessary in order to capture the hidden frequencies.
As such, we set $M=5000$. We then implement our Dirichlet process based Bayesian method with $r=10$ and $M=5000$. 
Figure \ref{fig:soi_freq} shows the results of our implementation.
Panel (a) of the figure shows convergence of the relevant posterior of $p_{25,j}+p_{27,j}$ approximately to the frequency $0.02$, while panel (b) shows convergence 
of $p_{16,j}+p_{18,j}+p_{19,j}+p_{21,j}+p_{22,j}$ approximately to $0.08$. The fine partition of $[0,1]$ is the reason for dissipating of the
proportions to many intervals $(\tilde p_{m-1,j},\tilde p_{m,j}]$. Other than the aforementioned $p_{m,j}$'s contributing to the frequencies, 
the rest of the $p_{m,j}$'s, except $p_{1,j}$, converged to zero. Thus, our results are consistent with the periodogram analysis of \ctn{Shumway06}. 
\begin{figure}
\centering
\subfigure [SOI: converging frequency slightly exceeds $0.02$.]{ \label{fig:soi_freq1}
\includegraphics[width=6cm,height=6cm]{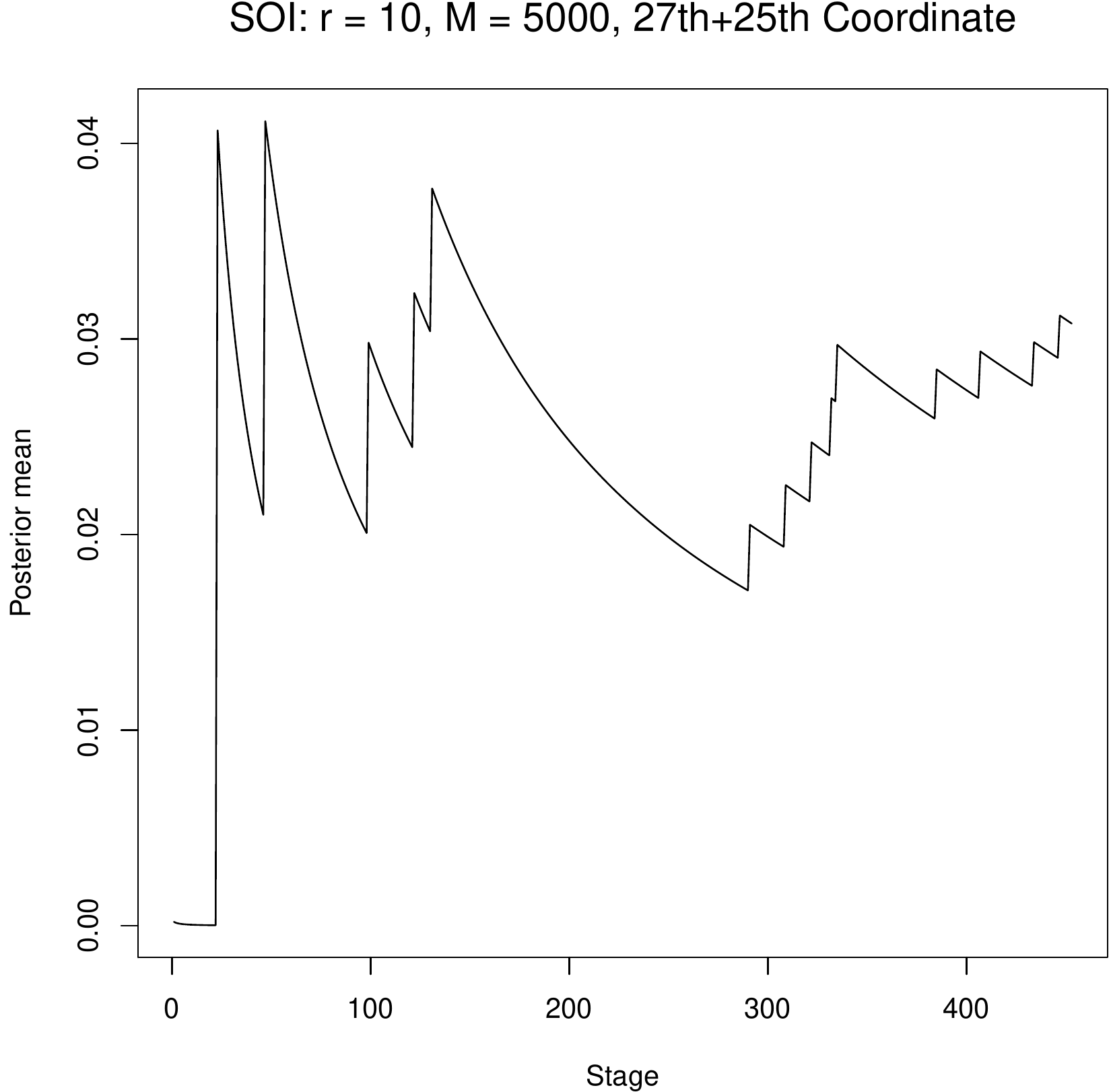}}
\hspace{2mm}
\subfigure [SOI: converging frequency slightly exceeds $0.08$.]{ \label{fig:soi_freq2}
\includegraphics[width=6cm,height=6cm]{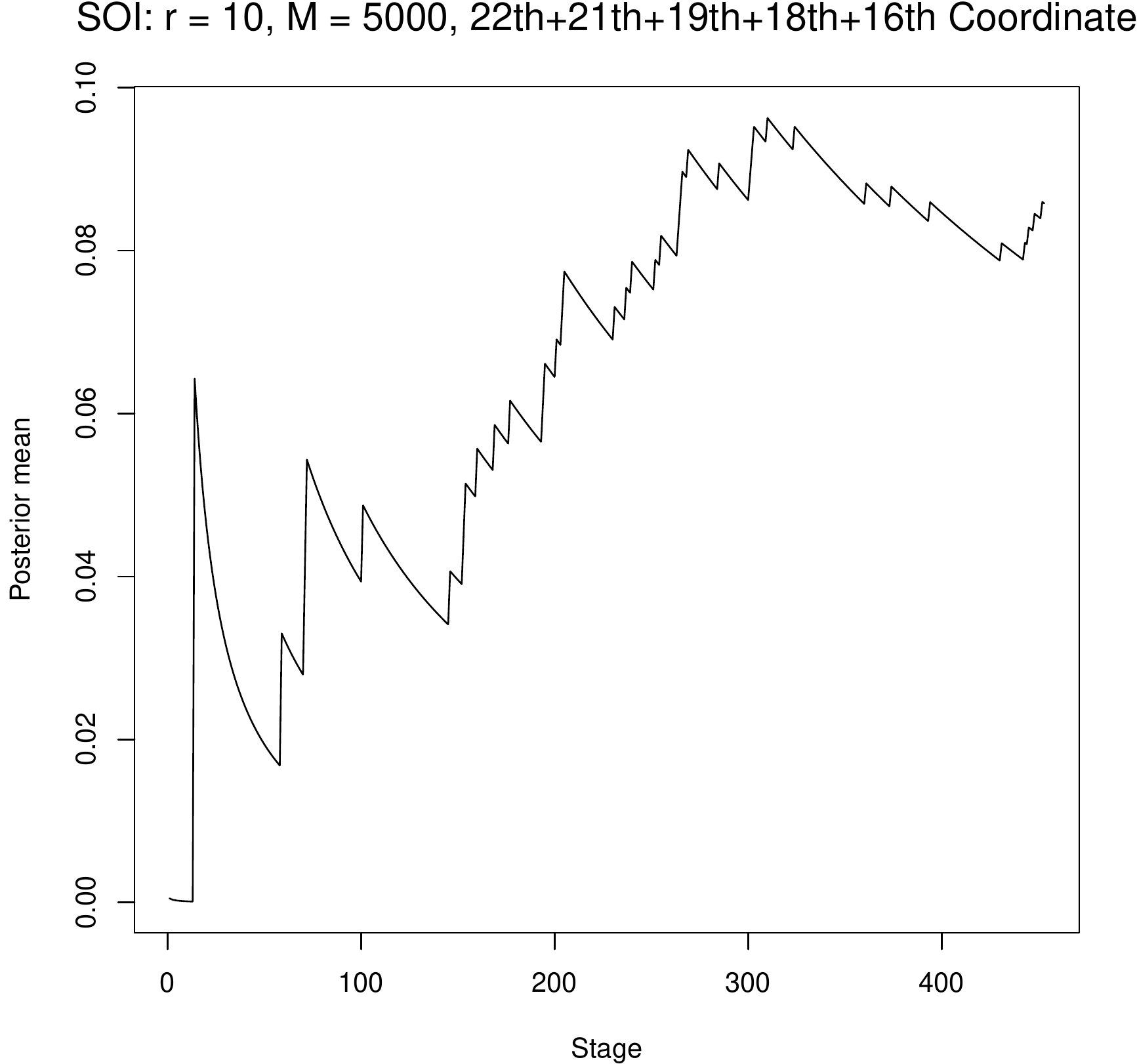}}
\caption{Bayesian results for frequency determination of the SOI time series.}
\label{fig:soi_freq}
\end{figure}

The above analysis requires very fine partition of $[0,1]$, using large values of $M$. This considerably increases the number of $p_{m,j}$s in the Bayesian model,
most of which do not contribute to frequency determination. Apart from being wasteful, this also slows down the implementation of the Bayesian code. Since the small range of the
transformed time series $\bZ^{10}$ is responsible for these issues, it makes sense to consider a transformation that increases the range, while preserving easy
visualization of the oscillations. In this particular example, simply multiplication of $\bZ^{10}$ by $10$ seems to have the desired effect. Figure \ref{fig:soi_transf_2}
shows the series $10\times\bZ^{10}$. Here, considering $M=1000$ turned out to be sufficient. Indeed, Figure \ref{fig:soi_freq_2} shows that the relevant frequencies
to which our Bayesian posteriors converged to, are consistent with those obtained for $\bZ^{10}$ and $M=5000$, and hence again approximately in keeping with
the periodogram analysis of \ctn{Shumway06}. Here we remark that the choice $M=1000$ is still somewhat large, but smaller values such as $100$ and $500$ did not yield
enough (almost) empty intervals $(\tilde p_{m-1},\tilde p_m]$ between the strings of intervals contributing significantly to the frequencies. Hence, these smaller choices
did not enable us to easily identify the different frequencies characterizing the SOI time series.
\begin{figure}
\centering
\includegraphics[width=10cm,height=6cm]{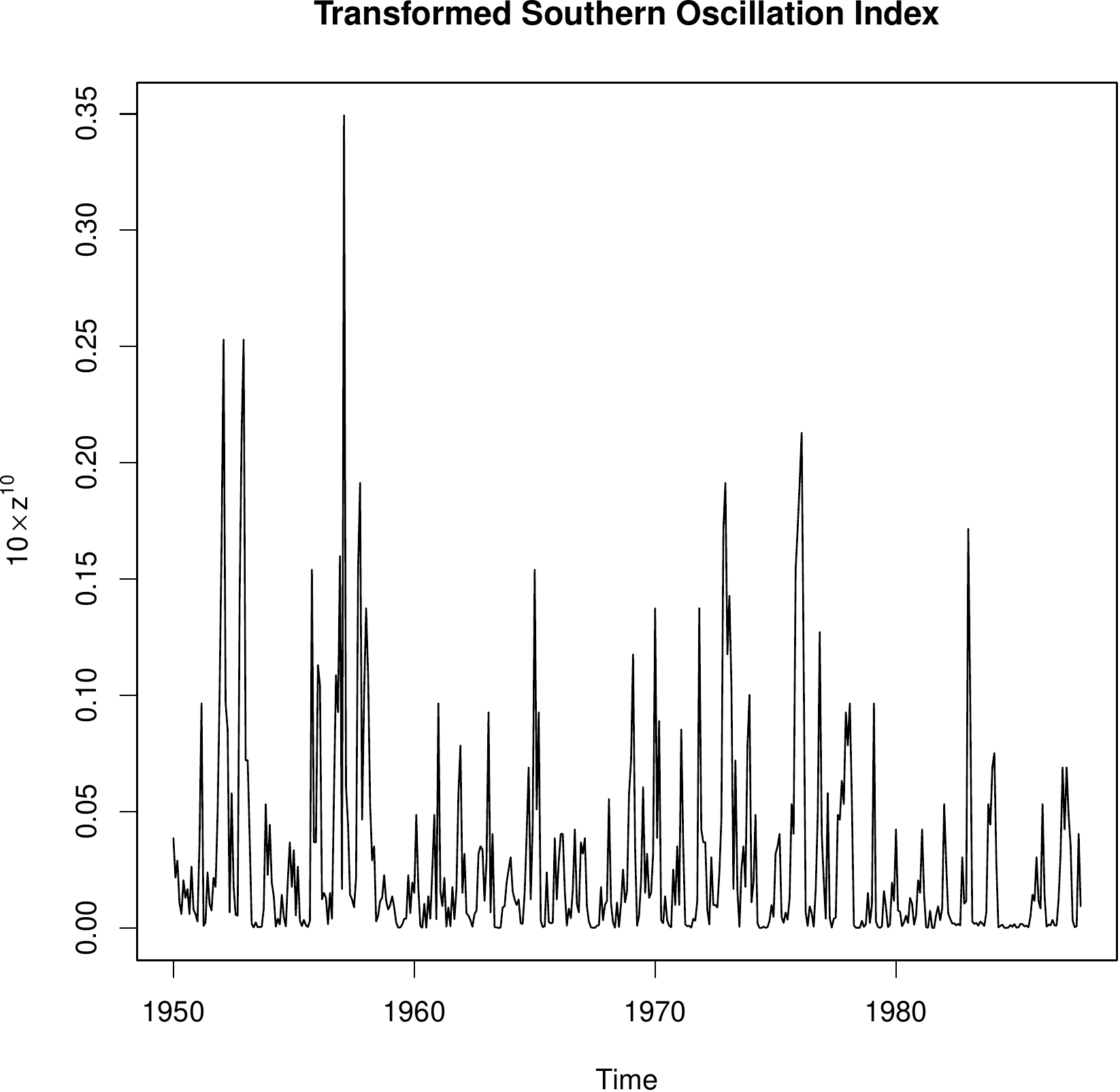}
\caption{The transformed SOI time series $10\times\bZ^{10}$.}
\label{fig:soi_transf_2}
\end{figure}

\begin{figure}
\centering
\subfigure [SOI transformation $10\times\bZ^{10}$: converging frequency slightly exceeds $0.02$.]{ \label{fig:soi_freq1_2}
\includegraphics[width=6cm,height=6cm]{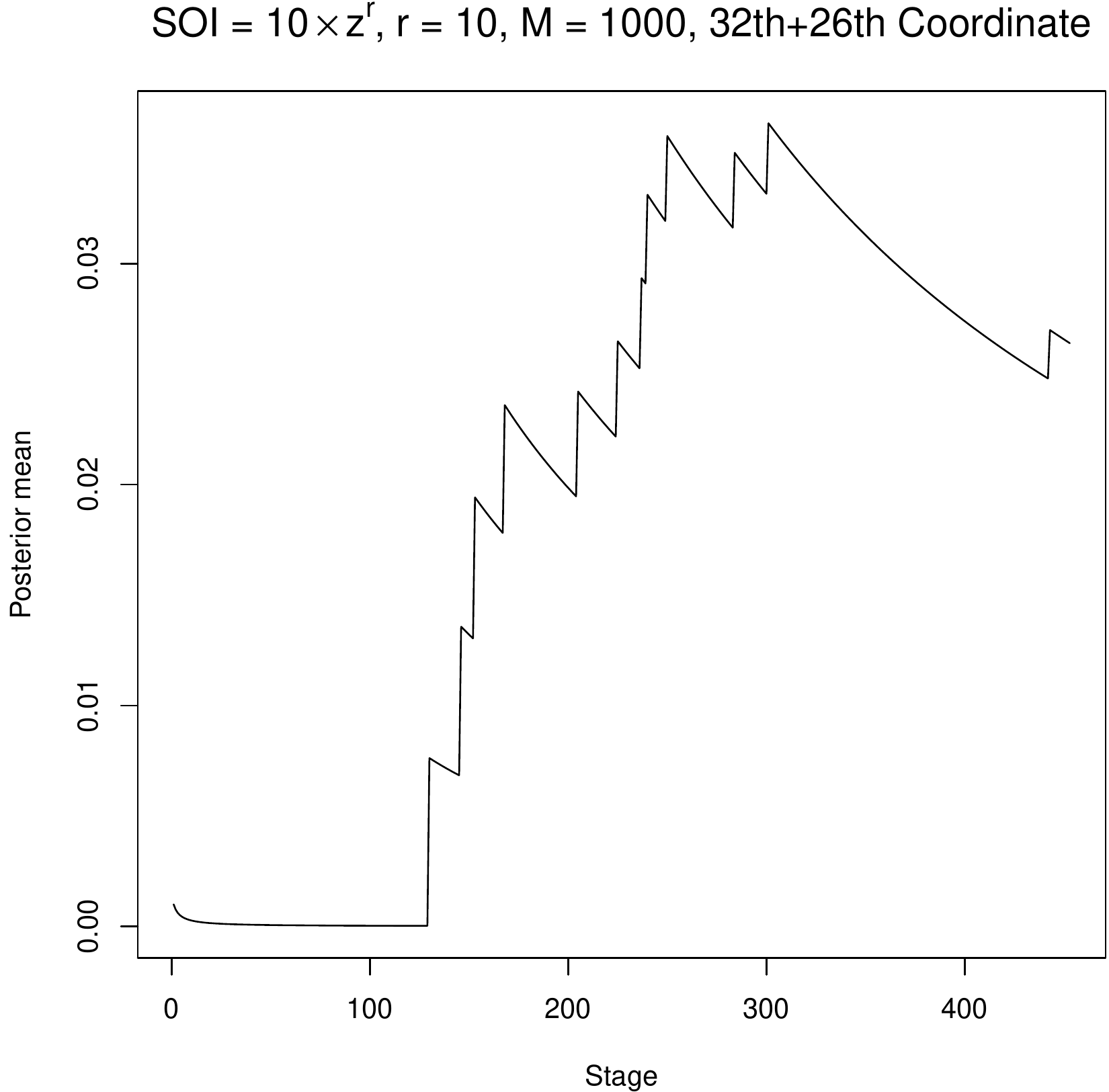}}
\hspace{2mm}
\subfigure [SOI transformation $10\times\bZ^{10}$: converging frequency slightly exceeds $0.08$.]{ \label{fig:soi_freq2_2}
\includegraphics[width=6cm,height=6cm]{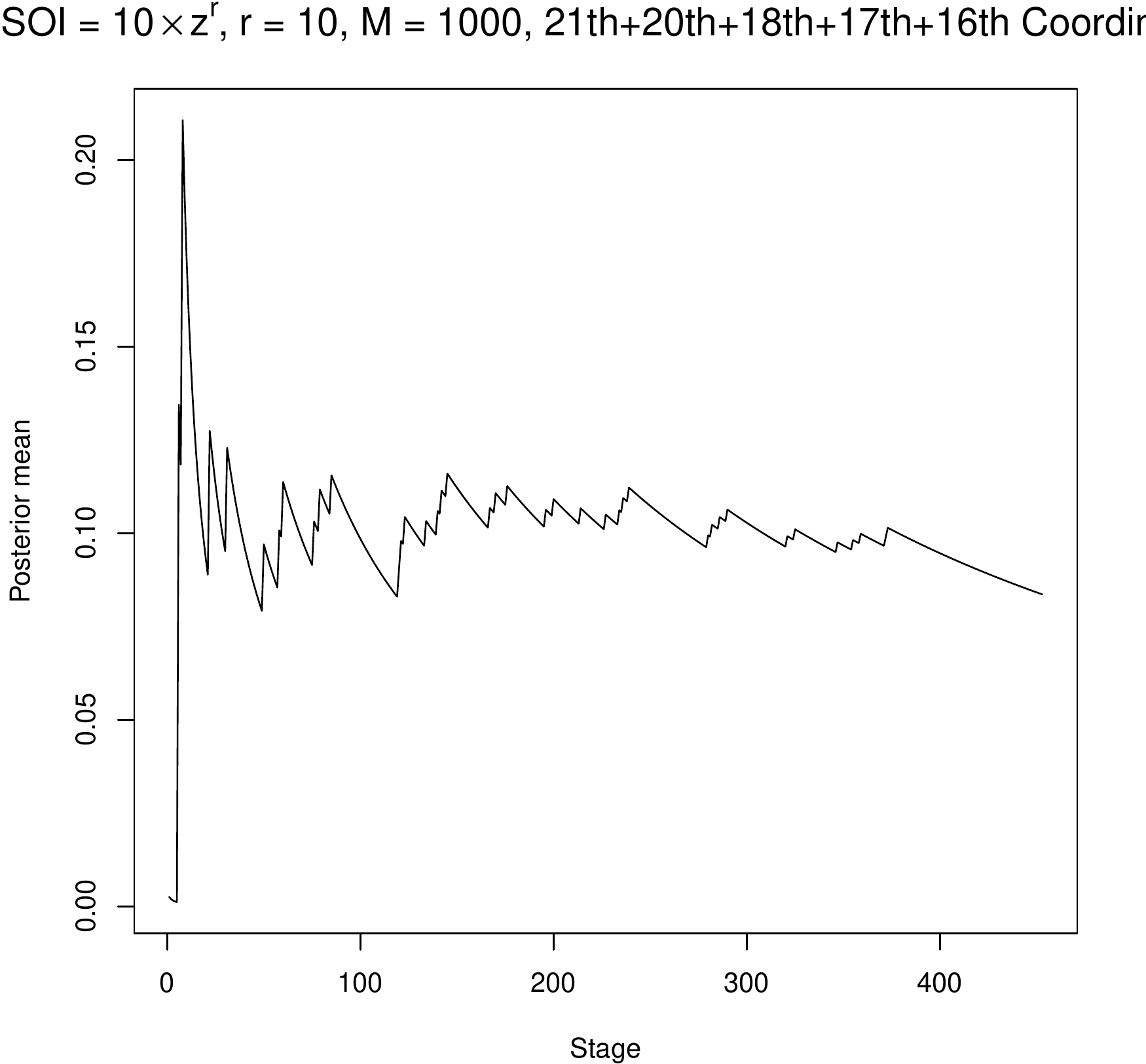}}
\caption{Bayesian results for frequency determination of the SOI time series with transformed time series $10\times\bZ^{10}$.}
\label{fig:soi_freq_2}
\end{figure}

We now turn to the Recruitment time series; as in SOI, we first center the time series. The original Recruitment series and the transformation
$\exp(25)\times\bZ^{50}$ are displayed in Figure \ref{fig:rec1}. This transformation enabled the most explicit visualization of the oscillations, among those
that we experimented with. The multiplicative factor $\exp(25)$ raises the range to a reasonable limit. We consider $M=1000$ for our Bayesian implementation
based on Dirichlet process.
Figure \ref{fig:rec_freq} depicts the posterior convergence path to the relevant frequencies. Note that the convergences in panel (a) occurs 
towards slightly larger than $0.02$, while that in panel (b) occurs around $0.08$. 
In contrast, for the SOI series, the convergences in panels (b) of Figure \ref{fig:soi_freq} and \ref{fig:soi_freq2} seemed to take place at values 
somewhat larger than $0.08$. 
\begin{figure}
\centering
\subfigure [The original Recruitment time series.]{ \label{fig:rec_original1}
\includegraphics[width=10cm,height=6cm]{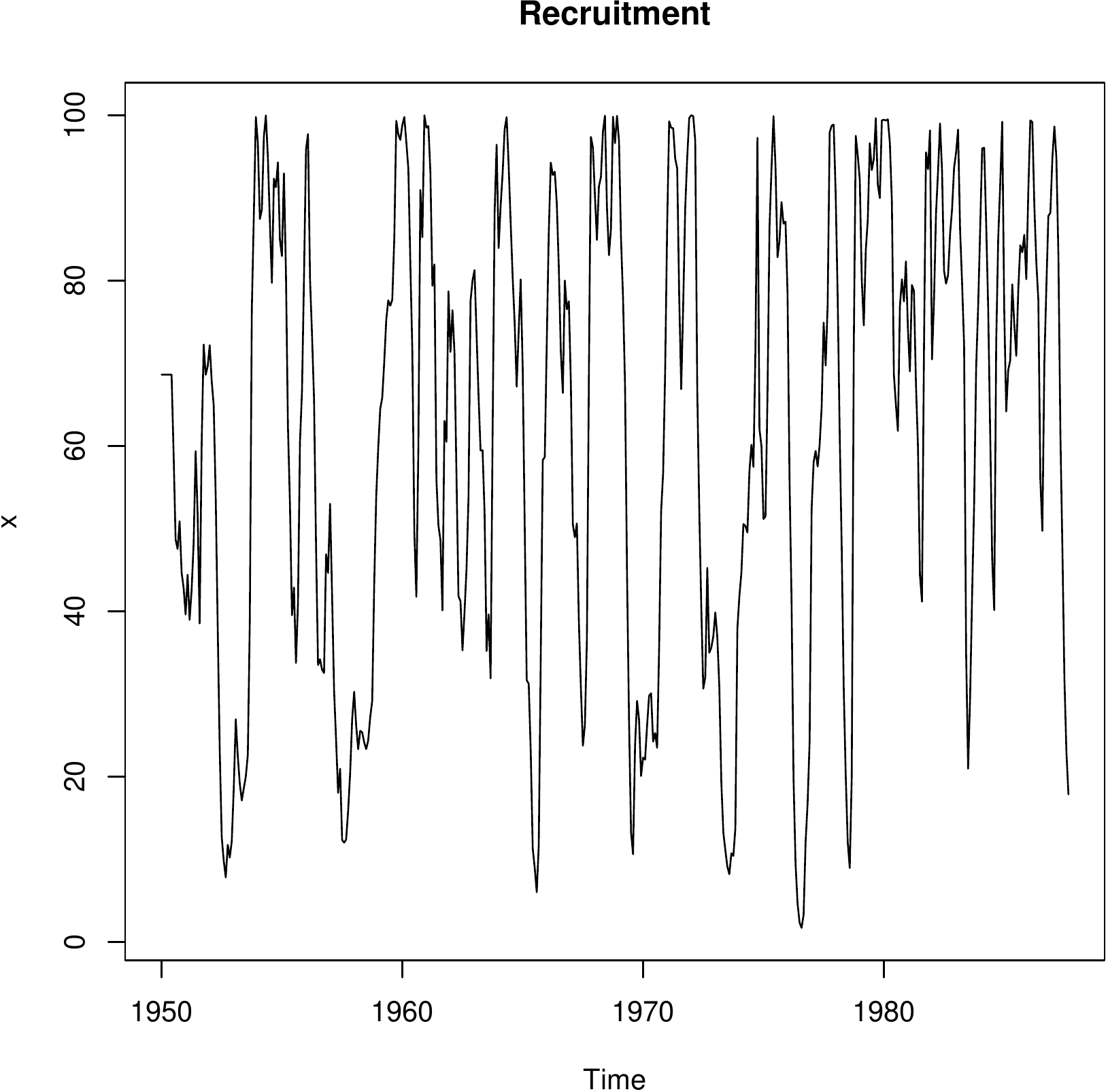}}
\vspace{2mm}
\subfigure [The transformed Recruitment time series.]{ \label{fig:rec_transf1}
\includegraphics[width=10cm,height=6cm]{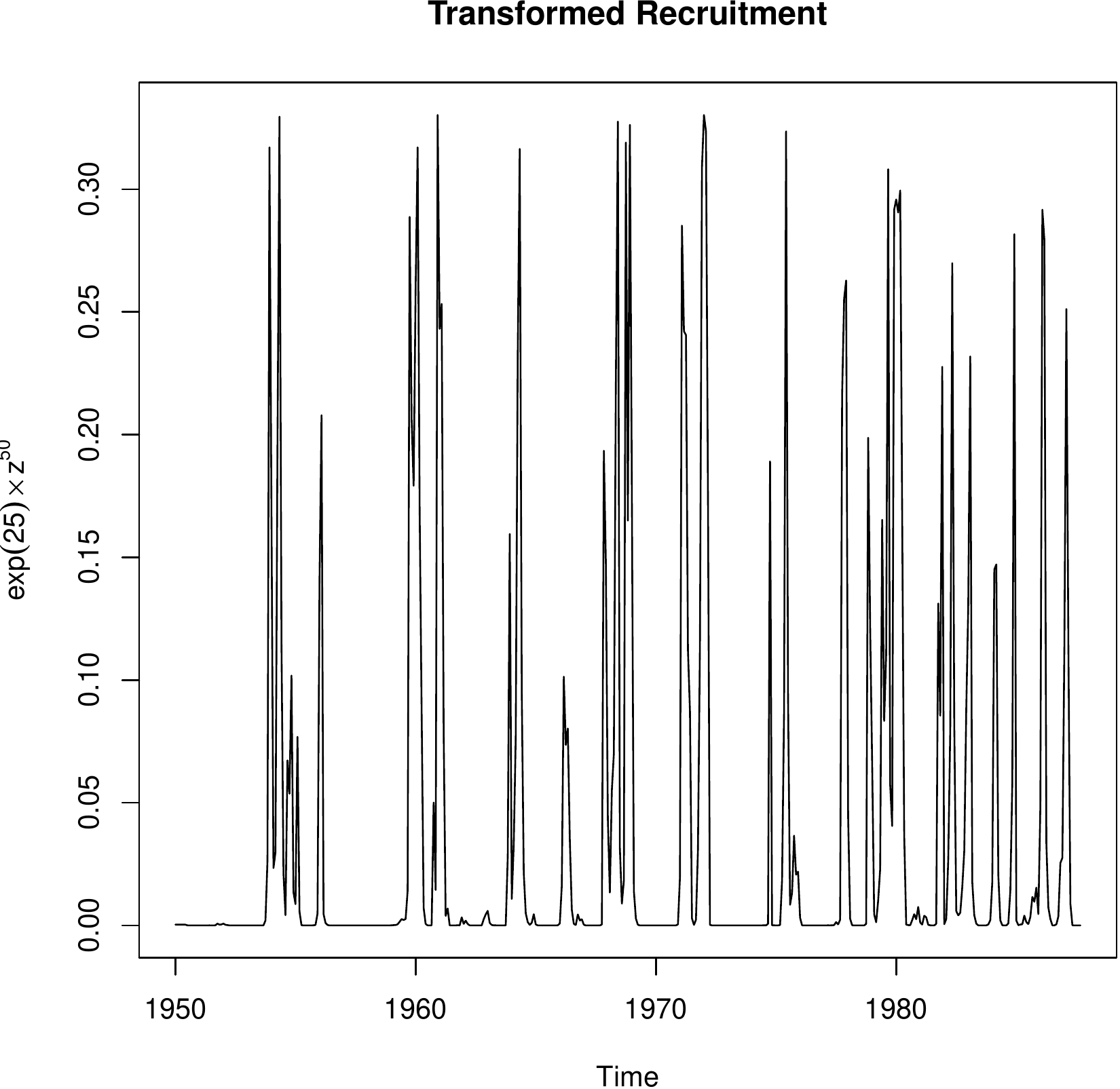}}
\caption{The original and the transformed Recruitment time series.}
\label{fig:rec1}
\end{figure}

\begin{figure}
\centering
\subfigure [Rec transformation $\exp\left(25\right)\times\bZ^{50}$: converging frequency slightly larger than $0.02$.]{ \label{fig:rec_freq1}
\includegraphics[width=6cm,height=6cm]{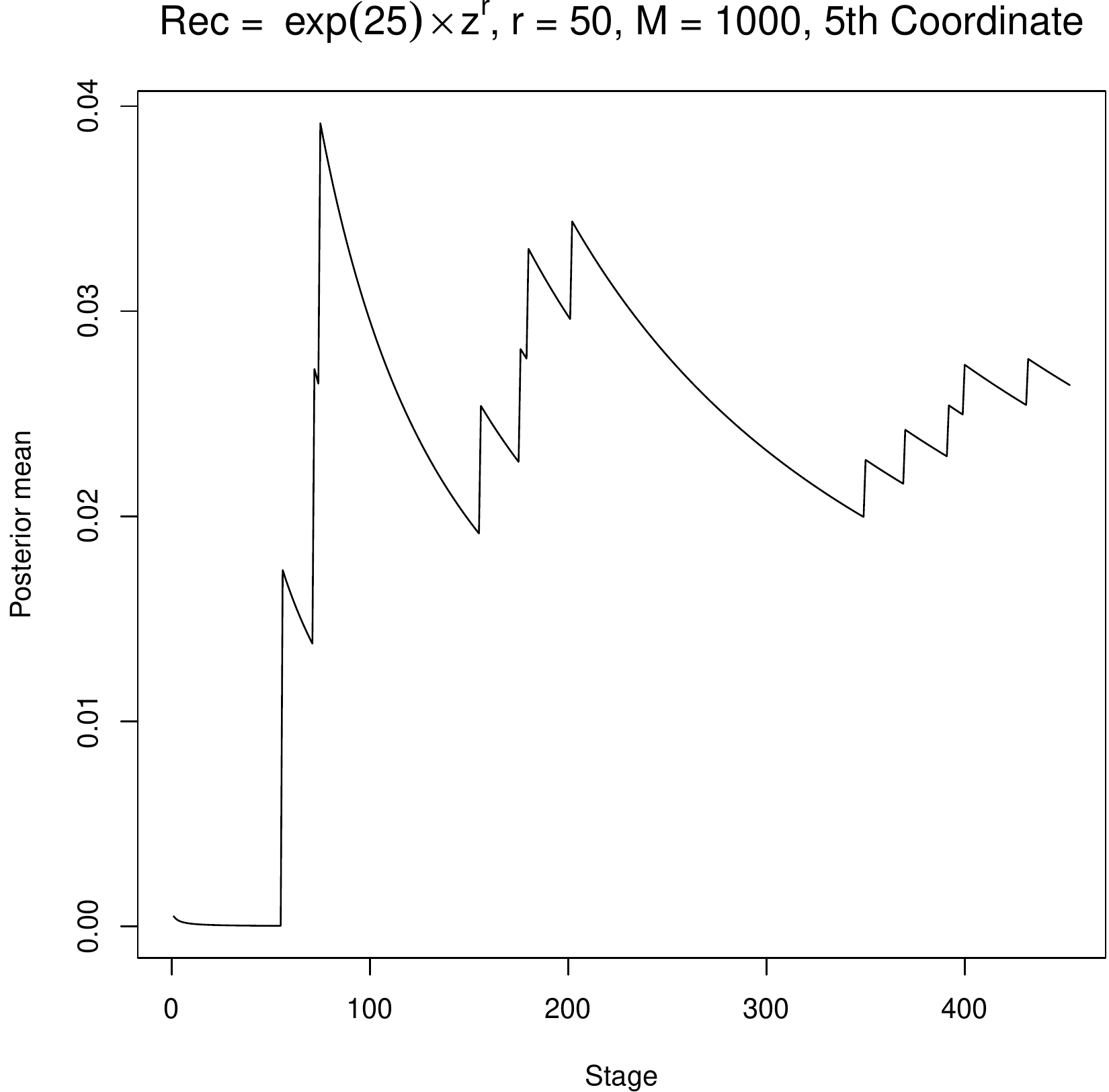}}
\hspace{2mm}
\subfigure [Rec transformation $\exp\left(25\right)\times\bZ^{50}$: converging frequency around $0.08$.]{ \label{fig:rec_freq2}
\includegraphics[width=6cm,height=6cm]{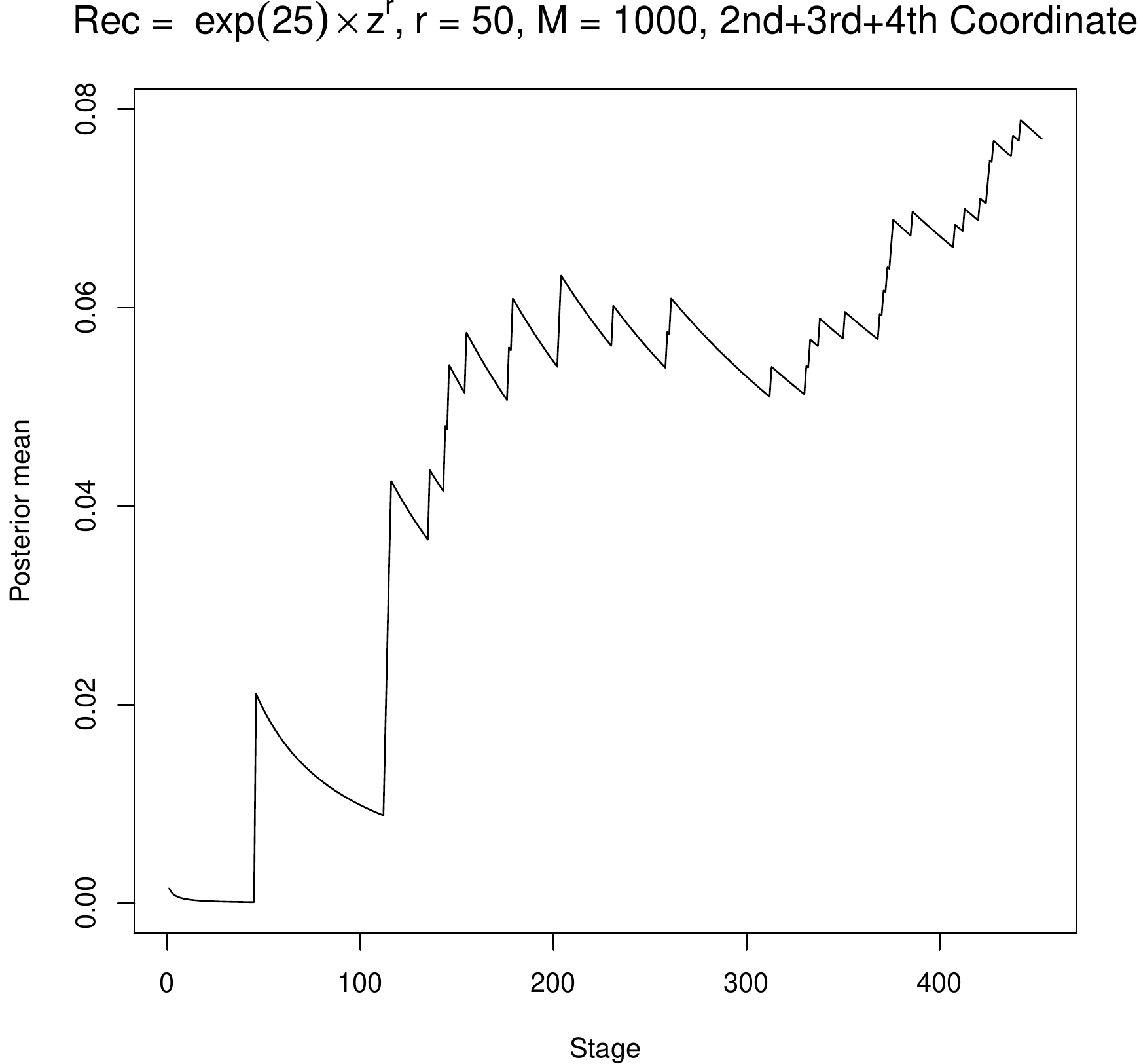}}
\caption{Bayesian results for frequency determination of the Recruitment time series with transformed time series $\exp\left(25\right)\times\bZ^{50}$.}
\label{fig:rec_freq}
\end{figure}

\subsubsection{Harmonics}
\label{subsubsec:harmonics}
Since in reality most signals are not sinusoidal, it is preferable to use harmonics to model such signals. In this respect, we consider Example 4.12 of
\ctn{Shumway06} where a signal is constructed using a sinusoid oscillating at two cycles per unit time, and 5 harmonics obtained from the sinusoid oscillating at
decreasing amplitudes. Specifically, their signal is given by
\begin{equation}
x_t = \sin(2\pi 2t)+0.5\sin(2\pi 4t)+0.4\sin(2\pi 6t)+0.3\sin(2\pi 8t)+0.2\sin(2\pi 10t)+0.1\sin(2\pi 12t),
\label{eq:harmonics}
\end{equation}
for $0\leq t\leq 1$. The original signal $\bX$ and the transformation $\bZ^{2}$ are displayed in Figure \ref{fig:harmonics1}, after considering $201$ equidistant
points in the time interval $[0,1]$. Note that the original signal is not even close to sinusoidal. For the transformation $\bZ^r$, we chose $r=2$ such that 
the structure of $\bX$ is essentially retained, but the gaps between the oscillations are increased to facilitate detection of the frequencies. 
\begin{figure}
\centering
\subfigure [The original signal with 6 harmonics.]{ \label{fig:harmonics_original1}
\includegraphics[width=10cm,height=6cm]{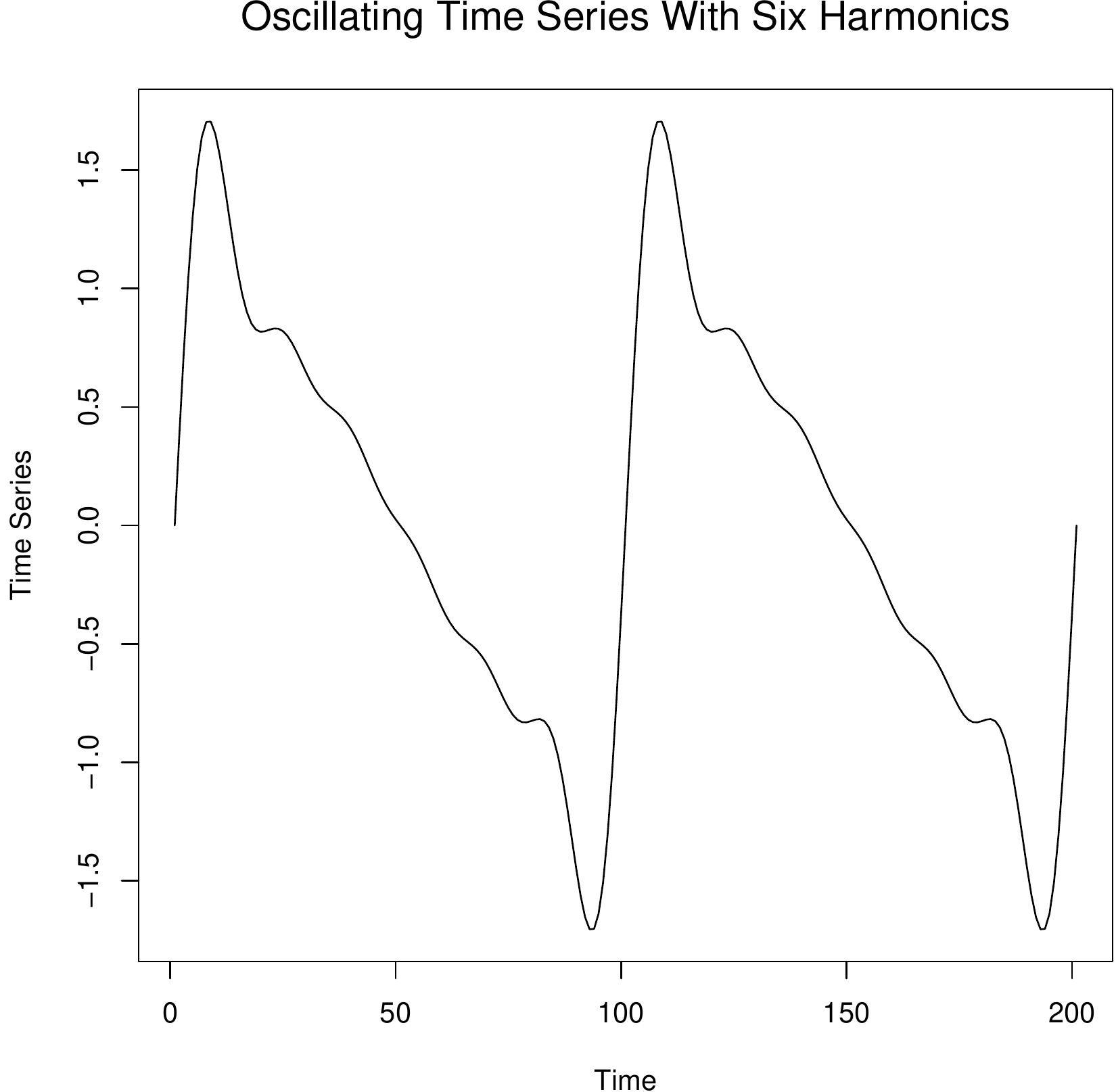}}
\vspace{2mm}
\subfigure [The transformed signal with 6 harmonics.]{ \label{fig:harmonics_transf1}
\includegraphics[width=10cm,height=6cm]{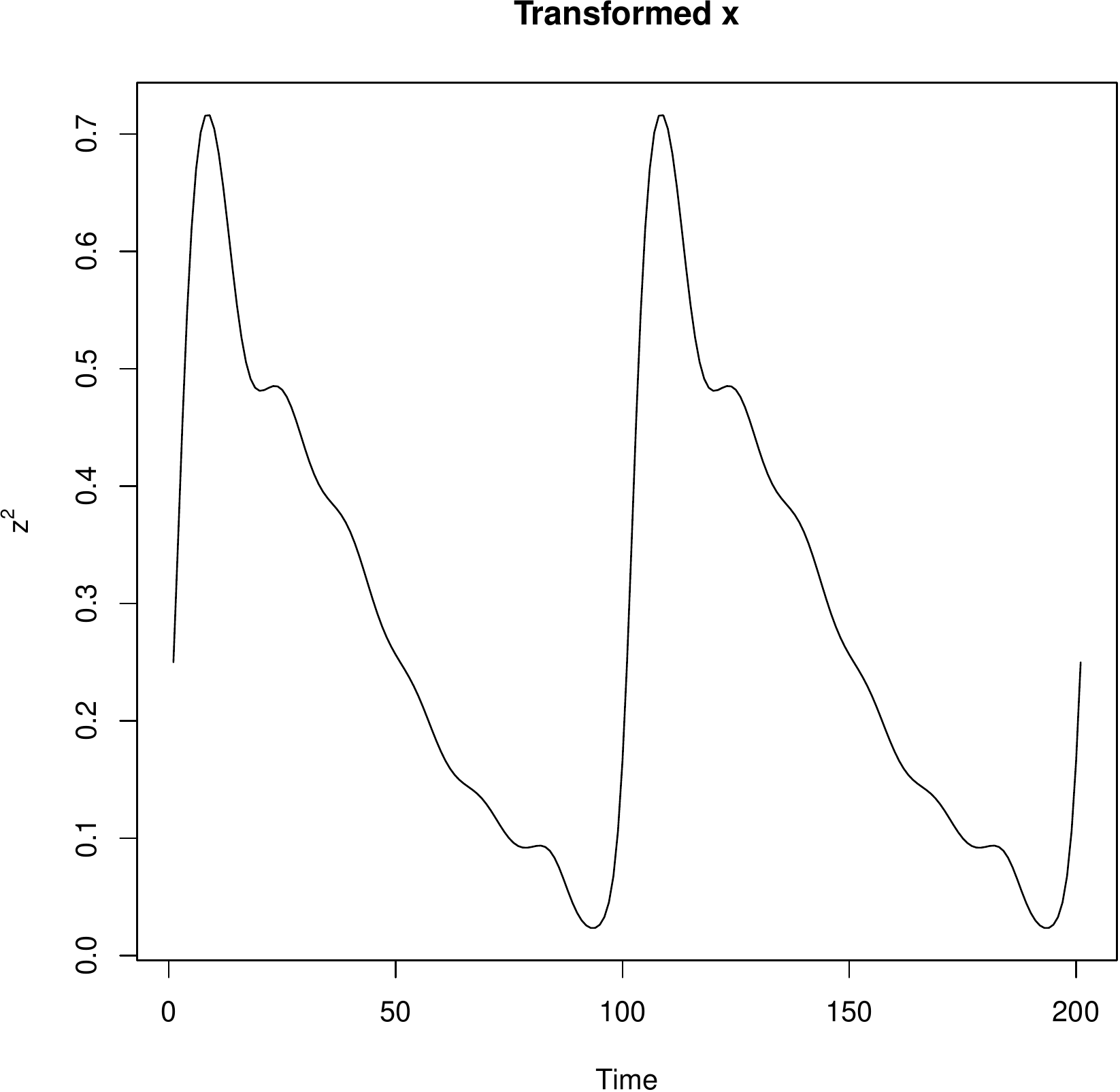}}
\caption{The original and the transformed signal with 6 harmonics.}
\label{fig:harmonics1}
\end{figure}

Since $\bZ^2$ suggests multiple frequencies that are likely to be close to each other, we chose $M=150$ to divide $[0,1]$ into larger number of finer sub-intervals
compared to the previous examples to properly detect the oscillations. Application of our Bayesian procedure revealed 6 distinct values out of $M=150$ at the end of the $201$-th
iteration, while the rest converged to zero. We take the averages of the co-ordinates yielding the same distinct value, and present the results in Figure \ref{fig:harmonics2},
after multiplication by $201$, to yield the Bayesian results on frequencies per unit time. As is evident from the diagrams, the final iterations produced
the frequencies $2,4,6,8,10,14$, obtained after rounding off the values. Except the frequency $14$, which somewhat overestimates the true frequency $12$, the others
are indeed the true frequencies. That so accurate results are obtained by our Bayesian method even for a challenging time series with small length, is really encouraging.
\begin{figure}
\centering
\subfigure [$r=10,M=100$. True frequency $=2$.]{ \label{fig:har1}
\includegraphics[width=4.5cm,height=4.5cm]{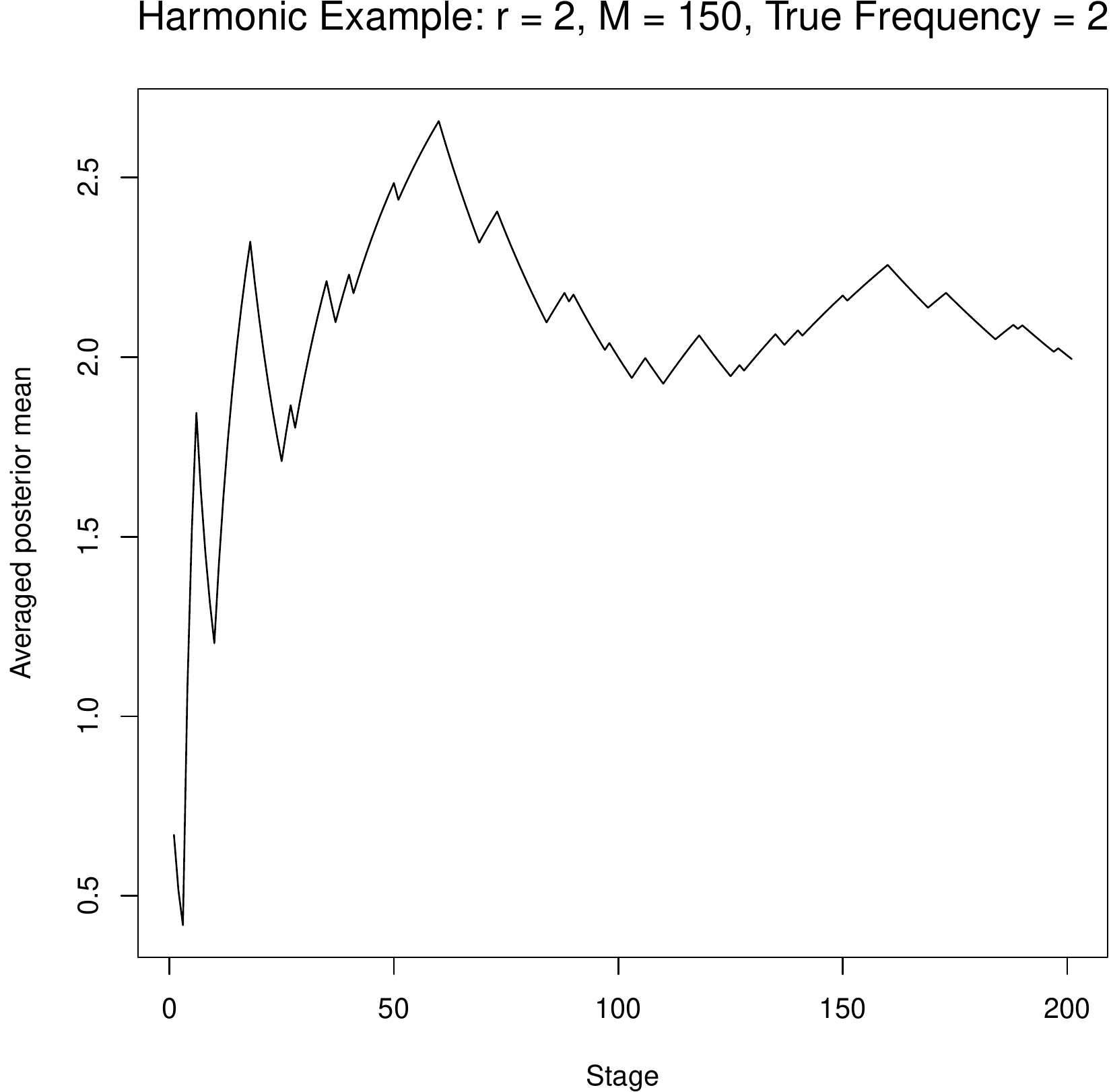}}
\hspace{2mm}
\subfigure [$r=10,M=100$. True frequency $=4$.]{ \label{fig:har2}
\includegraphics[width=4.5cm,height=4.5cm]{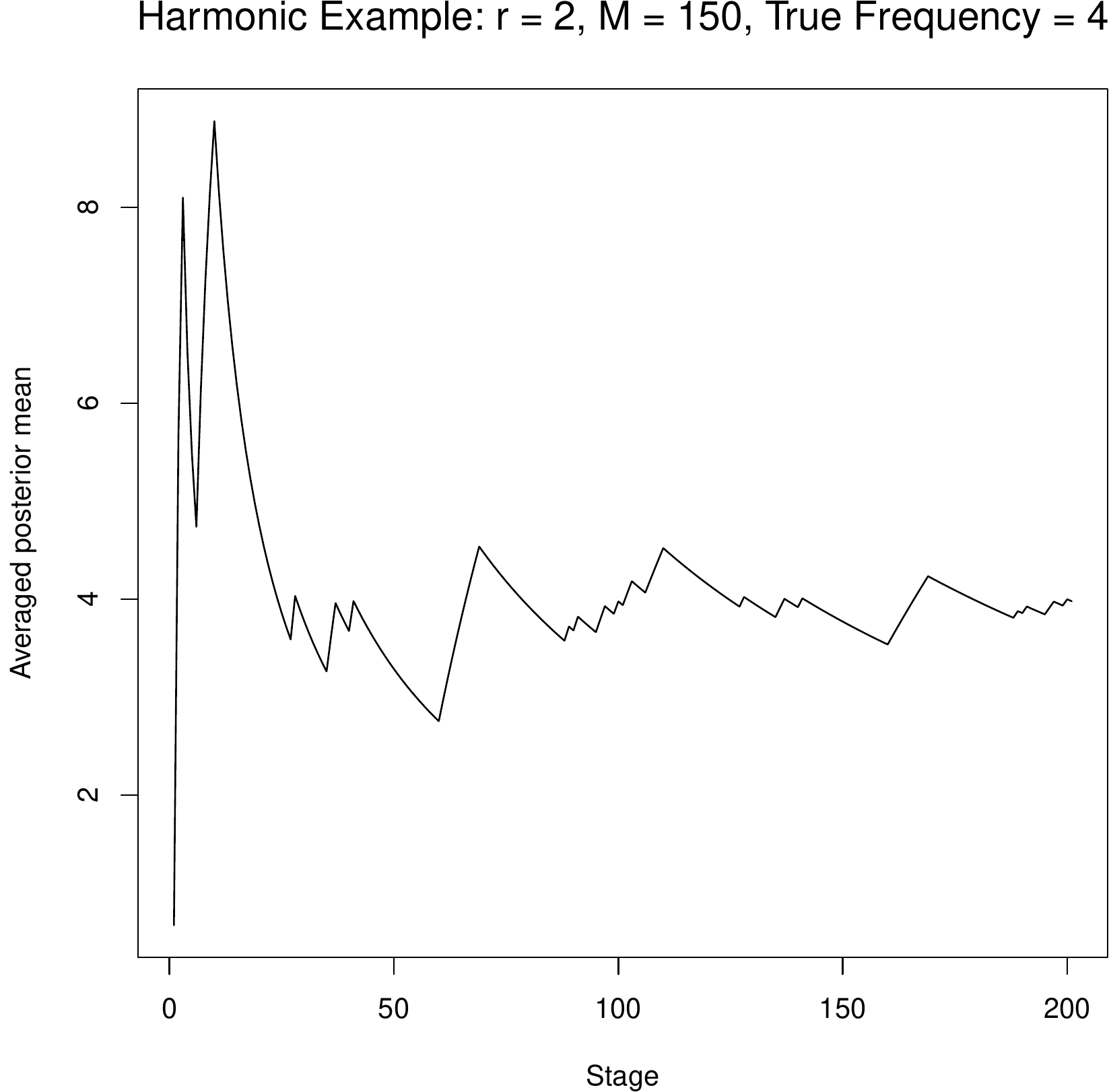}}
\hspace{2mm}
\subfigure [$r=10,M=100$. True frequency $=6$.]{ \label{fig:har3}
\includegraphics[width=4.5cm,height=4.5cm]{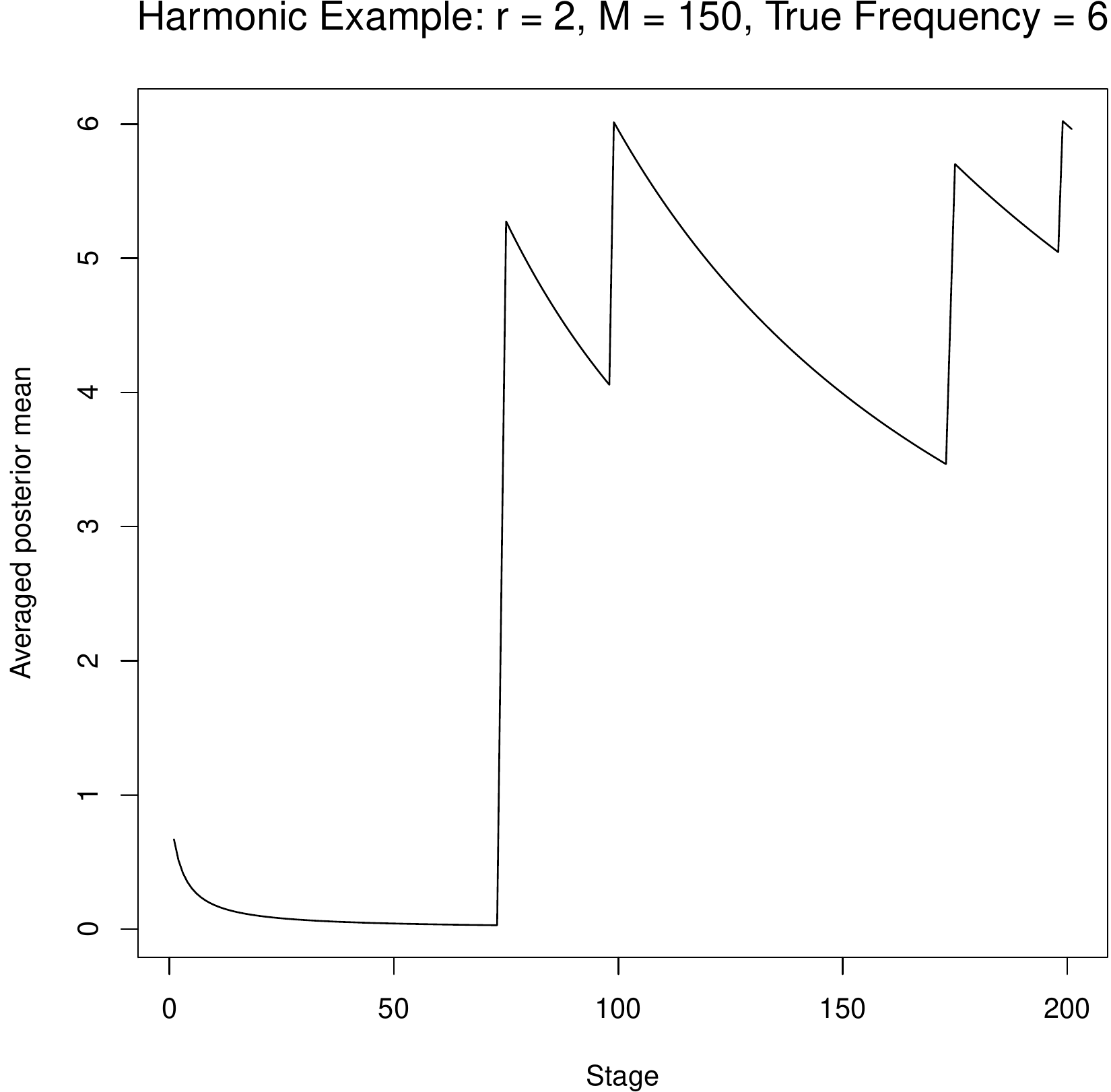}}\\
\vspace{2mm}
\subfigure [$r=10,M=100$. True frequency $=8$.]{ \label{fig:har4}
\includegraphics[width=4.5cm,height=4.5cm]{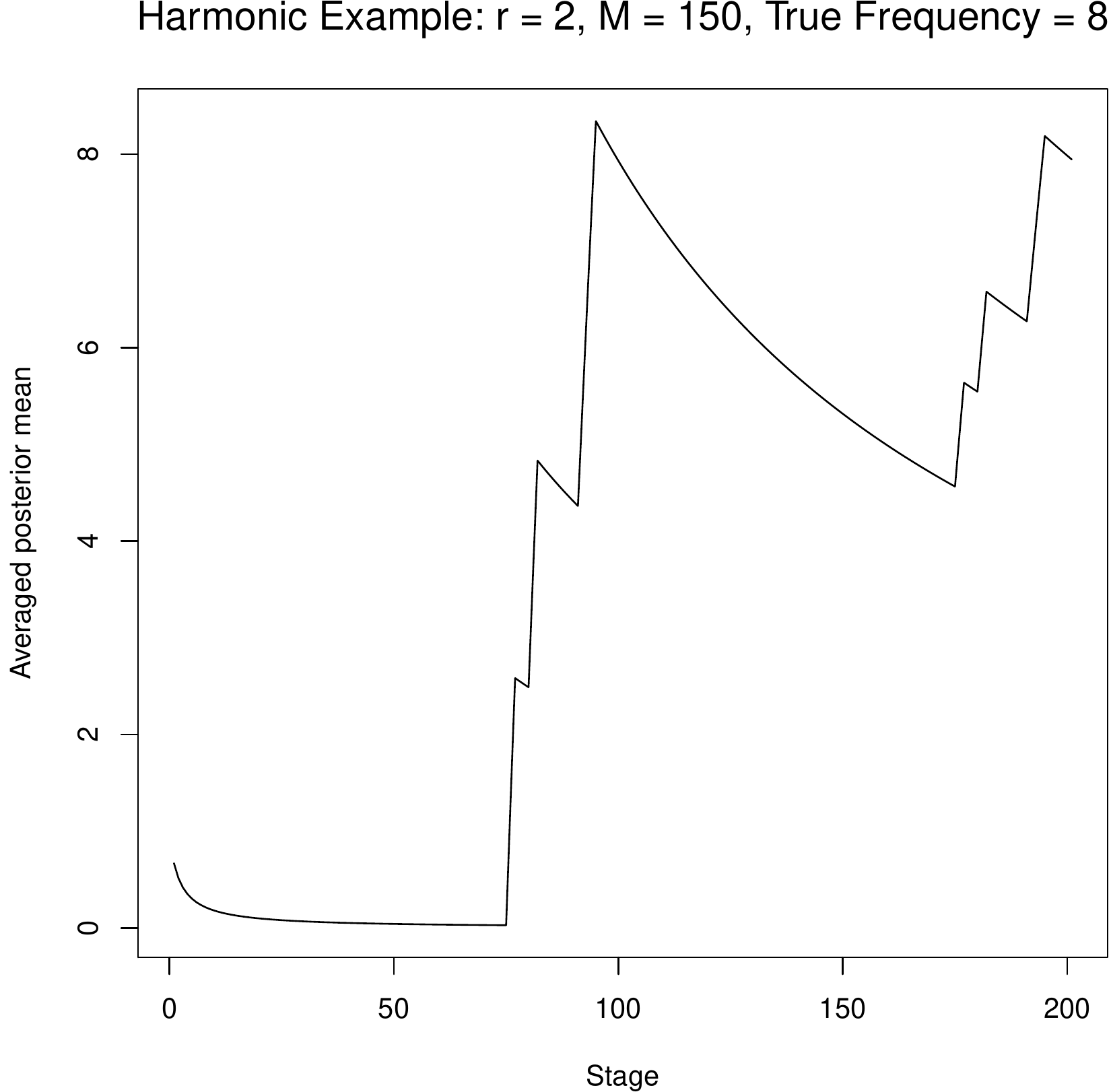}}
\hspace{2mm}
\subfigure [$r=10,M=100$. True frequency $=10$.]{ \label{fig:har5}
\includegraphics[width=4.5cm,height=4.5cm]{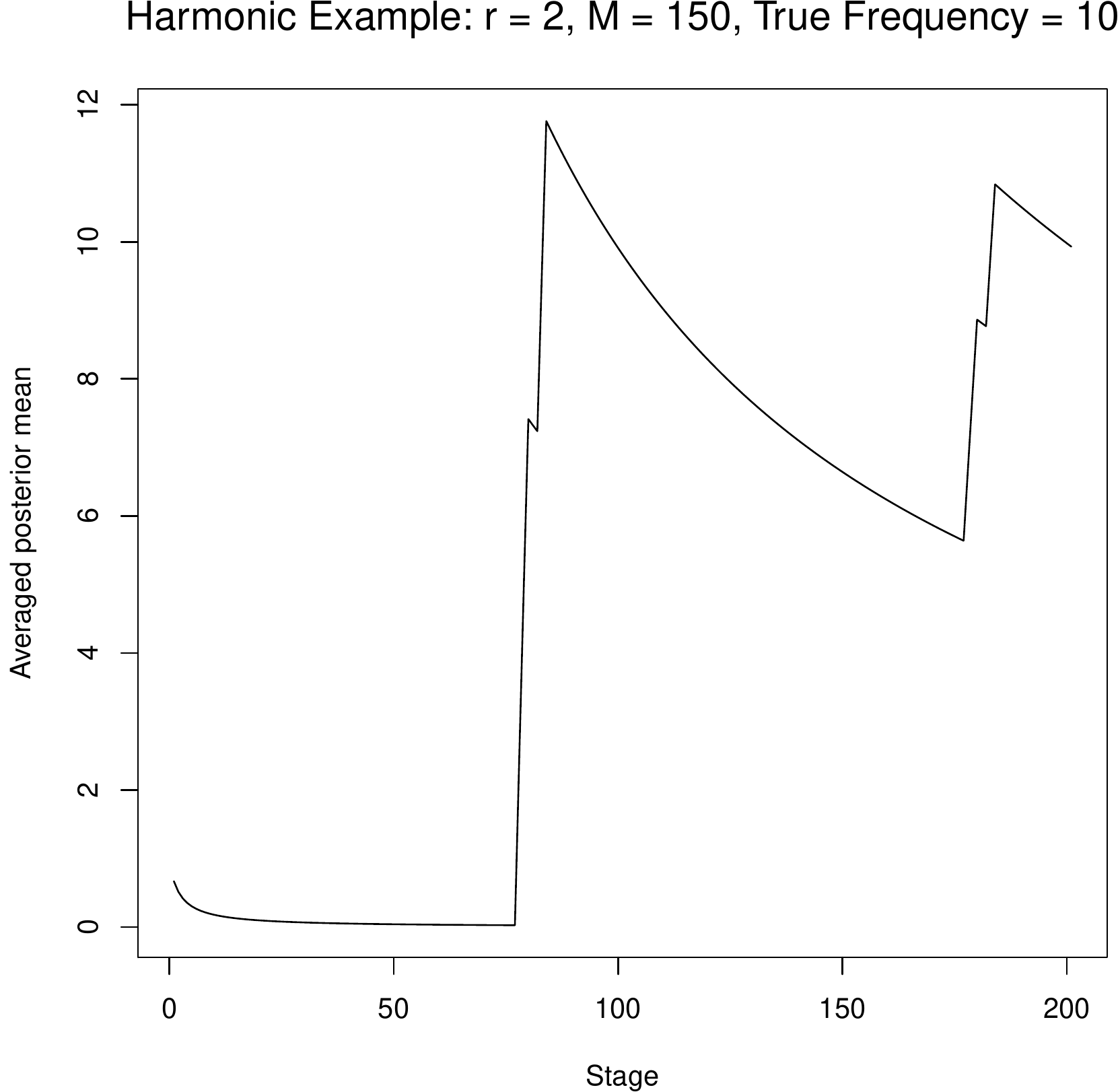}}
\hspace{2mm}
\subfigure [$r=10,M=100$. True frequency $=12$.]{ \label{fig:har6}
\includegraphics[width=4.5cm,height=4.5cm]{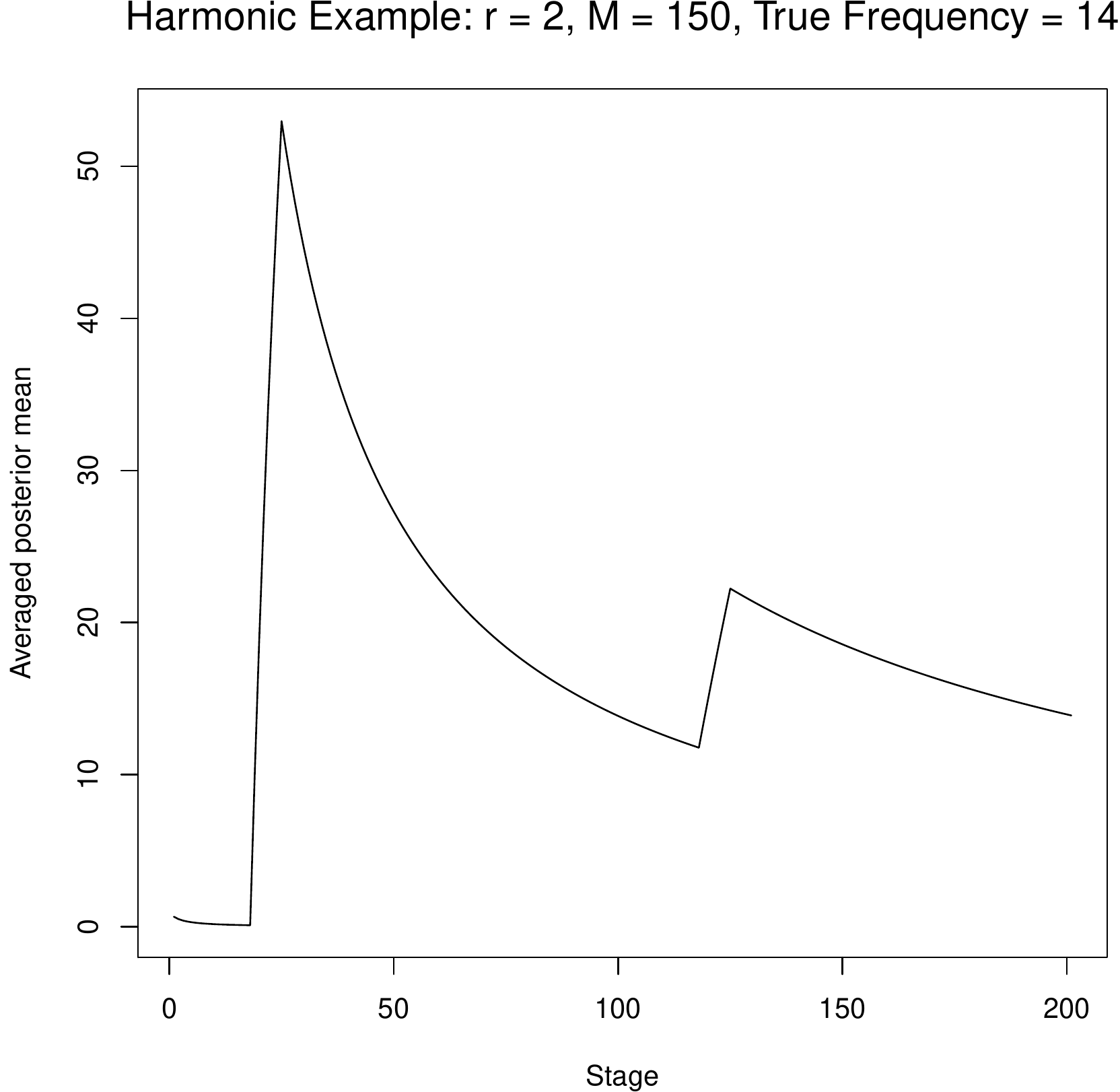}}
\caption{Illustration of our Bayesian method for determining multiple frequencies in non-sinusoidal signals. 
Here the true frequencies are $2$, $4$, $6$, $8$, $10$ and $12$ oscillations per unit time.}
\label{fig:harmonics2}
\end{figure}

\section*{Acknowledgment}
We are grateful to Dr. Satyaki Mazumder for helpful comments.

\newpage

\normalsize
\bibliographystyle{natbib}
\bibliography{irmcmc}

\end{document}